\documentclass{imsart}
\RequirePackage[OT1]{fontenc}

\usepackage[a4paper,left=10mm,top=10mm,right=10mm,bottom=10mm]{geometry}
\usepackage{mypackage}
\usepackage{mymacros2}
\usepackage{myenv}
\usepackage{mymacros_brackets}
\usepackage{DTR}

\makeatletter
\newcommand*{\addFileDependency}[1]{
  \typeout{(#1)}
  \@addtofilelist{#1}
  \IfFileExists{#1}{}{\typeout{No file #1.}}
}
\makeatother

\newcommand*{\myexternaldocument}[1]{%
    \externaldocument{#1}%
    \addFileDependency{#1.tex}%
    \addFileDependency{#1.aux}%
}
\myexternaldocument{Supplement}

\numberwithin{equation}{section}
 \counterwithin{lemma}{section}
  \counterwithin{theorem}{section}
   \counterwithin{fact}{section}
   \counterwithin{definition}{section}

\begin{document}
\begin{frontmatter}


\title{On Fisher Consistency of Surrogate Losses for  Optimal Dynamic Treatment Regimes with Multiple Categorical Treatments per Stage}
%

  \begin{aug}
\author{\fnms{Nilanjana} \snm{Laha}\thanksref{b,cor}\ead[label=e1]{nlaha@tamu.edu}}
\thankstext{cor}{Corresponding author.},
\author{\fnms{Nilson} \snm{Chapagain}\thanksref{b}\ead[label=e2]{nchapagain@stat.tamu.edu}}, 
\author{\fnms{Victoria} \snm{Cicherski}\thanksref{b}\ead[label=e3]{cicherskiv114@tamu.edu}}, 
\and
\author{\fnms{Aaron} \snm{Sonabend-W}\thanksref{c}\ead[label=e4]{asonabend@gmail.com}}
\address[b]{Department of Statistics, Texas A\&M,  College Station, TX 77843\printead[presep={,\\ }]{e1,e2,e3}}
\address[c]{Google Research, Mountain View, CA 94043, USA\printead[presep={,\ }]{e4}}

\end{aug}

\runauthor{}


\begin{abstract}
   Patients with chronic diseases often receive treatments at multiple time points, or stages. Our goal is to learn  the optimal dynamic treatment regime (DTR) from longitudinal patient data.  When both the number of stages and the number of treatment levels per stage are arbitrary, estimating the optimal DTR reduces to a sequential, weighted, multiclass classification problem \citep{kosorok2019}.  In this paper, we aim to  solve the above-mentioned classification problem, simultaneously across all stages,  through Fisher consistent surrogate losses.  Although computationally feasible Fisher consistent surrogates exist in special cases—e.g., the binary treatment setting—a unified theory of Fisher consistency remains largely unexplored. 
    We establish necessary and sufficient conditions for DTR Fisher consistency within the class of non-negative, stagewise separable surrogate losses. To our knowledge, this is the first result in the DTR literature to provide necessary conditions for Fisher consistency within a non-trivial surrogate class. Furthermore, we show that many convex surrogate losses fail to be Fisher consistent for the DTR classification problem, and we formally establish this inconsistency for  smooth, permutation equivariant, and relative-margin-based convex losses. Building on this, we propose  SDSS (Simultaneous Direct Search with Surrogates), which uses smooth, non-concave surrogate losses to learn the optimal DTR. We design a computationally fast, gradient-based algorithm for implementing SDSS. When the optimization error is small, we establish a sharp upper bound on SDSS's regret decay rate. We assess the numerical performance of SDSS through simulations. Finally, we demonstrate its real-world performance by applying it to estimate optimal fluid resuscitation strategies for severe septic patients using electronic health record data.
    

\end{abstract}
\begin{keyword}
\kwd{Dynamic treatment regimes}
\kwd{Classification}
\kwd{Fisher consistency}
\kwd{Classification calibration}
\kwd{Non-convex optimization}
\end{keyword}
  \end{frontmatter}  
  
\section{Introduction}
\label{sec: introduction}

Personalized medicine seeks to customize the treatments according to individual patient characteristics \citep{kosorok2019,deliu2022reinforcement}. During the treatment of chronic diseases, such as cancer, sepsis, and diabetes, patients may receive treatments multiple times \citep{johnson2018comparative,Raghu,sonabendw2021semisupervised}. In recent times, plenty of such longitudinal data has been available in the form of electronic health records (EHR) \citep{Raghu}. 
Our objective is to learn the optimal sequence of treatments, collectively referred to as a treatment policy or dynamic treatment regime (DTR), from such patient data. In line with the contemporary DTR literature, we define the optimal DTR to be the DTR that maximizes the expected counterfactual outcome, also known as the value function \citep{chakraborty2013,tsiatis2019dynamic}. We refer to the corresponding problem as the DTR learning problem. We consider the setting where patient data are collected over $T$ stages for some $T>1$, and at each stage $t\in\{1,\ldots,T\}$, there are $k_t\geq 2$  available treatment options. We assume that each $k_t$ is fixed and does not grow with the sample size $n$. To distinguish this setting from the binary-treatment setting, i.e., where   $k_t=2$ for all  $t\in\{1,\ldots,T\}$, we refer to it as the general DTR setting throughout this paper.


Plenty of approaches have been proposed in recent decades for estimating the optimal DTR. Traditional methods either model the full data distribution \citep{xu2016, zajonc2012} or focus on partial modeling.  Popular examples of the latter include Q-learning \citep{moodie2014q,murphy2005,Watkins1989,zhang2018interpretable,chakraborty2013},  stochastic tree search algorithm \citep{sun2021stochastic}, A-learning,
and marginal structural mean models \citep{chakraborty2013,moodie2007demystifying,murphy2003,Robins2004,schulte2014}.
Model-based methods can be computationally  efficient. 
 However, in this case, modeling carries a substantial risk of misspecification \citep{kosorok2019}. Model-based approaches typically fit the models using parametric or nonparametric methods via backward recursion: the model for stage $T$ is fitted first, and models for stages $T-1$, $T-2$, etc., are recursively fitted using the estimates from later stages, causing model misspecification errors to propagate and accumulate over stages  \citep{murphy2005}. 
 A-learning and marginal structural mean models are more robust to model misspecification than Q-learning. However, they still require the correct specification of the contrasts of certain functions called Q-functions \citep{schulte2014}. The cost of model misspecification is expected to be higher in high-dimensional data or when the optimal DTR has a nonlinear decision boundary, as is likely in real-world medical datasets  \citep{Laha2024}.

   Our work focuses on a more recent alternative approach to learning the optimal DTR. This approach, called the direct search or value search, frames the DTR learning problem as a weighted, sequential, multiclass classification\footnote{Multiclass classification is a generalization of binary classification where the number of classes can exceed two \citep{tewari2007,zhang2004}.} problem \citep{kosorok2019, zhao2015}.
    The advantage of this approach is that it  does not require modeling any quantities except the treatment assignment probabilities, which are also known in experimental designs. 
    Hereafter, we refer to the above sequential multiclass classification  as the sequential DTR classification, or simply the DTR classification.
     The  DTR classification problem transforms into  a discontinuous and non-convex loss-minimization problem, which is atypical of classification problems \citep{bartlett2006,tewari2007,ramaswamy2016convex}.  A common strategy to address this is to replace the discontinuous loss with a well-behaved surrogate loss, leading to the surrogate loss minimization approach 
\citep{bartlett2006,tewari2007,ramaswamy2016convex}. The  surrogate losses that lead to the same population-level solution as the original discontinuous loss are  called calibrated or Fisher consistent for that loss \citep{bartlett2006,tewari2007,ramaswamy2016convex}.  For example, binary classification  leads to an optimization problem with the discontinuous 0-1 loss $x\mapsto 1[x\leq 0]$.   The support vector machine replaces the  0-1 loss with the hinge loss $x\mapsto \max\{0,x\}$, which is Fisher consistent for the 0-1 loss  \citep{bartlett2006}.  
     
      However, the  discontinuous loss  corresponding to the sequential DTR classification, referred to hereafter as the DTR loss, is more convoluted than the discontinuous loss of the binary or multiclass classification.
      Since  no general recipe exists for constructing computationally efficient Fisher consistent surrogates for an arbitrary loss  \citep{finocchiaro2019embedding}, finding a Fisher-consistent surrogate loss for the DTR loss is not straightforward \citep{kosorok2019}.
    As a consequence, to date, Fisher consistent surrogate losses for  the sequential DTR classification have been obtained only in  special cases, e.g.,
 the single-stage case ($T=1$) \citep{bennett2020efficient, chen2018estimating, luedtke2016, meng2020near, wang2018learning,pan2021improved,zhang2020multicategory},  in the context of maximizing the conditional survival function \citep{xue2022multicategory}, and, most notably, the widely studied  binary-treatment case  \citep{liu2018augmented, zhao2015,Laha2024,jiang2019entropy,liu2024controlling,liu2024learning,zhao2020constructing}. As a result, the general DTR case $(k_t\geq 2)$ has received much less attention in the direct search literature, despite its practical relevance  \citep{tao2017adaptive,xue2022multicategory}. 

 In addition, most Fisher-consistency-related work on DTR focuses on identifying Fisher consistent surrogate losses. In contrast, research across machine learning has moved beyond  finding Fisher consistent surrogates to understanding the specific properties of surrogates that drive Fisher consistency. For instance, in binary and multiclass classification, considerable effort has been devoted to establishing necessary and sufficient conditions for Fisher consistency \citep{bartlett2006,tewari2007,ramaswamy2016convex,wang2023unified}.
 Other areas where  Fisher consistency of surrogate losses has been investigated in depth include multi-label learning \citep{gao2011}, ranking \citep{duchi2010,calauzenes2012}, and structured prediction \citep{osokin2017structured}. In comparison, the properties of surrogate losses that elicit Fisher consistency for the sequential DTR classification remain poorly understood even in the binary treatment case. Yet, such knowledge is important for understanding the scope of  model-free  DTR learning. In addition to constructing computationally feasible Fisher consistent surrogate problems for DTR classification, this paper also attempts to address this conceptual gap. Before outlining our contributions, we provide some important clarifications.

 First, many DTR approaches break the sequential DTR classification problem into $T$ single-stage DTR classification problems, which are then solved recursively using tree-based search \citep{tao2017adaptive,tao2018tree} or Fisher consistent surrogates for the single-stage case \citep{liu2024learning,liu2018augmented,hager2018optimal}. We refer to this approach as the stagewise direct search, which bypasses the need to solve the full DTR classification problem. Most existing work on direct search for the general DTR setting falls within this category.  Although computationally efficient, these methods either incur a loss of sample size \citep[e.g., backward outcome weighted learning of][]{zhao2015} or depend on model-based objects such as pseudo-outcomes (e.g., the augmented outcome weighted learning of \cite{liu2018augmented} or the tree-based methods of \cite{tao2017adaptive,tao2018tree}). Thus, stagewise direct search, a hybrid between direct search and model-based methods, does not align with the goal of this paper. We aim to directly attack the sequential DTR classification, and learn the optimal treatment assignment for each stage simultaneously without estimating any pseudo-outcomes. To emphasize the distinction from stagewise direct search, we refer to our approach as Simultaneous Direct Search with Surrogates (SDSS).

Second, the surrogate loss theory for the general DTR setting differs substantially  from the binary treatment setting. The binary treatment case  turns into a sequential binary classification, where the general case turns into a sequential multiclass classification. An advantage of binary classification is that it facilitates a special class of surrogate losses called margin-based losses, which has no direct extension to  the multiclass setting    \citep{wang2023unified}. 
 Moreover, DTR classifications in the binary  and  general settings are at least as different as binary  and multiclass classification, which have spurred separate research trajectories \citep{tewari2007, zhang2004, zou2008, ramaswamy2016convex}.

Finally, Fisher consistency is quite well understood in the special case when $T=1$ (also referred to as individualized treatment regimes or ITR) due to the efforts of \cite{zhang2020multicategory, pan2021improved, bennett2020efficient}, among others. However, the theoretical insights from this setting do not directly extend to the general $T$-stage case. The primary reason is that when $T=1$, DTR classification reduces to  a weighted multiclass classification problem, allowing the exploitation of the  extensive Fisher consistency literature on multiclass classification \citep[cf.][]{tewari2007, ramaswamy2016convex, zhang2004, zhang2014multicategory}.
 In contrast, sequential DTR classification differs substantially from multiclass classification and, consequently, from single-stage DTR. We will discuss these differences in detail in  Section \ref{sec: necessity}.

   \subsection{Main contributions}
 Given the main focus of this paper is the Fisher consistency of   surrogate losses, similar to \cite{liu2024learning,liu2024controlling,zhao2015,Laha2024}, we restrict our attention to the case where the propensity scores, i.e., the  treatment assignment probabilities, are known. Our contribution is two-fold. The first part of the paper focuses on the subject of Fisher consistency in the  DTR classification problem. The second part focuses on the construction of our method, SDSS, and investigates its theoretical and empirical performance.

 \subsubsection{Thoretical results on Fisher consistency}
 \label{sec: contribution: part 1}
  \paragraph*{Concave losses} 
 
  Concave surrogate losses  are attractive because they result in a convex optimization problem.
      Since DTR learning is a maximization problem, concave—rather than convex—surrogates are more appropriate.  In Section \ref{sec: convex loss}, we show that many popular concave surrogates that are Fisher consistent in the  single-stage setting fail to be  Fisher consistent when extended to the $T= 2$ case. Our examples include both smooth and non-smooth (e.g., hinge loss) losses, as well as losses using sum-zero constraints.  In particular, Theorem \ref{theorem: CC} shows that, under  smoothness conditions,   permutation equivariant and  relative-margin-based (PERM)  surrogate losses can not be Fisher consistent if they are also concave  \citep{wang2023unified}. 
Here, we use  PERM losses because (a)
  they provide enough structure that can be exploited to derive interpretable Fisher consistency-related results and (b) it is a rich class including numerous popular  multivariate losses (see Section \ref{sec: convex loss} for details).

Our previous work in  \cite{Laha2024} established the inconsistency of  margin-based, smooth, concave losses in the binary treatment setting.  However, margin-based losses constitute an useful but small subset of all possible surrogate losses in that setting \citep{wang2023unified}. Hence, our Theorem \ref{theorem: CC} does not follow from \cite{Laha2024}'s impossibility results. Moreover, our proof technique differs substantially from that of \cite{Laha2024} because multiclass PERM losses  exhibit more degrees of variability than the margin-based losses of the binary treatment setting (detailed comparison with \cite{Laha2024} is provided in Section \ref{sec: related: laha 2021}; see also Section \ref{sec: CC: proof outline}).

  At a high level, these results provide a negative answer to the possibility of convexifying sequential DTR classification.  This negative finding has two implications. First, model-free DTR learning will likely incur the computational burden of nonconvex optimization. Second, since single-stage DTR classification problems can be convexified \citep[cf.][]{zhang2020multicategory,pan2021improved}, this highlights the intrinsic difficulty of the multi-stage ($T>1$) setting compared to the single-stage setting.

    \paragraph*{Necessary and sufficient conditions} 
 The  negative results on concave surrogates forces us to look beyond concave surrogate losses. In Section \ref{sec: necessity}, we consider the class of non-negative stagewise separable surrogates, i.e., surrogates that can be written as a product of non-negative single-stage surrogates (see \eqref{def: product psi}). This class  was chosen because it has   been used in the DTR Fisher-consistency literature due to its computational and theoretical advantages \citep{xue2022multicategory,Laha2024}. Informally, we show that surrogates in this class are Fisher consistent if and only if the corresponding single-stage surrogates  (a) are Fisher consistent for single-stage DTR and (b) for $t \geq 2$, have image sets with  sufficiently large convex hulls (see Theorems~\ref{theorem: sufficient conditions} and~\ref{theorem: necessity} for the precise result).  
    The second condition restricts the class of Fisher consistent separable surrogates for $T \geq 2$, which underscores  the inherent challenge of the multi-stage setting compared to the single-stage setting. To the best of our knowledge, this is the first result establishing necessary conditions for DTR Fisher consistency in non-trivial surrogate loss classes.
    In Section~\ref{sec: examples of psi satisfying N1 and N2}, we establish the existence of smooth, non-convex  Fisher consistent losses. This finding indicates that sequential DTR classification can at least be smoothed, even though convexification is not currently attainable.
   Moreover, we show that our surrogates preserve Fisher consistency even when the search space of policies is restricted to a smaller class, e.g., the class of linear policies, provided the optimal DTR lies in that class  (Lemma \ref{lemma: theorem with linear policy}). 

    \subsubsection{Methodological contribution}
 
The proposed DTR learning method, SDSS, solves a surrogate value optimization problem based on the above-mentioned  smooth, non-convex, Fisher consistent surrogate losses. The smoothness of the surrogates allows for gradient-based optimization, enabling fast implementation that scales to the sample sizes  of modern EHR data. However, the non-convexity of these losses introduces optimization challenges, illustrated with a toy example in Section~\ref{sec: implementation}. To mitigate these issues, we incorporate strategies such as random reinitialization and minibatching, which help prevent optimization iterates from stagnating in suboptimal regions. 

SDSS can accommodate any number of stages and any number of treatments per stage. It is simultaneous in that it learns the treatment assignments for all stages simultaneously. Therefore, unlike the stagewise methods, SDSS does not require modeling pseudo-outcomes and is completely model-free when the propensity scores are known.   SDSS is flexible with policy classes in that it can be implemented with any policy class. 

 
\paragraph*{Regret decay rate} In Section \ref{sec: regret decay}, we show that if the policy class is sufficiently rich, e.g., neural network classes, and the optimization error is small, then we can provide a tight upper bound on the regret of SDSS—defined as the difference between the expected total reward of the optimal DTR and that of the DTR estimated by SDSS.   To this end, we use a small noise assumption \citep{tsybakov2004} (see Assumption \ref{assump: small noise}) and standard smoothness assumptions similar to those in \cite{zhao2015,sun2021stochastic}. For specific surrogates and neural network policy classes, SDSS's regret decay rate matches the minimax rate of risk decay for binary classification under similar assumptions  (Corollary \ref{cor: neural network}). Our empirical analysis in Section \ref{sec: empirical} suggests that despite the non-convexity, SDSS may asymptotically outperform traditional DTR learning methods in scenarios with high model misspecification risk.

\subsubsection{Organization}  In Section \ref{sec: set-up}, we discuss the mathematical formulation of the DTR classification                                  problem and formally define  Fisher consistency for the DTR classification problem. Sections \ref{sec: convex loss} and \ref{sec: necessity} describe the Fisher consistency properties of concave and separable losses, respectively, as discussed under Contribution \ref{sec: contribution: part 1}. Section  \ref{sec: method} introduces SDSS and discusses its implementation. In Section \ref{sec: regret decay}, we provide the regret analysis of SDSS.  Section \ref{sec: empirical} presents an empirical comparison of SDSS with other DTR methods using various simulation setups (see Section \ref{sec: simulation}) and EHR data from sepsis patients (see Section \ref{sec: application}). Section \ref{sec: related} compares our work with related literature. We conclude with some discussion and suggestions for future research in section \ref{sec: discussion}. For the reader’s convenience, a list of notation is provided in Table \ref{tab: list of notation}.
 
 \subsection{Notation}
We let $\NN$ denote the set of natural numbers.  For any integer $k\in\NN$, let $[k]$ denote the set $\{1,2,\ldots,k\}$.  For any vector $\mx\in\RR^k$ and any $j\in[k]$, $\mx_j$ denotes the $j$th element of $\mx$, so that $\mx=(\mx_1,\ldots,\mx_k)$. For a vector $\mx\in\RR^k$, $\argmax(\mx)$ denotes the set $\argmax_{i\in[k]} \mx_i$, and in particular, $\max(\argmax(\mx))=\max\{i\in[k]:\mx_i=\max(\mx)\}$. 

We let $\RR$ denote the set of real numbers and $\bRR$ the extended real line, i.e., $\RR\cup{\pm \infty}$. Denote by $\RR_{\geq 0}$ and $\RR_{>0}$ the sets $[0,\infty)$ and $(0,\infty)$, respectively.  For any $k\in\NN$, define $\mathcal S^{k-1}$ to be the simplex in $\RR^k$, i.e. $\S^{k-1}:=\{\mp\in\RR^k: \mp_i\geq 0\text{ for all }i\in[k],\ \sum_{i=1}^k\mp_i=1\}$.  The $k$-dimensional vectors of all zeros and all ones are denoted by $\mz_k$ and $\mo_k$, respectively, and $I_k$ denotes the identity matrix of order $k$. A permutation $v$ on $[k]$ is a bijection from $[k]$ to itself.  By an abuse of notation, for any $\mx\in\RR^k$,  we will denote by $v(\mx)$ the $k$-dimensional vector
 $(\mx_{v(1)},\ldots, \mx_{v(k)})$, which extends  the permutation $v$  to a map from $\RR^k$ to 
 itself. For $k\in\NN$, we let $|\cdot|_k$ denote the $\ell_k$ norm; that is, for $\mx\in\RR^k$, $\|\mx\|_k=\sum_{i=1}^k (|\mx_i|^k)^{1/k}$. If $\mx\in\NN^k$, we denote by $|\mx|_1$ the sum $\sum_{i=1}^k\mx_i$. Unless otherwise mentioned, lowercase letters such as $x$ and $y$ will denote real numbers, whereas boldface letters such as $\mx$, $\my$, $\mp$, $\mw$, $\mbu$, and $\mv$ will denote vectors.
 For any set $\C$, the indicator function $1[X\in \C]$ takes the value one if $X\in \C$ and zero otherwise. We denote by $\iint(\C)$ and $\overline{C}$ the interior and the closure of the set $\C$, respectively, and  its cardinality is denoted by $|\C|$. When  $\C \subset \RR^k$, for any  $x \in \RR$, we define  
$
x\C = \{\, x\mx : \mx \in \C \}
$. The closed convex hull of $\C$ is denoted by $\conv(\C)$. For a finite set $\C$, $\max(\C)$ denotes its maximum element. If $\C=\emptyset$, we define $\inf(\C)=\infty$ and $\sup(\C)=-\infty$. 

For any $k\in\NN$, a probability measure $P$ on $\mathcal X$, and any $P$-measurable function $h:\mathcal X\mapsto\RR$, we denote the $L_k(P)$ norm of $h$ by  $\|h\|_{P,k}=\left(\int|h(x)|^k dP(x)\right)^{1/k}$ and the uniform norm by $\|h\|_\infty=\sup_{x\in\X}|h(x)|$.
  Also,  $P[h]$ denotes the integral $\int h dP$.  We define the domain (effective domain) of a convex function  $h:\RR^k\mapsto\RR$ as in page 23 of \cite{rockafellar}, i.e.
$\dom(h)=\{\mx\in\RR^k: h(\mx)<\infty\}$. For any set $\C\subset \RR^k$, we define its support function $ \varsigma_\C$ as in Definition 2.1.1 (p. 104) of \cite{hiriart}:
\begin{equation}
    \label{def: support function}
    \varsigma_{\C}(\mx)=\sup_{\mw\in \C} \mx^T\mw \quad\text{ for all }\mx\in\RR^k.
\end{equation}
For any differentiable function $h:\RR^k\mapsto\RR$, $\grad h$ will denote the gradient of $h$.  Throughout this paper, we use the convention $\pm\infty\times 0=0$. In addition, we use $C$ and $c$ to denote generic constants that may vary from line to line.

 Many  results in this paper are  asymptotic (in $n$) in nature and thus require some standard asymptotic notations.  If $\myb_n$ and $b_n$ are two sequences of real numbers then $\myb_n \gg b_n$ (and $\myb_n \ll b_n$) implies that ${\myb_n}/{b_n} \rightarrow \infty$ (and ${\myb_n}/{b_n} \rightarrow 0$) as $n \rightarrow \infty$, respectively. Similarly $\myb_n \gtrsim b_n$ (and $\myb_n \lesssim b_n$) implies that $\liminf_{n \rightarrow \infty} {{\myb_n}/{b_n}} = C$ for some $C \in (0,\infty]$ (and $\limsup_{n \rightarrow \infty} {{\myb_n}/{b_n}} =C$ for some $C \in \RR_{\geq 0}$). Alternatively, $\myb_n = o(b_n)$ will also imply $\myb_n \ll b_n$ and $\myb_n=O(b_n)$ will imply that $\limsup_{n \rightarrow \infty} \ \myb_n / b_n = C$ for some $C \in \RR_{\geq 0}$.

\section{Mathematical setup}
\label{sec: set-up}

We assume that each patient receives $T$ treatments $A_1,\ldots,A_T$ at stages $1,\ldots,T$, respectively. 
Each treatment $A_t$ is a random variable taking value in a finite set $\mathcal A_t$ with cardinality $k_t$.
For sake of simplicity, we let $\mathcal A_t=\{1,\ldots, k_t\}$ for $t\in[T]$. We denote $\sum_{t=1}^T k_t=\K$. After each treatment, the patient observes  responses (outcomes) $Y_1,\ldots,Y_T$, which are real-valued random variables. Without loss of generality, we assume higher values of outcomes  are desirable. Hence the outcomes are also referred to as rewards. We also observe  covariates $O_1,\ldots,O_T$ for every patient, which are random vectors. The dimension of $O_t$ can change with stage $t$. The following diagram gives the sequence of events:
\[O_1\rightarrow A_1\rightarrow Y_1\rightarrow \cdots \rightarrow O_{T}\rightarrow A_{T}\rightarrow Y_{T}.\]
The $t$-th stage history $H_t$ contains everything observed before $A_t$, that is $H_1=O_1$, and $H_t=(O_1,A_1,Y_1,\ldots,O_t)$ for $2\leq t\leq T$. We denote the space of each $H_t$ by $\H_t$, i.e.,  $H_t\in\H_t$ for all $t\in[T]$. Therefore, $\H_t\subset \mathcal O_1\times \mathcal A_1\times Y_1\times \ldots\times\A_{t-1}\times  Y_{t-1}\times \mathcal O_t$, where the inclusion can be strict. Let $q$ denote the dimension of $H_T$.  The $t$-th stage individualized treatment assignment $d_t$ is a map from $\H_t$ to the treatment space $\A_t$. 
A treatment policy or DTR  is the collection of all such $T$ treatments; that is, $d=(d_1,\ldots,d_T)$ is a DTR. 

For a patient, the trajectory $\D$  will denote the sequence of all covariates, actions, and outcomes, i.e., 
\begin{equation}
    \label{def: trajectory}
    \D=(O_1,A_1,Y_1,\ldots,O_T,A_T,Y_T).
\end{equation}
Suppose $\PP$ is the  distribution of $\D$, i.e.,  $\PP$ is the distribution of the observed data. The expectation corresponding to $\PP$ will be denoted  by $\E$. We  observe $n$ independent trajectories $\D_i$'s from $n$ patients. 
The empirical distribution of $\D_1,\ldots,\D_n$ will be denoted by $\Pm$. For any function $h$ mapping $\D$ to a real number, we denote by
$\Pm[h]$ the average $n^{-1}\sum_{i=1}^nh(\D_i)$. We will denote the $t$-th stage propensity score, i.e., the  treatment assignment probability  $\PP(A_t=a_t|H_t)$, by $\pi_t(a_t\mid H_t)$.

We define the value function of a DTR or policy $d$ by
$V(d)=\E_{d}[\sum_{t=1}^TY_t]$, 
where the expectation $\E_d$ is with respect to  a distribution which is similar to $\PP$, except the treatment assignments follow   the potentially unobserved  DTR  $d$. Thus $\E_{d}$ is different from $\E$, the expectation with respect to  the observed data. Therefore, 
we need additional assumptions on the observed data distribution $\PP$ so that $V(d)$ becomes identifiable under $\PP$. \\

\noindent \textbf{Assumptions for identifiability:}
\begin{itemize}
    \item[I.] \textbf{Positivity:} There exists a constant $C_{\pi}\in(0,1)$ so that $\pi_t(a_t \mid h_t)>C_\pi$ for all 
$a_t\in[k_t]$, $h_t\in\H_t$, and $t\in[T]$.
    \item[II.]\textbf{Consistency:} For all $t\in[T]$, the observed outcomes $Y_t$ and covariates $O_t$ agree with the potential outcomes and covariates  under the treatments actually received. See \cite{schulte2014, robins1994estimation} for more details.
    \item[III.]\textbf{Sequential ignorability:} For each $t=1,\ldots,T$, the treatment assignment $A_t$ is conditionally independent of the future potential outcomes $Y_{t}$ and future potential clinical profile $O_{t+1}$ given $H_t$. Here we take $O_{T+1}$ to be the empty set.
\end{itemize}
Our version of sequential ignorability follows from  \cite{robins1997causal,MurphySA2001MMMf}. Assumptions I-III are standard in DTR literature  \citep[cf.][]{schulte2014,MurphySA2001MMMf,tsiatis2019dynamic}. Assumption III is generally satisfied under sequentially randomized trials \citep{chakraborty2013}. Recent research shows promise for relaxing Assumption III in observational data, provided proxy variables are available \citep{han2021identification,zhang2024identification}. However, 
we do not pursue this direction in the current paper, as we first need to understand Fisher consistency in the simpler unconfounded setting before tackling the added complexity of unmeasured confounders.

Under Assumptions I-III, it holds that \citep[cf. p. 224 of][]{tsiatis2019dynamic}
\begin{equation}
\label{identification: value function}
    V(d)=\E\left[\prod_{t=1}^T\frac{1[d_t(H_t)=A_t]}{\pi_t(A_t\mid H_t)}\sum_{j=1}^TY_j\right].
\end{equation}

We define the optimal policy or DTR $d^*$ to be a  maximizer of $V(d)$ over all possible DTRs, i.e., $ d^*\in \argmax_{d}V(d)$. We denote the optimal value function $V(d^*)$ by $V_*$. 
Under Assumptions I-III, $d^*$ can be represented using the optimal Q-functions. We define the $T$-stage optimal Q-function as
 \[Q_T^*(H_T,A_T)=\E[Y_T\mid H_{T},A_T],\]
 and for
 $t\in\{1,\ldots,T-1\}$, we define the  optimal Q-functions recursively as follows:
 \[Q_t^*(H_{t}, A_t)=\E\left[Y_{t}+\max_{i\in[k_{t+1}]}Q_{t+1}^*(H_{t+1},i)\mid H_{t},A_t\right].\]
It can be shown that \citep[cf.][]{chakraborty2013} under Assumptions I-III, any $d$ satisfying
\begin{align}
  \label{def: d star argmax form}
    d_t(H_t)\in\argmax_{i\in[k_t]}Q_t^*(H_t,i)\quad \text{ for all }t\in[T]
\end{align}
is a candidate for the optimal DTR $d^*$, and any version of the optimal DTR $d^*$ satisfies \eqref{def: d star argmax form} with probability one.

 It is common to transform the maximization of  $V(d)$ into an optimization  problem with real-valued functions, thereby avoiding optimization over discrete-valued maps \citep{tewari2007,ramaswamy2016convex}. To formalize, for any $t\in[T]$, let $\F_{t}$ denote the space of all Borel measurable functions from $\H_t$ to $\Rt$. We consider  $T$-fold product  of these function-classes, denoted by $\F=\F_1\times\ldots\times\F_T$. Any $f\in\F$ is of the form $f=(f_1,\ldots,f_T)$, where   $f_t:\H_t\mapsto \RR^{k_t}$ is Borel measurable. Thus any $f\in\F$ is a map from $\H_1\times \ldots\times\H_T$ to $\RR^\K$, where we denoted $\K=\sum_{t=1}^Tk_t$. For any $(h_1,\ldots,h_T)\in \H_1\times\ldots\times \H_T$, 
\[f(h_1,\ldots,h_T)=\left(f_1(h_1),\ldots,f_T(h_T)\right),\text{ which is in }\RR^{\sum_{t\in[T]}k_t}=\RR^{\K}.\]
 For all $i\in[k_t]$ and $t\in[T]$, the $i$-th element of  $f_t(h_t)$ will be denoted as $f_{ti}(h_t)$.

Similar to \cite{wang2023classification,wang2023unified,zhang2004}, we will refer to $f$ as the class score function, and  $f_t$ as the class score function for stage $t$. In the classification literature, these functions are also known as the prediction function, decision function vector, or classification function. We require a link function to map each class score function to a treatment. Following contemporary classification  literature \citep{zhang2004,wang2023classification,wang2023unified}, we call this link function the pred function, defined by the operator $\pred(\mx)=\max(\text{argmax}_{1\leq j\leq k_t}\mx_j)$  for each vector $\mx$.
  The outer max operator ensures that $\pred$ returns a number rather than a set.   We omit the dimension of $\mx$ from the notation of $\pred$ because, similar to $\max$, the domain of $\pred$  should be understandable from the context. 
Notice that each $f_t$  defines a treatment assignment $d_t:\H_t\mapsto[k_t]$ via $\pred(f_t)$. Let $\pred(f)$ denote the function $(\pred(f_1),\ldots,\pred(f_T))$. Then each $f\in\F$ defines a DTR (or policy) $d$ via $\pred(f)$. We say that a DTR $d$ is linear if each $f_t$ is linear; otherwise, it is non-linear.

 By an abuse of notation, we denote $V(\pred(f))$ by  $V(f)$, so that
\begin{align}\label{def: V(f) original}
    V(f)=\E\lbt \prod_{t=1}^T\frac{1[A_t=\pred(f_t(H_t))]}{\pi_t(A_t\mid H_t)}\sum_{j=1}^TY_j\rbt.
\end{align}
If $f^*=(f^*_1,\ldots,f^*_T)$ is a maximizer of \eqref{def: V(f) original}, then it can be shown that  $d^*=\pred(f^*)$ is an optimal DTR, and vice versa \citep{tsiatis}. Furthermore,  $\sup_{f\in\F}V(f)=\Vt$. We say that $d_t^*$ is linear if there exists $f_t^*$ such that $f^*_{t}:\H_t\mapsto\RR^{k_t}$ is linear for all $t\in[T]$, and non-linear otherwise.
Similar to \cite{zhao2015,jiang2019entropy}, we  use the inverse probability weighted (IPW)  estimator to estimate  $V(f)$,  given by \begin{align}\label{def: V(f) }
    \Vipw(f)=\Pm\lbt \prod_{i=1}^T\frac{1[A_t=\pred(f_t(H_t))]}{\pi_t(A_t\mid H_t)}\sum_{j=1}^TY_j\rbt. 
 \end{align}
 When $T=1$,  the maximization of \eqref{def: V(f) } corresponds to a multiclass classification problem with class labels $1,\ldots,k_1$, feature space $\H_1$, and weight $\E[Y_1\mid H_1]/\pi_1(A_1\mid H_1)$  \citep{tao2017adaptive,xue2022multicategory}. Assigning treatment $A_1$ to a patient with history $H_1$ corresponds to classifying this patient to one label from the set  $[k_1]$.  
 When $T>1$, the maximization  of \eqref{def: V(f) } can be described as a weighted sequential multiclass classification problem, where  a patient is assigned to the class $\pred(f_t(H_t))$ at the $t$-th stage \citep[cf.][]{tsiatis}. Hereafter,  DTR classification refers to the maximization of \eqref{def: V(f) }. Before moving on to the surrogate problem construction, we introduce two additional assumptions on $\PP$. Assumption IV is a boundedness assumption on the outcomes.
\begin{itemize}
    \item \textbf{Assumption IV.}  The $Y_t$'s are bounded above by a constant, potentially depending on $\PP$, for all  $t\in[T]$. 
\end{itemize}
Boundedness assumption on the rewards is common in the DTR and ITR literature \citep{jiang2019entropy,zhao2019efficient,Laha2024}. 
This is not an overly stringent assumption since, in most of our applications, $Y_t$'s  represent medical measurements, and are automatically bounded.  Weaker versions of Assumption IV, such as the boundedness of the optimal Q-functions, may be sufficient for some of our results. However, the boundedness of the $Y_t$'s is required  for showing the sufficiency of the proposed necessary conditions.
It is unclear whether interpretable sufficient conditions can still be obtained for the separable surrogates in Section~\ref{sec: necessity} without Assumption~IV.

Since maximizing $\E_d[\sum_{t=1}^TY_t]$ is equivalent to maximizing $\E_d[\sum_{t=1}^T(Y_t+c)]$ for any $c\in\RR$, without loss of generality, we take $Y_t>C_{min}$ for all $t=1,\ldots,T$ for some $C_{min}>0$, leading to the following assumption.
\begin{itemize}
    \item \textbf{Assumption V.} 
   There is a constant $C_{min}>0$, potentially depending on $\PP$, so that $Y_t>C_{min}$ for all $t\in[T]$.
\end{itemize}
Although we refer to the above as an assumption, it is not technically an assumption on $\PP$ since a location transformation of any outcome data will ensure Assumption V holds, provided their supports are bounded below.  Assumption V has appeared in related DTR literature \citep{Laha2024,zhao2015,zhao2012}.

\smallskip
\subsubsection{Construction of the surrogate problem} 
\label{sec: surrogate construction}

 For any $k\in\NN$, $i\in[k]$, and $\mx\in\RR^k$, let us define the operator  $\phi_{\text{dis}}$ as $\phi_{\text{dis}}(\mx;i)=1[i=\pred(\mx)]$, where ``dis" in the subscript stands for ``discontinuous". Although $\phi_{\text{dis}}$ depends on the dimension of $\mx$, we suppress this dependence to streamline the notation.
For $\mx_1\in\RR^{k_1},\ldots,\mx_T\in\RR^{k_T}$, and $a_1\in[k_1],\ldots,a_T\in[k_T]$, let us define the loss
\begin{equation}
    \label{def: 01 loss}
\psid(\mx_1,\ldots,\mx_T;a_1,\ldots,a_T)=\prod_{t=1}^T\phi_{\text{dis}}(\mx_t;a_t).
\end{equation}
Equation \ref{def: V(f) } implies that 
\[\Vipw(f)=\PP_n\lbt\psid(f_1(H_1),\ldots,f_T(H_T);A_1,\ldots,A_T)\frac{\sum_{j=1}^TY_j}{\prod_{t=1}^T\pi_t(A_t\mid H_t)}\rbt.\]
An ideal estimator of $f^*$ would be the maximizer of $\Vipw(f)$. However, direct optimization of $\widehat V(f)$ is computationally challenging due to the discontinuity of $\psid$, which arises from the $\pred$ operator and the indicator function \citep{tewari2007,bhattacharyya2018hardness}. In both machine learning and DTR literature, direct optimization of discontinuous losses is typically avoided, and they are  replaced by a more well-behaved function,  called a surrogate loss \citep{xu2014model,horowitz1992smoothed,feng2022nonregular,gao2011,calauzenes2012}. Moreover, empirical evidence from simpler machine learning problems, such as binary classification, shows no significant advantage in directly optimizing the 0-1 loss over using certain surrogates like the sigmoidal loss \citep{nguyen2013algorithms}. 
 In view of the above, we adopt the surrogate loss approach,  replacing  $\psid$ with a   surrogate loss $\psi:\RR^{\K}\times \prod_{t=1}^T[k_t]\mapsto\RR$.

The resulting surrogate  problem maximizes the surrogate value function
\begin{align}\label{def:  surrogate V(f) estimated}
    \widehat V^\psi(f)=\PP_n\lbt\psi(f_1(H_t),\ldots,f_T(H_T);A_1,\ldots,A_T)\frac{\sum_{j=1}^TY_j}{\prod_{t=1}^T\pi_t(A_t\mid H_t)}\rbt.
 \end{align}
However, the surrogate $\psi$ needs to be chosen carefully so that the DTR resulting from the maximization of $ \widehat V^\psi_{}(f)$ is a consistent estimator of $d^*$. The relevant concept in literature is Fisher consistency or calibration \citep{bartlett2006,tewari2007,liu2018augmented,zhao2015}.
We need to introduce some terminology before formally defining Fisher consistency for the DTR classification problem.

For $f\in\F$, we define the population-level surrogate value function by 
\begin{align}\label{def:  surrogate V(f)}
V^\psi(f)=\E\lbt\psi(f_1(H_t),\ldots,f_T(H_T);A_1,\ldots,A_T)\frac{\sum_{j=1}^TY_j}{\prod_{t=1}^T\pi_t(A_t\mid H_t)}\rbt.
 \end{align}
 For any surrogate $\psi$, let us denote $\Vst=\sup_{f\in\F}V^\psi(f)$. Let $\tilde f=(\tilde f_1,\ldots,\tilde f_T)$ be a maximizer of $V^\psi(f)$ over $\F$. Note that $\tilde f$ may not exist or be unique. In some cases, the maximum of $V^\psi(f)$ is not attained in $\F$, but there exists an extended-real-valued score function $\tilde f$ such that $V^\psi(\tilde f)=V^\psi_*$. We will denote $\pred(\tf)$ by $\tilde d$. 
 
\begin{definition}[DTR Fisher Consistency w.r.t. $\mP$]
\label{def: Fisher consistency multiclass} Let $\mP$ be a set of distributions satisfying Assumptions I-V. 
We say that a surrogate $\psi$ is Fisher consistent with respect to $\mP$ if, for any probability measure $\PP \in \mP$, every sequence $\{f_m\}_{m \geq 1} \subset \F$ satisfying $V^\psi(f_m) \to_m \Vst$ also satisfies $V(f_m) \to_m \Vt$, where the expectations in $V^\psi$ and $V$ are taken with respect to $\PP$.

\end{definition}
Denote by $\mP^0$ the class of all probability distribution functions  satisfying Assumptions I-V. Throughout this paper,  we say \textit{ $\psi$ is Fisher consistent if it is Fisher consistent  with respect to $\mP^0$. } Our definition of Fisher consistency allows  $\argmax_{f\in\F}V^\psi(f)$ to be empty, aligning with the definition of Fisher consistency in broader machine learning literature \citep{bartlett2006,tewari2007,ramaswamy2016convex}.
If  $\tilde f$ exists, then Fisher consistency implies that $\tilde d=\pred(\tilde f)$ is a candidate for the optimal DTR $d^*$. We say that a surrogate loss is single-stage Fisher consistent if it is Fisher consistent for the DTR classification problem when $T=1$.
Fisher consistency is an important requirement for  surrogate losses because otherwise the optimization of $\Vipw^{\psi}$ can lead to sub-optimal DTRs, even if the latter efficiently estimates $ V^\psi$.

In the multiclass classification literature, surrogate frameworks often involve the sum-zero constraint, requiring the class score functions to sum to zero \citep{tewari2007,zhang2004}. However, in this paper, we do not consider a constrained optimization framework because constraints can complicate the optimization, especially given that DTR classification would impose 
$T$ sum-zero constraints—one for each stage \citep{xue2022multicategory}. 

Although the
$\pred$ function considered above is the most traditional link function for policy learning and classification, alternatives exist.
Notable among these is the recently proposed angle-based link function, which, although originally developed for multiclass classification, has been applied to ITR \citep{zhang2020multicategory}, and to  DTR for  maximizing survival probability  \citep{xue2022multicategory}. In this framework, $\pred(\mx)$ does not rely on $\argmax(\mx)$; instead, it depends on the angle between $\mx$ and a set of prespecified vectors determined by the dimension of $\mx$. See Remark \ref{remark: angle-based framework} for further discussion, where we explain why we adopt the $\argmax$-based $\pred$ over the angle-based alternative for our problem.

Before closing this section, we note the distinction between the general and the binary-treatment case. In the latter case, treatments $A_t$ can be represented as ${\pm 1}$ rather than $\{1,2\}$, so that $A_1,\ldots, A_T\in\{\pm 1\}$. This embedding has an advantage. Letting $g_t:\H_t\mapsto\RR$ for $t\in[T]$, we can represent $d_t$  by $\sign(g_t)$. Consequently,  $1[A_t=\pred(f_t(H_t))]$ can be replaced with the 0-1 loss function $1[A_tg_t(H_t)>0]$ in \eqref{def: V(f) original} and \eqref{def: V(f) }. Therefore, $\phi_{\text{dis}}$ simplifies to the  univariate 0-1 loss $x\mapsto 1[x\geq 0]$,  allowing margin-based surrogates.  See \cite{wang2023unified} for discussion on the advantages of margin-based surrogate losses.

 \section{Concave surrogates}
 \label{sec: convex loss}
Convex surrogates are attractive because they yield convex optimization problems with unique minima, facilitating efficient optimization techniques \citep{calauzenes2012}. This section presents some negative results regarding the Fisher consistency of concave surrogates in DTR classification. Readers primarily interested in Fisher consistent surrogates may skip to Section~\ref{sec: necessity}.

If   concave surrogates fail to be Fisher consistent in the $T=2$ case, they can't succeed in the $T\geq 2$ case.  Hence,   we will use $T=2$ throughout this section  to streamline our presentation.  To our knowledge, of date, no concave surrogate loss has been proposed for the DTR classification in the general case, described in Section \ref{sec: surrogate construction}. Consequently, no reference concave surrogate  exists for our setting.  However, Fisher consistent surrogates for multiclass classification yield Fisher consistent surrogates for the $T=1$ case, since, as noted in Section \ref{sec: set-up}, these two problems have a direct correspondence  \citep[see also][]{zhang2020multicategory,tao2017adaptive}.  The concave surrogate losses considered in this section are direct extensions of these single-stage losses to the $T=2$ case.  

The corresponding losses in multiclass classification are convex, as that setting involves minimization. Throughout this section, references to adaptations from multiclass classification should be understood to mean the concave versions of those losses.

  We begin this section with an example of the concave version of the exponential loss, which is Fisher inconsistent when $T=2$. Then we establish the  inconsistency of concave PERM losses under smoothness conditions. Finally, we present examples of non-PERM concave surrogates, demonstrating that they may also fail to be Fisher consistent.

 \begin{exampleboxenv}[Exponential loss]
 \label{ex: Example 1}
Let us consider the concave surrogate loss
\begin{align}
\label{inex: loss: exp non-additive}
    \psi(\mx,\my; a_1,a_2)=-\sum_{i\in[k_1]}  \sum_{j\in[k_2]}\exp(-(\mx_{a_1}-\mx_i+\my_{a_2}-\my_j)),
\end{align}
where $\mx\in\RR^k$, $a_1\in[k_1]$, and $a_2\in[k_2]$. It is a concave two-stage extension of the following exponential surrogate loss:
\begin{align}
\label{inex: single-stage: pairwise: exp}
   \phi(\mx;a)=-\sum_{i\in[k_1],i\neq a}\exp(-(\mx_{a}-\mx_i))\quad\text{for all } \mx\in\RR^k, a\in[k_1],
\end{align}
which is a Fisher consistent  surrogate for multiclass classification with $k_1$ classes  \citep{zhang2004,tewari2007}.
Using the above, it can be shown that the  surrogate  in \eqref{inex: single-stage: pairwise: exp} is also single-stage Fisher consistent, i.e., Fisher consistent for DTR classification when $T=1$. Result \ref{result: ex: exp non-additive},  proved in Supplement \ref{sec: proof of result ex: exp non-additive}, provides the formula of $\tilde d=\pred(\tilde f)$ when $T=2$.

\begin{result}
    \label{result: ex: exp non-additive}
   Under Assumptions I-V, 
     the surrogate loss defined in \eqref{inex: loss: exp non-additive} satisfies
    \begin{align}
    \label{inresult: example: exp loss}
 \tilde d_1(H_1)\in  &\ \argmax_{a_1\in[k_1]}\E\lbt \slb\sum_{a_2\in[k_2]}\sqrt{Y_1+Q_2^*(H_2,a_2)}\srb^2\bl H_1, A_1=a_1\rbt,\nn\\
 \tilde d_2(H_2)\in &\   \argmax_{a_2\in[k_2]}Q_2^*(H_2,a_2).
\end{align}

\end{result}
Equation \ref{def: d star argmax form} implies that the second stage treatment assignment $\tilde d_2(H_2)$ is optimal but 
the first stage treatment assignment $\tilde d_1(H_1)$  may be suboptimal 
because it need not agree with maximizer of $Q_1^*(H_1,a_1)$, i.e., $\E[ Y_1+\max_{i\in [k_2]}Q_2(H_2,i)\mid H_1,A_1=a_1]$, over $a_1\in[k_1]$. 
 Since Fisher consistency implies $\tilde d=d^*$, the loss in \eqref{inex: loss: exp non-additive} can not be Fisher inconsistent. 
\end{exampleboxenv}

Example~\ref{ex: Example 1} is not an isolated case. In all our examples (see Section~\ref{sec: concave: additional examples} for illustrations when $T = 2$), $T$-stage extensions of  Fisher consistent surrogates from multiclass classification yield optimal treatment only in the final stage, while earlier treatment assignments become suboptimal.  This probably occurs because the last stage of DTR classification resembles a multiclass classification problem.

\subsection{PERM losses}
\label{sec: PERM}
Now, we will show that a large class of concave surrogates can not be Fisher consistent for the sequential DTR classification. To this end, we first define the class of permutation equivariant and relative margin-based surrogate losses, i.e., PERM losses, which were introduced in \cite{wang2023unified} for the multiclass classification problem. This requires introducing the notion of permutation symmetry. 
We say 
$\eta:\RR^{k_1+k_2-2}\mapsto\RR$ is permutation symmetric if for any  permutations $\ppi:[k_1-1]\mapsto[k_1-1]$ and $\ppi':[k_2-1]\mapsto[k_2-1]$,
\begin{equation}
    \label{intheorem: cc: permutation symmetry of phi}
\eta(\mbu,\mv)=\eta(\ppi(\mbu),\ppi'(\mv))\text{ for all }\mbu\in\RR^{k_1-1},\mv\in\RR^{k_2-1}.
\end{equation}
Extending the definition from \cite{wang2023unified} to the two-stage case, we say that a surrogate loss $\psi$ is PERM if there exists a permutation-symmetric function $\eta$ such that for all $a_1\in[k_1]$, $a_2\in[k_2]$, $\mx\in\RR^{k_1}$, and $\my\in\RR^{k_2}$,
\begin{align}
    \label{def: PERM}
    \psi(\mx,\my;a_1,a_2)=\eta(\mx_{a_1}-\mx_1,\ldots,\mx_{a_1}-\mx_{k_1}, \my_{a_2}-\my_1,\ldots,\my_{a_2}-\my_{k_2}) 
\end{align}

Terms such as $\mx_i-\mx_i$ and $\my_j-\my_j$ are omitted from the expression in \eqref{def: PERM}. PERM losses are basically surrogate losses that are both permutation equivariant and relative margin-based. 
They  are  permutation equivariant in that they do not give special weight to any specific treatment label. This is a mild restriction because, to the best of our knowledge, most surrogate losses in practical use are  permutation equivariant. PERM losses are relative-margin-based in the sense that $\psi(\mx, \my; a_1, a_2)$ depends on $\mx$ (or $\my$) only through pairwise differences of its elements, as evident from \eqref{def: PERM}. Relative margin-based losses are also common in the multiclass classification literature \citep{rosset2003margin, glasmachers2016unified, fathony2016adversarial}.

 We used the class of PERM losses here because  demonstrating Fisher inconsistency for all concave surrogates, even under smoothness conditions, is  challenging, because the class of all concave surrogates is too broad. As demonstrated by the insightful work of  \cite{wang2023unified}, 
the PERM losses provide enough structure that can be manipulated to establish interpretable theoretical results on Fisher consistency. Moreover, PERM losses include  well-known  surrogate  loss classes. Below, we present some examples of surrogate loss classes that are PERM. See \cite{wang2023unified} for more examples.

\begin{exampleboxenv}[Pairwise comparison loss \citep{weston1999support}]
    \label{sec: pairwise comparison}
    In the multiclass classification context, this loss is defined by
\begin{equation}
    \label{eq: CC: single-stage pairwise}
    \phi(\mx;a)=\sum_{i\in[k_1],i\neq a}\tph(\mx_{a}-\mx_i)\quad\text{ for }\mx\in\RR^{k_1},a\in[k_1],
\end{equation}
where $\tph$ is a real-valued function. 
From \cite{zhang2004} \citep[see also][]{tewari2007}, it follows that if $\tph$ is bounded above, non-decreasing, concave, and $\tph'(0)>0$ then the $\phi$ in \eqref{eq: CC: single-stage pairwise}
is Fisher consistent for the single-stage DTR classification.  Two-stage extension of $\phi$ can be defined as 
 \begin{equation}
\label{def: psi pairwise}
    \psi(\mx,\my; a_1,a_2)=\sum_{i\in[k_1]}  \sum_{j\in[k_2]}\tph(\mx_{a_1}-\mx_i, \my_{a_2}-\my_j)
\end{equation}
\begin{align}
\label{def: psi pairwise 2}
  \text{or }\quad  \psi(\mx,\my; a_1,a_2)=\sum_{i\in[k_1],i\neq a_1}\tph_1(\mx_{a_1}-\mx_i)+\sum_{j\in[k_2],j\neq a_2}\tph_2(\my_{a_2}-\my_j),
\end{align}
where $\tph_1$ and $\tph_2$ are real valued functions. 
The exponential loss in Example \ref{ex: Example 1} is of the form \eqref{def: psi pairwise}. 

\end{exampleboxenv}

\begin{exampleboxenv}[Gamma-Phi losses \citep{beijbom2014guess}]
 \label{sec: gamma-phi losses}
 In the multiclass classification context, the Gamma-Phi loss is defined as
\begin{equation}
    \label{def:cc: gamma phi: T=1}
    \phi(\mx;a_1)=\tpho\slb \sum_{i\in[k_1]:j\neq a_1}\tphi(\mx_{a_1}-\mx_i)\srb\quad \text{for all }\mx\in\RR^{k-1} \text{ and }a_1\in[k_1],
\end{equation}
where $\tphi:\RR\mapsto\RR$ and $\tpho:\RR\mapsto\RR$ are  real-valued functions. Many popular multiclass losses  are Gamma-Phi losses.  For example, when $\tpho(x)=\log(1+x)$ and $\tphi(t)=\exp(-x)$, \eqref{def:cc: gamma phi: T=1} reduces to the cross-entropy loss or the multinomial logistic loss. When $\tpho(x)=c\log(1+x)$ and  $\tphi(x)=\exp((1-x)/c)$ for some $c>0$, \eqref{def:cc: gamma phi: T=1} reduces to the coherence loss, which is used in multiclass boosting algorithms \citep{zhang2009coherence}.   Both these losses are convex and Fisher consistent for multiclass classification \citep{wang2023classification}. In the DTR problem, however, the concave version of $\phi$ is relevant. 
From \cite{wang2023classification}, it follows that if (a) $\tphi$ and $\tpho$ are strictly decreasing, real-valued, and continuously differentiable functions, (b) $\tpho$ is bounded above, and (c) $\inf_{x\in\RR}\tphi(x)=0$, then the   Gamma-Phi loss in \eqref{def:cc: gamma phi: T=1} is Fisher consistent for the single-stage DTR case. Two possibilities of two-stage extensions of \eqref{def:cc: gamma phi: T=1} are given below:
\begin{gather*}
    \phi(\mx,\my;a_1,a_2)=\tpho\slb \sum_{i\in[k_1]:j\neq a_1}\sum_{j\in[k_2]:j\neq a_2}\tphi(\mx_{a_1}-\mx_i,\my_{a_2}-\my_j)\srb,\\
    \phi(\mx,\my;a_1,a_2)=\tpho\slb \sum_{i\in[k_1]:j\neq a_1}\tphi(\mx_{a_1}-\mx_i)+\sum_{j\in[k_2]:j\neq a_2}\tphi(\my_{a_2}-\my_j)\srb,
\end{gather*}
where $a_1\in\RR^{k_1}$, $a_2\in\RR^{k_2}$, $\mx\in\RR^{k_1}$, and $\my\in\RR^{k_2}$.
Pairwise losses are a special case of gamma-phi losses when $\tpho$ is the identity function.
\end{exampleboxenv}

 Now we are ready to state the main theorem of this section. To do so, we first review some definitions from convex analysis. A convex function $f:\RR^k\mapsto\RR$ is called proper if $f(\mx) > -\infty$ for all $\mx \in \RR^k$ and if $\dom(f)$ is non-empty \citep[cf. pp. 24,][]{rockafellar}. Also, $f$ is called closed if its sublevel sets are either closed or empty \citep[cf. Definition 1.2.3, pp. 78,][]{hiriart}. Note that a PERM loss is concave if its template $\eta$ is concave. Therefore, we will assume  that the template $\eta$ is concave. Because we define domain, closedness, and properness in the context of convex functions, the statement of Theorem \ref{theorem: CC} below uses $-\eta$ and $-\psi$ instead of $\eta$ and $\psi$, respectively.
\begin{theorem}
\label{theorem: CC}
   Suppose $T=2$, $k_1,k_2\geq 2$,   and $\psi$ is an above-bounded concave PERM loss such that $\cap_{i=1}^{k_1}\cap_{j=1}^{k_2}\iint(\dom(-\psi(\cdot;i,j))\neq \emptyset$. Further suppose $-\eta$ is proper, closed, and strictly convex,  where $\eta$ is the template of $\psi$. Also, we assume $\eta$ is thrice continuously differentiable on $\iint(\dom(-\eta))$. Then $\psi$ can not be Fisher consistent.
\end{theorem}

The conditions in Theorem \ref{theorem: CC} merits some discussion. (i) \textbf{Closedness and properness:} These are mild regularity assumptions on $-\eta$ intended to exclude pathological cases. Improper convex functions are often omitted from formal definition of convex functions \cite[cf. p.~74][]{hiriart}. Closedness affects behavior only at the boundary and, like properness, is a standard assumption for convex functions in the statistical literature \citep{saumard2014log, laha2021adaptive, seregin2010nonparametric, dossglobal, dosssymmetric}. To our knowledge, all convex  surrogates in practical use are proper and closed. 
(ii) \textbf{Domain condition:}
The condition $\bigcap_{i=1}^{k_1} \bigcap_{j=1}^{k_2} \iint(\dom(-\psi(\cdot; i, j))) \neq \emptyset$ is a mild technical requirement. The proof of Theorem~\ref{theorem: CC} involves the construction of a bad set of distributions. This condition ensures the existence of $\tilde f_1$ and $\tilde f_2$ for such distributions, where $(\tilde f_1, \tilde f_2) \in \arg\max_{f \in \F} V^\psi(f)$. While the theorem may still hold without this assumption, the associated analysis becomes technically more involved.  All multiclass PERM  losses in practical use satisfy this domain assumption. (iii) \textbf{Boundedness and strict convexity: }
These conditions are mainly required for establishing the existence and uniqueness of $\tilde f_1$ and $\tilde f_2$ for the above-mentioned bad set of distributions. \cite{wang2023classification, wang2023unified} also used such assumptions on PERM losses. (iv) \textbf{Smoothness condition:} The proof of Theorem \ref{theorem: CC} depends crucially on the smoothness condition. However, we are not aware of any non-smooth concave loss, PERM or otherwise, that is Fisher consistent for $T \geq 2$.  Later in this section, we provide an example of a popular non-smooth concave surrogate that is Fisher consistent in the single-stage case but inconsistent in two-stage DTR.  Smoothness-related assumptions on surrogates are common in the Fisher consistency literature \citep{zhang2004,xue2022multicategory,wang2023classification,wang2023unified,zhang2020multicategory}.  Supplement~\ref{sec: cond of Theorem cc} discusses the conditions under which our PERM loss examples satisfy the assumptions of Theorem~\ref{theorem: CC}.

 \subsubsection{Proof outline and challenges} 
\label{sec: CC: proof outline}
 
 Theorem~\ref{theorem: CC} is proved in Supplement~\ref{sec: pf of theorem cc}. The proof, which occupies a substantial part of our proof sections, draws  on tools from convex analysis.
  First, we identify a subset of $\mathcal{P}^0$ where $\tilde{f}$ exists and is analytically tractable. Fisher inconsistency is then demonstrated within this subset. The next step  is a reduction to the binary treatment case.  At a high level, we show that under the conditions of Theorem~\ref{theorem: CC}, the existence of a Fisher consistent, PERM loss $\psi$ in the general setting implies the existence of a Fisher consistent, relative-margin-based surrogate $\vartheta$ in the binary treatment setting.   Note that, we lose the permutation equivariance property of  $\psi$ in this reduction.  The remainder of the proof establishes that if $\vartheta$ is concave, it cannot be Fisher consistent.   Loss of permutation equivariance implies $\vartheta$ has more degrees of variability than a PERM loss in the binary treatment case. Interestingly, in the binary treatment case,  PERM losses  coincide with \cite{Laha2024}'s margin-based losses. Therefore, while \cite{Laha2024} developed tools to prove the inconsistency of margin-based concave surrogates in the binary-treatment case, those techniques fail to apply here. The loss of permutation equivariance substantially alters the structure of the problem, making our analysis significantly more intricate than that in \cite{Laha2024}. See Section \ref{sec: related: laha 2021} for further comparisons between the current paper and \cite{Laha2024}.

\subsubsection{Additional examples of non-PERM concave surrogates}
\label{sec: concave: additional examples}
 The negative result in Theorem~\ref{theorem: CC} is unlikely to be an artifact of smoothness or the PERM structure. In our search across other classes of concave losses, we found no surrogate for which $\tilde d_1$ agrees with $d_1^*$. Below, we provide examples of inconsistency arising from non-PERM concave surrogate classes. Our examples  include a non-smooth loss, which is a two-stage extension of  multivariate hinge loss.

 \begin{exampleboxenv}[One vs all loss]
\label{ex 4: cc: one v all}
    Here we consider a two-stage concave version of the one versus all surrogate loss of \cite{zhang2004}. Following the multiclass classification literature, in the single-stage case, we can define the one vs all loss by
    \[\phi(\mx;a_1)=\tph(\mx_{a_1})+\sum_{i\in[k_1]:i\neq a_1}\tph(-\mx_i)\text{ for all }\mx\in\RR^{k_1}\text{ and }a_1\in[k_1],\]
    where $\tph$ a univariate function. 
  Using \cite{zhang2004}, it can be shown that  $\phi$ is single-stage Fisher consistent if $\tph$ is concave, bounded above, differentiable, and $\tph(-x)<\tph(x)$ for all $x>0$.
  A two stage extension of the above surrogate is given by 
\begin{align}
\label{inex: psi: one vs all: general}
    \psi(\mx,\my; a_1,a_2)=&\ \tph(\mx_{a_1},\my_{a_2})+\sum_{j\in[k_2]:j\neq  a_2}\tph(\mx_{a_1},-\my_{j})+\sum_{i\in[k_1]:i\neq  a_1}\tph(-\mx_{i},\my_{a_2})\nn\\
    &\ +\sum_{i\in[k_1]:i\neq  a_1}\sum_{j\in[k_2]:j\neq  a_2}\tph(-\mx_{i},-\my_{j}),
\end{align}
 where $\mathbf{x} \in \mathbb{R}^{k_1}$, $\mathbf{y} \in \mathbb{R}^{k_2}$, $a_1 \in [k_1]$, and $a_2 \in [k_2]$, and $\tph$ is a bivariate function. In particular, let us take   
$\tph(x,y)=-\exp(-x-y)$\footnote{Note that this $\tph$ is an above-bounded, differentiable concave function satisfying $\tph(x,y)>\tph(x,-y)$ when $y>0$ and $\tph(x,y)>\tph(-x,y)$ when $x>0$.}. In that case, the loss in \eqref{inex: psi: one vs all: general} takes the form 
\begin{align}
\label{inex: psi: one vs all}
    \psi(\mx,\my; a_1,a_2)=&\ -\slb e^{-\mx_{a_1}}+\sum_{i\in[k_1]:i\neq  a_1}e^{\mx_i}\srb\slb e^{-\my_{a_2}}+\sum_{j\in[k_2]:j\neq  a_2}e^{\my_j}\srb.
\end{align}
Result \ref{result: conv: one vs all}, proved in Supplement \ref{sec: proof of result one vs all},  implies that, similar to Examples \ref{ex: Example 1} and \ref{ex 3: cc: additive},   the first stage treatment assignment may be suboptimal  for the loss in \eqref{inex: psi: one vs all}.
\begin{result}
\label{result: conv: one vs all}
    Consider the surrogate $\psi$ in \eqref{inex: psi: one vs all}. Under the setup of Result~\ref{result: ex: exp non-additive},  $\tilde d_2(H_2)$ is as in \eqref{inresult: example: exp loss}, and any element of the following set is a  candidate for  $\tilde d_1(H_1)$:
     \[\argmax_{i\in[k_1]}\sum_{a_2\in[k_2]}\E\lbt \sqrt{\E[Y_1+Y_2\mid H_2,a_2]\sum_{j\in[k_2]:j\neq  a_2}\E[Y_1+Y_2\mid H_2,j]}\bl H_1, A_1=i\rbt.\]
\end{result}

\end{exampleboxenv}

\begin{exampleboxenv}[Additive losses]
\label{ex 3: cc: additive}
A simple way of constructing a two-stage concave surrogate loss is adding two concave single-stage surrogate losses. 
We consider losses of the form
    \begin{align}
        \label{inex: additive psi}
        \psi(\mx,\my;a_1,a_2)=\phi^{(1)}(\mx,a_1)+\phi^{(2)}(\my,a_2)
   \end{align}
   for all $\mx\in\RR^{k_1}$, $\my\in\RR^{k_2}$ ,$a_1\in[k_1]$, and $a_2\in[k_2]$.
   The loss $\psi$ in \eqref{inex: additive psi} will be concave if both $\phi^{(1)}$ and $\phi^{(2)}$ are concave.
Suppose $\phi^{(1)}$ and $\phi^{(2)}$ are single-stage Fisher consistent losses. Result \ref{result: single-stage FC} below, which is proved in Supplement \ref{sec: pf Result: additive convex loss}, shows that the first stage treatment assignment $\tilde d_1$ does not match $d_1^*$ although  the second stage treatment assignment remains optimal. This result is general in that concavity of the single-stage losses is not required.

 \begin{result}
 \label{result: single-stage FC}
  Suppose Assumptions I-V hold and     $\psi$ is as  in \eqref{inex: additive psi}. Further suppose,  $\phi^{(1)}$ and $\phi^{(2)}$ are each single-stage Fisher consistent surrogates. Then whenever they exist, $\tilde d_2$ is as in \eqref{inresult: example: exp loss} and
     \begin{align*}
 \tilde d_1(H_1)\in  \argmax_{a_1\in[k_1]}\E\lbt \sum_{a_2\in[k_2]}\E[Y_1+Y_2\mid H_2, A_2=a_2]\bl H_1, A_1=a_1\rbt.
\end{align*}
 \end{result}

\end{exampleboxenv} 

\begin{exampleboxenv}[Sum-zero constraints]
\label{ex: sum zero}
    In multiclass classification, a popular method is the constrained comparison method, where the sum-zero constraint $\sum_{i=1}^{k_1}f_{1i}=0$ is imposed on the score function   when the surrogate risk is minimized \citep{zhang2004}. The corresponding surrogate is given by  
    \begin{equation}
     \label{def: cc: sum-zero constraint loss single-stage}
     \phi(\mx;a_1)=\sum_{i\in[k_1]:i\neq a_1}\tph(-\mx_i)\quad\text{ for all }\mx\in\RR^{k_1} \text{ and }a_1\in[k_1],
    \end{equation}
where $\tph$ is a univariate function. 
    An example is the multi-category support vector machine (SVM) of \cite{lee2004multicategory}, where $\tph$ is the convex loss $\tph(x)=\max(1-x,0)$ for all $x\in\RR$, also known as the hinge loss.  \cite{zhang2004,liu2007} show that the resulting surrogate is  Fisher consistent for the multiclass classification problem.
A two-stage extension  of \eqref{def: cc: sum-zero constraint loss single-stage} is given by
    \begin{equation}
     \label{def: cc: sum-zero constraint loss two stage}
     \psi(\mx,\my;a_1,a_2)=\sum_{i\in[k_1]:i\neq a_1}\sum_{j\in[k_2]: j\neq a_2}\tph(-\mx_i,-\my_j)
    \end{equation}
   for all $\mx\in\RR^{k_1}$, $\my\in\RR^{k_2}$ ,$a_1\in[k_1]$, and $a_2\in[k_2]$, where the constrained maximization problem becomes
    \begin{maxi}|l|
  {f_1\in\F_1,f_2\in\F_2}{V^\psi(f_1,f_2)}{\label{opti: hinge loss: surrogate value fn}}{}
  \addConstraint{\sum_{i=1}^{k_t}f_{ti}(h_t)}{=0,}{\text{ for all }h_t\in\H_t\text{ and } t\in\{1,2\}.}
 \end{maxi}
   Let us consider $\tph(x,y)=\min(x-1,y-1,0)$, the concave version of the multivariate hinge loss \citep{zhao2015}. 
  Exact solutions of \eqref{opti: hinge loss: surrogate value fn}  with this choice of $\psi$ is analytically intractable and beyond the scope of this paper. Instead, we consider toy settings where the solutions $\tilde f_1$ and $\tilde f_2$ can be found numerically.

    To make the calculations simple, we consider the scenario where  $O_1,O_2=\emptyset$, and $Y_1=0$. In this case, $H_2\equiv A_1$. Since $H_1=O_1=\emptyset$, $\tilde f_1$ is a fixed vector in $\RR^3$ and $d_1^*$ and $\tilde d_1$ are fixed numbers. 
   It can be shown that $d^*$ is totally determined by the   expected (conditional) outcome matrix $(\E[Y_2\mid A_1=i,A_2=j])_{i\in[k_3],j\in[k_2]}$ in this setting. Letting  $k_1=3$ and $k_2=2$, we consider 3 choices of this matrix,  which leads to the three settings  in Table \ref{tab:hinge}. In all these settings, $d_1^*$ is unique. Table \ref{tab:hinge}  displays $\tilde f_1$, $\argmax(\tilde f_1)$, and $d_1^*$ for each setting. The calculations behind the derivation of $\tilde f$ for the settings in  Table \ref{tab:hinge} can be found in Supplement  \ref{secpf: hinge loss}. Table \ref{tab:hinge} shows that  for the first two settings, $\tilde d_1=\pred(\tilde f_1)$ disagrees with $d_1^*$. It can also be shown that $V(\tilde f)<V_*$ in these two settings. Therefore, the surrogate  under consideration is not Fisher consistent.  Nevertheless, in the third setting,  $\tilde d_1$ agrees with $d_1^*$. 

   In this case, it turns out that $Q_1^*(1)=2$,  $Q_1^*(2)=80$ ,but  $Q_1^*(3)=3$,  indicating treatment level 2, which is also the optimal treatment for the first stage, has much higher treatment effect compared to the 
 remaining treatments. The higher effect size makes optimal treatment assignment an easier task in the first stage of third setting. 

\begin{table}[H]
    \centering
    \renewcommand{\arraystretch}{1.2}
    \resizebox{\textwidth}{!}{%
    \begin{tabular}{|c|ccccccccc|cc|}
       \hline
       Setting  &  $\E[Y_2\mid 1,1]$ & $\E[Y_2\mid 1,2]$ & $\E[Y_2\mid 2,1]$ & $\E[Y_2\mid 2,2]$ & $\E[Y_2\mid 3,1]$ & $\E[Y_2\mid 3,2]$ & $\tilde f_{11}$ & $\tilde f_{12}$ & $\tilde f_{13}$ & $\argmax(\tilde f_1)$ & $d_1^*$\\
       \hline
       \large \rule{0pt}{2.5ex}1  & \large \rule{0pt}{2.5ex}5 & \large 1 & \large 3 & \large 4 & \large 4 & \large 4 & \large 0 & \large 0 & \large 0 & \large (1,2,3) & \large 1\\
       \large \rule{0pt}{2.5ex}2  & \large \rule{0pt}{2.5ex}5 & \large 7 & \large 3 & \large 4 & \large 4 & \large 4 & \large 0 & \large 0 & \large 0 & \large (1,2,3) & \large 1\\
       \large \rule{0pt}{2.5ex}3  & \large \rule{0pt}{2.5ex}1 & \large 2 & \large 80 & \large 1 & \large 3 & \large 3 & \large -1 & \large 2 & \large -1 & \large 2 & \large 2 \\
       \hline
    \end{tabular}
    }
    \caption{{\bf Toy settings for Example~\ref{ex: sum zero}.} This table displays three settings under the toy setup of Example~\ref{ex: sum zero}. It shows the $\E[Y_2\mid A_1=i, A_2=j]$ values and the corresponding optimal first-stage treatment $d_1^*$. The provided $\tilde f_1 \equiv (\tilde f_{11}, \tilde f_{12}, \tilde f_{13})$ is the solution to~\eqref{opti: hinge loss: surrogate value fn}, using the multivariate hinge loss surrogate of Example \ref{ex: sum zero}. Here, $\E[Y_2 \mid i, j]$ abbreviates $\E[Y_2 \mid A_1 = i, A_2 = j]$ for $i \in [3]$ and $j \in [2]$.}
    \label{tab:hinge}
\end{table}
\vspace{-10pt} 
\end{exampleboxenv}
To conclude, our results and examples in this section do not show much promises for concave surrogates in the DTR classification problem.  In fact, it may be possible that convexifying simultaneous direct search is fundamentally infeasible without sacrificing Fisher consistency. Thus,  the  computational burden of non-convex optimization may be necessary for theoretically valid simultaneous direct search, which can be viewed as the cost of reducing reliance on models. This limitation may narrow the scope of simultaneous DTR direct search, but Fisher inconsistency of concave/convex surrogate losses is neither surprising nor uncommon.  Convex surrogates, especially smooth ones, struggle with Fisher consistency across a range of machine learning problems, including multilabel classification with ranking loss \citep{gao2011}, ranking with pairwise disjoint loss \citep{calauzenes2012,duchi2010}, and maximum score estimation \citep{feng2022nonregular}. Even in DTR classification with risk constraints, no concave surrogates have been identified to date, and non-concave losses such as the ramp loss remain in use \citep{liu2024learning,liu2024controlling}.
According to \cite{calauzenes2012}, the failure of concave loss occurs in problems where the convex surrogate objective function ($V^\psi(f)$ in our case) fails to closely approximate the discontinuous objective function  ($V(f)$ in our case). 


We note that properties related to Fisher consistency rely   on  the choice of the link function $\pred$ \citep{ramaswamy2016convex}. Consequently, the aforementioned negative results are specific to our  $\argmax$-based $\pred$ function, which is arguably the most common $\pred$ function in DTR, ITR, and even classification literature.   Smooth concave surrogate losses may achieve Fisher consistency with alternative choices of $\pred$, although whether such $\pred$ functions can yield computationally feasible methods is a separate question.

\section{Necessary and Sufficient  Conditions for Fisher Consistency}
\label{sec: necessity}
The inconsistency results for  concave surrogates in Section \ref{sec: convex loss} prompt us to look beyond the concave surrogates. In this section, we establish necessary and sufficient conditions for Fisher consistency among  separable  surrogates, which we define as
\begin{equation}
    \label{def: product psi}
    \psi(\mx_1,\ldots,\mx_T;a_1,\ldots, a_T)=\prod_{t=1}^T\phi_t(\mx_t;a_t)\text{ for all }\mx_t\in\RR^{k_t},\ a_t\in[k_t],\  t\in[T], 
\end{equation}
where each $\phi_t:\RR^{k_t}\times[k_t]\mapsto \RR$ can be thought of as a single-stage surrogate loss. We also assume $\phi_t\geq 0$ for each $t\in[T]$. We then present examples of Fisher-consistent surrogates within this class and discuss their properties. Here we consider products of single-stage surrogates rather than sums, as Example~\ref{ex 3: cc: additive} shows that addition can lead to Fisher inconsistency.

Non-negativity of the $\phi_t$’s offers analytical advantages. In particular, 
it enables iterative closed-form expressions for $V^\psi(f)$ and $V^\psi_*$, facilitating a  detailed analysis of the $\psi$-regret $V^\psi_*-V^\psi(f)$. However, non-negativity implies non-concavity for non-constant $\phi_t$'s. That said, concavity of the $\phi_t$’s may offer limited benefit, since a separable $\psi$ can still be non-concave even if each $\phi_t$ is concave. Notably, the original discontinuous surrogate $\psi_{\text{dis}}$ is  separable with non-negative $\phi_{\text{dis}}$.
Non-negative separable surrogates have been used in DTR classification, particularly for survival data \citep{xue2022multicategory} and binary treatments \citep{Laha2024}. 
Even in  single-stage direct search, non-negative, non-concave surrogates  are  used  when  the presence of risk constraints complicate the classification \citep{liu2024learning,wang2018learning}. Similar surrogates  also appear  in the multiclass classification literature \citep{liu2006multicategory,wu2007robust}.

\subsubsection{Single-stage setting $(T=1)$}
The necessary and sufficient condition of Fisher consistency in the single-stage case is analogous to multiclass classification, and can be concisely expressed using functionals of $\phi_t$'s. Let $k\in\NN$. For any single-stage surrogate $\phi:\RR^k\times [k]\mapsto\RR$, $\mx\in\RR^k$, and any $\mp\in\RR^k$, let us define the functionals
\begin{equation}
    \label{def: Psi and psi star main text}
\Psi(\mx;\mp)=\sum_{i=1}^k\mp_i\phi(\mx;i)\quad \text{and}\quad\Psi^*(\mp)=\sup_{\mx\in\RR^k}\Psi(\mx;\mp).
\end{equation}
Equation \ref{def: support function} implies  that $\Psi^*$ is simply the support function of the set 
 \begin{align}
     \label{def: image set}
     \mathcal V_\phi=\lbs \mv\in\RR^k: \mv_i=\phi(\mx;i)\text{ for some }\mx\in\RR^k\rbs.
 \end{align} 
 By a slight abuse of terminology, we will refer to $\mathcal V_\phi$ as the image set of $\phi$.
When $\phi=\phi_t$, $\Psi$ and $\Psi^*$  will be denoted by $\Psi_t$ and $\Psi_t^*$, respectively. We will now introduce a condition.
\begin{assumptionp}{N1}
\label{assump: N1}
  For any $\mp\in\RR_{\geq 0}^{k}$ and $\mx\in\RR^{k}$,  $\phi:\RR^{k}\times[k]\mapsto\RR_{\geq 0}$ satisfies 
\begin{align}
\label{inlemma: necessity: single-stage FC}
  \Psi^*(\mp)- \sup_{\mx\in\RR^k:\mp_{\pred(\mx)}<\max(\mp) }\Psi(\mx;\mp)>0.
\end{align}
\end{assumptionp}
Since we defined the supremum of an empty set to be $-\infty$,  \eqref{inlemma: necessity: single-stage FC} holds trivially when the elements of $\mp$ are all equal. Condition~\ref{assump: N1} essentially states that if $\pred(\mx) \notin \argmax(\mp)$, then $\mx$ can not maximize $\Psi(\cdot; \mp)$.  Such conditions are common in multiclass classification literature \citep{ramaswamy2016convex,tewari2007}. \cite{tewari2007} also provides  geometric interpretations of conditions similar to  Condition \ref{assump: N1}, though under the link function $\pred(\mx) = \argmax_{i \in [k]} \phi(\mx; i)$. Lemma~\ref{lemma: single-stage} below shows that Condition~\ref{assump: N1} is both necessary and sufficient for Fisher consistency in the single-stage setting. The analogous result for multiclass classification has been proved by \cite{tewari2007}. 

\begin{lemma}
\label{lemma: single-stage}
    Condition \ref{assump: N1} is  necessary and sufficient for the Fisher consistency of $\phi:\RR^k\times[k]\mapsto\RR_{\geq 0}$ in the single-stage $(T=1)$ setting  with respect to $\mP_0$, i.e., the class of all distributions satisfying Assumptions I-V.
\end{lemma}

Under Condition \ref{assump: N1}, $\phi$ satisfies several interesting properties. For example, it is bounded (cf. Supplementary Lemma \ref{lemma: necessity: psi bounded}), and, given any bounded set $\mathfrak{ B}\subset \RR_{\geq 0}^k$, 
  there exists a non-negative, non-decreasing,  convex function $\varrho$ such that 
 \begin{align}
     \label{ineq: cond N1 equivalence}
     \Psi^*(\mp)-\Psi(\mx;\mp)\geq \varrho(\max(\mp)-\mp_{\pred(\mx)})
 \end{align}
    for  all $\mx\in\RR^{k}$ and all  $\mp\in\mathfrak{B}$   (cf. Supplementary Lemma \ref{lemma: necessity}). Analogous results also appear in the multiclass classification literature \citep{zhang2004}.

\subsubsection{Multi-stage setting $(T\geq 2)$}
\label{sec: necessetity T geq 2}
As we shall see, Condition~\ref{assump: N1} is no longer sufficient for Fisher consistency of separable surrogates when $T \geq 2$. In this case, we will require more restrictions on the $\phi_t$'s. To this end, we introduce another condition on the $\Psi^*$-transformation.
\begin{assumptionp}{N2}
\label{assump: N2}
 Let $k\in\NN$. Then  $\phi:\RR^{k}\times[k]\mapsto\RR$ satisfies \begin{align}
            \label{cond: N2}
            \Psi^*(\mp)=C_{\phi} \max(\mp) \text{ for all }\mp\in\RR_{\geq 0}^{k} \text{ for some }C_{\phi}>0.
        \end{align}  
\end{assumptionp}
It is straightforward to verify that $C_{\phi} = \Psi^*(\mo_k)$. Interestingly, when $\phi = \phi_{\text{dis}}$, we have $\Psi^*(\mp) = \max(\mp)$. Thus Condition~\ref{assump: N2} essentially requires the $\Psi^*$-functionals of $\phi$ and $\phi_{\text{dis}}$ to be proportional.  Lemma~\ref{lemma: necessity: cond N2 and Cond N2'} below shows that Condition~\ref{assump: N2} can also be characterized in terms of $\conv(\mV_\phi)$, the closed convex hull of the image set $\mV_\phi$ defined in \eqref{def: image set}.

\begin{lemma}
\label{lemma: necessity: cond N2 and Cond N2'}
   Suppose  $k\in\NN$ and   $\phi:\RR^{k}\times[k]\mapsto\RR_{\geq 0}$ satisfies  Condition \ref{assump: N1}. Then $\phi$ satisfies Condition \ref{assump: N2} if and only if 
   $C_\phi\S^{k-1}\subset \conv(\mV_{\phi})$, where $\S^{k-1}$ is the simplex in $\RR^k$, $C_\phi=\Psi^*(\mo_k)$, and $\mV_\phi$ is as in \eqref{def: image set}.

\end{lemma}
Note that the $C_\phi$ in Lemma \ref{lemma: necessity: cond N2 and Cond N2'} agrees with that in Condition \ref{assump: N2}.
The lemma further implies that when $C_\phi = 1$, Condition~\ref{assump: N2} reduces to $\conv(\mV_\phi) \supset \S^{k-1}$. The appearance of $\S^{k-1}$ may seem unexpected, but as shown in the proof of Lemma~\ref{lemma: necessity: cond N2 and Cond N2'}, it is simply the closed convex hull of the image set for $\phi_{\text{dis}}$. Thus, Condition~\ref{assump: N2} requires the image set of $\phi$ to contain, in convex hull, the image set of the original discontinuous surrogate $\phi_{\text{dis}}$.

    If a surrogate satisfies both Conditions \ref{assump: N1} and \ref{assump: N2},
then the $\varrho$ in 
\eqref{ineq: cond N1 equivalence} 
becomes linear. In particular,  Supplementary Lemma \ref{lemma: necessity: linear bound} shows that there exists $\CC_\phi\in(0,1]$, depending only on $\psi$, so that 
\begin{align}
    \label{def of mathcal C psi}
    \Psi^*(\mp)-\Psi(\mx;\mp)\geq C_\phi\CC_\phi(\max(\mp)-\mp_{\pred(\mx)})
\end{align}
for all $\mx\in\RR^k$ and $\mp\in\RR^k_{\geq 0}$, where $C_\phi$ is the constant appearing in Condition \ref{assump: N2}. The original discontinuous loss $\phi_{\text{dis}}$ also satisfies \eqref{def of mathcal C psi}.  
Thus Conditions \ref{assump: N1} and \ref{assump: N2} together enforce a stricter restriction on the surrogates than  Condition  \ref{assump: N1} alone. 
Theorem~\ref{theorem: sufficient conditions} below establishes that these two conditions are sufficient for Fisher consistency.
\begin{theorem}[Sufficient conditions]
\label{theorem: sufficient conditions}
    Suppose  $\psi$ is a separable surrogate, i.e., it is as in \eqref{def: product psi}, where the $\phi_t$'s are non-negative. Then the followings are sufficient conditions for Fisher consistency (with respect to $\mP_0$, i.e., the set of all distributions satisfying Assumptions I-V):
    \begin{itemize}
        \item[F1.] Each $\phi_t$ satisfies Condition \ref{assump: N1} for $t\in[T]$.
        \item[F2.] For $t\geq 2$, $\phi_t$ also satisfies Condition \ref{assump: N2}.
        \end{itemize}
    \end{theorem}
The proof of Theorem \ref{theorem: necessity} can be found in Supplement \ref{secpf: necessity}. The proof exploits properties of surrogates satisfying Conditions \ref{assump: N1} and \ref{assump: N2}, such as  inequalities \eqref{ineq: cond N1 equivalence} and \eqref{def of mathcal C psi}. However, the main technical  challenge lies in showing that $\phi_t(f_t;\pred(f_t(H_t))$ stays bounded away from zero when $ V^\psi_*-V^\psi(f)$ is small.
 Section \ref{sec: examples of psi satisfying N1 and N2} provides examples of $\psi$ that satisfy conditions F1 and F2. It turns out that these conditions are also necessary for Fisher consistency if we slightly relax Assumption V, which requires $Y_t > 0$, to allow $Y_t \geq 0$. The resulting class of distributions is slightly larger than $\mathcal{P}_0$.  
 
    \begin{theorem}[Necessary conditions]
\label{theorem: necessity}
    Suppose  $\psi$ is a separable surrogate with non-negative $\phi_t$'s. Further suppose $\psi$ is Fisher consistent with respect to the set of all distributions satisfying Assumptions I-IV and $Y_t\geq 0$ for all $t\in[T]$. Then $\psi$ must satisfy the conditions F1 and F2 of Theorem \ref{theorem: sufficient conditions}.
    \end{theorem}
 
We chose to work with the slightly broader class that allows $Y_t \geq 0$ in Theorem~\ref{theorem: necessity} for technical convenience in the proof. It may be possible to show that
F1 and F2 are necessary for Fisher consistency even under $\mathcal{P}_0$ (i.e., when $Y_t > 0$), but the corresponding proof would be considerably more technical without offering additional conceptual insight. 
Despite this minor discrepancy, we will refer to F1 and F2 as the necessary and sufficient conditions for Fisher consistency within the class of separable surrogates in the sequel.

Similar to Theorem~\ref{theorem: sufficient conditions}, the proof of Theorem~\ref{theorem: necessity} is non-trivial, and together they occupy a substantial portion of the Supplement.
The proof of Theorem~\ref{theorem: necessity} relies on identifying good sets of distribution under which the optimal policy and  quantities such as  $\tilde f$ and $V^\psi(f)$ are tractable. To prove Condition F1, we consider distributions that resemble a one-stage DTR, where treatment affects the outcome at only one stage. However, proving Condition F2 requires distributions that embed a two-stage DTR structure, and the corresponding good set is adjusted accordingly.  A key step in proving F2 is to show that for any $t\geq 2$, and $\mx\in\RR^{k_t}_{\geq 0}$,  $\Psi_t^*(\mx)=G_t(\max(\mx))$ for some univariate function $G_t$. The linearity of  $G_t$ is then derived from the positive homogeneity of  $\Psi_t^*$.

  
While Theorems~\ref{theorem: sufficient conditions} and~\ref{theorem: necessity} do not explicitly refer to the original loss, they can be interpreted as suggesting that, compared to the single-stage setting, Fisher consistency in the multi-stage setting requires the $\phi_t$'s to more closely resemble the original discontinuous loss $\phi_{\text{dis}}$ for $t \geq 2$. 
 As a downside, similar to $\psi_{\text{dis}}$, such surrogates do not preserve the ranking among suboptimal treatments. In particular, for two suboptimal treatments $i$ and $j$, it is possible that $\tilde f_{ti}(H_t) < \tilde f_{tj}(H_t)$ even when $Q_t^*(H_t, i) > Q_t^*(H_t, j)$, where $\tilde f$ satisfies $V^\psi(\tilde f) = V^\psi_*$.   In fact, our proof implies that, if $\tilde f$ exists, $\phi_t(\tilde f_t(H_t),j)=0$ for any suboptimal treatment $j\in[k_t]$.  

Theorems \ref{theorem: sufficient conditions}-\ref{theorem: necessity} imply that Fisher consistency does not require  
$\phi_1$ to satisfy Condition \ref{assump: N2}. However, if $\phi_1$ does satisfy Condition \ref{assump: N2}, 
we  can obtain a stronger result. 
\begin{proposition}
    \label{prop: multi-cat FC}
 Suppose Assumptions I-V are satisfied and $\psi$ is as in Theorem \ref{theorem: sufficient conditions} except $\phi_1$ also satisfies Condition \ref{assump: N2}. Further suppose for each $t\in[T]$, there exists $\J_t\geq 0$ so that  $\phi_t(\mx,\pred(\mx))\geq \J_t$  for all $\mx\in\RR^{k_t}$.
 Then for all $f\in\F$,
 \begin{align}
     \label{instatement: sufficiency: lower bound linear}
       V^\psi_*-V^\psi(f)\geq C_*\slb V_*-V(f)\srb\text{ where }C_*=\slb\prod_{t=1}^T\J_t \srb\min_{1\leq t\leq T}\CC_{\phi_t}.
 \end{align}
 Here  the $C_{\phi_t}$'s are  as in Condition \ref{assump: N2},  and the $\CC_{\phi_t}$'s are as in \eqref{def of mathcal C psi}.
\end{proposition}
The proof of Proposition~\ref{prop: multi-cat FC} is provided in Supplement~\ref{sec: proof of dorollary multi-cat FC}. The constant $C_*$  depends on $\psi$, but we omit this dependence from notaion for the sake of simplicity.
The lower bound in \eqref{instatement: sufficiency: lower bound linear} is nontrivial only when $\J_t > 0$, which requires the $\phi_t(\mx; \pred(\mx))$'s  to be  bounded away from zero. The original discontinuous loss $\phi_{\text{dis}}$ and all surrogates considered in the next section satisfy this lower boundedness condition. Supplementary Lemma~\ref{lemma: necessity psi(xk, pred xk) is lower bounded}  implies that this condition is automatically met under Condition~\ref{assump: N1} if the $\phi_t$'s are symmetric, i.e., if $\sum_{i=1}^{k_t} \phi_t(\mx; i)$ is constant for all $\mx \in \RR^{k_t}$. In general, however, Conditions~\ref{assump: N1} and~\ref{assump: N2} do not ensure that $\phi_t(\mx, \pred(\mx))$ is bounded away from zero uniformly over $\mx \in \RR^{k_t}$, and counterexamples can be easily constructed.

There exist surrogates for which \eqref{instatement: sufficiency: lower bound linear} holds with equality, e.g., when the $\phi_t$'s are constant multiples of $\phi_{\text{dis}}$. To see this, note that $\J_t = \CC_{\phi_t} = 1$ for $\phi_{\text{dis}}$. However, for specific surrogates, the constant $C_*$ may not be tight, and sharper constants may exist. As such, $C_*$ may not serve as a reliable criterion for selecting among surrogates. Moreover, since our separable surrogates are  non-convex, a surrogate that is easier to optimize are  ultimately more practical choice. Further discussion on $C_*$, including upper bounds on $C_*$ and the roles of $\J_t$ and $\CC_{\phi_t}$, is provided in Supplement~\ref{sec: supple: discussion on C star}. 

 \begin{remark}
 We anticipate that non-separable Fisher consistent surrogates may also exist. From \cite{liu2024controlling}, it follows that  the multivariate ramp loss—which is non-separable—is Fisher consistent for the DTR classification problem in the binary-treatment setting. However, this loss is margin-based, and its extension to the $k_t > 2$ case  remains an open question.
  \end{remark} 

\paragraph*{Location and scale transformation} Lemma~\ref{lemma: location-scale transformation} below shows that if a single-stage surrogate $\phi_t$ satisfies Conditions~\ref{assump: N1} and~\ref{assump: N2}, it continues to do so under a scale transformation. However, a location transformation by a positive constant violates Condition~\ref{assump: N2} although it preserves Condition~\ref{assump: N1}. The proof of Lemma~\ref{lemma: location-scale transformation} is provided in Supplement~\ref{secpf: lemma location-scale}.

\begin{lemma}
    \label{lemma: location-scale transformation}
   Suppose $\phi: \RR^k \times [k] \rightarrow \RR_{\geq 0}$ satisfies Conditions~\ref{assump: N1} and~\ref{assump: N2}. Then, for any $a, b > 0$, the scaled surrogate defined by $\phi_{a,b}(\mx; i) = b\phi(a\mx; i)$ for $\mx \in \RR^k$ and $i \in [k]$ also satisfies Conditions~\ref{assump: N1} and~\ref{assump: N2}. The shifted surrogate $\phi + c$ satisfies Condition~\ref{assump: N1} for all $c\in\RR$ but fails to satisfy Condition~\ref{assump: N2} unless $c=0$. 
\end{lemma}

Lemma \ref{lemma: location-scale transformation} proves the existence of non-negative $\phi_t$'s that satisfy Condition \ref{assump: N1} but not  Condition \ref{assump: N2}. Conversely, there are surrogates that satisfy Condition~\ref{assump: N2} but not Condition~\ref{assump: N1}. As a trivial example, consider $\phi_t(\mx, i) = 1[\pred(\mx) = k_t - i+1]$ for $\mx \in \RR^{k_t}$ and $i \in [k_t]$. Thus, neither Condition~\ref{assump: N1} nor Condition~\ref{assump: N2} implies the other.

\subsection{Restricted classes}
\label{sec: linear policies}

In the upcoming Section~\ref{sec: regret decay}, we will see that if we search for the best DTR within a sufficiently rich policy class, e.g., neural networks, decision trees, or basis expansion classes, we may expect a DTR with high value in large samples. However, in many applications, practitioners may be interested in simpler, interpretable policy classes that may not be universal approximation classes. For instance, linear policies are often favored for their interpretability. Suppose $\mL\subset \F$ is such a restricted class of policies and $f^*\in\mL$ is the maximizer of $V(f)$ over $\F$. Informally, we say $\psi$ is Fisher consistent over $\mL$ if, for any sequence $\{f_n\}_{n \geq 1} \subset \mL$, the convergence $V^\psi(f_n) \to \sup_{f \in \mL} V^\psi(f)$ implies $V(f_n) \to \sup_{f \in \mL} V(f)$.  An interesting question is: under what conditions a surrogate $\psi$ is Fisher consistent over restricted classes? The answer depends on the specific structure of $\mL$, but we anticipate that Conditions~\ref{assump: N1} and~\ref{assump: N2}—which are essentially restrictions on $\Psi^*$—will generally not suffice. These conditions sufficed in our earlier discussion because the maximizer of $V^\psi(f)$ over $\F$ can be explicitly characterized via $\Psi^*$ (see Lemma~\ref{lemma: sufficiency: p t star and Q t star} in the Supplement). 
 The maximizer of $V^\psi(f)$ over  $\mL$, in general, has little connection with $\Psi^*$. 

Developing a comprehensive theory for restricted classes is out of the scope of the present paper.  However, in this section, we show that in the special case where the optimal DTR  within $\mL$ coincides with $d^*$, Conditions~\ref{assump: N1} and~\ref{assump: N2} can still be used to investigate the DTR obtained by maximizing $V^\psi(f)$ over $\mL$.
Nevertheless, even in this case, we need more structure on $\psi$. 
 There are many ways to impose such structures. We choose to restrict the behavior of the $\phi_t(\mx;j)$'s when the elements of $\mx$ are large, since the upcoming examples in Section~\ref{sec: examples of psi satisfying N1 and N2} satisfy this condition. This leads to Condition~\ref{assump: N3} below.
\begin{assumptionp}{N3}
 \label{assump: N3}
  Suppose $\mx\in\RR^k$ is such that $\argmax(\mx)$ is singleton. Then  for any real sequence $\{\myb_n\}_{n\geq 1}\subset \RR_{\geq 0}$ such that $\myb_n\to\infty$, 
        \[\phi(\myb_n\mx;j)\to_n C_\phi 1[j=\argmax(\mx)]\text{ for all }j\in[k],\]
   where $C_\phi$ is as in Condition \ref{assump: N2}.
 \end{assumptionp}

  Lemma~\ref{lemma: theorem with linear policy} below shows that, when the optimal DTR within $\mL$ coincides with $d^*$, and $\mL$ is closed under multiplication by positive scalars, Conditions~\ref{assump: N1}, \ref{assump: N2}, and \ref{assump: N3} together guarantee a regret bound analogous to Proposition~\ref{prop: multi-cat FC}.

\begin{lemma}
\label{lemma: theorem with linear policy}
Suppose $\mL$ is closed under multiplication by positive scalars, i.e., if $f \in \mL$, then $af \in \mL$ for any $a > 0$.  
Assume there exists $f^* \in \mL$ such that, for all $t \in [T]$, the set $\argmax(f^*_t(H_t))$ is a singleton and agrees with $d_t^*(h_t)$ with $\PP$-probability one. Let $\psi$ be as in Proposition~\ref{prop: multi-cat FC}, and suppose it also satisfies Condition~\ref{assump: N3}. Let $C_*$ be as in Proposition~\ref{prop: multi-cat FC}.
Then the following assertion holds for any $f \in \mL$:
\[
\sup_{f \in \mL} V^\psi(f) - V^\psi(f) \geq C_* \left(V_* - V(f)\right).
\]
\end{lemma}

Lemma~\ref{lemma: theorem with linear policy} is proved in Supplement~\ref{sec: pf of linear}. 
The assumptions in Lemma~\ref{lemma: theorem with linear policy} imply that the optimal treatment assignments are unique with probability one. This condition often arises in the theoretical analysis of model-based DTR methods, such as parametric Q-learning and A-learning, which require uniqueness of the optimal DTR for efficient estimation of model parameters \citep{laber2014,schulte2014,robins1994estimation,chakraborty2013,tsiatis}. It has also appeared in the context of direct search with linear policies \citep{Laha2024}. An important example of $\mL$ is the class of linear policies \citep{murphy2005,wallace2015doubly,robins1994estimation,sonabendw2021semisupervised}. Here, the scores are linear functions of $H_t$ or its feature transformations. Lemma~\ref{lemma: theorem with linear policy} implies that if the optimal DTR $d^*$ is linear, then direct search over $\mL$ with suitable surrogates can, in principle, recover the optimal DTR.

\subsection{Examples of Fisher-consistent $\phi$}
\label{sec: examples of psi satisfying N1 and N2}

In this section, we present examples of surrogates $\phi_t$ that satisfy Conditions~\ref{assump: N1}, \ref{assump: N2}, and \ref{assump: N3}, and that $\phi(\mx; \pred(\mx)) > \J$ for some $\J > 0$. These surrogates are defined for a general dimension $k$, and the corresponding expression for $\phi_t$ can be obtained by replacing $k$ with $k_t$.  Note that the scale of $\phi$ can be adjusted by multiplying it by any positive constant;  however, such scaling has no impact on the  theoretical or practical performance of our method.


 \subsubsection{Kernel-based surrogates}
 \label{sec: kernel based surrogates}

For any $k\in\NN$, let 
$K:\RR^k\mapsto\RR_{\geq 0}$ be  a kernel which is positive everywhere and integrates to one.  For sake of simplicity, suppose $K$ is the joint distribution of $k$ i.i.d. random variables $Z_1$, $\ldots$, $Z_k$ with density \KK. This implies  $K(\mx)=\prod_{i=1}^k \KK(\mx_i)$ for each $\mx\in\RR^k$. Examples of $\KK$ include densities such as Gaussian, logistic, etc. 
For $\mx\in\RR^k$ and $j\in[k]$, consider the single-stage  surrogate loss 
\begin{align}
\label{def: psi: smoothed pred}
    \phi(\mx;j)=C\dint_{\RR^k} 1[\pred(\mx-\mbu)=j]K(\mbu)d\mbu=C\dint_{\RR^k} 1[\pred(\mbu)=j]K(\mx-\mbu)d\mbu.
\end{align}
For this surrogate,  $\Psi^*(\mo_k) = C$.
To see this, observe that $\sum_{i \in [k]} \phi(\mx; i) = C$ for all $\mx \in \RR^k$. Since the sum is constant, $\phi$ is a symmetric surrogate loss, as previously discussed in Section \ref{sec: necessetity T geq 2}. Symmetric losses are known to be more robust to label corruption in multiclass classification problems \citep{patrini2017making, charoenphakdee2019symmetric}. As previously mentioned, symmetric losses also satisfy $\inf_{\mx \in \RR^k} \phi(\mx, \pred(\mx)) > 0$ (see Lemma~\ref{lemma: necessity psi(xk, pred xk) is lower bounded} in the Supplement). In fact, as shown in Lemma~\ref{lemma: kernel based satisfies J condition}, this quantity is lower bounded by $1/k$ for the kernel-based surrogates.
 Lemma~\ref{lemma: kernel based satisfies J condition} is proved in Supplement \ref{secpf: kernel based}.
\begin{lemma}
\label{lemma: kernel based satisfies J condition}
Suppose $\phi$ is as in  \eqref{def: psi: smoothed pred}.  Then 
\begin{align}
   \label{kernel based: alt form}
   \phi(\mx,j)=CE_{\KK}\lbt \prod_{i\neq j}\slb  1-F_{\KK}(Z+\mx_i-\mx_j)\srb\rbt, 
\end{align}
where $Z$ is a random variable with density $\KK$, and $F_{\KK}$ and $E_{\KK}$ are the distribution function and the expectation operators corresponding to $\KK$, respectively. 
    Moreover, $\phi(\mx,\pred(\mx))>\J$ for all $\mx\in\RR^k$ where $\J=C/k$.
\end{lemma}
Using \eqref{kernel based: alt form}, we can derive closed-form expressions for $\phi$ in specific cases. In particular, Supplement~\ref{sec: short: calc for kernel logistic} provides explicit formulas for $\phi$ when $k = 3$ and $\KK$ is either the standard logistic or standard Gumbel density.
Lemma~\ref{lemma: FC of smoothed FC}, proved in Supplement~\ref{secpf: kernel based surrogate N1 and N2}, establishes that the kernel-based surrogate $\phi$ satisfies the conditions in Proposition~\ref{prop: multi-cat FC} and Lemma~\ref{lemma: theorem with linear policy}.
\begin{lemma}
 \label{lemma: FC of smoothed FC}
   The  kernel-based  $\phi$ in \eqref{def: psi: smoothed pred}  satisfies Conditions \ref{assump: N1}, \ref{assump: N2}, and  \ref{assump: N3}, provided $\KK(x)>0$ for each $x\in\RR$. Moreover, the constant $C_\phi$ in Condition~\ref{assump: N2} equals $C$, and the constant $\CC_\phi$ in \eqref{def of mathcal C psi} equals $2^{-(k-1)}$.

 \end{lemma}
Suppose the $C$ in \eqref{def: psi: smoothed pred} is one. Since $\J=1/k$ and $\CC_\phi=2^{-(k-1)}$ by Lemmas \ref{lemma: kernel based satisfies J condition} and \ref{lemma: FC of smoothed FC}, respectively, we have
\[C_*=\prod_{t=1}^T\frac{1}{k_t}\min_{t\in[T]}2^{-k_t}=\frac{2^{-\max_{t\in[T]}{k_t}}}{\prod_{t=1}^T k_t},\text{ where }C_* \text{ is as in Proposition \ref{prop: multi-cat FC}}.\]
 In particular, if $k_t=k$ for all $i\in[T]$, then $C_*=2^{-k}/k^T$.

\subsubsection{Product-based surrogates}
\label{sec: product-type surrogate loss}
Suppose  $\KK$ is a density supported on $\RR$ as in Section \ref{sec: kernel based surrogates}. For $\mx\in\RR^k$ and $C>0$, the product-based  surrogate is of the form
\begin{equation}
 \label{def: multi-cat: Fisher consistent psi}
     \phi(\mx;j)=\prod_{i\neq j, 1\leq i\leq k}\tau(\mx_j-\mx_i) \quad\text{ where } \quad\tau(x)=C(1-F_{\KK}(x)),
 \end{equation}
with $F_\KK$ being the  distribution function of  the density $\KK$.
  To understand the intuition behind this surrogate, assuming $C=1$, note that  $\tau(x)=\int_{\RR} 1[x-u\geq 0]\KK(u)du$ can be represented as a smoothed (using kernel $\KK$)  version of the univariate 0-1 losses $1[x\geq 0]$ and $1[x> 0]$. The  discontinuous loss $I[j=\pred(x)]$ can be written as a product of the univariate 0-1 losses as follows:
 \begin{equation}
  \label{def: multi-cat: motivation behind FC}
  I[j=\pred(\mx)]=\prod_{1\leq i<j}I[x_j-x_i\geq 0]\prod_{j+1\leq i\leq k}I[x_j-x_i> 0].
 \end{equation}
 If we smooth  each 0-1 loss in \eqref{def: multi-cat: motivation behind FC} using the kernel $\KK$, we  obtain the product-based surrogate in \eqref{def: multi-cat: Fisher consistent psi} with $C=1$. The main difference between the kernel-based surrogates and the product-based surrogates is as follows: the kernel-based surrogate smooths  the multivariate 0-1 loss $1[\pred(\mx)=j]$ at once, where the product-based surrogate splits this loss into a product of univariate 0-1 losses, and smooths each loss separately. Unlike the kernel-based surrogate,  the product-based surrogate may not satisfy that $\sum_{i\in[k]}\phi(\mx;i)$ is a constant.

Lemma \ref{lemma: FC: product type surrogate loss} below implies that the product-based  surrogate  satisfies Conditions \ref{assump: N1}, \ref{assump: N2}, and \ref{assump: N3}. This lemma is proved in Supplement \ref{secpf: product loss FC}.

\begin{lemma}
     \label{lemma: FC: product type surrogate loss}
  Suppose $k \in \NN$. If the density $\KK$ is symmetric about zero, then the product-based surrogate loss $\phi$, as defined in \eqref{def: multi-cat: Fisher consistent psi}, satisfies Assumptions~\ref{assump: N1}, \ref{assump: N2}, and~\ref{assump: N3} with $C_\phi = C$. Moreover, it satisfies \eqref{def of mathcal C psi} with $\CC_\phi = 1/2$, and $\phi(\mx, \pred(\mx)) \geq \J$ with $\J=C2^{-(k-1)}$ for all $\mx \in \RR^k$.

 \end{lemma}
Lemma~\ref{lemma: FC: product type surrogate loss} implies Proposition~\ref{prop: multi-cat FC} and 
Lemma~\ref{lemma: theorem with linear policy} apply for the product-based surrogates with  $C_* = C_\phi^T2^{-\sum_{t=1}^T k_t + T - 1}$.
In particular, when $C_\phi=1$ and $k_t = k$ for all $t \in [T]$, we have $C_* = 2^{-kT + T - 1}$.
Examples of  $\KK$ that are symmetric about zero include  the logistic, Cauchy, and Student’s $t$-density (for suitable degrees of freedom).
In the binary treatment case, the product-based surrogate  with  symmetric $\KK$ corresponds to the margin-based Fisher consistent surrogates studied by \citet{Laha2024}. We will revisit this connection in Section~\ref{sec: relative margin based}.

To the best of our knowledge, our kernel-based and  product-based surrogates have not been used previously in DTR or ITR research. However, smoothing the 0-1 loss is a widely adopted approach for constructing surrogate methods, particularly in machine learning and statistics problems where suitable convex surrogates are unavailable.
For example, smoothed 0-1 loss  has been used for constructing surrogates in multilabel classification with ranking loss \citep{gao2011}, maximum score estimation \citep{feng2022nonregular,xu2014model}, covariate-adjusted Youden index estimation, and one-bit compressed sensing \citep{feng2022nonregular}.  In the context of dynamic treatment regimes, the smoothed 0-1 loss appears in \cite{Laha2024} and \cite{xue2022multicategory}. In binary classification, where Fisher-consistent convex surrogates are available, \cite{nguyen2013algorithms} demonstrate that the smoothed 0–1 loss exhibits greater robustness to outliers and data contamination compared to popular convex surrogates.

\subsection{Relative-margin-based representation} 
\label{sec: relative margin based}
The product-based and kernel-based surrogates share a useful property. They are relative-margin-based, which means that  $\phi$ depends on $\mx$ only through the pairwise  differences $\mx_i-\mx_j$'s. This property can lead to a dimension reduction during the optimization. 
\begin{definition}[Relative-margin-based losses]
\label{def: relative margin}
 Let $k\in\NN$. The surrogate    $\phi:\RR^k\times[k]\mapsto\RR$ is relative-margin-based if there exists a function $\Gamma:\RR^{k-1}\times[k]\mapsto\RR$, so that for each $i\in[k]$ and $\mx\in\RR^k$,
 \[\phi(\mx;i)=\Gamma(\Delta \mx_1;i),\quad\text{ where }\quad \Delta \mx_1=(\mx_1-\mx_2,\ldots,\mx_1-\mx_k).\]
 The function $\Gamma$ is referred to as the template for $\phi$.
\end{definition}
Relative margin-based surrogate losses are quite common in multiclass classification theory, where they have appeared in the context of multiclass support vector machines \citep{glasmachers2016unified}, Gamma-Phi losses \citep{wang2023classification}, and PERM losses  \citep{wang2023unified}.
Lemma \ref{lemma: relative margin} below states that the product-based and kernel-based surrogates are relative margin-based, and provides the corresponding templates. We will say that the separable $\psi$ defined in \eqref{def: product psi} is relative-margin-based if each $\phi_t$ is relative-margin based.  We will denote the corresponding templates by $\Gamma_1,\ldots,\Gamma_T$, respectively.

 \begin{lemma}
     \label{lemma: relative margin}
  If $\phi:\RR^k\times[k]\mapsto\RR_{\geq 0}$ is either the kernel-based surrogate in \eqref{def: psi: smoothed pred}  or the product-based surrogate in \eqref{def: multi-cat: Fisher consistent psi}, then $\phi$ is relative-margin-based. 
For the product-based surrogate, $\Gamma$ takes the form
\begin{align}
\label{def: Gamma: product bases}
 \Gamma (\my;i)= \begin{cases}
       \prod_{j\in[k]} \tau(\my_j) & \text{if }i=1\\
        \tau(-\my_i)\prod_{j\in[k]\setminus\{1,i\}}\tau(\my_j-\my_i) & \text{if }i\neq 1,
    \end{cases}
\end{align}
where $\tau$ is as in \eqref{def: multi-cat: Fisher consistent psi}.
For the kernel-based surrogate loss, $\Gamma$ takes the form
\begin{align}
\label{def: Gamma: kernel based}
     \Gamma (\my;i)= \begin{cases}
        C_\phi E_{\KK}\lbt \prod_{j\in[k]\setminus\{i\}} \slb 1- F_{\KK}(Z-\my_j)\srb\rbt & \text{if }i=1\\
        C_\phi E_{\KK}\lbt \slb 1- F_{\KK}(Z+\my_i)\srb\prod_{j\in[k]\setminus\{1,i\}}^k \slb 1- F_{\KK}(Z+\my_i-\my_j)\srb\rbt & \text{if }i\neq 1,
    \end{cases}
\end{align}
where $Z$ is a random variable with density $\KK$, as in \eqref{kernel based: alt form}. 
 \end{lemma}
 The proof of Lemma \ref{lemma: relative margin} follows directly from the definitions of the surrogates in \eqref{kernel based: alt form} and \eqref{def: multi-cat: Fisher consistent psi}, and is therefore omitted.  In the binary treatment case, i.e., when $k_t = 2$ for all $t \in [T]$, $\Delta \mx_1$ reduces to a real number,  leading to the margin-based formulation of the surrogate problem. In this case, the $\Gamma$ corresponding to the product-based surrogate leads to \cite{Laha2024}'s surrogate losses, provided the smoothing kernel $\KK$ is symmetric. 

  \begin{figure}
    \centering
    \includegraphics[width=1\linewidth,height=2.3in]{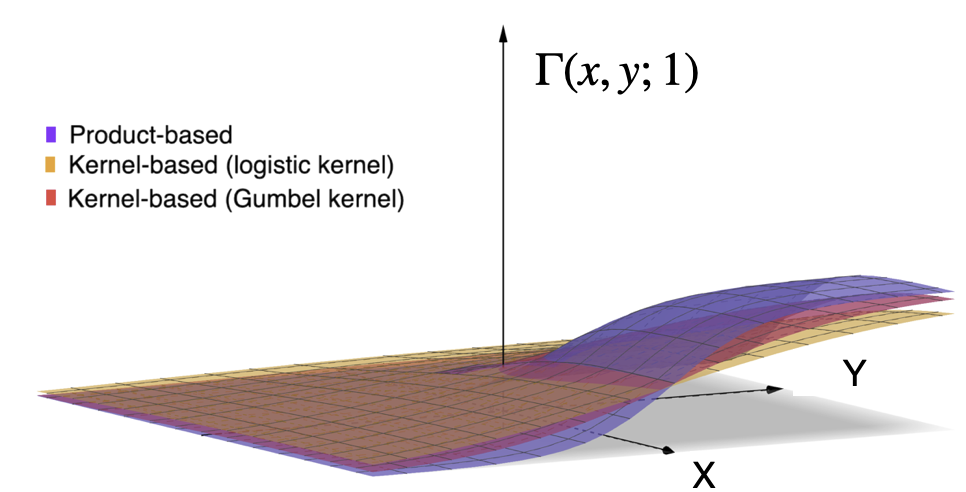}
    \caption{{\bf Plot of $\Gamma(x,y;1)$ when  $k=3$.}  For the product-based surrogate, $\Gamma$ is as in \eqref{def: Gamma: product bases}, and  $\tau(x)=(1+\tanh(x))/2$, which is the distribution function of the centered logistic distribution with scale $2$. For kernel-based surrogates, the template $\Gamma$ is provided in \eqref{def: Gamma: kernel based}. Its closed formulas for the logistic and Gumbel densities are provided in Supplement~\ref{sec: short: calc for kernel logistic}. }
    \label{Plot: Gamma function}
\end{figure}
  See Figure \ref{Plot: Gamma function} for the plot of $\Gamma(x, y; 1)$ for some product-based and kernel-based surrogates when $k = 3$. The function has no local extrema but attains its maximum at $(\infty, \infty)$, resulting in a bump in the first quadrant. The plots of $\Gamma(x, y; 2)$ and $\Gamma(x, y; 3)$ are similar, except their bumps appear in the second and fourth quadrants, respectively, since they are maximized at $(-\infty, \infty)$ and $(\infty, -\infty)$.


\begin{remark}
\label{remark: angle-based framework}
In Section \ref{sec: set-up}, we mentioned the angle-based framework, an alternative surrogate framework based  on an angle-based $\pred$ function.  Similar to the relative-margin-based surrogates, the angle-based framework  requires only $k_t-1$ scores. To the best of our knowledge, under this framework, concave Fisher consistent surrogates are currently  available only for the $T=1$ case \citep{zhang2020multicategory}. For $T>1$, Fisher consistency guarantees exist for some non-concave surrogates, albeit under a setting quite different from ours \citep[see Section \ref{sec: lit: existing research on kt geq 2} for further discussion on][]{xue2022multicategory}.   Similar to our product-based losses, these surrogates rely on smoothed 0-1 loss functions. In this paper, we focus on the argmax-based link function, as it is better suited to the mathematical manipulations needed for deriving the necessary and sufficient conditions. We anticipate that some parallels to our results would hold under the angle-based framework as well, although the corresponding calculations can be more analytically challenging.
    \end{remark}

   \section{SDSS method}
   \label{sec: method}

 In this section, we introduce the SDSS method. Section~\ref{sec: implementation} then addresses the optimization challenges posed by the non-convexity of our Fisher consistent surrogates and presents an algorithm for the optimization component of SDSS.
 In what follows, we assume that $\psi$ is separable as in  \eqref{def: product psi}, with each $\phi_t$ being non-negative and relative-margin-based. As we will see, relative-margin-based forms offer a computational advantage. 

If $\psi$ is relative-margin-based, i.e., if each $\phi_t$ is relative-margin-based for $t \in [T]$, then by Definition~\ref{def: relative margin}, $\phi_t(f_t(H_t); A_t)$ depends only on the pairwise differences among the components of $f_t(H_t)$, not on their absolute values.
As a result, $\widehat{V}^\psi(f)$ can not have a unique maximizer, even when maximizers exist. 
  For such surrogates, imposing linear constraints on $f_t$, such as $\sum_{i=1}^{k_t} f_{ti} = 0$ or $f_{t1} = 0$, does not affect the supremum of $\widehat{V}^\psi(f)$.
We prefer  the constraint  $f_{t1}=0$ because then we need to optimize only over the remaining $k_t-1$ many score functions, thereby lowering the problem dimension. Instead of setting $f_{t1} = 0$, one could equivalently choose $f_{t2} = 0$ or, more generally, $f_{ti} = 0$ for any $i \in [k_t]$. This choice is unlikely to impact SDSS's computational complexity or theoretical performance.   For simplicity, we take $i = 1$. Specifically, we replace the maximization of $\widehat{V}^\psi(f)$ with the constrained problem: \begin{maxi}|l| {f \in \F}{\widehat{V}^\psi(f)}{\label{maxi: V psi over constrained F}}{} \addConstraint{f_{t1}(H_t)}{= 0}{\quad \text{for all } t \in [T],} \end{maxi} whose optimal value coincides with $\sup_{f \in \F} \widehat{V}^\psi(f)$ when $\psi$ is relative-margin-based.
Similarly, maximizing the constrained version of $V^\psi(f)$ yields $V^\psi_*$ for relative-margin-based $\psi$'s.

 We now reformulate \eqref{maxi: V psi over constrained F} as an unconstrained optimization problem.
To this end, we introduce some notation. For each $t\in[T]$, let $\W_t$ denote the class of all Borel measurable functions from $\H_t$ to $\RR^{k_t - 1}$, and define $\W = \W_1\times\ldots\times \W_T$. Let $g = (g_1, \ldots, g_T)$, where each $g_t \in \W_t$. We refer to $g$ as the relative class score function and each $g_t$ as the relative score function at stage $t$.  For all $t\in[T]$ and $i\in[k_t-1]$, the $i$-th element of the vector-valued function $g_t$ will be denoted by  $g_{ti}$.
 For any $g_t \in \W_t$, define the transformation $\trans(g_t) = (0, -g_t)$. Also, we let $\trans(g)=(\trans(g_1),\ldots,\trans(g_T))$. Thus, if $g\in\W$, then $\trans(g)\in\F$. Then it is straightforward to verify that \eqref{maxi: V psi over constrained F} is equivalent to maximizing $\widehat{V}^\psi(\trans(g))$ over $g \in \W$.  In particular, if $\widehat{f}$ solves \eqref{maxi: V psi over constrained F} and $\widehat{g} \in \W$ maximizes $\widehat{V}^\psi(\trans(g))$, then they are related by
 \begin{equation}
 \label{eq: relative class cores} \widehat{f}_t = \trans(\widehat{g}_t) = (0, -\widehat{g}_t) \quad \text{for all } t \in [T]. \end{equation}
For any $g\in\W$, let us denote $\widehat{V}^{\psi,\text{rel}}(g) := \widehat{V}^\psi(\trans(g))$. Here the word ``rel" stands for relative-margin-based.  Straightforward algebra yields that if  $\psi$ is relative margin-based, then 
\begin{equation}
  \label{relation: relative margin}  
\widehat V^{\psi,\text{rel}}(g) =\PP_n\lbt \overbrace{\sum_{t=1}^TY_{t} \prod_{t=1}^T\frac{\Gamma_t(g_t;A_t)}{\pi_t(A_t\mid H_t)}}^{\L(\D;g)}\rbt.
 \end{equation}
Therefore, solving \eqref{maxi: V psi over constrained F} reduces to maximizing $\widehat{V}^{\psi,\text{rel}}(g)$ over $g \in \W$.
Note that the latter maximization is over $\sum_{t=1}^T (k_t - 1)$ score functions, whereas maximizing $\widehat{V}^\psi(f)$ over $f \in \F$ would involve $\sum_{t=1}^T k_t$ score functions.

In practice, maximization over $\W$ is generally infeasible. Therefore, SDSS restricts optimization to a subset $\U_n \subset \W$.   Let $\U_{tn}$ be a subset of the space of all Borel measurable functions from $\H_t$ to $\RR$.  Then, we define $\U_n$ to be  
$\U_n=\U_{1n}^{k_1-1}\times\ldots\times\U_{Tn}^{k_t-1}$. For the sake of simplicity, we will refer to $\U_n$, the search space of relative class scores, as the policy class.
Algorithm \ref{alg: SDSS} provides a pseudo-code for  SDSS.  The optimization part of SDSS will be discussed in detail in Section \ref{sec: implementation}.
\begin{algorithm}[H]
\caption{SDSS}
\label{alg: SDSS}
\begin{algorithmic}[1]
\Require (i)  $\U_n$: function class for relative class scores, (ii) $\psi$: a relative-margin-based, non-negative, smooth, separable surrogate (typically chosen to be Fisher consistent),  and  (iii) $\{\D_i\}_{i \in [n]}$: $n$ i.i.d. trajectories.
\State Let $\widehat{g}$ denote an approximate minimizer of $-\widehat{V}^{\psi,\text{rel}}(g)$ over $g \in \U_n$, where $\widehat{V}^{\psi,\text{rel}}(g)$ is as in \eqref{relation: relative margin}. We use  Algorithm \ref{alg:multi_stage_opt} for the minimization, which will be explained later.

\State Set $\widehat{f} = \trans(\widehat{g})$.
\State Return estimated policy: $\widehat{d} = \pred(\widehat{f})$.
\end{algorithmic}
\end{algorithm}

\subsection{Optimization part of SDSS}
\label{sec: implementation}

If $\psi$ is smooth, then $\widehat{V}^{\psi,\text{rel}}(g)$ is a smooth functional of $g$ and can be optimized using standard gradient-based methods. Gradient-based methods scale well to large sample sizes and are fast \citep{bottou2018optimization}.  Therefore, we  implement SDSS with smooth $\psi$ (cf. examples in Figure~\ref{Plot: Gamma function}).  
In this section, we  discuss the optimization part of SDSS for a  pre-fixed $\psi$. A discussion on the choice of surrogate and when to use SDSS is provided in Section \ref{sec: practical guidelines}. Throughout this section, any reference to gradient descent specifically refers to the minimization problem in Step 1 of the SDSS Algorithm (Algorithm \ref{alg: SDSS}).

 We assume $\U_n$ is parameterized by a real vector $\theta $ so that $\U_n \subset \{g_\theta : \theta \in \RR^{k_{\texttt{dim}}}\}$, where $k_{\texttt{dim}}$ is the dimension of $\theta$. 
 While $\theta$ and $k_{\texttt{dim}}$  may depend on $n$, we omit this dependence for notational simplicity. In the case of linear relative class scores, we may take $g_{ti}(H_t) = \theta_{ti}^\top H_t$ for $t \in [T]$ and $i \in [k_t - 1]$, where each $\theta_{ti}$ is a vector of the same dimension as $H_t$, and $\theta = (\theta_{ti})_{t \in [T],i \in [k_t - 1]}$.
 The function-class $\U_n$ can also be non-linear, e.g., neural networks, decision trees, lists, basis expansion classes, etc.  In the neural network case, $\theta$ contains the associated weight and bias parameters. In the basis expansion case, we may take $g_{ti}(H_t) = \theta_{ti}^\top J(H_t)$, where $J(H_t)$ is a vector of basis functions.

While gradient-based methods can speed up optimization, they still face challenges due to non-convexity.
We illustrate this with a toy example where $T = 1$, $k_1 = 3$, $n = 7$, and $\H_1 = \RR$, using the dataset in Table~\ref{table:toy_data}. We assume $\pi(A_1 = i \mid H_1) = 1/3$ for all $i \in [3]$.
In this illustration, we use the product-based surrogate in \eqref{def: multi-cat: Fisher consistent psi} with $\tau(x)=1+\tanh(x)$, though the same challenges arise for all Fisher consistent surrogates considered in this paper. For each $z \in \RR$, let $g_{11}(z) = xz$ and $g_{12}(z) = yz$ for $x, y \in \RR$, so that $\theta = (x, y)$ and
\[\U_{n}=\{ g_\theta: g_\theta(z)=(xz, yz)\text{ for all }z\in\RR\equiv\H_1, \text{ where } x,y\in\RR\}.\]
This parameterization reduces $-\widehat{V}^{\psi,\text{rel}}(g_\theta)$ to a function of only two variables, $x$ and $y$, with the form
  \begin{align}
     \label{opti: value fn: toy data}
    \widehat V^{\psi,\text{rel}}(x,y)\equiv \widehat V^{\psi,\text{rel}}(g_\theta)\equiv   =\frac{3}{n}\sum_{i=1}^nY_{1i}\Gamma(xH_{1i},yH_{1i};A_{1i}),
 \end{align}
where $\theta=(x,y)$, allowing visual exploration of the surface of $\widehat{V}^{\psi,\text{rel}}(g_\theta)$. 

\begin{table}[H]
    \centering
    \begin{tabular}{|c|ccccccc|}
        \hline
       i & 1 & 2 & 3 & 4 & 5 & 6 & 7 \\ \hline
        $H_{1i}$ & 2 & 1 & -1 & 0.5 & -0.5 & -1 & 0.5 \\
        $A_{1i}$ & 1 & 2 & 3 & 1 & 2 & 2 & 3 \\ 
        $Y_{1i}$ & 0.33 & 0.67 & 0.67 & 0.33 & 0.23 & 1 & 0.13 \\ \hline
    \end{tabular}
    \caption{Toy Data of $(H_1,A_1,Y_1)$ triplets  when $T=1$, $k_1=3$, and $H_1\in\RR$. Here $n=7$.}
    \label{table:toy_data}
\end{table}

The surface  plot of $ \widehat V^{\psi,\text{rel}}$ for this toy dataset is given in Figure \ref{fig:value function for toy data}. In this plot,  $ \widehat V^{\psi,\text{rel}}$ does not have any local optimum but it has a horizontal asymptote.  
Moreover, $\widehat{V}^{\psi,\text{rel}}$ exhibits several plateau regions for the toy data. While the gradient in these regions is not exactly zero, it can be very small, as illustrated in Figure~\ref{fig:contour: grad}.  Gradient descent iterates can become trapped if they enter these regions, leading to what is known as the vanishing gradient problem \citep{hochreiter2001gradient,ven2021regularization}. However, for our toy data, getting trapped in a plateau region is not always bad.  Figure \ref{fig:value function for toy data} shows that there is a conical plateau region in the first quadrant  (yellow in Figure \ref{fig: contour}), where   $ \widehat V^{\psi,\text{rel}}$ becomes concave, and slowly plateaus to the maxima. The contour plot of $\widehat{V}_\psi$ in Figure~\ref{fig: contour} shows that $\widehat{V}^{\psi,\text{rel}}(x, y)$ can be very close to the optimal value when $(x, y)$ lies inside the optimal plateau. For instance, at the point $(10, 4)$ inside the plateau, the value is 3.24993, while the global optimum, attained at infinity, is approximately 3.26. The main concern is whether gradient descent iterates get trapped in suboptimal plateaus.

The sample size was seven in our toy data, but  plateau regions and horizontal asymptotes  still persist  in larger samples.   This is unsurprising because 
the surface of $\widehat V^{\psi,\text{rel}}$ in Figure \ref{fig:value function for toy data} is not not an artifact of the toy data, but is inherited from the geometry  of $\Gamma$ (see Figure \ref{Plot: Gamma function}). Moreover, all surrogate functions we are aware of that satisfy Conditions~\ref{assump: N1}–\ref{assump: N3} exhibit similar suboptimal plateau regions, and the vanishing gradient issue persists across them.  Supplement~\ref{supp: vanishing gradient} explores in more detail how the geometry of these surrogates gives rise to this behavior. We now illustrate the optimization challenges the plateau regions pose and demonstrate how we address them using the toy dataset.  To this end, we use Figure~\ref{fig: grad descent plots}, which displays the optimization trajectories of several gradient-based methods initialized from six different starting points in our toy data example. 

 \paragraph*{Slow convergence} The presence of  plateau regions imply that  gradient descent iterates can move very slowly, even when approaching the optimal region. For example, see Path 6 in Figure~\ref{fig: grad descent 0.05}, traced by iterates initialized at (5, –5) in our toy data example. This issue can be partially addressed by  momentum-based gradient descent techniques, such as Adaptive Moment Estimation or ADAM  (see Figure \ref{fig: ADAM 0.05} for its application on our toy data example),  because momentum helps in  accelerating the convergence \citep[cf.][]{Boyd}.

  \paragraph*{Iterates diverging to suboptimal region} 
  Suboptimal plateau regions and valleys can cause gradient descent iterates to diverge in suboptimal directions.  Figures~\ref{fig: grad descent 0.05} and~\ref{fig: ADAM 0.05} show that, both vanilla gradient descent and ADAM diverge toward plateau regions with suboptimal value (around 2.29) when initialized at  $(-1, -1)$ (Path 1) and $(-5, 8)$ (Path 4) in our toy data example. We verified that tuning the learning rate or increasing iterations does not prevent the  divergence along Paths 1 and 4 for either vanilla gradient descent or ADAM. This finding  also establishes that SDSS with vanilla gradient descent or  ADAM may fail to reach the global optimum  if initiated from certain suboptimal regions. Such sensitivity to initialization is typical in non-convex optimization. In our implementation, we use random initialization, but we adopt two strategies to prevent stagnation of the gradients:
(i) restarting gradient descent with new initial points if the iterates get stuck, and
(ii) injecting noise via minibatches, i.e., computing gradients using small, random subsets of the data. For example, in Figures \ref{fig: sgd lr 0.05} and \ref{fig: ADAm w SGD lr 0.10}, we use stochastic gradient descent with minibatch size one.  In  Figure \ref{fig: ADAm w SGD lr 0.10}, when we use SGD with ADAM,  all six paths of the iterates enter the optimal region eventually.
\vspace{10pt}

While our toy example did not exhibit local optima and saddle points, they may appear in more complex settings.  To address this, alongside the earlier strategies, we use ReduceLROnPlateau \citep{goodfellow2016deep}, a learning rate scheduler that adaptively adjusts the learning rate based on validation-set performance. Below, we detail how these strategies are integrated into the optimization step of SDSS, which is summarized in  Algorithm~\ref{alg:multi_stage_opt}. In our empirical study in Section~\ref{sec: empirical}, $\U_n$ consists of either linear functions or deep neural networks with ELU/ReLU activations. Accordingly, the discussion below focuses on these function classes, though Algorithm~\ref{alg:multi_stage_opt} is applicable to any choice of $\U_n$.

\subsubsection{SDSS Optimization procedure}
\label{sec: measures we take}

First, SDSS splits the data into two parts: 80\% for training (\(\mathcal{D}_{\text{train}}\)) and 20\% for validation (\(\mathcal{D}_{\text{val}}\)). 
As we will see shortly, the validation set is used to monitor loss improvement and trigger necessary adjustments.

\paragraph*{Initialization}

In our empirical study, we implement SDSS with random initialization.  Random initialization is a  popular strategy for non-convex problems, especially when deep neural networks are involved  \citep{chen2019gradient,he2015delving, goodfellow2016deep}. An advantage of random-initialization is that it is  agnostic to the model or policy class being used \citep{chen2019gradient}.  However, when $\theta$ is high-dimensional—as in the case where $\U_n$ consists of deep neural networks—we need to scale the variance of the randomly initiated  parameter  $\theta^{(0)}$ appropriately. Otherwise, the variance of the gradient  may explode \citep{he2015delving}. Several initialization methods account for this, including He initialization \citep{he2015delving} and Xavier initialization \citep{glorot2010understanding}, both widely used to stabilize the variance of the gradients during training.
Supplement~\ref{supp: he and xavier} details how these initializers sample $\theta^{(0)}$ for neural network and linear policy classes.   In our numerical experiments, we used the He initializer, as it consistently performed better or comparably to the Xavier initializer for both linear and neural network policies.
 It may be  possible to customize  the initialization for specific policy classes, which can be an interesting direction of future research. 

\paragraph*{Minibatch stochastic gradient descent with {ADAM}}
For the minimization of  $- \widehat V^{\psi,\text{rel}}$, we use minibatch stochastic gradient descent with {ADAM}  \citep{kingma2014adam}. The momentum helps to accelerate  convergence, where the noise due to minibatching  helps prevent stagnation in suboptimal regions \citep{goodfellow2016deep, bottou2018optimization}.  ADAM is widely used in non-convex optimization  \citep{kingma2014adam, goodfellow2016deep}. Alternative adaptive learning rate optimization algorithms, such as RMSprop exist, but ADAM consistently outperformed RMSprop in our simulations.

We now explain how ADAM is implemented in SDSS, as detailed in Phase~\ref{phase: tarining} of Algorithm~\ref{alg:multi_stage_opt}.
At iteration \( r\in\{1,2,\ldots,\} \), we sample a minibatch \( \mathcal{B}  \) from the training set, and the gradient of the  loss function is computed as:
\( 
\mgrad_r = -\sum_{i \in \mathcal{B}} \nabla_{\theta} \L(\D_i;\theta^{(r-1)})/|\B|
\) for $r\geq 1$, where $\L$ is as defined in \eqref{relation: relative margin}.  
ADAM uses estimators of the first and second (elementwise) moments of $\mgrad_r$, which we denote by \(\texttt{mom}_{1r}\) and \(  \texttt{mom}_{2r}\), respectively \citep[cf.][for details]{kingma2014adam}. These terms are initialized at zero, and at each iteration \( r \), are updated as 
\begin{equation}
   \label{inalg: adam moments first def}
   \texttt{mom}_{1r}= D_1   \texttt{mom}_{1,r-1} + (1 - D_1) \mgrad_r,\quad
  \texttt{mom}_{2r} = D_2   \texttt{mom}_{2,r-1} + (1 - D_2) \mgrad_r^2,
\end{equation}
where $\mgrad_r^2$ is the vector obtained by squaring each element of $\mgrad_r$. The constants $D_1>0$ and $D_2>0$, 
known as weight decay parameters, control the decay rates of the moment estimates. They are typically set as \( D_1 = 0.9 \) and \( D_2 = 0.999 \) \citep{kingma2014adam}. 

Both moment estimates are typically initialized at zero, which introduces a bias toward smaller values in early iterations \citep{kingma2014adam}. To address this, ADAM applies the following bias corrections to $\texttt{mom}_{1r}$ and $  \texttt{mom}_{2r}$:
\begin{align}
    \label{inalg: ADAM: bias correction}
    \texttt{mom}_{1r} \gets\frac{\texttt{mom}_{1r}}{1 - D_1^r}\quad\text{and}\quad
\texttt{mom}_{2r} \gets\frac{  \texttt{mom}_{2r}}{1 - D_2^r}.
\end{align}

The $r$-th parameter update is then computed as:
\[
\theta^{(r)} \leftarrow \theta^{(r-1)} - \texttt{lr}_r \frac{{\texttt{mom}^{(1)}}_r}{\sqrt{{\texttt{mom}^{(2)}}_r} + \epsilon_{\text{num}}}.
\]
Here $\texttt{lr}_r$ is the learning rate and \( \epsilon_{\text{num}} \) is a small constant (typically \( 10^{-8} \)) added for numerical stability to prevent division by zero. The updating schedule of $\texttt{lr}_r$ is discussed below.

\paragraph*{Learning rate scheduler  ReduceLROnPlateau}
 
We start with an initial learning rate $\texttt{lr}_0$, 
which is then dynamically adjusted using a learning rate scheduler called ReduceLROnPlateau \citep{goodfellow2016deep}.
This scheduler tracks an exponential moving average (EMA) of the validation loss, as described in Phase \ref{phase: validation} of Algorithm \ref{alg:multi_stage_opt}. After every $r_{\text{eval}}$ iterations, the EMA is updated as
 \[
        \text{EMA}_{\text{val}}^{(r)} := \varkappa\,\hL_{\text{val}} (\theta^{(r)})+ (1-\varkappa)\,\text{EMA}_{\text{val}}^{(r-1)},
        \]
      where $\hL_{\text{val}}$ is the validation loss defined in \eqref{def: validation loss}, and the constant $\varkappa$ is  called the EMA smoothing parameter. $\varkappa$  is typically chosen from $(0.1, 0.8)$.
      The number $r_{\text{eval}}$, called the evaluation frequency, is typically chosen from $\{2, 3, 4\}$. If  the previously saved best  value of the EMA is larger than $\text{EMA}_{\text{val}}^{(r)}$, we update it to $\text{EMA}_{\text{val}}^{(r)}$,  and set the $r$-th iterate $\theta^{(r)}$ as the best iterate $\theta_{\text{best}}$ up to iteration $r$. Otherwise, the best EMA and  $\theta_{\text{best}}$ remain unchanged.

If the EMA fails to improve beyond a small threshold for $N_{\text{patience}}$ (typically set to  ten) consecutive updates of $\theta^{(r)}$, the learning rate $\texttt{lr}$ is reduced to $R_F \times \texttt{lr}$, where the reduction factor $R_F$ is usually chosen from $(0.5, 0.9)$ \citep{goodfellow2016deep}. This update rule is described in Phase \ref{phase: learning rate} of Algorithm \ref{alg:multi_stage_opt}. Adaptive learning rate reduction helps prevent overshooting and improves stability as parameters approach a local optimum.

\paragraph*{Random re-initialization}
If learning rate reductions fail to improve the validation loss, we infer that the iterates are  trapped in a plateau region \citep{dauphin2014identifying} and reinitiate  from a new random starting point. 
Formally, reinitialization is triggered  when the EMA fails to improve beyond a small threshold  after $N_{\text{restart}}$ many  learning rate updates, where $N_{\text{restart}}$ is typically chosen from $\{3,\ldots,10\}$  (see Phase \ref{phase: random  restart} of Algorithm \ref{alg:multi_stage_opt}). Upon reinitiation, the learning rate resets to $\texttt{lr}_0$, but SDSS retains the best $\theta$, i.e., $\theta_{\text{best}}$, and EMA found so far.  

\paragraph*{Stopping rule}
The total number of iterations, $N_{\text{epoch}}$, is pre-specified. SDSS tracks the best-performing parameter throughout optimization using the validation loss. After $N_{\text{epoch}}$ iterations, the algorithm terminates and returns the best recorded value of $\theta$.

\vspace{10pt}

\begin{spacing}{0.907} 
\small 

\refstepcounter{algorithm}
\noindent\textbf{Algorithm \thealgorithm: SDSS-OPTIM (Optimization part of SDSS)}\label{alg:multi_stage_opt}

\begin{algorithmic}[1]

\State \textbf{Input:} 
\begin{enumerate}
    \setlength\itemsep{0.3em} 
    \item Training set: $\mathcal{D}_{\text{train}}$, Validation set: $\mathcal{D}_{\text{val}}$
    \item Maximum epochs: $N_{\text{epoch}}$
    \item Initial learning rate: $\texttt{lr}_0 \in [10^{-3},0.2]$, reduction factor: $R_F \in (0.5,0.8)$
    \item Evaluation frequency: $r_{\text{eval}} \in \{2,\ldots,4\}$
    \item Patience: $N_{\text{patience}} \in \mathbb{N}$ (default 10) and Restart Threshold: $N_{\text{restart}} \in \{3,\ldots,10\}$
    \item ADAM parameters: $D_1 > 0$ (default $0.9$), $D_2 > 0$ (default $0.99$)
    \item Minibatch size: $n_{\text{mini}} \in \mathbb{N}$
    \item EMA smoothing parameter: $\varkappa \in (0.1,0.8)$ (default 0.8)
    \item Numerical constant: $\epsilon_{\text{num}} > 0$ (default $10^{-8}$)
    \item Improvement threshold: $\delta_{\text{imp}} \in (0,0.001)$ (default 0)
\end{enumerate}

\State \textbf{Initialize:} 
\[
\theta^{(0)} \text{ using He or Xavier initializer},\quad \texttt{mom}^{(1)}_0 \gets \mathbf{0},\quad \texttt{mom}^{(2)}_0 \gets \mathbf{0},\quad \texttt{lr} \gets \texttt{lr}_0,
\]
\[
R_2 \gets 0,\quad R_1 \gets 0,\quad \text{best\_val} \gets (+\infty)^{\dim(\theta^{(0)})},\quad \text{EMA}_{\text{val}}^{(0)} \gets +\infty.
\]

\For{\(r = 1,2,\dots, N_{\text{epoch}}\)}
    
    \phase{Training Step}
    \label{phase: tarining}
    \Procedure{TrainingStep}{}
        \State Sample a minibatch \(\mathcal{B} \subset \mathcal{D}_{\text{train}}\) of size \(n_{\text{mini}}\).
        \State Compute gradient:
        \[
        \texttt{grad}_r \gets -\frac{1}{|\mathcal{B}|}\sum_{i \in \mathcal{B}} \nabla_{\theta}\L(\mathcal{D}_{i}; \theta^{(r-1)}).
        \]
        \State {Gradient clipping:} 
        $
        \texttt{grad}_r \gets \frac{\texttt{grad}_r}{\max\{1, \|\texttt{grad}_r\|_2\}}.
        $
        \State Update  ADAM moments  ${\texttt{mom}}^{(1)}_r$ and ${\texttt{mom}}^{(2)}_r$ using $D_1$ and $D_2$ as in \eqref{inalg: adam moments first def}.
        \State Bias-correct moments as in \eqref{inalg: ADAM: bias correction}. 

        \State Update parameters:
        $
        \theta^{(r)} \gets \theta^{(r-1)} - \texttt{lr}\,\frac{{\texttt{mom}}^{(1)}_r}{\sqrt{v_r}+\epsilon_{\text{num}}}$.
    \EndProcedure

    \If{$r$ is a multiple of $r_{\text{eval}}$}
        \phase{Updating validation loss}
        \label{phase: validation}
        \Procedure{ValidationStep}{}
            \State Compute validation loss:
            \begin{equation}
                \label{def: validation loss}
                 \hL_{\text{val}}(\theta^{(r)}) \gets -\frac{1}{|\mathcal{D}_{\text{val}}|}\sum_{i \in \mathcal{D}_{\text{val}}} \L(\mathcal{D}_{i}; \theta^{(r)}) \text{ where }\L\text{ is as in \eqref{relation: relative margin}.}
            \end{equation}
            \State EMA update:
           $ \text{EMA}_{\text{val}}^{(r)} \gets \varkappa\,\hL_{\text{val}} (\theta^{(r)})+ (1-\varkappa)\,\text{EMA}_{\text{val}}^{(r-1)}$. 
           
            \If{\(\text{EMA}_{\text{val}}^{(r)} < \text{best\_val} - \delta_{\text{imp}}\)}
                \State \(\text{best\_val} \gets \text{EMA}_{\text{val}}^{(r)}\).
                \State Save model: \(\theta_{\text{best}} \gets \theta^{(r)}\).
                \State Reset counters: \(R_2 \gets 0,\quad R_1 \gets 0\).
            \Else
                \State \(R_2 \gets R_2 + 1\).
                \If{\(R_2 \ge N_{\text{patience}}\)}
                    \phase{Learning Rate Reduction}
                    \label{phase: learning rate}
                    \Procedure{ReduceLearningRate}{}
                        \State \(\texttt{lr} \gets R_F \times \texttt{lr}\).
                        \State Reset \(R_2 \gets 0\).
                        \State \(R_1 \gets R_1 + 1\).
                    \EndProcedure
                \EndIf
                \If{\(R_1 \ge N_{\text{restart}}\)}
                    \phase{Reinitialization}
                    \label{phase: random  restart}
                    \Procedure{ReinitializeModel}{}
                        \State Reinitialize \(\theta^{(r)}\) using He or Xavier initialization.
                        \State Reset ADAM moments: 
                        $
                        \texttt{mom}_{1r}\gets \mathbf{0},\quad   \texttt{mom}_{2r} \gets \mathbf{0}
                      $.
                        \State Reset learning rate: \(\texttt{lr} \gets \texttt{lr}_0\).
                        \State Reset counters: \(R_2 \gets 0,\quad R_1 \gets 0\).
                    \EndProcedure
                \EndIf
            \EndIf
        \EndProcedure
    \Else
        \State \(\text{EMA}_{\text{val}}^{(r)} \gets \text{EMA}_{\text{val}}^{(r-1)}\).
    \EndIf
\EndFor

\State \textbf{Output:} \(\theta^{*} \gets \theta_{\text{best}}\).
\end{algorithmic}
\end{spacing}

\vspace{5pt}

\section{Regret decay}
\label{sec: regret decay}

In this section, we analyze the regret decay rate of the SDSS policy $\widehat{d}\equiv\pred(\hf)$ from Algorithm~\ref{alg: SDSS}. The regret of  $\hd$ equals $V_*-V(\widehat d)$, which can also be presented using the estimated class scores as $V_*-V(\widehat f)$. Our theoretical results focus on non-linear function classes, where the sequence $\{\U_n\}_{n\geq 1}$ is assumed to be dense in $\W$. In the following discussion, the $\psi$-regret of any relative class score $g$ will be defined as $V^\psi_*-V^\psi(f)$, where $f=\trans(g)$ is the  class score corresponding to $g$. 

\subsection{Regret decomposition}
\label{sec: regret decomposition}
For any $g \in \W$, define $\Vr(g) := V^\psi(\trans(g))$. Since $\Vr(g)=\E[\hVr(g)]$, $\Vr(g)$ can be interpreted as the population version of $\hVr(g)$. The regret of the SDSS-estimated policy $\hd$ can be bounded using $\Vr(g)$ and $\hVr(g)$.
Lemma~\ref{lemma: error decomp}, proved in Supplement~\ref{secpf: error decomp}, provides a decomposition of the regret.

\begin{lemma}
\label{lemma: error decomp}
Suppose $\psi$ is relative-margin-based, and it satisfies  the conditions of Proposition \ref{prop: multi-cat FC} with $\J_t>0$ for all $t\in[T]$.  Suppose $\U_n\subset\W_n$. Let $\hf$ and $\hg$ denote the class scores and relative class scores, respectively, associated with the policy class $\U_n$, as defined in Algorithm \ref{alg: SDSS}.
Then for   any $\tilde g\in \U_n$, 
    \begin{align}
    \label{regret decomposition}
    V_*-V(\hf)\leq &\  C_*^{-1}  \underbrace{\lbs \sup_{g\in\W}\Vr(g)-\Vr(\tilde g)}_{\text{Approximation error}}+\underbrace{(\Vr-\hVr)(\tga-\hg)}_{\text{Estimation error}}\nn\\
    &\ +\underbrace{\sup_{g\in\U_n}\hVr(g)-\hVr(\hg)}_{\text{Optimization error: }\Opn}\rbs, \text{ where }C_*\text{ is as in Proposition \ref{prop: multi-cat FC}.}
    \end{align}
\end{lemma}
Our definitions of approximation and estimation errors depend on the choice of $\tilde{g}$. Ideally, these definitions should use $\tilde{g} = \argmax_{g \in \U_n} \Vr(g)$. However, since the maximizer may not always exist,  in our theoretical analysis, we  use a   $\tilde{g} \in \U_n$ with small  approximation error. This motivates our use of a more general definition of approximation and estimation errors.

The approximation error arises because SDSS optimizes $\hVr$ over the policy class $\U_n$ rather than the full function space $\W$.  
The estimation error arises due to the finite sample size, and it depends  on the complexity of $\U_n$.
 The optimization error occurs because our problem is non-convex.  There is a trade-off between the approximation error and the estimation  error. On one hand, complex universal approximation classes $\U_n$ may have lower approximation error, although the estimation error increases with the complexity of $\U_n$.   On the other hand, a simple $\U_n$, e.g., class of linear functions, may have lower estimation error, but its approximation error can be high if the optimal policy does not belong to that class. 

 We have discussed the optimization procedure in detail in Section~\ref{sec: implementation}. 
Since our optimization problem is non-convex, providing global-convergence-type results on the optimization error is challenging.  
In fact, as noted in Section~\ref{sec: implementation}, global convergence does not hold if initiated from  bad starting points even in simple toy settings. Moreover, it can be easily shown that the  Polyak- \L{}ojasiewicz inequality, which  guarantees global convergence for gradient-based algorithms, fails to hold in our setting. On the other hand, $ -\widehat V^{\psi,\text{rel}}$ is unlikely to be strongly convex near the optimum  since we expect $ \widehat V^{\psi,\text{rel}}$ to plateau as in the toy example. This lack of strong convexity prevents application of traditional  results, such as those in \cite{bottou2018optimization},  for the convergence of gradient descent to local optimum. Similar issues were encountered in \cite{Laha2024} due to the same structural limitation.  To the best of our knowledge, the recent work of \cite{dudik2022convex} is the only study on gradient descent's optimization error that is  relevant to our setting.
 Although their focus is convex optimization, they consider functions without finite minimizer, similar to ours. However,  they rely on a new theoretical framework based on astral planes to treat such functions. While it may be possible to analyze $\Opn$ using their approach, extending their astral plane theory to our setting would likely require a separate paper. Although this is an interesting direction for future research, it is beyond the scope of the current work.

In light of the above, this paper focuses only on the approximation and estimation errors.
Our regret bound is derived through a sharp analysis of both components, under Tsybakov’s small noise condition \citep{tsybakov2004} and complexity assumptions on $\U_n$. We will see that the resulting bound matches the best known regret decay rates for DTR under these conditions.

\begin{remark}[Requirement of relative-margin-based surrogates]
   The rate results presented in this paper do not require $\psi$ to be relative-margin-based. Analogous results to Theorems~\ref{theorem: approx error theorem}–\ref{thm: est error} and Corollary~\ref{cor: neural network} can be established for potentially non-relative-margin-based surrogates with minimal modifications to the proofs, provided the other conditions are satisfied.
However, since the SDSS method explicitly uses relative-margin-based surrogates, we restrict our attention to those surrogates in this section.
\end{remark}

\subsection{Approximation error}
\label{sec: approximation error}

In this section, we derive a sharp upper bound on SDSS's approximation error, which will be used in Section~\ref{sec: estimation error} to bound the regret of $\widehat{d}$.
While the approximation error primarily depends on the choice of $\U_n$, it also depends on the surrogate $\psi$ and the underlying distribution $\PP$. In particular, the approximation error can be small even for simpler policy classes if the optimal treatment is well-separated from suboptimal ones with high probability. To obtain a sharp upper bound, we therefore impose appropriate conditions on both $\psi$ and $\PP$.

\subsubsection{Assumptions on $\psi$}
Even under Conditions~\ref{assump: N1}–\ref{assump: N3}, the class of relative-margin-based surrogates $\psi$ is too broad to permit a sharp analysis of the approximation error. However, for specific surrogates where the minimizer of $\Vr(g)$ over $g \in \W$ is better understood, a more detailed analysis becomes possible.
To this end, we use a stronger version of Condition~\ref{assump: N3}. Recall that, if a single-stage surrogate $\phi$ satisfies Condition~\ref{assump: N3}, then $\phi(\mx; j)$ is small for large values of $\mx$ when $j \notin \argmax(\mx)$. The stronger version, Condition~\ref{assump: N3: strong}, imposes tighter control by requiring a polynomial decay of $\phi(\mx; j)$ whenever $j \neq \pred(\mx)$. If a surrogate is symmetric, i.e., if $\sum_{i \in [k]} \phi(\mx; i)$ is constant in $\mx$, then Condition~\ref{assump: N3: strong} implies Condition~\ref{assump: N3}.

\begin{assumptionp}{ N3.s}
\label{assump: N3: strong}
Consider  $\phi:\RR^k\times[k]\mapsto\RR$.  There exists $C_a>0$ (``a" in $C_a$ refers to approximation error) so that
     $\phi(\mx;j)\leq C_a (\max(\mx)- \mx_j)^{-2}$
    for all $\mx\in\RR^k$ and $j\in[k_t]$.
\end{assumptionp}
We assume that the $\phi_t$'s satisfy Condition~\ref{assump: N3: strong} and are symmetric. The symmetry assumption is used because symmetric surrogates yield sharper upper bounds on the approximation error. We can still bound the approximation error without the symmetry assumption, but the bounds will be looser. We revisit this point in the discussion following Theorem~\ref{theorem: approx error theorem}.
The kernel-based surrogate losses introduced in Section~\ref{sec: kernel based surrogates} are symmetric, and Result~\ref{result: smoothed psi satisfies approx error conditions} (proved in Supplement~\ref{secpf: kernel loss satisfies approx condition}) shows that they also satisfy Condition~\ref{assump: N3: strong}. Thus, the rate results developed in this section apply to these kernel-based surrogates.
\begin{result}
\label{result: smoothed psi satisfies approx error conditions}
  Suppose $\phi$ is a kernel-based surrogate as defined in \eqref{def: psi: smoothed pred}, and the random variable $Z$ associated with the density $\KK$ satisfies $0<E_\KK[Z^2]<\infty$.  Then $\phi$ satisfies Condition \ref{assump: N3: strong} with $C_a=2E_\KK[Z^2]$.
\end{result} 
The finite second moment condition is mild. Especially, the examples in Figure \ref{Plot: Gamma function} satisfy this condition.

\subsubsection{Small noise condition on $\PP$}
 Although it was introduced by \cite{tsybakov2004} in the context of binary classification, we will use the multiclass version of the small noise condition, which was also used by \cite{Rigollet2010NonparametricBW,qian2011,liu2006multicategory}. To introduce this condition, we define the operator
 \begin{equation}
     \label{def: mu}
     \mu(\mx)=\max(\mx)-\max_{i\in[k]:i\neq\pred(\mx)}\mx_i,
 \end{equation}
 where $\mx$ is any real vector.
 Note that $\mu(\mx)>0$ implies $\mx$ has a unique maximum, and $\mu(\mx)=0$ implies $\max(\mx)$ is attained at multiple elements of $\mx$. 
\begin{assumption}[Small noise]
\label{assump: small noise}
For each $t \in [T]$, let $Q_t^*(H_t)$ denote the vector of optimal Q-functions, i.e.,
\begin{equation}
    \label{def: vec of q functions}
    Q_t^*(H_t) = (Q_t^*(H_t,1), \ldots, Q_t^*(H_t,k_t)).
\end{equation} 
Suppose $\mu$ is as defined in \eqref{def: mu}. 
Then there exists $C>0$ such that for all sufficiently small $z>0$, 
\[\PP\slb\mu(Q_t^*(H_t))\leq z\text{ for some }t\in[T]\srb\leq Cz^{1+\alpha}.\]
\end{assumption}
The optimal Q-function $Q_t^*(H_t,a_t)$ can be seen as  the overall utility of assigning treatment $a_t$ at stage $t$ provided the patient receives the optimal treatment in the subsequent stages. A large value of $\mu(Q_t^*(H_t))$ indicates that, for a patient with history $H_t$, the optimal treatment at stage $t$ yields substantially higher utility than the second-best option—making the optimal treatment easier to detect. Assumption \ref{assump: small noise}  ensures that the probability of this utility gap being small at any stage is small. An assumption of this type is typically required to obtain sharp theoretical guarantees on the estimated policy,  as the problem becomes too hard if the utility gap is small with high probability \citep{Robins2004,luedtke2016}.   Therefore, the   small noise assumption, or similar assumptions, e.g., the geometric noise assumption  \citep{steinwart2007}, is common in the DTR and ITR literature \citep{luedtke2016,qian2011,liu2024controlling,liu2024learning,zhao2015}.

Now we are ready to state the main theorem on approximation error, which essentially says that  this error can be made small  if for each $j\in[k_t-1]$, functions of the form
\begin{equation}
    \label{def: blip cont}
    \texttt{blip}_{tj}(\mx)=Q_t^*(\mx;1)-Q_t^*(\mx;1+j),\text{  }\mx\in\H_{t}
\end{equation}
 are well approximated within the class $\U_{tn}$ for all  $t\in[T]$. Following the DTR literature, we  refer to the  above functions as blip functions \citep{schulte2014}. 

\begin{theorem}
\label{theorem: approx error theorem}
Suppose $\psi$ is as in \eqref{def: product psi} with non-negative $\phi_t$'s, and each $\phi_t$ satisfies Conditions \ref{assump: N1}, \ref{assump: N2}, and \ref{assump: N3: strong}. Moreover, for each  $t\in[T]$, we assume $\phi_t$ to be relative-margin-based (Definition \ref{def: relative margin}) and  symmetric.
 Further suppose $\PP$ satisfies  Assumptions I-V and Assumption \ref{assump: small noise}. Let $\tga=(\tga_1,\ldots,\tga_T)\in\mathcal U_n$ be such that
 \begin{equation}
     \label{instatement: approx error}
     \max_{t\in[T]}\max_{j\in[k_t-1]}\|\tga_{tj}- \texttt{blip}_{tj}\|_\infty\leq \delta_n.
 \end{equation}
    Then there exists a constant $C>0$, depending only on $\psi$ and $\PP$, so that for any $\myb_n>\delta_n^{-(1+\alpha/2)}$, 
     $\sup_{g\in\W}\Vr(g)-\Vr(\myb_n\tga)\leq C\delta_n^{1+\alpha}$.
\end{theorem}
The proof of Theorem \ref{theorem: approx error theorem} can be found in Supplement \ref{sec: approx error proofs}. The $\tilde{g}$ in Theorem~\ref{theorem: approx error theorem} may depend on $n$, but we omit this dependence from the notation for simplicity. We will now briefly discuss the conditions of Theorem \ref{theorem: approx error theorem}.  
The condition in \eqref{instatement: approx error} will be satisfied if, for example, the blip function belongs to a smoothness class such as the \Holder\ class, and $\U_n$ is dense within these classes.  Examples of such $\U_n$ include  neural network, cubic spline,  wavelet, and  decision tree classes \citep{sun2021stochastic}. See Section \ref{sec: neural network} for more details on the neural network example.  The conditions on the surrogates are satisfied by the kernel-based surrogates, as shown by Result \ref{result: smoothed psi satisfies approx error conditions}. 
The symmetry condition plays an important role. In its absence, we obtain a looser upper bound on the approximation error. For instance, the product-based surrogate violates the symmetry condition, and it can be shown that its approximation error bound becomes $\delta_n^{\alpha}$, leading to slower overall regret decay when combined with the estimation error. However, as previously discussed, the relative-margin-based condition is not necessary. 
 Finally, the original discontinuous loss $\phi_{\text{dis}}$ satisfies  Conditions \ref{assump: N1}, \ref{assump: N2}, and \ref{assump: N3: strong} with $C_a=0$.
Therefore,   the assertions of Theorem \ref{theorem: approx error theorem}  hold for this loss.

\subsection{Combined regret decay rate}
\label{sec: estimation error}
In this section, we present the final regret bound by combining the approximation and estimation errors. Recall that $\U_n = \U_{1n}^{k_1} \times \ldots \times \U_{Tn}^{k_T}$, where each $\U_{tn}$ is a class of  Borel measurable functions from $\H_t$ to $\RR$. We begin with a general theorem that provides a regret bound under complexity restriction on the $\U_{tn}$ classes (Theorem~\ref{thm: est error}). Then, in Section~\ref{sec: neural network},  we  consider the special case when the $\U_{tn}$'s correspond to ReLU neural networks and the blip functions are smooth.

\paragraph*{Complexity condition on $\U_{tn}$} Complexity assumptions on  policy classes are standard in reinforcement learning and DTR literature  for bounding the regret \citep{jin2021bellman,athey2021policy,sun2021stochastic}.
In this paper, we use covering numbers to measure the complexity of function classes. 
  For any $\epsilon>0$, the covering number \( N(\epsilon, \mathcal{U}_{tn}, \|\cdot\|_\infty) \) is the minimum number of balls of radius \( \epsilon \) 
(with respect to the uniform norm \( \|\cdot\|_\infty \)) needed to cover $\U_{tn}$. We will assume $\U_{tn}$ satisfies
\begin{equation}
    \label{ineq: bracketing entropy: ub}
    N(\e,\U_{tn},\|\cdot\|_\infty)\lesssim (\I_n/\epsilon)^{\rho_n}\quad \text{for all }t\in[T]
\end{equation}
for some $\I_n,\rho_n>0$, which will be specified later. Function classes satisfying \eqref{ineq: bracketing entropy: ub} are referred to as VC-type classes \citep[p. 41]{koltchinskii2009}. 

As an example, suppose $\U_{tn}$ is a class of ReLU networks with depth $\N_n \in \mathbb{N}$, width vector $W_n$, sparsity $s_n \in \mathbb{N}$, a linear output layer, and weights bounded by one \citep[cf.][]{schmidt2020}. Here, sparsity refers to the number of non-zero weights in the network. We denote this class by $\mathcal{F}(\N_n, W_n, s_n)$.
From Lemma 5 of \cite{schmidt2020} (see Supplementary Fact~\ref{fact: NN covering number}), it follows that $\U_{tn} = \mathcal{F}(\N_n, W_n, s_n)$ satisfies \eqref{ineq: bracketing entropy: ub} with $\rho_n = s_n + 1$ and $\I_n = (\N_n + 1)(s_n + 1)^{2\N_n + 4}$.
Other important examples of classes satisfying \eqref{ineq: bracketing entropy: ub} include wavelets \citep{GineNickl, Laha2024} and classification and regression trees  \citep{athey2021policy}.

\paragraph*{Lipschitz condition on the $\phi_t$'s} In addition to the complexity condition, we require a Lipschitz condition on the surrogates. 
\begin{condition}
\label{cond: estimation error}
For each $t\in[T]$ and $i\in[k_t]$, the function $\mx\mapsto \phi_t(\mx;i)$ is Lipschitz with constant $C>0$ in that  for all $\mx,\my\in\RR^{k_t}$,
        \[|\phi_t(\mx;i)-\phi_t(\my;i)|\leq C\|\mx-\my\|_2.\]
\end{condition}
Condition \ref{cond: estimation error}  ensures that the complexity restrictions on $\U_{tn}$'s transfer to function-classes of the form  $\mx\mapsto \phi_t(f_t(\mx);i)$, which is required for bounding the Rademacher complexity of some classes during our regret analysis. Lipschitz assumptions on surrogate losses are common in theoretical analyses within both DTR and classification literature  \citep{xue2022multicategory,bartlett2006}.
If the functions $\phi_t(\cdot; i)$ are differentiable with gradients bounded in $\ell_2$, then Condition~\ref{cond: estimation error} is automatically satisfied. Under mild differentiability conditions, the kernel-based surrogates satisfy this requirement, as will be discussed following Theorem~\ref{thm: est error}.

Theorem~\ref{thm: est error} shows that, when the optimization error is negligible and both the above conditions and those in Theorem~\ref{theorem: approx error theorem} are satisfied, the regret of SDSS is of order $O_p(n^{-(1+\alpha)/(2+\alpha)})$, up to a polylogarithmic factor.  This matches the regret decay rate obtained by \cite{Laha2024} in the binary-treatment case.

\begin{theorem}
    \label{thm: est error}
    Suppose $\mathbb P$ satisfies Assumptions I-V and Assumption \ref{assump: small noise} with $\alpha>0$.  Let $\psi$ be as in Theorem~\ref{theorem: approx error theorem}. Additionally, the $\phi_t$'s satisfy Condition~\ref{cond: estimation error} and that $\phi_t(\mx;\pred(\mx))>\J$ for  for all $t\in[T]$ and $\mx\in\RR^{k_t}$, where $\J>0$ is a constant.   Let  $\U_n=\U_{1n}^{k_1}\times\ldots\times\U_{Tn}^{k_T}$. For all $t\in[T]$, suppose the $\U_{tn}$'s satisfy \eqref{ineq: bracketing entropy: ub}  with constants $\I_n>0$ and  $\rho_n>0$ such that   $\liminf_n\rho_n>0$, $\rho_n\log \I_n=o(n)$, and  $\liminf_n\rho_n\log \I_n>0$. 
  Further suppose there exists $\tilde{g} \in \U_n$ so that $\tilde g/\myb_n$ satisfies \eqref{instatement: approx error} with $\delta_n=(\rho_n n^{-1} \log \I_n)^{1/(2+\alpha)}$
for some $\myb_n>\sqrt{n/(\rho_n\log(\I_n))}$. 
 Then there exist constants $C > 0$ and $N_0 \geq 1$ such that, for all $n \geq N_0$ and all $x > 0$, the SDSS score estimator $\hf$ (as defined in Algorithm~\ref{alg: SDSS}) satisfies
\[V_*-V(\hf)\leq  C\max\lbs (1+x)^2 (\log n)^2 \lb \frac{\rho_n\log \I_n}{n}\rb^{\frac{1+\alpha}{2+\alpha}},\Opn\rbs\]
with probability at least $1-\exp(-x)$.
\end{theorem}

The proof of Theorem~\ref{thm: est error} is provided in Supplement~\ref{sec: est error proof} and relies on tools from empirical risk minimization theory. The rate-related conditions in Theorem~\ref{thm: est error} merits some discussion. The conditions on $\rho_n$ and $\I_n$ are required to apply  Rademacher complexity bound results on certain function classes in our proof.
 The choice of $\delta_n$  ensures that the estimation error is of  order $O_p(\delta_n^{1+\alpha})$ up to a poly-logarithmic term.  Theorem \ref{theorem: approx error theorem} established that the approximation error of $\U_n$ is of the same order. Hence, the approximation error  matches the estimation error under the conditions of Theorem \ref{theorem: approx error theorem}. 
 Theorem~\ref{thm: est error} also requires that $\U_n$ contain a function $\tilde{g}$ such that $\tilde{g}/\myb_n$ satisfies \eqref{instatement: approx error}. This implies $\mathcal{U}_n$  includes functions with magnitudes at least on the order of $\myb_n$.  The result in Theorem~\ref{thm: est error} is nonparametric, as it imposes no  other structural restrictions on the function classes.

 \paragraph*{Example of surrogates satisfying the conditions of Theorem \ref{thm: est error}} 
Result \ref{result: kernel surrogate is Lipshitcz} show that kernel-based surrogates satisfy Condition \ref{cond: estimation error} under mild conditions.
\begin{result}
\label{result: kernel surrogate is Lipshitcz}
Suppose $\phi : \mathbb{R}^k \times [k] \to \mathbb{R}$ is a kernel-based surrogate as defined in \eqref{def: psi: smoothed pred}, where the kernel $\KK$ is bounded and continuously differentiable, and  $\KK'$ is bounded and integrable with respect to the Lebesgue measure. Then $\phi$ satisfies Condition \ref{cond: estimation error}.
\end{result}
The proof of Result~\ref{result: kernel surrogate is Lipshitcz} in Supplement~\ref{secpf: result kernel surrogate is Lipshitcz} relies on showing that the gradient of $\phi$ is bounded, which requires the integrability of $\KK'$. We require $\KK$ and $\KK'$ to be bounded for applying the Leibniz rule for differentiating an integral while deriving the partial derivatives of $\phi$.
A kernel-based surrogate satisfies all  conditions of Theorem~\ref{thm: est error} if $\KK$ satisfies the assumptions of both Results~\ref{result: smoothed psi satisfies approx error conditions} and~\ref{result: kernel surrogate is Lipshitcz}. Examples of such $\KK$ include the logistic and Gumbel densities (cf. Figure~\ref{Plot: Gamma function}), as well as the Gaussian density, Weibull and  Gamma densities with shape parameter greater than two, and the beta density with both parameters greater than one.

\begin{remark}[Product-based surrogates] As noted earlier, their approximation error is of order $O_p(\delta_n^{\alpha})$. Applying  proof techniques similar to Theorem~\ref{thm: est error}, one would obtain a slower regret decay rate of $O_p(n^{-\alpha/(2+\alpha)})$ for the product-based surrogates under conditions similar to  Theorem~\ref{thm: est error}.
    \end{remark}
    
 \subsubsection{Special case: neural network}
 \label{sec: neural network}
In this example, we will consider the neural network class ${\mathcal F}(\N_n,W_n,s_n)$ as $\U_{tn}$. 
As previously noted, \eqref{ineq: bracketing entropy: ub} holds for this class. To ensure that \eqref{instatement: approx error} is satisfied, i.e., that the blip functions are well-approximable by $\U_{tn}$, we impose structural assumptions on the blip functions. Specifically, we assume they belong to a \Holder\ class.
 \Holder\  assumption on the Q- and blip functions are quite common in the DTR literature \citep{zhao2015,sun2021stochastic,Laha2024}. Moreover, such smoothness assumptions are weaker than parametric restrictions on the Q- and blip functions, which are also common in the DTR literature \citep{schulte2014,kosorok2019}.

We now define a \Holder\ class. Let $k \in \NN$. Recall that for any multi-index $\mv \in \NN^k$, we denote $|\mv|_1 = \sum_{i \in [k]} \mv_i$. For any $\X \subset \RR^k$, a function $u: \X \to \RR$ is said to have \Holder\ smoothness index $\mybeta > 0$ if, for all multi-indices $\mv = (\mv_1, \ldots, \mv_k) \in \NN^k$ satisfying $|\mv|_1 < \mybeta$, the partial derivative $\partial^\mv u = \partial^{\mv_1} \cdots \partial^{\mv_k} u$ exists, and there exists a constant $C > 0$ such that
\[\frac{|\partial^\mv u(\mx)-\partial^\mv u(\my)|}{\|\mx-\my\|_2^{\mybeta-\floor*{\mybeta}}}<C\quad \text{ for all }\mx,\my\in \X\text{ such that }\mx\neq\my.\]
Here, and throughout, $\floor*{\mybeta}$ denotes the largest integer less than or equal to $\mybeta$.
For some constant $\fy > 0$, we define the \Holder\  class $\C^\mybeta(\X, \fy)$ as
\begin{align*}
\C^\mybeta(\X, \fy)=\lbs u: \X\mapsto\RR\ \bl\   \sum_{\mv:|\mv|_1<\mybeta}\|\partial^\mv u\|_\infty+\sum_{\mv:|\mv|_1=\floor*{\mybeta}}\sup_{\substack{\mx,\my\in \mathcal \X\\ 
\mx\neq \my}}\frac{|\partial^\mv u(\mx)-\partial^\mv u(\my)|}{|\mx-\my|^{\mybeta-\floor*{\mybeta}}}\leq \fy\rbs.
\end{align*}
Since $H_t$ may include categorical variables, we separate its continuous and categorical components. To this end, for each $t \in [T]$, we write $H_t = (H_{ts},\ H_{tc})$, where $H_{ts} $ and $H_{tc}$ denote the continuous and categorical parts of $H_t$, respectively. We also assume that $H_{ts} \in \H_{ts}$ and $H_{tc} \in \H_{tc}$ where $\H_t=\H_{ts}\times \H_{tc}$ for each $t\in[T]$.
Our smoothness assumption, presented in Assumption \ref{assumption: smmothness}, is on the blip function restricted to $\H_{ts}$. 
\begin{assumption}[Smoothness assumption]
\label{assumption: smmothness}
 For each $t \in [T]$, $\H_t$ is compact and $\H_{tc}$ is finite. There exist constants $\mybeta > 0$ and $\fy > 0$ such that, for all  $t\in[T]$, for any fixed $h_t \in \H_{tc}$, and any $i \in [k_t]$, the function $\texttt{blip}_{ti}(\cdot, h_t)$ is in $\C^\mybeta(\H_{ts}, \fy)$. Here, the blip function  $\texttt{blip}_{tj}$ is as defined in \eqref{def: blip cont}.
\end{assumption}
Recall that we denote the dimension of $H_T$ by $q$.  
Corollary \ref{cor: neural network} below implies  that under Assumption \ref{assumption: smmothness}, the regret decays at the rate $O_p(n^{-\frac{1+\alpha}{2+\alpha+q/\mybeta}})$ up to a logarithmic factor, modulo the optimization error. 
\begin{corollary}
 \label{cor: neural network}
 Suppose $\mathbb{P}$ and $\psi$ are as in Theorem~\ref{thm: est error}, and that $\mathbb{P}$ also satisfies Assumption~\ref{assumption: smmothness} with parameter $\mybeta > 0$.
Let $\mathcal{U}_{tn}$ be the neural network class $\mathcal{F}(\N_n, \p_n, s_n)$ with depth $\N_n = c_1 \log n$, sparsity $s_n = c_2 n^{q/((2 + \alpha)\mybeta + q)}$, and maximal width satisfying $\max W_n \leq c_3 s_n / \N_n$, where $c_1, c_2, c_3 > 0$.
 Then there exist constants $N_0>0$ and $C>0$, depending on $\PP$ and $\psi$,  such that if $c_1,c_2,c_3>C$,  then for any $n\geq N_0$ and $x>0$, the following holds with probability at least $1-\exp(-x)$:
 \[V^*-V(\hf)\leq C\max\lbs (1+x)^2(\log n)^{\frac{6+4\alpha}{2+\alpha}}n^{-\frac{1+\alpha}{2+\alpha+q/\mybeta}},\ \Opn\rbs.\]
 \end{corollary}
The proof of Corollary~\ref{cor: neural network} is provided in Supplement~\ref{secpf: cor: NN}. Under similar smoothness and small noise conditions, the rate $n^{-\frac{1+\alpha}{2+\alpha+q/\mybeta}}$ is  the minimax rate of risk decay in binary classification \citep{audibert2007}. It is also  the sharpest available rate, up to polylog terms, for DTR regret decay under analogous conditions \citep{Laha2024}. \cite{Laha2024} conjectured that this rate is  minimax-optimal for regret decay in DTR settings, under assumptions analogous to Assumptions~\ref{assump: small noise} and~\ref{assumption: smmothness}.
Under such conditions, the available  upper bound on the regret decay rate of nonparametric Q-learning is $n^{-\frac{(1+\alpha)}{2+\alpha}\frac{2}{2+q/\mybeta}}$. This upper bound is calculated using  \cite{qian2011}'s regret bounds for single-stage case; see \cite{Laha2024} for more details. The Q-learning regret decay rate is slower because it relies on the estimation of the Q-functions, and the minimax   $L_2(\PP)$  error rate of nonparametrically estimating  $\mybeta$-smooth functions is  $n^{-\frac{2}{2+q/\mybeta}}$ \citep[cf.][]{yang1999minimax}.  In principle, under general conditions such as ours, Q-function estimation   is a   harder nonparametric problem than  learning the optimal DTR, since the optimal DTR depends only on the signs of pairwise Q-function differences, not their exact values.
That said, sharper regret decay rates for Q-learning can be achieved under stronger restrictions on the Q-function class, as such restrictions simplify the underlying estimation problem   \citep{qian2011,Laha2024}.

\subsection{Practical guidance regarding SDSS}
\label{sec: practical guidelines}

Although the approximation and estimation errors of SDSS decays to zero, $\Opn$ can remain non-trivial. We are currently unable to reliably quantify $\Opn$ because the minimum attainable value of the loss $-\widehat V^\psi$ is unknown, unlike standard losses such as the squared error. Instead, we offer several guidelines for using SDSS.

First, although simultaneous optimization  uses the sample effectively by pulling all information together,  it can become highly overparameterized as $T$ increases, especially with nonparametric policy classes. Thus, we recommend using SDSS only when the sample size is large. For instance, when $T = 2$ and $k_1 = k_2 = 3$, SDSS typically needs several thousand observations to be competitive with stagewise methods. Such sample sizes are common in EHR data.

Second, SDSS generally incurs higher optimization error than stagewise methods, which solve separate single-stage optimization problems sequentially, each involving fewer parameters than SDSS.
 Therefore, in simpler settings, where all methods are expected to incur low approximation and estimation error,  SDSS may have no clear  advantage due to its higher optimization error. However, SDSS can have an advantage when the approximation and estimation errors are expected to be substantial for all methods, such as in cases with high model misspecification risk (e.g., with interpretable policies) or when covariates are high-dimensional or noisy. 
In such challenging scenarios, SDSS may outperform stagewise methods despite its higher optimization error because (i)  it does not need outcome or pseudo-outcome modeling,  (ii) it has no error propagation across stages due to its simultaneous optimization framework.\footnote{The stagewise methods generally rely on  a model-based quantity called pseudo-outcomes.  Model misspecification  for the pseudo-outcomes can result in higher estimation and approximation error, which also propagates through stages \citep{murphy2005}.} Moreover, both the kernel-based and product-based surrogates are bounded and taper off at infinity. In the context of classification, a growing body of research supports that such surrogates exhibit superior  robustness to outliers and data contamination (\citeauthor{liu2006multicategory}, \citeyear{liu2006multicategory}; \citeauthor{wu2007robust}, \citeyear{wu2007robust}; \citeauthor{masnadi2008design}, \citeyear{masnadi2008design}; \citeauthor{nguyen2013algorithms}, \citeyear{nguyen2013algorithms}; \citeauthor{wang2020robust}, \citeyear{wang2020robust}). 
     The challenges of model misspecification and noise are common in real-world EHR data. It is therefore unsurprising that SDSS outperforms both Q-learning and ACWL in our sepsis data example (Section \ref{sec: application}). Hence, we recommend SDSS in settings with noisy or high-dimensional covariates, or when model misspecification is a concern.
      \begin{figure}[!htb]
    \centering
       \begin{subfigure}{\textwidth}
        \centering
      \includegraphics[ height=1.7in]{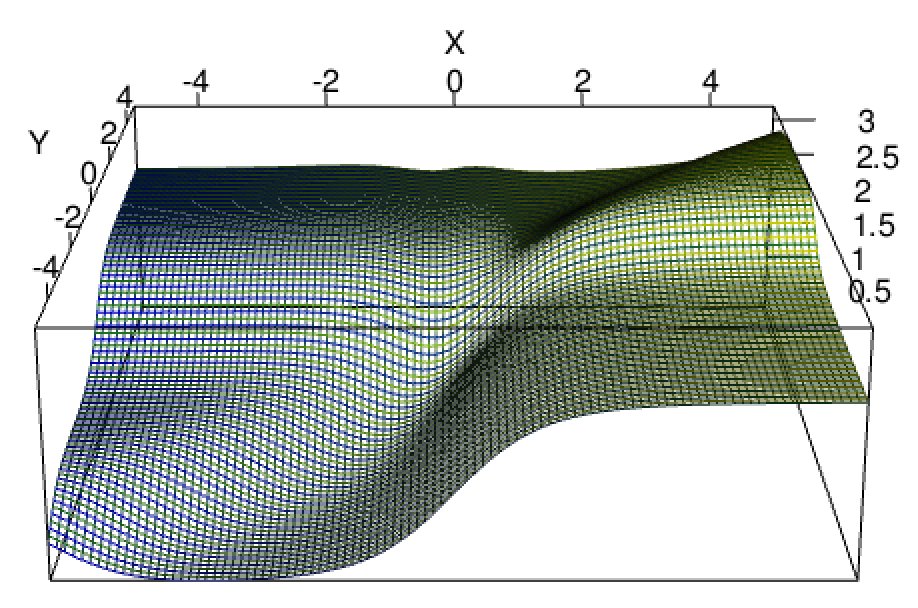}
        \caption{Plot of $ \widehat V^{\psi,\text{rel}}(x,y)$ (plotted in the Z axis) vs $x$ and $y$ }
        \label{fig:value function for toy data}
    \end{subfigure} \\
    \begin{subfigure}{0.44\textwidth}
        \centering
      \includegraphics[width=\linewidth]{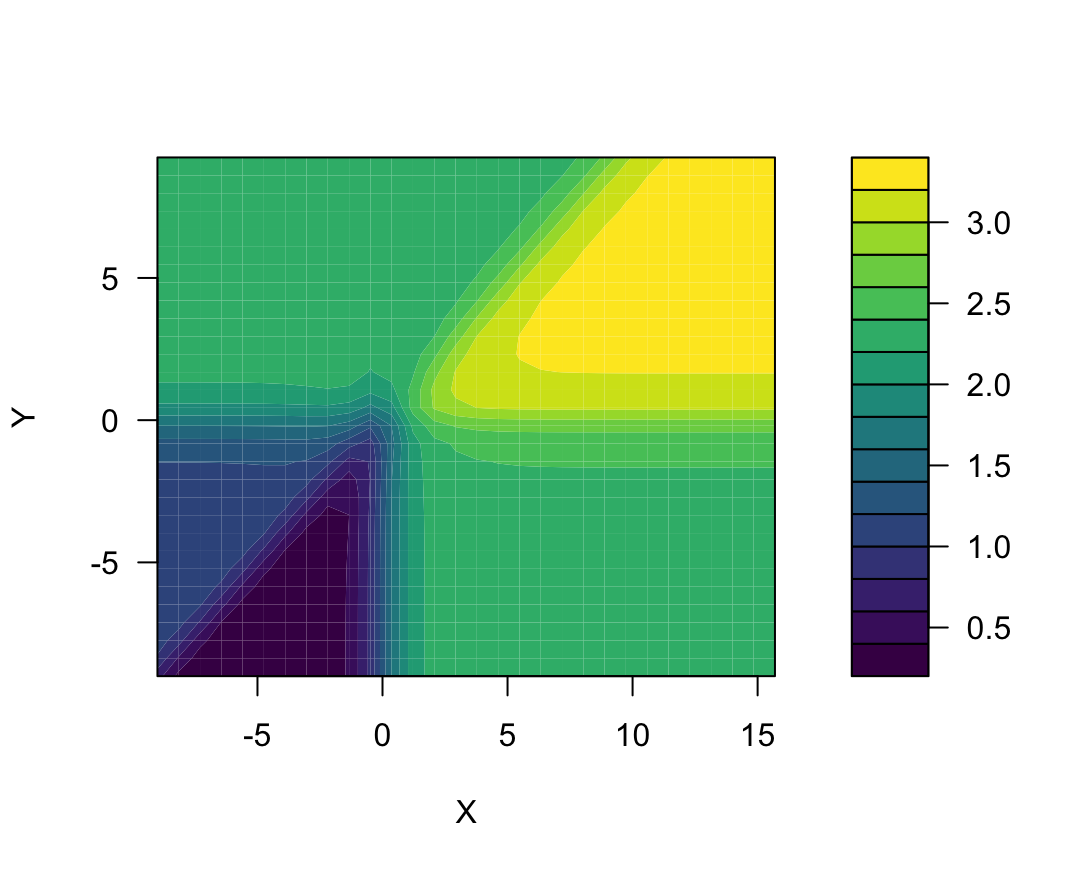} 
        \caption{Contour plot of   $ \widehat V^{\psi,\text{rel}}$ }
        \label{fig: contour}
    \end{subfigure} 
    \begin{subfigure}{0.44\textwidth}
        \centering
       \includegraphics[width=\linewidth]{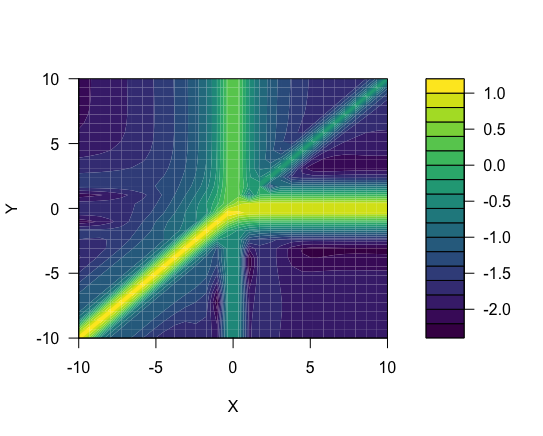} 
        \caption{Contour plot of the logarithm of the $\ell_2$ norm
of the gradient of  $ \widehat V^{\psi,\text{rel}}$ }
        \label{fig:contour: grad}
    \end{subfigure}
   \caption{{\bf Plots related to  $ \widehat V^{\psi,\text{rel}}$ for the toy example in Section \ref{sec: implementation}.} The formula of $ \widehat V^{\psi,\text{rel}}(x,y)$ is provided in \eqref{opti: value fn: toy data} and the toy data is provided in Table \ref{table:toy_data}. 
   In plot (c), the logarithm base is 10.  Higher negative values in this plot indicate plateau regions, where the gradients become  close to zero.}
    \label{fig:contour plot}
\end{figure}

 Finally, multiple Fisher consistent surrogates are available for SDSS. One way to combine several surrogates is through an ensemble approach, where we estimate the optimal DTR using multiple $\psi$'s, and the final treatment assignment for each subject is determined by the majority vote. However, ensembling may increase variance, particularly with non-linear policy classes (see Section \ref{sec: simulation} for details).
\begin{figure}[!htb]
     \centering
    \begin{subfigure}{0.47\textwidth}
        \centering
        \includegraphics[width=\linewidth]{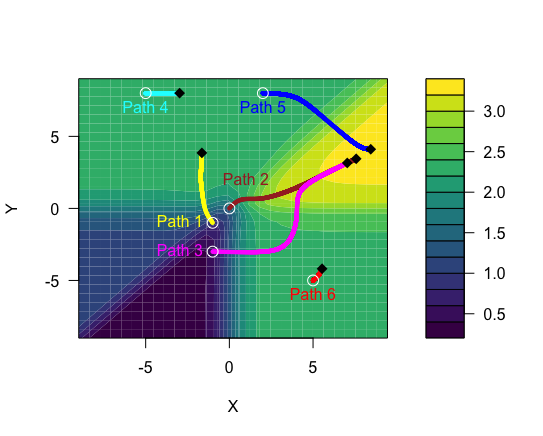} 
        \caption{Vanilla gradient descent (learning rate 0.05)}
        \label{fig: grad descent 0.05}
    \end{subfigure}
    ~
     \begin{subfigure}{0.47\textwidth}
        \centering
        \includegraphics[width=\linewidth]{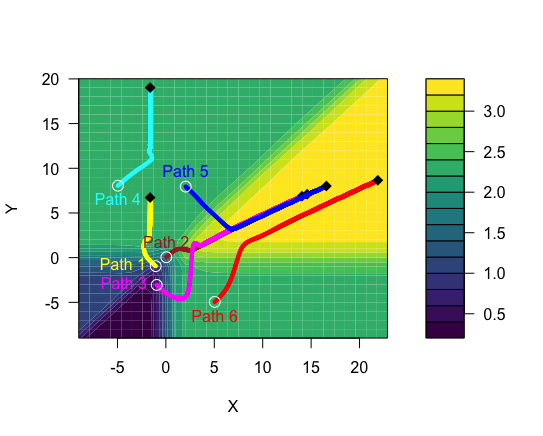} 
        \caption{ADAM without minibatches  (learning rate: 0.05)}
        \label{fig: ADAM 0.05}
    \end{subfigure}
    ~
     \begin{subfigure}{0.47\textwidth}
        \centering
      \includegraphics[width=\linewidth]{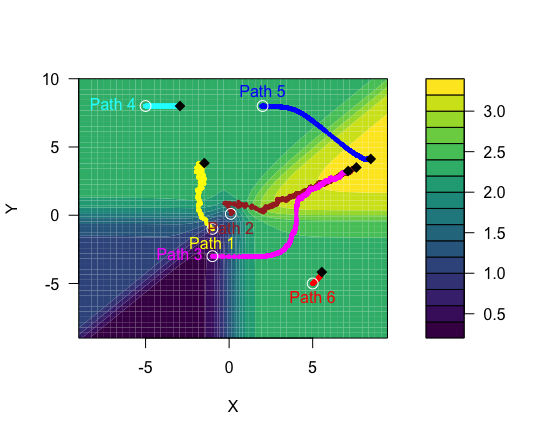} 
        \caption{SGD (learning rate: 0.05) }
        \label{fig: sgd lr 0.05}
    \end{subfigure} 
    ~
     \begin{subfigure}{0.47\textwidth}
        \centering
      \includegraphics[width=\linewidth]{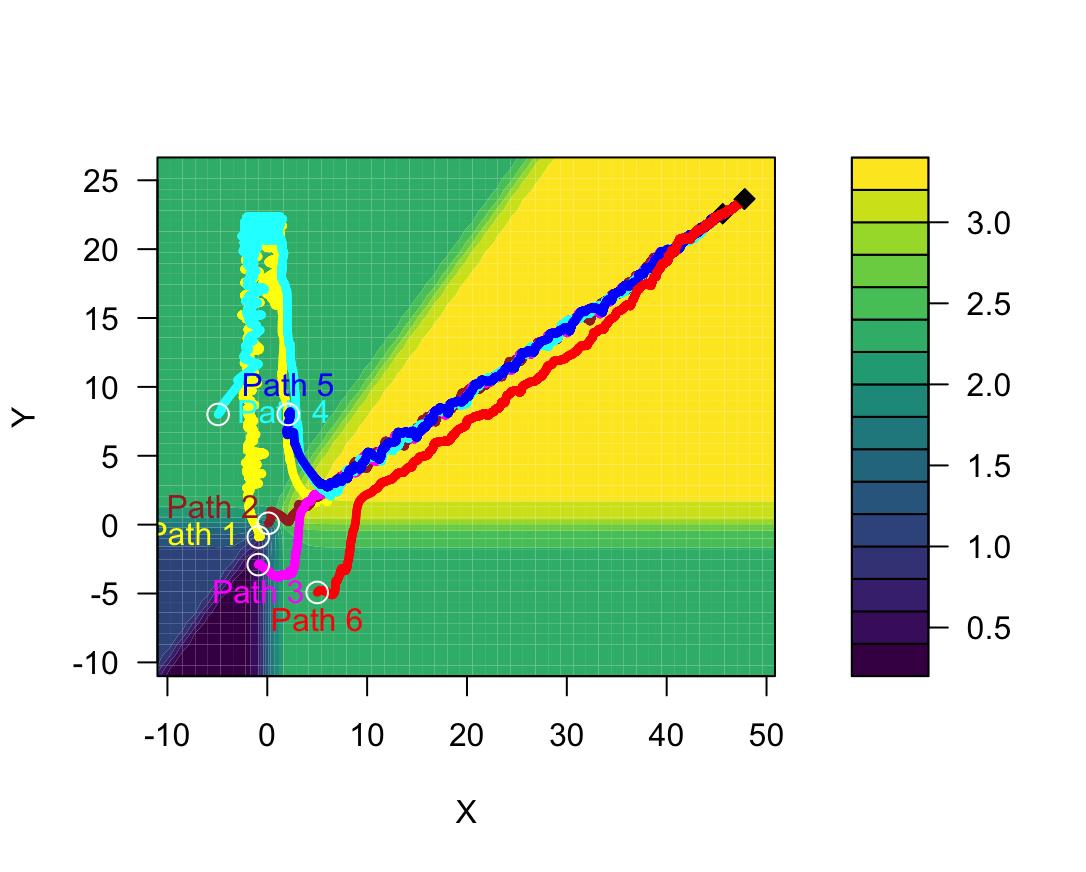} 
        \caption{ADAM with SGD (learning rate 0.10) }
        \label{fig: ADAm w SGD lr 0.10}
    \end{subfigure} 
     \caption{{\bf Gradient descent for the toy data in Section \ref{sec: implementation}}  The plots display the paths traced by iterates initiated from 6 different initialization points for (a) vanilla gradient descent, (b) ADAM (without minibatching), (c) SGD, and (d) ADAM with SGD. The white circle and the solid black rectangle mark the starting point and the end point of the paths, respectively. Although the algorithms minimize $-\hVr$, the paths are shown on the contour plot of $\hVr$ (also provided in Figure~\ref{fig: contour}).  The legend to the right presents the color scale. The yellow region indicates the optimal plateau, where the objective value is close to the optimum.      The first three plots used 5,000 iterations, while the last used 20,000. The SGD batch size was one. The ADAM parameters  were set to their default values as specified in Algorithm~\ref{alg:multi_stage_opt}.
     }
    \label{fig: grad descent plots}
\end{figure}

\section{Empirical analysis}
\label{sec: empirical}
As discussed in Section \ref{sec: implementation},  we recommend   SDSS  as an alternative to the conventional methods only in settings where the sample size is large and modeling may be challenging. Therefore, we focus only on such scenarios. 
Despite the non-convexity discussed in Section \ref{sec: implementation}, our optimization strategy involving ADAM, random restart, and minibatching consistently allowed SDSS to find high-value policies in these simulation settings, often outperforming comparators. We evaluate the performance of our method both in synthetic datasets (Section \ref{sec: simulation}) as well as EHR data on septic patients (Section \ref{sec: application}).   To investigate the benefit of our   SDSS  in such scenarios,  we compare   SDSS  with  two well-known DTR learning methods: (i)  Q-learning \citep{Watkins1989} and (ii) the Adaptive Contrast Weighted Learning or ACWL  \citep{tao2017adaptive}.

\subsection{Simulations}
\label{sec: simulation}

We consider two simulation schemes, each comprising two or three settings.
\paragraph*{Scheme 1}
It is a two-stage scenario with treatments  \(A_1, A_2 \in \{1, 2, 3\}\), assigned uniformly at random with  \(
\pi_1(A_1 = i \mid H_1) = \pi_2(A_2 = i \mid H_2) = 1/3 \) for $i\in\{1,2,3\}$. Therefore, $k_1=k_2=3$. The baseline covariates are
\[
O_1 = (X_{11}, X_{12}, X_{13}) \sim N(\mathbf{0}_3, 10 I_3),
\]
which have a substantial spread due to their high variance. 
 The first-stage outcome model is linear in the covariates. Specifically, 
\[
Y_1 = A_1 (X_{11} + X_{12} + X_{13}) + 3 + \varepsilon_1, \text{ where }
\quad \varepsilon_1 \sim N(0, 1)
\]
is the random error independent of the treatments and covariates. We let $O_2\sim N(0,1)$. 
The second-stage outcome model, which  includes a nonlinear component, is given by
\[
Y_2 = \omega X_{1A_2}^2  + O_2 + 3 + \varepsilon_2,
\]
where $\omega>0$ is a simulation parameter  and $\varepsilon_2\sim N(0,1)$ is the random error, which is  independent of the other random variables. 
Larger values of \(\omega\) increase the variance of \(Y_2\), as well as the nonlinear contribution of the second stage to the value function. Therefore, as we shall see, modeling  the Q-functions becomes more challenging as $\omega$ increases.  
Hence, we will refer to $\omega$ as the complexity parameter. 
We let $\omega$ vary in $\{10,20,40\}$, generating three distinct settings.
The optimal policies are given by
\[
d_1^*(H_1) = \begin{cases}
    3 & \text{ if }X_{11} + X_{12} + X_{13} > 0\\
    1 & \text{ o.w.}.
\end{cases}\quad\text{ and }\quad d_2^*(H_2) = \argmax_{i \in \{1,2,3\}} X_{1i}^2.
\]
According to our definition of linear policies in Section \ref{sec: set-up}, 
 $d_2^*$ is non-linear since the corresponding class scores are quadratic, but $d_1^*$ is linear because the corresponding class-scores are linear.

\paragraph*{Scheme 2}

This scheme has high-dimensional covariates. The baseline covariate \(O_1 \in \mathbb{R}^{n \times p}\) is distributed as a centered  $p$-variate Gaussian random vector with identity covariance matrix, where \(p \in \{50, 100\}\). The covariate space of the sepsis data in Section \ref{sec: application} has similar dimension. 
We provide the stage-specific outcome models of Scheme 2 below:

\[
Y_1 
= \slb\myD^{(1)}_{A_1} \sin(O_1^T \mathbf{1}_p)\srb^2 
  + \myE^{(1)}_{A_1} \cos(O_1^T \mathbf{1}_p) 
  + \varepsilon_1,
\]

\[
Y_2 
= \slb\myD^{(2)}_{A_2} \cos(O_1^T \mathbf{1}_p)\srb^2 
  + \myE^{(2)}_{A_2} \sin(O_1^T \mathbf{1}_p)
  + \varepsilon_2,
\]
where \(\varepsilon_1, \varepsilon_2 \sim N(0, 1)\). The stage-specific parameters $\myD^{(1)},\myD^{(2)},\myE^{(1)},\myE^{(2)}\in\RR^3$,  whose values are provided in  Supplement \ref{supp: simulation}. Here we  remind the readers that for a vector $\mx\in\RR^k$, $\mx_i$ represents its $i$-th element.   Scheme 2 outcome models are non-linear functions of the history.  
The resulting optimal treatment assignments are also non-linear. Their formulas are provided in Supplement~\ref{supp: simulation}, which shows that, at both stages, the optimal assignments $d_t^*(H_t)$ depend solely on $O_1^T \mo_p$ and take values between two and three.  Figure \ref{fig:scheme 2 optimal treatments} plots the optimal treatment assignments as a function of $O_1^T\mo_p$. As can be seen from Figure \ref{fig:scheme 2 optimal treatments}, for each stage $t$, the regions where the optimal treatment is two (or three) are disconnected sets  of the form $\cup_{i=1}^\infty\{h_t\in\H_t: c_i\leq o_1^T\mo_p\leq c_i',\ o_1\subset h_t\}$ where $\{c_i\}_{i\geq 1}$ and $\{c_i'\}_{i\geq 1}$ are sequences of distinct numbers. The disconnected components appear due to the periodic nature of the outcome models. Decision boundaries of this form are hard to approximate with linear policies, but they are  approximable with suitable non-linear policies.

\begin{figure}[H]
    \centering
    \includegraphics[width=\linewidth]{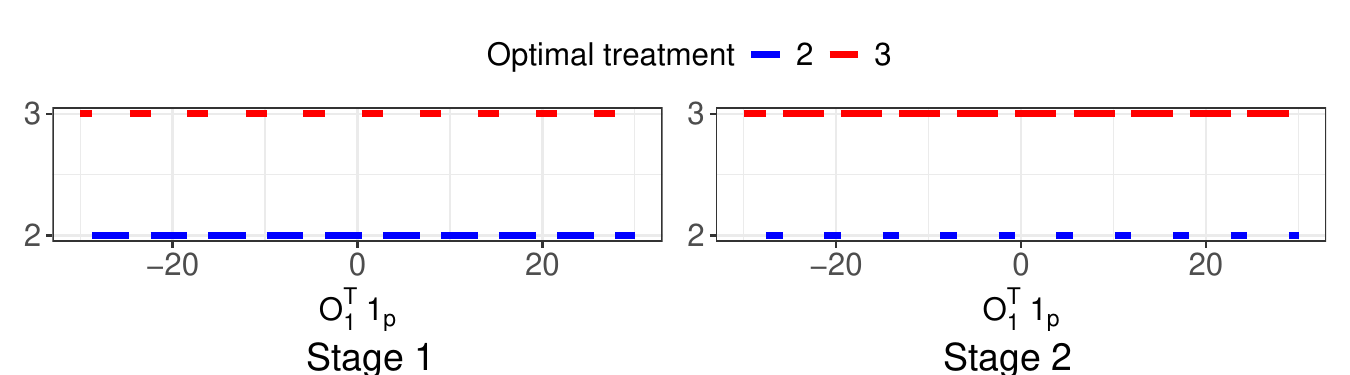}
    \caption{{\bf Optimal treatment assignments for Scheme 2 as a function of $O_1^T\mo_p$.} As mentioned in Section \ref{sec: simulation}, the optimal treatment assignments  in Scheme 2 are functions of  $O_1^T\mo_p$ only. The X-axis represents $O_1^T\mo_p$, while the Y-axis indicates the corresponding optimal treatment. Here, the optimal treatment assignments are either two or three, which are colored blue and red, respectively.   
    }
    \label{fig:scheme 2 optimal treatments}
\end{figure}

\emph{Propensity scores:} Here the propensity scores depend on the history. For $i\in[3]$, let us define the function $ \myeth_{t}:\H_t\times[3]\mapsto\RR^p$ as
\[
  \myeth_1(H_t;i) =
\begin{cases}
\mathbf{1_p} & \text{if } i = d_t^*(H_t) \\
-\mathbf{1_p} & \text{if } i \neq d_t^*(H_t).
\end{cases} 
\]
The propensity scores are given by
\begin{gather*}
   \pi_1(A_1 \mid H_1) = \frac{\exp\left(  O_1^T \myeth_1(H_1;A_1) \right)}{\sum_{i=1}^{3} \exp\left(  O_1 \myeth_1(H_1;i) \right)}\\
   \pi_2(A_2 \mid H_2) = \frac{\exp\left( O_1^\top \myeth_2(H_2;A_2) + 0.5 A_1 + 0.5 Y_1 \right)}{\sum_{i=1}^{3} \exp\left( O_1^\top \myeth_2(H_2,i) + 0.5 A_1 + 0.5 Y_1 \right)}.
\end{gather*}

\subsubsection{Common simulation setup}
For each of the 5 simulation settings resulting from the  2 schemes, we consider two sample sizes, \(n = 5,000\) and \(n = 15,000\). For each sample size and each setting, we generate 100  Monte Carlo samples. We compare the  methods using their value functions. We also use the value function of the optimal DTR as a benchmark. We estimate the value function of each DTR using an independent held-out sample of size \(n = 10,000\),  using the  formula $V(d)=\E_{d}[\sum_{t=1}^TY_t]$. 
  To ensure comparability across methods, the held-out set was the same for all methods in each Monte Carlo replication.

\begin{figure}[!htb]
    \centering
    \includegraphics[width=\textwidth]{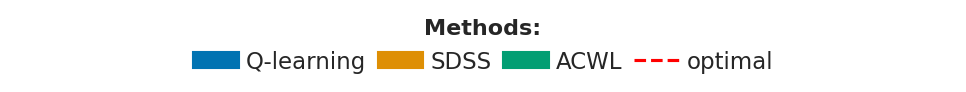}\\[1ex]
    \begin{subfigure}[b]{.49\textwidth}
        \includegraphics[width=\textwidth,height=1.5in]{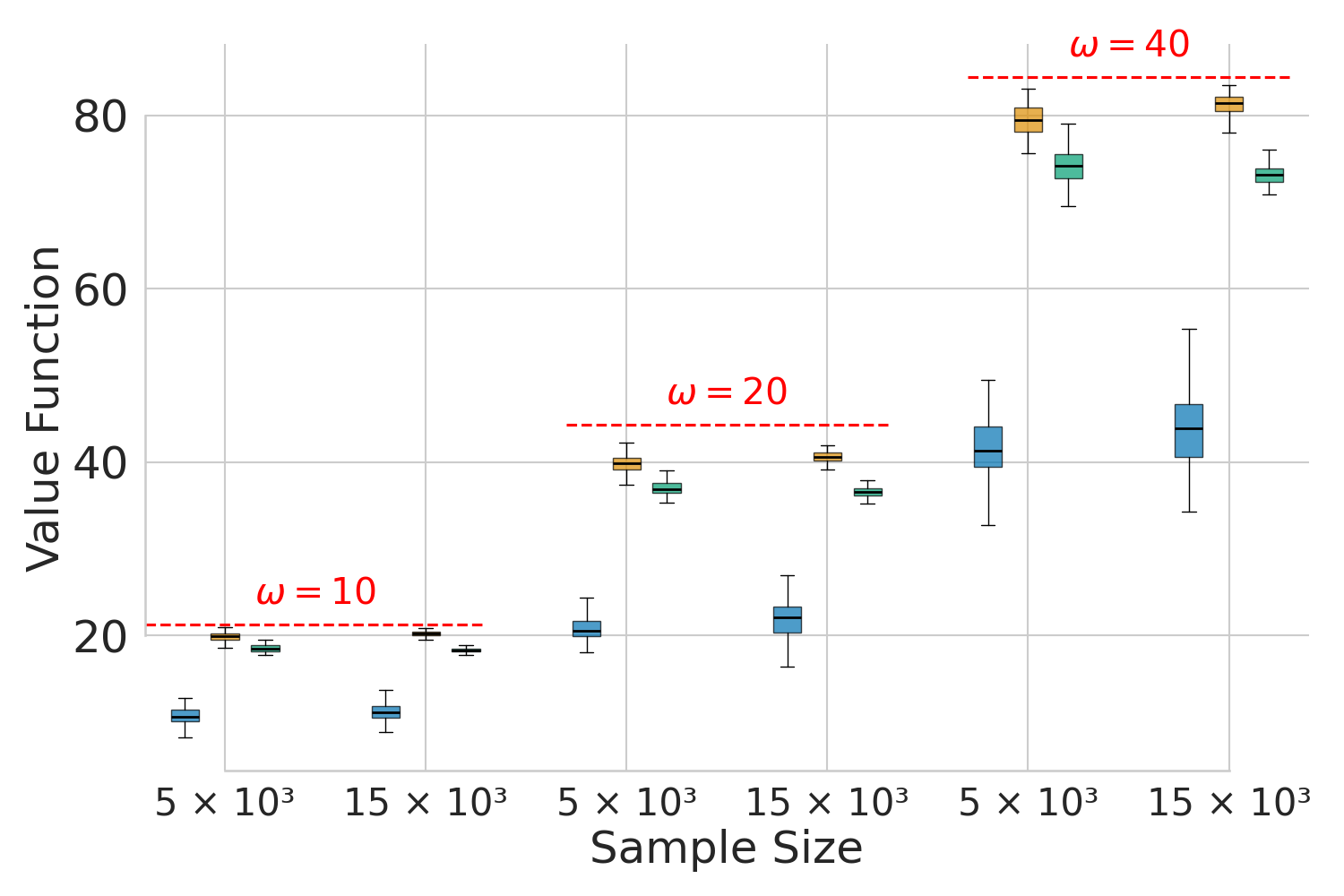}
        \caption{Non-linear policies (Scheme 1)}
        \label{fig:setup1_nonlinear}
    \end{subfigure}
    \hfill
    \begin{subfigure}[b]{.48\textwidth}
        \includegraphics[width=\textwidth,height=1.5in]{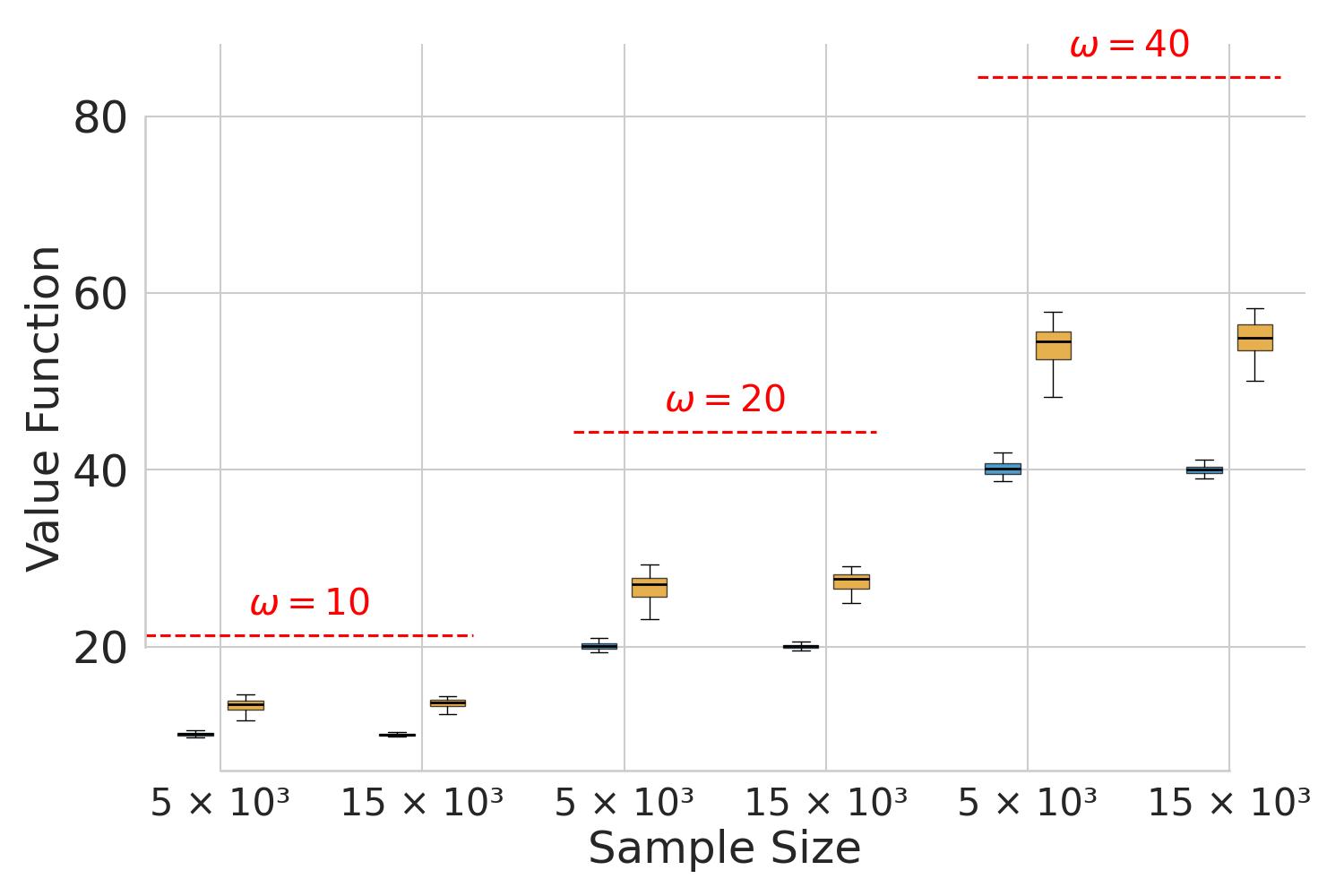}
        \caption{Linear policies (Scheme 1)}
        \label{fig:setup1_linear}
    \end{subfigure}
    \\
    \begin{subfigure}[b]{.49\textwidth}
        \includegraphics[width=\textwidth,height=1.5in]{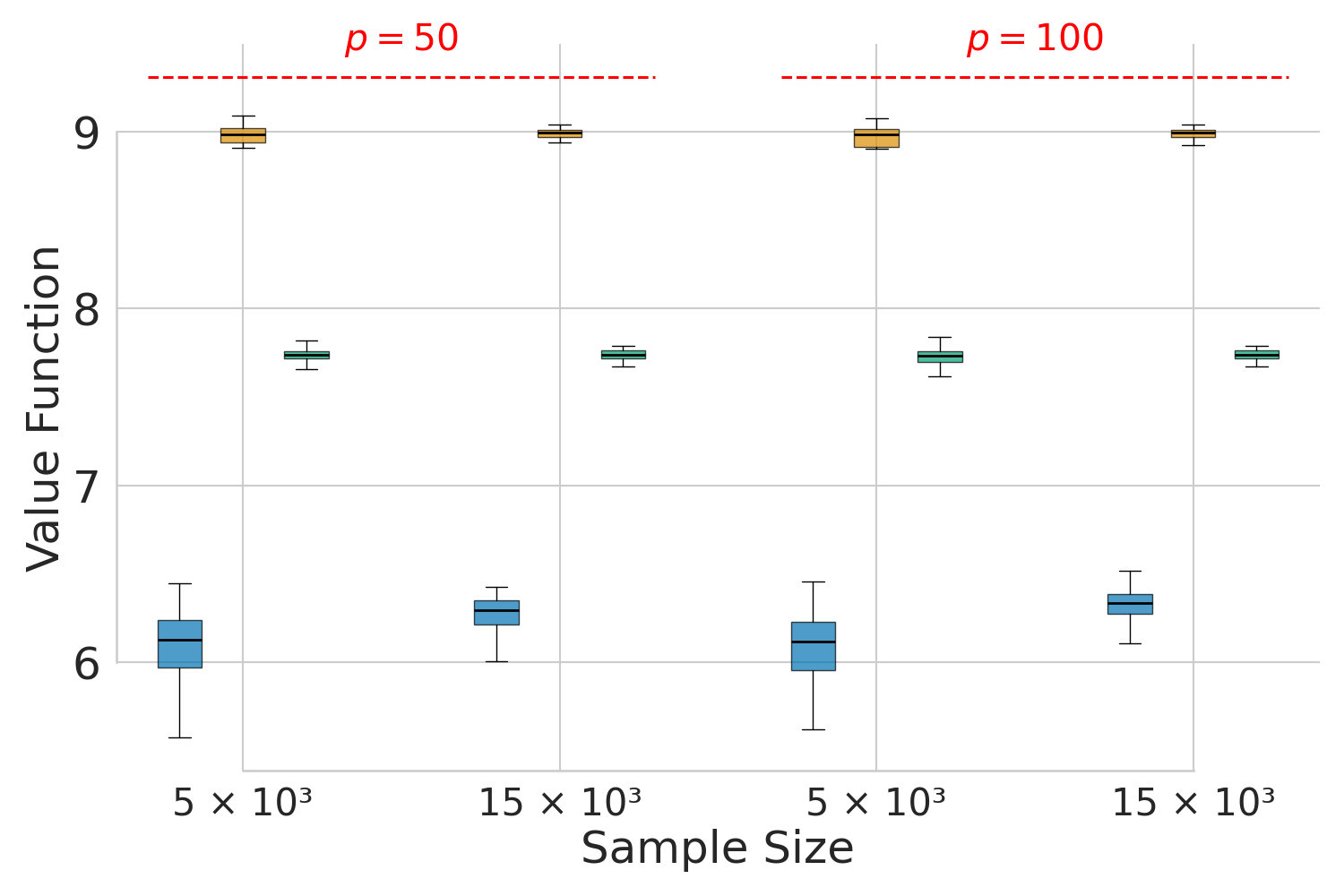}
        \caption{Non-linear policies (Scheme 2)}
        \label{fig:setup2_nonlinear}
    \end{subfigure}
    \hfill
    \begin{subfigure}[b]{.48\textwidth}
        \includegraphics[width=\textwidth,height=1.5in]{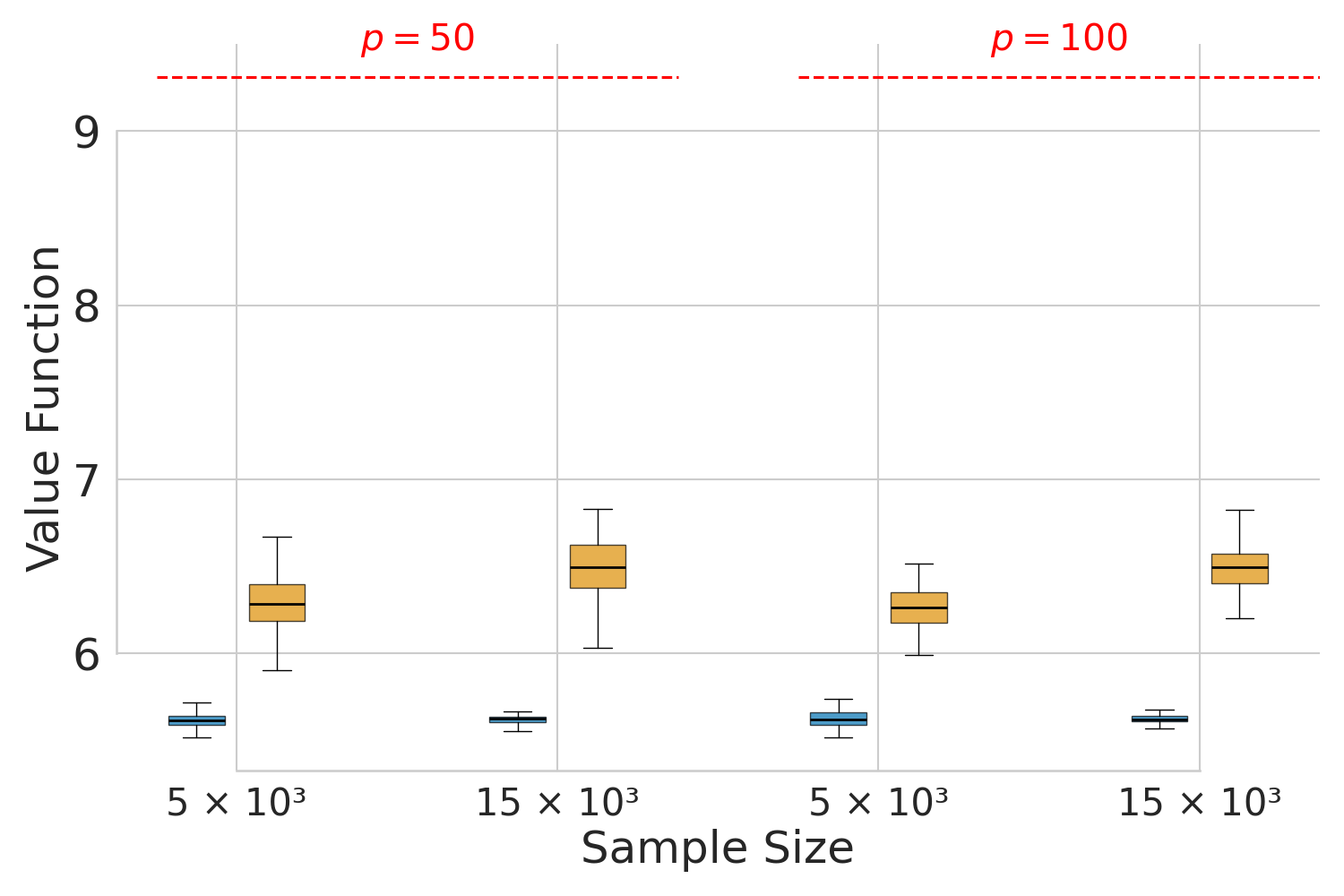}
        \caption{Linear policies (Scheme 2)}
        \label{fig:setup2_linear}
    \end{subfigure}
    \caption{\textbf{Boxplot of value functions:} Boxplot of estimated value function of different methods for 100 replications. The value functions in Scheme 1 are scaled by $10^{-2}$ for better visual comparability. 
     The X-axis indicates the sample size and the Y-axis represents the average value function over 100 replications. Each box shows the interquartile range (IQR), with edges at the 25th and 75th percentiles and a line for the median. Whiskers extend to 1.5$\cdot$IQR.
    The red dashed line represents the optimal value function. In Scheme one, $\omega$ corresponds to the complexity parameter, and in Scheme two,  $p$ corresponds to the dimension of the baseline covariate space.
    }
    \label{fig:simulation}
\end{figure}

\subsubsection{Methods}
 As mentioned previously, we benchmark the performance of SDSS against Q-learning and ACWL \citep{tao2017adaptive}. 
 All three methods, SDSS, Q-learning, and ACWL, are considered for nonlinear policies. The non-linear  DTR for SDSS  and Q-learning is implemented using deep neural network. For  ACWL, we use its default policy class  based on classification and regression trees (CART), which is nonlinear. Since currently ACWL  only uses tree-based policies, we exclude it from comparisons involving linear policies. For Scheme 1, whose first-stage outcome model includes  treatment-covariate interaction terms, we incorporated first-order interaction terms between treatment and covariates in the linear model used for Q-learning. In contrast, no interaction terms were included for Scheme 2 due to the high dimensionality of the covariates. The main implementation details for the methods are provided below.

\paragraph*{SDSS} SDSS was implemented using  Algorithm \ref{alg:multi_stage_opt} with 4:1 training:validation split. In case of the  neural network policies, the networks for the relative class score $g_{ti}$'s share parameters for each stage $t$, differing only at the outer layer of each stage.   This strategy was adopted to help reduce the burden of parameters.  For simplicity, we also used the same network configuration, e.g., depth, dropout rate, width, etc., for each relative score. SDSS was implemented using PyTorch, where the linear policy was implemented using neural network with no activation functions.  The SDSS hyperparameters are provided in Supplement  \ref{appendix:simulationParameters}.   The performance of SDSS was generally robust to moderate changes in hyperparameters such as initial learning rate, number of epochs, and the optimization parameters from Algorithm \ref{alg:multi_stage_opt}. Most optimization-related parameters were fixed throughout the simulations (see Table~\ref{tab: common parameters sim} in the Supplement), using default values from Algorithm~\ref{alg:multi_stage_opt}. For neural network policies, ReLU and ELU activation functions yielded the best, and comparable, results. In this case, varying the number of hidden layers between one and three, we observed the best performance in  Scheme 1 with one layer, while the high-dimensional Scheme 2  saw the best performance with three layers.   In practice, we recommend tuning the neural network  hyperparameters such as the number of layers, dropout rate,  learning rate, etc., by monitoring training and validation error curves. For the optimization-related hyperparameters, the values in Supplementary Table~\ref{tab: common parameters sim} may be used. The number of random restarts was capped at three in all replications. 
 The code for our simulation section, including the implementation of SDSS,  is available at the Github repository \cite{sdss-simu}. 

\subparagraph*{Choice of surrogate} 
Although the approximation and estimation error bounds for the kernel-based surrogates are sharper due to their symmetry property, they were outperformed by the product based surrogates in our simulations. We suspect that the simpler form of the product-based  surrogate may have optimization-related advantages, which makes this surrogate less sensitive to hyperparameters. 
 In view of the above, we present our results only for the product-based surrogate, where we take 
$
\tau(x) = 1+x/\sqrt{1+x^2}
$
as it performed slightly better than the other $\tau$'s we considered, which are 
$1 + \tanh(5x)$, $1 + x/(1+|x|)$, and $1 + 2\arctan((\pi x)/2)/\pi$.
That said, all four choices led to comparable performance.
If we ensemble these four surrogates using the strategy mentioned in Section \ref{sec: practical guidelines},  the   value function of the  estimated policy remains comparable to the above-mentioned best-performing $\tau$. However,  ensembling slightly increases  the variance of  $V(\widehat d)$ for  nonlinear policy classes, likely due to the higher parameter burden in this case.  Figure \ref{fig:setup2ensemble} in Supplement \ref{sec: simulation} illustrates this under  Scheme 2.

\paragraph*{Benchmark methods}
Both Q-learning and ACWL are stagewise methods relying on backpropagation. While Q-learning is purely model-based, ACWL uses both modeling and single-stage direct search. Q-learning requires modeling the Q-functions whereas ACWL require modeling its contrasts, both of which can become challenging when the outcome models are complex. However, since ACWL's stage-wise direct search is based on a doubly robust construction, it is supposed to be less affected by model misspecification as long as the propensity scores are consistently estimated. Here, the true propensity scores are provided to ensure comparability with SDSS.
\vspace{5pt}
\subparagraph*{\underline{Q-Learning}}
Similar to   SDSS, we implemented both linear and nonlinear versions of Q-learning using PyTorch. The nonlinear Q-learning uses a feed-forward ReLU neural network with two hidden layers with 64 and 32 neurons per layer. To ensure comparability with SDSS,   we used  the ADAM optimizer   with random restart and learning rate scheduler. We considered a training-validation split of $4:1$, because, similar to SDSS, the learning rate was updated when validation loss becomes stagnant. 
 More details on the network architecture and optimization hyperparameters are provided in Supplement \ref{supp: Q-learning simulation parameters}.  Q-learning was generally robust to moderate changes in neural network and optimization-related hyperparameters, and performed the best with ReLU networks with three hidden layers for the neural network policies.

\vspace{5pt}

\subparagraph*{\underline{ACWL}}\cite{tao2017adaptive}'s 
 ACWL is a stagewise method, which requires estimating contrasts of  Q-functions. ACWL uses tree-based policies to  directly optimize a doubly robust estimator of  single-stage value functions. 
We implemented ACWL using the authors’ publicly available R code, which uses the \texttt{rpart} package.   We implemented ACWL using the default setting. The provided functions for ACWL also do not take CART hyperparameters as input. 

  \subsubsection{Results}   
  Among the methods, Q-learning was fastest (4–5 seconds), followed by SDSS (7–8 seconds), and ACWL was slowest (about 11 seconds); overall, the runtimes were broadly comparable. The simulations were run on a MacBook Pro with a 2.6 GHz 6-core Intel Core i7 processor and 16GB RAM.
  
Figure \ref{fig:simulation} displays the boxplots of the  value function estimates. 
We discuss the case of non-linear policies first. Q-learning has the lowest value-function estimate, with the highest Monte Carlo variance, among all methods. This observation indicates that Q-learning may still be affected by model-misspecification even when we use universal approximation classes such as deep neural networks and fairly large sample sizes.  Although the stagewise method  ACWL is competitive with   SDSS, on average, the value function estimates of   SDSS  are consistently the highest among all methods, underscoring the advantage of simultaneous optimization in our simulation settings. Figure \ref{fig:setup1_nonlinear} shows that the gap between SDSS's  and the other methods' value function estimates slowly increases as the complexity parameter $\omega$ increases. To this end, note that both the range and variance of $Y_2$ increase with $\omega$. Since SDSS uses a bounded surrogate loss, it may be potentially less sensitive to higher variance in  $Y$ compared to the other methods, which use the squared error loss for modeling purposes. In a related context, bounded non-convex surrogate losses have been empirically demonstrated to offer increased robustness to label noise in classification tasks
\citep{natarajan2013learning,akhtar2024hawkeye}.

When considering linear policies,  Q-learning has lower variance than   SDSS, as shown by Figure \ref{fig:setup1_linear} and \ref{fig:setup2_linear}. This is unsurprising because being a linear-regression-based method,   linear Q-learning is computationally simpler than linear   SDSS, which still involves a non-convex optimization. In general, the value functions for linear policies are smaller than those of non-linear policies for both methods, which is expected because the optimal DTR is non-linear under both schemes. The value function estimates of SDSS are consistently higher than that of  Q-learning, suggesting that SDSS is less affected by the  policy class misspecification in our simulation settings. In the binary treatment setting, \cite{zhao2015,jiang2019entropy,Laha2024} also observed that direct-search-based methods often outperform or match Q-learning when the policy class is misspecified. 
As in the nonlinear case,  the performance gap between   SDSS and Q-learning increases as the complexity parameter $\omega$ increases in Scheme 1, demonstrating the advantage of SDSS as the variance of $Y$ increases. 

These  findings indicate that in  our simulation settings,   SDSS is either competitive to or exhibit better performance than the other methods, both with linear and non-linear policies.

\subsection{Application to an EHR Study of Sepsis}
\label{sec: application}
Sepsis is a critical condition where the body's immune response to infection causes severe, and potentially lethal, complications. Intravenous fluid (IV) resuscitation is often used to manage severe cases of sepsis \citep{Raghu}. Due to the significant variability in how sepsis affects individuals, personalized treatment approaches for IV fluid resuscitation are recommended   \citep{Lat2021, Komorowski2018}. Despite its critical importance, there still remains a lot of uncertainty about the best ways to administer  intravenous (IV) fluid, particularly across diverse patient subgroups and clinical conditions \citep{marik2015demise}. In this section, we investigate the effectiveness of SDSS and the comparators in estimating the optimal IV fluid resuscitation DTR using the Sepsis3 data from the Medical Information Mart for Intensive Care
version III (MIMIC-III) database \citep{Johnson2016MIMICIII}. It is an openly available  dataset comprised of
patients admitted to the Beth Israel Deaconess Medical Center.  Following  \cite{Raghu}, we use arterial lactate levels  to assess the effectiveness of IV-fluid resuscitation policies  since its elevated levels indicate worsening of sepsis.
Since the true value function of any policy/DTR is unknown,  we will estimate it using the doubly robust Augmented Inverse Probability Weighted (AIPW) estimator of \cite{kallus2020double}  \citep[][]{Jiang2016, Thomas2016}.

\paragraph*{Study Cohort and Data Description}
 We consider policies over two stages. 
We extracted a cohort of adult ICU patients from MIMIC-III, restricting the analysis to individuals with recorded IV fluid administration and relevant clinical measurements at two critical time points: admission (Stage 1) and four hours post-admission (Stage 2). The data was cleaned and pre-processed according to \cite{alistair_johnson_2018_1256723}.   
 We discretized the rate of IV fluid transfusion (in mL/Kg body weight) to apply our methods, which is a common practice in reinforcement learning \citep{Raghu,Komorowski2018}. 
We considered three treatment levels: no IV (no fluid administered), moderate (0-100 mL/Kg body weight), and high ($>$ 100 ml/Kg body weight). Therefore, $k_1=k_2=3$ in both stages.  
 Since lower values of lactate levels are associated with effective sepsis management, we transformed the lactate level $L_{\text{arterial}}$  as \(
\tilde{L}_{\text{arterial}} = 4 \, (22 - L_{\text{arterial}})
\) to use it as the reward or outcome. This transformation was selected because the original arterial lactate values range from 0.2 to 21.06. This transformation ensured a positive outcome variable with a broader range.

The covariates include diverse physiological measurements and laboratory values. 
A total of 98.42\% of patients had at least one missing value in their trajectory, with 27.95\% missing more than 50\% of their trajectory. 
Covariates  with more than 70\% missingness, e.g., temperature and PaO\(_2\)/FiO\(_2\) ratio, were excluded. 
Similar to prior works such as \cite{Raghu,johnson2018comparative} that used the Sepsis3 repository,  we imputed the remaining  missing values  using K-nearest neighbors (KNN) imputation with $K$=\textcolor{red}{3}.  Then, we estimated the stagewise propensity scores $\pi_t(A_t\mid H_t)$ of each patient  using multinomial logistic regression, and removed all patients with  propensity score $<$ 0.15 at any stage,  leading  to the exclusion of about 5\% patients. Clipping the observations with extreme probability estimates is  recommended for improving the prediction of logistic regression \citep{hirano2003efficient}. In our case,  this step also helps stabilize the inverse propensity scores, which are required not only during the training of SDSS and ACWL, but also at the testing phase, for estimating the value function using \cite{kallus2020double}'s AIPW estimator. 
After data pre-processing, we obtained a final sample of 16551 patients with complete data on 45 and 42 covariates in the first and second stages, respectively. The complete covariate list is provided in Supplement \ref{sec: list of covariates}. In the final data, percentages of patients who got no IV, moderate, and high dose of IV in stage 1 were  65.23\%, 12.80\%, and  21.97\%, respectively. The 
 corresponding percentages in stage 2 were 
37.51\%, 19.59\%, and  42.90\%.

\paragraph*{Training and testing}
We reshuffled the dataset 100 times and, after each shuffle, split the data into training and testing sets with a 50-50 ratio, resulting in 100 replications. 
In each replication, the training subset was used for training the policies following the procedures in Section \ref{sec: simulation}, and the test set was used exclusively for  estimating the value function of the trained policies. The hyperparameters for SDSS and Q-learning are provided in Supplement \ref{appendix:applicationParameters}. During training, the maximum number of random restarts was set to  six in all replications. 
 When $T=2$ and $k_1=k_2=3$, \cite{kallus2020double}'s AIPW estimator of the value function takes the form
 \begin{align}
 \label{def: AIPW}
     \widehat{V}_{AIPW}(d) = &\ \mathbb{P}_n\left[\widehat{Q}^d_1(H_1, d_1(H_1))\right] \nn\\
     &\ + \mathbb{P}_n \left[   \frac{1[A_1 = d_1(H_1)]}{\widehat\pi_1(A_1\mid H_1)} \left(Y_1 - \widehat{Q}^d_1(H_1, A_1) + \widehat{Q}_2(H_{2}, d_{2}(H_{2}))\right) \right]\nn\\
     &\ +  \mathbb{P}_n \left[  \prod_{i=1}^2 \frac{1[A_i = d_i(H_i)]}{\widehat\pi_2(A_i \mid H_i)} \left(Y_2 - \widehat{Q}_2(H_2, A_2) \right) \right],
 \end{align}
where  $\widehat\pi_1(A_1\mid H_1)$ and  $\widehat\pi_2(A_2\mid H_2)$ are the propensity score estimators, \(\mathbb{P}_n\) is the empirical distribution corresponding to the test data, and for any DTR $d$ and $i,j\in[3]$, the functions
$\widehat Q_2(H_2, i)$ and $\widehat Q_1^d(H_1, j)$ are the estimators of $Q_2(H_2, i)=\E[Y_2\mid H_2, A_2=i]$ and 
\[  \quad Q^{d}_1(H_1, j)=\E\slbt Y_1+ Q_2(H_2, d_2(H_2))\ \mid\ H_1,A_1=j\srbt,\]
respectively. The functions $Q_2$ and $Q_1^d$ are referred to as the Q-functions under the DTR $d$ \citep{sutton2018}. The above Q-functions were estimated nonparametrically using a 3-layer feed-forward neural network on the test data, whose details can be found in Supplementary Tables \ref{tab: app: validation network} and \ref{tab: common parameters apps}. The value function estimates of  different methods were stable for moderate variations of the parameters of this 3-layer neural network. For $\widehat\pi_1(A_1 \mid H_1)$ and $\widehat\pi_2(A_2 \mid H_2)$, we used probability estimates obtained from the  multinomial logistic regression fitted on the full data.

  Figure \ref{fig:application} provides the boxplots of the value function estimates obtained from the 100 replications. As in Section \ref{sec: simulation}, ACWL was excluded from  linear policy computations.  The Python code for the data application is available in the GitHub repository cited as \cite{sdss-app}.

\begin{figure}[!htb]
    \centering
    \includegraphics[width=\textwidth]{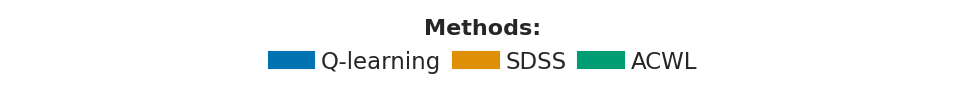}\\[1ex]
    \begin{subfigure}[t]{0.49\textwidth}
        \centering
        \includegraphics[width=\textwidth]{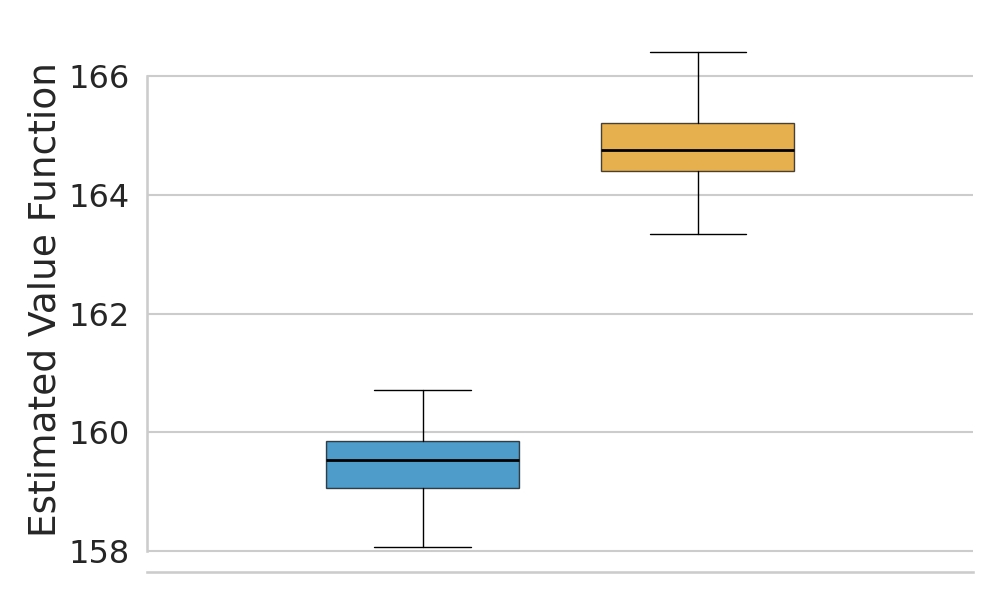}
        \caption{Linear policies}
    \end{subfigure}
    \hfill
    \begin{subfigure}[t]{0.49\textwidth}
        \centering
        \includegraphics[width=\textwidth]{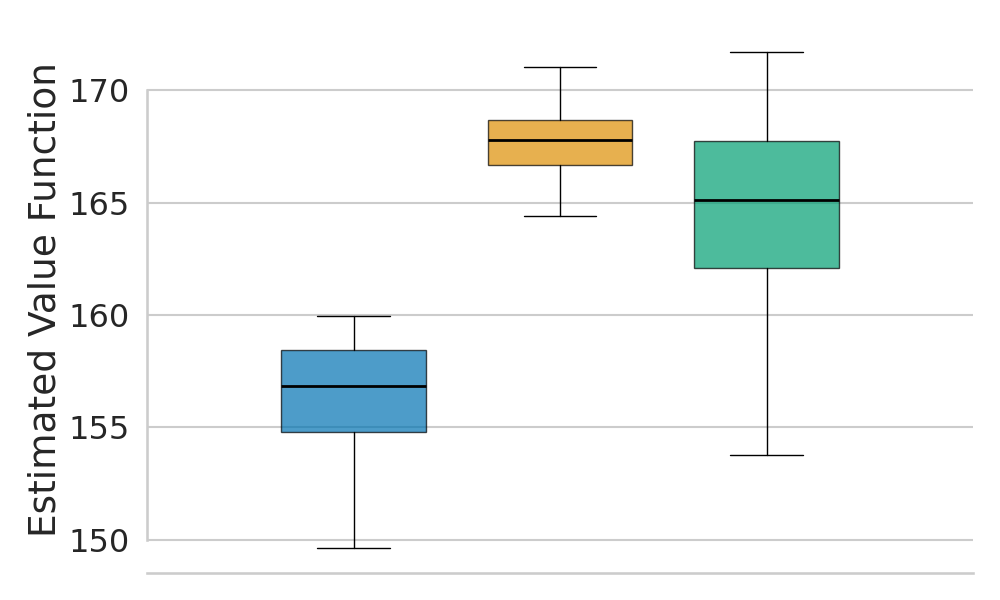}
        \caption{Nonlinear policies}
    \end{subfigure}
    \caption{{\bf Box plots of estimated value functions for the sepsis data.}  The y-axis shows AIPW estimates of the value function across 100 replications generated by reshuffling the data as described in Section \ref{sec: application}. Each box extends from the 25th to the 75th percentile, with the median indicated by the central line. 
     The whiskers extend up to 1.5 times the interquartile range. }
    \label{fig:application}
\end{figure}

\paragraph*{Results}

 In Figure \ref{fig:application}, non-linear SDSS exhibits higher  value function estimates with lower variability than the non-linear comparators.  SDSS also outperforms Q-learning when using linear policies. These findings aligns with our simulations   in Section \ref{sec: simulation},  where we saw that SDSS can be a competitive method for  high-dimensional data. Among Q-learning, SDSS, and ACWL, Q-learning exhibits the lowest value function estimates, which also aligh with our simulations  in Section \ref{sec: simulation}. 
 It appears that the policy estimated by Q-learning is more conservative than those by ACWL and SDSS. To illustrate this, we examine the SOFA (Sequential Organ Failure Assessment) score, which quantifies organ dysfunction in sepsis \citep{do2023sequential}. Encoding treatment levels as 1 (no IV), 2 (moderate), and 3 (high), Figure~\ref{fig:sofa_treatment_escalation} shows that, under non-linear policies, the average treatment level of Q-learning increases rather slowly with SOFA score, when compared with ACWL or SDSS. 
Figure \ref{fig:sofa_treatment_escalation} also shows that the SDSS policy exhibits the steepest escalation in treatment intensity as severity increases, aligning  with clinical expectations. The ACWL policy also increases treatment intensity with severity.

 While the value function estimate for non-linear policies is slightly higher than that of the linear policies for SDSS, the difference is modest. Linear SDSS exhibits slightly lower variability, as it involves optimization over fewer parameters. This observation suggests that in real-world applications, where interpretable policies are desirable but model misspecification is a concern, SDSS with a linear policy may be  promising.
\begin{figure}[!htb]
    \centering
    \includegraphics[width=\textwidth]{Plots/legendT.png}\\[1ex]
    \begin{subfigure}[t]{0.49\textwidth}
        \centering
        \includegraphics[height=1.5in]{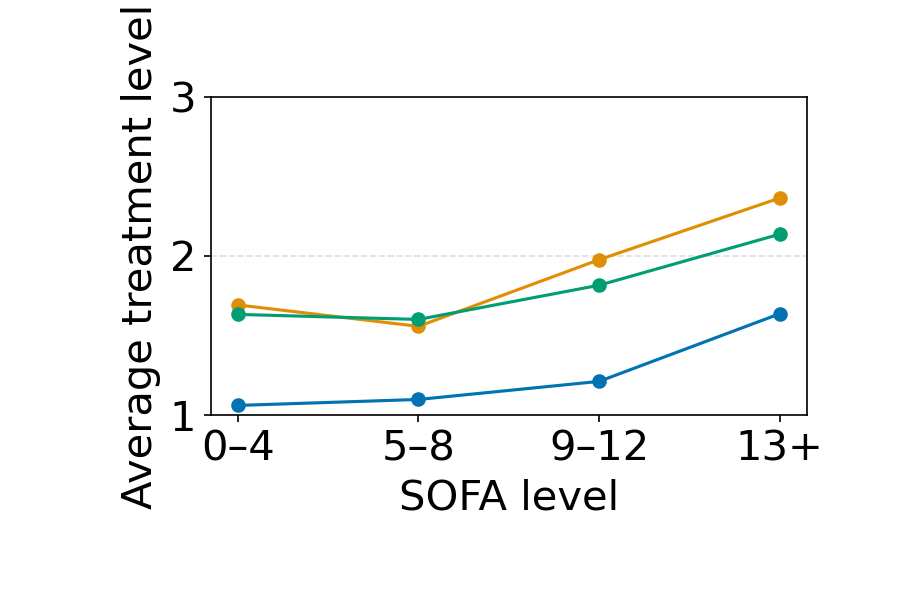}
        \caption{Stage 1}
    \end{subfigure}
    \hfill
    \begin{subfigure}[t]{0.49\textwidth}
        \centering
        \includegraphics[ height=1.5in]{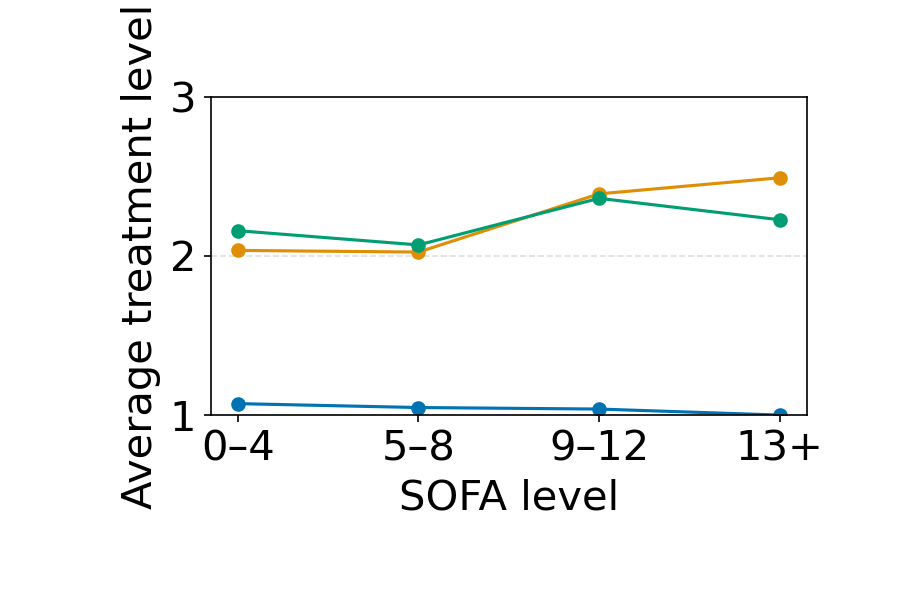}
        \caption{Stage 2}
    \end{subfigure}
    \caption{{\bf Plot of mean treatment level vs SOFA scores  in our data application for non-linear policies.}  
    We partition the SOFA score into four clinically meaningful bins, 0–4, 5–8, 9–12, and $\geq$13, corresponding to mild, moderate, severe, and critical organ dysfunction.  Treatment levels are encoded as 1 (no IV), 2 (moderate), and 3 (high). For each method and SOFA bin, the average treatment level, displayed on the Y axis, is computed using  the training set. Although SOFA scores range from 0 to 24, values above 13 were grouped due to their low frequency in the dataset.   }
    \label{fig:sofa_treatment_escalation}
\end{figure}

\section{Comparison with related literature}
\label{sec: related}
\subsection{Binary treatment setting} 
\label{sec: related: laha 2021}
Considerable work has been done on direct search methods for binary treatments, with extensions to specialized settings such as constrained and censored settings \citep{zhao2015, jiang2019entropy, liu2018augmented, zhao2020constructing, liu2024learning}. The key differences between the binary-treatment and the general DTR setting have already been discussed in Sections \ref{sec: introduction} and \ref{sec: set-up}. Here, we compare our work to \cite{Laha2024}, which is the most closely related to the current paper among works on the binary treatment setting.

\paragraph*{Comparison with \cite{Laha2024}} In \cite{Laha2024}, we characterized a class of margin-based, Fisher consistent surrogates for the sequential DTR classification in the binary treatment case, and provided sharp regret bounds for the resulting direct search method. 
In that paper, we also established the (Fisher)  inconsistency of margin-based, smooth, concave surrogates. Below, we highlight the major differences between \cite{Laha2024} and the present work.\\

 [1.]  The first major difference arises from the distinction between binary and multiclass classification, whose Fisher consistency literatures differ significantly. As discussed in Sections~\ref{sec: set-up} and~\ref{sec: convex loss}, binary classification—and by extension, binary-treatment DTRs— heavily exploit the embedding of the classes/treatments into $\{\pm 1\}$. As noted by \cite{zhou2022}, the absence of this structure complicates multiclass classification, and, consequently,  general DTR $ (k_t \geq 2)$ classification. As previously noted, the unavailability of margin-based losses in the general setting is a result of this.  More importantly,    
 tools for studying Fisher consistency in binary classification do not directly extend to multiclass settings, necessitating development of new techniques \citep{tewari2007,ramaswamy2016convex,neykov2016}. Likewise, while the insights from \cite{Laha2024} helped guide this work, their tools and methods were not directly applicable to the general setting.
     Nonetheless, we obtain regret decay rates  similar to \cite{Laha2024}'s, which indicates  that the general DTR problem is not inherently harder than the binary-treatment DTR case.

[2.]  Given the apparent failure of concave surrogates in the binary treatment case, \cite{Laha2024} speculated that DTR Fisher consistency might require stricter conditions on surrogate losses than in binary or multiclass classification, where concave, Fisher consistent surrogates exist. However, \cite{Laha2024} primarily focused on characterizing a class of Fisher consistent surrogates and did not fully explore necessary and sufficient conditions for Fisher consistency.
    Building on the initial findings of \cite{Laha2024}, the present paper shifts focus toward a deeper theoretical investigation of Fisher consistency by developing precise necessary and sufficient conditions. In Section~\ref{sec: necessity}, we show that for separable surrogates, DTR Fisher consistency indeed imposes strict restrictions on the image set of the surrogates. Moreover, \cite{Laha2024}'s surrogate losses satisfy these restrictions. In fact, they are a special case of our product-based surrogate losses (see Section~\ref{sec: relative margin based}).

 [3.] \cite{Laha2024} focused on the two-stage case ($T = 2$), whereas we address the general $T$-stage setting, which makes our theoretical analysis more technically challenging.

\subsection{ Direct search for general ($k_t\geq 2$) setting}
\label{sec: lit: existing research on kt geq 2}
As noted earlier, research on direct search in the general setting has been scarcer than in the binary treatment setting. The tree-based methods proposed by \cite{tao2017adaptive,tao2018tree} are among the few works in this category whose setting aligns with ours. Both methods follow the stagewise approach, which we have already extensively discussed  in Sections \ref{sec: introduction}, \ref{sec: simulation}, and \ref{sec: application}. Here, we instead focus on \cite{xue2022multicategory}, as it is, to the best of our knowledge, the only other DTR direct search paper adopting a simultaneous optimization framework. However, there are significant differences between \cite{xue2022multicategory} and our work.

\paragraph*{Comparison with \cite{xue2022multicategory}}
First,   \cite{xue2022multicategory}'s setting  and goal substantially differs from ours. Their focus is on censored survival data, with the goal of maximizing the conditional survival function of patients under censoring. While they demonstrate Fisher consistency of  their surrogate losses under conditions, exploring the theory of Fisher consistency was not a goal of that paper. 
Because the survival function is a probability, it factorizes into a product of stage-specific terms.
Taking the logarithm of this product yields a sum of $T$ terms, each corresponding to a different stage's treatment assignment.
This tensorization makes the discontinuous loss function additive in the $d_t$'s (the stage-specific treatment assignments). The clever work of \cite{xue2022multicategory} exploits this additive structure to gain a computational advantage. In contrast, our value function, which is the expected sum of potential rewards, does not tensorize.
Nevertheless, despite the advantage of tensorization,  \cite{xue2022multicategory}  requires non-convex surrogates, as we do. To this end, they adopt the angle-based $\pred$ function mentioned in Remark \ref{remark: angle-based framework}.

Second, the Fisher consistency results in \cite{xue2022multicategory} require the set of distributions $\mP$ (as denoted in our paper) to satisfy a set of inequalities involving discontinuous functions of the conditional survival probability (see their Condition 2). According to the authors, these restrictions ensure that receiving the optimal treatment increases the targeted survival probability at each stage. However, it is currently unknown whether these conditions are necessary, or whether they are verifiable under more standard assumptions of the DTR literature, as those in  \cite{zhao2015,murphy2005,kosorok2019,Robins2004}. These assumptions have no clear analogue in our setting and are therefore not directly comparable to ours.

 
\section{Discussion}
\label{sec: discussion}
In this paper, we have explored the limitations and possibilities of simultaneous direct search for finding the optimal DTR. When both the number of treatments and stages are arbitrary,  
simultaneous direct search reduces into an   optimization problem with a complex,  discontinuous loss.   
We begin by proving some negative results on the potential convexification of this problem through concave, Fisher consistent surrogates. Then, focusing on the class of separable surrogates, we show  that Fisher consistency  imposes restrictive geometric conditions on the 
 surrogates when $T\geq 2$. However, when $T=1$, the necessary conditions resemble those of multiclass classification. These  findings underscore  that, on top of the possible inconsistency of concave surrogates, the overall class of Fisher-consistent surrogates may be substantially narrower in multi-stage settings than in the single-stage (ITR) setting. 

Taken together, our findings suggest that computational challenges, particularly those stemming from non-convexity, may be an unavoidable cost of Fisher consistency in   simultaneous direct search.
 However, we demonstrate  existence of smooth, Fisher consistent  surrogate losses, thereby showing that the discontinuous optimization problem can be smoothed.  Building on this, we develop SDSS, a method for the surrogate optimization, that  exploits gradient-based optimizers to enable fast implementation. 
Our numerical experiments and real-data analysis suggest that, despite the higher computational burden, SDSS, when combined with ADAM, random restarts, and minibatching, may offer advantages in settings where modeling is difficult, provided the sample size is sufficiently large.

SDSS is the initial step towards DTR learning  with simultaneous optimization. There remains substantial room for improvement. Several modifications may further improve its performance, some of which are discussed below.

\paragraph*{\underline{ Tailored initialization}} We used He initialization for the non-convex optimization. Developing tailored initialization methods is a promising direction for improvement because  non-convex direct search has shown improved performance with customized initializations \citep{xue2022multicategory, liu2024learning}.

\paragraph*{\underline{ Location transformation}} We applied a location transformation to the $Y_t$’s when they were negative. However, this transformation can affect practical performance. It may be possible to avoid it by modifying our surrogates using the strategies in \cite{zhao2019efficient,liu2021outcome}.

\paragraph*{\underline{ Unknown propensity scores}} Our surrogate framework is based on an IPW estimator of value function, which can be unstable when the propensity scores need to be estimated   \citep{kosorok2019}. When propensity scores are unknown, the surrogate framework should ideally be based on the doubly robust estimator of the value function, as in \eqref{def: AIPW}. This presents an interesting avenue of future research. However, since the discontinuous loss in \eqref{def: AIPW} differs from $\psi_{\text{dis}}$ and involves Q-functions, it would alter the conditions required for Fisher consistency.






\begin{table}[!htb]
\caption{List of important notation}
\label{tab: list of notation}
\centering
\begin{tabular}{|cl|}
\hline
\textbf{Symbol} & \textbf{Description} \\
\hline
$t$ & Shorthand for stage\\
$T$ & Horizon or the total number of stages \\
($O_t$, $A_t$, $Y_t$) & covariate, treatment, and outcome triplet at stage $t$\\
$H_t$ & The history at stage $t$\\
$\D_i$ & Trajectory of the $i$-th patient, defined in \eqref{def: trajectory}\\
$k_t$ & Total number of treatments on stage $t$\\
$\K$ & the sum of all $k_t$'s; i.e., $\K=\sum_{t\in[T]}k_t$\\
$q$ & Dimension of $H_T$ \\
$\F$ & The class of all scores $f$\\
$\mP_0$ & Class of distributions satisfying Assumptions I-V\\
$d$ & A treatment policy or DTR\\
$d^*$ & The optimal DTR, need not be unique\\
$g$ & Relative score functions \\
$\pred$ & An operator defined as  $\pred(\mx)=\max(\argmax(\mx))$ for any vector $\mx$\\
$\trans$ & An operator, defined as $\trans(\mx)=(0,-\mx)$ for any vector $\mx$\\
$\psi$ & $T$-stage surrogate\\
$\phi$ & Single-stage surrogate, also denoted by $\phi_t$ when it varies with stage\\
$K,\KK$ & \makecell[l]{Kernels, where $\KK$ is generally a univariate\\ density  and $K$ is a multivariate density} \\
$V(f)$ & Value function of the score $f$\\
$V^\psi(f)$ & Surrogate value function of the score $f$, defined in \eqref{def:  surrogate V(f)}\\
$\tilde f$, $\tilde d$ & $\tilde f$ is the maximizer of $V^\psi(f)$ over $f\in\F$ and $\tilde d=\pred(\tilde f)$, both depend on $\psi$\\
$\widehat V^\psi(f)$ & Empirical surrogate value function of the score $f$, defined in \eqref{def:  surrogate V(f) estimated}\\
$\widehat{V}^{\psi,\text{rel}}(g)$ & \makecell[l]{Surrogate value function in the relative margin-based form,\\ defined as $\widehat{V}^{\psi,\text{rel}}(g) = \widehat{V}^\psi(\trans(g))$ for any relative score $g$}\\
$\L$ & A function related to  $\widehat{V}^{\psi,\text{rel}}$, see \eqref{relation: relative margin}\\
$Q^*$ & Optimal Q-function (state-action value function) \\
$Q^d$ & Q-function under policy $d$\\
$\mathcal{U}_n$ & Policy class, e.g., class of relative class scores $g$\\
$\mathcal{L}$ & Restricted classes of policies \\
$\mathcal{F}(\N, W, s)$ & ReLU network with depth $\N$, width $W$, and sparsity $s$\\
$\Gamma$ & Template of a relative margin-based surrogate (see Definition \ref{def: relative margin})\\
$\theta^{(r)}$ & The r-th parameter update in SDSS (see Algorithm \ref{alg: SDSS})\\
$\Opn$ & Optimization error\\
$\beta$ & Smoothness parameter for \Holder\ classes \\
\hline
\end{tabular}
\end{table}

\section{Acknowledgments}
 Nilanjana Laha and Nilson Chapagain's research was partially supported by National Science Foundation  grant DMS-2311098.  

\section*{Supplementary Material}
The supplementary material contains additional discussions, additional details on the tuning parameters used in our empirical study in Section \ref{sec: empirical}, and the proofs of the theorems and lemmas.

\textbf{Code:} The Python code for the simulations and data applications are  made available on Github \citep{sdss-simu,sdss-app}.
\FloatBarrier

\begin{center}
    \LARGE\textbf{Supplementary Material}
\end{center}



\setcounter{section}{0}
\renewcommand{\thesection}{S\arabic{section}}

\vspace{2em}
The supplement is organized as follows. Section~\ref{sec: additional discussion} provides additional discussion on various conditions related to our theoretical results, and  additional details on the simulations in Section~\ref{sec: simulation} and the data application in Section~\ref{sec: application}. Section~\ref{supp: simulation} presents tables listing the tuning parameters used for SDSS and Q-learning in Sections~\ref{sec: simulation} and~\ref{sec: application}. The remainder of the supplement contains the proofs.

\section{Additional discussion}
\label{sec: additional discussion}
 \subsection{Further discussion on the conditions in Theorem \ref{theorem: CC}}
  \label{sec: cond of Theorem cc}
  The  gamma-phi losses in Example \ref{sec: gamma-phi losses}  satisfy the boundedness condition if $\tpho$ is bounded above, and satisfy the domain condition if $\tpho$ and $\tphi$ are real-valued.  The smoothness condition  holds if both $\tphi$ and $\tpho$ are smooth. The pairwise loss in \eqref{def: psi pairwise} satisfies the conditions of Theorem~\ref{theorem: CC} if $-\tph$ is strictly convex, bounded below, closed, proper, thrice continuously differentiable on the interior of its domain, and $(0,0) \in \iint(\dom(-\tph))$. Similarly, the pairwise loss in \eqref{def: psi pairwise 2} satisfies all the conditions of Theorem~\ref{theorem: CC} if $-\tph_1$ and $-\tph_2$ are strictly convex, bounded below, closed, proper, thrice continuously differentiable on the interior of their respective domains, and $0 \in \iint(\dom(-\tph_1)) \cap \iint(\dom(-\tph_2))$. The exponential loss in Example~\ref{ex: Example 1}, as well as the concave versions of the cross-entropy and coherence losses in Example~\ref{sec: gamma-phi losses}, satisfy all conditions of Theorem~\ref{theorem: CC}.
\subsection{Further discussion on the constant $C_*$ in Proposition \ref{prop: multi-cat FC}} 
\label{sec: supple: discussion on C star}

Since the regret is bounded by $(V^\psi_* - V^\psi(f))/C_*$, a larger value of $C_*$ is preferable. This naturally raises the question: how large can $C_*$ be?
If $\phi_t$ is scaled by a constant $c > 0$,  $\inf_{\mx \in \RR^{k_t}} \phi_t(\mx; i)$ also scales by $c$, implying that $\J_t$ becomes $c\J_t$. From \eqref{def of mathcal C psi}, however, it follows that such scaling does not affect the positive fraction $\CC_{\phi_t}$. Consequently, $C_*$ scales linearly with $\phi_t$.
However, so does the $\psi$-regret term $V^\psi_* - V^\psi(f)$. For instance, multiplying $\phi_1$ by 2 results in both $C_*$ and the $\psi$-regret doubling, leaving the overall bound on the regret $V_*- V(f)$ in \eqref{instatement: sufficiency: lower bound linear} unchanged. Thus, inflating $\phi_t$ by a scaling factor does not change the rate of regret decay. Therefore, let us assume that the $\phi_t$'s are uniformly bounded  by one. In this case, the $\J_t$'s are also bounded by one. Since $\CC_{\phi_t} \leq 1$ by definition, it follows that $C_* \leq 1$ in this case. Therefore, although it may be possible for the true regret to be smaller than the $\psi$-regret, the bound in \eqref{instatement: sufficiency: lower bound linear} offers no such guarantee.

A few additional remarks on $C_*$ are in order.
  \textbf{(a) The role of $\CC_{\phi_t}$:} Large values of $\CC_{\phi_t}$ leads to large values of $C_*$. For a single-stage surrogate $\phi$, $\CC_{\phi}$ is large when disagreement between $\pred(\mx)$ and $\pred(\mp)$ causes a large gap between $\Psi^*(\mp)$ and  $\Psi(\mx; \mp)$. Such surrogates are likely to ensure that the surrogate value function of a suboptimal DTR is subatantially smaller than that of the optimal one. 
  \textbf{(d) Role of $\J_t$:} $C_*$ increases linearly with $\J_t$, hinting that surrogate losses for which  $\phi_t(\mx,\pred(\mx))$ is  larger compared to other  $\phi_t(\mx;i)$'s  ($i\neq \pred(\mx)$) may yield lower regret. The extreme case is $\phi_t=\phi_{dis}$, in which case $\phi(\mx;\pred(\mx))=1$  and all other $\phi(\mx;i)$'s are zero. Therefore, surrogate losses closer to $\phi_{dis}$ in geometry are expected to yield larger values of $C_*$. 
 
\subsection{Connection between our surrogates and the vanishing gradient problem}
\label{supp: vanishing gradient}
Figure \ref{fig:value function for toy data} shows that $\hVr$  exhibits horizontal asymptotes and plateau regions in the toy example from Section~\ref{sec: implementation}. These features, which lead to the vanishing gradient problem,  arise from $\Gamma$, since, as shown in \eqref{opti: value fn: toy data}, $\hVr$ is a positive linear combination of $\Gamma$ functions.
Figure \ref{Plot: Gamma function} shows that the maximum of $\Gamma(x,y;1)$ is not attained in $\RR^2$, and $\Gamma(x,y;1)$  exhibits horizontal asymptotes and  plateau regions for both the product-based and kernel-based surrogates.  The plots for $\Gamma(x,y;2)$ and  $\Gamma(x,y;3)$ are similar except the bumps (optimal plateau regions) are in the second and fourth quadrants, respectively. Although Figure \ref{Plot: Gamma function} uses a specific $\tau$ for the product-based surrogate, we verified that the surface plots have similar properties for the other examples of $\tau$ mentioned in Section \ref{sec: product-type surrogate loss}.

One may wonder whether an alternative surrogate could eliminate plateau-related issues entirely. If we could construct a $\phi$ satisfying Conditions~\ref{assump: N1} and~\ref{assump: N2} that also attains unique maximum in $\RR^k$, the issues could be at least partially eradicated. However, all Fisher consistent surrogates we have been able to construct that satisfy Conditions~\ref{assump: N1} and~\ref{assump: N2} also satisfy Condition~\ref{assump: N3}, and are variants of the kernel-based or product-based forms. 
 All these surrogates  lead to $\Gamma$'s similar to  Figure \ref{Plot: Gamma function}, i.e., they lack strong concavity, the optimum is attained at infinity, and  plateau regions exist. 
Therefore, the value functions resulting from these surrogates  have suboptimal plateau regions and the  vanishing gradient problem  persists. 
\subsection{More details on He and Xavier initialization}
\label{supp: he and xavier}
 When the relative class scores $g_{ti}(\cdot; \theta_{ti})$ are linear, He initialization samples the elements of the initial parameter, i.e.,  $\theta_{ti}^{(0)}$, independently from a centered Gaussian distribution with variance $2/n_f$, where $n_f$ is the dimension of the feature space. 
When  the relative class scores are deep neural network,   $\theta_{ti}$ consists of the weights and bias parameters of the network. Under He initialization, bias parameters in each layer are set to zero, while the weights are sampled independently from a centered Gaussian distribution with variance $2/n_f$, where $n_f$ is the number of input units in that layer. The Xavier initialization follows a similar approach in both cases, but samples weights from $U[-1/\sqrt{n_f}, 1/\sqrt{n_f}]$ instead of a normal distribution. 

\subsection{More details on Scheme 2}
The stage-specific parameters of Scheme 2 outcome models are given by
\begin{gather}
   \myD^{(1)} = (2.5, 2.7, 2.6), 
\quad
  \myD^{(2)} = (2.6, 2.5, 2.7),\\
 \myE^{(1)}=(2.1, 2.0, 2.2),\quad \myE^{(2)}=(2.2,2.1,2.3). 
\end{gather}
The optimal treatment assignments are given by
\[
d_2^*(H_2) 
\;=\; 
\underset{i \in \{1,2,3\}}{\mathrm{arg\,max}}
\Bigl\{
  \slb\myD^{(2)}_{i}\,\cos\bigl(O_1^T\mo_p\bigr)\srb^2 
  \;+\;
  \myE^{(2)}_{i}\,\sin\bigl(O_1^T\mo_p\bigr)
\Bigr\},
\]

  Then the optimal policies are:
\[
d_2^*(H_2)
= 
\underset{i\in\{1,2,3\}}{\arg\max}\,
\Bigl\{
  \bigl[\myD^{(2)}_{i}\cos(O_1^T \mathbf1_p)\bigr]^2
  + \myE^{(2)}_{i}\sin(O_1^T \mathbf1_p)
\Bigr\},
\]
\[
d_1^*(H_1)
= 
\underset{i\in\{1,2,3\}}{\arg\max}\,
\Bigl\{
  \bigl[\myD^{(1)}_{i}\sin(O_1^T \mathbf1_p)\bigr]^2
  + \myE^{(1)}_{i}\cos(O_1^T \mathbf1_p)
 \Bigr\}.
\]

\subsection{Comparing regular SDSS with ensembled SDSS for Scheme 2}
\begin{figure}[H]
    \centering
    \includegraphics[width=\textwidth]{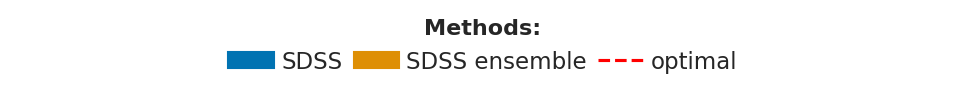}\\[1ex]    
    \begin{subfigure}[t]{0.49\textwidth}
        \centering
        \includegraphics[width=\textwidth]{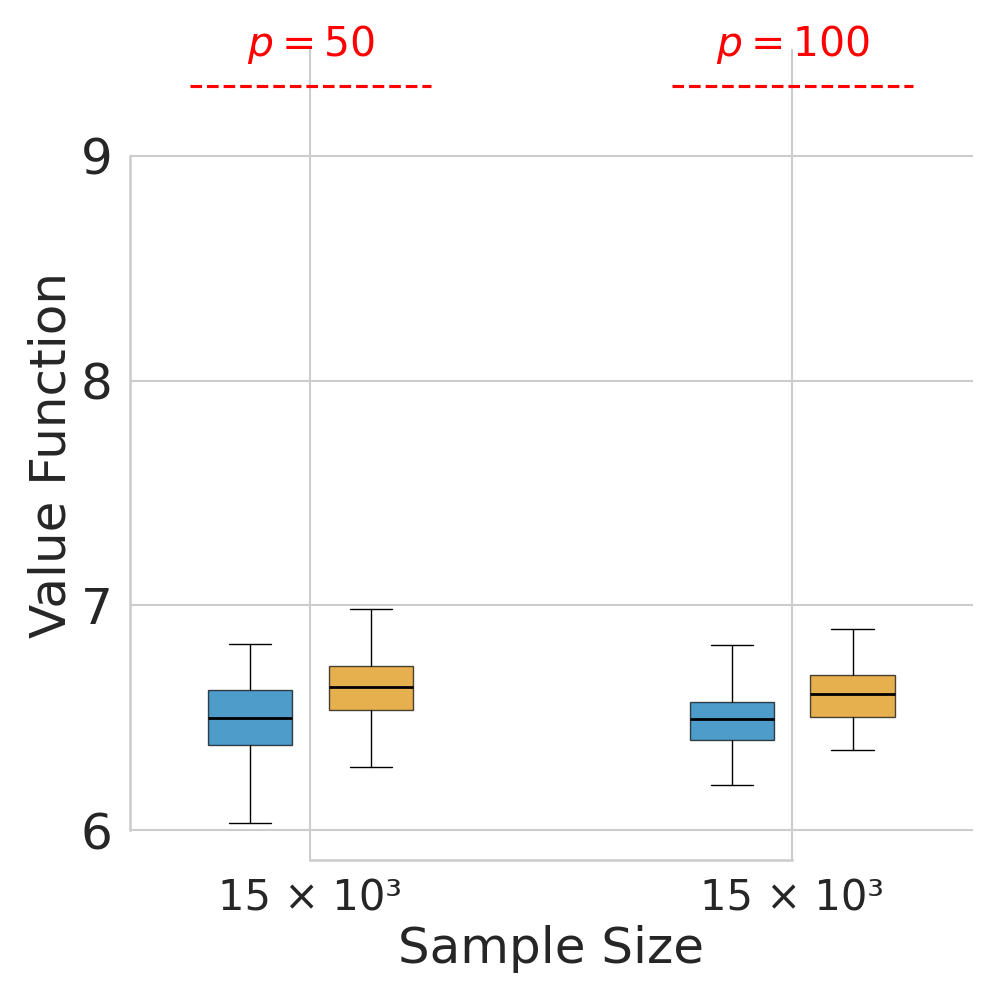}
        \caption{Linear policies.}
    \end{subfigure}
    \hfill
    \begin{subfigure}[t]{0.49\textwidth}
        \centering
        \includegraphics[width=\textwidth]{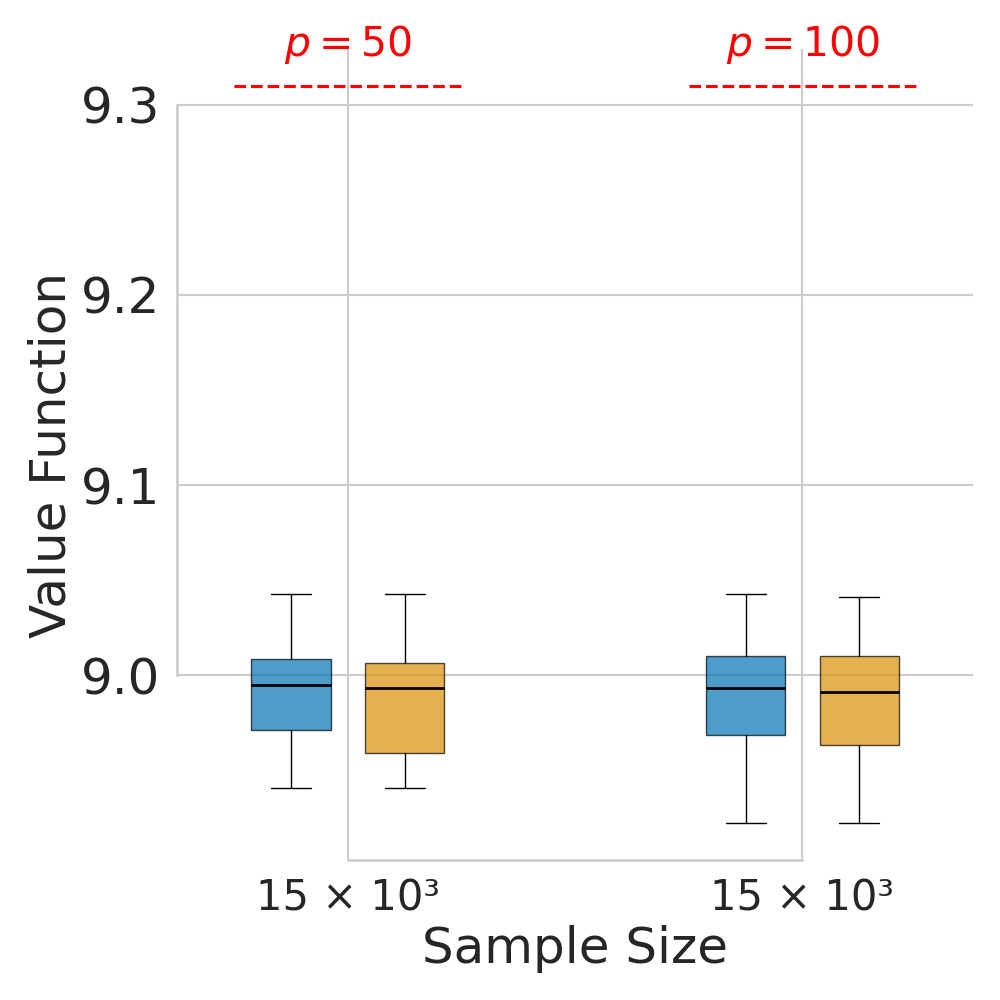}
        \caption{Nonlinear policies.}
    \end{subfigure}
    \caption{{\bf Value function of SDSS with and without ensembling under Scheme 2.} Ensembling is performed as described in Section \ref{sec: simulation}. The regular  SDSS (colored blue) uses $\tau(x) = 1 + x/\sqrt{1 + x^2}$. The left panel shows results for linear policies; the right panel shows results for nonlinear policies, both uses $p=50$ and $100$. The x-axis indicates the sample size and the y-axis represents the value function estimates from Monte Carlo 100 replications.  Each box spans the interquartile range (IQR) with edges at the 25th and 75th percentiles, a median line, and whiskers extending to 1.5\(\times\)IQR.  The red dashed line marks the optimal value. In all four cases, SDSS was implemented following the procedures in Section \ref{sec: simulation}. Ensembling slightly improves value function for linear policies and yields comparable value function for nonlinear policies, albeit with slightly higher variance. }
    \label{fig:setup2ensemble}
\end{figure}
\FloatBarrier

\subsection{The list of covariates for  sepsis data}
\label{sec: list of covariates}
After data pre-processing, we had 45 covariates in stage 1 and 42 covariates in stage 2. The  following 42 covariates were present in both stages: Mechanical ventilation status, maximum vasopressor dose, ICU readmission status, median vasopressor dose, Glasgow Coma Scale (GCS), heart rate (HR), systolic blood pressure (SysBP), mean arterial pressure (MeanBP), diastolic blood pressure (DiaBP), respiratory rate (RR), fraction of inspired oxygen (FiO2), potassium level, sodium level, chloride level, glucose level, magnesium level, calcium level, hemoglobin (Hb), white blood cell count (WBC), platelet count, partial thromboplastin time (PTT), prothrombin time (PT), arterial pH, arterial partial pressure of oxygen (paO2), arterial partial pressure of carbon dioxide (paCO2), arterial base excess (BE), serum bicarbonate (HCO3), arterial lactate, Sequential Organ Failure Assessment (SOFA) score, Systemic Inflammatory Response Syndrome (SIRS) score, shock index (HR/SBP), cumulative fluid balance, peripheral capillary oxygen saturation (SpO2), blood urea nitrogen (BUN), serum creatinine, serum glutamic oxaloacetic transaminase (SGOT/AST), serum glutamic pyruvic transaminase (SGPT/ALT), total bilirubin, international normalized ratio (INR), total fluid input, total fluid output, four-hourly fluid output. The extra three covariates in stage 1 were gender, age, and weight.


\section{Tuning parameters for Q-learning and SDSS}
\label{supp: simulation}
\subsection{Tuning parameters for the simulations in Section \ref{sec: simulation}}
\label{appendix:simulationParameters}
\subsubsection{Summary and definition of settings}
\begin{table}[!htb]
\centering
 \caption{Summary of Simulation schemes}
 \label{table: sim scheme summary}
\begin{tabular}{@{}llccc@{}}
\toprule
\makecell{\textbf{Simulation}\\{\bf Scheme}} & \makecell{\textbf{Policy}\\{\bf Type}} & \makecell{\textbf{Varied}\\{\bf Values}} & \makecell{\textbf{Sample}\\ {\bf Sizes}} & \textbf{Comparators} \\ \midrule
Scheme 1                 & Linear       & $\omega \in \{10,20,40\}$     & $\{5000,15000\}$     & SDSS, Q-learning         \\
Scheme 1                 & Non-linear   & $\omega \in \{10,20,40\}$     & $\{5000,15000\}$     & Q-learning, SDSS, ACWL   \\
Scheme 2                 & Linear       & $p \in \{50,100\}$            & $\{5000,15000\}$     & SDSS, Q-learning         \\
Scheme 2                 & Non-linear   & $p \in \{50,100\}$            & $\{5000,15000\}$     & Q-learning, SDSS, ACWL   \\ \midrule
\bottomrule
\end{tabular}
\label{tab:simulation_cases}
\end{table}

\begin{table}[!htb]
\centering
 \caption{Definition of the settings for the Simulation. These definition should be used in the rest of the tables in this section. }
 \label{tab: def simulation schemes}
\label{tab:settings-definition}
\begin{tabular}{ll}
\toprule
\textbf{Setting} & \textbf{Description} \\
\midrule
1 & Scheme 1, $\omega=10$, $n=5\,000$ \\
2 & Scheme 1, $\omega=10$, $n=15\,000$ \\
3 & Scheme 1, $\omega=20$, $n=5\,000$ \\
4 & Scheme 1, $\omega=20$, $n=15\,000$ \\
5 & Scheme 1, $\omega=40$, $n=5\,000$ \\
6 & Scheme 1, $\omega=40$, $n=15\,000$ \\
7 & Scheme 2,  $p=50$, $n=5\,000$  \\
8 & Scheme 2,  $p=50$, $n=15\,000$ \\
9 & Scheme 2,  $p=100$, $n=5\,000$ \\
10 & Scheme 2,  $p=100$, $n=15\,000$ \\
\bottomrule
\end{tabular}
\end{table}
 The common simulation parameters are provided in Table \ref{tab: common parameters sim}.
\begin{table}[!htb]
\centering
\caption{Tuning parameters that did not change across simulation settings. These values were used for both SDSS and Q-learning.}
\label{tab: common parameters sim}
\begin{tabular}{ll}
\toprule
Parameter & Value/Configuration \\
\midrule
Train/Validation Split & 70/30 \\
Weight Initialization & He initializer \\
$L_2$ Regularization (Weight Decay) Parameter & 0.001 \\
Reduction Factor ($R_F$) & 0.8 \\
Patience Threshold ($N_{\text{patience}}$) & 30 \\
Restart Threshold $(N_{\text{restart}})$ & 3\\
Evaluation Frequency ($r_{\text{eval}}$) & 2 \\
ADAM Parameter ($D_1$) & 0.9 \\
ADAM Parameter ($D_2$) & 0.999 \\
EMA Smoothing Parameter ($\varkappa$) & 0.8 \\
Numerical Constant ($\epsilon_{\text{num}}$) & $10^{-8}$ \\
Improvement Threshold ($\delta_{\text{imp}}$) & 0 \\
\bottomrule
\end{tabular}
\end{table}
\subsubsection{SDSS network parameters that varied across schemes}

\begin{table}[!htb]
\centering
\caption{Table of the neural network parameters for {\bf Linear SDSS} that varied across settings in our data simulations. Here hidden dim. corresponds to the number of units per hidden layer.}
\label{tab:sdss-linear}
\resizebox{\textwidth}{!}{%
\begin{tabular}{l*{10}{>{\large}c}}
\toprule
\large Parameter & Setting~1 & Setting~2 & Setting~3 & Setting~4 & Setting~5 & Setting~6 & Setting~7 & Setting~8 & Setting~9 & Setting~10 \\
\midrule
\large
\makecell[l]{Initial Learning\\ Rate $(\texttt{lr}_0 )$} & 0.07  & 0.07  & 0.07  & 0.07  & 0.07  & 0.07  & 0.1  & 0.1  & 0.2  & 0.2  \\
\large Batchsize ($n_{\text{mini}}$)                & 1000     & 1000     & 1000     & 1000      & 200      & 1000      & 800     & 800      & 1600     & 1600     \\
\large $N_{\text{epoch}}$ & 10  & 10  & 10  & 10  & 10  & 10  & 20  & 20  & 60  & 60  \\
\large Activation Function & None & None & None & None & None & None & None & None & None & None \\
\large Hidden Dim. Stage~1  & 10  & 10  & 20  & 20  & 20  & 20  & 20  & 20  & 10  & 20  \\
\large Hidden Dim. Stage~2  & 20  & 40  & 20  & 40  & 40  & 40  & 40  & 40  & 20  & 40  \\
\large Number of Layers & 1 & 1 & 1 & 1 & 1 & 1 & 1 & 1 & 1 & 1 \\
\bottomrule
\end{tabular}
}
\end{table}

\begin{table}[!htb]
\centering
\caption{Table of the neural network parameters for {\bf Non-linear SDSS} that varied across settings in our data simulations. Here hidden dim. corresponds to the number of units per hidden layer.}

\label{tab:sdss-nonlinear}
\resizebox{\textwidth}{!}{%
\begin{tabular}{l*{10}{c}}
\toprule
Parameter & Setting~1 & Setting~2 & Setting~3 & Setting~4 & Setting~5 & Setting~6 & Setting~7 & Setting~8 & Setting~9 & Setting~10 \\
\midrule
\makecell[l]{Initial Learning\\ Rate $(\texttt{lr}_0 )$} & 0.07  & 0.07  & 0.07  & 0.07  & 0.07  & 0.07  & 0.07  & 0.07  & 0.2  & 0.2  \\
\large Batchsize ($n_{\text{mini}}$)                & 5000     & 5000     & 5000     & 5000     & 5000     & 5000     & 5000     & 5000     & 1600     & 1600     \\
\large $N_{\text{epoch}}$ & 20  & 10  & 20  & 10  & 10  & 10  & 10  & 10  & 60  & 60  \\
Activation Function & ReLU  & ReLU  & ReLU  & ReLU  & ReLU  & ReLU  & eLU  & eLU  & eLU  & eLU  \\
Dropout Rate & 0 & 0 & 0 & 0 & 0 & 0 & 0.4 & 0.4 & 0.4 & 0.4 \\
Hidden Dim. Stage~1   & 40  & 20  & 20  & 20  & 40  & 40  & 40  & 40  & 40  & 40  \\
Hidden Dim. Stage~2   & 40  & 40  & 40  & 40  & 40  & 40  & 20  & 40  & 40  & 40  \\
Number of Layers & 1 & 1 & 1 & 1 & 1 & 1 & 3 & 3 & 3 & 3 \\
\bottomrule
\end{tabular}
}
\end{table}

\FloatBarrier


\subsubsection{Q-learning  parameters  that varied across schemes}
\label{supp: Q-learning simulation parameters}

\begin{table}[!htb]
\centering
\caption{Table of the neural network parameters for {\bf Non-linear Q-learning} that varied across settings in our data simulations. Here hidden dim. corresponds to the number of units per hidden layer.}

\label{tab:qnonlinear}
\resizebox{\textwidth}{!}{%
\begin{tabular}{l*{10}{c}}
\toprule
Parameter & Setting~1 & Setting~2 & Setting~3 & Setting~4 & Setting~5 & Setting~6 & Setting~7 & Setting~8 & Setting~9 & Setting~10 \\
\midrule
\makecell[l]{Initial Learning\\ Rate $(\texttt{lr}_0 )$} & 0.07  & 0.07  & 0.07  & 0.07  & 0.07  & 0.07  & 0.05  & 0.05  & 0.05  & 0.05  \\
\large Batchsize ($n_{\text{mini}}$)                 & 5000     & 5000     & 5000     & 5000     & 5000     & 5000     & 5000     & 5000     & 5000     & 5000     \\
\large $N_{\text{epoch}}$ & 20  & 10  & 20  & 10  & 10  & 10  & 10  & 10  & 20  & 10  \\
Activation Function & ReLU  & ReLU  & ReLU  & ReLU  & ReLU  & ReLU  & ReLU  & ReLU  & ReLU  & ReLU  \\
Dropout Rate & 0.1 & 0.1 & 0.1 & 0.1 & 0.1 & 0.1 & 0.1 & 0.1 & 0.1 & 0.1 \\
Hidden Dim. Stage~1   & 40  & 40  & 40  & 20  & 20  & 40  & 40  & 20  & 40  & 20  \\
Hidden Dim. Stage~2   & 20  & 40  & 40  & 40  & 40  & 40  & 40  & 40  & 40  & 40  \\
Number of Layers & 3 & 3 & 3 & 3 & 3 & 3 & 3 & 3 & 3 & 3 \\

\bottomrule
\end{tabular}
}
\end{table}

\FloatBarrier


\subsection{Tuning parameters for the sepsis data application}
\label{appendix:applicationParameters}

\begin{table}[!htb]
\centering
\caption{Tuning parameters that were common to all methods  in the application. These parameters did not vary across methods or policies 
unless otherwise specified on  Tables ~\ref{tab:appsLinear}--\ref{tab: app: validation network} 
}
\label{tab: common parameters apps}
\begin{tabular}{ll}
\toprule
Parameter & Value/Configuration \\
\midrule
Train/Validation Split & 70/30 \\
Weight Initialization & He initializer \\
$L_2$ Regularization (Weight Decay) Parameter & 0.001 \\
Reduction Factor ($R_F$) & 0.8 \\
Patience Threshold ($N_{\text{patience}}$) & 30 \\
Evaluation Frequency ($r_{\text{eval}}$) & 2 \\
ADAM Parameter ($D_1$) & 0.9 \\
ADAM Parameter ($D_2$) & 0.999 \\
EMA Smoothing Parameter ($\varkappa$) & 0.8 \\
Numerical Constant ($\epsilon_{\text{num}}$) & $10^{-8}$ \\
Improvement Threshold ($\delta_{\text{imp}}$) & 0 \\
\bottomrule
\end{tabular}
\end{table}

\FloatBarrier

\subsubsection{Parameters for SDSS in data application} 


\begin{table}[H]
\centering
\caption{Table of the neural network parameters for {\bf Linear SDSS} in our data application. Here hidden dimensions correspond to the number of units per hidden layer.}
\label{tab:appsLinear}
\begin{tabular}{ll}
\toprule
Parameter & Value/Configuration \\
\midrule
\makecell[l]{Initial Learning\\ Rate $(\texttt{lr}_0 )$}& 0.007 \\
\large $N_{\text{epoch}}$ & 120 \\
Train/Validation Split & 80/20 \\
Network Architecture & 3-layer network \\
Activation Function & None \\
Hidden Dimensions (Stage 1) & 40 \\
Hidden Dimensions (Stage 2) & 40 \\
\bottomrule
\end{tabular}
\end{table}

\begin{table}[!htb]
\centering
\caption{Table of the neural network parameters for {\bf Non-linear SDSS}  in our data application. Here hidden dimension corresponds to the number of units per hidden layer.}

\label{tab:appsNonLinear}
\begin{tabular}{ll}
\toprule
Parameter & Value/Configuration \\
\midrule
\makecell[l]{Initial Learning\\ Rate $(\texttt{lr}_0 )$} & 0.007 \\
\large $N_{\text{epoch}}$ & 120 \\
Train/Validation Split & 80/20 \\
Network Architecture & 3-layer network \\
Activation Function & ELU \\
Hidden Dimensions (Stage 1) & 40 \\
Hidden Dimensions (Stage 2) & 40 \\
Dropout Rate & 0.4 \\
\bottomrule
\end{tabular}
\end{table}
\FloatBarrier
\subsubsection{Parameters for Q-learning and Q-function estimation in data application}






\begin{table}[!htb]
\centering
\caption{Table of the neural network parameters for {\bf Non-linear Q-Learning} in our data application. Here hidden dimension corresponds to the number of units per hidden layer.}

\label{tab: non linear apps q learning}
\begin{tabular}{ll}
\toprule
Parameter & Value/Configuration \\
\midrule
Learning Rate & 0.07 \\
\large $N_{\text{epoch}}$ & 150 \\
Network Architecture & 1-layer network \\
Activation Function & ELU \\
Hidden Dimensions (Stage 1) & 4 \\
Hidden Dimensions (Stage 2) & 4 \\
Dropout Rate & 0 \\
\bottomrule
\end{tabular}
\end{table}



\begin{table}[!htb]
\centering
\caption{Table of neural network parameters for estimating \( \widehat{Q}^d_1 \text{ and } \widehat{Q}_2 \) during the computation of the AIPW value function estimator in our data application. Here hidden dimension corresponds to the number of units per hidden layer.}
\label{tab: app: validation network}
\begin{tabular}{ll}
\toprule
Parameter & Value/Configuration \\
\midrule
Learning Rate & 0.07 \\
Train/Validation Split & 70/30 \\
\large $N_{\text{epoch}}$ & 150 \\
Network Architecture & 3-layer network \\
Activation Function & ELU \\
Hidden Dimension (Stage 1) & 40 \\
Hidden Dimension (Stage 2) & 40 \\
Dropout Rate & 0.4 \\
\bottomrule
\end{tabular}
\end{table}
\FloatBarrier
\section{New notation for proofs}
\label{sec: notation app}
For two integers $m$ and $n$, $m \mod n$ denotes the remainder when $m$ is divided by $n$. For $m\leq n$, $[m:n]$ indicates the set $\{m,m+1,\ldots, n\}$.  
 We denote by $\mathbf e^{(i)}_k$ the unit vector of length $k$ with one in the $i$-th position and zero everywhere else. When $k$ is clear from the context, we will denote $\mathbf e^{(i)}_k$ only by $\mathbf e^{(i)}$.   $\mz_{m\times n}$ will denote a  matrix with $m$ rows and $n$ columns, whose all elements are zero. For two sets $A$ and $B$ in $\RR^k$, we define the Minkowski sum $A+B$ as the set $\{x+y:x\in A, y\in B\}$. For two real numbers $x$ snd $y$, we denote $\min(x,y)$ by $x\wedge y$ and $\max(x,y)$ by $x\vee y$.

 The order statistics of $x$ will be denoted by $(\mx_{[1]}, \mx_{[2]},\ldots,\mx_{[k]})$. 
We also introduce some shorthand. For any $t\in[T]$, we will denote the vector $(Q_t^*(H_t,i))_{i\in[k_t]}$ by $Q_t^*(H_t)$.  As mentioned earlier,  the $d^*$ defined by \eqref{def: d star argmax form} need not be unique. For ease of representation, unless otherwise mentioned, we will stick to a particular version of $d^*$: 
\begin{align}
    \label{def: d star from Q}
    d_t^*(H_t)=\max(\argmax_{a_t\in[k_t]}Q_t^*(H_t,a_t))\quad \text{ for all }t\in[T]. 
\end{align}
\section{Proofs for the convex surrogate  examples in Section \ref{sec: convex loss}}
\label{sec: proofs for convex surrogate loss example}
In the proofs for this section, $\tph$ will be used to denote an arbitrary function, whose definition can change from proof to proof.
 \subsubsection{Proof of Result \ref{result: ex: exp non-additive}}
\label{sec: proof of result ex: exp non-additive}
\begin{proof}[Proof of Result \ref{result: ex: exp non-additive}]
    Note that in this case $V^\psi(f)$ equals
\begin{align*}
    \MoveEqLeft \E\left[ -\sum\limits\limits_{i\in[k_1],j\in [k_2]}\frac{\exp\lb -\slb f_{1A_1}(H_1)-f_{1i}(H_1)+f_{2A_2}(H_2)-f_{2j}(H_2)\srb \rb(Y_1+Y_2)}{\pi_1(A_1\mid H_1)\pi_2(A_2\mid H_2)}\right]\\
    =&\ -\E\lbt \sum\limits\limits_{i\in[k_1],j\in [k_2]}\frac{e^{-\slb f_{1A_1}(H_1)-f_{1i}(H_1)+f_{2A_2}(H_2)-f_{2j}(H_2)\srb}\E[Y_1+Y_2\mid H_2, A_2]}{\pi_1(A_1\mid H_1)\pi_2(A_2\mid H_2)} \rbt\\
    =&\ \E\lbt-\sum\limits\limits_{a_2\in[k_2]} \lbs\sum\limits\limits_{i\in[k_1],j\in [k_2]}\exp\lb -\slb f_{1A_1}(H_1)-f_{1i}(H_1)+f_{2a_2}(H_2)-f_{2j}(H_2)\srb\rb\\
    &\ \times \frac{\E[Y_1+Y_2\mid H_2, A_2=a_2]}{\pi_1(A_1\mid H_1)} \rbs\rbt
\end{align*}
Therefore, for fixed $H_2$, if it exists, $\tilde f_2(H_2)$ is in  
\begin{align*}
\argmax_{\my\in\RR^{k_2}}\sum\limits\limits_{a_2\in[k_2]}\frac{\lbs\sum\limits\limits_{i\in[k_1],j\in [k_2]} -e^{ -f_{1A_1}(H_1)+f_{1i}(H_1)-\my_{a_2}+\my_j}\rbs\E[Y_1+Y_2\mid H_2, A_2=a_2]}{\pi_1(A_1\mid H_1)},
\end{align*}
provided  the above set is non-empty.
However,
\begin{align*}
   \MoveEqLeft  \sup_{\my\in\RR^{k_2}}\sum\limits\limits_{a_2\in[k_2]}\frac{\lbs\sum\limits\limits_{i\in[k_1],j\in [k_2]} -e^{ -f_{1A_1}(H_1)+f_{1i}(H_1)-\my_{a_2}+\my_j}\rbs\E[Y_1+Y_2\mid H_2, A_2=a_2]}{\pi_1(A_1\mid H_1)}\\
   =&\ -e^{-f_{1A_1}(H_1)}\inf_{\my\in\RR^{k_2}}\lb\sum\limits\limits_{a_2\in[k_2]} \frac{\E[Y_1+Y_2\mid H_2, A_2=a_2]}{\pi_1(A_1\mid H_1)}\sum\limits\limits_{i\in[k_1],j\in [k_2]}e^{f_{1i}(H_1)-\my_{a_2}+\my_j}\rb
   \end{align*}
which equals
\begin{align*}
\MoveEqLeft
 -e^{-f_{1A_1}(H_1)}\inf_{\my\in\RR^{k_2}}\lbs\sum\limits\limits_{a_2\in[k_2]} \frac{e^{-\my_{a_2}}\E[Y_1+Y_2\mid H_2, A_2=a_2]}{\pi_1(A_1\mid H_1)}\lb\sum\limits\limits_{j\in[k_2]}e^{\my_j} \sum\limits\limits_{i\in[k_1]}e^{f_{1i}(H_1)}\rb\rbs\\
   =&\ -\sum\limits\limits_{i\in[k_1]}e^{f_{1i}(H_1)-f_{1A_1}(H_1)} \inf_{\my\in\RR^{k_2}}\lbs\sum\limits\limits_{a_2\in[k_2]} \underbrace{\frac{\E[Y_1+Y_2\mid H_2, A_2=a_2]}{\pi_1(A_1\mid H_1)}}_{c_{a_2}(H_2)}\lb\sum\limits\limits_{j\in[k_2]}e^{\my_j-\my_{a_2}}\rb\rbs.
\end{align*}
Lemma \ref{lemma: exp opti} implies $\tilde f_2(H_2)$ uniquely exists and 
\[\argmax(\tilde f_2(H_2))=\argmax_{a_2\in[k_2]}\E[Y_1+Y_2\mid H_2,A_2=a_2].\]
Here Lemma \ref{lemma: exp opti} applies because $\E[Y_1+Y_2\mid H_2, A_2=a_2]>0$ by Assumption V, which implies  $c_{a_2}(H_2)$'s are positive real numbers. 
Since $Q_2^*(H_2,A_2)=\E[Y_1+Y_2\mid H_2,A_2]$ and $\tilde d_2=\pred(\tilde f_2)$, it follows that $\tilde d_2(H_2)\in\argmax_{a_2\in[k_2]}Q_2^*(H_2,a_2)$.

Lemma \ref{lemma: exp opti}  also implies that  $\sup_{f_2}V^\psi(f)$ equals
\begin{align*}
\MoveEqLeft -\E\lbt \sum\limits\limits_{i\in[k_1]}e^{f_{1i}(H_1)-f_{1A_1}(H_1)}\slb\sum\limits\limits_{j\in[k_2]}\sqrt{C_j(H_2)}\srb^2 \rbt\\
=&\ -\E\lbt \sum\limits\limits_{i\in[k_1]} \frac{e^{f_{1i}(H_1)}}{e^{f_{1A_1}(H_1)}}\E\lbt\slb\sum\limits\limits_{j\in[k_2]}\sqrt{C_j(H_2)}\srb^2\mid H_1,A_1\rbt \rbt\\
=&\ - \E\lbt\sum\limits\limits_{a_1\in [k_1]}\E\lbt\slb\sum\limits\limits_{j\in[k_2]}\sqrt{\E[Y_1+Y_2\mid H_2, A_2=j]}\srb^2\mid H_1,A_1=a_1\rbt \sum\limits\limits_{i\in[k_1]} \frac{e^{f_{1i}(H_1)}}{e^{f_{1A_1}(H_1)}}\rbt.
\end{align*}
Therefore, for fixed $H_1$, $\tilde f_1(H_1)$ exists if 
\begin{align*}
\argmin\limits_{\mx\in\RR^{k_1}}\sum\limits\limits_{a_1\in[k_1]} \E\slbt\slb\sum\limits\limits_{j\in[k_2]}\sqrt{\E[Y_1+Y_2\mid H_2, A_2=j]}\srb^2\mid H_1,A_1=a_1\srbt \sum\limits\limits_{i\in[k_1]} e^{\mx_i-\mx_{a_1}}
\end{align*}
is non-empty and in that case, any member of the above set is a candidate for $\tilde f_1(H_1)$.
Note that for any $H_1$ and $A_1=a_1$,
\[\slb\sum\limits\limits_{j\in[k_2]}\sqrt{\E[Y_1+Y_2\mid H_2, A_2=j]}\srb^2>0\]
because $\E[Y_1+Y_2\mid H_2, A_2]>0$ due to Assumption V.
Applying Lemma \ref{lemma: exp opti}, we can show that
the above set is non-empty  and 
\[\argmax(\tilde f_1(H_1))=\argmax_{a_1\in[k_1]}\E\lbt\slb\sum\limits\limits_{j\in[k_2]}\sqrt{\E[Y_1+Y_2\mid H_2, A_2=j]}\srb^2\mid H_1,A_1=a_1\rbt.\]
The proof follows noting
\[\E[Y_1+Y_2\mid H_2, A_2]=Y_1+Q^*_2(H_2,A_2)\]
and 
$\tilde d_1(H_1)=\pred(f_1(H_1))$. 
\end{proof}

\begin{lemma}
\label{lemma: exp opti}
    Suppose
    \begin{align*}
\tph(\my)=\sum\limits\limits_{i\in[k]}c_{i}\sum\limits\limits_{j\in[k]}e^{\my_j-\my_{i}},\quad \text{ for all }\my\in\RR^{k_2},
\end{align*}
where $k\in\NN$ and the $c_i$'s are positive real numbers.
Then $\inf_{\my\in\RR^{k}}\tph(\my)=(\sum\limits\limits_{i\in k}\sqrt{c_i})^2$ and $\argmin_{\my}\tph(\my)$ is a singletone set. Moreover, 
 $\my^*=\argmin_{\my}\tph(\my)$ satisfies  $\argmax(\my^*)=\argmax(\mathbf{c})$.
\end{lemma}

\begin{proof}[Proof of Lemma \ref{lemma: exp opti}]
    $\inf_{\my\in\RR^{k}}\tph(\my)$ is equivalent to the infimum of
\begin{align*}
g_1(\mx)=\sum\limits\limits_{i\in[k]}c_{i}\sum\limits\limits_{j\in[k]}\frac{\mx_j}{\mx_i},
\end{align*}
where $\mx_i=\exp^{\my_i}>0$ for each $i\in[k]$. However, the above is the infimum of
\begin{align}
\label{inlemma: CC: exp: opti 2nd stage}
g_2(\mbu)=\sum\limits\limits_{i\in[k]}c_{i}/\mbu_i
\end{align}
where $\mbu_i=\mx_i/\sum\limits\limits_{j\in[k]}\mx_j$ satisfies the constraints $\sum\limits\limits_{i=1}^{k}\mbu_i=1$ and $\mbu_i\in(0,1)$. Therefore, we need to
minimize the  function $\mbu\mapsto \sum\limits\limits_{i\in[k]}c_{i}/\mbu_i$ over the simplex  $\{\mbu\in\RR^k: \mbu^T\mo_{k}=1, \mbu_j\in(0,1)\text{ for all }j\in[k]\}$.   By Lagrange multiplier method, we can show that any critical point of the  minimization program 
\begin{mini}|s|
{\mbu\in\RR^{k}}{\sum\limits\limits_{i\in[k]}c_{i}/\mbu_i}
{\label{inlemma: cc: fact for result 1}}{}
\addConstraint{\sum\limits\limits_{i\in[k]}\mbu_i}{=1}
\addConstraint{\mbu_i}{\in(0,1)\quad\text{ for all }i\in[k].}
\end{mini}
satisfies
$\mbu_i=\sqrt{c_i/c_{\text{sum}}}$ where $c_{\text{sum}}\in \RR$ is such that $\sum\limits\limits_{i\in[k]}\mbu_i=1$ holds. This yields $c_{\text{sum}}=(\sum\limits\limits_{i\in k} \sqrt{c_i})^2$, leading to $\mbu_i=\sqrt{c_i}/\sum\limits\limits_{i\in[k]}\sqrt{c_i}$. Since $c_i>0$ for all $i\in[k]$, it follows that $\mbu_i\in(0,1)$ for $i\in[k]$. Therefore, $\mbu_i^*=\sqrt{c_i}/\sum\limits\limits_{i\in[k]}\sqrt{c_i}$ is also the unique  solution to the optimization problem \eqref{inlemma: cc: fact for result 1}.
 The minimum of \eqref{inlemma: cc: fact for result 1} is therefore $(\sum\limits\limits_{i\in[k]}\sqrt{c_i})^2$. The $\my^*_i$ corresponding to $\mbu_i=\sqrt{c_i}/\sum\limits\limits_{i\in[k]}\sqrt{c_i}$ satisfies $\my_i^*=\log(c_i)/2$. Thus $\my^*$ is the unique maximizer of  $\tph(\my)$. The proof is complete noticing  $\argmax(\my^*)=\argmax(c)$.
\end{proof}

\subsubsection{Proof of Result \ref{result: conv: one vs all}}
\label{sec: proof of result one vs all}
\begin{proof}[Proof of Result \ref{result: conv: one vs all}]
Note that $ \psi(\mx,\my;a_1,a_2)$ equals
\begin{align*}
  \MoveEqLeft -e^{-\my_{a_2}}\slb e^{-\mx_{a_1}}+\sum\limits\limits_{i\in[k_1]:i\neq a_1}e^{\mx_i}\srb+ \sum\limits\limits_{j\in[k_2]:j\neq a_2}e^{\my_j}\slb e^{-\mx_{a_1}}+\sum\limits\limits_{i\in[k_1]:i\neq a_1}e^{\mx_i}\srb\\
    =&\ -\slb e^{-\mx_{a_1}}+\sum\limits\limits_{i\in[k_1]:i\neq a_1}e^{\mx_i}\srb\slb e^{-\my_{a_2}}+\sum\limits\limits_{j\in[k_2]:j\neq a_2}e^{\my_j}\srb.
\end{align*}
    When $\psi$ is the loss in \eqref{inex: psi: one vs all}, $V^\psi(f)$ equals
    \begin{align*}
      \MoveEqLeft  \E\left[ \frac{-\slb e^{-f_{1A_1}(H_1)}+\sum\limits\limits_{i\in[k_1]:i\neq A_1}e^{f_{1i}(H_1)}\srb\slb e^{-f_{2A_2}(H_2)}+\sum\limits\limits_{j\in[k_2]:j\neq A_2}e^{f_{2j}(H_2)}\srb(Y_1+Y_2)}{\pi_1(A_1\mid H_1)\pi_2(A_2\mid H_2)}\right]\\
      =&\  \E\lbt-\slb e^{-f_{1A_1}(H_1)}+\sum\limits\limits_{i\in[k_1]:i\neq A_1}e^{f_{1i}(H_1)}\srb\\
      &\ \times \sum\limits\limits_{a_2\in[k_2]}\slb e^{-f_{2a_2}(H_2)}+\sum\limits\limits_{j\in[k_2]:j\neq a_2}e^{f_{2j}(H_2)}\srb\frac{\E[Y_1+Y_2\mid H_2, A_2=a_2]}{\pi_1(A_1\mid H_1)}\rbt
    \end{align*}
    Note that
    \begin{align}
    \label{inlemma: for result 3: def of V psi}
        \sup_{f_2} V^\psi(f_1,f_2)=&\ \E\lbt\slb e^{-f_{1A_1}(H_1)}+\sum\limits\limits_{i\in[k_1]:i\neq A_1}e^{f_{1i}(H_1)}\srb\nn\\
        &\ \times\sup_{\my\in\RR^{k_2}}\lbs -\sum\limits\limits_{a_2\in[k_2]}\slb e^{-\my_{a_2}}+\sum\limits\limits_{j\in[k_2]:j\neq a_2}e^{\my_j}\srb\frac{\E[Y_1+Y_2\mid h_2, a_2]}{\pi_1(a_1\mid h_1)}\rbs \rbt
    \end{align}
    Therefore, for fixed $h_2=(h_1,a_1,y_1,o_2)\in\H_2$,  $\tilde f_2(h_2)$ equals
    \begin{align*}
    \MoveEqLeft \argmax_{\my\in\RR^{k_2}}\lbs -\sum\limits\limits_{a_2\in[k_2]}\slb e^{-\my_{a_2}}+\sum\limits\limits_{j\in[k_2]:j\neq a_2}e^{\my_j}\srb\frac{\E[Y_1+Y_2\mid h_2, a_2]}{\pi_1(a_1\mid h_1)}\rbs\\
    =&\ \argmin\limits_{\my\in\RR^{k_2}}\sum\limits\limits_{a_2\in[k_2]}\slb e^{-\my_{a_2}}\E[Y_1+Y_2\mid H_2= h_2, A_2=a_2] \\
    &\ +e^{\my_{a_2}}\sum\limits\limits_{j\in[k_2]:j\neq a_2}\E[Y_1+Y_2\mid H_2= h_2,A_2=j] \srb\\
    =&\ \argmin\limits_{\my\in\RR^{k_2}}\sum\limits\limits_{a_2\in[k_2]}\slb e^{-\my_{a_2}}\mp_{a_2}(h_2) +e^{\my_{a_2}}(1-\mp_{a_2}(h_2)) \srb
    \end{align*}
    where
    \[\mp_j(h_2)=\frac{\E[Y_1+Y_2\mid H_2=h_2,A_2=j]}{\sum\limits\limits_{a_2\in[k_2]}\E[Y_1+Y_2\mid H_2=h_2, A_2=a_2]}\text{ and }\mp(h_2)=(\mp_j(h_2))_{j\in[k_2]}\]
     for all $h_2\in\H_2$. 
    For the rest of the proof, we will denote terms such as $\E[Y_1+Y_2\mid H_2=h_2, A_2=a_2]$  by $\E[Y_1+Y_2\mid h_2, a_2]$ for the sake of simplicity. 
Due to Assumption V, $\E[Y_1+Y_2\mid H_2,A_2]>0$ for all $H_2$ and $A_2$. Therefore,  $\mp_j(H_2)\in(0,1)$ for each $j\in[k_2]$. Hence, the vector $\mp(h_2)\in\S^{k_2-1}$. 
Fact \ref{fact: CC: one vs all} implies that $\argmax(\tilde f_2(h_2))\subset\argmax(\mp(h_2))=\argmax_{a_2\in[k_2]}\E[Y_1+Y_2\mid h_2, a_2]$. Therefore, $\tilde d_2(H_2)=\pred(\tilde f_2(H_2))\in\argmax_{a_2\in[k_2]}\E[Y_1+Y_2\mid H_2, a_2]$. Also, Fact \ref{fact: CC: one vs all} implies that  
\begin{align*}
 \MoveEqLeft \sup_{\my\in\RR^{k_2}}\lbs -\sum\limits\limits_{a_2\in[k_2]}\slb e^{-\my_{a_2}}+\sum\limits\limits_{j\in[k_2]:j\neq a_2}e^{\my_j}\srb\frac{\E[Y_1+Y_2\mid h_2, a_2]}{\pi_1(a_1\mid h_1)}\rbs\\ 
 =&\ 2 \frac{\sum\limits\limits_{j\in[k_2]}\E[Y_1+Y_2\mid h_2,j]}{\pi_1(a_1\mid h_1)}\sum\limits\limits_{j\in[k_2]}\sqrt{\mp(H_2)_j(1-\mp(H_2)_j)}\\
 =&\ 2\frac{\sum\limits\limits_{a_2\in[k_2]}\sqrt{\E[Y_1+Y_2\mid h_2,a_2]\sum\limits\limits_{j\in[k_2]:j\neq a_2}\E[Y_1+Y_2\mid h_2,j]}}{\pi_1(a_1\mid h_1)}.
\end{align*}
Therefore, \eqref{inlemma: for result 3: def of V psi} implies that $\sup_{f_2}V^\psi(f)$ equals
\begin{align*}
   2 & \E\lbt-\slb e^{-f_{1A_1}(H_1)}+\sum\limits\limits_{i\in[k_1]:i\neq A_1}e^{f_{1i}(H_1)}\srb\\
    &\ \times\sum\limits\limits_{a_2\in[k_2]}\frac{\sqrt{\E[Y_1+Y_2\mid H_2,A_2=a_2]\sum\limits\limits_{j\in[k_2]:j\neq a_2}\E[Y_1+Y_2\mid H_2,j]}}{\pi_1(A_1\mid H_1)}\rbt.
   \end{align*}
Denoting conditional expectation of a random variable $U$ with respect to $H_1$ and $A_1=a_1$ as $\E[U\mid H_1, a_1]$, we obtain that $-\sup_{f_2}V^\psi(f)$ equals
  \begin{align*}
  &\  2\E\lbt\sum\limits\limits_{a_1\in[k_1]}\slb e^{-f_{1a_1}(H_1)}+\sum\limits\limits_{i\in[k_1]:i\neq a_1}e^{f_{1i}(H_1)}\srb\\
   &\ \times\sum\limits\limits_{a_2\in[k_2]}\E\lbt \sqrt{\E[Y_1+Y_2\mid H_2,A_2=a_2]\sum\limits_{j\in[k_2]:j\neq a_2}\E[Y_1+Y_2\mid H_2,j]}\bl H_1, a_1\rbt\rbt\\
   =&\ 2\E\lbt\sum\limits_{a_1\in[k_1]} e^{-f_{1a_1}(H_1)}\sum\limits_{a_2\in[k_2]}\E\lbt \sqrt{\E[Y_1+Y_2\mid H_2,A_2=a_2]\sum\limits_{j\neq a_2}\E[Y_1+Y_2\mid H_2,j]}\bl H_1, a_1\rbt\rbt\\
  + &\ 2\E\lbt\sum\limits_{a_1\in[k_1]} e^{f_{1a_1}(H_1)}\sum\limits_{i\neq a_1}\sum\limits_{a_2\in[k_2]}\E\lbt \sqrt{\E[Y_1+Y_2\mid H_2,A_2=a_2]\sum\limits_{j\neq a_2}\E[Y_1+Y_2\mid H_2,j]}\bl H_1, i\rbt\rbt,
\end{align*}

where for any random variable $X$ and $i\in[k_1]$, we use the shorthand $\E[X\mid H_1,i]$ for $\E[X\mid H_1, A_1=i]$.
We overload notation and for all $h_1\in\H_1$ and $a_1\in[k_1]$, we define
\[\mp_{a_1}(h_1)=\frac{\sum\limits_{a_2\in[k_2]}\E\lbt \sqrt{\E[Y_1+Y_2\mid H_2,A_2=a_2]\sum\limits_{j\in[k_2]:j\neq a_2}\E[Y_1+Y_2\mid H_2,j]}\bl H_1, a_1\rbt}{\sum\limits_{i\in[k_1]}\sum\limits_{a_2\in[k_2]}\E\lbt \sqrt{\E[Y_1+Y_2\mid H_2,A_2=a_2]\sum\limits_{j\in[k_2]:j\neq a_2}\E[Y_1+Y_2\mid H_2,j]}\bl H_1, i\rbt}\]
and $\mp(h_1)=(\mp_i(h_1))_{i\in[k_1]}$.
That $\mp(h_1)\in\S^{k_1-1}$ follows since $\E[Y_1+Y_2\mid H_2,A_2]>0$  for all $H_2$ and $A_2$ by our assumption. 
Therefore, for fixed $h_1\in\H_1$, we observe that $\tilde f_1(h_1)$ equals
\begin{align*}
\argmin\limits_{\mx\in\RR^{k_1}}\sum\limits_{a_1\in[k_1]}\slb e^{-\mx_{a_1}}\mp_{a_1}(h_1)+e^{\mx_{a_1}}(1-\mp_{a_1}(h_1))\srb.
\end{align*}
Another application of Fact \ref{fact: CC: one vs all} implies that
$\tilde f_1(h_1)$ uniquely exists in $\RR^{k_1}$ and
\begin{align*}
& \argmax(\tilde f_1(h_1))\subset \argmax(\mp(h_1))\\
   &\  = \argmax\limits_{i\in[k_1]}\sum\limits_{a_2\in[k_2]}\E\lbt \sqrt{\E[Y_1+Y_2\mid H_2,A_2=a_2]\sum\limits_{j\neq a_2}\E[Y_1+Y_2\mid H_2,j]}\bl H_1, i\rbt.
\end{align*}
The proof follows noting $\tilde d_1(H_1)=\pred(\tilde f_1(H_1))$.
\end{proof}

\begin{fact}
\label{fact: CC: one vs all}
For any $\mp\in\S^{k-1}$ optimization problem
\begin{mini}|s|
    {\my\in\RR^{k}}{\sum\limits_{j\in[k]}\slb e^{-\my_j}\mp_j+(1-\mp_j)e^{\my_j}\srb}{\label{instatement: fact: cc: one vs all}}{}
\end{mini}
has a unique minimizer $\my^*$ satisfying $\argmax(\my^*)\subset\argmax(\mp)$ and the minimum value is $2\sum\limits_{j\in[k]}\sqrt{\mp_j(1-\mp_j)}$.
\end{fact}

\begin{proof}[Proof of Fact \ref{fact: CC: one vs all}]

    The given optimization problem is equivalent to
 \begin{mini*}|s|
    {\mx\in\RR^{k}_{>0}}{\sum\limits_{j\in[k]}\slb \mp_j/\mx_j+(1-\mp_j)\mx_j\srb}{}{}
\end{mini*}   
where we used the variable transformation $\mx_j=e^{\my_j}$ for $j\in[k]$. The function $\tph(\mx)=\sum\limits_{j\in[k]}\slb \mp_j/\mx_j+(1-\mp_j)\mx_j\srb$ is convex on $\RR^{k}_{>0}$. Hence, any critical point $\mx^*$ of $\tph$ is the unique global minimum over $\RR^{k}_{>0}$. Since $\tph$ is differentiable on $\RR^{k}_{>0}$, it follows that $\mx^*$ satisfies $\grad \tph(\mx^*)=0$. Therefore,  $\mx^*_j=\sqrt{\mp_j/(1-\mp_j)}$ for all $j\in[k]$, which implies  $\tph(\mx^*)=2\sum\limits_{j\in[k]}\sqrt{\mp_j(1-\mp_j)}$. 
Note that the optimal solution of \eqref{instatement: fact: cc: one vs all} satisfies $\my^*_j=\log(\mx^*_j)$. Therefore, $\my^*$ is unique and $\my^*_j=\frac{1}{2}\log(\mp_j/(1-\mp_j))$. Since $\phi(x)=\exp(-x)$ is convex, bounded below, differentiable, and decreasing,  Theorem 9 of \cite{zhang2004} implies  that $\argmax(\my^*)\subset\argmax(\mp)$.

\end{proof}

\subsubsection{Proof of Result \ref{result: single-stage FC}}
\label{sec: pf Result: additive convex loss}
\begin{proof}[Proof of Result \ref{result: single-stage FC}] 
     $V^\psi(f)$ equals
     \begin{align*}
  \MoveEqLeft  \E\lbt  \frac{(Y_1+Y_2)\phi^{(1)}(f_1(H_1);A_1) }{\pi_1(A_1\mid H_1)\pi_2(A_2\mid H_2)}+ \frac{(Y_1+Y_2)\phi^{(2)}(f_1(H_1);A_2) }{\pi_1(A_1\mid H_1)\pi_2(A_2\mid H_2)}\rbt
\end{align*}
Since the propensity scores are positive by Assumption I, 
\[ \begin{split}
   \MoveEqLeft \E\lbt  \phi^{(1)}(f_1(H_1);A_1) \frac{Y_1+Y_2}{\pi_1(A_1\mid H_1)\pi_2(A_2\mid H_2)}\rbt\\
    &\ =\E\lbt \sum\limits_{a_1\in[k_1]} \E\lbt \frac{Y_1+Y_2}{\pi_2(A_2\mid H_2)}\bl H_1=h_1, A_1=a_1\rbt\phi^{(1)}(f_1(H_1);a_1)\rbt
\end{split}\]
Thus for any $h_1\in\H_1$, a version of $\tilde f_1(h_1)$ exists if
\begin{align*}
\argmax_{\mx\in\RR^{k_1}}\sum\limits\limits_{a_1\in[k_1]} \E\lbt \frac{Y_1+Y_2}{\pi_2(A_2\mid H_2)}\bl H_1=h_1, A_1=a_1\rbt\phi^{(1)}(\mx;a_1)
\end{align*}
is non-empty. Let us define 
\begin{align*}
    \mp^{(1)}_i=\E\lbt \frac{Y_1+Y_2}{\pi_2(A_2\mid H_2)}\bl H_1=h_1, A_1=a_1\rbt,
\end{align*}
which is positive because of Assumptions  V.
In that case, 
\[\tilde f_1(H_1)\in \argmax_{\mx\in\RR^{k_1}}\sum_{a_1\in[k_1]}\mp_{i}^{(1)}\phi^{(1)}(\mx;a_1).\]
However, since $\mp^{(1)}\in\RR^{k_1}_{>0}$,  Lemma \ref{lemma: CC: FC single stage} below implies that if
\[\mx^*\in \argmax_{\mx\in\RR^{k_1}}\sum_{a_1\in[k_1]}\mp_{i}^{(1)}\phi^{(1)}(\mx;a_1),\]
then $\pred(\mx^*)\in\argmax(\mp^{(1)})$. Therefore, we have shown that
whenever $\tilde f_1(h_1)$ exists,
\begin{align*}
 \pred(\tilde f_1(h_1))\in   \argmax(\mp^{(1)})= \argmax_{a_1\in[k_1]}\E\lbt \frac{(Y_1+Y_2)}{\pi_2(A_2\mid H_2)}\bl H_1=h_1, A_1=a_1\rbt.
\end{align*}
Similarly, since $\phi^{(2)}$ is also single-stage Fisher consistent, taking $k_1$ to be $k_2$ in Lemma \ref{lemma: CC: FC single stage}, we can show that for a fixed $h_2=(h_1,a_1,y_1,o_2)$, whenever $\tilde f_2(h_2)$ exists, 
\begin{align*}
\pred(\tilde f_2(h_2))\in  &\ \argmax_{a_2\in[k_2]}\E\lbt \frac{(Y_1+Y_2)}{\pi_1(A_1\mid H_1)}\bl H_2=h_2, A_2=a_2\rbt\\
=&\ \argmax_{a_2\in[k_2]}\frac{1}{\pi_1(a_1\mid h_1)}\slb y_1+\E[Y_2 \mid H_2=h_2, A_2=a_2]\srb\\
=&\ \argmax_{a_2\in[k_2]}\E[Y_2\mid H_2=h_2, A_2=a_2],
\end{align*}
where we  used the fact that $\pi_1(a_1\mid h_1)>0$ (by Assumption I). The proof follows noting $\tilde d_t(H_t)=\pred(f_t(H_t))$ for $t=1,2$. 
 \end{proof}

\begin{lemma}
\label{lemma: CC: FC single stage}
    Suppose $\phi:\RR^{k_1}\mapsto\RR$ is  a Fisher consistent surrogate for $T=1$. Then if $\argmax_{\mx\in\RR^{k_1}}\Psi(\mx;\mp)$ is non-empty for some $\mp\in\RR^{k_1}_{\geq 0}$ with $\Psi$ as defined in \eqref{def: Psi and psi star main text}, then for any $\mx^*\in\argmax_{\mx\in\RR^{k_1}}\Psi(\mx;\mp)$, it follows that $\mp_{\pred(\mx^*)}=\max(\mp)$.
    
\end{lemma}

\begin{proof}
Let $\mp\in\RR^{k_1}_{>0}$. Consider the distribution $\PP\in\mP^0$ so that $H_1=\emptyset$ and  $\E[Y_1\mid  A_1=i]=\mp_i$ for all $i\in[k_1]$. For such $\PP$, $f_1:\H_1\mapsto\RR^{k_1}$ is just a vector in $\RR^{k_1}$. Since $\PP$ also satisfies Assumption I-III, for any such $f_1$, it holds that 
\begin{align*}
    V(f_1)= \E[Y_1 1[\pred(f_1)=A_1]/\pi_1(A_1\mid H_1)]=\mp_{\pred(f_1)},
\end{align*}
and
\begin{align*}
    V^\phi(f_1)= \E[Y_1 \phi(f_1,A_1)/\pi_1(A_1\mid H_1)]=\Psi(f_1;\mp).
\end{align*}
Let $\tilde f_1$ be the maximizer of $V^\phi(f_1)$ over $f_1\in\F_1$. 
Therefore, if  $\argmax_{\mx\in\RR^{k_1}}\Psi(\mx;\mp)$ is non-empty, then  $\tilde f_1$ exists. In that case, any  $\mx^*\in\argmax_{\mx\in\RR^{k_1}}\Psi(f_1;\mp)$ is a candidate of $\tilde f_1$. Since $\tilde f_1$ exists and $\phi$ is Fisher consistent for $T=1$, Definition \ref{def: Fisher consistency multiclass} (definition of  Fisher consistency)   implies $V(\tilde f_1)=\sup_{f_1\in\F_1}V(f_1)$. Therefore, it follows that  $\mp_{\pred(\mx^*)}=\max(\mp)$ for any  $\mx^*\in\argmax_{\mx\in\RR^{k_1}}\Psi(f_1;\mp)$.
\end{proof}

\subsubsection{Proofs for Example \ref{ex: sum zero}}
\label{secpf: hinge loss}
First, we derive a formula for $\tilde f_1(h_1)$ under the general two-stage setup. It will be used in Section \ref{sec: toy setup}, where we derive a simpler form of $\tilde f_1(h_1)$ under the toy setup. 
In this case, for all $\mx\in\RR^{k_1}$, $\my\in\RR^{k_2}$, $a_1\in[k_1]$, and $a_2\in[k_2]$,
\begin{align}
\label{def: cc: sum-zero constraint loss two stage hinge}
\psi(\mx,\my;a_1,a_2)=&\ \sum\limits_{i\in[k_1]:i\neq A_1}\sum_{j\in[k_2]: j\neq a_2}\min(-\mx_i-1,-\my_j-1,0)\nn\\
=&\ -\sum\limits_{i\in[k_1]:i\neq A_1}\sum_{j\in[k_2]: j\neq a_2}\max(1+\mx_i,1+\my_j,0).
\end{align}
For any set $\C\subset \F_1\times \F_2$, $\sup_{(f_1,f_2)\in\C}V^\psi(f_1,f_2)$ takes the form
\begin{align*}
   \MoveEqLeft\sup_{(f_1,f_2)\in\C}-\E\lbtt \frac{(Y_1+Y_2)\sum\limits_{i\in[k_1]:i\neq A_1}\sum\limits_{j\in[k_2]: j\neq A_2}\max(1+f_{1i}(H_1),1+f_{2j}(H_2),0)}{\pi_1(A_1\mid H_1)\pi_2(A_2\mid H_2)}\rbtt\\
   =&\ \sup_{(f_1,f_2)\in\C}-\E\lbtt \frac{(Y_1+Y_2)\sum\limits_{i\in[k_1]:i\neq A_1}\sum\limits_{j\in[k_2]: j\neq A_2}\max(1+f_{1i}(H_1),1+f_{2j}(H_2),0)}{\pi_1(A_1\mid H_1)\pi_2(A_2\mid H_2)}\rbtt\\
   =&\ - \inf_{(f_1,f_2)\in\C}\E\lbtt \frac{(Y_1+Y_2)\sum\limits_{i\in[k_1]:i\neq A_1}\sum\limits_{j\in[k_2]: j\neq A_2}\max(1+f_{1i}(H_1),1+f_{2j}(H_2),0)}{\pi_1(A_1\mid H_1)\pi_2(A_2\mid H_2)}\rbtt.
\end{align*}
Therefore, it suffices to consider the minimization problem. 
Let us denote the scores corresponding to $f_1$ as $f_{1j}$'s for $j\in[k_1]$ so that $f_1=(f_{11},\ldots,f_{1k_1})$. Similarly, let us denote the scores corresponding to $f_2$ as $f_{2j}$'s for $j\in[k_2]$ so that $f_2=(f_{21},\ldots,f_{2k_2})$.
Let us denote
\[\R(f_1,f_2)=\E\lbt (Y_1+Y_2)\frac{\sum\limits_{i\in[k_1]:i\neq A_1}\sum\limits_{j\in[k_2]: j\neq A_2}\max(1+f_{1i}(H_1),1+f_{2j}(H_2),0)}{\pi_1(A_1\mid H_1)\pi_2(A_2\mid H_2)}\rbt,\]
which equals
\begin{align*}
&\ \E\lbt \frac{\E[Y_1+Y_2\mid  H_2,A_2]}{\pi(A_1\mid H_1)} \frac{\sum\limits_{i\in[k_1]:i\neq A_1}\sum\limits_{j\in[k_2]: j\neq A_2}\max(1+f_{1i}(H_1),1+f_{2j}(H_2),0)}{\pi_2(A_2\mid H_2)}\rbt\\
  &= \E\lbt \sum_{a_2\in[k_2]}\frac{\E[Y_1+Y_2\mid  H_2, A_2=a_2]}{\pi(A_1\mid H_1)} \sum\limits_{i\in[k_1]:i\neq A_1}\sum\limits_{j\in[k_2]: j\neq a_2}\max(1+f_{1i}(H_1),1+f_{2j}(H_2),0)\rbt\\
  &= \E\lbt \sum_{j\in[k_2]}\sum\limits_{i\in[k_1]:i\neq A_1}\max(1+f_{1i}(H_1),1+f_{2j}(H_2),0)\sum\limits_{a_2\in[k_2]: a_2\neq j}\frac{\E[Y_1+Y_2\mid H_2, A_2=a_2]}{\pi(A_1\mid H_1)} \rbt.
\end{align*}
Then $\inf_{f_2\in \F_2: \sum_{j\in [k_2]}f_{2j}=0}\R(f_1,f_2)$ equals
\begin{align*}
    =&\ \E\lbt \inf_{\my\in\RR^{k_2}:\sum_{j=1}^{k_2}\my_j=0}\sum_{j\in[k_2]}\sum\limits_{i\in[k_1]:i\neq A_1}\max(1+f_{1i}(H_1),1+\my_j,0)\\
    &\  \times \sum\limits_{a_2\in[k_2]: a_2\neq j}\frac{\E[Y_1+Y_2\mid H_2, A_2=a_2]}{\pi(A_1\mid H_1)} \rbt
\end{align*}
For $h_2\in\H_2$, let us denote
\[\mp_{a_2}(h_2)=\frac{\E[Y_1+Y_2\mid h_2,A_2=a_2]}{\sum\limits_{j\in[k_2]}\E[Y_1+Y_2\mid h_2,A_2=j]}\quad\text{ for all }a_2\in[k_2]\text{ and }\mp(h_2)=(\mp_{a_2}(h_2))_{a_2\in[k_2]},\]
which are well-defined since $\E[Y_1+Y_2\mid h_2,A_2=j]>0$  for all $h_2$ and $j\in[k_2]$ due to Assumption V. 
Then
\begin{align*}
\MoveEqLeft\inf_{\my\in\RR^{k_2}:\sum_{j=1}^{k_2}\my_j=0}\lbs\sum_{j\in[k_2]}\sum\limits_{i\in[k_1]:i\neq A_1}\max(1+f_{1i}(H_1),1+\my_j,0)\\
&\ \times\sum\limits_{a_2\in[k_2]: a_2\neq j}\E[Y_1+Y_2\mid H_2, A_2=a_2] \rbs   \\
=&\ \lb\sum\limits_{j\in[k_2]}\E[Y_1+Y_2\mid h_2,A_2=j]\rb\\
&\ \times\inf_{\my\in\RR^{k_2}:\sum_{j=1}^{k_2}\my_j=0}\sum_{j\in[k_2]}\sum\limits_{i\in[k_1]:i\neq A_1}\max(1+f_{1i}(H_1),1+\my_j,0)\slb 1-p_j(H_2)\srb.
\end{align*}
For fixed $H_1$, 
\begin{align*}
\MoveEqLeft\sum_{j\in[k_2]}\slb 1-p_j(H_2)\srb\sum\limits_{i\in[k_1]:i\neq A_1}\max(1+f_{1i}(H_1),1+\my_j,0)\\
=&\ k_1 \sum_{j\in[k_2]}\slb 1-p_j(H_2)\srb+\sum_{j\in[k_2]}\slb 1-p_j(H_2)\srb\sum\limits_{i\in[k_1]:i\neq A_1}\max(f_{1i}(H_1),\my_j,-1)\\
=&\ k_1 \sum_{j\in[k_2]}\slb 1-p_j(H_2)\srb+\sum_{j\in[k_2]}\slb 1-p_j(H_2)\srb\sum\limits_{i\in[k_1-1]}\max(\mv_i(H_1;A_1),\my_j),
\end{align*}
where $\mv(H_1;A_1)=(\max(f_{1i}(H_1),-1))_{i\in[k_1]:i\neq A_1}$ is a vector of length $2$.
When  $k_1=3$ and $k_2=2$, Lemma \ref{lemma: hinge: special case} implies that
\begin{align*}
\MoveEqLeft\inf_{\my\in\RR^{k_2}:\sum_{j\in[k_2]}\my_j=0}\sum_{j\in[k_2]}\slb 1-p_j(H_2)\srb\sum\limits_{i\in[k_1-1]}\max(\mv_i(H_1;A_1),\my_j)\\
=&\ \begin{cases}
     \mv_1(H_1;A_1)+\mv_2(H_1;A_1) &\ \text{ if }\min(\mv(H_1;A_1))\geq 0\\
      \upkappa(\mv(H_1;A_1),\mp(H_2)) & \text{ otherwise} 
      \end{cases}\\
       =&\  \begin{cases}\sum\limits_{i\in[k_1]:i\neq A_1}f_{1i}(H_1) &\ \text{ if }f_{1i}(H_1)\geq 0\text{ for }i\neq A_1\\
      \upkappa(\mv(H_1;A_1),\mp(H_2)) & \text{ otherwise,}
  \end{cases}
\end{align*}
 where the last step follows because $\min(\mv(H_1;A_1))\geq 0$ if and only if $f_{1i}(H_1)>0$ for all $i\neq A_1$, in which case, $\max(f_{1i}(H_1),-1)=f_{1i}(H_1)$. 
Since 
$\sum_{j\in[k_2]}( 1-p_j(H_2))=k_2-1$, taking the infimum of
\[\sum_{j\in[k_2]}\sum\limits_{i\in[k_1]:i\neq A_1}\max(1+f_{1i}(H_1),1+\my_j,0)\sum\limits_{a_2\in[k_2]: a_2\neq j}\E[Y_1+Y_2\mid H_2, A_2=a_2] \]
over $\{\my\in\RR^{k_2}:\sum_{j=1}^{k_2}\my_j=0\}$  yields
\begin{align*}
\MoveEqLeft k_1(k_2-1)\sum\limits_{j\in[k_2]}\E[Y_1+Y_2\mid h_2,A_2=j]\\
&\ + \lb\sum\limits_{j\in[k_2]}\E[Y_1+Y_2\mid h_2,A_2=j]\rb \sum\limits_{i\in[k_1]:i\neq A_1}f_{1i}(H_1) 1\slbt \min\limits_{i\in[k_1]:i\neq A_1}f_{1i}(H_1)\geq 0 \srbt\\
&\ + \lb\sum\limits_{j\in[k_2]}\E[Y_1+Y_2\mid h_2,A_2=j]\rb\upkappa(\mv(H_1;A_1),\mp(H_2)) 1\slbt \min\limits_{i\in[k_1]:i\neq A_1}f_{1i}(H_1)< 0 \srbt.
\end{align*}
Therefore, maximizing $\inf_{f_2}\R(f_1,f_2)$ with respect to $f_1:\sum_{i\in[k_1]}f_{1i}=0$ is equivalent to minimizing
\begin{align*}
 \MoveEqLeft \E\left[\frac{\sum\limits_{j\in[k_2]}\E[Y_1+Y_2\mid H_2,A_2=j]\sum\limits_{i\in[k_1]:i\neq A_1}\max(f_{1i}(H_1),-1) 1\slbt \min\limits_{i\in[k_1]:i\neq A_1}f_{1i}(H_1)\geq 0 \srbt}{\pi_1(A_1\mid H_1)}\rbtt\\
 &\ +\E\lbtt\frac{\sum\limits_{j\in[k_2]}\E[Y_1+Y_2\mid H_2,A_2=j]\upkappa(\mv(H_1;A_1),\mp(H_2)) 1\slbt \min\limits_{i\in[k_1]:i\neq A_1}f_{1i}(H_1)< 0 \srbt\rb} {\pi_1(A_1\mid H_1)}  \right]
\end{align*}
with respect to $f_1:\sum_{i\in[k_1]}f_{1i}=0$ . However, the last display equals 
\begin{align*}
 &\ \E\left[\pi_1(A_1\mid H_1)^{-1} \begin{Bmatrix}\sum\limits_{j\in[k_2]}\E\slbt\E[Y_1+Y_2\mid H_2,A_2=j]\mid H_1, A_1\srbt&\\
  \times \lb \sum\limits_{i\in[k_1]:i\neq A_1}f_{1i}(H_1) 1\slbt \min\limits_{i\in[k_1]:i\neq A_1}f_{1i}(H_1)\geq 0 \srbt &\\
  +\upkappa(\mv(H_1;A_1),\mp(H_2)) 1\slbt \min\limits_{i\in[k_1]:i\neq A_1}f_{1i}(H_1)< 0 \srbt\rb \end{Bmatrix}  \right] \\
  =&\ \E\left[\sum_{a_1\in[k_1]}\sum\limits_{j\in[k_2]}\E\slbt\E[Y_1+Y_2\mid H_2,A_2=j]\mid H_1, a_1\srbt\sum\limits_{i\neq A_1}f_{1i}(H_1) 1\slbt \min\limits_{i\neq A_1}f_{1i}(H_1)\geq 0 \srbt\right] \\
  &\ +\E\lbt\sum_{a_1\in[k_1]}\slb\sum\limits_{j\in[k_2]}\E\slbt\E[Y_1+Y_2\mid H_2,A_2=j]\mid H_1, a_1\srbt\srb\\
  &\ \times \upkappa(\mv(H_1;A_1),\mp(H_2)) 1\slbt \min\limits_{i\in[k_1]:i\neq A_1}f_{1i}(H_1)< 0 \srbt \rbt. 
\end{align*}
Let us denote
\[\upalpha(H_1,a_1)=\sum\limits_{j\in[k_2]}\E\slbt\E[Y_1+Y_2\mid H_2,A_2=j]\mid H_1, a_1\srbt.\]
Then $\tilde f_1(H_1)$ is the maximizer of
\begin{equation}
    \label{def: hinge: uppsi}
   \begin{split}
 \MoveEqLeft   \sum_{a_1\in[k_1]}\upalpha(H_1,a_1)\lb \sum\limits_{i\in[k_1]:i\neq A_1}\mx_i 1\slbt \min\limits_{i\in[k_1]:i\neq A_1}\mx_i\geq 0 \srbt\\
    &\ +\upkappa(\mv(a_1),\mp(H_2)) 1\slbt \min\limits_{i\in[k_1]:i\neq A_1}\mx_i< 0 \srbt\rb
    \end{split}
\end{equation}
subject to  $\sum_{i\in[k_1]}\mx_i=0$, 
where $\mv(a_1)\in\RR^{k_1-1}$ satisfies 
$\mv(a_1)=(\max(\mx_i,-1))_{i\in[k_1]:i\neq a_1}$. 

\paragraph{Toy setup of the sum-zero loss example}
\label{sec: toy setup}
Under the toy setup of Example \ref{ex: sum zero}, $H_1=\emptyset$. Therefore, $f_1(H_1)\equiv f_1\in\RR^3$ and 
\[\upalpha(H_1,a_1)\equiv \upalpha(a_1)=\sum\limits_{j\in[k_2]}\E[Y_2\mid A_1=a_1,A_2=j].\]
Moreover, since $H_2\equiv A_1$, $\mp(H_2)=\mp(A_1)$, and
\[\mp_j(A_1)=\frac{\E[Y_2\mid A_1,A_2=j]}{\sum_{j\in[k_2]}\E[Y_2\mid A_1,A_2=j]}\text{ for }j\in[k_2].\]
In our example, $H_1=\emptyset$, which implies $f_1(H_1)=f_1$. Also, $k_1=3$.
In this case, \eqref{def: hinge: uppsi} reduces to the following optimization problem: 
 \begin{mini}|s|
{\mx\in\RR^3:\sum_{i\in[k_1]}\mx_i=0}{\sum_{a_1\in[k_1]}\upalpha(a_1)\left(\splitfrac{\sum_{j\in[k_1]:j\neq a_1}\mx_j 1\slbt \min\limits_{i\in[k_1]:i\neq A_1}\mx_i\geq 0 \srbt}{+\upkappa(\mv(a_1),\mp(a_1)) 1\slbt \min\limits_{i\in[k_1]:i\neq a_1}\mx_i< 0 \srbt}\right)}
{\label{opti: hinge: with x}}{}
\addConstraint{\mx_1+\mx_2+\mx_3}{=0}{}
\addConstraint{\mv(a_1)}{=(\max(\mx_i,-1))_{i\in[3]:i\neq a_1}}{}
\end{mini}
Any solution to solution to \eqref{opti: hinge: with x} is a candidate of $\tilde f_1$. 
Let us define
\begin{equation}
    \label{def: uppsi}
\uppsi(\mx_1,\mx_2)=\sum_{a_1\in[k_1]}\upalpha(a_1)\left(\splitfrac{\sum_{j\in[k_1]:j\neq a_1}\mx_j 1\slbt \min\limits_{i\in[k_1]:i\neq A_1}\mx_i\geq 0 \srbt}{+\upkappa(\mv(a_1),\mp(a_1)) 1\slbt \min\limits_{i\in[k_1]:i\neq a_1}\mx_i< 0 \srbt}\right)
\end{equation}
where the $\mx$ used in \eqref{def: uppsi} is $\mx=(\mx_1,\mx_2,-(\mx_1+\mx_2))$.
Note that $\tilde f$ is a  solution to \eqref{opti: hinge: with x} if and only if $(\tilde f_1,\tilde f_2)$ is a solution to
\begin{mini}|s|
{(\mx_1,\mx_2)\in\RR^2}{\uppsi(\mx_1,\mx_2)}
{\label{opti: hinge: with x reduced}}{}
\end{mini}
and $\tilde f_3=-(\tilde f_1+\tilde f_2)$. The optimization program \eqref{opti: hinge: with x reduced} is an unconstrained optimization problem with 2 variables.
 We obtained the surface plot of $\uppsi$ using the software R in the three settings mentioned in Table \ref{tab:hinge} to find out the optimal solution. Figure \ref{fig: hinge: surface plot} gives the surface plots for settings 2 and 3. The plot for setting 1 is skipped because it is similar to that of Setting 2. Finally, in this case, $Q_2^*(H_2,A_2)=\E[Y_1+Y_2\mid H_2, A_2]=\E[Y_2\mid A_1,A_2]$ and
    \[Q_1^*(H_1,A_1)\equiv Q_1^*(A_1)=\E\slbt\max\limits_{j\in[k_2]}Q_2^*(H_2,j)\mid A_1\srbt=\E\slbt\max\limits_{j\in[k_2]}\E[Y_2\mid A_1,A_2=j]\mid A_1\srbt,\]
    which equals $\max_{j\in[k_2]}\E[Y_2\mid A_1,A_2=j]$. 
    Hence, any member of
    \[ \argmax_{i\in[3]}\slb \max\limits_{j\in[2]}\E[Y_2\mid A_1=i,A_2=j]\srb\]
    is a candidate for $d_1^*$, the first stage optimal treatment assignment. This explains the $d_1^*$ values in Table \ref{tab:hinge}.

\begin{figure}[htbp]
    \centering
    \begin{subfigure}{0.45\textwidth}
        \centering
        \includegraphics[width=\textwidth]{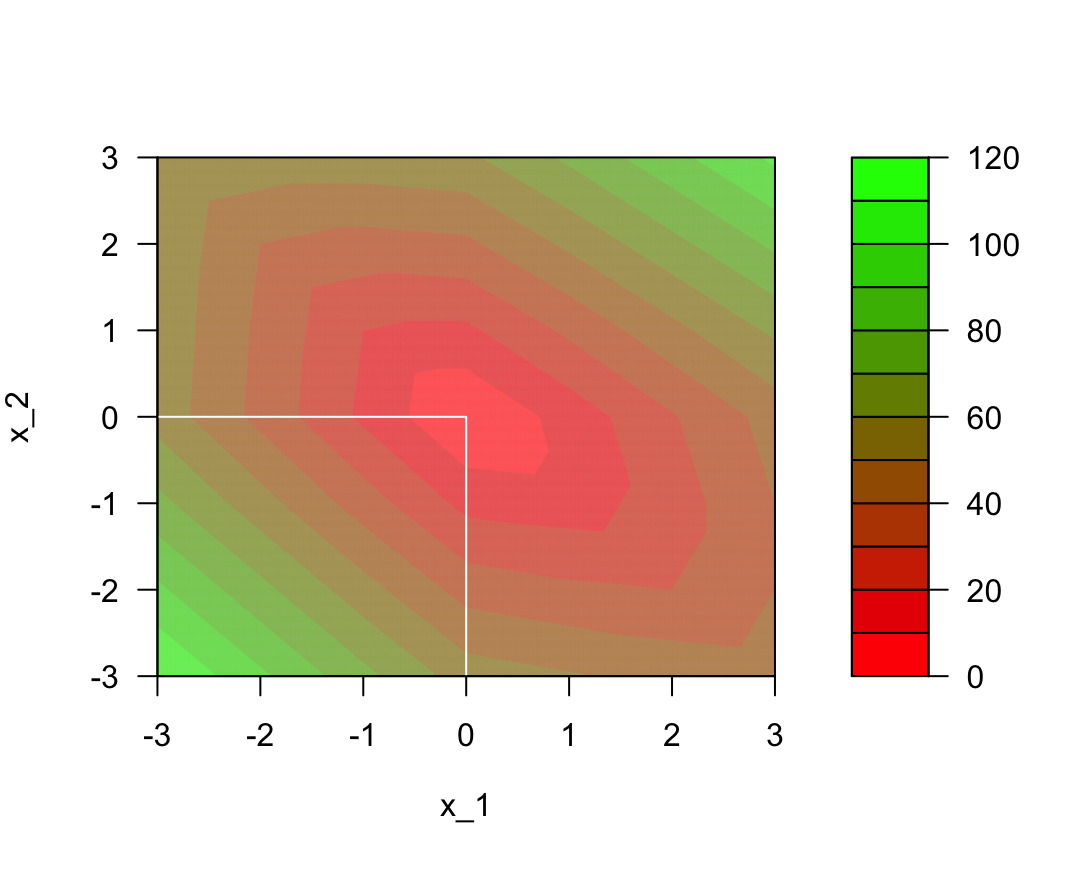}  
        \caption{Setting 2 in Table \ref{tab:hinge}}
        \label{fig:subfigure1}
    \end{subfigure}%
    \hfill
    \begin{subfigure}{0.45\textwidth}
        \centering
        \includegraphics[width=\textwidth]{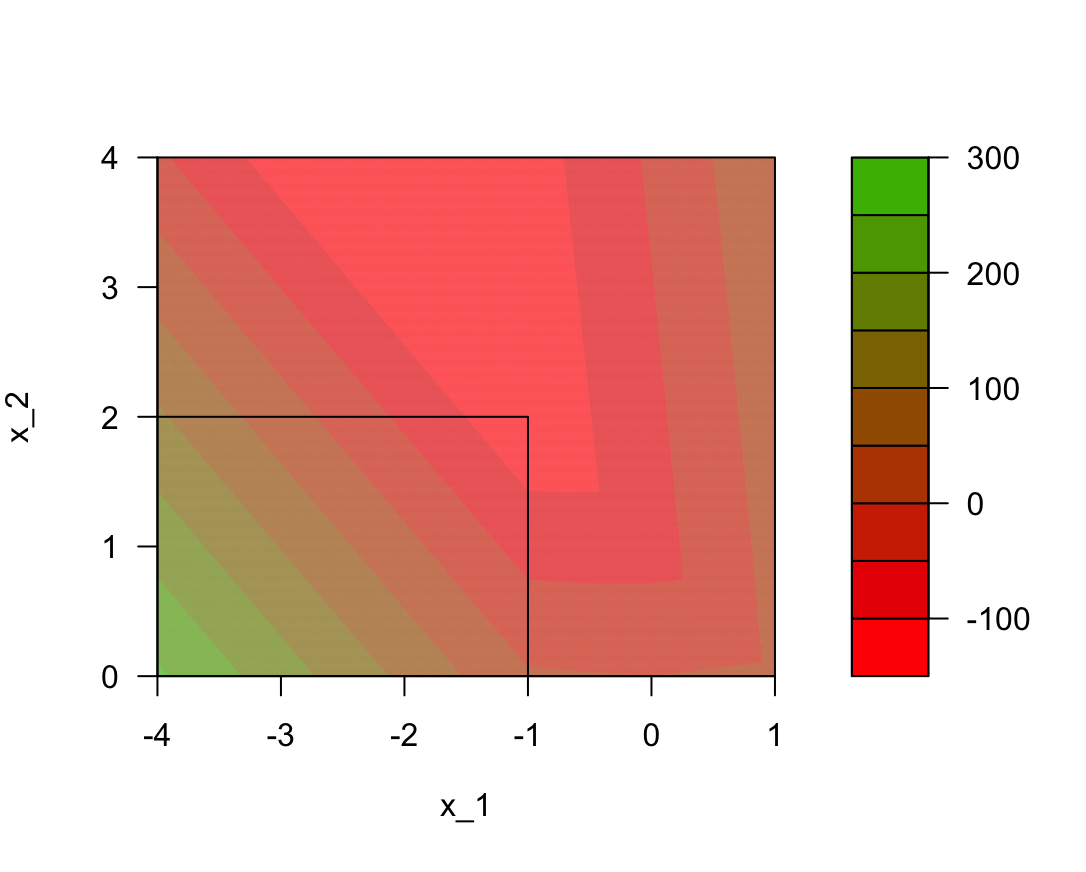}  
        \caption{Setting 3 in Table \ref{tab:hinge}}
        \label{fig:subfigure2}
    \end{subfigure}
    \caption{{\bf Surface  plot of $\uppsi$ defined in  \eqref{def: uppsi} for the last two settings in Table \ref{tab:hinge}}. The point where the vertical and horizontal lines meet is the point where the minima occurs. The value of the function decreases as the color turns red from green. The plot for Setting 1 is skipped because it is similar to that of Setting 2.}
    \label{fig: hinge: surface plot}
\end{figure}

\paragraph{Additional lemmas for the proofs of Section \ref{secpf: hinge loss}}

\begin{lemma}
\label{lemma: hinge: special case}
    Let $\mv\in\RR^{k_1-1}$ and $\mp\in\S^{k_2-1}$.  If $\mv_{[1]}< 0$, then the optimum value of the program 
         \begin{mini}|s|
{\my\in\RR^{k_2}}{\sum_{j\in[k_2]}\slb 1-\mp_j\srb\sum\limits_{i\in[k_1-1]}\max(\mv_i,\my_j)}
{\label{opt: hinge with v}}{}
\addConstraint{\sum_{j\in[k_2]}\my_j}{=0,}{}
\end{mini}
is $\upkappa(\mv,\mp)$ if $k_1=3$ and $k_2=2$, where 
\begin{align*}
    \upkappa(\mv,\mp)= \begin{cases}
   \leftfrac{\mv_{[1]}|\mp_1-\mp_2|+\max(\mv_{[2]},-\mv_{[1]})(\mp_1\wedge \mp_2)}{+\mv_{[2]}(\mp_1\vee\mp_2)}& \text{ if }\mp_1\wedge \mp_2\leq(\mp_1\vee\mp_2)/2,\\[10pt] \splitfrac{-|\min(\mv_{[2]},-\mv_{[1]})(\mp_2-\mp_1)|}{+\mv_{[2]}(\mp_1\vee \mp_2)+|\mv_{[2]}|(\mp_1\wedge \mp_2)} & \text{ otherwise.}
    \end{cases}
\end{align*}
If $\mv_{[1]}\geq 0$, then the optimal value of the program  in \eqref{opt: hinge with v} is $\mv_1+\mv_2$.
\end{lemma}

\begin{proof}
Let us define
\begin{align}
    \label{inlemma: hinge: def: tph with y}
    \tph(\my)=\sum_{j\in[k_2]}\slb 1-\mp_j\srb\sum\limits_{i\in[k_1-1]}\max(\mv_i,\my_j).
\end{align}
Note that 
\begin{equation*}
    \tph(\my)\geq \sum\limits_{i\in[k_1-1]}\mv_i\sum_{j\in[k_2]}\slb 1-\mp_j\srb=(k_2-1)\sum\limits_{i\in[k_1-1]}\mv_i
\end{equation*}
for all $\my\in\RR^{k_2}$.
If $\mv_{[1]}\geq 0$, then any $\my\in\RR^{k_1}$ with $\max(\my)\leq \mv_{[1]}$ satisfies $\tph(\my)=\sum\limits_{i\in[k_1-1]}\mv_i\sum_{j\in[k_2]}\slb 1-\mp_j\srb$. In particular, we can take $\my=\mz_{k_2}$.  
The result for $\mv_{[1]}>0$ follows noting $\mp\in\S^{k_2-1}$, which implies $\sum_{j\in[k_2]}( 1-\mp_j)=1$ when $k_2=2$. Hence, we consider $\mv_{[1]}<0$. 
Note also that
the optimum value of the program \eqref{opt: hinge with v}
        equals 
 \begin{mini*}|s|
{\my\in\RR^{k_2}}{\sum_{j\in[k_2]}\slb 1-\mp_j\srb\sum\limits_{i\in[k_1-1]}\max(\mv_{[i]},\my_j)}
{}{}
\addConstraint{\sum_{j\in[k_2]}\my_j}{=0.}{}
\end{mini*}
Therefore, it suffices to show the proof considering $\mv$ is sorted, i.e., $\mv_1\leq\ldots\leq \mv_{k_1-1}$ and $\mv_1<0$. 
In this case, from Lemma \ref{hinge lemma: aux: 1} it follows that the optimum value of the optimization program in \eqref{opt: hinge with v} equals the optimal value of the program
   \begin{mini}|s|
{\my\in\RR^{k_2}}{\sum_{j\in[k_2]}\slb 1-\mp_j\srb\sum\limits_{i\in[k_1-1]}\max(\mv_i,\my_j)}
{\label{opti: hinge: small case 0}}{}
\addConstraint{\sum_{j\in[k_2]}\my_j}{=0}{}
\addConstraint{\my_j}{\geq \mv_1}{\text{ for all }j\in[k_2].}
\end{mini}
When $k_2=2$, we can represent $\my$ by $\my=(x,-x)$ for some $x\in\RR$. Then the optimization problem in \eqref{opti: hinge: small case 0} becomes
   \begin{mini}|s|
{x\in\RR}{\mp_2\slb \max(\mv_1,x)+\max(\mv_2,x)\srb+\mp_1\slb \max(\mv_1,-x)+\max(\mv_2,-x)\srb}
{\label{inlemma: hinge: small case}}{}
\addConstraint{x,-x}{\geq \mv_1.}{}
\end{mini}
When $x,-x\geq \mv_1$, it follows that
\begin{align*}
   \MoveEqLeft \mp_2\slb \max(\mv_1,x)+\max(\mv_2,x)\srb+\mp_1\slb \max(\mv_1,-x)+\max(\mv_2,-x)\srb\\
   =&\ x(\mp_2-\mp_1)+\mp_2\max(\mv_2,x)+\mp_1\max(\mv_2,-x).
\end{align*}
Thus the optimization problem in \eqref{inlemma: hinge: small case} rewrites as 
    \begin{mini}|s|
{x\in\RR}{ x(\mp_2-\mp_1)+\mp_2\max(\mv_2,x)+\mp_1\max(\mv_2,-x)}
{\label{inlemma: hinge: small case 2}}{}
\addConstraint{|x|}{\leq -\mv_1,}{}
\end{mini}
where we used the fact that $x,-x\geq \mv_1$ is equivalent to $|x|\leq -\mv_1$ when $\mv_1<0$.
We overload notation and define
\[\tph(x):= x(\mp_2-\mp_1)+\mp_2\max(\mv_2,x)+\mp_1\max(\mv_2,-x).\]
Let us define $-v_1=c>0$. Therefore,
the optimal value of \eqref{inlemma: hinge: small case 2} equals $\inf_{x\in[-c,c]}\tph(x)$.
Since $v_2\geq v_1$, there can be three cases.
\subparagraph{Case 1: $\mv_2=\mv_1$} Since $x\in[-c,c]$, it follows that $x,-x\geq c_2$ implying   $\tph(x)=2x(\mp_2-\mp_1)$. Then 
\[\inf_{x\in[-c,c]}\tph(x)=-2|\mp_2-\mp_1|c=-2|\mv_2(\mp_1-\mp_2)|\]
\begin{align*}
  \argmin\limits_{x\in[-c,c]}\tph(x)=\begin{cases}
      c &\ \mp_1>\mp_2\\
      -c & \mp_2< \mp_1\\
      [-c,c] &\ \text{ if }\mp_1=\mp_2.
  \end{cases}  
\end{align*}
Let us call the above number $x^*$.
Suppose $\my^*=(x^*,-x^*) $
Then
\begin{align*}
\argmax(\my^*)=   \begin{cases}
      1 &\ \text{ if }\mp_1>\mp_2\\
      2 & \text{ if }\mp_2< \mp_1\\
     \{1,2\} &\ \text{ if }\mp_1=\mp_2.
  \end{cases}  
\end{align*}

\subparagraph{Case 2: $\mv_1<\mv_2<0$}
In this case, $-c<\mv_2<0<-\mv_2<c$, indicating

  \[  \tph(x)= \begin{cases}x(\mp_2-\mp_1)+\mp_2\mv_2-\mp_1 x & \text{ if } x\in[-c,\mv_2)\\
  2x(\mp_2-\mp_1)   & \text{ if } x\in(\mv_2,-\mv_2]\\
  2x\mp_2-\mp_1x+\mp_1\mv_2 & x\in(-\mv_2,c].
\end{cases}\]
Since $\tph(x)$ is a piecewise linear function, the knots will be in $\argmin_{x\in[-c,c]}\tph(x)$ where the knots belong to the set $\{-c, \mv_2,0,-\mv_2,c\}$.
Note that 
\begin{align*}
    \tph(x)=\begin{cases} 2\mp_1 c-\mp_2c+\mp_2\mv_2 & \text{if } x=-c\\
    2\mv_2(\mp_2-\mp_1) & \text{if }x=\mv_2\\
    -2\mv_2(\mp_2-\mp_1) & \text{if }x=-\mv_2\\
    2c\mp_2-\mp_1 c+\mp_1\mv_2 & \text{if }x=c.
     \end{cases}
\end{align*}
 Therefore,
\begin{align*}
    \inf_{x\in[-c,c]}\tph(x)=\min\slb 2\mp_1 c-\mp_2c+\mp_2\mv_2,\  -2|\mv_2(\mp_2-\mp_1)|,\  2c\mp_2-\mp_1 c+\mp_1\mv_2\srb.
\end{align*}
Let us denote
\[h_c(y;c,v_2)=2y c-(1-y)c+(1-y)\mv_2=3yc-c+(1-y)\mv_2.\]
Then \begin{align*}
    \inf_{x\in[-c,c]}\tph(x)=\min\slb h_c(\mp_1),\  h_c(\mp_2),\ -2|\mv_2(\mp_2-\mp_1)|\srb.
\end{align*} 
It can be easily seen that $h_c(\mp_1)\gtreqless h_c(\mp_2)$ iff $\mp_1\gtreqless \mp_2$. Therefore,
\begin{align}
\label{inlemma: hinge: hc p1 p2}
\min(h_c(\mp_1),h_c(\mp_2))=h_c(\min(\mp_1,\mp_2)).
\end{align}
Hence,
\begin{align*}
    \inf_{x\in[-c,c]}\tph(x)=\min\slb h_c(\mp_1\wedge\mp_2),\ -2|\mv_2(\mp_2-\mp_1)|\srb.
\end{align*} 
Now note that
\begin{equation}
    \label{inlemma: hinge: hc: alt form 1}
    h_c(\mp_1\wedge \mp_2)=2c(\mp_1\wedge \mp_2)+(\mp_1\vee\mp_2 )(\mv_2-c)=-c|\mp_1-\mp_2|+c(\mp_1\wedge \mp_2)+\mv_2(\mp_1\vee\mp_2).
\end{equation}
Therefore,
\[h_c(\mp_1\wedge\mp_2)\gtreqless -2|\mv_2(\mp_2-\mp_1)|\]
if and only if 
\[-c|\mp_1-\mp_2|+c(\mp_1\wedge \mp_2)+\mv_2(\mp_1\vee\mp_2)\gtreqless 2\mv_2|\mp_2-\mp_1|\]
where we used the fact that $\mv_2<0$. The last display is equivalent to 
\[(\mv_2+c)(\mp_1\wedge\mp_2)\gtreqless (\mv_2+c)|\mp_2-\mp_1|.\]
Since $c+\mv_2>0$, the above is equivalent to 
\[\mp_1\wedge\mp_2\gtreqless |\mp_2-\mp_1|\text{ or }\mp_1\wedge\mp_2\gtreqless \frac{\mp_1\vee\mp_2}{2}.\]
Therefore, we obtained that
\begin{align*}
 \inf_{x\in[-c,c]}\tph(x)=\begin{cases}
   -c|\mp_1-\mp_2|+c(\mp_1\wedge \mp_2)+\mv_2(\mp_1\vee\mp_2) &   \text{ if }\mp_1\wedge\mp_2\leq \frac{\mp_1\vee\mp_2}{2}\\
   -2|\mv_2(\mp_2-\mp_1)| & \text{ otherwise.}
 \end{cases}    
\end{align*}
\subparagraph{Case 3: $0\leq \mv_2<c$}
In this case, $-c<0\leq \mv_2<c$. Therefore,
\begin{align*}
    \tph(x)=\begin{cases}
   x(\mp_2-\mp_1)+\mp_2\mv_2-\mp_1 x & \text{ if } x\in[-c,-\mv_2)\\
  x(\mp_2-\mp_1) +\mv_2  & \text{ if } x\in(-\mv_2,\mv_2]\\
  2x\mp_2-\mp_1x+\mp_1\mv_2 & x\in(\mv_2,c]     
    \end{cases}
\end{align*}
Note that 
\begin{align*}
    \tph(x)=\begin{cases} 2\mp_1 c-\mp_2c+\mp_2\mv_2 & \text{if } x=-c\\
    \mv_2(\mp_1-\mp_2+1) & \text{if }x=-\mv_2\\
    \mv_2(\mp_2-\mp_1+1) & \text{if }x=\mv_2\\
    2c\mp_2-\mp_1 c+\mp_1\mv_2 & \text{if }x=c.
     \end{cases}
\end{align*}
 Therefore, $  \inf_{x\in[-c,c]}\tph(x)$ equals
\begin{align*}
  \MoveEqLeft \min\slb 2\mp_1 c-\mp_2c+\mp_2\mv_2,\  -\mv_2|\mp_2-\mp_1|+\mv_2,\  2c\mp_2-\mp_1 c+\mp_1\mv_2\srb\\
=&\ \min\slb h_c(\mp_1),\  h(\mp_2),\ -\mv_2|\mp_2-\mp_1|+\mv_2\srb\\
\stackrel{(a)}{=}&\ \min\slb h_c(\mp_1\wedge\mp_2),\ -\mv_2|\mp_2-\mp_1|+\mv_2\srb\\
\stackrel{(b)}{=}&\  \min\slb -c|\mp_1-\mp_2|+c(\mp_1\wedge \mp_2)+\mv_2(\mp_1\vee\mp_2),\ -\mv_2|\mp_2-\mp_1|+\mv_2\srb.
\end{align*}
 where (a) follows from \eqref{inlemma: hinge: hc p1 p2} and (b) follows from \eqref{inlemma: hinge: hc: alt form 1}.
 It can be easily seen that
 \[-c|\mp_1-\mp_2|+c(\mp_1\wedge \mp_2)+\mv_2(\mp_1\vee\mp_2)\lesseqgtr -\mv_2|\mp_2-\mp_1|+\mv_2\]
 if and only if
 \[(c-\mv_2)(\mp_1\wedge \mp_2)\lesseqgtr (c-\mv_2)|\mp_1-\mp_2|.\]
 Since $c>\mv_2$ in this case, the above is equivalent to $\mp_1\wedge \mp_2\lesseqgtr|\mp_1-\mp_2|$ or $\mp_1\wedge \mp_2\lesseqgtr(\mp_1\vee\mp_2)/2$. Therefore, we have obtained that
 \begin{align*}
   \inf_{x\in[-c,c]}\tph(x)=\begin{cases}
    -c|\mp_1-\mp_2|+c(\mp_1\wedge \mp_2)+\mv_2(\mp_1\vee\mp_2)   & \text{ if }\mp_1\wedge \mp_2\leq(\mp_1\vee\mp_2)/2\\
  -|\mv_2(\mp_2-\mp_1)|+\mv_2  & \text{ otherwise. }
   \end{cases}    
 \end{align*}
\subparagraph{Case 4: $c\leq  \mv_2$}
In this case, $-\mv_2\leq -c<c\leq\mv_2$. Thus $-x,x\leq \mv_2$ if $x\in[-c,c]$ Therefore,
\begin{align*}
    \tph(x)=
   x(\mp_2-\mp_1)+\mv_2   
    \end{align*}
     Therefore,
\begin{align*}
    \inf_{x\in[-c,c]}\tph(x)=-|\mp_2-\mp_1|c+\mv_2
\end{align*}
\subparagraph{Combining cases 1-4  for $\mp_1\wedge \mp_2\leq(\mp_1\vee\mp_2)/2$}
The above steps imply that if $\mp_1\wedge \mp_2\leq(\mp_1\vee\mp_2)/2$, then
\begin{align*}
    \inf_{x\in[-c,c]}\tph(x)=\begin{cases}
       2 \mv_2|\mp_1-\mp_2| &  \text{ if } \mv_1=\mv_2\\
    \mv_1|\mp_1-\mp_2|-\mv_1(\mp_1\wedge \mp_2)+\mv_2(\mp_1\vee\mp_2)    & \text{ if }\mv_1< \mv_2<-\mv_1\\
        \mv_1|\mp_1-\mp_2|+\mv_2 & \text{ if }-\mv_1\leq \mv_2.
    \end{cases}
\end{align*}
When $\mv_1=\mv_2$,
\[  -\mv_1(\mp_1\wedge \mp_2)+\mv_2(\mp_1\vee\mp_2) =\mv_1\slb \mp_1\vee\mp_2-\mp_1\wedge\mp_2\srb=\mv_1\abs{\mp_1-\mp_2}.\]
Therefore, if $\mp_1\wedge \mp_2\leq(\mp_1\vee\mp_2)/2$, then
\begin{align*}
    \inf_{x\in[-c,c]}\tph(x)=\begin{cases}
    \mv_1|\mp_1-\mp_2|-\mv_1(\mp_1\wedge \mp_2)+\mv_2(\mp_1\vee\mp_2)    & \text{ if }\mv_1\leq \mv_2<-\mv_1\\
        \mv_1|\mp_1-\mp_2|+\mv_2 & \text{ if }-\mv_1\leq \mv_2.
    \end{cases}
\end{align*}
However, when $-\mv_1\leq \mv_2$,
\begin{align*}
   \mv_1|\mp_1-\mp_2|+\mv_2= &\ \mv_1|\mp_1-\mp_2|+\mv_2(\mp_1\vee\mp_2)+\mv_1\mp_1\wedge\mp_2\\
   =&\   \mv_1|\mp_1-\mp_2|+\mv_2(\mp_1\vee\mp_2)+\max(\mv_2,-\mv_1)\mp_1\wedge\mp_2,
\end{align*}
and when $-\mv_1> \mv_2$,
\[ \mv_1|\mp_1-\mp_2|-\mv_1(\mp_1\wedge \mp_2)+\mv_2(\mp_1\vee\mp_2)= \mv_1|\mp_1-\mp_2|+\max(\mv_2,-\mv_1)\mp_1\wedge \mp_2+\mv_2(\mp_1\vee\mp_2).\]
Therefore,
if $\mp_1\wedge \mp_2\leq(\mp_1\vee\mp_2)/2$, then
\begin{align}
\label{inlemma: hinge: sol for one case}
    \inf_{x\in[-c,c]}\tph(x)=\mv_1|\mp_1-\mp_2|+\max(\mv_2,-\mv_1)(\mp_1\wedge \mp_2)+\mv_2(\mp_1\vee\mp_2).
\end{align}
\subparagraph{Combining cases 1-4 for $\mp_1\wedge \mp_2>(\mp_1\vee\mp_2)/2$}

The above cases imply that when  $\mp_1\wedge \mp_2>(\mp_1\vee\mp_2)/2$, then
\begin{align*}
  \inf_{x\in[\mv_1,-\mv_1]}  \tph(x)=\begin{cases}
   -2|\mv_2(\mp_2-\mp_1)| & \text{ if } \mv_1<\mv_2<0\\
   -|\mv_2(\mp_2-\mp_1)|+\mv_2 & \text{ if }0\leq \mv_2 <-\mv_1\\
        \mv_1|\mp_1-\mp_2|+\mv_2 & \text{ if }-\mv_1\leq \mv_2,
    \end{cases}
\end{align*}
which is equivalent to
\begin{align*}
 \inf_{x\in[\mv_1,-\mv_1]}   \tph(x)=\begin{cases}
   -2|\mv_2(\mp_2-\mp_1)| & \text{ if } \mv_1<\mv_2<0\\
   -|\min(\mv_2,-\mv_1)(\mp_2-\mp_1)|+\mv_2 & \text{ if }0\leq \mv_2
    \end{cases}  
\end{align*}
When $\mv_1<\mv_2<0$,
\begin{align*}
 \MoveEqLeft -2|\mv_2(\mp_2-\mp_1)|\\
 =&\ -|\mv_2(\mp_2-\mp_1)| +\mv_2(\mp_1\vee \mp_2)-\mv_2(\mp_1\wedge \mp_2)\\
 =&\  -|\min(\mv_2,-\mv_1)(\mp_2-\mp_1)|+\mv_2(\mp_1\vee \mp_2)-\mv_2(\mp_1\wedge \mp_2)\\
 =&\ -|\min(\mv_2,-\mv_1)(\mp_2-\mp_1)|+\mv_2(\mp_1\vee \mp_2)+|\mv_2|(\mp_1\wedge \mp_2).
\end{align*}
When $0\leq \mv_2$,
\[\mv_2=\mv_2(\mp_1\vee \mp_2)+\mv_2(\mp_1\wedge \mp_2)=\mv_2(\mp_1\vee \mp_2)+|\mv_2|(\mp_1\wedge \mp_2).\]
Therefore, we have obtained that when  $\mp_1\wedge \mp_2>(\mp_1\vee\mp_2)/2$,
\begin{equation}
    \label{inlemma: hinge: final: case 2}
   \inf_{x\in[\mv_1,-\mv_1]}\tph(x)=  -|\min(\mv_2,-\mv_1)(\mp_2-\mp_1)|+\mv_2(\mp_1\vee \mp_2)+|\mv_2|(\mp_1\wedge \mp_2)
\end{equation}
Hence, the proof follows combining \eqref{inlemma: hinge: sol for one case} and \eqref{inlemma: hinge: final: case 2}.


\end{proof}

\begin{lemma}
    \label{hinge lemma: aux: 1}
    Consider the optimization problem in \eqref{opt: hinge with v}.
     If $\mv_1\leq\ldots\leq \mv_{k_1-1}$ and  $\mv_1< 0$, then for any $k_1,k_2\geq 2$,  the optimal value of \eqref{opt: hinge with v} equals 
\[\inf_
{\my\in\C_1}{\sum_{j\in[k_2]}\slb 1-\mp_j\srb\sum\limits_{i\in[k_1]:i\neq A_1}\max(\mv_i,\my_j)},
\]
where
\[\C_1=\lbs\my\in\RR^{k_2}: \sum_{j\in[k_2]}\my_j=0,\ \my_i\geq \mv_1\rbs.\]
\end{lemma}

\begin{proof}[Proof of Lemma \ref{hinge lemma: aux: 1}]
    Let us denote the set $\C=\{\my\in\RR^{k_2}: \sum_{j\in[k_2]}\my_j=0\}$.  Hence, the optimal value of \eqref{opt: hinge with v} equals  $\inf_{\my\in\C}\tph(\my)$ where $\tph$ is as defined in \eqref{inlemma: hinge: def: tph with y}.
We will show that $\inf_{\my\in\C}\tph(\my)=\inf_{\my\in\C_1}\tph(\my)$.
If possible, suppose $i\in[k_2]$ is such that $\my_i<\mv_1$. Let $l\neq i$ and $l\in[k_2]$.  Our next step is to show that  
$\tph(\my)\geq \tph(\my')$ where $\my'$ is the vector obtained by replacing $\my_{i}$ with $\mv_{1}$  and replacing $\my_l$ with $\my'_l=\my_l+\my_{i}-\mv_1$. 
 Clearly, $\my_l'<\my_l$ and $\my_i'>\my_i$. Note that $\my'\in\C$ and $ \tph(\my)-\tph(\my')$ equals
\begin{align*}
   \MoveEqLeft \slb 1-\mp_i\srb \sum\limits_{r\in[k_1-1]}\max(\mv_r,\my_i)+ \slb 1-\mp_l\srb \sum\limits_{r\in[k_1-1]}\max(\mv_r,\my_l)\\
    &\ -\slb 1-\mp_i\srb \sum\limits_{r\in[k_1-1]}\max(\mv_r,\mv_1)- \slb 1-\mp_l\srb \sum\limits_{r\in[k_1-1]}\max(\mv_r,\my_l+\my_{i}-\mv_1)\\
    =&\ (1-\mp_i)\lb \sum\limits_{r\in[k_1-1]}\slb \max(\mv_r,\my_i)-\max(\mv_r,\mv_1)\srb\rb\\
    &\ + (1-\mp_l)\lb \sum\limits_{r\in[k_1-1]}\slb \max(\mv_r,\my_l)- \max(\mv_r,\my_l+\my_{i}-\mv_1)\srb\rb\\
     =&\     (1-\mp_l)\lb \sum\limits_{r\in[k_1-1]}\slb \max(\mv_r,\my_l)- \max(\mv_r,\my_l+\my_{i}-\mv_1)\srb\rb
\end{align*}
where the last equality follows because $\max(\mv_r,\my_i)=\max(\mv_r,\mv_1)=\mv_r$ or all $r\in[k_1-1]$ since $\my_i<\mv_1\leq \mv_r$ for all $r\in[k_1-1]$. However, since $\my_l+\my_{i}-\mv_1<\my_l$, we obtain $\tph(\my)\geq \tph(\my')$. We can keep replacing $\my_i$ with $\mv_1$ if $\my_i<\mv_1$ until all elements of $\my'$ are greater than equal to $\mv_1$. Thus we have shown that $\inf_{\my\in\C_1}\tph(\my)\leq \inf_{\my\in\C}\tph(\my)$, which, combined with $\C_1\subset\C$, implies $\inf_{\my\in\C_1}\tph(\my)=\inf_{\my\in\C}\tph(\my)$.

\end{proof}
\section{Proof of Theorem \ref{theorem: CC}}
\label{sec: pf of theorem cc}

\subsection{Proof preparation}
\label{sec: intheorem: cc: terminologies}

Before proceeding with the proof of Theorem \ref{theorem: CC}, we first gather some key facts that will be utilized throughout the proof and briefly outline the main steps involved.

We say $f:\RR^k\mapsto\RR$ is lower semicontinuous at $\mx_0\in\RR^k$ if $\liminf_{\mx\to \mx_0}f(\mx)\geq f(\mx_0)$. Following Definition 1.2.3., pp. 78 of \cite{hiriart},  $f$ is  lower semicontinuous everywhere if and only if it is closed. 
Following  Section 3.2 of \cite{hiriart}, we define a 0-coercive convex function  as follows. 
  \begin{definition}[0-coercivity]
  \label{def: coercivity}
  A closed convex function $f:\RR^k\mapsto\RR$ is called 0-coercive if $f(x)\to\infty$ as $\|x\|_2\to\infty$.
  \end{definition}
 A 0-coercive convex function is necessarily bounded below.
     An associated notion is that of the recession function.
     \begin{definition}
  \label{def: recession functions}
  Suppose the function $f:\RR^k\to\RR$ is closed and convex. Then for any $x\in\RR^k$, the function $f_\infty'$ defined by
  \[f_\infty'(x)=\lim_{t\to\infty}\frac{f(tx)-f(0)}{t}\]
  is known as the recession function of $f$.
  \end{definition}
   It is evident that $f_\infty'(0) = 0$. The behavior of $f_\infty'$ for $x \neq 0$ is closely related to the coercivity of $f$ and can inform us on the existence of a minimum. The following result, synthesizing Remark 3.2.7 and Proposition 3.2.4 from \cite{hiriart}, establishes key properties of recession functions in relation to coercivity.

 \begin{fact}
  \label{fact: concave: coercivity and recession for concave}
   Suppose $f$ is  convex and bounded below. Then
    $f_\infty'(x)\geq 0$ for all $x\neq 0$. In fact, there are only two possibilities: (1)
      $f$ is 0-coercive, and the set of minimum points of $f$ is a non-empty compact set (singletone if $f$ is strictly convex) -- this happens if $f_\infty'(x)>0$ for all $x\neq 0$, or (2) the minimum of $f$ is unattained -- this happens  if $f_\infty'(x)=0$ for some $x\neq 0$.
  \end{fact}
Fact~\ref{fact: concave: coercivity and recession for concave} suggests that we can deduce whether the minimum  of $f$ is achieved  by evaluating the recession functions of $f$. The subsequent fact provides a useful tool for this purpose.
  \begin{fact}[Proposition 3.2.8 of \cite{hiriart}]
  \label{fact: concave: prop 3.2.8 of hiriat}
      \vphantom{here\\}
  \begin{enumerate}
      \item[a] Consider $m\in\NN$ and let $a_1,\ldots,a_m$ be positive numbers. Then for closed convex functions $f_1,\ldots,f_m:\RR^k\mapsto\RR$, we have
  \[\slb \sum_{i=1}^ma_if_i\srb_\infty'=\sum _{i=1}^ma_i(f_i)_\infty'.\]
  \item[b] Suppose $h(x_1,\ldots,x_k)=f(a_1x_1,\ldots,a_kx_k)$ where $a_i\in\{\pm 1\}$ for $k=1,\ldots,m$ and $f$ is convex. Then \[h_\infty'(x_1,\ldots,x_k)=f_\infty'(a_1x_1,\ldots,a_kx_k).\]
  \item[c] Let $\text{Range}(A)$ denote the range of $A$ for any matrix $A$ with $k$ columns, where $k\in\NN$. If $h(\mx)=f(A\mx)$ for all $\mx\in\RR^k$ for a convex function $f$ and $\dom(f)\cap \text{Range}(A)\neq \emptyset$, then  $(f\circ A)'_\infty=f_\infty'\circ A$.  
  \end{enumerate}
  \end{fact}



 Now we will collect some initial results that will be used throughout the proof.
 From \cite{wang2023unified}, it follows that  
 if $\ppi:[k_1]\mapsto[k_1]$, $\tpi:[k_2]\mapsto[k_2]$ are two permutations and $\mx\in\RR^{k_1}$ and $\mw\in\RR^{k_2}$, then 
 \begin{align}
     \label{def: cc: permutation equivariance 2 stage}
     \psi(\mx, \mw; i,j)=\psi(\ppi(\mx),\tpi(\mw);\ppi^{-1}(i),\tpi^{-1}(j))=\psi(\ppi^{-1}(\mx),\tpi^{-1}(\mw); \ppi(i),\tpi(j)).
 \end{align}
 We refer to \eqref{def: cc: permutation equivariance 2 stage} as the clockwise permutation symmetry of $\psi$. 
 The above implies $\psi(\mx,\mw;i,j)$ is permutation equivariant w.r.t. $\mx$ and $i$ when $\mw$ and $j$ are fixed and vice versa. 
 Using this result, we can prove that $\mz_{k_1+k_2-2}\in\iint(\dom(-\eta))$ under the setup of Theorem \ref{theorem: CC}. Lemma \ref{lemma: CC: 0 in int} establishes this fact.  The proof of Lemma \ref{lemma: CC: 0 in int} can be found in Section \ref{secpf: of lemma  CC: 0 in int}. 
\begin{lemma}
\label{lemma: CC: 0 in int}
    Suppose $\psi$ is a PERM loss with template $\eta$, which is concave and bounded above. If  $\cap_{i=1}^{k_1}\cap_{j=1}^{k_2}\iint(\dom(\psi(\cdot;i,j))\neq \emptyset$, then $\mz_{k_1+k_2-2}\in\iint(\dom(-\eta))$.
\end{lemma}
Since $-\eta$ is bounded below, if $\dom(-\eta)$ is non-empty, then $-\eta$ is proper. Therefore Lemma \ref{lemma: CC: 0 in int} implies  $-\eta$ is a proper convex function.
\subsubsection{The main steps}
We will prove this theorem in two main steps. 
The main goal of the first step is to reduce the general problem to a  problem with $k_1=2$ and $k_2=2$. 
 In the second step, from a high-level, we first show that the concavity of $\psi$ enforces some geometric restrictions on $\tilde f_1$ and $\tilde f_2$ in the  corresponding binary-treatment  problem. Then  we will show that if  $\psi$ is Fisher consistent, then it also induces some restrictions on  $\tilde f_1$ and $\tilde f_2$. Using a  bad set of distributions, we will show that these restrictions  get violated in presence of the abovementioned concavity-induced restrictions. The structure of this section is as follows: Step 1 is proved in Section \ref{sec: intheorem: cc: step 1}, and Step 2 is proved in Section \ref{sec: CC: step 2}. The main lemmas are proved in Section \ref{sec: cc: proof of main lemmas}, while the proofs of additional results are deferred to Section \ref{secpf: cc: proof of additional lemmas}.
 


\subsection{Step 1}
\label{sec: intheorem: cc: step 1}
The first task of this step is to construct a sublass of distributions for which the $\tilde f$ corresponding to our $\psi$ exists and the forms of $\tilde f$, $\tilde d$, $d^*$, etc. are tractable. Then we will find a further subclass of distributions, for which, the DTR optimization problem for $k_1$ and $k_2$ treatment levels reduce to a DTR optimization problem with 2 treatment levels per stage. 

We consider a subset $\mP^{k_1,k_2}$  of  $\mP^0$, the space all distributions satisfying Assumptions I-V. The distributions in $\mP^{k_1,k_2}$ entail the following setting.
For these distributions, $Y_1=0$ and $O_1,O_2=\emptyset$. Therefore, there is no covariate. We could add covariates in this counterexample, but the main idea of the construction of the counterexample would be similar. However, the calculation would be more technical. Since $\mP^{k_1,k_2}\subset\mP^0$, it follows that $\E[Y_2\mid i,j]>0$ for all $i\in[k_1]$ and $j\in[k_2]$.

 Since  $H_1=\emptyset$, under this setting, the first stage class-score function $f_1$ can be represented by a vector in $\RR^k$. The corresponding $d_1=\argmax_{i\in[k_1]}f_1(i)$ will be a number.  Also, since $H_2=A_1$, $\H_2=[k_1]$. Thus $f_2$ is a map from $[k_1]$ to $\RR^{k_2}$ and the second stage DTR $d_2$ is a map from $[k_1]$ to $[k_2]$. For each $i\in [k_1]$, the class score  $f_2(i)$ is a $k_2$-dimesnional real vector. Therefore, the function  $f_2$ can be completely explained by the $k_1$ vectors $\{f_2(1),\ldots,f_2(k_1)\}$. 
 Since $H_2=A_1$ for distributions in $\mP^{k_1,k_2}$, the expectation  $\E[Y_1+Y_2\mid H_2,A_2]=\E[Y_2\mid A_1,A_2]$.
  For $i\in[k_1]$ and $j\in[k_2]$, let us denote $\E[Y_2\mid A_1=i, A_2=j]$ by $\mp_{ij}$. Since for $\PP\in\mP^{k_1,k_2}$,  $\E[Y_2\mid A_1, A_2]>0$,  it follows that $\mp_{ij}>0$ for all $i\in[k_1]$ and $j\in[k_2]$. Note that every $\PP\in\mP^{k_1,k_2}$ elicits one set of $\mp_{ij}$'s. Moreover, given any set  $\{\mp_{ij}\}_{i\in[k_1],j\in[k_2]}$ with positive $\mp_{ij}$'s, we can construct a $\PP\in\mP^{k_1,k_2}$ so that $\E[Y_2\mid A_1=i, A_2=j]=\mp_{ij}$. It can also be shown that for any DTR $d=(d_1,d_2)$, the value function $V(d_1,d_2)=\mp_{d_1,d_2(d_1)}$. Thus,  the contribution of $\PP\in\mP^{k_1,k_2}$ on the value function reflects only via  $(\mp_{ij})$. 
  
  For each $A_1=i$, any number in the set  $\argmax_{j\in[k_2]} \mp_{ij}$ is a candidate of   the optimal second stage treatment assignment $d^*_2(i)$. Also, since $H_1=\emptyset$, it follows that   $d_1^*$ can be any number in the set $\argmax_{i\in[k_1]} Q^*_1(i)$.
Under our setting, it follows that 
\[Q^*_1(i)=\E[\max_{j\in[k_2]}\E[Y_2\mid A_1, A_2=j]\mid A_1=i]=\E[\max_{j\in[k_2]}\mp_{A_1 j}\mid A_1=i]=\max\limits_{j\in[k_2]}\mp_{i j}\]
Thus any member of $ \argmax_{i\in[k_1]}(\max_{j\in[k_2]}\mp_{i j})$ is a candidate for $d_1^*$. Also, since $V_*=\max_{i\in[k_1]} Q^*_1(i)$, 
\begin{align}
\label{intheorem: eq: V star: initial}
    V_*=\max\limits_{i\in[k_1],j\in[k_2]}\mp_{ij}.
\end{align}
The surrogate value function  for $\PP\in\mP^{k_1,k_2}$ takes the form
 \begin{align}
 \label{intheorem: cc: expression: surrogate loss Pp}
  V^\psi(f_1,f_2) :=&\ \E\lbt \frac{Y_2}{\pi_1(A_1)\pi_2(A_2\mid A_1)}\psi(f_1,f_2(A_1); A_1,A_2)\rbt\nn\\
  =&\ \E\lbt \frac{\E[Y_2\mid A_1, A_2]}{\pi_1(A_1)\pi_2(A_2\mid A_1)}\psi(f_1,f_2(A_1);A_1,A_2)\rbt\nn\\
 =&\ \sum_{j\in[k_2]}\E\lbt \frac{\E[Y_2\mid A_1, A_2=j]}{\pi_1(A_1)}\psi(f_1,f_2(A_1);A_1,j)\rbt\nn\\
 =&\ \sum_{i\in[k_1]}\sum_{j\in[k_2]}\E[Y_2\mid i,j]\psi(f_1,f_2(i);i,j)\nn\\
 =&\ \sum_{i\in[k_1]}\sum_{j\in[k_2]}\mp_{ij}\psi(f_1,f_2(i);i,j).
    \end{align}
   Therefore  $V^{\psi}(f_1,f_2)$ depends on $\PP$ only via $\mp_{ij}$'s as long as $\PP\in\mP^{k_1,k_2}$.  Let us denote $\mp=(\mp_{ij})$. From now on, we will denote $V^\psi(f_1,f_2)$ by $V^\psi(f_1,f_2;\mp)$ to make the dependence on the underlying probability distribution $\PP$ explicit.
  
We will introduce some more new notation.  Since $f_1\in\RR^k$ and $f_2(i)\in\RR^{k_2}$,  $f_1$ can be represented by a vector $\mx\in\RR^{k_1}$ and each $f_2(i)$ can be represented by a vector $\my_i\in\RR^{k_2}$, $i\in[k_1]$. 
As $f_2=(f_2(1),\ldots,f_2(k_1))$, we can describe $f_2$  by the $k_1k_2$-dimensional vector $\Y=(\my_1,\ldots, \my_{k_1})$. For the rest of this proof, unless otherwise specified, we will use the notation $\mx$ instead of $f_1$ and $(\my_1,\ldots,\my_{k_1})$ or $\Y$ instead of $f_2$ unless otherwise specified. Therefore, 
   $V^{\psi}(f_1,f_2;\mp)$ will also be denoted by $V^{\psi}(\mx,\my_1,\ldots,\my_{k_1};\mp)$ for the rest of this proof.  Using the above notation,  \eqref{intheorem: cc: expression: surrogate loss Pp} can be rewritten as
  \begin{equation}
      \label{def: V psi: many ps}
      V^{\psi}(\mx,\my_1,\ldots,\my_{k_1};\mp)=\sum_{i\in[k_1]}\sum_{j\in[k_2]}\mp_{ij}\psi(\mx,\my_i;i,j).
  \end{equation}
Suppose $\argmax_{\mx\in\RR^{k_1},\my_1,\ldots,\my_{k_1}\in\RR^{k_2}} V^{\psi}(\mx,\my_1,\ldots,\my_{k_1};\mp)$ is non-empty for some $\mp$ and
  \begin{align}
      \label{intheorem: cc: def x star and y star}
(\mx^*(\mp),\my_{k_1}^*(\mp),\ldots,\my_{k_2}^*(\mp))&\ \equiv (\mx^*,\my_{k_1}^*,\ldots,\my_{k_2}^*)\nn\\
&\ \in\argmax_{\mx\in\RR^{k_1},\my_1,\ldots,\my_{k_1}\in\RR^{k_2}} V^{\psi}(\mx,\my_1,\ldots,\my_{k_1};\mp).
  \end{align}
Then for any $\PP$ corresponding to $\mp$, $\mx^*(\mp)$ is a candidate for $\tilde f_1$ and $\my_i^*(\mp)$ is a candidate for $\tilde f_2(i)$ for all $i\in[k_1]$. If $\tilde d_1$ and $\tilde d_2$ are the policies generated by $\tilde f_1$ and $\tilde f_2$, then $\tilde d_1=\pred(\mx^*(\mp))$ and $\tilde d_2(i)=\pred(\my_i^*(\mp))$ for each $i\in[k_1]$. Since $V(d_1,d_2)=\mp_{d_1,d_2(d_1)}$ for any DTR $d$, $V(\tilde d_1,\tilde d_2)=\mp_{\tilde d_1, \tilde d_2(\tilde d_1)}$.

    
For an arbitrary $\mp$, there is no guarantee that the maximizer of $V^\psi(\cdot;\mp)$ exists.
However, Lemma \ref{lemma: cc: 0 maximum for smooth}, which is proved in Section \ref{secpf: proof of 0 solution lemma}, shows that \\$\argmax_{\mx\in\RR^{k_1},\my_1,\ldots,\my_{k_1}\in\RR^{k_2}} V^{\psi}(\mx,\my_1,\ldots,\my_{k_1};\mp)$ is none-empty   when  $\mp_{ij}=1$ for all $i\in[k_1]$ and $j\in[k_2]$. We will later see that the existence of the maximizer for one $\mp$ is sufficient for showing the existence of maximizer for all $\mp$'s. 

\begin{lemma}
\label{lemma: cc: 0 maximum for smooth}
Suppose 
\[\C_e=\lbs (x\mo_{k_1}, y\mo_{k_2},\ldots, y\mo_{k_2})\in\RR^{k_1(1+k_2)}: x,y\in\RR\rbs\]
    and  $\mp_{ij}=1$ for all $i\in[k_1]$ and $j\in[k_2]$. Under the setup of Theorem \ref{theorem: CC},  the $V_{\psi}$ defined in \eqref{def: V psi: many ps} satisfies
    \[\argmax\limits_{\mx\in\RR^{k_1},\my_1,\ldots,\my_{k_1}\in\RR^{k_2}}V^{\psi}(\mx,\my_1,\ldots,\my_{k_1};\mp)\supset \C_e.\]
\end{lemma}

We now consider a further subset $\mP_b^{k_1,k_2}\subset\mP^{k_1,k_2}$ so that for any $\PP\in\mP_b^{k_1,k_2}$, $\mp_{1,j}=\mp_{12}$ for all $j\geq 2$, $\mp_{i,1}=\mp_{21}$ for all $i>1$, and  $\mp_{ij}=\mp_{22}$ for all $j>1$ and $i>1$. Since it is possible to construct a $\PP \in \mP^{k_1,k_2}$ such that $\E[Y_2 \mid A_1 = i, A_2 = j] = \mp_{ij}$ for any $\mp \equiv {\mp_{ij}}{i \in [k_1], j \in [k_2]}$ with positive $\mp_{ij}$ values, it follows that for any $(\mp_{11}, \mp_{12}, \mp_{21}, \mp_{22}) \in \RR^4_{>0}$, there exists a $\PP \in \mP_b^{k_1,k_2}$ corresponding to these values. Since we will only consider $\PP\in\mP_b^{k_1,k_2}$ from now on, by an abuse of notation, we will denote $(\mp_{11}, \mp_{12}, \mp_{21}, \mp_{22})$ by $\mp$.
The surrogate loss reduces to
\begin{align}
\label{intheorem: CC: expansion: of Psi}
V^{\psi}(\mx,\my_1,\ldots,\my_{k_1};\mp)
=&\ \mp_{11}\psi(\mx,\my_1;1,1)+ \mp_{12} \sum_{j=2}^{k_2}\psi(\mx,\my_1;1,j)\nn\\
&\ +\mp_{21}\sum_{i=2}^{k_1}\psi(\mx,\my_i;i,1)
+\mp_{22}\sum_{i=2}^{k_1}\sum_{j=2}^{k_2}\psi(\mx, \my_i;i,j).   
\end{align}
Since $\psi$ is PERM with template $\eta$, \eqref{def: PERM} implies that $V^{\psi}(\mx,\my_1,\ldots,\my_{k_1};\mp)$ equals
\begin{align}
\label{intheorem: CC: psi to phi}
\MoveEqLeft \mp_{11}\eta(\mx_1-\mx_2,\ldots,\mx_1-\mx_{k_1},\my_{11}-\my_{12},\ldots,\my_{11}-\my_{1k_2})\nn\\
&\ +\mp_{12} \sum_{j=2}^{k_2}\eta(\mx_1-\mx_2,\ldots,\mx_1-\mx_{k_1},\my_{1j}-\my_{11},\ldots,\my_{1j}-\my_{1,k_2-1})\nn\\
&\ +\mp_{21}\sum_{i=2}^{k_1}\eta(\mx_i-\mx_1,\ldots,\mx_i-\mx_{k_1},\my_{i1}-\my_{i2},\ldots,\my_{i1}-\my_{ik_2})\nn\\
&\ +\mp_{22}\sum_{i=2}^{k_1}\sum_{j=2}^{k_2}\eta(\mx_i-\mx_1,\ldots,\mx_i-\mx_{k_1},\my_{ij}-\my_{i1},\ldots,\my_{ij}-\my_{ik_2})   
\end{align}
where  terms such as  $\mx_i-\mx_i$ and $\my_{ij}-\my_{ij}$ are omitted from the argument of $\eta$.
Let us denote $\Delta \mx_i=\mx_1-\mx_i$ for $i\in [2:k_1]$ and $\Delta \my_{ij}=\my_{i1}-\my_{ij}$ for $i\in[k_1]$ and $j\in[2:k_2]$. Observe that $\mx_i-\mx_j=\mx_i-\mx_1+(\mx_1-\mx_j)=\Delta x_j-\Delta \mx_i$ and $\my_{ij}-\my_{ij'}=\my_{ij}-\my_{i1}+(\my_{i1}-\my_{ij'})=\Delta \my_{ij'}-\Delta \my_{ij}$. Let $\Delta \mx=(\Delta \mx_{2},\ldots,\Delta \mx_{k_1})$, and for $i\in[k_1]$, let $\Delta \my_{i}=(\Delta \my_{i2},\ldots,\Delta \my_{ik_2})$. 
With this notation, for any $\mx\in\RR^{k_1}$ and $\my_1,\ldots,\my_{k_2}\in\RR^{k_2}$, the quantity $V^\psi$ reduces to
\begin{align*}
 \MoveEqLeft  V^{\psi}(\mx,\my_1,\ldots,\my_{k_1};\mp)
= \mp_{11}\eta(\Delta \mx,\Delta \my_1)\nn\\
&\ + \mp_{12} \sum_{j=2}^{k_2}\eta(\Delta \mx,-\Delta \my_{1j},\Delta \my_{12}-\Delta \my_{1j}\ldots,\Delta \my_{1k_2}-\Delta \my_{1j})\nn\\
&\ +\mp_{21}\sum_{i=2}^{k_1}\eta(-\Delta \mx_i,\Delta \mx_{2}-\Delta \mx_i\ldots,\Delta \mx_{k_1}-\Delta \mx_i,\Delta \my_{i})\nn\\
&\ +\mp_{22}\sum_{i=2}^{k_1}\sum_{j=2}^{k_2}\eta\begin{pmatrix}-\Delta \mx_i,\Delta \mx_{2}-\Delta \mx_i\ldots,\Delta \mx_{k_1}-\Delta \mx_i,&\\-\Delta \my_{ij},\Delta \my_{i2}-\Delta \my_{ij}\ldots,\Delta \my_{ik_2}-\Delta \my_{ij}\end{pmatrix}\nn\\
=&\ \mp_{11}\eta(\Delta \mx,\Delta \my_1)\nn\\
&\ + \mp_{12} \sum_{j=1}^{k_2-1}\eta(\Delta \mx,-\Delta \my_{1,1+j},\Delta \my_{12}-\Delta \my_{1,1+j}\ldots,\Delta \my_{1k_2}-\Delta \my_{1,1+j})\nn\\
&\ +\mp_{21}\sum_{i=1}^{k_1-1}\eta(-\Delta \mx_{1+i},\Delta \mx_{2}-\Delta \mx_{1+i}\ldots,\Delta \mx_{k_1}-\Delta \mx_{1+i},\Delta \my_{1+i})\nn\\
&\ +\mp_{22}\sum_{i=1}^{k_1-1}\sum_{j=1}^{k_2-1}\eta\begin{pmatrix}-\Delta \mx_{1+i},\Delta \mx_{2}-\Delta \mx_{1+i}\ldots,\Delta \mx_{k_1}-\Delta \mx_{1+i},-\Delta \my_{1+i,1+j},\\
\Delta \my_{1+i.2}-\Delta \my_{1+i,1+j}\ldots,\Delta \my_{1+i,k_2}-\Delta \my_{1+i.1+j}\end{pmatrix}.
\end{align*}
For $\mbu\in\RR^{k_1-1}$ and $\mv_1,\ldots,\mv_{k_1}\in\RR^{k_1-1}$, by an slight abuse of notation, let us define the function
\begin{align}
\label{intheorem: CC: def: of V phi for PERM}  \MoveEqLeft  V^{\eta}(\mbu,\mv_1,\ldots,\mv_{k_1};\mp)
=  \mp_{11}\eta(\mbu,\mv_1)\nn\\
&\ + \mp_{12} \sum_{j=1}^{k_2-1}\eta(\mbu,-\mv_{1j},\mv_{11}-\mv_{1j}\ldots,\mv_{1,k_2-1}-\mv_{1j})\nn\\
&\ +\mp_{21}\sum_{i=1}^{k_1-1}\eta(-\mbu_{i},\mbu_{1}-\mbu_{i}\ldots,\mbu_{k_1-1}-\mbu_{i},\mv_{1+i})\nn\\
&\ +\mp_{22}\sum_{i=1}^{k_1-1}\sum_{j=1}^{k_2-1}\eta
\begin{pmatrix}-\mbu_{i},\mbu_{1}-\mbu_{i}\ldots,\mbu_{k_1-1}-\mbu_{i},&\\
-\mv_{1+i,j},\mv_{1+i,1}-\mv_{1+i,j}\ldots, \mv_{1+i,k_2-1}-\mv_{1+i.j}
\end{pmatrix},
\end{align}
where terms such as $\mbu_i-\mbu_i$ and $\mv_{ij}-\mv_{ij}$ are omitted from the argument of $\eta$ in the above definition. Note that $V^\eta$ is different from $V^\psi$ as the domain of $\psi$ and $\eta$ are different. 
Since $\Delta \mx=(\Delta \mx_{2},\ldots,\Delta \mx_{k_1})$ and  $\Delta \my_{i}=(\Delta \my_{i2},\ldots,\Delta \my_{ik_2})$ for $i\in[k_1]$, letting
 $\Delta x=\mbu$ and $\Delta \my_i=\mv_i$, we can show that
 \begin{align}
     \label{intheorem: CC: connection between V psi and V eta}
     V^{\psi}(\mx,\my_1,\ldots,\my_{k_1};\mp)=V^\eta(\Delta \mx,\Delta \my_1,\ldots,\Delta\my_{k_1};\mp)
 \end{align}
 for all  $\mx\in\RR^{k_1}$ and $\my_1,\ldots,\my_{k_2}\in\RR^{k_2}$.
Therefore, $V^{\psi}$  depends on $\mx$ and the $\my_i$'s only via the $\Delta \mx$ and $\Delta \my_i$'s.  To maximize $V^{\psi}$ with respect to $\mx$ and the $\my_i$'s, one actually need to maximize it over the $\Delta \mx$ and $\Delta \my_i$'s which are in $\RR^{k_1-1}$ and $\RR^{k_2-1}$, respectively. 
If $\mbu^*\in\RR^{k_1-1} $, $\mv_1^*,\ldots,\mv_{k_1}^*\in\RR^{k_2-1}$ maximizes $V^\eta(\cdot;\mp)$, then any $\mx\in\RR^{k_1},\my_1,\ldots,\my_{k_1}\in\RR^{k_2}$ with $\Delta \mx=\mbu^* $ and $\Delta \my_i=\mv_i^*$ ($i\in[k_1]$) is a maximizer of $V^{\psi}(\cdot;\mp)$. In particular 
\begin{align}
\label{intheorem: CC: connection from u star to x star}
    \mx^*=(0,-\mbu^* ),\quad \my_1^*=(0,-\mv_1^*),\ldots, \my_{k_1}^*=(0,-\mv_{k_1}^*) 
\end{align}
is a version of $(\mx^*(\mp),\my^*_1(\mp),\ldots,\my_{k_1}^*(\mp))$. 
Alternatively, 
\begin{equation}
    \label{intheorem: cc: connection x star and mbu star}
(\Delta \mx^*(\mp),\Delta \my_1^*(\mp),\ldots,\Delta \my_{k_1}^*(\mp))\in\argmax\limits_{\mbu\in\RR^{k_1-1},\mv_1,\ldots,\mv_{k_1}\in\RR^{k_2-1}}V^\eta(\mbu,\mv_1,\ldots,\mv_{k_1})
\end{equation}
if $(\mx^*(\mp),\my^*_1(\mp),\ldots,\my_{k_1}^*(\mp))$ exists.  However, it is not obvious when\\ $(\mx^*(\mp),\my^*_1(\mp),\ldots,\my_{k_1}^*(\mp))$  exists. 
To this end, first we derive a simpler expression for $V^\eta(\cdot;\mp)$.

Since $\eta$ is blockwise permutation symmetric  by \eqref{intheorem: cc: permutation symmetry of phi}, for any $\mbu\in\RR^{k_1-1}$ and $\mw\in\RR^{k_2-1}$, 
\[\eta(\mbu,-\mw_{j},\mw_{1}-\mw_{j}\ldots,\mw_{k_2-1}-\mw_{j})=\eta(\mbu,\mw_{1}-\mw_{j}\ldots,-\mw_{j},\ldots,\mw_{k_2-1}-\mw_{j})\]
and
\[\eta(-\mbu_i,\mbu_{1}-\mbu_i\ldots,\mbu_{k_1-1}-\mbu_i,\mw)=\eta(\mbu_{1}-\mbu_i\ldots,-\mbu_i,\ldots,\mbu_{k_1-1}-\mbu_i,\mw)\]
for any $i\in[k_1-1]$, $j\in[k_2-1]$,  $\mbu\in\RR^{k_1-1}$ and $\mw\in\RR^{k_2-1}$. In the above, terms such as $\mbu_i-\mbu_i$ and $\mw_j-\mw_j$ are omitted from the argument of $\eta$. Thus 
\begin{align}
\label{intheorem: CC: def: of V phi for PERM 1} 
\MoveEqLeft  V^{\eta}(\mbu,\mv_1,\ldots,\mv_{k_1};\mp)
= \mp_{11}\eta(\mbu,\mv_1)\nn\\
&\ + \mp_{12} \sum_{j=1}^{k_2-1}\eta(\mbu,\mv_{11}-\mv_{1j}\ldots,-\mv_{1j},\ldots,\mv_{1,k_2-1}-\mv_{1j})\nn\\
&\ +\mp_{21}\sum_{i=1}^{k_1-1}\eta(\mbu_{1}-\mbu_i\ldots,-\mbu_i,\ldots,\mbu_{k_1-1}-\mbu_i,\mv_{i+1})\\
&\ +\mp_{22}\sum_{i=1}^{k_1-1}\sum_{j=1}^{k_2-1}\eta\begin{pmatrix}
    \mbu_{1}-\mbu_i\ldots,-\mbu_i,\ldots,\mbu_{k_1-1}-\mbu_i,&\\\mv_{i+1,1}-\mv_{i+1,j}\ldots,-\mv_{i+1,j},\ldots,\mv_{i+1,k_2-1}-\mv_{i+1,j}.
\end{pmatrix}\nn
\end{align}
 Let us define $\AA_0=I_{k_1-1}$, 
 \begin{align}
     \label{inlemma: CC: def of Ai}
     \AA_1=\begin{bmatrix}
 -1 & \mz^T_{k_1-2}\\
     -\mo_{k_1-1} & I_{k_1-2}
 \end{bmatrix} \text{ and }\AA_i=\begin{bmatrix}
     I_{i-1} & -\mo_{i-1} & \mz_{(i-1)\times (k_1-1-i)}\\
     \mz_{i-1}^T & -1 & -\mz_{k_1-1-i}^T\\
     \mz_{(k_1-1-i)\times(i-1)} & -\mo_{k_1-1-i} & I_{k_1-1-i}
 \end{bmatrix}
 \end{align}
 for $i\in[2:k_1-1]$ 
 in case the set $[2:k_1-1]$ is non-empty. 
Therefore,
\begin{align*}
    \AA_i\mbu=(\mbu_1-\mbu_i,\ldots,-\mbu_i,\ldots,\mbu_{k_1-1}-\mbu_i)\text{ for all }i\in[k_1].
\end{align*}
One can verify that we can obtain $\AA_i$ from $\AA_j$ by swapping the $i$th column with the $jth$ column and then swapping the ith row with the jth row. Therefore, $\AA_i=P_{ij}\AA_jP_{ij}$ where, by a slight abuse of notation, we define  $P_{ij}$ to be the permutation matrix of appropriate dimesnion for swapping ith row with jth row. Also, we let $\BB_0$ be $I_{k_2-1}$ and 
\begin{align}
\label{inlemma: CC: def of Bj}
  \BB_1=\begin{bmatrix}
 -1 & \mz^T_{k_2-2}\\
     -\mo_{k_2-1} & I_{k_2-2}
 \end{bmatrix} \text{ and }\BB_j=\begin{bmatrix}
     I_{j-1} & -\mo_{j-1} & \mz_{(j-1)\times (k_2-1-j)}\\
     \mz_{j-1}^T & -1 & -\mz_{k_2-1-j}^T\\
     \mz_{(k_2-1-j)\times(j-1)} & -\mo_{k_2-1-j} & I_{k_2-1-j}
 \end{bmatrix}
\end{align}
for $j\in[2:k_2-1]$ in case the set $[2:k_2-1]$ is non-empty. 
It can similarly be shown that $\BB_j=P_{1j}\BB_1P_{1j}$. Note that
\begin{align*}
   \BB_j\mw=(\mw_1-\mw_j,\ldots,-\mw_j,\ldots,\mw_{k_2-1}-\mw_j) \text{ for any }\mw\in\RR^{k_2-1}.
\end{align*}
With this new notation,  \eqref{intheorem: CC: def: of V phi for PERM 1} reduces to 
\begin{align}
\label{intheorem: CC: def: of V phi for PERM 2}  \MoveEqLeft  V^{\eta}(\mbu,\mv_1,\ldots,\mv_{k_1};\mp)
= \mp_{11}\eta(\mbu,\mv_1)\nn + \mp_{12} \sum_{j=1}^{k_2-1}\eta(\mbu,\BB_j \mv_1)\nn\\
&\ +\mp_{21}\sum_{i=1}^{k_1-1}\eta(\AA_i \mbu,\mv_{i+1}) +\mp_{22}\sum_{i=1}^{k_1-1}\sum_{j=1}^{k_2-1}\eta(\AA_i \mbu, \BB_j\mv_{i+1})
\end{align}
for all $\mbu\in\RR^{k_1-1}$, $\mv_1,\ldots,\mv_{k_1}\in\RR^{k_2-1}$. 

We will now show that the maximizer of $V^\eta(\cdot;\mp)$ is  attained for all  $\mp\in\RR^4_{>0}$.
If possible, suppose there exists $\mp\in\RR^4_{>0}$ so that the maximizer of $V^\eta(\cdot;\mp)$ is not attained. 
Then Lemma \ref{lemma: CC: coercivity main lemma} implies that  the maximizer of $V^\eta(\cdot;\mp)$ is not attained for  any $\mp\in\RR^4_{>0}$. The proof of Lemma  \ref{lemma: CC: coercivity main lemma} is provided in Section \ref{secpf: cc: coercive main lemma}. Note that the function $f$ in Lemma~\ref{lemma: CC: coercivity main lemma} does not represent scores, and $h$ denotes a function rather than a history.
 \begin{lemma}
\label{lemma: CC: coercivity main lemma}
  Let $r,k\in\NN$.  Suppose $C_1,\ldots, C_m$ are matrices in $\RR^{r\times k}$ and $h:\RR^r\mapsto\RR$ is a  convex and bounded below. Further suppose $\mz_{k}\in\dom(h)$. Define 
     \[f(\mx;\mw)=\sum_{i=1}^m\mw_ih(C_i\mx),\quad \mx\in\RR^k,\mw\in\RR^m_{>0}.\]
     Suppose there exists  $\mw_0\in\RR^m_{>0}$ so that the minimum of the convex function $f(\cdot;\mw_0)$
   is not attained  in $\RR^k$. Then the minima  of $f(\cdot;\mw)$ is not attained  in $\RR^k$ for any $\mw\in\RR^m_{>0}$.
\end{lemma}
Here Lemma \ref{lemma: CC: coercivity main lemma} applies because  $-\eta$ is convex and bounded below and $\dom(-\eta)$ $\ni \mz_{k_1+k_2-2}$ by Lemma \ref{lemma: CC: 0 in int}. In particular, the maximizer of $V^\eta(\cdot;\mp)$ is not attained when $\mp_{11}=\mp_{12}=\mp_{21}=\mp_{22}=1$. Then the maximizer of $V^\psi(\cdot;\mp)$ is also not attained because otherwise \eqref{intheorem: CC: connection from u star to x star} would give a maximizer of $V^\eta(\cdot;\mp)$. However, Lemma \ref{lemma: cc: 0 maximum for smooth} implies that $(\mz_{k_1},\mz_{k_2},\ldots,\mz_{k_2})$ is a  maximizer of $V^\psi$ when the $\mp_{ij}$'s are all one, which presents a contradiction. Therefore, the maximizer of $V^\eta(\cdot;\mp)$ is attained for all  $\mp\in\RR^4_{>0}$.
 Since $\eta$ is strictly concave, $V^\eta(\cdot;\mp)$ is strictly concave, implying that the maximizer is unique. Let us denote this maximizer by $(\mbu^*(\mp) $, $\mv_1^*(\mp),\ldots,\mv_{k_1}^*(\mp))$. The following lemma, which is proved in Section \ref{secpf: claim PERM sup}, shows that the maximizer lies in the lower dimensional space 
 \begin{equation}
 \label{def: CC: mathcal C}
 \begin{split}
      \C=\lbs (\mbu,\mv_1,\ldots,\mv_{k_1})\in\RR^{k_1-1}\times \RR^{k_1(k_2-1)}: \mbu=x\mo_{k_1-1},\\
      \ \mv_1=y\mo_{k_2-1},\ \mv_2=\ldots=\mv_{k_1}=z\mo_{k_2-1},\ \text{where }x,y,z\in\RR\rbs.
 \end{split}
    \end{equation}

\begin{lemma}
\label{lemma: cc: supremum}
For  $\mp\in\RR^4_{>0}$, let
\[(\mbu^*(\mp) , \mv_1^*(\mp),\ldots,\mv_{k_1}^*(\mp))=\argmax\limits_{\mbu\in\RR^{k_1-1},\mv_1,\ldots,\mv_{k_1}\in\RR^{k_2-1}}V^\eta(\mbu,\mv_1,\ldots,\mv_{k_1};\mp).\] 
Then  $
(\mbu^*(\mp) ,\mv_1^*(\mp),\ldots,\mv_{k_1}^*(\mp))\in\mathcal C,
$
 where $\mathcal C$ is as defined in \eqref{def: CC: mathcal C}.
\end{lemma}

For all $x$, $y$, $z\in\RR$, let us define $\vartheta:\RR^2\times\{1,2\}^2\mapsto\RR$ i as follows:
  \begin{align}
    \label{intheorem: CC: def: vartheta}
  \vartheta(x,y;1,1)=&\  \eta(x\mo_{k_1-1},y\mo_{k_2-1}) \nn\\
   \vartheta(x,y;1,2)=&\  \sum_{j=1}^{k_2-1}\eta(x\mo_{k_1-1},yB_j\mo_{k_2-1})\nn\\
    \vartheta(x,y;2,1)=&\  \sum_{i=1}^{k_1-1}\eta(xA_i\mo_{k_1-1},y\mo_{k_2-1})\nn\\
     \vartheta(x,y;2,2)=&\  \sum_{i=1}^{k_1-1}\sum_{j=1}^{k_2-1}\eta(xA_i\mo_{k_1-1},yB_j\mo_{k_2-1}),
    \end{align}
    where for $i\in[k_1]$ and $j\in[k_2]$, and the $\AA_i$ and $\BB_j$'s are as defined in \eqref{inlemma: CC: def of Ai} and \eqref{inlemma: CC: def of Bj}, respectively. 
For $ x,y,z\in\RR$, define
\begin{align}
    \label{def: CC: Phi: binary}
\Phi(x,y,z;\mp)=-\mp_{11}\vartheta(x,y;1,1)-\mp_{12}\vartheta(x,y;1,2)-\mp_{21}\vartheta(x,z;2,1)-\mp_{22}\vartheta(x,z;2,2).
\end{align}
Note that $\Phi$ is well defined because $\eta$ is bounded above. If $\mbu=x\mo_{k_1-1}$,\\ $\mv_1=y\mo_{k_2-1}$, and $\mv_2=\ldots=\mv_{k_1}=z\mo_{k_2-1}$, then for all $\mp\in\RR^4_{>0}$ and $x,y,z\in\RR$, 
\begin{equation}
\label{intheorem: connection between V eta and Phi}
  V^\eta(\mbu,\mv_1,\ldots,\mv_{k_1};\mp)=  \Phi(x,y,z;\mp).
\end{equation}
   Therefore
Lemma \ref{lemma: cc: supremum} and \eqref{intheorem: CC: def: of V phi for PERM 2} imply that 
\begin{align*}
\sup_{(\mbu,\mv_1,\ldots,\mv_{k_1})\in\RR^{k_1-1}\times \RR^{k_1(k_2-1)}}V^\psi(\mbu,\mv_1,\ldots,\mv_{k_1};\mp)=\sup_{x,y,z\in\RR}\Phi(x,y,z;\mp),
\end{align*}

It is  possible to show that if $\psi$ is Fisher consistent for the class $\mP_b^{k_1,k_2}$, then $\vartheta$ induces a relative-margin-based Fisher consistent loss for  2-stage DTR problem with  binary treatments. When we mention  ``relative-margin-based loss induced by $\vartheta$", we refer to the loss defined by $\vartheta^{\text{loss}}(x_1,x_2,y_1,y_2;i,j)=\vartheta(x_1-x_2,y_1-y_2;i,j)$ for all $x_1,x_2,y_1,y_2\in\RR$.  For the purpose of showing the contradiction in Step 2, it will be sufficient to show that the Fisher consistency of $\psi$ induces a weaker version of Fisher consistency on  $\vartheta^{\text{loss}}$ or $\vartheta$, which is later elicited in Property \ref{property: CC}. Therefore, we will not get into the Fisher consistency of $\vartheta^{\text{loss}}(\cdot;i,j)$. However, before discussing the weaker version of Fisher consistency, we need to establish more properties of $\Phi$ and the $\vartheta(\cdot;i,j)$'s.



 Lemma \ref{lemma: CC: closed}  establishes that $\Phi(\cdot;\mp)$ has a unique minimizer for any $\mp\in\RR^4_{>0}$. This lemma also establishes that $-\vartheta$ and $\Phi$ are proper, closed, and strictly convex, where a  proper convex function was defined in Section \ref{sec: intheorem: cc: terminologies}. Lemma \ref{lemma: CC: closed} is proved in Section \ref{secpf: cc: closed}.

\begin{lemma}
 \label{lemma: CC: closed}
 Under the setup of Theorem \ref{theorem: CC}, the $\Phi(\cdot;\mp)$ defined in \eqref{def: CC: Phi: binary} has a unique minimizer over $\RR^3$ for all $\mp\in\RR^4_{>0}$. Moreover, 
   \begin{equation}
       \label{def: intheorem: cc: x star y star}
       (x^*(\mp),y^*(\mp),z^*(\mp))\equiv (x^*,y^*,z^*)=\argmin\limits_{x,y,z\in\RR} \Phi(x,y,z;\mp).
   \end{equation} 
   satisfies $x^*(\mp)=\mbu_1^*(\mp)$, $y^*(\mp)=\mv_{11}^*(\mp)$, and $z^*(\mp)=\mv^*_{21}(\mp)$. 
   Moreover,  the  $-\vartheta(\cdot;i,j)$'s, as well as $\Phi(\cdot;\mp)$, are closed and proper strictly convex  functions for all $i,j\in[2]$ and $\mp\in\RR^4_{>0}$. Moreover, $\Phi(\cdot;\mp)$ is 0-coercive for all $i,j\in[2]$ and $\mp\in\RR^4_{>0}$.
\end{lemma}

 An important quantity for us will be 
\begin{align*}
    \Lambda^*(\mp)=(x^*(\mp),y^*(\mp),z^*(\mp)),
\end{align*}
which uniquely exists for each $\mp\in\RR^4_{>0}$ by Lemma \ref{lemma: CC: closed}. 
 When $\mp$ is clear from the context, we will drop $\mp$ from the notation of $x^*(\mp),y^*(\mp)$, and $z^*(\mp)$.
\subsection{Step 2}
\label{sec: CC: step 2}
In this step, we will show that the $\Lambda^*(\mp)$ defined in \eqref{intheorem: CC: def: vartheta} satisfies some  restrictions due to the concavity of $\vartheta$. Then we will show that,
if $\psi$ is Fisher consistent, then $\vartheta$  satisfies some properties, which violates the abovementioned  restrictions. This contradiction would imply that $\psi$ can not be Fisher consistent, thus completing the proof. 

  It is hard to find a closed form expression of $\Lambda^*(\mp)$ or analyze its behavior for any arbitrary $\mp$. Especially, for an arbitrary $\mp\in\RR^4_{>0}$, it is not  guaranteed that $\Lambda^*(\mp)$ lies in the interior or the relative interior of $\dom(\Phi(\cdot;\mp))$. However, if $\Lambda^*(\mp)\notin\iint(\dom(\Phi(\cdot;\mp))$, we will need to deal with several pathological issues for applying the tools of convex analysis.  Therefore, for the rest of the proof, we will focus on a  neighborhood of a special $\mp$, namely $\mp=\mp^0=(1,1,1,1)$. We will later show that if $\mp$ is in a small neighborhood of $\mp^0$, then $\Lambda^*(\mp)\in \iint(\dom(\Phi(\cdot;\mp))$, which makes application of convex analysis tools easier.  Lemma \ref{lemma: CC: value of Lambda p knot } implies that $\Lambda^*(\mp^0)=(0,0,0)$. 
  \begin{lemma}
  \label{lemma: CC: value of Lambda p knot }
      Suppose $\mp^0=(1,1,1,1)$. Then $\Lambda^*(\mp^0)=(0,0,0)$.
  \end{lemma}
   \begin{proof}[Proof of Lemma \ref{lemma: CC: value of Lambda p knot }]
This lemma follows trivially from \eqref{intheorem: CC: def: vartheta} and Lemma \ref{lemma: cc: 0 maximum for smooth}.

  \end{proof}
 
We  break down step 2 into further substeps.
\begin{itemize}
\item[2a.] In this step, we establish smoothness properties of $\vartheta$ and related functions in neighborhood of $\mp^0$. 
\item[2b.] In this step, we  establish the continuity properties of the map $\mp\mapsto \Lambda^*(\mp)$.
We will also show that $\mp\mapsto y^*(\mp)$ and $\mp\mapsto z^*(\mp)$ are locally Lipschitz at $\mp=\mp^0$. These results enable control over the behavior of $\Lambda^*(\mp)$ for $\mp$'s in a neighborhood of $\mp^0$.
    \item[2c.] We will show that the concavity of  $\vartheta$ induces some restrictions on $\Lambda^*(\mp)$  for all $\mp$ in a neighborhood of $\mp^0$. 
    \item[2d.] We will show that if $\psi$ is Fisher consistent, then $\vartheta$ satisfies some properties, which will collectively be referred to as Property \ref{property: CC}. 
    \item[2e.] We will use the restrictions on $\Lambda^*(\mp)$  derived in Step 2c to show that $\vartheta$ can not satisfy Property \ref{property: CC}, which would imply $\psi$ can not be Fisher consistent. 
\end{itemize}

\subsubsection{ Step 2a: properties of $\vartheta$ and related functions}
In this step, we will mainly establish different  properties of $\vartheta$, $\Phi$, and $\Lambda^*(\mp)$. The function $\Phi(\cdot;\mp)$ is an important function for us because $\Lambda^*(\mp)$ relies on it and, as mentioned previously, the proof hinges on $\Lambda^*(\mp)$. Another important function for us is 
\begin{equation}
    \label{def: CC: Phi}
\Phi^*(\mp)= \inf_{\mw\in\RR^3} \Phi(\mw;\mp),\ \mp\in\RR^4_{>0}.
\end{equation}
Note that as $\Lambda^*(p)$ is the solution to above minimization problem, we also have
\[\Phi^*(\mp)=\Phi(\Lambda^*(\mp);\mp)\]
for any $\mp\in\RR^4_{>0}$. Fact \ref{fact: Phi star mp finite} below implies that  $-\Phi^*$ is well-behaved in that it is a  convex function  that is  continuous and bounded on $ \RR^4_{>0}$. This fact requires $-\eta$ to be bounded below, which is satisfied under the setup of Theorem \ref{theorem: CC}. Fact \ref{fact: Phi star mp finite} is proved in Section \ref{sec: proof of Fact Phi star mp finite}.
\begin{fact}
    \label{fact: Phi star mp finite}
    Suppose $-\eta$ is bounded below. Then 
    $-\Phi^*$ is a convex function  with $ \iint(\dom(-\Phi^*))\supset \RR^4_{>0}$.  Also $\Phi^*$ is continuous on $\RR^4_{>0}$. Moreover, 
 $|\Phi^*(\mp)|<\infty$ for all $\mp\in\RR^4_{>0}$.
\end{fact}

Lemma \ref{lemma: CC: differentiability} below implies that the $\vartheta(\cdot;i,j)$'s are thrice continuously differentiable at $(0,0)$ for $i,j\in[2]$.  In particular,  there is an open neighborhood $U_0$ of $0$ so that the $\vartheta(\cdot;i.j)$'s are thrice continuously differentiable at $U_0^2$ for all $i,j\in[2]$ and  $\Phi(\cdot;\mp)$ is thrice continuously differentiable at $U_0^3$ for all $\mp\in\RR^4_{>0}$. 
Lemma \ref{lemma: CC: differentiability} is proved in Section \ref{secpf: Lemma differentiability}.

\begin{lemma}
\label{lemma: CC: differentiability}
 Under the setup of Theorem \ref{theorem: CC}, 
 there exists an open neighborhood $U_0$ of $0$ so that
 \begin{enumerate}
     \item $U_0^2\subset \iint(\dom(\vartheta(\cdot;i,j)))$ for all $i,j\in[2]$.
     \item $U_0^3\subset \iint(\dom(\Phi(\cdot;\mp)))$ for all $\mp\in\RR^4_{>0}$.
     \item The $\vartheta(\cdot;i.j)$'s are thrice continuously differentiable at $U_0^2$ for all $i,j\in[2]$.
     \item $\Phi(\cdot;\mp)$ is thrice continuously differentiable at $U_0^3$ for all $\mp\in\RR^4_{>0}$. 
 \end{enumerate}
\end{lemma}

 
\subsubsection{Step 2b: continuity properties of $\Lambda^*$}
\label{sec: CC: 2b}
 In this section, we will prove the continuity properties of $\Lambda^*$. In particular, we will show that $\Lambda^*$ is continuous. Moreover, we will show that $y^*$ and $z^*$ are locally Lipschitz at $\mp^0$.   Lemma \ref{lemma: CC: continuity of Lambda} proves the continuity of the map $\mp\mapsto\Lambda^*(\mp)$. This lemma  is proved in Section \ref{secpf:  CC: continuity of Lambda}. 

  
\begin{lemma}
   \label{lemma: CC: continuity of Lambda}
      Under the setup of Theorem \ref{theorem: CC}, $\mp\mapsto\Lambda^*(\mp)$ is a continuous map at any $\mp\in\RR^4_{>0}$. 
   \end{lemma}
  

One immediate corollary of the continuity of $\Lambda^*$ is that  $\Lambda^*(\mp)\in U_0^3$ for all $\mp$ in a small neighborhood of $\mp^0$. 
  \begin{lemma}
\label{lemma: CC: hessian cont.}
  Under the setup of Theorem \ref{theorem: CC},  there exists $\delta>0$ so that if $\mp\in B(\mp^0,\delta)$, then $\Lambda^*(\mp)\in U_0^3$, where $U_0$ is as in Lemma \ref{lemma: CC: differentiability}. 
\end{lemma}

\begin{proof}[Proof of Lemma \ref{lemma: CC: hessian cont.}]
 
Lemma \ref{lemma: CC: value of Lambda p knot } implies that $\Lambda^*(\mp^0)=(0,0,0)$. By the continuity of $\Lambda^*$ by Lemma \ref{lemma: CC: continuity of Lambda}, we can choose $\delta$ small enough so that if $\mp\in B(\mp^0,\delta)$, then $\Lambda^*(\mp)\in U_0^3$.

\end{proof}

Lemma \ref{lemma: CC: Lipschitz}  below establishes the Lipschitz continuity of $y^*$ and $z^*$ at $\mp^0$.
   \begin{lemma}
     \label{lemma: CC: Lipschitz}  
  Consider the setup of Theorem \ref{theorem: CC}.  Then there exists a $\delta>0$ and $C\geq 0$ such that for all  $\mp\in B(\mp^0,\delta)$,
    \[|y^*(\mp)|\leq C\|\mp-\mp^0\|_2\]
    and
    \[|z^*(\mp)|\leq C\|\mp-\mp^0\|_2.\]
   \end{lemma}
   Lemma \ref{lemma: CC: Lipschitz} is proved in Section \ref{secpf: proof of lemma: CC: Lipschitz }.
   Since $\Lambda(\mp^0)=(0,0,0)$ by Lemma \ref{lemma: CC: value of Lambda p knot }, the above can also be re-written as 
     \[|y^*(\mp)-y^*(\mp^0)|\leq C\|\mp-\mp^0\|_2\text{ and }|z^*(\mp)-z^*(\mp^0)|\leq C\|\mp-\mp^0\|_2,\]
     which yields the Lipschitz continuity of $z^*$ and $y^*$ at $\mp^0$.

   \subsubsection{Step 2c: geometric restrictions on $\vartheta$}
   In this step, we show that $\vartheta$ and $\Lambda^*(\mp)$ need to satisfy a system of inequalities when  $\mp\in B(\mp^0,\delta)$ where $\delta$ is smaller than the $\delta$ in Lemma \ref{lemma: CC: hessian cont.}. 
For such $\mp$'s,  $\Lambda^*(\mp) \in U_0^3$, which is necessary to guarantee the differentiability of $\Phi(\cdot; \mp)$ at $\Lambda^*(\mp)$, as well as to ensure that $\Lambda^*(\mp) \in \iint(\dom(-\Phi(\cdot; \mp)))$ and $\Lambda^*(\mp) \in \iint(\dom(-\vartheta(\cdot; i, j)))$ for all $i, j \in [2]$.
Suppose $x\in U_0$, which indicates $(x,y^*(\mp),z^*(\mp))\in U_0^3$. Therefore, $\Phi(\cdot;\mp)$ is differentiable at $(x,y^*(\mp),z^*(\mp))$ by Lemma \ref{lemma: CC: differentiability}. Since $\Lambda^*(p)-(x,y^*(p),z^*(p))=(x^*(p)-x,0,0)$ and $\Phi$ is convex, from the definition of subdifferential \citep[pp. 167, (1.2.1),][]{hiriart}, it follows that 
   \[\Phi(\Lambda^*(\mp);\mp)-\Phi(x,y^*(\mp),z^*(0);\mp)\geq (x^*(\mp)-x)\pdv{\Phi(x,y^*(\mp),z^*(\mp);\mp)}{x}.\]
   In particular, if $x\neq x^*(\mp)$, then by strict convexity, the above inequality is strict.
   Thus, we have for any $x\neq x^*(\mp)$,
   \[(x^*(\mp)-x)\pdv{\Phi(x,y^*(\mp),z^*(\mp);\mp)}{x}< 0.\]
  Since $0\in U_0$, we can choose $x=0$, which implies that for all $\mp\in B(\mp^0,\delta)$,
   \begin{equation}
       \label{intheorem: cc: inequality}
       x^*(\mp)\pdv{\Phi(0,y^*(\mp),z^*(\mp);\mp)}{x}< 0.
   \end{equation}
   Similarly, we can show that
   \begin{align}
       \label{intheorem: cc: inequality 2}
       z^*(\mp)\pdv{\Phi(x^*(\mp),y^*(\mp),0;\mp)}{z}<&\  0,\nn\\
       y^*(\mp)\pdv{\Phi(x^*(\mp),0,z^*(\mp);\mp)}{y}<&\  0
   \end{align}
for all $\mp\in B(\mp^0,\delta)$.  Note that the concavity of $\vartheta$, which led to the convexity of $\Phi(\cdot;\mp)$, is the main reason behind the above inequalities.

   Now we will calculate the partial derivatives. To this end, note that 
  \begin{align}
  \label{intheorem: derivatives}
      \pdv{\Phi(x^*(\mp),0,z^*(\mp);\mp)}{y}=&\ -\mp_{11}\pdv{\vartheta(x^*(\mp),0;1,1)}{y}-\mp_{12}\pdv{\vartheta(x^*(\mp),0; 1,2)}{y}\nn\\
       \pdv{\Phi(x^*(\mp),z^*(\mp),0;\mp)}{z}=&\ -\mp_{21}\pdv{\vartheta(x^*(\mp),0;2,1)}{z}-\mp_{22}\pdv{\vartheta(x^*(\mp),0; 2,2)}{z}\nn\\
        \pdv{\Phi(0,y^*(\mp),z^*(\mp);\mp)}{x}=&\ -\mp_{11}\pdv{\vartheta(0,y^*(\mp);1,1)}{x}-\mp_{12}\pdv{\vartheta(0,y^*(\mp); 1,2)}{x}\nn\\
        &\ -\mp_{21}\pdv{\vartheta(0,z^*(\mp);2,1)}{x}-\mp_{22}\pdv{\vartheta(0,z^*(\mp); 2,2)}{x}
  \end{align}
for all $\mp\in B(\mp^0,\delta)$,  provided $
  \delta$ is smaller than the $\delta$ in Lemma \ref{lemma: CC: hessian cont.}. 
The inequalities \eqref{intheorem: cc: inequality} and \eqref{intheorem: cc: inequality 2} entail a restriction on $\vartheta$ and $\Lambda^*(\mp)$.

\subsubsection{Step 2d: implication of  Fisher consistency}
In this section, we will show that the Fisher consistency of $\psi$ enforces some restriction on $\vartheta$. To this end, we define a set of conditions, which we will refer to us Property \ref{property: CC}. One may think of these conditions as a weak version Fisher consistency. 
As mentioned earlier, one can show that when $\psi$ is Fisher consistent, the margin-based loss defined by $\vartheta^{\text{loss}}(x_1,x_2,y_1,y_2;i,j)=\vartheta(x_1-x_2,y_1-y_2;i,j)$ for all $x_1,x_2,y_1,y_2\in\RR$ is Fisher consistent in the two-stage binary treatment case. However, to prove this theorem, we do not need to show the above result, which would be technically more cumbersome. For the purpose of this proof, it is enough to show that, when $\psi$ is Fisher consistent, $\vartheta$ satisfies some conditions,  which are collected in Property \ref{property: CC}.
\begin{property}
    \label{property: CC}
  The function $\vartheta:\RR^2\times\{1,2\}\mapsto\RR$ satisfies Property \ref{property: CC} if there exists $\delta>0$ so that $x^*(\mp)$, $y^*(\mp)$, and $z^*(\mp)$ defined in \eqref{def: intheorem: cc: x star y star} satisfy the following for all  $\mp\in B(\mp^0,\delta)$:
 \begin{itemize}
    \item[P1.] If  $\max (\mp_{11},\mp_{12})> \max(\mp_{21},\mp_{22})$, then $x^*(\mp)>0$. Also, if $\max (\mp_{11},\mp_{12})< \max(\mp_{21},\mp_{22})$, then $x^*(\mp)<0$.
    \item[P2.] Suppose $\max (\mp_{11},\mp_{12})> \max(\mp_{21},\mp_{22})$. Then   then $y^*(\mp)> 0$ if $\mp_{11}>\mp_{12}$ and $y^*(\mp)<0$ if $\mp_{11}<\mp_{12}$. Suppose $\max (\mp_{21},\mp_{22})> \max(\mp_{11},\mp_{12})$. Then  $z^*(\mp)> 0$ if $\mp_{21}>\mp_{22}$ and $z^*(\mp)<0$ if $\mp_{21}<\mp_{22}$.
\end{itemize}
\end{property}
Lemma \ref{lemma: cc: connecting general to margin-based}, which is proved in Section \ref{secpf: proof of  lemma: cc: connecting general to margin-based}, entails that  $\vartheta$ satisfies Property \ref{property: CC} if $\psi$ is Fisher consistent.
\begin{lemma}
 \label{lemma: cc: connecting general to margin-based}
  If $\psi$ is fisher consistent, then the $\vartheta$ defined in \eqref{intheorem: CC: def: vartheta}  satisfies Property \ref{property: CC} under the setup of Theorem \ref{theorem: CC}
\end{lemma}

Lemma \ref{lemma: cc: connecting general to margin-based} implies that to prove that $\psi$ is not Fisher consistent, it suffices to show that $\vartheta$ fails to satisfy Property \ref{property: CC}. 
We have $\mp\in B(\mp^0,\delta)$ in Property \ref{property: CC} because in this case Lemma \ref{lemma: CC: hessian cont.}  ensures that $\Lambda^*(\mp)\in U^2_0$ for sufficiently small $\delta>0$, ,which, in combination with Lemma \ref{lemma: CC: differentiability}, yields that $\Lambda^*(\mp)\in\iint(\dom(\Phi(\cdot;\mp))$. This ensures that we can use tools of convex analysis on $\Phi(\cdot;\mp)$ to prove Lemma \ref{lemma: cc: connecting general to margin-based}. Unless $\Lambda^*(\mp)\notin\iint(\dom(\Phi(\cdot;\mp))$, showing P1 and P2 of Property \ref{property: CC}  is difficult for arbitrary $\mp$'s. Proving Lemma \ref{lemma: cc: connecting general to margin-based},  which relies on the continuity of $\Phi(\cdot;\mp)$ at $\Lambda^*(\mp)$, is challenging  in this case.



  \subsubsection{Step 2e: proving the contradiction}

 This is the final step of the proof. In this step, we will show that $\vartheta$ can not satisfy Property \ref{property: CC}. Then  Lemma \ref{lemma: cc: connecting general to margin-based} would imply that $\psi$ can not be Fisher consistent, completing the proof.
Since $\vartheta(\cdot;i,j)$ is differentiable at $(0,0)$ by Lemma \ref{lemma: CC: differentiability} for all $i,j\in[2]$, at least one of the following three cases must hold: (i) $\partial \vartheta(0,0;1,1)/\partial x$, $\partial \vartheta(0,0;1,2)/\partial x$, $\partial \vartheta(0,0;2,1)/\partial x$, and $\partial \vartheta(0,0;2,2)/\partial x$ are all zero,
  (ii) at least one number between  $\partial \vartheta(0,0;1,1)/\partial x$ and  $\partial \vartheta(0,0;1,2)/\partial x$ is non-zero, and (iii) at least one number between $\partial \vartheta(0,0;2,1)/\partial x$ and   $\partial \vartheta(0,0;2,2)/\partial x$ is non-zero.   
 We will show that  in none of theses cases, $\vartheta$ satisfies Property \ref{property: CC}.  Since the proofs are similar for case (ii) and (iii), we will show the proof only for case (i) and (ii). 
\paragraph{Case (i): $\partial \vartheta(0,0;i,j)/\partial x$'s are  zero for all  $i,j\in[2]$}
 The following lemma implies that if the  $\partial \vartheta(0,0;i,j)/\partial x$'s are all zero, then $\vartheta(0,0;i,j)$'s can not satisfy Property \ref{property: CC}. 
\begin{lemma}
\label{lemma: CC: varthetas non-zero}
  Under the setup of Theorem \ref{theorem: CC},  if \[\partial\vartheta(0,0;1,1)/\partial x=\partial \vartheta(0,0;1,2)/\partial x=\partial\vartheta(0,0;2,1)/\partial x=\partial\vartheta(0,0;2,2)/\partial x=0,\]
    then the $\vartheta(\cdot;i,j)$'s can not satisfy Property \ref{property: CC}.
\end{lemma}

\begin{proof}[Proof of Lemma \ref{lemma: CC: varthetas non-zero}]
   Note that $\partial \Phi(0,0,0;\mp)/\partial x=0$ for all $\mp\in\RR^4_{>0}$ if
    \[\partial\vartheta(0,0;1,1)/\partial x=\partial \vartheta(0,0;1,2)/\partial x=\partial\vartheta(0,0;2,1)/\partial x=\partial\vartheta(0,0;2,2)/\partial x=0.\]
    Consider $\mp\in\RR^4_{>0}$ of the form $(c_1,c_1,c_2,c_2)$ where $c_1>c_2>0$ and $\mp$. If  $\vartheta$ satisfies Property \ref{property: CC}, then $x^*(\mp)>0$ for all such $\mp$ provided $c_1$ and $c_2$ are sufficiently close to one.  However, for  $\mp$'s of such form,  \eqref{intheorem: cc: partial derivatives sum zero} implies that $\partial \Phi(0,0,0;\mp)/\partial y=0$ and  $\partial \Phi(0,0,0;\mp)/\partial z=0$.
    Therefore, $\grad\Phi(0,0,0;\mp)=0$. Since $\Phi(\cdot;\mp)$ is strictly convex for all $\mp\in\RR^4_{>0}$, $\grad\Phi(0,0,0;\mp)=0$ implies $(0,0,0)$ is the unique minimizer of $\Phi(\cdot;\mp)$ \citep[cf. Theorem 2.2.1, 
 pp. 177,][]{hiriart}. Therefore, for our chosen $\mp$, $\Lambda^*(\mp)=(0,0,0)$ or $x^*(\mp)=0$ although $\max(\mp_{11},\mp_{12})>\max(\mp_{21},\mp_{22})$. Given any $\delta>0$, we can find $c_1>c_2>0$ such that $(c_1,c_1,c_2,c_2)\in B(\mp^0,\delta)$.  Therefore,  Property \ref{property: CC} can not hold.
\end{proof}
\paragraph{Case (ii):    $\partial \vartheta(0,0;1,1)/\partial x$ and/or  $\partial \vartheta(0,0;1,2)/\partial x$ is non-zero}
Our first task is to choose a $\delta_0>0$. The choice of $\delta$ is explained below. 
Since the $\vartheta(\cdot;i,j)$'s are thrice continuously differentiable at (0,0) by Lemma \ref{lemma: CC: differentiability}, 
there exists $\epsilon_0>0$  so  that
\begin{align}
    \label{intheorem: CC: third derivative bound}
    \sup_{|y|<\epsilon_0}\bl \diffp[2,1]{\vartheta(0,y;i,j)}{y,x}\bl<\bl\diffp[2,1]{\vartheta(0,0;i,j)}{y,x}\bl+1
\end{align}
 for all $i,j\in[2]$. By Lemma \ref{lemma: CC: continuity of Lambda}, there exists  $\delta>0$ such that  for all $\mp\in B(\mp^0,\delta)$, it holds that $\|\Lambda^*(\mp)-\Lambda^*(\mp^0)\|_2\leq \epsilon_0$. We choose our $\delta_0>0$ to be smaller than   this $\delta$, one, and also the $\delta$'s in  Lemmas \ref{lemma: CC: hessian cont.} and \ref{lemma: CC: Lipschitz}. If Property \ref{property: CC} holds, we will take $\delta_0$ to be smaller than the $\delta$ mentioned in Property \ref{property: CC}.  In this proof, we will consider $\mp$'s of the form $\mp=(1+\Delta_1,1+\Delta_2,1,1)$ where $\Delta_1$ and $\Delta_2$ are reals satisfying $\Delta_1^2+\Delta_2^2<\delta_0^2$, which ensures  $\mp\in B(\mp^0,\delta_0)$. We will choose $\Delta_1$ and $\Delta_2$ so that $\max(1+\Delta_1,1+\Delta_2)>1$. If Property \ref{property: CC} holds, and  $\delta_0$ is chosen as mentioned above, our  $\mp$'s will satisfy  $x^*(\mp)>0$ by P1 of Property \ref{property: CC}.



 \begin{lemma}
 \label{lemma: CC: final key inequality}
 Suppose $\vartheta$ satisfies Property \ref{property: CC} and $\delta_0$ is chosen as described above. If $x^*(\mp)>0$ for some $\mp=(1+\Delta_1,1+\Delta_2,1,1)\in B(\mp^0,\delta_0)$, then 
\begin{align}
\label{intheorem: cc: final: key inequality}
  \MoveEqLeft  \Delta_1\slb \pdv{\vartheta(0,0;1,1)}{x}+y^*(\mp)\pdv[2]{\vartheta(0,0;1,1)}{y}{x}+C\Delta_1\srb\nn\\
  &\ +\Delta_2\slb \pdv{\vartheta(0,0;1,2)}{x}+ y^*(\mp)\pdv[2]{\vartheta(0,0;1,2)}{y}{x}+C\Delta_2\srb  > 0
\end{align}
for some constant $C> 0$ depending only on  $\vartheta$.
 \end{lemma}

\begin{proof}[Proof of Lemma \ref{lemma: CC: final key inequality}]
Because $\delta_0$ is smaller that the $\delta$ in Lemma \ref{lemma: CC: hessian cont.}, $\Lambda^*(\mp)\in U_0^3$ for all $\mp\in B(\mp^0,\delta_0)$. 
 Since $\Lambda^*(\mp)\in U_0^3$, $(0,y^*(\mp),z^*(\mp))\in U_0^3$. Therefore $\Phi(\cdot;\mp)$ is thrice continuously differentiable at $(0,y^*(\mp),z^*(\mp))$. 
Therefore using \eqref{intheorem: derivatives} we obtain that
 \begin{align}
 \label{intheorem: cc: final: 1st expansion}
-\pdv{\Phi(0,y^*(\mp),z^*(\mp);\mp)}{x}= &\   \mp_{11}\pdv{\vartheta(0,y^*(\mp);1,1)}{x}+\mp_{12}\pdv{\vartheta(0,y^*(\mp); 1,2)}{x}\nn\\
&\ +\mp_{21}\pdv{\vartheta(0,z^*(\mp);2,1)}{x}+\mp_{22}\pdv{\vartheta(0,z^*(\mp); 2,2)}{x}.
 \end{align}
 Since $\Lambda^*(\mp)\in U_0^3$, $(0,y^*(\mp))\in U_0^2$. Therefore, the $\vartheta(\cdot;i,j)$'s are thrice continuously differentiable at $(0,y^*(\mp))$. Since the line joining $y^*(\mp)$ and $0$ is in $U_0$, we can take a second order Taylor series expansion of  $y\mapsto \pdv{\vartheta(0,y;1,1)}{x}$  at $y^*(\mp)$ around zero, which yields
 \[\pdv{\vartheta(0,y^*(\mp);1,1)}{x}=\pdv{\vartheta(0,0;1,1)}{x}+y^*(\mp)\pdv[2]{\vartheta(0,0;1,1)}{y}{x}+\frac{y^*(\mp)^2}{2}\diffp[2,1]{\vartheta(0,\xi_1;1,1)}{y,x},\]
 where $\xi_1$ is between $0$ and $y^*(\mp)$. 
 We can have a similar Taylor series expansion for all the other terms, leading to
 \begin{align*}
   -\pdv{\Phi(0,y^*(\mp),z^*(\mp);\mp)}{x}=&\ \pdv{\Phi(0;\mp)}{x}+ y^*(\mp) \slb \Delta_1 \pdv[2]{\vartheta(0,0;1,1)}{y}{x}+ \Delta_2\pdv[2]{\vartheta(0,0;1,2)}{y}{x}\srb\\
    &\  +y^*(\mp)M_{12}(0,0)+z^*(\mp)N_{12}(0,0)\\
    &\   +(1+\Delta_1)\frac{y^*(\mp)^2}{2}\diffp[2,1]{\vartheta(0,\xi_1;1,1)}{y,x}\\
    &\ +(1+\Delta_2)\frac{y^*(\mp)^2}{2}\diffp[2,1]{\vartheta(0,\xi_2;1,2)}{y,x}\\
    &\ +\frac{z^*(\mp)^2}{2}\diffp[2,1]{\vartheta(0,\xi_3;2,1)}{y,x} +\frac{z^*(\mp)^2}{2}\diffp[2,1]{\vartheta(0,\xi_4;1,1)}{y,x}.
 \end{align*}
 where $\xi_2$ is between $0$ and $y^*(\mp)$ and $\xi_3$, $\xi_4$ is between $0$ and $z^*(\mp)$, and $M_{12}$ and $N_{12}$ are as defined in \eqref{def: cc: M12 and N12}.
 Note that $M_{12}(0,0)$ and $N_{12}(0,0)$ vanish by Lemma \ref{lemma: CC: M12=0}. Also, 
 \[\pdv{\Phi(0;\mp)}{x}=\pdv{\Phi(0;\mp^0)}{x}+\Delta_1\pdv{\vartheta(0,0;1,1)}{x}+\Delta_2\pdv{\vartheta(0,0;1,2)}{x},\]
whose first term is  zero by \eqref{intheorem: cc: partial derivatives sum zero}. Note that \eqref{intheorem: cc: inequality} implies when $x^*(\mp)>0$, we must have $-\pdv{\Phi(0,y^*,z^*;\mp)}{x}>0$. Therefore, \eqref{intheorem: cc: final: 1st expansion} yields that 
 \begin{align}
 \label{intheorem: cc: main proof big expansion inequality}
   \MoveEqLeft   \Delta_1\pdv{\vartheta(0,0;1,1)}{x}+\Delta_2\pdv{\vartheta(0,0;1,2)}{x} +y^*(\mp) \slb \Delta_1 \pdv[2]{\vartheta(0,0;1,1)}{y}{x}+ \Delta_2\pdv[2]{\vartheta(0,0;1,2)}{y}{x}\srb\nn\\
    &\ +(1+\Delta_1)\frac{y^*(\mp)^2}{2}\diffp[2,1]{\vartheta(0,\xi_1;1,1)}{y,x}+(1+\Delta_2)\frac{y^*(\mp)^2}{2}\diffp[2,1]{\vartheta(0,\xi_2;1,2)}{y,x}\nn\\
    &\ +\frac{z^*(\mp)^2}{2}\diffp[2,1]{\vartheta(0,\xi_3;2,1)}{y,x} +\frac{z^*(\mp)^2}{2}\diffp[2,1]{\vartheta(0,\xi_4;2,2)}{y,x}>0.
 \end{align}
 We have chosen $\delta_0$ so that $\|\Lambda^*(\mp)\|_2\leq \epsilon_0$ for all $\mp\in B(\mp^0,\delta)$ where  $\epsilon_0$ is as in \eqref{intheorem: CC: third derivative bound}. Since $\|\Lambda^*(\mp)\|_2<\epsilon_0$, $|y^*(\mp)|$ and $|z^*(\mp)|$ are smaller than $\epsilon_0$. Therefore, $\xi_1$, $\xi_2$, $\xi_3$ and $\xi_4$ are smaller than $\e_0$ in absolute value. Thus \eqref{intheorem: CC: third derivative bound} yields 
 \begin{gather*}
    \bl \diffp[2,1]{\vartheta(0,\xi_1;1,1)}{y,x}\bl<\bl\diffp[2,1]{\vartheta(0,0;1,1)}{y,x}\bl+1.\\
     \bl\diffp[2,1]{\vartheta(0,\xi_2;1,2)}{y,x}\bl<\bl\diffp[2,1]{\vartheta(0,0;1,2)}{y,x}\bl+1.\\
     \bl\diffp[2,1]{\vartheta(0,\xi_3;2,1)}{y,x}\bl<\bl\diffp[2,1]{\vartheta(0,0;2,1)}{y,x}\bl+1.\\
     \bl\diffp[2,1]{\vartheta(0,\xi_4;2,2)}{y,x}\bl<\bl\diffp[2,1]{\vartheta(0,0;2,2)}{y,x}\bl+1.
 \end{gather*}
 In addition, since $\delta_0$is chosen to be smaller than one, $|\Delta_1|,|\Delta_2|<1$. 
 Thus
 \begin{align*}
     \MoveEqLeft (1+\Delta_1)\frac{y^*(\mp)^2}{2}\diffp[2,1]{\vartheta(0,\xi_1;1,1)}{y,x}+(1+\Delta_2)\frac{y^*(\mp)^2}{2}\diffp[2,1]{\vartheta(0,\xi_2;1,2)}{y,x}\\
    &\ +\frac{z^*(\mp)^2}{2}\diffp[2,1]{\vartheta(0,\xi_3;2,1)}{y,x} +\frac{z^*(\mp)^2}{2}\diffp[2,1]{\vartheta(0,\xi_4;2,2)}{y,x}\leq C(y^*(\mp)^2+z^*(\mp)^2)
 \end{align*}
 where the constant $C>0$ depends only on $\vartheta(\cdot;i,j)$'s.
However, by Lemma \ref{lemma: CC: Lipschitz}, $y^*(\mp)^2$ and $z^*(\mp)^2$ are bounded by $C\|\mp-\mp^0\|_2^2= C(\Delta_1^2+\Delta_2^2)$ where $C>0$ is a constant depending on the $\vartheta(\cdot;i,j)$'s. Therefore,
\begin{align*}
     \MoveEqLeft (1+\Delta_1)\frac{y^*(\mp)^2}{2}\diffp[2,1]{\vartheta(0,\xi_1;1,1)}{y,x}+(1+\Delta_2)\frac{y^*(\mp)^2}{2}\diffp[2,1]{\vartheta(0,\xi_2;1,2)}{y,x}\\
    &\ +\frac{z^*(\mp)^2}{2}\diffp[2,1]{\vartheta(0,\xi_3;2,1)}{y,x} +\frac{z^*(\mp)^2}{2}\diffp[2,1]{\vartheta(0,\xi_4;2,2)}{y,x}\leq C(\Delta_1^2+\Delta_2^2).
 \end{align*}
The above, combined with \eqref{intheorem: cc: main proof big expansion inequality} implies \eqref{intheorem: cc: final: key inequality}.
\end{proof}



Now we are ready to show that under case (ii), $\vartheta$ does not satisfy Property \ref{property: CC}. We will divide case (ii) in two subcases:  (a) either $\partial{\vartheta(0,0;1,1)}/\partial{x}< 0$  or $\partial{\vartheta(0,0;1,2)}/\partial{x}< 0$ and (b) either $\partial{\vartheta(0,0;1,1)}/\partial{x}> 0$ or  $\partial{\vartheta(0,0;1,2)}/\partial{x}> 0$. If (ii) holds then either (ii) a or (ii) b needs to hold. There can be overlap between cases (a) and (b), but it will not affect the proof.  We will show that, if  $\vartheta$ satisfies Property \ref{property: CC}, then we will encounter contradiction  in either cases. 

\paragraph{Case (ii) a: $\partial{\vartheta(0,0;1,1)}/\partial{x}< 0$  or $\partial{\vartheta(0,0;1,2)}/\partial{x}< 0$ } 
First consider the case $\pdv{\vartheta(0,0;1,1)}{x}< 0$. If possible, suppose Property \ref{property: CC} holds. 
Consider $\Delta_1=1/k$ and $\Delta_2=1/k^2$, where $k\in\NN$ is a large number so that the resulting $\mp_k=(1+1/k,1+1/k^2,1,1)$ is in  $B(\mp^0,\delta_0)$. Observe that $\mp=\mp_k$ satisfies $\max(\mp_{11},\mp_{12})>\max(\mp_{21},\mp_{22})$. Therefore Property \ref{property: CC}  implies $x^*(\mp_k)>0$ for all $k\in\NN$. Hence, \eqref{intheorem: cc: final: key inequality} holds by Lemma \ref{lemma: CC: final key inequality}. For this $\Delta_1$ and $\Delta_2$,  \eqref{intheorem: cc: final: key inequality} reduces to
\begin{align*}
  \MoveEqLeft  \frac{1}{k}\slb \pdv{\vartheta(0,0;1,1)}{x}+y^*(\mp_k)\pdv[2]{\vartheta(0,0;1,1)}{y}{x}+C\frac{1}{k}\srb\\
  &\ +\frac{1}{k^2}\slb \pdv{\vartheta(0,0;1,2)}{x}+ y^*(\mp_k)\pdv[2]{\vartheta(0,0;1,2)}{y}{x}+C\frac{1}{k^2}\srb  > 0
\end{align*}
 for all  $\mp_k$ of the form $(1+1/k,1+1/k^2,1,1)$ for all sufficiently large  $k\in\NN$.
Multiplying both sides with $k$, we get 
\begin{align*}
  \MoveEqLeft   \pdv{\vartheta(0,0;1,1)}{x}+y^*(\mp_k)\pdv[2]{\vartheta(0,0;1,1)}{y}{x}+C\frac{1}{k}\\
  &\ +\frac{1}{k}\slb \pdv{\vartheta(0,0;1,2)}{x}+ y^*(\mp_k)\pdv[2]{\vartheta(0,0;1,2)}{y}{x}+C\frac{1}{k^2}\srb  > 0.
\end{align*}
Letting $k\to\infty$, and observing Lemma \ref{lemma: CC: continuity of Lambda} implies  $y^*(\mp_k)\to 0$ as $k\to\infty$ we obtain
\[ \pdv{\vartheta(0,0;1,1)}{x}\geq 0,\]
which leads to a contradiction. 
Therefore, $\vartheta$ can not satisfy Property \ref{property: CC} if \\$\partial{\vartheta(0,0;1,1)}/\partial{x}< 0$.
Similarly, we can show that $\vartheta$ can not satisfy Property \ref{property: CC} if $\partial{\vartheta(0,0;1,2)}/\partial{x}< 0$. by taking $\Delta_1=1/k^2$ and $\Delta_2=1/k$. Therefore, we have showed that if either $ \pdv{\vartheta(0,0;1,1)}{x}< 0$ or  $\pdv{\vartheta(0,0;1,2)}{x}< 0$, Property \ref{property: CC} can not hold.


\paragraph{Case ii(b): $\pdv{\vartheta(0,0;1,1)}{x}> 0$ or $\pdv{\vartheta(0,0;1,2)}{x}> 0$}  First consider the case $\pdv{\vartheta(0,0;1,1)}{x}> 0$. If possible, suppose $\vartheta$ satisfies Property \ref{property: CC}. Take $\Delta_1=-1/k$ and $\Delta_2=1/k^2$. Let us refer to the resulting sequence $(1-1/k,1+1/k^2,1,1)$ by $\mp_k$.  Note that $\mp_k\in B(\mp^0,\delta_0)$ if $k$ is sufficiently large and $x^*(\mp_k)>0$ for all $k\in\NN$ by Property \ref{property: CC}.  Then \eqref{intheorem: cc: final: key inequality}  holds by Lemma \ref{lemma: CC: final key inequality}.  
In this case, \eqref{intheorem: cc: final: key inequality} reduces to 
\begin{align*}
  \MoveEqLeft  -\frac{1}{k}\slb \pdv{\vartheta(0,0;1,1)}{x}+y^*(\mp_k)\pdv[2]{\vartheta(0,0;1,1)}{y}{x}-C\frac{1}{k}\srb\\
  &\ +\frac{1}{k^2}\slb \pdv{\vartheta(0,0;1,2)}{x}+ y^*(\mp_k)\pdv[2]{\vartheta(0,0;1,2)}{y}{x}+C\frac{1}{k^2}\srb  > 0
\end{align*}
for all sufficiently large $k$.
Multiplying both sides by $k$, we get
\begin{align*}
  \MoveEqLeft  -\slb \pdv{\vartheta(0,0;1,1)}{x}+y^*(\mp_k)\pdv[2]{\vartheta(0,0;1,1)}{y}{x}-C\frac{1}{k}\srb\\
  &\ +\frac{1}{k}\slb \pdv{\vartheta(0,0;1,2)}{x}+ y^*(\mp_k)\pdv[2]{\vartheta(0,0;1,2)}{y}{x}+C\frac{1}{k^2}\srb  > 0.
\end{align*}
Letting $k\to\infty$ and observing $y^*(\mp_k)\to 0$ by Lemma \ref{lemma: CC: continuity of Lambda}, we obtain that
\begin{align*}
  - \pdv{\vartheta(0,0;1,1)}{x}\geq 0,
\end{align*}
which is a contradiction. Therefore, $\vartheta$ can not satisfy Property \ref{property: CC} in this case. The proof for the  $\partial\vartheta(0,0;1,2)/\partial x> 0$ follows similarly by taking $\Delta_1=1/k^2$ and $\Delta_2=-1/k$ and letting $k\to\infty$.

\subsection{Proofs of main   lemmas}
\label{sec: cc: proof of main lemmas}
\subsubsection{Proof of Lemma \ref{lemma: CC: 0 in int}}
\label{secpf: of lemma  CC: 0 in int}
\begin{proof}[Proof of Lemma \ref{lemma: CC: 0 in int}]
In this proof, we will work with $-\eta$ because it is convex.
Because $\cap_{i=1}^{k_1}\cap_{j=1}^{k_2}\iint(\dom(\psi(\cdot;i,j))\neq \emptyset$, there exist $\mx\in\RR^{k_1}$ and $\mw\in\RR^{k_2}$ so that
 $(\mx,\mw)\in\iint(\dom(\psi(\cdot;i,j))$ for all $i\in[k_1]$ and $j\in[k_2]$. 
Let us fix $i\in[k_1]$. Our first task is to show that $\mz_{k_1+k_2-2}\in\iint(\dom(-\eta))$. 
Let us denote
  \[\dmxi=(\mx_i-\mx_1,\ldots,\mx_i-\mx_{i-1},\mx_i-\mx_{i+1},\ldots,\mx_i-\mx_{k_1}).\]
    Then for all $j\in[k_2]$, 
    \begin{gather*}
        -\eta(\dmxi, \mw_{j}-\mw_{1},\ldots,\mw_{j}-\mw_{k_2})<\infty
    \end{gather*}
    where the term $\mw_{i}-\mw_{i}$ is omitted.
 Since $\eta$ is blockwise permutation symmetric by \eqref{def: cc: permutation equivariance 2 stage}, for all $j\in[2:k_2]$, 
     \begin{align*}
      \MoveEqLeft  -\eta(\dmxi, \mw_{j}-\mw_{j+1},\mw_{j}-\mw_{j+2}\ldots,\mw_{j}-\mw_{k_2},\mw_{j}-\mw_{1},\ldots,\mw_{j}-\mw_{j-1})\\
        =&\  -\eta(\dmxi, \mw_{j}-\mw_{1},\ldots,\mw_{j}-\mw_{k_2})<\infty.
    \end{align*}
   For the sake of algebra, we will introduce $\mw_j$ for $j>k_2$. We  extend the definition of $\mw_j$ to $\NN$ letting $\mw_{j}=\mw_{j \mod k_2}$ for all $j\in\NN$. Therefore, we can write $\mw_1=\mw_{k_2+1}$, $\ldots$, $\mw_{j-1}=\mw_{j+k_2-1}$, and so on. According to our notation
  \begin{align*}
      \MoveEqLeft  \eta(\dmxi, \mw_{j}-\mw_{j+1},\ldots,\mw_{j}-\mw_{k_2},\mw_{j}-\mw_{1},\ldots,\mw_{j}-\mw_{j-1})\\
        =&\  \eta(\dmxi, \mw_{j}-\mw_{j+1},\ldots,\mw_{j}-\mw_{j+k_2-1})
    \end{align*}
for all $j\in[k_2]$. Therefore, using this notation, 
\begin{align}
\label{inlemma: CC: dom lemma: finity}
 -\eta(\dmxi, \mw_{j}-\mw_{j+1},\ldots,\mw_{j}-\mw_{j+k_2-1})<\infty
\end{align}
for all $j\in[k_2]$.
On the other hand, since $-\eta$ is convex, Jensen inequality implies that
\begin{align}
\label{inlemma: CC: 0 in dom: jensen}
\MoveEqLeft -\frac{1}{k_2}\sum_{j\in[k_2]} \eta(\dmxi, \mw_{j}-\mw_{j+1},\ldots,\mw_{j}-\mw_{j+k_2-1})\nn\\
\geq &\ -\eta\lb \dmxi, \frac{\sum_{j\in[k_2]}(\mw_i-\mw_{j+1})}{k_2},  \ldots, \frac{\sum_{j\in[k_2]}(\mw_i-\mw_{j+k_2-1})}{k_2}\rb
\end{align}
    where $\mw_j=\mw_{j\mod 2}$ as per our notation. Moreover, 
    \begin{align*}
   \sum_{j\in[k_2]}(\mw_j-\mw_{j+1})=\sum_{j\in[k_2]}\mw_j-\sum_{j\in[k_2]}\mw_{j+1}=\sum_{j\in[k_2]}\mw_j-\sum_{j=2}^{k_2}\mw_{j}-\mw_{k_2+1}=0
    \end{align*}
because $\mw_{k_2+1}=\mw_1$.
    Similarly,
 \begin{align*}
     \sum_{j\in[k_2]}(\mw_j-\mw_{j+k_2-1})= \sum_{j\in[k_2]}\mw_j-\sum_{j\in[k_2]}\mw_{j+k_2-1}=\sum_{j\in[k_2]}\mw_j-\sum_{j=2}^{k_2}\mw_{j+k_2-1}-\mw_{k_2},
 \end{align*}  
which vanishes  because $\sum_{j=2}^{k_2}\mw_{j+k_2-1}=\sum_{j=2}^{k_2}\mw_{j-1}=\sum_{j=1}^{k_2-1}\mw_{j}$. In general, for any  $1\leq j\leq k_2-1$,  we obtain that
\begin{align*}
&\  \sum_{j\in[k_2]}\mw_{j+j}=\sum_{j=1}^{k_2-j}\mw_{j+j}+\sum_{j=k_2-j+1}^{k_2}\mw_{j+j}=\sum_{j=j+1}^{k_2}\mw_j+\sum_{j=k_2+1}^{k_2+j}\mw_j\\
=&\ \sum_{j=j+1}^{k_2}\mw_j+\sum_{j=1}^{j}\mw_j=\sum_{j\in[k_2]}\mw_j.  \end{align*}
Therefore, $  \sum_{j\in[k_2]}(\mw_j-\mw_{j+j})=0$ for any  $1\leq j\leq k_2-1$, which, combined with \eqref{inlemma: CC: dom lemma: finity} and \eqref{inlemma: CC: 0 in dom: jensen}, implies that
$-\eta(\dmxi,\mz_{k_2-1})<\infty$.
Since $i$ was arbitrary, we obtain $-\eta(\dmxi,\mz_{k_2-1})<\infty$ for each $i\in[k_1]$.

For each $i\in[k_1]$, the block-wise permutation-symmetry of $\eta$ as in \eqref{def: cc: permutation equivariance 2 stage} implies that 
\begin{align*}
    \eta(\dmxi,\mz_{k_2-1})=\eta(\mx_i-\mx_{i+1},\ldots,\mx_i-\mx_{i+k_1-1},\mz_{k_2-1}),
\end{align*}
where for $i>k_1$, $\mx_i=\mx_{i \mod k_1}$. Since 
\[-\eta(\mx_i-\mx_{i+1},\ldots,\mx_i-\mx_{i+k_1-1},\mz_{k_2-1})<\infty\]
for all $i\in[k_1]$, by Jensen's inequality, it follows that
\begin{align*}
\infty> &\ -\frac{1}{k_2}\sum\limits_{i\in[k_2]} \eta( \mx_i-\mx_{i+1},\ldots,\mx_i-\mx_{i+k_1-1},\mzk)\\
\geq &\ -\eta\llb \frac{\sum\limits_{i\in[k_1]}(\mx_i-\mx_{i+1})}{k_1}, \frac{\sum\limits_{i\in[k_1]}(\mx_i-\mx_{i+2})}{k_1}, \ldots, \frac{\sum\limits_{i\in[k_1]}(\mx_i-\mx_{i+k_1-1})}{k_1},\mzk\rrb . 
\end{align*}
Proceeding as in the case for $\mw$, we can prove that the right hand side of the above inequality equals $-\eta(\mz_{k_1-1},\mz_{k_2-2})$. Thus we have shown that $\mz_{k_1+k_2-2}\in\dom(-\eta)$.

It remains to prove $\mz_{k_1+k_2-2}\in\iint(\dom(-\eta))$, which will be proved by contradiction. 
If possible, suppose $\mz_{k_1+k_2-2}\notin\iint(\dom(-\eta))$. Then we must have $\mz_{k_1+k_2-2}\in\partial(\dom(-\eta))$.  We claim that, in that case, there exists an open orthant in $\RR^{k_1+k_2-2}$, which we will soon define, whose intersection with $\dom(-\eta)$ is the nullset.  For any $k\in\NN$ and $\bse\in\{\pm 1\}^{k}$, we define the open orthant $\mathcal{Q}(\bse)$ in $\RR^k$  as 
\[\mathcal{Q}(\bse)=\{\mx\in\RR^k: \mx_i\epsilon_i>0\text{ for all }i\in[k]\}.\]
Note that an open orthant in $\RR^{k_1+k_2-2}$ is of the form $\Q(\bse)\times \Q(\bse')$ where $\Q(\bse)$ is an open orthant in $\RR^{k_1-1}$ corresponding to some $\bse\in\{\pm 1\}^{k_1-1}$ and $\Q(\bse')$ is an open orthant in $\RR^{k_2-1}$ corresponding to some $\bse'\in\{\pm 1\}^{k_2-1}$.
\begin{fact}
\label{claim: CC: orthant}
If $\eta$ is concave with $\mz_{k_1+k_2-2}\in\partial(\dom(-\eta))$, then there exist $\bse,\bse'\in\{\pm 1\}^{k}$  so that the open orthant  $\Q(\bse)\times \Q(\bse')$ satisfies $\Q(\bse)\times \Q(\bse')\cap\dom(-\eta)=\emptyset$. 
\end{fact}
Fact \ref{claim: CC: orthant} is proved in Section \ref{secpf: claim: cc: orthant}.  We will prove that under the setup of the current lemma,  $\{\Q(\bse)\times \Q(\bse')\}\cap\dom(-\eta)\neq \emptyset$  for all  $\bse$ and $\bse'$ in $\{\pm 1\}^{k}$, which will lead to the desired contradiction. Fix $\bse$ and $\bse'$ in $\{\pm 1\}^{k}$. We denote the number of positive elements in $\bse$ and $\bse'$ by $a_1$ and $a_2$, respectively. Since $\cap_{i=1}^{k_1}\cap_{j=1}^{k_2}\iint(\dom(\psi(\cdot;i,j))\neq \emptyset$, there is an open set $U\subset \RR^{k_1+k_2}$ so that for all $(\mx,\mw)\in U$, $-\psi(\mx,\mw;i,j)<\infty$ for all $i\in[k_1]$ and $j\in[k_2]$. In particular, we can choose $\mx\in\RR^{k_1}$ and $\mw\in\RR^{k_2}$ such that the elements in $\mx$ and $\mw$ are all distinct.
Suppose $r\in[k_2]$ is such that $\mw_r=\mw_{(a_2+1)}$. 
Since $\mw$ has distinct elements, $r$ is unique  and there are $a_2$ many $j$'s such that $\mw_j<\mw_r$ and $k_2-1-a_2$ many $\mw_j$'s are strictly greater than $\mw_r$. Note that
\begin{align*}
 \infty>   -\psi(\mx,\mw;i,j)=&\ -\eta(\dmxi, \mw_j-\mw_1,\ldots,\mw_j-\mw_{k_2})
\end{align*}
where the term $\mw_j-\mw_j$ is omitted. 
The vector $\mv=(\mw_i-\mw_1,\ldots,\mw_i-\mw_{k_2})$ has exactly $a_2$ many positive and $k_2-1-a_2$ many negative elements. Therefore, we can permute the vector $\mv$ in a way so that $\ppi(\mv)\in \Q(\bse')$. However, by \eqref{def: cc: permutation equivariance 2 stage}, 
\[\eta(\dmxi, \mv)=\eta(\dmxi, \ppi(\mv))\]
for any permutation $\ppi:[k_2-1]\mapsto[k_2-1]$. Therefore, there is a permutation $\ppi$ so that $\ppi(\mv)\in \Q(\bse')$ and $(\dmxi,\ppi(\mv))\in\dom(-\eta)$ both holds. Thus we have shown that
for each $i\in[k_1]$, there exists $\mv_i\in \Q(\bse')$ such that 
\begin{align*}
    (\dmxi,\mv_i)\in\dom(-\eta).
\end{align*}
Since the elements of $\mx$ is distinct, there is an $i\in[k_1]$ so that $\mx_{(a_1+1)}=\mx_i$. Then $a_1$ many elements of $\dmxi$ is positive, and the rest are negative. The block permutation symmetry of   \eqref{def: cc: permutation equivariance 2 stage}implies that $\eta(\dmxi, \mv_i)=\eta(\ppi(\dmxi), \mv_i)$ for any permutation $\ppi:[k_1]\mapsto[k_1]$. In particular, we can choose $\ppi$ in a way so that $\ppi(\dmxi)\in \Q(\bse)$. Therefore, $\ppi(\dmxi)\in  \Q(\bse)$ and $-\eta(\ppi(\dmxi),\mv_i)<\infty$. This implies $(\ppi(\dmxi),\mv_i)\in \Q(\bse)\times \Q(\bse')$. However, since $\eta(\dmxi, \mv_i)=\eta(\ppi(\dmxi), \mv_i)$, it holds that $-\eta(\ppi(\dmxi), \mv_i)<\infty$. Hence $\{\Q(\bse)\times \Q(\bse')\}\cap \dom(-\eta)$ is non-empty.  Then Fact \ref{claim: CC: orthant} implies   $\mz_{k_1+k_2-2}\notin\iint(\dom(-\eta))$ can not  hold. Hence, we must have $\mz_{k_1+k_2-2}\in\iint(\dom(-\eta))$.

\end{proof}

\subsubsection{Proof of Lemma \ref{lemma: cc: 0 maximum for smooth}}
\label{secpf: proof of 0 solution lemma}
\begin{proof}[Proof of Lemma \ref{lemma: cc: 0 maximum for smooth}]

Note that \eqref{intheorem: CC: psi to phi} implies when $\mp_{ij}=1$ for all $i\in[k_1]$ and $j\in[k_2]$, then 
\begin{align*}
V^\psi(\mx, \my_1,\ldots, \my_{k_1};\mp)= \sum\limits_{i\in[k_1]}\sum\limits_{j\in[k_2]}\eta(\mx_i-\mx_1,\ldots,\mx_i-\mx_{k_1},\my_{ij}-\my_{i1},\ldots, \my_{ij}-\my_{ik_2}),
\end{align*}
for all $\mx\in\RR^{k_1}$ and $\my_1,\ldots,\my_{k_1}\in\RR^{k_2}$. 
Note that if $(\mx,\my_1,\ldots,\my_{k_1})\in \C_e$, then 
\begin{align*}
V^\psi(\mx, \my_1,\ldots, \my_{k_1};\mp)= \sum\limits_{i\in[k_1]}\sum\limits_{j\in[k_2]}\eta(\mathbf 0_{k_1-1}, \mathbf 0_{k_2-1})=V^\psi(\mathbf 0_{k_1},\mathbf 0_{k_2},\ldots, \mathbf 0_{k_2};\mp).
\end{align*}
Since  $(\mathbf 0_{k_1-1}, \mathbf 0_{k_2-1})\in\iint(\dom(-\eta))$ by Lemma \ref{lemma: CC: 0 in int}, the above implies  that $(\mathbf 0_{k_1},\mathbf 0_{k_2},\ldots, \mathbf 0_{k_2})\in \iint(\dom(-V^{\psi}(\cdot;\mp)))$.
The above display also implies that the proof will be complete if we can show that $(\mathbf 0_{k_1},\mathbf 0_{k_2},\ldots, \mathbf 0_{k_2})$ lies in \\
$ \argmax_{\mx,\my_1,\ldots,\my_{k_1}}V^{\psi}(\mx,\my_1,\ldots,\my_{k_1};\mp)$. Note that $(\mathbf 0_{k_1},\mathbf 0_{k_2},\ldots, \mathbf 0_{k_2})$ is just the zero-vector $\mz_{k_1+k_1k_2}$ and $(\mathbf 0_{k_1-1}, \mathbf 0_{k_2-1})$
is just the zero-vector $\mz_{k_1+k_2-2}$. To streamline notation, we will use the expressions $\mzb$ and $\mzs$ for those vectors from now on.

Let us denote $\myf=-V^\psi(\cdot;\mp)$  so that $\myf$ is a  convex function with $\mzb\in\dom(\myf)$. Thus we need to show that $\mzb$ is a minimizer of $\myf$. Hence, it is enough to prove $\partial \myf(\mzb)\ni \mzb$ \citep[Theorem 2.2.1, pp. 177,][]{hiriart}.  In particular, we will show that $\partial \myf(\mzb)$ is singletone, and equals $\mzb$.  For $i\in[k_1]$ and $j\in[k_2]$, let us define 
\begin{align}
    \label{intheorem: CC: def: gij and hij}
   \MoveEqLeft \tph_{ij}(\mx,\mw)=-\eta(\mx_i-\mx_1,\ldots,\mx_i-\mx_{k_1},\mw_{j}-\mw_{1},\ldots, \mw_{j}-\mw_{k_2})\nn\\
    &\ \ \text{and}\ \myh_{ij}(\mx,\my_1,\ldots,\my_{k_1})=\tph_{ij}(\mx,\my_i),
\end{align}
where $\mx\in\RR^{k_1}$, $\my\in\RR^{k_2}$, $\my_i\in\RR^{k_2}$ for $i\in[k_1]$, and the terms $\mx_i-\mx_i$ and $\my_{j}-\my_{j}$ are omitted in the above definitions. Then by Theorem 4.1.1 of \cite{hiriart}, 
\begin{align*}
\partial \myf(\mx, \my_1,\ldots, \my_{k_1})=\sum\limits_{i\in[k_1]}\sum\limits_{j\in[k_2]}\partial \myh_{ij}(\mx, \my_1,\ldots, \my_{k_1})
\end{align*}
where for sets $C_1,\ldots,C_n$, the sum $C_1+\ldots+C_n$ is the Minkowski sum defined in Section \ref{sec: notation app}. 
If $(\mx, \my_1,\ldots, \my_{k_1})$ lies in  $\iint(\dom(\myh_{ij}))$ for some $i\in[k_1]$ and $j\in[k_2]$, then the corresponding $\partial \myh_{ij}(\mx, \my_1,\ldots, \my_{k_1})$ is non-empty \citep[see theorem 23.4, pp.214, of][for a proof]{rockafellar}. Otherwise, it is just an empty set. Thus the above expression continues to hold even if $(\mx, \my_1,\ldots, \my_{k_1})\notin\iint(\dom(\myh_{ij}))$ for  some $i\in[k_1]$ and $j\in[k_2]$.

Consider $i\in[k_1]$ and $j\in[k_2]$. Equation \ref{intheorem: CC: def: gij and hij} implies that among the $\my_r$'s ($r\in[k_1]$),  $\myh_{ij}$ depends only on $\my_i$. Hence, if $(\mt,\mzz_1,\ldots, \mzz_{k_1})\in \partial \myh_{ij}(\mx, \my_1,\ldots, \my_{k_1})$ for some $\mt\in\RR^{k_1}$ and $\mzz_1,\ldots,\mzz_{k_1}\in\RR^{k_2}$, then $\mzz_r=0$ for $r\neq i$ \citep[cf. Remark 4.1.2, pp. 184,][]{hiriart}.

Moreover, $(\mt,\mz_{k_2}\ldots, \mzz_i,\ldots,\mz_{k_2})$ is in $\partial \myh_{ij}(\mx, \my_1,\ldots, \my_{k_1})$ if and only if $(\mt,\mzz_i)\in \partial \tph_{ij}(\mx,\my_i)$. Using the above, we can specify $\partial \myh_{ij}(\mx, \my_1,\ldots, \my_{k_1})$ for any $\mx\in\RR^{k_1}$ and $\my_1,\ldots,\my_{k_1}\in\RR^{k_2}$  as follows:
\begin{equation}
    \label{intheorem: cc: delgij to del hij}
   \partial \myh_{ij}(\mx, \my_1,\mydots, \my_{k_1})=  \{(\mt,\mz_{k_2}\mydots, \mzz,\mydots,\mz_{k_2}): (\mt,\mzz)\in\partial \tph_{ij}(\mx,\my_i), \mt\in\RR^{k_1},\mzz\in\RR^{k_2}\}.
\end{equation}
Therefore, we need to find $\partial \tph_{ij}$ if we want to infer on $\partial \myh_{ij}$. Equation
\ref{intheorem: CC: def: gij and hij} also implies that for each $i,j$, $\tph_{ij}(\mx,\my)=h(C_{ij}(\mx,\my))$ where $C_{ij}=\begin{bmatrix}
    \AA_{i} & \mz_{(k_1-1)\times k_2}\\
   \mz_{(k_2-1)\times k_1} & \BB_{j}
\end{bmatrix}$ 
is a  matrix with $\AA_{i}\in \RR^{(k_1-1)}\times \RR^{k_1}$ and $\BB_{j}\in \RR^{(k_2-1)}\times \RR^{k_2}$ satisfying 
\begin{align*}
    \AA_{i} \mx= \begin{bmatrix} \mx_i-\mx_1\\ \vdots \\ \mx_i-\mx_{i-1}\\\mx_i-\mx_{i+1}\\ \vdots \\ \mx_i-\mx_{k-1}\end{bmatrix} \quad\text{and}\quad   \BB_{j} \mw= \begin{bmatrix} \mw_j-\mw_1\\ \vdots \\ \mw_j-\mw_{j-1}\\ \mw_j-\mw_{j+1}\\ \vdots\\ \mw_j-\mw_{k-2}\end{bmatrix}
\end{align*}
for any $i\in[k_1]$, $j\in[k_2]$,  $\mx\in\RR^{k_1}$,  and $\mw\in\RR^{k_2}$. The above implies
\begin{align}
\label{inlemma: CC: 0 solutions: Ai and Bj def}
&\ \AA_1= \begin{bmatrix}
    \mo_{k_1-1} & -I_{k_1-1}
\end{bmatrix},\ \AA_{k_1}= \begin{bmatrix}
   -I_{k_1-1} & \mo_{k_1-1} 
\end{bmatrix},\nn\\
&\ \BB_1=\begin{bmatrix}
    \mo_{k_2-1} & -I_{k_2-1}
\end{bmatrix},\ \BB_{k_2}= \begin{bmatrix}
   -I_{k_2-1} & \mo_{k_2-1},
\end{bmatrix},\nn\\
 &\ \AA_i= \begin{bmatrix}
     -\mathbf{I}_{i-1} & \mo_{i-1} & \mathbf 0_{(i-1)\times (k_1-i)}\\
     \mathbf {0}_{(k_1-i)\times(i-1)} & \mo_{k_1-i} & -\mathbf{I}_{k_1-i}
    \end{bmatrix},\\
    \text{and}  &\  \BB_j=\begin{bmatrix}
     -\mathbf{I}_{j-1} & \mo_{j-1} & \mathbf 0_{(j-1)\times (k_2-i)}\nn\\
     \mathbf 0_{(k_2-j)\times(j-1)} & \mo_{k_2-j} & -\mathbf{I}_{k_2-j}
    \end{bmatrix}
\end{align}
for $i\in[2:k_1-1]$ and $j\in[2: k_2-1]$. Here we did a slight  overload of notation because $\AA_i$ and $\BB_j$'s are different from those defined in Section \ref{sec: intheorem: cc: step 1}.
Since $k_1,k_2\geq 2$ by our assumption on $k_1$ and $k_2$, the above matrices are well-defined. 
Therefore,  $\partial \tph_{ij}(\mx,\mw)=C_{ij}^T\partial (-\eta)(C_{ij}(\mx,\mw))$ for $\mx\in\RR^{k_1}$ and $\mw\in\RR^{k_2}$  \cite[cf. Theorem 4.2.1 of][]{hiriart}.  Since $\mzs\in\iint(\dom(-\eta))$ by Lemma \ref{lemma: CC: 0 in int} and $-\eta$ is differentiable by our assumption, it follows that $-\eta$ is differentiable at $\mzs$. Thus $\partial (-\eta)(\mzs)$ is singletone and equals $-\grad \eta(\mzs)$.   Therefore, $\tph_{ij}$ is differentiable at $\mz_{k_1+k_2}$ and 
\begin{align}
    \label{inlemma: CC: gij gradient}
    \grad \tph_{ij}(\mz_{k_1+k_2})=-C_{ij}^T\grad \eta(\mzs).
\end{align}
Equation \ref{intheorem: CC: def: gij and hij} implies, $\myh_{ij}$, and consequently, $\myf=-V^\psi$, also are  differentiable at $\mzb$.

 We will calculate $\grad \eta(\mzs)$ now. To this end, we will use \eqref{def: cc: permutation equivariance 2 stage}, which implies 
\[\eta(\mbu_1,\mbu_2,\ldots,\mbu_{k_1-1},\mv)=\eta(\mbu_2,u_1,\ldots,\mbu_{k_1-1},\mv)\text{ for any }\mbu\in\RR^{k_1-1}\text{ and }\mv\in\RR^{k_2-1},\]
leading to
\[\pdv{\eta(\mbu_1,\mbu_2,\ldots,\mbu_{k_1-1},\mv)}{\mbu_1}=\pdv{\eta(\mbu_2,\mbu_1,\ldots,\mbu_{k_1-1},\mv)}{\mbu_1}\]
whenever $\eta$ is differentiable at $(\mbu,\mv)$. Therefore,
\[\eta_1(\mbu_1,\mbu_2,\ldots,\mbu_{k_1-1},\mv)=\eta_2(\mbu_2,\mbu_1,\ldots,\mbu_{k_1-1},\mv).\]
In particular, when $\mbu=\mz_{k_1-1}$, then it follows that $\eta_1(\mz_{k_1-1}, \mv)=\eta_2(\mz_{k_1-1},\mv)$. In the same way, we can show that $\eta_i(\mz_{k_1-1},\mv)$ is constant across $i\in[k_1-1]$. Similarly, we can show that $\eta_{k_1-1+j}(\mbu,\mz_{k_2-1})$ is constant across $j\in[k_2-1]$. In particular,
\[\eta_{i}(\mz_{k_1+k_2-2})=\eta_1\mz_{k_1+k_2-2})\text{ for all }i\in[k_1]\text{ and }\eta_{j}(\mz_{k_1+k_2-2})=\eta_{k_1}(\mz_{k_1+k_2-2})\]
for all $j\in[k_1:k_1+k_2-2]$. 
Thus
\[\grad \eta(\mzs)=(\eta_1(\mzs)\mo_{k_1-1}, \eta_{k_1}(\mzs)\mo_{k_2-1}).\]
Then  \eqref{inlemma: CC: gij gradient} implies that
\[\grad \tph_{ij}(\mz_{k_1+k_2})=-C_{ij}^T\grad \eta(\mzs)=-\begin{bmatrix}\eta_1(\mzs)\AA_{i}^T\mo_{k_1-1} \\\eta_{k_1}(\mzs)\BB_{j}^T\mo_{k_2-1}\end{bmatrix}.\]
We have already shown that $\myf$ is differentiable at $\mzb$. 
It remains to calculate $\grad \myf(\mzb)$.  To this end,   \eqref{intheorem: CC: def: gij and hij} and \eqref{intheorem: cc: delgij to del hij} imply that 
\begin{align*}
    \grad \myf(\mzb)=&\ -\sum\limits_{i\in[k_1]}\sum\limits_{j\in[k_2]}\begin{bmatrix}
    \eta_1(\mzs)\AA_{i}^T\mo_{k_1-1}\\ 0 \\ \vdots\\ 0\\ \eta_{k_1+1}(\mzs)\BB_{j}^T\mo_{k_2-1}\ (i\text{th position})\\ 0\\ \vdots\\ 0 \end{bmatrix}\\
    =&\ -\sum\limits_{j\in[k_2]}\begin{bmatrix}
    \eta_1(\mzs)\sum\limits_{i\in[k_1]}\AA_{i}^T\mo_{k_1-1}\\  \eta_{k_1+1}(\mzs)\BB_{j}^T\mo_{k_2-1}\ \\  \vdots\\ \eta_{k_1+1}(\mzs)\BB_{j}^T\mo_{k_2-1} \end{bmatrix}\\
    =&\ -\begin{bmatrix}
   k_2 \eta_1(\mzs)\sum\limits_{i\in[k_1]}\AA_{i}^T\mo_{k_1-1} \\\eta_{k_1+1}(\mzs)\sum\limits_{j\in[k_2]}\BB_{j}^T\mo_{k_2-1} \ \\  \vdots\\ h_{k_1+1}(\mzs)\sum\limits_{j\in[k_2]} \BB_{j}^T\mo_{k_2-1} \end{bmatrix}.
\end{align*}
Note that
\eqref{inlemma: CC: 0 solutions: Ai and Bj def} implies that
\begin{align*}
    \AA_i^T=\begin{bmatrix}
     -\mathbf{I}_{i-1} & \mathbf 0_{(i-1)\times(k_1-i)}\\  \mo_{i-1}^T & \mo_{k_1-i}^T\\ \mathbf 0_{(k_1-i)\times (i-1)} & -\mathbf{I}_{k_1-i}
    \end{bmatrix}\ \text{and}\  \BB_j^T=\begin{bmatrix}
     -\mathbf{I}_{j-1} & \mathbf 0_{(j-1)\times(j-1)}\\  \mo_{j-1}^T & \mo_{k_2-j}^T\\ \mathbf 0_{(k_2-j)\times (j-1)} & -\mathbf{I}_{k_2-j}
    \end{bmatrix}.
\end{align*}
Hence,
\begin{align*}
    \AA_i^T\mo_{k_1-1}=\begin{bmatrix}
        -\mo_{i-1}\\ k_1-1\\ -\mo_{k_1-i}
    \end{bmatrix} \text{ and } \BB_j^T\mo_{k_2-1}=\begin{bmatrix}
        -\mo_{j-1}\\ k_2-1\\ -\mo_{k_2-j}
    \end{bmatrix}.
\end{align*}
Thus $\sum\limits_{i\in[k_1]}\AA_{i}^T\mo_{k_1-1}=\mz_{k_1}$ and $\sum\limits_{j\in[k_2]} \BB_{j}^T\mo_{k_2-1}=\mz_{k_2}$, which indicates\\
$\grad \myf(\mzb)=\mzb$, which completes the proof.
\end{proof}

\subsubsection{Proof of Lemma \ref{lemma: CC: coercivity main lemma}}
\label{secpf: cc: coercive main lemma}
  \begin{proof}[Proof of Lemma \ref{lemma: CC: coercivity main lemma}]
       Since $\mz_{k}\in\dom(h)$ and $\mz_{k}\in \dom(h)\cap \text{Range}(C_i)$ for each $i\in[m]$, $(h\circ C_i)'_\infty=h_\infty'\circ C_i$ for each $i\in[m]$ by part (c) of Fact  \ref{fact: concave: prop 3.2.8 of hiriat}. Using part (a) of the same Fact, we obtain that for each  $\mw\in\RR^m_{>0}$,
       \[f_\infty'(\mx;\mw)=\sum\limits_{i=1}^m\mw_ih_\infty'(C_i\mx),\quad \mx\in\RR^k,\]
       where $f_\infty'(\cdot;\mw)$ is the recession function of $f(\cdot;\mw)$. Since $h$ is convex and bounded below, $f_\infty'(\cdot;\mw)$ is convex and bounded below for each $\mw\in\RR^m_{>0}$. If there exists $\mw_0\in\RR^m_{>0}$ so that the minima of the convex function $f(\cdot;\mw_0)$
   is not attained  in $\RR^k$, then $f_\infty'(\mx;\mw_0)=0$ for some $\mx\neq \mz_k$ by Fact \ref{fact: concave: coercivity and recession for concave}. Since $\mw_0\in\RR^m_{>0}$, it must hold that $h_\infty'(C_i\mx)=0$ for each $i\in[m]$ for this $\mx$. 
   Therefore,
   \[\sum\limits_{i=1}^m\mw_ih_\infty'(C_i\mx)=0\quad \text{ for all }\mw\in\RR^m_{>0}.\]
   Another application of Fact \ref{fact: concave: coercivity and recession for concave} yields that  the minima of the  convex function $\mx\mapsto \sum\limits_{i=1}^m\mw_ih(C_i\mx)$ is not attained for any $\mw\in\RR^m$. Hence, proved.
   \end{proof}

\subsubsection{Proof of Lemma \ref{lemma: cc: supremum}}
\label{secpf: claim PERM sup}
\begin{proof}[Proof of Lemma \ref{lemma: cc: supremum}]
We will fix $\mp\in\RR^4_{>0}$ and denote $\mbu^*(\mp)$, $\mv_1^*(\mp),\ldots,\mv_{k_1}^*(\mp)$ by $\mbu^*$, $\mv_1^*,\ldots,\mv_{k_1}^*$, respectively.
 We will prove Lemma \ref{lemma: cc: supremum} in three steps.
\begin{itemize}
    \item[(a)] We will show that there exists $x\in\RR$ so that $\mbu^* =x\mo_{k_1-1}$.
\item[(b)] We will next show that
$\mv_2^*=\ldots=\mv_{k_1}^*$.
\item[(c)] Finally, we will show that
$\mv_i^*=c\mo_{k_2-1}$ for some $c\in\RR$ for each $i\in[k_1]$.
\end{itemize}



\paragraph{Proving (a) the $\mbu_i^*$'s are equal}
For fixed $i$ and $j$ in $[k_1]$, let us denote $C_{ij}=\{\mbu\in\RR^{k_1-1}: \mbu_i=\mbu_j\}$. We will first show that $\mbu^* \in C_{ij}$ for all $i,j\in [k_1]$. Since the proof is trivial for $i=j$, we assume $i\neq j$. 

Fix $i,j\in[k_1]$ so that $i\neq j$. 
Let us denote by $\ppi$ the permutation on $[k_1]$ that swaps $i$ with $j$, i.e., $\ppi(i)=j$ and $\ppi(j)=i$. To show $\mbu^* \in C_{ij}$, our first step is to prove
\begin{align}
    \label{intheorem: cc: pi pi inv for u}
 V^{\eta}(\mbu,\mv_1,\ldots,\mv_{k_1};\mp) =V^{\eta}(\ppi \mbu,\mv_1,\mv_{\ppi(1)+1},\ldots, \mv_{\ppi(k_1-1)+1};\mp)
\end{align}
for all $\mbu\in\RR^{k_1-1}$ and $\mv_1,\ldots,\mv_{k_1}\in\RR^{k_2-1}$.
To that end, note that $\ppi \mbu=P_{ij}\mbu$ where $P_{ij}$ is the permutation matrix for swapping $i$th row with $j$th row. 
For any $\mbu\in\RR^{k_1-1}$ and $\mw\in\RR^{k_2-1}$, it follows that
$\eta(\ppi \mbu,\mw)=\eta(\mbu,\mw)$
by the blockwise permutation symmetry of $\eta$ in \eqref{def: cc: permutation equivariance 2 stage}. Also,
\begin{align*}
    \eta(\AA_i\ppi \mbu,\mw)=\eta(\AA_i P_{ij} \mbu,\mw)\stackrel{(a)}{=}\eta(P_{ij}\AA_jP_{ij}^2\mbu,\mw)\stackrel{(b)}{=}\eta(P_{ij}\AA_j\mbu,\mw)\stackrel{(b)}{=}\eta(P_{ij}^2\AA_j\mbu,\mw),
\end{align*}
which equals $\eta(\AA_j \mbu,\mw)$, 
where (a) uses the fact that $\AA_i=P_{ij}A_jP_{ij}$ for any $i,j\in[k_1]$, (b) uses the fact that $P_{ij}^2=I$ and (c) uses the blockwise permutation symmetry of $\eta$  as in \eqref{def: cc: permutation equivariance 2 stage}. Similarly we can show that $\eta(\AA_j\ppi \mbu,\mw)=\eta(\AA_i \mbu,\mw)$ for all $\mbu\in\RR^{k_1-1}$ and $\mw\in\RR^{k_2-1}$.
Now suppose $r\neq i\neq j$. Consider the case $i<r<j$. Then for all $\mbu\in\RR^{k_1-1}$ and $\mw\in\RR^{k_2-1}$,
\begin{align*}
\MoveEqLeft  \eta(\AA_r \ppi  \mbu,\mw)=\eta(\AA_r P_{ij} \mbu,\mw)  \\
=&\ \eta(\ppi \mbu_1-\ppi \mbu_r,\ldots, -\ppi \mbu_r,\ldots,\ppi \mbu_i-\ppi \mbu_r,\ldots,\ppi \mbu_j-\ppi \mbu_r,\ldots,\ppi \mbu_{k_1-1}-\ppi \mbu_r,\mw)\\
=&\ \eta(\mbu_1- \mbu_r,\ldots,  \mbu_j- \mbu_r,\ldots,-\mbu_r,\ldots, \mbu_i- \mbu_r,\ldots, \mbu_{k_1-1}-\mbu_r,\mw)\\
\stackrel{(a)}{=}&\  \eta(\mbu_1- \mbu_r,\ldots,  \mbu_i- \mbu_r,\ldots,-\mbu_r,\ldots, \mbu_j- \mbu_r,\ldots, \mbu_{k_1-1}-\mbu_r,\mw)\\
=&\ \eta(\AA_r \mbu,\mw),
\end{align*}
where (a) follows by the blockwise permutation symmetry of $\eta$. For all other orderings of the triplet $(i,r,j)$, we can similarly prove that $\eta(\AA_r \ppi  \mbu,\mw)=\eta(\AA_r \mbu,\mw)$.
Therefore \eqref{intheorem: CC: def: of V phi for PERM 2} implies that 
\begin{align*}
 \MoveEqLeft   V^{\eta}(\ppi \mbu,\mv_1,\ldots,\mv_{k_1};\mp)
= \mp_{11}\eta(\ppi \mbu,\mv_1)\nn + \mp_{12} \sum\limits_{j'=1}^{k_2-1}\eta(\ppi \mbu,\BB_{j'} \mv_1)\nn\\
&\ +\mp_{21}\sum\limits_{i'=1}^{k_1-1}\eta(\AA_{i'} \ppi \mbu,\mv_{i'+1}) +\mp_{22}\sum\limits_{i'=1}^{k_1-1}\sum\limits_{j'=1}^{k_2-1}\eta(\AA_{i'} \ppi \mbu, \BB_{j'}\mv_{i'+1})\\
=&\  \mp_{11}\eta(\mbu,\mv_1)\nn + \mp_{12} \sum\limits_{j'=1}^{k_2-1}\eta( \mbu,\BB_{j'} \mv_1)\nn\\
&\ +\mp_{21}\sum\limits_{i'\in[k_1]:i'\neq i,j}\eta(\AA_{i'} \ppi \mbu,\mv_{i'+1}) +\mp_{22}\sum\limits_{i'\in[k_1]:i'\neq i,j}\sum\limits_{j'=1}^{k_2-1}\eta(\AA_{i'}\ppi \mbu, \BB_{j'}\mv_{i'+1})\\
&\ +\mp_{21}\sum\limits_{i'\in (i,j)}\eta(\AA_{i'} \ppi \mbu,\mv_{i'+1}) +\mp_{22}\sum\limits_{i'\in (i,j)}\sum\limits_{j'=1}^{k_2-1}\eta(\AA_{i'}\ppi \mbu, \BB_{j'}\mv_{i'+1}).
\end{align*}
By our previous calculations,
\begin{align*}
    \MoveEqLeft \mp_{21}\sum\limits_{i'\in[k_1]:i'\neq i,j}\eta(\AA_{i'} \ppi \mbu,\mv_{i'+1}) +\mp_{22}\sum\limits_{i'\in[k_1]:i'\neq i,j}\sum\limits_{j'=1}^{k_2-1}\eta(\AA_{i'}\ppi \mbu, \BB_{j'}\mv_{i'+1})\\
   =  &\ \mp_{21}\sum\limits_{i'\in[k_1]:i'\neq i,j}\eta(\AA_{i'} \mbu,\mv_{i'+1}) +\mp_{22}\sum\limits_{i'\in[k_1]:i'\neq i,j}\sum\limits_{j'=1}^{k_2-1}\eta(\AA_{i'} \mbu, \BB_{j'}\mv_{i'+1})\\
\end{align*}
and
\begin{align*}
\MoveEqLeft \mp_{21}\sum\limits_{i'\in (i,j)}\eta(\AA_{i'} \ppi \mbu,\mv_{i'+1}) +\mp_{22}\sum\limits_{i'\in (i,j)}\sum\limits_{j'=1}^{k_2-1}\eta(\AA_{i'}\ppi \mbu, \BB_{j'}\mv_{i'+1})\\
=&\ \mp_{21}\eta(\AA_{i} \ppi \mbu,\mv_{i+1}) +\mp_{22}\sum\limits_{j'=1}^{k_2-1}\eta(\AA_{i}\ppi \mbu, \BB_{j'}\mv_{i+1})\\
&\ +\mp_{21}\eta(\AA_{j} \ppi \mbu,\mv_{j+1}) +\mp_{22}\sum\limits_{j'=1}^{k_2-1}\eta(\AA_{j}\ppi \mbu, \BB_{j'}\mv_{j+1})\\
=&\  \mp_{21}\eta(\AA_{j}\mbu,\mv_{i+1}) +\mp_{22}\sum\limits_{j'=1}^{k_2-1}\eta(\AA_{j} \mbu, \BB_{j'}\mv_{i+1})\\
&\ +\mp_{21}\eta(\AA_{i} \mbu,\mv_{j+1}) +\mp_{22}\sum\limits_{j'=1}^{k_2-1}\eta(\AA_{i}\mbu, \BB_{j'}\mv_{j+1}),
\end{align*}
which, noting $j=\ppi(i)$ and $i=\ppi(j)$, equals
\[\mp_{21}\sum\limits_{i'\in (i,j)}\eta(\AA_{i'}  \mbu,\mv_{\ppi(i')+1}) +\mp_{22}\sum\limits_{i'\in (i,j)}\sum\limits_{j'=1}^{k_2-1}\eta(\AA_{i'} \mbu, \BB_{j'}\mv_{\ppi(i')+1}).\]
Combining all the above pieces, we have
\begin{align*}
\MoveEqLeft   V^{\eta}(\ppi \mbu,\mv_1,\ldots,\mv_{k_1};\mp)\\
=&\  \mp_{11}\eta(\mbu,\mv_1)\nn + \mp_{12} \sum\limits_{j'=1}^{k_2-1}\eta( \mbu,\BB_{j'} \mv_1)\nn\\
&\ +\mp_{21}\sum\limits_{i'\in[k_1]:i'\neq i,j}\eta(\AA_{i'} \mbu,\mv_{i'+1}) +\mp_{22}\sum\limits_{i'\in[k_1]:i'\neq i,j}\sum\limits_{j'=1}^{k_2-1}\eta(\AA_{i'} \mbu, \BB_{j'}\mv_{i'+1})\nn\\
&\ +\mp_{21}\sum\limits_{i'\in (i,j)}\eta(\AA_{i'}  \mbu,\mv_{\ppi(i')+1}) +\mp_{22}\sum\limits_{i'\in (i,j)}\sum\limits_{j'=1}^{k_2-1}\eta(\AA_{i'} \mbu, \BB_{j'}\mv_{\ppi(i')+1}),
\end{align*}
which can be rewritten as
\begin{align*}
    \MoveEqLeft  \mp_{11}\eta(\mbu,\mv_1)\nn + \mp_{12} \sum\limits_{j'=1}^{k_2-1}\eta( \mbu,\BB_{j'} \mv_1)\nn\\
&\ +\mp_{21}\sum\limits_{i'=1}^{k_1-1}\eta(\AA_{i'} \mbu,\mv_{\ppi(i')+1}) +\mp_{22}\sum\limits_{i'=1,}^{k_1-1}\sum\limits_{j'=1}^{k_2-1}\eta(\AA_{i'} \mbu, \BB_{j'}\mv_{\ppi(i')+1})
\end{align*}
because $\ppi(i')=i'$ for $i'\neq (i,j)$. However, the above implies
\begin{equation}
    \label{inlemma: CC: intermediate: sol: intermediate}
    V^{\eta}(\ppi \mbu,\mv_1,\ldots,\mv_{k_1};\mp)=V^{\eta}( \mbu,\mv_1,\mv_{\ppi(1)+1},\ldots, \mv_{\ppi(k_1-1)+1};\mp)
\end{equation}
for all $\mbu\in\RR^{k_1-1}$ and $\mv_1,\ldots,\mv_{k_1}\in\RR^{k_2-1}$. 
Now observe that since $\ppi=\ppi^{-1}$, we  have
\begin{align*}
 V^{\eta}( \ppi \mbu,\mv_1,\mv_{\ppi(1)+1},\ldots, \mv_{\ppi(k_1-1)+1};\mp)=&\ V^{\eta}( \ppi \mbu,\mv_1,\mv_{\ppi^{-1}(1)+1},\ldots, \mv_{\ppi^{-1}(k_1-1)+1};\mp).
\end{align*}
However, \eqref{inlemma: CC: intermediate: sol: intermediate} applied on $(\mbu, \mv_1,\mv_{\ppi^{-1}(1)+1},\ldots, \mv_{\ppi^{-1}(k_1-1)+1}) $ leads to  
\[ V^{\eta}( \ppi \mbu,\mv_1,\mv_{\ppi^{-1}(1)+1},\ldots, \mv_{\ppi^{-1}(k_1-1)+1};\mp)=V^{\eta}(\mbu,\mv_1,\ldots,\mv_{k_1};\mp),\]
which proves \eqref{intheorem: cc: pi pi inv for u}.
  Now we will prove that $\mbu^*\in C_{ij}$.

Suppose  $\mbu^*_i\neq \mbu^*_j$. Since $\eta$ is bounded, there exists $c\in\RR$ so that\\
$V^{\eta}( \mbu^* ,\mv^*_1,\ldots,\mv^*_{k_1};\mp)=c$.  Then \eqref{intheorem: cc: pi pi inv for u} implies that 
\[V^{\eta}( \ppi \mbu^* ,\mv^*_1,\mv^*_{\ppi(1)+1},\ldots, \mv^*_{\ppi(k_1-1)+1};\mp)=c\]
as well. Since $V^\eta$ is a sum of strictly concave functions, it is strictly concave. Therefore, Jensen's inequality yields that 
\begin{align*}
    c=\frac{V^{\eta}(\mbu^* ,\mv^*_1,\ldots,\mv^*_{k_1};\mp)+V^{\eta}( \ppi \mbu^* ,\mv^*_1,\mv^*_{\ppi(1)+1},\ldots, \mv^*_{\ppi(k_1-1)+1};\mp)}{2}\\
    \stackrel{(a)}{<} V^{\eta}\lb \frac{\mbu^* +\ppi \mbu^* }{2}, \mv^*_1, \frac{\mv^*_2+\mv^*_{\ppi(1)+1}}{2},\ldots, \frac{\mv^*_{k_1}+\mv^*_{\ppi(k_1-1)+1}}{2}\rb
\end{align*}
where the strict inequality holds because of strict concavity.
If $\mbu'=\frac{\mbu^* +\ppi \mbu^* }{2}$, then $\mbu'_i=\mbu'_j$.  Then the above implies that if $\mbu_i\neq \mbu_j$, then 
\begin{align}
    \label{intheorem: cc: Cij inequality}
    V^{\eta}(\mbu^* ,\mv^*_1,\ldots,\mv^*_{k_1};\mp)<\sup_{\mbu\in C_{ij},\mv_1,\ldots, \mv_{k_1}\in\RR^{k_2-1}}V^{\eta}(\mbu,\mv_1,\ldots,\mv_{k_1};\mp),
\end{align}
which contradicts with the fact that $(\mbu^* ,\mv_1^*,\ldots,\mv_{k_1}^*)$ is the unique maximizer of $V^\eta$. Hence, we must have $\mbu^*_i=\mbu^*_j$. Since $i$ and $j$ were arbitrary, we must have $\mbu^* =x\mo_{k_1-1}$ for some $x\in\RR$.


\paragraph{Showing $\mv_j^*$'s are equal for $j\geq 2$:}
Suppose $\ppi$ is a permutaion of $[2:k_1]$. Then $(\ppi(2)-1,\ldots,\ppi(k_1)-1)$ is a permutaion of $[k_1-1]$. Let us denote $\ppi(2)-1=\varsigma(1),\ldots, \ppi(k_1)-1=\varsigma(k_1-1)$. Then $\varsigma$ is a permutaion of $[k_1-1]$. Therefore, we have shown that for $i\in[k_1-1]$, we can write  $\ppi(i+1)=\varsigma(i)+1$ for a permutaion $\varsigma$ of $[k_1-1]$. Therefore, 
\begin{align*}
\MoveEqLeft V^{\eta}(  \mbu,\mv_1,\mv_{\ppi(2)},\ldots, \mv_{\ppi(k_1)};\mp) \\
=&\ V^{\eta}(  \mbu,\mv_1,\mv_{\ppi(1+1)},\ldots, \mv_{\ppi(1+k_1-1)};\mp)\\
=&\  V^{\eta}(  \mbu,\mv_1,\mv_{1+\varsigma(1)},\ldots, \mv_{1+\varsigma(k_1-1)};\mp).
\end{align*}
However, applying \eqref{intheorem: cc: pi pi inv for u} on $(\varsigma^{-1}\mbu,\mv_1,\ldots,\mv_{k_1})$ we obtain that 
\[ V^{\eta}(  \varsigma^{-1}(\mbu),\mv_1,\mv_{2},\ldots, \mv_{k_1};\mp)=V^{\eta}(  \mbu,\mv_1,\mv_{1+\varsigma(1)},\ldots, \mv_{1+\varsigma(k_1-1)};\mp).\]
Therefore, we have showed that
\[V^{\eta}(  \mbu,\mv_1,\mv_{\ppi(2)},\ldots, \mv_{\ppi(k_1)};\mp)=V^{\eta}(  \varsigma^{-1}(\mbu),\mv_1,\mv_{2},\ldots, \mv_{k_1};\mp).\]
In particular,
\begin{align*}
  V^{\eta}(  \mbu^* ,\mv^*_1,\mv^*_{\ppi(2)},\ldots, \mv^*_{\ppi(k_1)};\mp)=&\ V^{\eta}(  \varsigma^{-1}(\mbu^* ),\mv^*_1,\mv^*_{2},\ldots, \mv^*_{k_1};\mp)\\
  =&\ V^{\eta}(  \mbu^* ,\mv^*_1,\mv^*_{2},\ldots, \mv^*_{k_1};\mp)
\end{align*}
because we have proved in the last step that 
 $\mbu^* _1=\ldots=\mbu^* _{k_1-1}$. However, $(\mbu^* ,\mv^*_1,\mv^*_{2},\ldots, \mv^*_{k_1})$ is the unique maximizer of $V^\eta$, implying $(\mv^*_{\ppi(2)},\ldots, \mv^*_{\ppi(k_1)})=(\mv^*_{2},\ldots, \mv^*_{k_1})$.  Since $\ppi$ is any arbitrary permutation of $[2:k_1]$, this implies that $\mv_2^*=\ldots=\mv_{k_1}^*$.

\paragraph{Showing $v_j$'s have equal elements}

Suppose $\ppi$ is any arbitrary permutation of $[k_2-1]$. Then by \eqref{intheorem: CC: def: of V phi for PERM 2}, 
\begin{align}
\label{inlemma: CC: vj s equal Vphi}
 V^{\eta}(  \mbu,\ppi(\mv_1),\mv_2,\ldots, \mv_{k_1};\mp) 
= \mp_{11}\eta( \mbu,\ppi(\mv_1))\nn + \mp_{12} \sum\limits_{j'=1}^{k_2-1}\eta( \mbu,\BB_{j'} \ppi(\mv_1))\nn\\
 +\mp_{21}\sum\limits_{i'=1}^{k_1-1}\eta(\AA_{i'}  \mbu,\mv_{i'+1}) +\mp_{22}\sum\limits_{i'=1}^{k_1-1}\sum\limits_{j'=1}^{k_2-1}\eta(\AA_{i'}  \mbu, \BB_{j'}\mv_{i'+1}),
\end{align}
where the $\BB_{j}$'s and $\AA_{i}$'s are as in \eqref{inlemma: CC: def of Ai} and \eqref{inlemma: CC: def of Bj}. First note that \\
$\eta(\mbu,\ppi(\mv_1))=\eta(\mbu,\mv_1)$ by \eqref{def: cc: permutation equivariance 2 stage}. By definition of $\BB_j$, 
\begin{align*}
 \MoveEqLeft \eta( \mbu,\BB_{j'} \ppi(\mv_1)) =&\ \eta(\mbu, \mv_{1\ppi(1)}-\mv_{1\ppi(j)},\ldots, -\mv_{1\ppi(j)},\ldots, \mv_{1\ppi(k_2-1)}-\mv_{1\ppi(j)}).
\end{align*}
However, by \eqref{intheorem: cc: permutation symmetry of phi}, $ \eta( \mbu,\BB_{j'} \ppi(\mv_1))= \eta( \mbu,\ppi^{-1}(\BB_{j'} \ppi(\mv_1)))$. Let 
\begin{align*}
l=\BB_{j'} \ppi(\mv_1)=(\mv_{1\ppi(1)}-\mv_{1\ppi(j)},\ldots, -\mv_{1\ppi(j)},\ldots, \mv_{1\ppi(k_2-1)}-\mv_{1\ppi(j)}).
\end{align*}
We want to figure out what is $\ppi^{-1}(l)$. Suppose $i=j$. Then 
\[\ppi^{-1}(l_i)=l_{\ppi^{-1}(i)}=\begin{cases}
    \mv_{1i}-\mv_{1\ppi(j)} & \text{ if }i\neq \ppi(j)\\
   -\mv_{1\ppi(j)} & \text{ if }i= \ppi(j)
\end{cases}\]
Therefore,
\[\ppi^{-1}(\BB_{j'} \ppi(\mv_1))=\left(\begin{smallmatrix}\underbrace{\mv_{1i}-\mv_{1\ppi(j)},\ldots, \mv_{1,\ppi(j)-1}-\mv_{1\ppi(j)}}_{\ppi(j)-1\text{ elements}},-\mv_{1\ppi(j)},&\\
\underbrace{\mv_{1,\ppi(j)+1}-\mv_{1\ppi(j)},\ldots,\mv_{1k_1}-\mv_{1\ppi(j)}}_{k_2-1-\ppi(j)\text{ elements}}\end{smallmatrix}\right),\]
which is just $\BB_{\ppi(j)}\mv_1$. 
Thus we have shown that $\eta(\mbu, \BB_{j'}\ppi(\mv_1))=\eta(\mbu, \BB_{\ppi(j')}\mv_1)$, indicating that
\begin{align*}
 \sum\limits_{j'=1}^{k_2-1}\eta( \mbu,\BB_{j'} \ppi(\mv_1))=    \sum\limits_{j'=1}^{k_2-1}\eta(\mbu, \BB_{\ppi(j')}\mv_1)=\sum\limits_{j=1}^{k_2-1}\eta(\mbu, \BB_{j}\mv_1).
\end{align*}
Hence, \eqref{inlemma: CC: vj s equal Vphi} implies
\[ V^{\eta}(  \mbu,\ppi(\mv_1),\mv_2,\ldots, \mv_{k_1};\mp) = V^{\eta}(  \mbu,\mv_1,\mv_2,\ldots, \mv_{k_1};\mp) \]
for all $\mbu\in\RR^{k_1}$ and $\mv_1,\mv_2,\ldots,\mv_{k_1}\in\RR^{k_2}$. 
In particular,
\[ V^{\eta}(  \mbu^* ,\ppi(\mv^*_1),\mv^*_2,\ldots, \mv^*_{k_1};\mp) = V^{\eta}(  \mbu^* ,\mv^*_1,\mv^*_2,\ldots, \mv^*_{k_1};\mp) \]
The uniqueness of $\mv_1^*$ implies $\mv_1^*=\ppi(\mv_1^*)$. Since $\ppi$ is an arbitrary permutaion, it follows that $\mv_1^*=c\mo_{k_2-1}$ for some $c\in\RR$. The proof follows similarly for other $\mv_i^*$'s replacing $\mbu$ by $\AA_{i-1}\mbu$, and hence skipped.

\end{proof}

\subsubsection{Proof of Lemma \ref{lemma: CC: closed}}
\label{secpf: cc: closed}
\begin{proof}[Proof of Lemma \ref{lemma: CC: closed}]
We have previously shown that the minimizer of $-V^\eta(\cdot;\mp)$ is unique, and 
 by  Lemma \ref{lemma: cc: supremum}, it  is of the form $(u_1^*\mo_{k_1-1},\mv^*_{11}\mo_{k_2-1},\mv^*_{21}\mo_{k_2-1},\ldots,\mv^*_{21}\mo_{k_2-1})$ for some $\mbu^* ,\mv_{11}^*,\mv_{21}^*\in\RR$, possibly depending on $\mp$. Therefore, if $(x,y,z)\neq (\mbu^* _1,\mv_{11}^*,\mv_{21}^*)$, it holds that 
 \begin{align*}
   \MoveEqLeft  -V^\eta(x\mo_{k_1-1},y\mo_{k_2-1},z\mo_{k_2-1},\ldots,z\mo_{k_2-1})\\
     &\ >-V^\eta(\mbu_1^*\mo_{k_1-1},\mv_{11}^*\mo_{k_2-1},\mv_{21}^*\mo_{k_2-1},\ldots,\mv_{21}^*\mo_{k_2-1}),
 \end{align*}
 where the strict inequality follows from the uniqueness of the minimum. The above implies
 \[\Phi(x,y,z;\mp)>\Phi(\mbu^* _1,\mv_{11}^*,\mv_{21}^*;\mp)\text{ for all }x,y,z\in\RR\]
 because  \eqref{intheorem: CC: def: of V phi for PERM 2} implies that 
    \[\Phi(x,y,z;\mp)=-V^\eta(x\mo_{k_1-1},y\mo_{k_2-1},z\mo_{k_2-1},\ldots,z\mo_{k_2-1}).\]
   Therefore, 
 $\Phi(x,y,z;\mp)$ has a unique minimizer $(x^*(\mp),y^*(\mp),z^*(\mp))$ and $x^*(\mp)=\mbu_1^*(\mp)$, $y^*(\mp)=\mv_{11}^*(\mp)$, $z^*(\mp)=\mv_{21}^*(\mp)$, thus proving the first part of the lemma.
 
 
Now we will show that $-\vartheta(\cdot;i,j)'s$ and $\Phi(\cdot;\mp)$ are  proper, closed, and strictly convex  for any $i,j\in[2]$ and $\mp\in\RR^4_{>0}$. Note that $-\vartheta(\cdot;i,j)$'s and $\Phi$ are strictly convex  because $-\eta$ is strictly convex. To show a convex function is proper, we need to show that (1) it never attains the value $-\infty$ and (2) its domain is non-empty \citep[pp. 24 of][]{rockafellar}. 
Since $-\eta$ is bounded below, (1) follows trivially for $-\vartheta(\cdot;i,j)'s$ and $\Phi$. Since $\mz_{k_1+K_2-2}\in\dom(-\eta)$ by Lemma \ref{lemma: CC: 0 in int}, \eqref{intheorem: CC: def: vartheta} implies that $(0,0)\in\dom(-\vartheta(\cdot;i,j))$ for all $i,j\in[2]$. Similarly, \eqref{def: CC: Phi: binary} implies that $(0,0,0)\in\dom(\Phi(\cdot;\mp))$ for all $\mp\in\RR^4_{>0}$. Therefore, (2) follows  for $-\vartheta(\cdot;i,j)'s$ and $\Phi(\cdot;\mp)$. Therefore  $-\eta$, $-\vartheta$, and $\Phi(\cdot;\mp)$ are proper.  Using Facts \ref{fact: closedness} and \ref{fact: bivariate lsc} (see Section \ref{sec: additional facts}), one can easily verify that  $-\vartheta(\cdot;i,j)$'s and $\Phi(\cdot;\mp)$ are  closed.
We have shown that, for each $\mp\in\RR^4_{>0}$, the minimum of $\Phi(\cdot;\mp)$ is attained. Therefore, $\Phi(\cdot;\mp)$is 0-coercive by Fact \ref{fact: concave: coercivity and recession for concave}.



 \end{proof}
 \subsubsection{Proof of Fact \ref{fact: Phi star mp finite}}
\label{sec: proof of Fact Phi star mp finite}
\begin{proof}[Proof of Fact \ref{fact: Phi star mp finite}]
 Note that
 \begin{align}
   -\Phi^*(\mp)=&\ \sup_{\mw\in\RR^3} \{-\Phi(\mw;\mp)\}\nn\\
      &\ =\sup_{(x,y,z)\in\RR^3}\slb \mp_{11}\vartheta(x,y;1,1)+\mp_{12}\vartheta(x,y;1,2)+ \mp_{21}\vartheta(x,z;2,1)\nn\\
      &\ +\mp_{22}\vartheta(x,z;2,2)\srb\nn\\
      &\ =\varsigma_{Im(\vartheta)}(\mp)
   \end{align} 
    where $\varsigma$ is the support function defined in \eqref{def: support function} and 
    \begin{align*}
        Im(\vartheta)=&\ \{\mw\in\RR^4: \mw=\{\vartheta(x,y;1,1), \vartheta(x,y;1,2), \vartheta(x,z;2,1), \vartheta(x,z;2,2)\}\\
        &\ \text{ for some }(x,y,z)\in\RR^3\}
    \end{align*}
    is the image set of $\vartheta$. 
   Thus $-\Phi^*$ is the support function of $Im(\vartheta)$. Thus $-\Phi^*$ is a closed convex function. Hence, it is continuous on the relative interior of its domain \citep[cf. Remark 3.1.3, pp. 104,][]{hiriart}. However, because  $\vartheta$ is bounded above, $-\Phi^*(\mp)<\infty$ for $\mp\in\RR^4_{>0}$. Hence,
   \begin{align}
       \label{intheorem: CC: dom Psi star}
       \dom(-\Phi^*)\supset \RR^4_{>0}.
   \end{align}
 Since $\RR^4_{>0}$  is open, $\RR^4_{>0}\subset\iint(\dom(-\Phi^*))$.  Thus $\Phi^*$ is continuous on $\RR^4_{>0}$  everywhere. 
 
To show the finiteness of $\Phi^*$, note  that, if  $\Phi^*(\mp)=\inf_{\mw\in\RR^3}\Phi(\mw;\mp)=\infty$ for some $\mp\in\RR^4_{>0}$, then   $\Phi(\mw;\mp)=\infty$ for all $\mw\in\RR^3$, which would imply that $\Phi(\cdot;\mp)$ is improper. However, Lemma \ref{lemma: CC: closed} implies that $\Phi(\cdot;\mp)$ is proper. Thus $\Phi^*(\mp)<\infty$. Also, $-\Phi^*(\mp)<\infty$ for all $\mp\in\RR^4_{>0}$ because $-\Phi^*$ is a convex function with $\mp\in\dom(-\Phi^*)$ by \eqref{intheorem: CC: dom Psi star}. Therefore, $|\Phi^*(\mp)|<\infty$.
\end{proof}

 \subsubsection{Proof of Lemma \ref{lemma: CC: differentiability}}
\label{secpf: Lemma differentiability}
\begin{proof}[Proof of Lemma \ref{lemma: CC: differentiability}]

First we prove the results on $\vartheta$.
Lemma \ref{lemma: CC: 0 in int} implies that under the setup of Theorem \ref{theorem: CC}, $\mz_{k_1+k_2-2}\in\iint(\dom(-\eta))$.
If $x,y$ are in a small neighborhood $U_0$ of $0$, then $(xA\mo_{k_1-1},yB\mo_{k_2-1})$ is in a small neighborhood of $\mz_{k_1+k_2-2}$ for all matrices $A$ and $B$ of appropriate dimensions. Noting $\mz_{k_1+k_2-2}\in\iint(\dom(-\eta))$, we deduce that if $U_0$ is small enough, then $(xA_i\mo_{k_1},yB_j\mo_{k_2-1})\in \iint(\dom(\eta))$ for all $(x,y)\in U_0$ and $i\in\{0,\ldots,k_1-1\}$ and $j\in\{0,\ldots,k_2-1\}$ where the $\AA_i$ and $\BB_j$'s are as defined in \eqref{inlemma: CC: def of Ai} and \eqref{inlemma: CC: def of Bj}.  Thus, $U_0^2\subset \iint(\dom(\vartheta(\cdot,i,j)))$ for $i,j\in[2]$, which proves part 1 of the current lemma.  On the other hand, if $(xA_i\mo_{k_1-1},yB_j\mo_{k_2-1})\in\iint(\dom(-\eta))$, $\eta$ is thrice continuously differentiable at $(xA_i\mo_{k_1},yB_j\mo_{k_2-1})$. Therefore, $\vartheta(\cdot;i,j)$ is also thrice continuously differentiable on $U_0^2$ if $U_0$ is a sufficiently small neighborhood, which proves part 3 of the current lemma. 

Now we will prove the results on $\Phi$.
Since for any $\mp\in\RR^4_{>0}$, the function $\Phi(\cdot;\mp)$ is a non-negative linear combination of convex functions, namely the $-\vartheta(\cdot;i,j)$'s,  by \eqref{def: CC: Phi: binary}, it follows that $\dom(\Phi(\cdot;\mp))= \cap_{i,j\in[2]}\dom(\vartheta(\cdot;i,j))$ \citep[pp. 33][]{rockafellar}. Therefore, part 2 follows from part 1. Part 4 follows from part 3 due to \eqref{def: CC: Phi: binary}.
\end{proof}

\subsubsection{Proof of Lemma \ref{lemma: CC: continuity of Lambda}}
\label{secpf:  CC: continuity of Lambda}
\begin{proof}[Proof of Lemma \ref{lemma: CC: continuity of Lambda}]
      The bulk of the proof is devoted to showing that there exists $\delta^*>0$ so that if $\mq\in B(\mp,\delta)$ for some $\delta\leq \delta^*$, then $\Lambda^*(\mq)$ is uniformly bounded in the sense that it lies in a compact set depending only on $\mp$ and $\vartheta$. Denote by
       \[\S^{\mq}_r=\{v\in\RR^3: \Phi(v;\mq)\leq r\},\quad \text{ for all }\mq\in\RR^4_{>0},\]
which means $\S^{\mq}_r$ is the level-$r$ sublevel set of $\Phi(\cdot;\mq)$.  We will prove the boundedness of $\Lambda^*(\mq)$ by proving two main results. The first step is to show that
\[ \Lambda^*(\mq)\in \S^{\mq}_{\Phi^*(\mp)+C_1\delta}\]
for some constant $C_1>0$. The second step will show that $\S^{\mq}_r\subset \S^\mp_{C(\mp,r)}$ for some bounded function $C(\mp,r)$ for all $r\in\RR$ and $\mq\in B(\mp,\delta)$ as long as $\delta<\min(\mp)/2\wedge 1$. The rest of the boundedness proof will follow from the uniqueness of the minimizer of $\Phi(\cdot;\mp)$, which follows from Lemma \ref{lemma: CC: closed}. We will take $\delta^*=\min(\mp)/2\wedge 1$. Note that if $\mq\in B(\mp,\delta)$ with $\delta\leq \delta^*$, then $\mq\in\RR^4_{>0}$ automatically because $\min(\mq)> \min(\mp)-\delta\geq \min(\mp)/2>0$.

First of all, since $-\Phi^*$ is a convex function, it is locally Lipschitz on its domain \citep[Theorem 3.1.2, pp. 103,][]{hiriart}. However, \eqref{intheorem: CC: dom Psi star} implies that $\dom(-\Phi^*)\supset \RR^4_{>0}$. Therefore,  for all $\mq\in B(\mp,\delta)$, 
\[|\Phi^*(\mp)-\Phi^*(\mq)|\leq L_{\mp}\|\mp-\mq\|_2\]
  for some constant $L_{\mp}>0$. Therefore, for  $\mq\in B(\mp,\delta)$, 
\[\Phi(\Lambda^*(\mq);\mq)=\Phi^*(\mq)\leq \Phi^*(\mp)+L_{\mp}\|\mp-\mq\|_2.\]
Noting $\Phi^*(\mp)+L_p\|\mp-\mq\|_2\in\RR$ by Fact \ref{fact: Phi star mp finite}, we deduce that when  $\delta<\delta^*$ ,
\begin{align}
            \label{inlemma: CC: Lammbda mp: sol  inclusion pq}
            \Lambda^*(\mq)\subset \S^{\mq}_{\Phi^*(\mp)+L_{\mp}\|\mp-\mq\|_2}\text{ for all }\mq\in B(\mp,\delta),
            \end{align} 
            which implies
  \begin{align}
            \label{inlemma: CC: Lammbda mp: sol  inclusion}
            \Lambda^*(\mq)\subset \S^{\mq}_{\Phi^*(\mp)+L_{\mp}\delta}\text{ for all }\mq\in B(\mp,\delta).
            \end{align}

        Now we will prove that if $\delta<\delta^*$, then 
        \begin{align}
            \label{inlemma: CC: Lammbda mp: subset inclusion}
    \S^{\mq}_r\subset\S^\mp_{r+8\delta\frac{C\|\mp\|_1+2C+|r|}{\min(\mp)}}\quad\text{for all }\mq\in B(\mp,\delta).
        \end{align} 
        Suppose $v=(x,y,z)\in\S^{\mq}_r$ for some $r\in\RR$. Thus
        \[-\vartheta(x,y;1,1)\mq_{11}-\vartheta(x,y;1,2)\mq_{12}-\vartheta(x,y;2,1)\mq_{21}-\vartheta(x,y;2,2)\mq_{22}\leq r.\]
        Note that  the $\vartheta(\cdot;i,j)$'s are  bounded above by some constant $C>0$ for $i,j\in[2]$ by Lemma \ref{lemma: cc: connecting general to margin-based}. Thus $-\vartheta(x,y;i,j)>-C$ for all $i,j\in[2]$. Therefore, the above display implies 
        \[-\vartheta(x,y;1,1)\mq_{11}-C(\mq_{12}+\mq_{21}+\mq_{22})\leq r.\]
        Therefore, we have obtained that
        \[-\vartheta(x,y;1,1)\leq \frac{C(\mq_{11}+\mq_{21}+\mq_{22})+r}{\mq_{11}}<\frac{C\|\mq\|_1+|r|}{\min(\mq)},\]
        where we used the fact that $\min(\mq)>0$.
        Note that 
        \[\vartheta(x,y;1,1)\leq C \leq \frac{C\|\mq\|_1}{\min(\mq)}.\]
        Therefore,
        \[|\vartheta(x,y;1,1)| \leq \frac{C\|\mq\|_1+|r|}{\min(\mq)}.\]
        We will express the above bound in terms of $\mp$.
      By the triangle inequality, 
        \[\|\mq\|_1\leq \|\mq-\mp\|_1+\|\mp\|_1\leq 2\|\mp-\mq\|_2+\|\mp\|_1\]
    because $\|\mp-\mq\|_1\leq 2\|\mp-\mq\|_2$ by Cauchy Schwartz inequality. Also, 
    \[\min(\mq)\geq \min(\mp)-\|\mp-\mq\|_\infty\geq \min(\mp)-\|\mp-\mq\|_2\]
    because $\|\mp\|_\infty\leq \|\mp\|_2$ for any vector $\mp$. Since $\mq\in B(\mp,\delta)$ and $\delta<\delta^*$, we have  $\|\mp-\mq\|_2<\delta^*<\min(\mp)/2$. Hence $\min(\mp)-\|\mp-\mq\|_2\geq \min(\mp)/2$, and  the following  bound holds:
      \[|\vartheta(x,y;1,1)|\leq \frac{C\|\mq\|_1+|r|}{\min(\mq)}\leq 2\frac{C\|\mp\|_1+2C\|\mp-\mq\|_2+|r|}{\min(\mp)}.\]
      Similarly, we can show that 
       \[|\vartheta(x,y;1,2)|,|\vartheta(x,z;2,1)|,|\vartheta(x,z;2,2)|\leq  2\frac{C\|\mp\|_1+2C\|\mp-\mq\|_2+|r|}{\min(\mp)}.\]
       Therefore, if $v\in\S^{\mq}_r$, then 
       \begin{align}
       \label{inlemma: CC: cont: diff}
         \MoveEqLeft  |\Phi(v;\mp)-\Phi(v;\mq)|\nn\\
           \leq &\ \slb |\vartheta(x,y;1,1)|+|\vartheta(x,y;1,2)|+|\vartheta(x,z;2,1)|+|\vartheta(x,z;2,2)|\srb \|\mp-\mq\|_\infty\nn\\
           \leq &\ 8\|\mp-\mq\|_2\frac{C\|\mp\|_1+2C\|\mp-\mq\|_2+|r|}{\min(\mp)},
       \end{align}
       where we again used the fact that  $\|\mp\|_\infty\leq \|\mp\|_2$ for any vector $\mp$.
 Since $\Phi(v;\mq)\leq r$ and $\|\mp-\mq\|_2<\delta\leq \delta^*\leq 1$, it follows that for all $\mq\in B(\mp,\delta)$,
 \begin{align}
 \label{inlemma: CC: Lammbda mp: subset inclusion pq}
    \Phi(v;\mp)\leq &\  r+8\|\mp-\mq\|_2\frac{C\|\mp\|_1+2C\|\mp-\mq\|_2+|r|}{\min(\mp)}\nn\\
    \leq &\ r+8\|\mp-\mq\|_2\frac{C\|\mp\|_1+2C+|r|}{\min(\mp)},
 \end{align}
    which completes the proof of \eqref{inlemma: CC: Lammbda mp: subset inclusion}.

Now suppose $r=\Phi^*(\mp)+L_p\|\mp-\mq\|_2$.
Then 
\begin{align}
\label{inlemma: CC: Lammbda mp: combination slo and subset}
 \MoveEqLeft r+8\|\mp-\mq\|_2\frac{C\|\mp\|_1+2C+|r|}{\min(\mp)}\nn\\
 \leq &\  \Phi^*(\mp)+L_p\|\mp-\mq\|_2+8\|\mp-\mq\|_2\frac{C\|\mp\|_1+2C+|\Phi^*(\mp)+L_p\|\mp-\mq\|_2|}{\min(\mp)}\nn\\
  =&\ \Phi^*(\mp)+\|\mp-\mq\|_2\underbrace{\slb L_p+8\frac{C\|\mp\|_1+2C+|\Phi^*(\mp)|+L_p}{\min(\mp)}\srb}_{C_p}
\end{align}
where we used the fact that $\|\mp-\mq\|_2\leq \delta\leq\delta^*\leq 1$. Since $|\Phi^*(\mp)|<\infty$ by Fact \ref{fact: Phi star mp finite}, $C_p<\infty$.
Hence, if $\delta\leq \delta^*$, then \eqref{inlemma: CC: Lammbda mp: sol  inclusion pq}, \eqref{inlemma: CC: Lammbda mp: subset inclusion pq} and \eqref{inlemma: CC: Lammbda mp: combination slo and subset} imply that
 \begin{align}
           \label{inlemma: CC: cont: bdd}
           \Lambda^*(\mq)\in \S^\mp_{\Phi^*(\mp)+C_p\|\mp-\mq\|_2 }\subset  \S^\mp_{\Phi^*(\mp)+\delta C_p}.
       \end{align}
       for all $\mq\in B(\mp,\delta)$. 
       Since $\Phi(\cdot;\mp)$ is 0-coercive  for all $\mp\in\RR^4_{>0}$ by Lemma \ref{lemma: cc: connecting general to margin-based}, the sublevel sets of $\Phi(\cdot;\mp)$  are compact  by Proposition 3.2.4, pp. 107,  of \cite{hiriart}. Hence, $\S^\mp_{\Phi^*(\mp)+\delta C_p}$ is compact.

       Now it remains to prove that if $\mq_k\to \mp$, then $\Lambda^*(\mq)\to\Lambda^*(\mp)$. First of all, note that if $\mq_k\to \mp$, then given any small $\delta>0$, 
       $\mq_k\in B(\mp,\delta)$ for all sufficiently large $k$. Let us take $\delta<\delta^*$. Then the above calculations imply that the sequence $\Lambda^*(\mq_k)$ lies in a compact set for all large $k$ and hence is bounded. To show $\Lambda^*(\mq_k)\to\Lambda^*(\mp)$, it suffices to show that given any subsequence of $\Lambda^*(\mq_k)$, there exists a further subsequence that converges to $\Lambda^*(\mp)$. Since  any subsequence  of $\Lambda^*(\mq_k)$ is bounded, we can  always extract a converging subsequence from it. We will show that limit is $\Lambda^*(\mp)$. If possible, suppose that limit is $\mw\in\RR^3$ for some subsequence where $\mw$ may depend on the particular subsequence. The proof will be complete if we can show that $\mw=\Lambda^*(\mp)$ for any such subsequence. To streamline notation, we will denote this subsequence by $\Lambda^*(\mq_k)$ as well. $\Phi(\cdot;\mp)$ is closed by Lemma \ref{lemma: CC: closed}. By Proposition 1.2.2 of \cite{hiriart}, the sublevel sets of a closed function  are closed (possibly empty). Since  $\mq_k\in B(\mp,\delta^*)$ for all sufficiently large $k$,  using \eqref{inlemma: CC: cont: bdd}  we obtain that $\mq_k\in \S^\mp_{\Phi^*(\mp)+\delta C_p}$  for all large $k\in\NN$. Since $ \S^\mp_{\Phi^*(\mp)+\delta C_p}$ is closed, if $\Lambda^*(\mq_k)\to \mw$, then $\mw\in \S^\mp_{\Phi^*(\mp)+\delta C_p}$. Thus $\mw\in\dom(\Phi(\cdot;\mp))$. However, we do not know if $\mw\in\iint(\dom(\Phi(\cdot;\mp)))$. Therefore, we do not know if $\Phi(\cdot;\mp)$ is continuous at $\mw$. However, since $\Phi(\cdot;\mp)$ is closed, $\liminf_k \Phi(\Lambda^*(\mq_k),\mp)\geq \Phi(\mw;\mp)$ since $\Lambda^*(\mq_k)\to_k \mw$ \citep[for a lower-semi continuous function $f$, $\liminf_{y\to x}f(y)\geq f(x)$; cf. pp. 55 of][]{rockafellar}.
       Note that 
       \begin{align}
       \label{inlemma: CC: cont: final}
\MoveEqLeft \limsup_k  \lbs\Phi(\mw,\mp)-\Phi(\Lambda^*(\mq_k);\mq_k)\rbs\nn\\
\leq &\ \limsup_k\lbt\lbs\Phi(\mw;\mp)-\Phi(\Lambda^*(\mq_k),\mp)\rbs+\lbs \Phi(\Lambda^*(\mq_k),\mp)-\Phi(\Lambda^*(\mq_k);\mq_k)\rbs\rbt
\end{align}
\begin{align*}
 \limsup_k\lbs\Phi(\mw;\mp)-\Phi(\Lambda^*(\mq_k),\mp)\rbs =\Phi(\mw;\mp)-\liminf_k \Phi(\Lambda^*(\mq_k),\mp)\leq 0
\end{align*}
because we have already argued that $\liminf_k \Phi(\Lambda^*(\mq_k),\mp)\geq \Phi(\mw;\mp)$. 

\eqref{inlemma: CC: Lammbda mp: sol  inclusion} implies that
$  \Lambda^*(\mq)\in \S^{\mq}_{r}$ with $r=\Phi^*(\mp)+L_{\mp}\delta^*$ as long as $\mq\in B(\mp,\delta^*)$.  Then by \eqref{inlemma: CC: cont: diff},
\begin{align*}
\Phi(\Lambda^*(\mq_k),\mp)-\Phi(\Lambda^*(\mq_k);\mq_k) 
 \leq 8\|\mp-\mq_k\|_2\frac{C\|\mp\|_1+4C\|\mp-\mq_k\|_2+|r|}{\min(\mp)}.
\end{align*}
Since $|\Phi^*(\mp)|<\infty$, it follows that $|r|<\infty$. Therefore, using the fact that $\mq_k\to \mp$, we deduce that 
\[\limsup_k\lbs \Phi(\Lambda^*(\mq_k),\mp)-\Phi(\Lambda^*(\mq_k);\mq_k) \rbs =0.\]
Hence,
\begin{align*}
    \limsup_k\lbs\Phi(\mw;\mp)-\Phi(\Lambda^*(\mq_k),\mp)\rbs+ \limsup_k\lbs \Phi(\Lambda^*(\mq_k),\mp)-\Phi(\Lambda^*(\mq_k);\mq_k)\rbs
\end{align*}
is well-defined and is non-positive. Therefore,
\begin{align*}
 \MoveEqLeft\limsup_k\lbt\lbs\Phi(\mw;\mp)-\Phi(\Lambda^*(\mq_k),\mp)\rbs+\lbs \Phi(\Lambda^*(\mq_k),\mp)-\Phi(\Lambda^*(\mq_k);\mq_k)\rbs\rbt  \\
 \leq &\   \limsup_k\lbs\Phi(\mw;\mp)-\Phi(\Lambda^*(\mq_k),\mp)\rbs+ \limsup_k\lbs \Phi(\Lambda^*(\mq_k),\mp)-\Phi(\Lambda^*(\mq_k);\mq_k)\rbs
\end{align*}
is non-positive.
Therefore, \eqref{inlemma: CC: cont: final} implies that
\[0\geq \Phi(\mw;\mp)-\liminf_k\Phi(\Lambda^*(\mq_k);\mq_k)=\Phi(\mw;\mp)-\liminf_k\Phi^*(\mq_k)=\Phi(\mw;\mp)-\Phi^*(\mp)\]
  where the last step follows from \eqref{intheorem: CC: dom Psi star} because $-\Phi^*$ is  continuous on $\RR^4_{>0}$ by Lemma \ref{fact: Phi star mp finite}. Therefore, we have obtained that $\Phi(\mw;\mp)\leq \Phi^*(\mp)$. Lemma \ref{lemma: CC: closed} implies that $\Phi(\cdot;\mp)$ has unique minimum for each $\mp\in\RR^4_{>0}$. Therefore, $\mw=\Lambda^*(\mp)$, which completes the proof.

   \end{proof}
     \subsubsection{Proof of Lemma \ref{lemma: CC: Lipschitz}}
     \label{secpf: proof of lemma: CC: Lipschitz }
   \begin{proof}[Proof of Lemma \ref{lemma: CC: Lipschitz}  ]
   We will prove the result for $y^*(\mp)$ only because the proof for $z^*(\mp)$ follows similarly. 
 Before we prove this lemma, we will collect some results that will be required later in the proof.  
 Since Lemma \ref{lemma: CC: value of Lambda p knot } implies $\Lambda^*(\mp^0)=(0,0,0)$, $\Phi(\cdot;\mp^0)$ is minimized at $(0,0,0)$. 
Lemma \ref{lemma: CC: differentiability} implies that  
 $\Phi(\cdot;\mp)$ is differentiable at $(0,0,0)$ for all $\mp\in\RR^3_{>0}$, and in particular, at $\mp^0$, which implies $\grad \Phi(0,0,0;\mp^0)=0$,  yielding
\begin{align}
\label{intheorem: cc: partial derivatives sum zero}
 \pdv{\vartheta(0,0;1,1)}{x}+\pdv{\vartheta(0,0;1,2)}{x}+\pdv{\vartheta(0,0;2,1)}{x}+\pdv{\vartheta(0,0;2,2)}{x}=&\ 0,\nn\\
    \pdv{\vartheta(0,0;1,1)}{y}+\pdv{\vartheta(0,0;1,2)}{y} =&\ 0,\nn\\
    \pdv{\vartheta(0,0;2,1)}{y}+\pdv{\vartheta(0,0;2,2)}{y}    =&\ 0.
\end{align}
Let us also define
\begin{gather}
    \label{def: cc: M12 and N12}
    M_{12}(x,y)=\pdv[2]{\vartheta(x,y;1,1)}{x}{y}+\pdv[2]{\vartheta(x,y;1,2)}{x}{y}\nn\\
    N_{12}(x,y)=\pdv[2]{\vartheta(x,y;2,1)}{x}{y}+\pdv[2]{\vartheta(x,y;2,2)}{x}{y},
\end{gather}
where $M_{12}(x,y)$ and $N_{12}(x,y)$ exist if $x,y\in U_0$, and in particular, at $(x,y)=(0,0)$. The following lemma, proved in Section \ref{secpf: lemma on M12}, provides the values of $M_{12}(0,0)$ and $N_{12}(0,0)$.
\begin{lemma}
    \label{lemma: CC: M12=0}
    $M_{12}(0,0)=N_{12}(0,0)=0$ when $M_{12}$ and $N_{12}$ are continuous functions. 
\end{lemma}
The following lemma proves the H\"older continuity of $\Lambda^*$ at all $\mp$ in a neighborhood of $\mp^0$.  Lemma  \ref{lemma: CC: Lipschitz continuity of Lambda} is proved in Section \ref{secpf: proof of lemma:  CC: Lipschitz continuity of Lambda }
\begin{lemma}
\label{lemma: CC: Lipschitz continuity of Lambda}
Consider the setup in Theorem \ref{theorem: CC}.
Suppose $\mp\in B(\mp^0,\delta^0)$ where $\delta^0$ is as in Lemma \ref{lemma: CC: hessian cont.}. Then there exists $\delta_p>0$, possibly depending on $\mp$, such that if $\mq\in B(\mp,\delta_p)$, then $\|\Lambda^*(\mp)-\Lambda^*(\mq)\|_2\leq C_p\sqrt{\|\mp-\mq\|_2}$ where $C_p$ depends only on $\vartheta$ and $\mp$.
\end{lemma}
We will choose $\delta>0$ in a way so that a few conditions are met. In particular, we let $\delta=\min(\delta_1,\delta_2,\delta_3,\delta_4,\delta_5)$ where $\delta_1$, $\delta_2$, $\delta_3$,  $\delta_4$, and $\delta_5$ are as explained below. 
\begin{enumerate}
    \item   Let $\delta_1>0$ be such that $\Lambda^*(\mp)\in U^3_0$ for all $\mp\in B(\mp^0,\delta_1)$. Such a $\delta_1$ exists by Lemma \ref{lemma: CC: hessian cont.} .
    \item Let $\delta_2>0$ be such that for all $\mp\in B(\mp^0,\delta_2)$, 
 $\|\Lambda^*(\mp)-\Lambda^*(\mp^0)\|_\infty=\|\Lambda^*(\mp)\|_\infty<2$. Such a $\delta_2$ exists by Lemma \ref{lemma: CC: continuity of Lambda}.
 \item Let us denote 
\[C=\pdv[2]{\vartheta(0,0;1,1)}{y}+\pdv[2]{\vartheta(0,0;1,2)}{y}.\]
Note that $C$ is a negative number because $\vartheta$ is strictly concave. If $\mp$ is in a small neighborhood of $\mp^0=(1,1,1,1)$, $|\mp_{11}-1|$ and $|\mp_{12}-1|$ are small. If $\delta$ is sufficiently small, then $(x^*(\mp),y^*(\mp))$ will be sufficiently close to $(0,0)$ for all $\mp\in B(\mp^0,\delta)$ by Lemma \ref{lemma: CC: continuity of Lambda}.  Also, $\pdv[2]{\vartheta(x,y;1,1)}{y}$ and $\pdv[2]{\vartheta(x,y;1,2)}{y}$ are close to $\pdv[2]{\vartheta(0,0;1,1)}{y}$ and $\pdv[2]{\vartheta(0,0;1,2)}{y}$ if $(x,y)$ is sufficiently  close to $(0,0)$ because the second-order partial derivatives of  $\vartheta(\cdot;i,j)$'s are continuous at $(0,0)$ by Lemma \ref{lemma: CC: differentiability}. 
Therefore, there exists $\delta_3>0$ so that if $\delta<\delta_3$,  then for all $\mp\in B(\mp^0,\delta)$, 
\begin{align}
    \label{intheorem: cc: bound derivatives c 2c}
 \lbs \mp_{11}\pdv[2]{\vartheta(x^*(\mp),\xi;1,1)}{y}+\mp_{12}\pdv[2]{\vartheta(x^*(\mp),\xi';1,2)}{y}\rbs\in(2C,C/2)
\end{align}
for all $\xi$ and $\xi'$ between $0$ and $y^*(\mp)$.
\item Since the third order partial derivatives of $\vartheta(\cdot;i,j)$'s are continuous on $U_0^2$ by Lemma \ref{lemma: CC: differentiability} and  $\Lambda^*(\mp)$ is continuous and $\Lambda^*(\mp)\in U_0^3$ for all $\mp$ in a small neighborhood of $\mp^0$ by Lemma \ref{lemma: CC: continuity of Lambda}, we 
can  choose $\delta$ so small such that
 for all $\mp\in B(\mp^0,\delta)$, 
\[ \begin{split}
    \bl \mp_{11}\diffp[2,1]{\vartheta(\xi_1,0;2,2)}{x,y}+\mp_{12}\diffp[2,1]{\vartheta(\xi_2,0;1,2)}{x,y}\bl \\
    \leq  \mp_{11}\bl \diffp[2,1]{\vartheta(0,0;2,2)}{x,y}\bl+\mp_{12}\bl\diffp[2,1]{\vartheta(0,0;1,2)}{x,y}\bl+1
\end{split}\]
for all $\xi_1$ and $\xi_2$ between $0$ and $x^*(\mp)$. We will refer to the above $\delta$ by $\delta_4$.
We will also choose $\delta_4$ to satisfy $\delta_4<1$. Then   $\mp_{11},\mp_{12}<2$ because  $\mp$ satisfies $\|\mp-\mp^0\|_2<\delta$ . Therefore, the last display would imply 
\begin{align}
   \label{intheorem: CC: Lipschitz: bound on 3-derivative} 
 \MoveEqLeft \bl \mp_{11}\diffp[2,1]{\vartheta(\xi_1,0;2,2)}{x,y}+\mp_{12}\diffp[2,1]{\vartheta(\xi_2,0;1,2)}{x,y}\bl \nn\\
 \leq &\ 2\lb \bl \diffp[2,1]{\vartheta(0,0;2,2)}{x,y}\bl+\bl\diffp[2,1]{\vartheta(0,0;1,2)}{x,y}\bl+1\rb
\end{align}
for all $\xi_1$ and $\xi_2$ between $0$ and $x^*(\mp)$.
\item Let $\delta_5>0$ be such that for all $\mp\in B(\mp^0,\delta_5)$, $\|\Lambda^*(\mp)-\Lambda^*(\mp^0)\|_2<C\sqrt{\|\mp-\mp^0\|_2}$ for some $C>0$ depending only on $\vartheta$. Existence of such a $\delta_5$ is guaranteed by Lemma \ref{lemma: CC: Lipschitz continuity of Lambda}. Note that since $\Lambda^*(\mp^0)=\mz_3$ by Lemma \ref{lemma: CC: value of Lambda p knot }, the above implies that $|x^*(\mp)|=|x^*(\mp)-x^*(\mp^0)|<C\sqrt{\|\mp-\mp^0\|_2}$  for all $\mp\in B(\mp^0,\delta_5)$.
\end{enumerate}



 The $\vartheta(\cdot;i,j)$'s defined in \eqref{intheorem: CC: def: vartheta} are thrice continuously differentiable on  $U_0^2$ for all $i,j\in[2]$  and $\Phi(\cdot;\mp)$  is thrice continuously differentiable on  $U_0^3$ for all $\mp\in\RR^4_{>0}$. By our choice of $\delta$,  $\Lambda^*(\mp)\in U_0^3$ for all $\mp\in B(\mp^0,\delta)$. Therefore, the $\vartheta(\cdot;i,j)$'s and $\Phi(\cdot;\mp)$ will be thrice continuously differentiable at  $\Lambda^*(\mp)$ if $\mp\in B(\mp^0,\delta)$. Therefore, for all  $\mp\in B(\mp^0,\delta)$,  $\partial{\Phi(\Lambda^*(\mp);\mp)}/\partial y$ exists, and
\[\pdv{\Phi(\Lambda^*(\mp);\mp)}{y}=-\mp_{11}\pdv{\vartheta(x^*(\mp),y^*(\mp);1,1)}{y}-\mp_{12}\pdv{\vartheta(x^*(\mp),y^*(\mp);1,2)}{y}=0,\]
where the last step follows because $\Lambda^*(\mp)$ is the  minimizer of $\Phi(\cdot;\mp)$.
Since $U_0$ contains both $0$ and $y^*(\mp)$, the real-valued functions $y\mapsto \partial{\vartheta(x^*(\mp),y;1,1)}/\partial y$ and $y\mapsto \partial{\vartheta(x^*(\mp),y;1,1)}/\partial y$   are differentiable on an interval containing $0$ and $y^*(\mp)$. Therefore,  we can take  Taylor series expansion of these functions at $y^*(\mp)$  around $y=0$, which leads to
\begin{align}
\label{inlemma: cc: Lipshitcz: Taylor}
  \MoveEqLeft  \mp_{11}\pdv{\vartheta(x^*(\mp),0;1,1)}{y}+\mp_{12}\pdv{\vartheta(x^*(\mp),0;1,2)}{y}\nn\\
  &\  +y^*(\mp)\lb \mp_{11}\pdv[2]{\vartheta(x^*(\mp),\xi;1,1)}{y}+\mp_{12}\pdv[2]{\vartheta(x^*(\mp),\xi';1,2)}{y}\rb=0
\end{align}
where $\xi$ and $\xi'$  are between $0$ and $y^*(\mp)$. 

 The case $y^*(\mp)=0$ will not be considered because the proof follows trivially when $y^*(\mp)=0$. 
We will prove that for any $y^*(\mp)\neq 0$, 
\begin{align}
    \label{intheorem: bound on y start by x step 1}
    |y^*(\mp)|\leq  \frac{2}{|C|}\abs{\mp_{11}\pdv{\vartheta(x^*(\mp),0;1,1)}{y}+\mp_{12}\pdv{\vartheta(x^*(\mp),0;1,2)}{y}}.
\end{align}
First, we consider the case when $y^*(\mp)$ is positive. 
\paragraph{$y^*(\mp)$ is positive}
In this case, \eqref{inlemma: cc: Lipshitcz: Taylor} leads to 
\begin{align*}
 \MoveEqLeft \mp_{11}\pdv{\vartheta(x^*(\mp),0;1,1)}{y}+\mp_{12}\pdv{\vartheta(x^*,0;1,2)}{y}\\
  =&\ -y^*(\mp)\lb \mp_{11}\pdv[2]{\vartheta(x^*(\mp),\xi;1,1)}{y}+\mp_{12}\pdv[2]{\vartheta(x^*(\mp),\xi';1,2)}{y}\rb\\
  \stackrel{(a)}{\geq} &\ -y^*(\mp) C/2= |y^*(\mp)||C|/2,
\end{align*}
where (a) follows from \eqref{intheorem: cc: bound derivatives c 2c} and the last step uses  $C<0$.
As mentioned previously, $|C|$ is non-zero, which implies \eqref{intheorem: bound on y start by x step 1}. 

\paragraph{$y^*(\mp)$ is negative}
In this case, using  \eqref{intheorem: cc: bound derivatives c 2c} again, we obtain that
\begin{align*}
 \MoveEqLeft \mp_{11}\pdv{\vartheta(x^*(\mp),0;1,1)}{y}+\mp_{12}\pdv{\vartheta(x^*(\mp),0;1,2)}{y}\\
  =&\ -y^*(\mp)\lb \mp_{11}\pdv[2]{\vartheta(x^*(\mp),\xi;1,1)}{y}+\mp_{12}\pdv[2]{\vartheta(x^*(\mp),\xi';1,2)}{y}\rb\\
  \leq &\ -\frac{C}{2}y^*(\mp) = -|y^*(\mp)|\frac{|C|}{2},
\end{align*}
where we use the fact $C<0$. 
Thus
\[|y^*(\mp)|\leq -\frac{2}{|C|}\slb \mp_{11}\pdv{\vartheta(x^*(\mp),0;1,1)}{y}+\mp_{12}\pdv{\vartheta(x^*,0;1,2)}{y}\srb,\]
which implies \eqref{intheorem: bound on y start by x step 1} because the last display implies that when $y^*(\mp)<0$,
\[\mp_{11}\pdv{\vartheta(x^*(\mp),0;1,1)}{y}+\mp_{12}\pdv{\vartheta(x^*,0;1,2)}{y}< 0.\]

To complete the proof, we need to upper bound the right hand side of \eqref{intheorem: bound on y start by x step 1}.
To this end, note that Lemma \ref{lemma: CC: differentiability} implies that the univariate functions  $x\mapsto \pdv{\vartheta(x,0;1,1)}{y}$ and $x\mapsto \pdv{\vartheta(x,0;1,2)}{y}$ are twice continuously differentiable on  $ U_0$. Our choice of $\delta$ ensures that $x^*(\mp)\in U_0$  for all $\mp\in B(\mp^0,\delta)$. Therefore,  the above univariate functions are twice continuously differentiable on an open interval containing $0$ and $x^*(\mp)$ for all $\mp\in B(\mp^0,\delta)$. A second order Taylor series expansion of  the function $x\mapsto \mp_{11}\pdv{\vartheta(x,0;1,1)}{y}+\mp_{12}\pdv{\vartheta(x,0;1,2)}{y}$ at $x^*(\mp)$ around $0$ yields 
\begin{align}
\label{intheorem: cc: second order Taylor series 1}
    \MoveEqLeft \mp_{11}\pdv{\vartheta(x^*(\mp),0;1,1)}{y}+\mp_{12}\pdv{\vartheta(x^*(\mp),0;1,2)}{y}\nn\\
    =&\ \mp_{11}\pdv{\vartheta(0,0;1,1)}{y}+\mp_{12}\pdv{\vartheta(0,0;1,2)}{y}  +x^*(\mp)\mp_{11}\pdv[2]{\vartheta(0,0;1,1)}{x}{y}\nn\\
    &\ +x^*(\mp)\mp_{12}\pdv[2]{\vartheta(0,0;1,2)}{x}{y} +\frac{x^*(\mp)^2}{2}\slb \mp_{11}\diffp[2,1]{\vartheta(\xi_1,0;2,2)}{x,y}+\mp_{12}\diffp[2,1]{\vartheta(\xi_2,0;1,2)}{x,y}\srb,
\end{align}
where $\xi_1$ and $\xi_2$ are between $0$ and $x^*(\mp)$. 
Note that by our choice of $\delta$ and \eqref{intheorem: CC: Lipschitz: bound on 3-derivative}, there exists a constant $C'$, depending only on $\vartheta$, so that 
\begin{equation}
    \label{intheorem: cc: second order Taylor series 2}
\bl\mp_{11}\diffp[2,1]{\vartheta(\xi_1,0;2,2)}{x,y}+\mp_{12}\diffp[2,1]{\vartheta(\xi_2,0;1,2)}{x,y}\bl<C'.
\end{equation}
Using Lemma \ref{lemma: CC: M12=0} and \eqref{intheorem: cc: partial derivatives sum zero}, we obtain that
\begin{align*}
  \MoveEqLeft \mp_{11}\pdv{\vartheta(0,0;1,1)}{y}+\mp_{12}\pdv{\vartheta(0,0;1,2)}{y}  +x^*(\mp)\lb\mp_{11}\pdv[2]{\vartheta(0,0;1,1)}{x}{y}+\mp_{12}\pdv[2]{\vartheta(0,0;1,2)}{x}{y}\rb\nn\\
   =&\ (\mp_{11}-\mp_{12})\underbrace{\slb \pdv{\vartheta(0,0;1,1)}{y}+x^*(\mp)\pdv[2]{\vartheta(0,0;1,1)}{x}{y}\srb}_{C_1(\mp)},
\end{align*}
which implies
\begin{align}
\label{intheorem: cc: second order Taylor series 3}
  \MoveEqLeft  \abs{\mp_{11}\pdv{\vartheta(0,0;1,1)}{y}+\mp_{12}\pdv{\vartheta(0,0;1,2)}{y}  +x^*(\mp)\lb\mp_{11}\pdv[2]{\vartheta(0,0;1,1)}{x}{y}+\mp_{12}\pdv[2]{\vartheta(0,0;1,2)}{x}{y}\rb}\nn\\
    \leq &\ |\mp_1-\mp_2||C_1(\mp)|.
\end{align}
Since we have chosen $\delta$ so that $\Lambda^*(\mp)<2$ for all $\mp\in B(\mp^0,\delta)$, it follows that  $|x^*(\mp)|<2$. Therefore, 
\begin{align}
\label{intheorem: CC: Lipschitz: constant bound}
    |C_1(\mp)|\leq \bl\pdv{\vartheta(0,0;1,1)}{y}\bl+\bl\pdv[2]{\vartheta(0,0;1,1)}{x}{y}\bl,
\end{align}
which depends only on $\vartheta$. 
On the other hand, since $\mp^0=(1,1,1,1)$, it holds that 
\[|\mp_1-\mp_2|=|\mp_1-\mp^0_1+\mp_2^0-\mp_2|\leq \|\mp-\mp^0\|_1\leq \sqrt{2}\|\mp-\mp^0\|_2,\]
where the last step follows by Cauchy-Schwartz inequality.  Hence, combining \eqref{intheorem: cc: second order Taylor series 1}, \eqref{intheorem: cc: second order Taylor series 2}, \eqref{intheorem: cc: second order Taylor series 3}, and \eqref{intheorem: CC: Lipschitz: constant bound}, we obtain that there exist constants $C'>0$ and $C_2>0$, depending only on $\vartheta$, so that
\begin{align}
\label{intheorem: Lipschitz: bound on y: first}
    \abs{\mp_{11}\pdv{\vartheta(x^*(\mp),0;1,1)}{y}+\mp_{12}\pdv{\vartheta(x^*(\mp),0;1,2)}{y}}\leq C'x^*(\mp)^2+C_2\|\mp-\mp^0\|_2.
\end{align}
Since $\delta<\delta_5$,  there exists $C>0$, depending only on $\vartheta$, so that 
\[x^*(\mp)^2\leq  C\|\mp-\mp^0\|_2.\]
 Hence, the proof follows from \eqref{intheorem: bound on y start by x step 1} and \eqref{intheorem: Lipschitz: bound on y: first}.

   \end{proof}
     \subsubsection{Proof of Lemma \ref{lemma: cc: connecting general to margin-based}}
\label{secpf: proof of  lemma: cc: connecting general to margin-based}
\begin{proof}[Proof of Lemma~\ref{lemma: cc: connecting general to margin-based}]
 

 
By Lemma \ref{lemma: CC: differentiability}, there exists an open neighborhood $U_0\ni 0$  such that $U_0^2\subset\iint(\dom(\Phi(\cdot;\mp))$ for all $\mp\in\RR^4_{>0}$. 
 Lemma \ref{lemma: CC: continuity of Lambda} implies that  we can choose a $\delta>0$ so small such that $\Lambda^*(\mp)\in U_0^3$ for all $\mp\in B(\mp^0,\delta)$. We will consider that $\mp\in B(\mp^0,\delta)$ for the rest of this proof.

\paragraph{Proving P1}
 Lemma \ref{lemma: CC: closed} implies that 
$\mbu^*(\mp)=x^*(\mp)\mo_{k_1-1}$, $\mv_1^*(\mp)=y^*(\mp)\mo_{k_2-1}$ and $\mv_i^*(\mp)=z^*(\mp)\mo_{k_2-1}$ for all $i\in[2:k_1]$, where these vectors are as defined in Lemma \ref{lemma: cc: supremum}.
Therefore \eqref{intheorem: CC: connection from u star to x star} implies that 
\[ \begin{split}
  &\  \mx^*(\mp)=(0,-x^*(\mp)\mo_{k_1-1}),\quad \my_1^*(\mp)=(0,-y^*(\mp)\mo_{k_2-1}),\\
  &\  \my_{2}^*(\mp)=\ldots, \my_{k_1}^*(\mp)=(0,-z^*(\mp)\mo_{k_2-1})
\end{split} \]
is a maximizer of $V^\psi$. Hence, a version of $\tilde d$ is given by $\tilde d_1=\pred(0,-x^*(\mp)\mo_{k_1-1})$ and
\begin{align}
    \label{def: tilde d 2 i in proving property}
    \tilde d_2(i)=\begin{cases}
    \pred(0,-y^*(\mp)\mo_{k_1-1}) & \text{ if } i=1\\
    \pred(0,-z^*(\mp)\mo_{k_2-1}) & \text{ if } i\in[2:k_1].
\end{cases}
\end{align}
Note that $\tilde d_1$ (or $\tilde d_2$) can take only two values, one and $k_1$ (or $k_2$). 
We remind the reader that although we consider $\mp\in\RR^4_{>0}$ for $\PP\in\mP^{k_1k_2}_b$,  this $\mp$ is the abbreviation of $(\mp_{ij})_{i\in[k_1],j\in[k_2]}$.  However, for  $\PP\in\mP^{k_1k_2}_b$, 
$\mp_{1j}=\mp_{12}$ if $j\geq 2$, $\mp_{i1}=\mp_{21}$ for $i\geq 2$ and $\mp_{i2}=\mp_{22}$ for $i\geq 2$. Therefore, to avoid redundancy, we only consider the four elements of $\mp$ although $\mp_{ij}$ exists for all $i\in[k_1]$ and $j\in[k_2]$. In light of the above, we can keep using the expression $V(\tilde d)=\mp_{\tilde d_1,\tilde d_2(\tilde d_1)}$ for $\PP\in\mP^{k_1,k_2}_b$ as well. Since $V(d)=\mp_{d_1,d_2(d_1)}$ for any DTR $d$, $V(\tilde d)=\mp_{\tilde d_1,\tilde d_2(\tilde d_1)}$.

{\bf Subcase $\max(\mp_{11},\mp_{12})>\max(\mp_{21},\mp_{22})$:} If $\psi$ is Fisher consistent, then Definition \ref{def: Fisher consistency multiclass} implies that $V(\tilde d)=V_*=\max(\mp_{11},\mp_{12})$. Therefore, $\mp_{\tilde d_1,\tilde d_2(\tilde d_1)}=\max(\mp_{11},\mp_{12})$. Hence, $\tilde d_1\neq 2,3,\ldots,k_1$ because in these cases, $\mp_{\tilde d_1,\tilde d_2(\tilde d_1)}\leq \max(\mp_{21},\mp_{22})<V_*$. Thus $\tilde d_1=1$, which implies $\pred(0,-x^*(\mp)\mo_{k_1-1})=1$, which implies $\max(\argmax(0,-x^*(\mp)))=1$. Therefore, $0>-x^*(\mp)$ or $x^*(\mp)>0$. 

 {\bf Subcase $\max(\mp_{21},\mp_{22})>\max(\mp_{11},\mp_{12})$:} In this case,  proceeding in the same way as the previous case, we can similarly prove that 
$\pred(0,-x^*(\mp)\mo_{k_1-1})=k_1$, which implies $\max(\argmax(0,-x^*(\mp)\mo_{k_1-1}))=k_1$. However, this only implies that $x^*(\mp)\leq 0$.  Therefore, we still need to show that $x^*(\mp)\neq 0$. If possible, suppose $x^*(\mp)=0$ for some $\mp\in B(\mp^0,\delta)$. Then for this $\mp$, $\mx^*(\mp)=\mz_{k_1}$. Consider the sequence $\mzz_n=(1/n,\mz_{k_1-1})$. 

Our next step is to show that  $V^\psi(\mzz_n,\my_1^*(\mp),\ldots,\my_{k_1}^*(\mp);\mp)\to_n V^\psi_*$. To this end, it suffices to show that $V^\eta(\Delta\mzz_n,\mv_1^*(\mp),\ldots,\mv_{k_1}^*(\mp);\mp)\to_n V^\eta_*$. Note that $\Delta\mzz_n=1_{k_1-1}/n$. Also, Lemma \ref{lemma: CC: closed} implies that $\mv_1^*(\mp)=y^*(\mp)\mo_{k_2-1}$ and $\mv_2^*(\mp)=\ldots=\mv_{k_1}^*(\mp)=z^*(\mp)\mo_{k_2-1}$. 
Therefore, by \eqref{intheorem: connection between V eta and Phi}, 
$V^\eta(\Delta\mzz_n,\mv_1^*(\mp),\ldots,\mv_{k_1}^*(\mp);\mp)=-\Phi(1/n, y^*(\mp),z^*(\mp);\mp) $ and $V^\eta_*=-\Phi^*(\mp)$. Therefore, it suffices to show that $\Phi(1/n, y^*(\mp),z^*(\mp);\mp)\to_n \Phi^*(\mp)$. Since  $\Lambda^*(\mp)=(0,y^*(\mp),z^*(\mp))\in U_0^3$, it follows that $(0,y^*(\mp),z^*(\mp))\in \iint(\dom(\Phi(\cdot;\mp)))$. Since $\Phi(\cdot;\mp)$ is convex by Lemma \ref{lemma: CC: closed}, the above implies that $\Phi(\cdot;\mp)$ is continuous at $(0,y^*(\mp),z^*(\mp))$ because convex functions are continuous on the interior of its domain  \citep[cf. pp. 104 of][]{hiriart}.
Therefore, 
\[\lim_n \Phi(1/n,y^*(\mp),z^*(\mp);\mp)= \Phi(0,y^*(\mp),z^*(\mp);\mp)=\Phi^*(\mp),\]
where we used the fact that $\Lambda^*(\mp)=(0,y^*(\mp),z^*(\mp))$. 
Therefore, it follows that $V^\psi(\mzz_n,\my_1^*(\mp),\ldots,\my_{k_1}^*(\mp);\mp)\to_n V^\psi_*$. If $\psi$ is Fisher consistent, then the above implies that $V(\tilde d_n)\to V_*=\max(\mp)=\max(\mp_{21},\mp_{22})$ where 
the DTR $\tilde d_n=(\tilde d_{1n}, \tilde d_{2n})$ corresponds to $(\mzz_n,\my_1^*(\mp),\ldots,\my_{k_1}^*(\mp))$. Note that $\tilde d_n$ satisfies
\begin{align}
    \label{inlemma: Prop proof: tilde d}
    \tilde d_{1n}=\pred(\Delta\mzz_n),\quad \tilde d_{2n}(i)=\begin{cases}
        \pred(0,-y^*(\mp)\mo_{k_2-1}) & \text{ if }i=1\\
        \pred(0,-z^*(\mp)\mo_{k_2-1}) & \text{ if }i\in[2:k_1].
    \end{cases}
\end{align}
Thus in this case, 
$\tilde d_{1n}=\pred(0,-\mo_{k_1-1}/n)=1$. Since $V(\tilde d)=V(d_1,d_2(d_1))$ for any DTR $d\equiv (d_1,d_2)$ in $\mP^{k_1,k_2}$,  $V(\tilde d_n)=\mp_{1, d_{2n}(1)}$. Therefore,
\[V(\tilde d_n)\leq \max(\mp_{11},\mp_{12})< \max(\mp_{21},\mp_{22})=V_*.\]
The above implies 
$\limsup_n V(\tilde d_n)<V_*$, which violates $V(\tilde d_n)\to V_*$. Therefore, we arrive at a contradiction if we assume $x^*(\mp)=0$. Therefore, if $\mp\in B(\mp^0,\delta)$, then $x^*(\mp)<0$ if $\max(\mp_{21},\mp_{22})>\max(\mp_{11},\mp_{12})$.




\paragraph{Proving P2}
\paragraph{Case 1: $\max(\mp_{11},\mp_{12})>\max(\mp_{21},\mp_{22})$}
 Suppose   $\mp_{11}>\mp_{12}$. Therefore, $V_*=\max(\mp)=\mp_{11}$ and $\mp_{ij}<V_*$ unless $(i,j)=(1,1)$.  If $\psi$ is Fisher consistent, then $V(\tilde d)=\mp_{\tilde d_1,\tilde d_2(\tilde d_1)}=\mp_{11}$.  Therefore, $d_2(\tilde d_1)=1$. We have already proved that $x^*(\mp)>0$ in this case, implying $\tilde d_1=\pred(0,-x^*(\mp)\mo_{k_1-1})=1$. Therefore, $\tilde d_2(1)=1$. Hence, \eqref{def: tilde d 2 i in proving property} implies that $\pred(0,-z^*(\mp)\mo_{k_2-1})=1$, which indicates that $\max(\argmax(0,-y^*(\mp)\mo_{k_2-1}))=1$. Therefore, $0>-y^*(\mp)$, which implies $y^*(\mp)>0$. Now if  $\mp_{11}<\mp_{12}$, then proceeding as before, we arrive at 
$\max(\argmax(0,-y^*(\mp)\mo_{k_2-1}))=k_2$, which only ensures that $y^*(\mp)\leq 0$. As before, let us assume $y^*(\mp)=0$ for some $\mp\in B(\mp^0,\delta)$. We will show that we arrive at a contradiction if $y^*(\mp)$. 

Consider the sequence $(x^*(\mp),1/n,z^*(\mp))$, which 
obviously converges to $\Lambda^*(\mp)$ as $n\to\infty$. Let us also denote $\mzz_n=(1/n,\mz_{k_2-1})$. 
Therefore, proceeding like the case of $x^*(\mp)=0$, we can show that
\[\Phi(x^*(\mp),1/n,z^*(\mp);\mp)\to_n \Phi^*(\mp).\]
As before, it follows that  $\Phi^*(\mp)=V^\psi_*$ and $V^\psi(\mx^*(\mp),\mzz_n, \my_2^*(\mp),\ldots,\my_{k_1}^*(\mp))=-\Phi(x^*(\mp),1/n,z^*(\mp);\mp)$. 
Therefore,  $V^\psi(\mx^*(\mp),\mzz_n, \my_2^*(\mp),\ldots,\my_{k_1}^*(\mp))\to_n V_*^\psi$. Suppose $\tilde d_n$ is the DTR associated with $(\mx^*,\my^*_1,\ldots,\my^*_{k_1})$. We can show that
\begin{align*}
    \tilde d_{1n}= \pred(0,-x^*(\mp)\mo_{k_1-1}),\quad \tilde d_{2n}(i)=\begin{cases}
        \pred(\Delta\mzz_n) & \text{ if }i=1\\
        \pred(0,-z^*(\mp)\mo_{k_2-1}) & \text{ if }i\in[2:k_1].
    \end{cases}
\end{align*}

Fisher consistency implies that $V(\tilde d_n)=\mp_{\tilde d_{1n},\tilde d_{2n}(\tilde d_{1n})}\to_n \mp_{12}$. However, we have shown that $\tilde d_{1n}=1$ in this case. Also,
\[\tilde d_{2n}(\tilde d_{1n})=\tilde d_{2n}(1)=\pred(0,-\mo_{k_2-1}/n)=1.\]
Therefore, we have $V(\tilde d_n)=\mp_{11}<V^\psi_*=\mp_{12}$ where the last equality follows because because  $\max(\mp)=\mp_{12}$.
Hence, we arrive at a contradiction, implying $y^*(\mp)$ can not be zero. 

\paragraph{Case 2: $\max(\mp_{11},\mp_{12})<\max(\mp_{21},\mp_{22})$}
We have shown that, in this case, $x^*(\mp)<0$. Therefore, $\pred(\mx^*(\mp))=\pred(0,-x^*(\mp)\mo_{k_1-1})=k_1$. Suppose   $\mp_{21}>\mp_{22}$. In this case, $V_*=\max(\mp)=\mp_{21}$. Using Fisher consistency as in Case 1, we can show that $\tilde d_2(\tilde d_1)=1$, implying $\tilde d_2(k_1)=1$. Therefore, $\pred(0,-z^*(\mp)\mo_{k_2-1})=1$, which implies $z^*(\mp)>0$. If   $\mp_{21}<\mp_{22}$, then we can show that $\pred(0,-z^*(\mp))=k_2$, which implies $z^*(\mp)\leq 0$. If possible, suppose $z^*(\mp)=0$ for some $\mp\in B(\mp^0,\delta)$. In this case, $\max(\mp)=\mp_{22}$.

Consider the sequence  $(x^*(\mp),y^*(\mp),1/n)$. Let us denote $\mzz_n=(1/n,\mz_{k_2-1})$. Proceeding as before, we can show that 
\[\Phi(x^*(\mp),y^*(\mp),1/n;\mp)\to_n \Phi^*(\mp).\]
From \eqref{intheorem: connection between V eta and Phi}, it follows that  $\Phi^*(\mp)=V^\psi_*$ and $V^\psi(\mx^*(\mp),\my_2^*(\mp),\mzz_n, ,\ldots,\mzz_n)=-\Phi(x^*(\mp),y^*(\mp),1/n,)$. 
Therefore,  $V^\psi(\mx^*(\mp),\my_2^*(\mp),\mzz_n, ,\ldots,\mzz_n)\to_n V_*^\psi$. If $\psi$ is Fisher consistent, we must also have $V(\tilde d_n)\to_n V_*$ where $\tilde d_{n}$ is the DTR corresponding to $(\mx^*(\mp),\my_2^*(\mp),\mzz_n, ,\ldots,\mzz_n)$. We can show that
\begin{align*}
    \tilde d_{1n}= \pred(0,-x^*(\mp)\mo_{k_1-1}),\quad \tilde d_{2n}(i)=\begin{cases}
     \pred(0,-y^*(\mp)\mo_{k_2-1})   & \text{ if }i=1\\
      \pred(\Delta\mzz_n)    & \text{ if }i\in[2:k_1].
    \end{cases}
\end{align*}
Since $x^*(\mp)<0$, the corresponding $\tilde d_{n}$ satisfies $\tilde d_{1n}=k_1$. Therefore, $\tilde d_{2n}(\tilde d_{1n})=\tilde d_{2n}(k_1)=\pred(0,-\mo_{k_2-1}/n)=1$. Hence, $\mp_{\tilde d_{1n},\tilde d_{2n}(\tilde d_{1n})}=\mp_{k_11}$. However, since $\PP\in\mP^{k_1k_2}_b$, $\mp_{k_11}=\mp_{21}$ by definition. Therefore, $V(\tilde d_n)=\mp_{21}$. Since  $\mp_{21}<\mp_{22}$, it follows that $\limsup_n V(\tilde d_n)<V_*$, which implies $V(\tilde d_n)$ does not converge to $V_*$. Thus we again encounter a contradiction, implying $z^*(\mp)$ can not be zero.

\end{proof}

\subsection{Proof of auxiliary  results}
\label{secpf: cc: proof of additional lemmas}


 \subsubsection{Proof of Lemma \ref{lemma: CC: M12=0}}
\label{secpf: lemma on M12}
\begin{proof}[Proof of Lemma \ref{lemma: CC: M12=0}]
    We prove the case for $M_{12}$ because the proof for $N_{12}$ follows in the similar way.  Lemma \ref{lemma: CC: differentiability} implies that the second-order partial derivatives of $\vartheta(\cdot;i,j)$'s are continuous on $U_0^2$ for all $i,j\in[2]$. Therefore, $M_{12}$ and $N_{12}$ are continuous on the open set $U_0^2$. 
Suppose $M_{12}(0,0)\neq 0$. If possible, suppose $M_{12}(0,0)>C$ for some $C>0$. Then there exists $\epsilon>0$ such that $|x|,|y|<\epsilon$ implies that
$M_{12}(x,y)>C/2$. By Lemma \ref{lemma: CC: continuity of Lambda}, can choose $\delta_1>0$ so small   such that $|x^*(\mp)-x^*(\mp^0)|<\epsilon$ for all $\mp\in B(\mp^0,\delta_1)$. Since $x^*(\mp)=0$ by Lemma \ref{lemma: CC: value of Lambda p knot }, it follows that $|x^*(\mp)|<\epsilon$ for all $\mp\in B(\mp^0,\delta_1)$.
By Lemma \ref{lemma: CC: maximizer of loss half} below (proved in Section \ref{secpf: auxiliary lemmas for step 2c}), we can choose $\mp\in B(\mp^0,\delta_1)$ so that $\mp_{11}=\mp_{12}>\mp_{21},\mp_{22}$, $\Lambda^*(\mp)\in U_0^3$, $y^*(\mp)=0$,  and $x^*(\mp)>0$.


    \begin{lemma}
  \label{lemma: CC: maximizer of loss half}
     There exists $\mp\in B(\mp^0,\delta_1)$ so that $\mp_{11}=\mp_{12}>\mp_{21},\mp_{22}$, $\Lambda^*(\mp)\in U_0^3$, $x^*(\mp)>0$, and $y^*(\mp)=0$. 
  \end{lemma}

Noting (a) $\Lambda^*(\mp)=(x^*(\mp),y^*(\mp),z^*(\mp))$ is the minimizer of $\Phi(\cdot;\mp)$ and (b) $\Phi(\cdot;\mp)$ is thrice differentiable on $U_0^3$, and hence at $\Lambda^*(\mp)$, we obtain that
\[0=\pdv{\Phi(x^*(\mp),y^*(\mp),z^*(\mp);\mp)}{y}.\]
 Then
by a first order Taylor series expansion (differentiating with respect to $x$), we have
\begin{align*}
    \pdv{\Phi(0,y^*(\mp),z^*(\mp);\mp)}{y}+x^*(\mp)\pdv[2]{\Phi(\xi,y^*(\mp),z^*(\mp);\mp)}{x}{y}=0
\end{align*}
where $\xi$ is a number between $0$ and $x^*$. Therefore,
\begin{align}
    \label{inlemma: CC: Taylor of M12 N12}
  &\  \mp_{11}\pdv{\vartheta(0,y^*(\mp);1,1)}{y}+\mp_{12}\pdv{\vartheta(0,y^*(\mp);1,2)}{y}\nn\\
    &\ +x^*(\mp)\lb\mp_{11}\pdv[2]{\vartheta(\xi,y^*(\mp);1,1)}{x}{y}+\mp_{12}\pdv[2]{\vartheta(\xi,y^*(\mp);1,2)}{x}{y}\rb=0.
\end{align}
Since $\mp_{11}=\mp_{12}$ and $y^*(\mp)=0$, we have 
\[\pdv{\vartheta(0,0;1,1)}{y}+\pdv{\vartheta(0,0;1,2)}{y}+x^*(\mp)\pdv[2]{\vartheta(\xi,0;1,1)}{x}{y}+x^*(\mp)\pdv[2]{\vartheta(\xi,0;1,2)}{x}{y}=0.\]
Using \eqref{intheorem: cc: partial derivatives sum zero} we obtain that
\[x^*(\mp)\pdv[2]{\vartheta(\xi,0;1,1)}{x}{y}+x^*(\mp)\pdv[2]{\vartheta(\xi,0;1,2)}{x}{y}=0.\]
Since $x^*(\mp)>0$,
\[M_{12}(\xi,0)=\pdv[2]{\vartheta(\xi,0;1,1)}{x}{y}+\pdv[2]{\vartheta(\xi,0;1,2)}{x}{y}=0.\]
Therefore, we have obtained that
there exists $\xi\in[0,x^*(\mp)]$ such that
$M_{12}(\xi,0)=0$.  However, since $|x^*(\mp)|<\epsilon$,  $\xi$ satisfies $|\xi|<\epsilon$. However, we have previously shown that $M_{12}(x,y)>C/2$ for all $x,y\in(-\epsilon,\epsilon)$. Therefore, we arrive at a contradiction.  Hence, our assumption $M_{12}(0,0)>0$ is incorrect Similarly, we can show that assuming $M_{12}(0,0)<0$ leads to a contradiction as well. Thus, $M_{12}(0,0)=0$.

\end{proof}

\subsubsection{Proof of Lemma \ref{lemma: CC: maximizer of loss half}}
\label{secpf: auxiliary lemmas for step 2c}

  \begin{proof}[Proof of Lemma \ref{lemma: CC: maximizer of loss half}]
  By Lemmas \ref{lemma: CC: differentiability} and \ref{lemma: cc: connecting general to margin-based},  we can choose $\delta\in (0,\delta_1)$ so small such that for all $\mp\in B(\mp^0,\delta)$, $\Lambda^*(\mp)\in U_0^3$ and $\Lambda^*(\mp)$ satisfies Properties P1 and P2. We consider a $\mp\in B(\mp^0,\delta/3)$ of the form $(c,c,c_1,c_2)$ where $c>c_1,c_2>0$. Such a $\mp$ exists in the ball $B(\mp^0,\delta/3)$  because $\mp^0=(1,1,1,1)$. 
      Suppose, if possible,  $y^*(\mp)>0$. We will show that this leads to a contradiction. If $y^*(\mp)>0$,  we can find an $\epsilon>0$ such that $y^*(\mp)>\epsilon$. Since $\Lambda^*$ is a continuous map by Lemma \ref{lemma: CC: continuity of Lambda}, we can find a $\delta'>0$ such that for all  $\mq\in B(\mp,\delta')$, $\|\Lambda^*(\mp)-\Lambda^*(\mq)\|_2\leq \epsilon/2$. We  choose $\delta'$ to be smaller than $\delta/3$ so that  $B(\mp,\delta')\subset B(\mp^0,2\delta/3)\subset B(\mp^0,\delta)$.  If $(x^*(\mq),y^*(\mq),z^*(\mq))=\Lambda^*(\mq)$, $|y^*(\mp)-y^*(\mq)|\leq \epsilon/2$. Hence, $y^*(\mq)>0$.  In particular, we consider $\mq=(c,c+\delta'/2,c_1,c_2)$. It is not hard to see that $\mq\in B(\mp,\delta')$. However, since $B(\mp,\delta')\subset  B(\mp^0,\delta)$, the corresponding   $\Lambda^*(\mq)$ satisfies Properties P1 and P2.  Since $\argmax(\mq)=(1,2)$, Property P2 implies $y^*(\mq)<0$, which leads to a contradiction. Thus $y^*(\mp)>0$ can not hold. In a similar way, we can show that $y^*(\mp)<0$ can not hold. Hence, we must have $y^*(\mp)=0$.

  \end{proof}
  \subsubsection{Proof of Lemma \ref{lemma: CC: Lipschitz continuity of Lambda}}
\label{secpf: proof of lemma:  CC: Lipschitz continuity of Lambda }
\begin{proof}[Proof of Lemma \ref{lemma: CC: Lipschitz continuity of Lambda}]
In this proof, we will use the fact that $(0,0)$ lies inside $\iint(\dom(\vartheta(\cdot;i,j)))$ for all $i,j\in[2]$, which follows from Lemma \ref{lemma: CC: differentiability}.
If $\mp\in B(\mp^0,\delta^0)$ where $\delta^0$ is as in Lemma \ref{lemma: CC: hessian cont.}, then $\Lambda^*(\mp)\in U_0^3$, where the latter is an open set containing the origin. Therefore, $\Phi(\cdot;\mp)$ is thrice continuously differentiable at $\Lambda^*(\mp)$. Therefore,  the hessian of $\Phi(\cdot;\mp)$ exists and all of its elements are continuous at $U_0^3\ni\Lambda^*(\mp)$.  Since $\Phi(\cdot;\mp)$ is a strictly convex function, the hessian of $\mw\mapsto\Phi(\mw;\mp)$ is  positive definite at $\Lambda^*(\mp)$.  The minimum eigenvalue of the hessian at $\Lambda^*(\mp)$  is thus bounded below by $2c_p$ for some $c_p>0$. Since the hessian of $\Phi(\cdot;\mp)$ is  continuous on $U_0^3$ and the minimum eigenvalue of a matrix is a continuous function of the matrix \citep[cf. Corollary 6.3.8, p.407 of][]{horn2012matrix}, the minimum eigenvalue of the hessian of $\mw\mapsto \Phi(\mw;\mp)$ is a continuous function of $\mw$ on $U_0^3$.  We can therefore choose a $\epsilon>0$ so small such that  the minimum eigenvalue of the hessian of $\Phi(\mw;\mp)$ is bounded away from $c_p$ for all $\mw\in B(\Lambda^*(\mp),\epsilon)\subset U_0^3$.

  Thus $\Phi(\cdot;\mp)$ is strongly convex on $B(\Lambda^*(\mp),\epsilon) $ with  constant $c_p>0$ that possibly depends only on $\vartheta$ and $\mp$.  Since $\Lambda^*(\mp)$ is continuous at $\mp$, there is $\delta'>0$ so that  $\Lambda^*(\mq)\in B(\Lambda^*(\mp),\epsilon)$ for all $\mq\in B(\mp,\delta')$. 
Let us take $\delta_p=\min(\delta',\delta^*)$ where $\delta^*$ is as in the proof of Lemma \ref{lemma: CC: continuity of Lambda}. Note that $\delta^*$ depends on $\mp$. Suppose $\delta<\delta_p$. Then for all $\mq\in B(\mp,\delta)$, 
\begin{align*}
    \Phi(\Lambda^*(\mq),\mp)\geq \Phi(\Lambda^*(\mp),\mp)+\langle s, \Lambda^*(\mq)-\Lambda^*(\mp)\rangle +\frac{c_p}{2} \|\Lambda^*(\mp)-\Lambda^*(\mq)\|_2^2,
\end{align*}
where $s\in\partial \Phi(\Lambda^*(\mp), \mp)$ by Theorem 6.1.2, pp. 200 of \cite{hiriart}.
 Since $\Lambda^*(\mp)$ is the minimizer of $\Phi(\cdot,\mp)$, $0\in\partial \Phi(\Lambda^*(\mp), \mp)$. Taking $s=0$, we obtain the following for all $\mq\in B(\mp,\delta)$:
\[ \Phi(\Lambda^*(\mq),\mp)\geq \Phi(\Lambda^*(\mp),\mp)+\frac{c_p}{2} \|\Lambda^*(\mp)-\Lambda^*(\mq)\|_2^2.\]
On the other hand, since $\delta<\delta_p\leq \delta^*$, \eqref{inlemma: CC: cont: bdd} implies that 
\[\Phi(\Lambda^*(\mq),\mp)\leq \Phi^*(\mp)+ C_p\|\mp-\mq\|_2\]
for all $\mq\in B(\mp,\delta)$. 
  Because $\Phi(\Lambda^*(\mp),\mp)=\Phi^*(\mp)$, we thus have
  \[\Phi^*(\mp)+\frac{c_p}{2} \|\Lambda^*(\mp)-\Lambda^*(\mq)\|_2^2\leq \Phi(\Lambda^*(\mq),\mp)\leq \Phi^*(\mp)+\frac{\|\mp-\mq\|_2}{2} C_p\]
  for all $\mq\in B(\mp,\delta)$, implying 
  \[ \|\Lambda^*(\mp)-\Lambda^*(\mq)\|_2^2\leq \frac{2C_p}{c_p}\|\mp-\mq\|_2.\]
  \end{proof}
\subsubsection{Proof of Fact \ref{claim: CC: orthant}}
\label{secpf: claim: cc: orthant}
\begin{proof}[Proof of Fact \ref{claim: CC: orthant}]
 If $\mz_{k_1+k_2-2}\in\partial(\dom(-\eta))$, then by the supporting hyperplane theorem \citep[pp. 100, Theorem 11.6][]{rockafellar}, there exists a hyperplane $\mathfrak{T}^T\mbu+\mathfrak{C}^T\mv+c=0$ with  $\mathfrak{T}\in\RR^{k_1-1}$, $\mathfrak{C}\in\RR^{k_2-1}$, and $c\in\RR$ so that 
$(\mathfrak{T},\mathfrak{C})\neq \mz_{k_1+k_2-2}$ and  $\mathfrak{T}^T\mbu+\mathfrak{C}^T\mv+c\leq 0$ for all $(\mbu,\mv)\in\dom(-\eta)$ and $\mz_{k_1+k_2-2}$ lies on the hyperplane.  The supporting hyperplane theorem applies because $\dom(-\eta)$ is a convex set due to $\eta$ being concave. Since $\mz_{k_1+k_2-2}$ lies on the hyperplane, $c=0$ and, hence, the hyperplane passes through the origin.
 Suppose $\bse=(\bse_i)_{i=1}^{k_1}\in\{\pm 1\}^{k_1-1}$ is such that $\bse_i>0$ if $\mathfrak{T}_i>0$ and $\bse_i<0$ if $\mathfrak{T}_i<0$. If $\mathfrak{T}_i=0$, then $\bse_i$ can be either $1$ or $-1$. Thus $\bse$ need not be unique. Simililarly, we define  $\bse'=(\bse_i')_{i=1}^{k_1}\in\{\pm 1\}^{k_2-1}$  such that $\bse'_i>0$ if $\beta_i>0$ and $\bse'_i<0$ if $\beta_i<0$. $\bse'$ need not be unique either since $\mathfrak{C}$ can contain zero elements. Consider the open orthant $\Q(\bse)\times \Q(\bse')$. Suppose $\mbu\in \Q(\bse) $ and $\mv\in  \Q(\bse')$. The definition of $\bse$ and $\bse'$ ensures that $\mathfrak{T}_i\mbu_i>0$ if $\mathfrak{T}_i\neq 0$ and $\beta_j\mv_j>0$ if $\beta_j\neq 0$.
Then
\[\mathfrak{T}^T\mbu+\mathfrak{C}^T\mv=\sum\limits_{i:\mathfrak{T}_i\neq 0}|\mathfrak{T}_i\mbu_i|+\sum\limits_{j:\beta_j\neq 0}|\beta_j\mv_j|.\]
Also $\mathfrak{T}$ and $\mathfrak{C}$ both can not be zero vectors because the theorem states that $(\mathfrak{T},\mathfrak{C})\neq \mz_{k_1+k_2-2}$. Therefore, there exists at least one $\mathfrak{T}_i$ or $\beta_j$ such that $\mathfrak{T}_i\neq 0$ or $\beta_j\neq 0$, leading to $\mathfrak{T}_i^T\mbu_i>0$ or $\beta_j\mv_j>0$. Therefore, $\mathfrak{T}^T\mbu+\mathfrak{C}^T\mv>0$. Therefore, $(\mbu,\mv)\notin \dom(-\eta)$ because  $\mathfrak{T}^T\mbu+\mathfrak{C}^T\mv\leq 0$  for all $(\mbu,\mv)\in\dom(-\eta)$. Since $(\mbu,\mv)$ was an arbitrary element of $\Q(\bse)\times \Q(\bse')$, it follows that
$\{\Q(\bse)\times \Q(\bse')\}\cap \dom(-\eta)=\emptyset$.  

\end{proof}

\section{Proofs of Theorem~\ref{theorem: sufficient conditions} and Proposition \ref{prop: multi-cat FC}}
\label{secpf: proof of suff}

This section is organized as follows. In Section \ref{secpf: suff: new notation}, we introduce some terminologies that will be used in the proofs of Theorem \ref{theorem: sufficient conditions} and Proposition~\ref{prop: multi-cat FC}. Section \ref{sec: necessity: preparatory results} 
lists some preliminary results  on $\psi$'s satisfying Conditions \ref{assump: N1}-\ref{assump: N2}. These results will be used not only for the proofs of this section, but also for  proving    Theorem \ref{theorem: necessity} and other results of Section \ref{sec: necessity}. Then, in 
Section \ref{sec: suff: preliminary lemmas}, we state and prove some  auxiliary lemmas on the value function under Assumptions I-V, which will be used in the proofs of this section. Sections \ref{sec: proof of theorem suff} and \ref{sec: proof of dorollary  multi-cat FC} provide the proofs of Theorem \ref{theorem: sufficient conditions} and Proposition \ref{prop: multi-cat FC}, respectively.

\subsubsection{New terminologies}
\label{secpf: suff: new notation}
\textbf{Notation and terminologies:}
In what follows, we will denote $\myq(\mx;\mp)=\mp_{\pred(\mx)}$. 
For the purpose of induction, we introduce some new notation. 
Let us define the function $\mp_T:\H_T\mapsto\RR^{k_T}$ so that 
\[\mp_T(H_T)_i=\E[Y_1+\ldots+Y_T\mid H_T, A_T=i]\text{ for } i\in [k_T],\]
and for $t=T-1,\ldots,1$, the functions $\mp_t:\H_t\mapsto\RR^{k_t}$ are recursively defined as
\begin{align}
    \label{def: suff: pt f}
    \mp_t(H_t;f)_i=\E[\Psi_{1+t}(f_{1+t}(H_{1+t});\mp_{1+t}(H_{1+t};f))\mid H_{t}, A_t=i] \text{ for all }i\in[k_t].
\end{align}
When $f$ is clear from the context, we will denote $\mp_t(H_t;f)$ simply by  $\mp_t(H_t)$.
We also define the functions $\mp_t^*:\H_t\mapsto\RR^{k_t}$ 
 so that $\mp_T^*=\mp_T$ and for $t=T-1,\ldots,1$, the $\mp_t^*$'s are recursively defined as  
\begin{align}
\label{def: sufficiency: pt star}
    \mp_t^*(H_t)_i= \E[\Psi_{1+t}^*(\mp_{1+t}^*(H_{1+t}))\mid H_t, A_t=i],\quad i\in[k_t].
    \end{align}
 For a function $h:\H_t\mapsto[k_t]$, expressions such as  $\mp_t(H_t)_{h(H_t)}$ and $\mp_t^*(H_t)_{h(H_t)}$ may also be denoted by $\mp_t(H_t,h(H_t))$ and $\mp_t^*(H_t,h(H_t))$, respectively, especially when the expression of $h(H_t)$ is long or convoluted.
   For the sake of notational simplicity, for any $\mx,\mp\in\RR^{k}$, 
   \begin{equation}
       \label{def: myq}
       \myq(\mx;\mp)=\mp_{\pred(\mx)}.
   \end{equation}


\subsection{Lemmas on the properties of $\phi$ under Conditions \ref{assump: N1}-\ref{assump: N2}}
\label{sec: necessity: preparatory results}
This section is organized as follows. 
First, we list some facts that we will repeatedly use in the proofs of Theorem \ref{theorem: sufficient conditions} and \ref{theorem: necessity}.   In Section \ref{sec: consequence of N1}, we establish properties that $\phi$ satisfies when Condition \ref{assump: N1} holds. In Section \ref{secpf: consequences of N1 and N2}, we derive properties of $\phi$ under both Conditions \ref{assump: N1} and \ref{assump: N2}.

Under Assumption I,
  \begin{align}
      \label{intheorem: necessity: IPW sum one}
  \E\lbt \frac{\phi_t(f_t(H_t);A_t)}{\pi_t(A_t\mid H_t)}\ \bl\ H_t\rbt=\Psi_t(f_t(H_t);\mo_{k_t})
  \text{ for all }t\in[T],
  \end{align}
and for any random variable $\VV$ and $i\in[k_t]$, 
\begin{align}
   \label{intheorem: necessity: IPW: 0-1}
        \E\lbt \VV\frac{1[A_t=i]}{\pi_t(A_t\mid H_t)}\ \bl\ H_t\rbt=\E[\VV\mid  H_t, A_t=i]
  \end{align}
  and
\begin{align}
    \label{intheorem: necessity: IPW: general p}
    \E\lbt \VV\frac{\phi_t(f_t(H_t);A_t)}{\pi_t(A_t\mid H_t)}\ \bl\ H_{t}\rbt=\Psi_t(f_t(H_t);\mp(H_t))
\end{align}
where $\mp(H_t)_j=\E[\VV\mid H_t, A_{t}=j]$ for all $j\in[k_t]$.
Therefore, for any $H_t$,
\begin{align}
    \label{intheorem: necessity: IPW: general p supremum}
   \sup_{f_t\in\F_t} \E\lbt \VV\frac{\phi_t(f_t(H_t);A_t)}{\pi_t(A_t\mid H_t)}\rbt=\E[\Psi_t^*(\mp(H_t))].
\end{align}
Moreover, if $h(H_{t})$ is any $\PP$-measurable function of $H_{t}$, then
\begin{align}
    \label{intheorem: necessity: IPW: supremum with other rv}
   \sup_{f_t\in\F_t} \E\lbt h(H_{t})\VV\frac{\phi_t(f_t(H_t);A_t)}{\pi_t(A_t\mid H_t)}\rbt=\E[h(H_{t})\Psi_t^*(\mp(H_t))]\text{ for all }t\in[T].
\end{align}
Also, if Assumption I holds, with $\mp(H_t)_j=\E[\VV\mid H_t, A_{t}=j]$, we obtain that 
\begin{align}
    \label{intheorem: necessity: IPW: diff}
\MoveEqLeft\E\lbt \frac{1[A_t=\pred(\mp(H_t))]-1[A_t=i]}{\pi_t(A_t\mid H_t)}h(H_t)\VV\bl H_t\rbt\nn\\
=&\ E\lbt \frac{1[A_t=\pred(\mp(H_t))]-1[A_t=i]}{\pi_t(A_t\mid H_t)}h(H_t)\E[\VV\mid H_t, A_t]\bl H_t\rbt\nn\\
=&\ E\lbt \frac{1[A_t=\pred(\mp(H_t))]-1[A_t=i]}{\pi_t(A_t\mid H_t)}h(H_t)\mp(H_t)_{A_t}\bl H_t\rbt\nn\\
=&\ \E\lbt \slb \max(\mp(H_t))-\mp(H_t)_i\srb h(H_t)\bl H_t\rbt.
\end{align}
We will also use the fact that
\begin{equation}
       \label{intheorem: fact: product of indicators} 
       \E\lbt \prod_{j=t}^T\frac{1[A_j=d_j^*(H_j)]}{\pi_j(A_j\mid H_j)}\ \bl\ H_{t}\rbt=1 \text{ for all }t\in[T].
    \end{equation}
\subsubsection{Properties of $\phi$ satisfying  Condition \ref{assump: N1}}
\label{sec: consequence of N1}
\begin{fact}
    \label{fact: Cond N1 implies pred x in argmax p}
Suppose   $\phi:\RR^{k}\times[k]\mapsto\RR$ satisfies Condition \ref{assump: N1}, where $k\in\NN$.  If $\phi(\xpm;\mp)\to_m\Psi^*(\mp)$  for some $\mp\in\RR_{\geq 0}^{k}$, 
   then  $\pred(\xpm)\in\argmax(\mp)$ for all sufficiently large $m$.
\end{fact}

\begin{proof}[Proof of Fact \ref{fact: Cond N1 implies pred x in argmax p}]
If the assertion in Fact \ref{fact: Cond N1 implies pred x in argmax p} does not hold, we can extract a subsequence of $\xpm$, if necessary, so that $\pred(\xpm)\not\in\argmax(\mp)$ but
    $\Psi(\xpm;\mp)\to\Psi^*_t(\mp)$. We denote this subsequence by $\{\xpm\}_{m\geq 1}$ for the sake of notational simplicity. Then,
    \[\sup_{\mx:\mp_{\pred(\mx)}<\max(\mp)}\Psi(\mx;\mp)\geq \sup_{m}\Psi(\xpm;\mp)= \Psi^*(\mp),\]
  which would contradict \eqref{inlemma: necessity: single-stage FC}.
    
\end{proof}
\begin{lemma}
  \label{lemma: necessity: psi bounded}  
  Suppose  $\phi:\RR^{k}\times[k]\mapsto\RR$ satisfies Condition \ref{assump: N1}, where $k\in\NN$. Then the $\phi(\cdot;i)$'s are bounded above for each  $i\in[k]$. Moreover, for each   $\mp\in\RR_{\geq 0}^{k}$ such that $\mp\neq\mz_{k}$, $\Psi^*(\mp)>0$.
\end{lemma}

\begin{proof}[Proof of Lemma \ref{lemma: necessity: psi bounded}]
First we prove the boundedness of $\phi$. 
    Suppose there exists $i\in[k]$ such that $\sup_{\mx\in\RR^{k}}\phi(\mx;i)=\infty$. In this case, $\Psi^*(\mp)=\infty$ for all   $\mp\in\RR_{\geq 0}^{k}$. Also there exists $\{\xpm\}\subset \RR^{k}$ so that $\phi(\xpm;i)\to_m\infty$. Therefore, 
    $\Psi(\xpm;\mp)\to_m\Psi^*(\mp)$ for all $\mp\in\RR_{\geq 0}^{k}$. Fact  \ref{fact: Cond N1 implies pred x in argmax p}  implies that $\pred(\xpm)\in\argmax(\mp)$ for all sufficiently large $m\in\NN$. 
  Consider $\mp_1,\ \mp_2$ in $\RR_{\geq 0}^{k}$ such that $\argmax(\mp_1)=1$ and  $\argmax(\mp_2)=2$.
  For both $i\in[1,2]$, it holds that  $\pred(\xpm)\in\argmax(\mp_i)$ for all sufficiently large $m$. 
Then there exists $M\in\NN$ such that for all $m\geq M$, $\pred(\xpm)\in\argmax(\mp_1)\cap\argmax(\mp_2)=\emptyset$, which can not hold. Therefore, $\sup\phi(\cdot;i)$ must be finite.

    Now suppose $\Psi^*(\mp')=0$ for some $\mp'\in\RR_{\geq 0}^{k}$, but $\mp'\neq\mz_{k}$. Let $\C=\{i\in[k]:\mp'_i>0\}$, which is non-empty because $\mp'\neq\mz_{k}$.  Since $\Psi^*(\mp')=0$, $\Psi(\mx;\mp')=0$ for all $\mx\in\RR^{k}$  if $i\in\C$. Therefore, if $i\in\C$, $\phi(\mx;i)=0$ for all $\mx\in\RR^{k}$. Then for any $\mp$ such that $\{i\in[k]:\mp_i>0\}=\C$, 
    $\Psi(\mx;\mp)=0$ for all $\mx\in\RR^{k}$ and  $\Psi^*(\mp)=0$. We can choose $\mp\in\RR^k_{\geq 0}$ so that
    $\{i\in[k]:\mp_i>0\}=\C$ and the non-zero  $\mp_i$'s are all different so that $[k_r]\setminus\argmax(\mp)$ is non-empty. Therefore, there exists $\mx\in\RR^k$ satisfying $\mp_{\pred(\mx)}<\max(\mp)$. Therefore, the set $\{\mx\in\RR^{k}:\mp_{\pred(\mx)}<\max(\mp)\}$ is non-empty, which implies
    \[\sup_{\mx\in\RR^{k}:\mp_{\pred(\mx)}<\max(\mp)}\Psi(\mx;\mp)>-\infty.\]
    Since $0\leq \Psi(\mx;\mp)\leq \Psi^*(\mp)$ for all $\mx\in\RR^k$ and $\Psi^*(\mp)=0$, it follows that the above supremum is $0$. Therefore, \eqref{inlemma: necessity: single-stage FC} is violated. Therefore, $\Psi^*(\mp)>0$ for all non-zero  $\mp\in\RR_{\geq 0}^{k}$.
\end{proof}
\begin{lemma}
    \label{lemma: necessity: Psi-t cont.}
    If $\phi:\RR^k\times[k]$ satisfies Condition \ref{assump: N1} for some $k\in\NN$, then $\Psi^*:\RR^k\mapsto\RR$ is positively homogenous, convex, and a continuous function.
\end{lemma}

\begin{proof}[Proof of Lemma \ref{lemma: necessity: Psi-t cont.}]
  In Section \ref{sec: necessity}, it was mentioned that $\Psi^*$ is the support function of the set $\mathcal V_\phi$. A support function of a nonempty set is sublinear \citep[cf. Proposition 2.1.2, pp. 134 of ][]{hiriart}. Sublinear functions are functions that are both convex and positively homogenous 
  \citep[cf. Definition 1.1.1, pp. 123 of ][]{hiriart}.
By Lemma \ref{lemma: necessity: psi bounded}, $\phi(\cdot;i)$ is bounded above for each $i$. Therefore, $|\Psi^*(\mp)|<\infty$ for all $\mp\in\RR^{k}$. Hence, $\dom(\Psi^*)=\RR^{k}$.  Since a convex function is continuous on the interior of its domain \citep[cf. pp. 104 of][]{hiriart}, $\Psi^*$ is continuous.

\end{proof}

\begin{lemma}
\label{lemma: necessity: Bartlett stuff}
Let $\mathfrak{B}$ be any bounded set in $\RR_{\geq 0}^k$.  
 Suppose   $\phi:\RR^{k}\times[k]\mapsto\RR_{\geq 0}$ satisfies Condition \ref{assump: N1}. Then there exists a non-negative convex function $\varrho:\RR_{\geq 0}\mapsto\RR_{\geq 0}$,  depending only on $\psi$ and $\mathfrak{B}$, so that 
\[\Psi^*(\mp)-\Psi(\mx;\mp)\geq \varrho(\max(p)-\myq(\mx;\mp))\]
for all $\mp\in \mathfrak{B}$, 
where $\myq$ is as in \eqref{def: myq}. Moreover, $\varrho(0)=0$, $\varrho$ is positive on $(0,\infty)$, and for any sequence $\{\e_m\}_{m\geq 1}\subset\RR_{\geq 0}$, $\lim_{m\to \infty}\varrho(\e_m)=0$ if and only if $\lim_{m\to\infty}\e_m=0$.
\end{lemma}

\begin{proof}
  The proof follows from \cite{zhang2004} but is given here for the sake of completeness.  Let us denote
  \[\myv(\epsilon)=\inf_{\mp\in \mathfrak{B},\mx\in \RR^k}\{\Psi^*(\mp)-\Psi(\mx;\mp):\max(\mp)-\myq(\mx;\mp)\geq \epsilon\}\]
  where the infimum is $\infty$ the set is empty. 

   Then by Proposition 23 of \cite{zhang2004}, $\myv$ is non-ngeative, $\myv(0)=0$, and non-decreasing on $\RR_{\geq 0}$. Proposition 23 of \cite{zhang2004} further implies that for any $\mp\in \mathfrak{B}$ and $\mx\in\RR^k$, 
  \[\Psi^*(\mp)-\Psi(\mx;\mp)\geq \myv(\max(\mp)-\myq(\mx;\mp)).\]
  We will now show that if $\epsilon>0$, then $\myv(\epsilon)>0$. For fixed $\mp\in \mathfrak{B}$, if $\myq(\mx;\mp)\leq \max(\mp)-\epsilon$, then $\myq(\mx;p)<\max(\mp)$, indicating $\pred(\mx)\notin\argmax(\mp)$. Moreover, $\mp\neq \mz_{k}$ since $\argmax(\mz_k)=[k]$. Therefore,
  \[\sup_{\mx: \max(\mp)-\myq(\mx;\mp)\geq\epsilon }\Psi(\mx;\mp)\leq \sup_{\mx: \pred(\mx)\notin\argmax(\mp) }\Psi(\mx;\mp)\stackrel{(a)}{<}\Psi^*(\mp),\]
  where (a) follows because $\phi$ satisfies Condition \ref{assump: N1}.
  Hence, for any $\mx\in\RR^k$ and $\mp\in \mathfrak{B}$, $\Psi^*(\mp)-\Psi(\mx;\mp)>0$ if $\max(\mp)-\myq(\mx;\mp)\geq\epsilon $.

 Now suppose $\myv(\epsilon)=0$. Then there exists a sequence $(\mxm,\mpm)_{m\geq 1}\subset\RR^k\times \mathfrak{B}$ such that $\max(\mpm)-\myq(\mxm;\mpm)\geq\epsilon$ for all $m\geq 1$ but $\Psi^*(\mpm)-\Psi(\mxm;\mpm)\to_m 0$. Since $\mathfrak{B}$ is bounded, $\mpm$ has a convergent subsequence that converges to $\mp^0\in\overline{\mathfrak{B}}$. 
  Then
  \begin{align*}
  \MoveEqLeft    \abs{\Psi^*(\mpm)-\Psi(\mxm;\mpm)-\slb   \Psi^*(\mp^0)-\Psi(\mxm;\mp^0)\srb}\\
  \leq &\ \abs{\Psi^*(\mpm)-\Psi^*(\mp^0)}+\abs{\Psi(\mxm;\mp^0)-\Psi(\mxm;\mpm)},
  \end{align*}
  whose first term goes to zero as $m\to\infty$ because $\Psi^*$ is continuous on $\RR^k$ by Lemma \ref{lemma: necessity: psi bounded}. 
 For the second term, note that since $\Psi$ is linear in $\mp$, 
  \begin{align*}
  \MoveEqLeft \abs{\Psi(\mxm;\mp^0)-\Psi(\mxm;\mpm)}= \abs{\Psi(\mxm;\mp^0-\mpm)}
  \leq \sup_{\mx\in\RR^k,i\in[k]}\phi(\mx;i)\|\mp^0-\mpm\|_1
  \end{align*}
  which goes to zero as $m\to \infty$ because $\phi$ is bounded by Lemma \ref{lemma: necessity: psi bounded} and $\mpm\to_m\mp^0$. Hence, we have showed that
  \[\abs{\Psi^*(\mpm)-\Psi(\mxm;\mpm)-\slb   \Psi^*(\mp^0)-\Psi(\mxm;\mp^0)\srb}\to_m 0,\]
  which, combined with the fact that $\Psi^*(\mpm)-\Psi(\mxm,\mpm)\to_m 0$,  implies that
  \[\Psi(\mxm;\mp^0)\to_m \Psi^*(\mp^0).\]
 On the other hand, $\max(\mpm)-\myq(\mxm;\mpm)\geq\epsilon$ for all $m\geq 1$. 
 Therefore, 
 \begin{align*}
    \MoveEqLeft \max(\mp^0)-\myq(\mxm;\mp^0)\\
     \geq &\ \max(\mpm)-\myq(\mxm;\mpm)-|\myq(\mxm;\mpm)-\myq(\mxm;\mp^0)|\\
     &\ -|\max(\mpm)-\max(\mp^0)|\\
     \geq&\  \epsilon -|\myq(\mxm;\mpm)-\myq(\mxm;\mp^0)|-|\max(\mpm)-\max(\mp^0)|.
 \end{align*}
Since $\mpm\to_m\mp^0$, we have  $\max(\mpm)\to_m\max(\mp^0)$. Also
 \[|\myq(\mxm;\mpm)-\myq(\mxm;\mp^0)|=|(\mpm)_{\pred(\mxm)}-(\mp^0)_{\pred(\mxm)}|\leq \|\mpm-\mp^0\|_2,\]
 which converges to $0$ as $m\to\infty$. 
 Hence, 
  $ \max(\mp^0)-\myq(\mxm;\mp^0)\geq \epsilon/2$ for all sufficiently large $m$. Hence, $\mp_{\pred(\mx)}<\max(\mp^0)$ for all sufficiently large $m$. However, $\Psi(\mxm;\mp^0)\to_m\Psi^*(\mp^0)$. 
  Therefore, 
  \[\sup_{\mp^0_{\pred(\mx)}<\max(\mp^0)}\Psi(\mxm;\mp^0)=\Psi^*(\mp^0),\]
  which is a contradiction since $\mp^0\in\overline {\mathfrak{B}}\subset \RR_{\geq 0}^k$. 
  Thus, the assumption that $\myv(\epsilon)=0$ was wrong. Hence,  $\myv(\epsilon)>0$. 


Let us extend $\myv$ to $\RR$ by setting $\myv$ to be $0$ on $(-\infty,0)$. Since $\myv\geq 0$, $\myv$ is bounded below by the affine function $f(0)=0$. Then 
Proposition 25 of \cite{zhang2004} applied to our case (we need to take $Q=\mathfrak{B}$, $l(\mp, d)=\max(\mp)-\myq(d,\mp)$, and $v\equiv 1$ in that proposition) implies there is a convex function $\varrho$ such that $\varrho(x)\leq \myv(x)$, $\varrho(\e)>0$ whenever $\epsilon>0$, $\varrho(0)=0$, and $\varrho$ is non-decreasing. Since $\varrho\leq\myv$ and $\myv<\infty$ on $\RR$, $\dom(\varrho)=\RR$. 

It remains to show that $\varrho(\e_m)\to_m 0$ if and only if $\e_m\to 0$. For the if part, note that $\varrho$ is continuous at $0$ because $\varrho$ is convex and $0\in\iint(\dom(\varrho))\equiv\RR$. Thus $\lim_{\e\to 0}\varrho(\e)=0$.
Now we prove the only if part. Suppose $\{\e_m\}_{m\geq 1}$ is a subsequence so that  $\e_m\geq 0$ for all $m\geq 1$ and $\varrho(\e_m)\to_m 0$. Then we claim that $\e_m\to 0$. We can replace $\{\e_m\}_{m\geq 1}$, if necessary, with a subsequence so that $\e_m\geq \e>0$ for all $m\geq 1$. Then $\varrho(\e_m)\geq \varrho(\e)>0$ for all $m$ along that sequence since $\varrho$ is non-decreasing. However, all subsequences of $\varrho(\e_m)$ must go to zero, leading to a contradiction. Hence, we have shown that $\varrho(\e_m)\to_m 0$  only if $\e_m\to 0$.

\end{proof}
\begin{lemma}
\label{lemma: necessity psi(xk, pred xk) is lower bounded}
   Suppose  $m\in\NN$ and the single-stage surrogate $\phi:\RR^k\times[k]$ satisfies Condition \ref{assump: N1} and $\Psi^*(\mo_k)=1$. 
    Let $\mp\in\RR_{\geq 0}^k$ satisfy $\mp\neq\mz_k$. Then there exists $\mytheta(\mp)>0$, depending only on $\mp$ and $\phi$, so that for any sequence $\mxm$ satisfying $\Psi(\mxm;\mp)\to_m \Psi^*(\mp)$,
    \[\liminf_m\phi(\mxm;\pred(\mxm))>\mytheta(\mp).\]
\end{lemma}

\begin{proof}
Had  the $\phi(\cdot;i)$'s been zero for all $i\in[k]$, then $\Psi(\mx;\mp)=\Psi^*(\mp)=0$ for all $\mp\in\RR_{\geq 0}^k$. This contradicts that $\Psi^*(\mo_k)= 1$. Thus, there exists a $i\in[k]$ so that  $\phi(\cdot;i)$ is not identically zero. 
Since $\mp\neq \mz_k$,  the above implies $\Psi^*(\mp)\geq \max(\mp)\sup_{\mx\in\RR}\phi(\mx;i)>0$.

Let us define the function 
\[\overline\Omega(\mp)=\sup_{\pred(\mx)\notin\argmax(\mp)}\Psi(\mx;\mp)\]
for all $\mp\in\RR_{\geq 0}^k$.
Since $\phi$ satisfies Condition \ref{assump: N1}, for all $\mp\in\RR_{\geq 0}^k$, it holds that $\overline\Omega(\mp)<\Psi^*(\mp)$. Recall that we denote by $\mathbf e^{(i)}_k$ the unit vector of length one with one in the $i$th position. 
For any $\mp$, we denote the set
\[\myG(\mp)=\lbs \mp- \mathbf e^{(i)}_k\mp_i: i\in\argmax(\mp)\rbs.\]
 Note that $|\myG(\mp)|=|\argmax(\mp)|$ is a finite number. 
 However, for any  $\mq\in\myG(\mp)$, since the $\phi$'s are non-negative, 
 \[\Psi(\mx;\mq)=\sum_{i\in[k]}\phi(\mx;i)\mq_i\leq\sum_{i\in[k]}\phi(\mx;i)\mp_i=\Psi(\mx;\mp)\]
 for all $\mx\in\RR^k$. Therefore, for any  $\mq\in\myG(\mp)$,
 \[\Psi^*(\mq)=\sup_{\mx\in\RR^k}\Psi(\mx;\mq)\leq \Psi^*(\mp).\]
By definition of $\overline\Omega$,  $\overline\Omega(\mq)< \Psi^*(\mq)\leq \Psi^*(\mp)$ for each $\mq\in\myG(\mp)$ satisfying $\mq\neq\mz_k$. Since $\myG(\mp)$ is a finite set, it therefore follows that   $\max_{\mq\in\myG(\mp),\mq\neq\mz_k}\overline\Omega(\mq)<\Psi^*(\mp)$.  Therefore
\begin{equation}
    \label{def: intheorem: suff: tilde mytheta}
    \tilde\mytheta(\mp):=\Psi^*(\mp)-\max_{\mq\in\myG(\mp),\mq\neq\mz_k}\overline\Omega(\mq)
\end{equation}
is positive. 
 Since $\mp\neq\mz_k$, there exists a $\mq\in\myG(\mp)$ so that $\mq\neq\mz_k$. Hence, the maximum $\max_{\mq\in\myG(\mp),\mq\neq\mz_k}\overline\Omega(\mq)>-\infty$.
 Therefore,   $\tilde\mytheta(\mp)\in\RR_{>0}$.

Since $\Psi(\mxm;\mp)\to_m \Psi^*(\mp)$, and $\phi$ satisfies Condition \ref{assump: N1},  $\pred(\mxm)\in\argmax(\mp)$ for all sufficiently large $m\in\NN$ by Fact \ref{fact: Cond N1 implies pred x in argmax p}. Therefore, given any $\e>0$,
 $\Psi(\mxm,\mp)>\Psi^*(\mp)-\e$ for all sufficiently large $m$. Let us choose $\epsilon<\tilde\mytheta(\mp)/2$. Note that
 \begin{align*}
     \Psi(\mxm;\mp) =&\ \sum_{i\in[k_1]:i\neq \pred(\mxm)}\phi(\mxm;i)\mp_i +\max(\mp)\phi(\mxm;\pred(\mxm)).
 \end{align*}
 Therefore,
 \begin{align*}
   \max(\mp)\phi(\mxm;\pred(\mxm))=&\ \Psi(\mxm;\mp)-\sum_{i\in[k_1]:i\neq \pred(\mxm)}\phi(\mxm;i)\mp_i\\
   \geq &\ \Psi^*(\mp)-\epsilon -\sum_{i\in[k_1]:i\neq \pred(\mxm)}\phi(\mxm;i)\mp_i
 \end{align*}
 for all sufficiently large $m$. 
Let us define $\mq^{(m)}$ so that  $\mq^{(m)}_i=\mp_i$ if $i\neq \pred(\mxm)$ and 0 otherwise. Since
 $\pred(\mxm)\in\argmax(\mp)$ for all sufficiently large $m\in\NN$, $\mq^{(m)}\in\myG(\mp)$ for all sufficiently large $m$. Then
\[\sum_{i\in[k_1]:i\neq \pred(\mxm)}\phi(\mxm;i)\mp_i=\Psi(\mxm;\mq^{(m)}).\]
Since $\mp\neq\mz_k$, $\max(\mp)>0$. There can be two situations. Either $\mq^{(m)}=\mz_{k}$ or $\max(\mq^{(m)})>0$. In the first case,
$\sum_{i\in[k_1]:i\neq \pred(\mxm)}\phi(\mxm;i)\mp_i=0$. Therefore,
\[\phi(\mxm;\pred(\mxm)\geq\frac{\Psi^*(\mp)-\epsilon}{\max(\mp)}>\frac{\Psi^*(\mp)}{2\max(\mp)}\]
because $\epsilon<\tilde \mytheta(\mp)/2\leq \Psi^*(\mp)/2$ by \eqref{def: intheorem: suff: tilde mytheta}. 
In the second case, $\pred(\mxm)\notin\argmax(\mq^{(m)})$
because $\mq^{(m)}_{\pred(\mxm)}=0$. 
Suppose  $\mx\in\RR^k$ satisfies $\pred(\mx)\notin\argmax(\mq)$ for some $\mq\in\myG(\mp)$ such that $\mq\neq \mz_{k}$. Then  it holds that  
\[\Psi(\mx;\mq)\leq \overline\Omega(\mq)\leq \max_{\mq\in\myG(\mp),\mq\neq \mz_{k}}\overline\Omega(\mq)=\Psi^*(\mp)-\tilde\mytheta(\mp).\]
Since $\mq^{(m)}\in\myG(\mp)$ for all sufficiently large $m\in\NN$,  it follows that
\[\Psi(\mxm;\mq^{(m)})<\Psi^*(\mp)-\tilde\mytheta(\mp)\]
given $m$ is sufficiently large. Therefore, in the second case, 
\[\phi(\mxm;\pred(\mxm)\geq\frac{\Psi^*(\mp)-\epsilon-\Psi(\mxm;\mq^{(m)})}{\max(\mp)}\geq \frac{\tilde\mytheta(\mp)-\epsilon}{\max(\mp)}\geq\frac{\tilde\mytheta(\mp)}{2\max(\mp)}\]
since $\epsilon\leq \tilde\mytheta(\mp)/2$.
Thus, we have shown that for 
 all sufficiently large $m$, 
 \[\phi(\mxm;\pred(\mxm)\geq \frac{\tilde\mytheta(\mp)}{2\max(\mp)}.\]
The definition of $\tilde\mytheta(\mp)$ in \eqref{def: intheorem: suff: tilde mytheta} implies that   $\tilde\mytheta(\mp)\leq \Psi^*(\mp)$. Hence, \[\phi(\mxm;\pred(\mxm))\geq \frac{\tilde\mytheta(\mp)}{2\max(\mp)}.\] Since $\tilde\mytheta(\mp)>0$, the proof follows with $\mytheta(\mp)=\tilde\mytheta(\mp)/2$.
\end{proof}

\begin{lemma}
\label{lemma: necessity: phi bdd away from zero step 1}
   Let  $k\in\NN$. Suppose $\phi:\RR^k\times[k]$, which is not identiacally zero,  satisfies Condition \ref{assump: N1}. Let $\mxm$ be such that there exists $\mp\in\RR^{k}_{\geq 0}$ and $\mp\neq 0$ so that $\Psi(\mxm;\mp)\to\Psi^*(\mp)$. Then there exists a constant $c>0$, depending only on $\psi$, so that $\liminf_{m\to\infty}\phi(\mxm;\pred(\mxm))\geq c$. 
\end{lemma}

\begin{proof}[Proof of Lemma \ref{lemma: necessity: phi bdd away from zero step 1}]
Since $\phi$ is not identically zero, $\Psi^*(\mo_k)>0$. Without loss of generality, we assume that   $\Psi^*(\mo_k)=1$. If not,  we can replace $\phi$ by $\phi/\Psi^*(\mo_k)$. This ensures that we can apply Lemma \ref{lemma: necessity psi(xk, pred xk) is lower bounded}, which implies
\[\liminf_{m\to\infty}\phi(\mxm;\pred(\mxm))\geq \mytheta(\mp).\]
Therefore, it suffices to show that  $c=\inf_{\mp\in\RR^{k}_{\geq 0},\mp\neq 0}\mytheta(\mp)>0$ where $\mytheta$ is as in Lemma \ref{lemma: necessity psi(xk, pred xk) is lower bounded}.
    Suppose, if possible, $c=0$.  Then 
    there exist a sequence $\{\mpm\}_{m\geq 1}\subset \RR^{k}_{\geq 0}$ and sequences $\{\mxrm\}_{r\geq 1}\subset \RR^k$ so that $\mpm\neq 0$, $\Psi(\mxrm;\mpm)\to_r \Psi^*(\mpm)$  for each $m\in\NN$ but $\liminf_r\phi(\mxrm;\pred(\mxrm))\to_m 0$. 
     Each $\{\mxrm\}_{r\geq 1}$ has a subsequence $\{\mx^{m,r'}\}_{r'\geq 1}$ so that  
     \[\lim_{r'\to\infty}\phi(\mx^{m,r'};\pred(\mx^{m,r'}))=\liminf_{r\to\infty}\phi(\mxrm;\pred(\mxrm)).\] Replacing the original sequences by these  subsequences, we obtain  sequences $\{\mxrm\}_{r\geq 1}\subset \RR^k$  so that $\Psi(\mxrm;\mpm)\to_r \Psi^*(\mpm)$ for each $m\in\NN$,  but \[\lim_{m\to\infty}\lim_{r\to\infty}\phi(\mxrm;\pred(\mxrm))=0.\] 
    
    Note that since $\mpm\neq \mz_{k}$, $\|\mpm\|_1=\sum_{l\in[k]}\mpm_l>0$. Therefore, we obtain that 
    \[\Psi(\mxrm;\mpm/\|\mpm\|_1)\stackrel{(a)}{=}\frac{\Psi(\mxrm;\mpm)}{\|\mpm\|_1}\stackrel{(b)}{\to_r} \frac{\Psi^*(\mpm)}{\|\mpm\|_1}\stackrel{(c)}{=}\Psi^*(\mpm/\|\mpm\|_1),\]
    where (a) follows because the map $\mp\mapsto\Psi(\mx,\mp)$ is linear in $\mp$, (b) follows because  $\Psi(\mxrm;\mpm)\to_r \Psi^*(\mpm)$, and (c) follows because $\Psi^*$ is positively homogeneous by Lemma \ref{lemma: necessity: Psi-t cont.}.
    Therefore, without the loss of generality, we can write that there exists $\{\mpm\}_{m\geq 1}\subset \S^{k-1}$ so that $\Psi(\mxrm;\mpm)\to_r \Psi^*(\mpm)$  for each $m\in\NN$ but \[\lim_{m\to\infty}\lim_{r\to\infty}\phi(\mxrm;\pred(\mxrm))=0.\] Since $\S^{k-1}$ is closed and bounded, there is a subsequence $\{m_l\}_{l\geq 1}\subset\{m\}_{m\geq 1}$ so that $\mp^{(m_l)}\to_l\mp^*\in\S^{k-1}$. For the sake of notational simplicity, we will denote the sequence $\{m_l\}_{l\geq 1}$ by $\{l\}_{l\geq 1}$. Hence $l\to\infty$ and  $\mpl\to_l\mp^*$ as $l\to\infty$. We will next show that if $\mpl\to_l\mp^*$, then there exists  a  sequence $\mxl$ so that $\Psi(\mxl;\mp^*)\to_l \Psi^*(\mp^*)$ but $\lim_{l\to\infty}\phi(\mxl;\pred(\mxl))=0$, which will contradict Lemma \ref{lemma: necessity psi(xk, pred xk) is lower bounded}, thus completing the proof.

  Fix $l\geq 1$. Since $\lim_{m\to\infty}\lim_{r\to\infty}\phi(\mxrm,\pred(\mxrm))=0$, there exists $M_l\in\NN$ so that for all $m\geq M_l$, 
   \[\lim_{r\to\infty}\phi(\mxrm,\pred(\mxrm))<1/l.\]
   This implies that for all $m\geq M_l$,  there exists $R_m>0$ so that for all $r\geq R_m$,
   \[\phi(\mxrm,\pred(\mxrm))<1/l.\]
   In particular,
   \[\phi(\mx^{M_l,R_{M_l}},\pred(\mx^{M_l,R_{M_l}}))<1/l.\]
 Since $\Psi(\mx^{m,r};\mpm)\to_r \Psi^*(\mpm)$ for each $m\in\NN$, given any $m$ and $l$, there exists $\mathfrak{K}_{m,l}>0$ so that  if $r\geq \mathfrak{K}_{m,l}$, it holds that 
    $|\Psi^*(\mpm)-\Psi(\mx^{m,r};\mpm)|<1/l$. 

We choose $\mx^{(1)}=\mx^{1,1}$. For any $l\in\NN$ so that $l\geq 2$, we choose $\mx^{(l)}=\mx^{m_l,n_l}$ so that $m_l= \max(m_{l-1}+1,M_l)$, and $n_l=\max(n_{l-1},R_{m_l},\mathfrak{K}_{m_l,l})$. We will now show that if  $\mpi\to_l\mp^*$, then  $\Psi(\mxl;\mp^*)\to_l \Psi^*(\mp^*)$.
For each $l\in\NN$, 
    \begin{align*}
   \Psi^*(\mp^*)-\Psi(\mxl;\mp^*) 
     \leq & | \Psi^*(\mp^*)-\Psi^*(\mp^{(m_l)})  |+|\Psi^*(\mp^{(m_l)})-\Psi(\mx^{m_l,n_l};\mp^{(m_l)})|\\
     &\ +|\Psi(\mx^{m_l,n_l};\mp^{(m_l)})-\Psi(\mx^{m_l,n_l};\mp^*)|.
    \end{align*}
   Since $m_l>m_{l-1}$, we have $m_l\to\infty$. Thus $\mp^{(m_l)}\to_l \mp^*$ as $l\to\infty$.  Since $\Psi^*$ is continuous on $\RR^k$ by Lemma \ref{lemma: necessity: Psi-t cont.}, the above implies  $\Psi^*(\mp^{(m_l)})\to_l \Psi^*(\mp^*)$, indicating that the first term in the above bound approaches  zero. Since  $n_l\geq \mathfrak{K}_{m_l,l}$,
   it holds that
   \[|\Psi^*(\mp^{(m_l)})-\Psi(\mx^{m_l,n_l};\mp^{(m_l)})|<1/l.\]
   Therefore,
   \[\lim_{l\to\infty}|\Psi^*(\mp^{(m_l)})-\Psi(\mx^{m_l,n_l};\mp^{(m_l)})|=0.\]
On the other hand,
\[|\Psi(\mx^{m_l,n_l};\mp^{(m_l)})-\Psi(\mx^{m_l,n_l};\mp^*)|\leq \sup_{\mx\in\RR^k,l\in[k]}|\phi(\mx,l)|\|\mp^{(m_l)}-\mp^*\|_1,\]
   which goes to zero as $l\to\infty$ because $\mp^{(m_l)}\to_l \mp^*$ and $\phi$ is bounded by Lemma \ref{lemma: necessity: psi bounded}. Therefore, we have shown that if $\mpi\to_l\mp^*$, then  $\Psi(\mxl;\mp^*)\to_l \Psi^*(\mp^*)$. Therefore, it just remains to show that $\lim_{m\to\infty}\phi(\mxl;\pred(\mxl))=0$. Since $m_l\geq M_l$ and $n_l\geq R_{m_l}$, 
   \[\phi(\mxl;\pred(\mxl))=\phi(\mx^{m_l,n_l};\pred(\mx^{m_l,n_l}))<1/l.\]
  Since $\phi$ is non-negative, the above implies that 
   \[\lim_{l\to\infty}\phi(\mx^{m_l,n_l};\pred(\mx^{m_l,n_l}))=0,\]
   which proves the desired contradiction.


\end{proof}

\begin{lemma}
\label{lemma: necessity: phi bdd away from zero}
   Let  $k\in\NN$. Suppose $\phi:\RR^k\times[k]\mapsto\RR$, which is not identiacally zero,  satisfies Condition \ref{assump: N1}. The sequences $\{\mxm\}\subset\RR^{k}$ and $\{\mpm\}\subset\RR^{k}_{\geq 0}$ are such that  $\Psi^*(\mpm)-\Psi(\mxm;\mpm)\to_m 0$. Moreover, the sequences $\mpm$ satisfy $\liminf_{m\to\infty}\|\mpm\|_1>0$. Then there exists a constant $\J>0$, depending only on $\psi$, so that $\liminf_{m\to\infty}\phi(\mxm;\pred(\mxm))\geq \J$. 
\end{lemma}

\begin{proof}[Proof of Lemma \ref{lemma: necessity: phi bdd away from zero}]

    Since $\liminf_{m\to\infty}\|\mpm\|_1>0$,  $\|\mpm\|_1$ is bounded away from $0$. Therefore,
   $\mqm=\mpm/\|\mpm\|_1$ is well defined for all sufficiently large $m\in\NN$ and 
   \[\Psi^*(\mqm)-\Psi(\mxm;\mqm)\stackrel{(a)}{=}\frac{\Psi^*(\mpm)-\Psi(\mxm;\mpm)}{\|\mpm\|_1}\to_m 0,\]
   where (a) follows because $\Psi^*$ is positively homogeneous by Lemma \ref{lemma: necessity: Psi-t cont.} and $\Psi(\mx;\cdot)$ is linear for each fixed $\mx\in\RR^k$.
   Therefore, we can replace $\mpm$ with $\mqm$.
   Note that $\mqm\in\S^{k-1}$. Hence, without loss of generality, we assume that $\mpm\in\S^{k-1}$.
   
   If possible, suppose the assertion in the statement of Lemma \ref{lemma: necessity: phi bdd away from zero} is false. This implies there exists $\epsilon>0$ and    sequences $\{\mxm\}\subset\RR^{k}$ and $\{\mpm\}\subset\S^{k-1}$ such that  $\Psi^*(\mpm)-\Psi(\mxm;\mpm)\to_m 0$ but  $\lim_{m\to\infty}\phi(\mxm;\pred(\mxm))=\J-\epsilon$. Therefore, all subsequences $\mxmr$ of $\mxm$ must satisfy 
   \[\lim_{r\to\infty}\phi(\mxmr;\pred(\mxmr))=\J-\epsilon.\] We will show that there exists a subsequence $\{m_r\}_{r\geq 1}\subset \{m\}$ so that 
   \[\liminf_{r\to\infty}\phi(\mx^{(m_r)};\pred(\mx^{(m_r)}))\geq \J,\]
   which will finish the proof by contradiction. We will consider the following subsequence. Since $\{\mpm\}_{m\geq 1}\subset\S^{k-1}$ and $\S^{k-1}$ is closed, there exists a subsequence $\{m_r\}_{r\geq 1}\subset \{m\}$ so that $\mpmr\to_r \mp^*\in\S^{k-1}$. It is clear that $\mp^*\neq 0$. If we can show $\Psi(\mxmr,\mp^*)\to_r\Psi^*(\mp^*)$,  then  Lemma \ref{lemma: necessity: phi bdd away from zero step 1} would imply  $\liminf_{r\to\infty}\phi(\mx^{(m_r)};\pred(\mx^{(m_r)}))\geq \J$, leading to the desired contradiction. Therefore, it is enough to prove that $\Psi(\mxmr,\mp^*)\to_r\Psi^*(\mp^*)$. 

   To this end, note that
   \begin{align*}
    \MoveEqLeft   \Psi^*(\mp)-\Psi(\mxmr,\mp^*)= \Psi^*(\mp)-\Psi^*(\mpmr)\\
    &\ +\Psi^*(\mpmr)-\Psi(\mxmr;\mpmr)+\Psi(\mxmr;\mpmr)-\Psi(\mxmr,\mp^*),
   \end{align*}
   which converges to zero as $r\to\infty $ because (a) $\Psi^*(\mpmr)\to_r\Psi^*(\mp)$ since $\Psi^*$ is continuous by Lemma \ref{lemma: necessity: Psi-t cont.} and $\mpmr\to_r \mp^*$, (b) $\Psi^*(\mpmr)-\Psi(\mxmr;\mpmr)\to_r 0$ because $\{m_r\}_{r\geq 1}\subset \{m\}$ and $\Psi^*(\mpm)-\Psi(\mxm;\mpm)\to_m 0$, and (c) $\Psi(\mxmr;\mpmr)-\Psi(\mxmr,\mp^*)\to_r 0$ since
    \[\Psi(\mxmr;\mpmr)-\Psi(\mxmr,\mp^*)\leq \sup_{\mx\in\RR^k,i\in[k]}\phi(\mx;i)\|\mpmr-\mp^*\|_1,\]
    which converges to $0$ as $\mpmr\to_r \mp^*$. Hence, the proof follows.

\end{proof}
\subsubsection{Properties of $\phi$ under both Conditions \ref{assump: N1} and \ref{assump: N2}}
\label{secpf: consequences of N1 and N2}
\begin{lemma}
\label{lemma: necessity: bound on psi(x;i) if i is not pred(x)}
    Suppose   $\phi:\RR^k\times[k]\mapsto\RR$ satisfies Conditions \ref{assump: N1} and \ref{assump: N2} with $C_\phi=1$.
    Then there exists
    $\Omega<1$ so that for all $\mx\in\RR^k$, 
    \[\sum_{i\in[k_1]:i\neq\pred(\mx)}\phi(\mx;i)<\Omega.\]
\end{lemma}

\begin{proof}[Proof of Lemma \ref{lemma: necessity: bound on psi(x;i) if i is not pred(x)}]
   Fix $i\in[k]$. Let $\C_i=\{\mx\in\RR^k:\pred(\mx)= i\}$. Recall that we denote by $\mathbf e^{(i)}_k$ the unit vector with one in the $i$th position. Consider $\mp^{(i)}=\mo_{k}-\mathbf e^{(i)}_k$, whose $i$th element is zero  but all other elements are one. Clearly, if $\mx\in\C_i$, then $\pred(\mx)=i\notin\argmax(\mp^{(i)})$. Let $\mathfrak{G}_i=\sup_{\mx\in \C_i}\Psi(\mx;\mp^{(i)})$. Note that $\Psi^*(\mp^{(i)})=\max(\mp^{(i)})$ because $\phi$ satisfies Condition \ref{assump: N2} with $C_\phi=1$. Thus, $\Psi^*(\mp^{(i)})=1$. Consequently, $\mathfrak{G}_i<1$ since 
   \[\mathfrak{G}_i=\sup_{\mx\in \C_i}\Psi(\mx;\mp^{(i)})\stackrel{(a)}{<}\Psi^*(\mp^{(i)})=1,\]
  where step (a) follows from Condition \ref{assump: N1}.
  
   Since $\Psi(\mx;\mp^{(i)})=\sum_{j\in[k_2]:j\neq i}\phi(\mx;j)$, we have shown that if $\mx\in\C_i$, then
   \[\sum_{j\in[k_2]:j\neq i}\phi(\mx;j)\leq \mathfrak{G}_i\]
   for some $\mathfrak{G}_i<1$. We can show that above holds for all for all $i\in[k]$. Let $\Omega=\max_{i\in[k]}\mathfrak{G}_i$. Clearly, $\Omega<1$. Since any $\mx\in\RR^k$ belongs to $\C_{\pred(\mx)}$, we have
   \[\sum_{j\in[k_2]:j\neq \pred(\mx)}\phi(\mx;j)\leq \alpha_{\pred(\mx)}\leq\Omega,\]
  which completes the proof.
\end{proof}

\begin{lemma}
    \label{lemma: necessity: linear bound}
    Suppose   $\phi:\RR^k\times[k]\mapsto\RR$ satisfies Conditions \ref{assump: N1} and \ref{assump: N2}. Then there exists $\CC_{\phi}\in(0,1]$, depending only on $\phi$, so that
    \[\Psi^*(\mp)-\Psi(\mx;\mp)\geq C_\phi\CC_{\phi}(\max(\mp)-\myq(\mx;\mp))\]
    for all $\mx\in\RR^k$ and $\mp\in \RR_{\geq 0}^k$. Here $C_\phi>0$ is as in Condition \ref{assump: N2}.
\end{lemma}

\begin{proof}[Proof of Lemma \ref{lemma: necessity: linear bound}]
First, we will consider the case when $C_\phi=1$. 
Recall from \eqref{def: myq} that $\myq(\mx;\mp)=\mp_{\pred(\mx)}$. Let us define $z=\myq(\mx;\mp)$. Then 
\begin{align*}
  \Psi^*(\mp)-\Psi(\mx;\mp)
    =&\ \Psi^*(\mp)-z+z\slb 1-\sum_{i\in[k]}\phi(\mx;i)\srb +\sum_{i\in[k_1]:i\neq\pred(\mx)}(z-\mp_i)\phi(\mx;i),
\end{align*}
  which is bounded below by 
  \begin{align*}
   \Psi^*(\mp)-z+\sum_{i\in[k_1]:i\neq\pred(\mx)}(z-\mp_i)\phi(\mx;i)   
  \end{align*}
  since $z\geq 0$ and 
  \[\sum_{i\in[k]}\phi(\mx;i)\leq \Psi^*(\mo_{k})\stackrel{(a)}{=}C_\phi=1,\]
  where (a) follows from Condition \ref{assump: N2}. 
However, since $\Psi^*(\mp)=\max(\mp)$, $\phi(\mx;i)\geq 0$, and $z-\mp_i\geq z-\max(\mp)$, the expression in the above display is bounded below by
\begin{align*}
 \MoveEqLeft  \max(\mp)-z+\sum_{i\in[k_1]:i\neq \pred(\mx)}(z-\max(\mp))\phi(\mx;i)\\
   =&\  ( \max(\mp)-z)\slb 1-\sum_{i\in[k_1]:i\neq \pred(\mx)}\phi(\mx;i)\srb.
     \end{align*}
 Lemma \ref{lemma: necessity: bound on psi(x;i) if i is not pred(x)} implies $\sum_{i\in[k_1]:i\neq \pred(\mx)}\phi(\mx;i)\leq \Omega$ where $\Omega<1$, and it depends only on $\phi$. Since $z=\myq(\mx;\mp)$, it follows that
 \[ \Psi^*(\mp)-\Psi(\mx;\mp)\geq (1-\Omega)(\max(\mp)-\myq(\mx;\mp)).\]
Since $\Omega\in[0,1)$, $1-\Omega\in(0,1]$. Hence, the proof follows for $C_\phi=1$. When $C_\phi\neq 1$, the proof follows taking $\tilde\phi$ to be $\phi/C_\phi$ and applying the above result to $\tilde\phi$.

\end{proof}

\begin{lemma}
    \label{lemma: symmetric surrogate lemma}
     Suppose the single-stage surrogate $\phi:\RR^k\times[k]\times\RR$ satisfies Conditions \ref{assump: N1} and \ref{assump: N2} for some $k\in\NN$. Further suppose, there exists $C>0$ so that  $\Psi(\mx,\mo_k)=C$ for all $\mx\in\RR^k$. Then there exists $\J>0$ so that $\phi(\mx;\pred(\mx))>\J$.
\end{lemma}

\begin{proof}[Proof of Lemma \ref{lemma: symmetric surrogate lemma}]
Without loss of generality, we assume that $C=1$. If not, we can replace $\phi$ by $\phi/C$ so that $\Psi(\mx,\mo_k)=1$. 
Hence, the conditions of  Lemma \ref{lemma: necessity psi(xk, pred xk) is lower bounded} are satisfied. We let $\overline\Omega$, $\myG$, and $\tilde\mytheta$ be as in the proof of Lemma \ref{lemma: necessity psi(xk, pred xk) is lower bounded}. We have shown in the proof of Lemma \ref{lemma: necessity psi(xk, pred xk) is lower bounded} that $\tilde\mytheta(\mp)>0$ for all $\mp\in\RR^k_{\geq 0}$ such that $\mp\neq \mz_k$. 
   Since $\phi$ satisfies $\Psi(\mx;\mo_{k})=1$, 
\begin{align*}
    1= \Psi(\mx;\mo_k) =&\ \sum_{i\in[k_1]:i\neq \pred(\mx)}\phi(\mx;i) +\phi(\mx;\pred(\mx)).
 \end{align*}
 Hence,
\begin{align*}
    \phi(\mx;\pred(\mx))=1-\sum_{i\in[k_1]:i\neq \pred(\mx)}\phi(\mx;i)=1-\Psi(\mx; \mq),
\end{align*}
where $\mq=\mo_k-\mathbf e^{(\pred(\mx))}_k$. Note that $\mq\in \myG(\mo_{k})$. Moreover, $\argmax(\mq)=[k]\setminus \{\pred(\mx)\}$,  indicating $\pred(\mx)\notin\argmax(\mq)$.
Therefore, $\Psi(\mx;\mq)\leq \overline{\Omega}(\mq)$. Hence, 
\begin{align*}
 \MoveEqLeft \phi(\mx;\pred(\mx))\geq 1- \overline{\Omega}(\mq) \geq  1- \sup_{\mq\in\myG(\mo_k)}\overline{\Omega}(\mq) =\max(\mo_k)-\sup_{\mq\in\myG(\mo_k)}\overline{\Omega}(\mq)\\
  =&\ \Psi^*(\mo_k)-\sup_{\mq\in\myG(\mo_k)}\overline{\Omega}(\mq)=\tilde\mytheta(\mo_k).
\end{align*}
The proof  follows since $\tilde\mytheta(\mp)>0$ for all $\mp\neq \mz_k$, as mentioned in the beginning of the proof.
\end{proof}

\begin{lemma}
\label{lemma: sufficiency: tilde f form}
 Let $k\in\NN$. 
 Suppose $\phi:\RR^k\times [k]\mapsto\RR_{\geq 0}$ satisfies Conditions  \ref{assump: N1} and \ref{assump: N2}. Given any $\mp\in\RR^k_{\geq 0}$, if $\Psi^*(\mp)=\langle \mv^*,\mp\rangle $ for some $\mv^*\in\RR^k$, then
    the vector $\mv^*$ satisfies 
   $\mv^*\in\conv(\C_{\text{set}})$, 
    where
     \[\C_{\text{set}}=\{\mathbf{e}^{(i)}_k: i\in\argmax(\mp)\}.\]
     Moreover, if $\Psi(\mx^{(m)};\mp)\to_m  \Psi^*(\mp)$, then
     $\phi(\mxm;i)\to_m 0$ if $i\notin\argmax(\mp)$.
   \end{lemma}

\begin{proof}[Proof of Lemma \ref{lemma: sufficiency: tilde f form}]
The proof is trivial if $\mp=\mz_{k}$. Thus we assume $\mp\neq \mz_k$. 
Without loss of generality, let us assume that $C_\phi=1$ because otherwise we can replace $\phi$ by $\phi/C_\phi$. Recall the image set $\mV$ defined in \eqref{def: image set}. Since $C_\phi=1$, $\Psi(\mx;\mo_k)\leq 1$, which implies $\phi(\mx;i)\leq 1$ for each $i\in[k]$ and 
\[\mV\subset \lbs \mx\in\RR^k_{\geq 0}:\mx_i\in[0,1],\quad \sum_{i\in[k]}\mx_i\leq 1\rbs:=\C.\]
Since $\Phi^*$ is the support function of $\mV$ (see the proof of Lemma \ref{lemma: necessity: Psi-t cont.}), and $\Phi^*(\mp)=\langle \mv^*,\mp\rangle$, it follows that $\mv^*$ lies on a face of the closed convex hull of $\mV$  \citep[cc. pp. 145 of][]{hiriart}. Since $\mV\subset\C$, and $\C$ is both closed and convex, $\overline{\conv(\mV)}\subset\C$.
Therefore, $\mv^*\in \C$, which implies $\mv_i^*\in[0,1]$ for each $i\in[k]$ and $\sum_{i\in[k]}\mv^*_i\leq 1$. Therefore,
\[\langle \mv^*,\mp\rangle\leq \max(\mp)\sum_{i\in[k]}\mv^*_i\leq \max(\mp).\]
Since $\phi$ satisfies Condition \ref{assump: N2}, it follows that $\Psi^*(\mp)=\max(\mp)$, implying
\[\langle \mv^*,\mp\rangle= \max(\mp)\sum_{i\in[k]}\mv^*_i= \max(\mp),\]
implying $\sum_{i\in[k]}\mv^*_i= 1$ and
\[\langle \mv^*,\mp\rangle= \max(\mp)\sum_{i\in[k]}\mv^*_i.\]
Therefore, 
\[\sum_{i\in[k]}(\max(\mp)-\mp_i)\mv_i^*=0,\]
implying $\mv_i^*>0$ only if $\max(\mp)=\mp_i$. Since $\sum_{i\in[k]}\mv^*_i= 1$, it follows that $\mv^*=\sum\limits_{i\in[k]}\mv_i^*\mathbf{e}^{(i)}_k=\sum\limits_{i\in\argmax(\mp)}\mv_i^*\mathbf{e}^{(i)}_k\in\conv(\cset)$.

Let $\mym_i=\phi(\mxm;i)$. We will show that, if $i\notin\argmax(\mp)$, given any subsequence of $\mym_i$, we can find a further subsequence that converges to zero. Therefore, it will follow that $\mym_i\to_m 0$. Since $\phi$ satisfies Condition \ref{assump: N1}, $\phi$ is bounded by Lemma \ref{lemma: necessity: psi bounded}. Therefore, given any subsequence of $\{\mxm\}_{m\geq 1}$, we can find a further subsequence so that the $\mym_i=\phi(\mxm;i)$'s converge (say to some $\mv_i^*\in\RR$) for all $i\in[k]$ along this subsequence. To simplify notation, we will denote the corresponding subsequences by $\mxm$ and $\mym$. Hence, $\mym\to_m \mv^*$. Therefore, it suffices to show that $\mv_i^*=0$ unless $i\in\argmax(\mp)$.

 Note that $\langle \mym,\mp\rangle\to_m  \langle\mv^*,\mp\rangle$ because the $l_2$ inner product  is a continuous function. Noting $\Psi(\mxm;\mp)=\langle \mym,\mp\rangle$ and $\Psi(\mxm;\mp)\to_m\Psi^*(\mp)$, we obtain that $\langle\mv^*,\mp\rangle=\Psi^*(\mp)$.  Therefore, from the first part, it follows that $\mv^*\in\conv(\cset)$, which implies $\mv_i^*=0$ if $i\notin\argmax(\mp)$.
 

\end{proof}

\subsection{Preliminary auxlliary lemmas}
\label{sec: suff: preliminary lemmas}

\begin{lemma}
\label{lemma:  sufficiency: induction: Psi t star p t star}
Suppose $T\geq 2$, $\PP$ satisfies Assumptions I-IV, and $(\phi_2,\ldots,\phi_T)$ satisfies Conditions \ref{assump: N1} and \ref{assump: N2} with $C_{\phi_t}=1$ for $t\in[2:T]$. Then 
    \begin{equation}
        \label{intheorem: sufficiency: induction: Psi t star p t star}
        \Psi_t^*(\mp_t^*(H_t))=\E\lbt\slb\sum_{i=1}^T Y_i\srb\prod_{j=t}^T\frac{1[A_j=d_j^*(H_j)]}{\pi_j(A_j\mid H_j)}\ \bl\  H_t\rbt
    \end{equation}
    and   $\argmax(\mp_t^*(H_t))=\argmax(Q_t^*(H_t))$ for $t\in[2:T]$. 
\end{lemma}

\begin{proof}[Proof of Lemma \ref{lemma:  sufficiency: induction: Psi t star p t star}]
We will show that 
\[\Psi_t^*(\mp_t^*(H_t))=\E\lbt\slb\sum_{i=1}^T Y_i\srb\prod_{j=t}^T\frac{1[A_j=d_j^*(H_j)]}{\pi_j(A_j\mid H_j)}\ \bl\  H_t\rbt\]
for $t\in[2:T]$ using induction. Our induction hypothesis is that the above and $\argmax(\mp_t^*(H_t))=\argmax(Q_t^*(H_t))$ holds for $t\in[2:T]$. 
  Let $t=T$. Note that \eqref{def: sufficiency: pt star} implies   
  $\mp^*_T(H_T)_i=\E\slbt \sum_{i=1}^TY_i\mid H_T, A_T=i\srbt$. Since $Q_T^*(H_T,i)=\E[Y_T\mid H_T, A_T=i]$ for all $i\in[k_T]$ and $\pred(\mx)=\max(\argmax(\mx))$ for any vector $\mx$, it follows that 
     \[\begin{split}
         \pred(\mp_T^*(H_T))=&\ \max(\argmax_{i\in[k_T]}\E[Y_T\mid H_T, A_T=i])=\max(\argmax(Q_T^*(H_T))\\
         =&\ \pred(Q_T^*(H_T))=d_T^*(H_T),
     \end{split}\] 
     where the last step follows because  \eqref{def: d star from Q} defines  $d_t^*(H_t)$ to be $\pred(Q_t^*(H_t))$ for all $t\in[T]$.
   Since $\phi_t$'s satisfy  Condition \ref{assump: N2} for $t\in[2:T]$ with $C_{\phi_t}=1$,  
    \begin{align*}
  \MoveEqLeft \Psi_t^*(\mp_T^*(H_t))= \max(\mp_T^*(H_T))\stackrel{(a)}{=}\mp_T^*(H_T)_{\pred(\mp_T^*(H_T))}\\
   =&\  \E\lbt\sum_{i=1}^TY_i\mid A_T=\pred(\mp_T^*(H_T)), H_T\rbt\stackrel{(b)}{=} \E\lbt\sum_{i=1}^TY_i\mid A_T=d_T^*(H_T), H_T\rbt,
    \end{align*}
    where (a) follows because $\pred(\mx)\in\argmax(\mx)$ for any vector $\mx$ and (b) follows because we just showed that $\pred(\mp_T^*(H_T))=d_T^*(H_T)$.
  Since $\PP$ satisfies Assumption I,  \eqref{intheorem: necessity: IPW: 0-1} yields   
    \[\E\lbt \sum_{i=1}^TY_i\ \bl\ H_T, A_T=d_T^*(H_T)\rbt=\E\lbt\slb\sum_{i=1}^T Y_i\srb\frac{1[A_T=d_T^*(H_T)]}{\pi_T(A_T\mid H_T)}\ \bl\  H_T\rbt,\]
     which implies that the induction hypothesis holds when $t=T$.

    Suppose the induction hypothesis holds for  some $1+t\in[3:T]$. We will show that the induction hypothesis holds for $t$. Since $t\in[2:T]$, $\phi_t$ satisfies  Condition \ref{assump: N2} with $C_{\phi_t}=1$. Therefore,   $\Psi_t^*(\mp_t^*(H_t))= \max(\mp_t^*(H_t))$ where by \eqref{def: sufficiency: pt star}, for any $i\in[k_t]$, 
     \begin{align}
     \label{def: p t star}
 \mp_t^*(H_t)_i &\  =\E[\Psi_{1+t}^*(\mp_{1+t}^*(H_{1+t}))\mid H_t, A_t=i]\nn\\
   =&\ \E\lbt\E\lbt \slb\sum_{i=1}^T Y_i\srb\prod_{j=1+t}^T\frac{1[A_j=d_j^*(H_j)]}{\pi_j(A_j\mid H_j)}\ \bl\ H_{1+t}\rbt\ \bl\ H_t, A_t=i\rbt\nn\\
   =&\ \E\lbt \slb\sum_{i=1}^T Y_i\srb\prod_{j=1+t}^T\frac{1[A_j=d_j^*(H_j)]}{\pi_j(A_j\mid H_j)}\ \bl\ H_t, A_t=i\rbt
    \end{align}
    since the Induction hypothesis holds for $1+t$.
    Thus 
    \begin{align*}
  \Psi_t^*(\mp_t^*(H_t))=&\  \E\lbt \slb\sum_{i=1}^T Y_i\srb\prod_{j=1+t}^T\frac{1[A_j=d_j^*(H_j)]}{\pi_j(A_j\mid H_j)}\ \bl\ H_t, A_t=\pred(\mp_t^*(H_t))\rbt  \\  
  =&\  \E\lbt \slb\sum_{i=1}^T Y_i\srb\frac{1[A_t=\pred(\mp_t^*(H_t))]\prod_{j=1+t}^T 1[A_j=d_j^*(H_j)]}{\prod_{j=t}^T\pi_j(A_j\mid H_j)}\ \bl\ H_t\rbt,
    \end{align*}
    where the last step follows from  \eqref{intheorem: necessity: IPW: 0-1}. Thus to show that the induction hypothesis follows for $t$, it only remains to show that $\pred(\mp_t^*(H_t))=d_t^*(H_t)$.
To show the latter, we will first prove a fact. Since $(Y_1,\ldots,Y_{t-1})\subset H_t$, 
  \begin{align}
  \label{intheorem: suff: Psi t star form argmax step}
      \MoveEqLeft \E\lbt \slb\sum_{i=1}^T Y_i\srb\prod_{j=1+t}^T\frac{1[A_j=d_j^*(H_j)]}{\pi_j(A_j\mid H_j)}\ \bl\ H_t, A_t=j\rbt\nn\\
      =&\ \slb\sum_{i=1}^{t-1} Y_i\srb\E\lbt \prod_{j=1+t}^T\frac{1[A_j=d_j^*(H_j)]}{\pi_j(A_j\mid H_j)}\ \bl\ H_t, A_t=i\rbt\nn\\
      &\ +\E\lbt \slb\sum_{i=t}^T Y_i\srb\prod_{j=1+t}^T\frac{1[A_j=d_j^*(H_j)]}{\pi_j(A_j\mid H_j)}\ \bl\ H_t, A_t=j\rbt\nn\\
      =&\ \sum_{i=1}^{t-1} Y_i+ \E\lbt \slb\sum_{i=t}^T Y_i\srb\prod_{j=1+t}^T\frac{1[A_j=d_j^*(H_j)]}{\pi_j(A_j\mid H_j)}\ \bl\ H_t, A_t=j\rbt,
  \end{align}  
  where the last step follows from \eqref{intheorem: fact: product of indicators}. 
    Since $\pred(\mp)=\max(\argmax(\mp))$ for any vector $\mp$, \eqref{def: p t star} implies $\pred(\mp_t^*(H_t))$ equals
\begin{align*}
    \MoveEqLeft \max\lbs \argmax_{j\in[k_t]}\E\lbt \slb\sum_{i=1}^T Y_i\srb\prod_{j=1+t}^T\frac{1[A_j=d_j^*(H_j)]}{\pi_j(A_j\mid H_j)}\ \bl\ H_t, A_t=j\rbt\rbs\\
    \stackrel{(a)}{=}&\ \max\lbs \argmax_{j\in[k_t]}\E\lbt \slb\sum_{i=t}^T Y_i\srb\prod_{j=1+t}^T\frac{1[A_j=d_j^*(H_j)]}{\pi_j(A_j\mid H_j)}\ \bl\ H_t, A_t=j\rbt\rbs,
\end{align*}
    where (a) follows due to \eqref{intheorem: suff: Psi t star form argmax step}.
      However, Fact \ref{fact: Q function expression} implies
    \[\E\lbt \slb\sum_{i=t}^T Y_i\srb\prod_{j=1+t}^T\frac{1[A_j=d_j^*(H_j)]}{\pi_j(A_j\mid H_j)}\ \bl\ H_t, A_t=i\rbt=Q_t^*(H_t,i),\]
    implying $\pred(\mp_t^*(H_t))=\pred(Q_t^*(H_t))=d_t^*(H_t)$.
    Therefore, the proof follows.
\end{proof}

\begin{lemma}
\label{lemma: suff: p t star is bdd}
  Under the setup of Lemma \ref{lemma:  sufficiency: induction: Psi t star p t star}, there exist $c_1,c_2>0$, depending only on $\PP$, so that the $\mp_t^*:\H_t\mapsto\RR$ defined in \eqref{def: sufficiency: pt star} satisfies $c_1<\mp_t^*(h_t)_i<c_2$ for all $h_t\in\H_t$, $i\in[k_t]$, and  $t\in[T]$. 
\end{lemma}

\begin{proof}[Proof of Lemma \ref{lemma: suff: p t star is bdd}]
We will use \eqref{def: p t star}. By our assumption, there is a constant depending on $\PP$, i.e., $C_{min}(\PP)$, so that $Y_t>C_{min}(\PP)$ for each $t\in[T]$. Hence,   \eqref{def: p t star} implies for each $i\in[k_t]$, 
\[\mp_t^*(H_t)_i\geq TC_{min}(\PP)\E\lbt \prod_{j=1+t}^T\frac{1[A_j=d_j^*(H_j)]}{\pi_j(A_j\mid H_j)}\rbt,\]
which equals $TC_{min}(\PP)$ by 
\eqref{intheorem: fact: product of indicators}. Moreover, Assumption IV implies that there exists a constant depending on $\mp$, say $C_{max}(\PP)$, so that $Y_j\leq C_{max}(\PP)$ for all $j\in[T]$. Hence, 
\[\mp_t^*(H_t)_i\leq TC_{max}(\PP)\E\lbt \prod_{j=1+t}^T\frac{1[A_j=d_j^*(H_j)]}{\pi_j(A_j\mid H_j)}\rbt=TC_{max}(\PP)\]
by the same argument. Hence, the proof follows.
\end{proof}

\begin{lemma}
\label{lemma: sufficiency: p t star and Q t star}
Suppose $\PP$ satisfies Assumptions I-V, and if $T\geq 2$, then for $t\in[2:T]$, $\phi_t$ satisfies Conditions \ref{assump: N1} and \ref{assump: N2}  with $C_{\phi_t}=1$. Then $\mp_1^*(H_1)= Q_1^*(H_1)$ and $V_*=\E[\max(\mp_1^*(H_1))]$.
\end{lemma}

\begin{proof}[Proof of Lemma \ref{lemma: sufficiency: p t star and Q t star}]
If $T=1$, then the proof follows trivially because in this case, $p_1^*(H_1)_i=p_1(H_1)_i=E[Y_1\mid H_1, A_1=i]=Q_1^*(H_1,i)$ for all $i\in[k_1]$. Hence, we assume that $T\geq 2$. 
For $t\in[2:T]$, $C_{\phi_t}=1$ implies $\Psi^*_t(\mo_{k_t})=1$, which indicates Lemmas \ref{lemma: necessity: Bartlett stuff}--\ref{lemma: necessity: linear bound} hold.
    By \eqref{def: sufficiency: pt star} and \eqref{intheorem: sufficiency: induction: Psi t star p t star}, for any $i\in[k_1]$, 
    \begin{align*}
        \mp_1^*(H_1)_i=&\ \E[\Psi_2^*(p_2^*(H_2))\mid H_1,A_1=i]\\
        =&\ \E\lbt \E\lbt\slb\sum_{i=1}^T Y_i\srb\prod_{j=2}^T\frac{1[A_j=d_j^*(H_j)]}{\pi_j(A_j\mid H_j)}\ \bl\ H_2\rbt\ \bl\ H_1,A_1=i\rbt
    \end{align*}
    where the last step follows by Lemma \ref{lemma:  sufficiency: induction: Psi t star p t star}.
   Using Fact \ref{fact: Q function expression} in step (a), we obtain that
    \begin{align*}
    \mp_1^*(H_1)_i=\E\lbt\slb\sum_{i=1}^T Y_i\srb\prod_{j=2}^T\frac{1[A_j=d_j^*(H_j)]}{\pi_j(A_j\mid H_j)}\ \bl\ H_1,A_1=i\rbt \stackrel{(a)}{=} Q_1^*(H_1,i). 
    \end{align*}
    That $V_*=\E[\max(\mp_1^*(H_1))]$ follows noting $V_*=\E[\max_{i\in [k_1]} Q_1^*(H_1,i)]$.
\end{proof}

 \begin{lemma}
        \label{lemma: sufficiency}
         Suppose $\PP$ satisfies Assumptions I-V. Then  $V^\psi_*=\E[\Psi_1^*(\mp_1^*(H_1))]$, and for any $f\in\F$, $V^\psi(f)=\E[\Psi(f_1(H_1);\mp_1(H_1))]$ where $\mp_1(H_1)\equiv \mp_1(H_1;f)$.  In the special case where $\phi_1$ also satisfies Condition  \ref{assump: N2} with $C_{\phi_1}=1$, then $V_*^\psi=V_*$.
          
  \end{lemma}
    
\begin{proof}[Proof of Lemma \ref{lemma: sufficiency}]
We will prove $V^\psi_*=\E[\Psi_1^*(\mp_1^*(H_1))]$  by induction. Since the proof of $V^\psi(f)=\E[\Psi(f_1(H_1);\mp_1(H_1)]$ is similar, it is skipped. \\
\textbf{Induction hypothesis:} For $t\in[T]$,
\begin{align*}
     \MoveEqLeft  \sup_{f_{i}\in\F_i:i\in [t:T]}\E\lbt \slb\sum_{i=1}^T Y_i\srb\prod_{j=1}^{T}\frac{\phi_j(f_j(H_j);A_j)}{\pi_j(A_j\mid H_j)}\rbt= \E\lbt \prod_{j=1}^{t-1}\frac{\phi_{j}(f_j(H_j);A_t)}{\pi_j(A_j\mid H_j)}\Psi_t^*(\mp_t^*(H_t))\rbt,
    \end{align*}
    where the product term is one when the range of the product is empty, which occurs only when $t=1$.
    When $t=T$, taking $\VV=\sum_{i=t}^TY_t$ and $h(H_{T})=\prod_{j=1}^{T-1}\{\phi_j(f_j(H_j);A_j)/ \pi_j(A_j\mid H_j)\}$ in \eqref{intheorem: necessity: IPW: supremum with other rv}, we obtain that 
    \begin{align*}
 \sup_{f_T\in\F_T}\E\lbt \slb\sum_{i=1}^T Y_i\srb\prod_{j=1}^{T}\frac{\phi_j(f_j(H_j);A_j)}{\pi_j(A_j\mid H_j)}\rbt
        =&\ \E\lbt \prod_{j=1}^{T-1}\frac{\phi_j(f_j(H_j);A_j)}{\pi_j(A_j\mid H_j)}\Psi_T^*(\mp_T^*(H_T))\rbt.
    \end{align*}
    Thus the induction hypothesis holds for $t=T$. Suppose it holds for  some $t\in[2:T]$. We will show that then the induction hypothesis holds for $t-1$.  Note that $t-1\in[T-1]$ and  
    \begin{align*}
     \MoveEqLeft  \sup_{f_i:i\in[t-1:T]}\E\lbt \slb\sum_{i=1}^T Y_i\srb\prod_{j=1}^{T}\frac{\phi_j(f_j(H_j);A_j)}{\pi_j(A_j\mid H_j)}\rbt\\
     \stackrel{(a)}{=} &\ \sup_{f_{t-1}}\E\lbt \prod_{j=1}^{t-1}\frac{\phi_j(f_j(H_j);A_j)}{\pi_j(A_j\mid H_j)}\Psi_t^*(\mp_t^*(H_t))\rbt\\
    = &\ \sup_{f_{t-1}\in\F_{t-1}}\E\lbt \prod_{j=1}^{t-2}\frac{\phi_j(f_T(H_t);A_t)}{\pi_j(A_j\mid H_j)}\lb \Psi_t^*(\mp_t^*(H_t))\frac{\phi_{t-1}(f_{t-1}(H_{t-1});A_{t-1})}{\pi_{t-1}(A_{t-1}\mid H_{t-1})}\rb\rbt.
    \end{align*}
Now we apply \eqref{intheorem: necessity: IPW: supremum with other rv} on the right hand side of the above display at time stage  $ t-1$ with  $\VV=\Psi_t^*(\mp_t^*(H_t))$ and  $h(H_{t})=\prod_{j=1}^{t-2}\{\phi_j(f_j(H_j);A_j)/ \pi_j(A_j\mid H_j)\}$. Noting that 
 \[\E[\VV\mid H_{t-1}, A_{t-1}=i]=\E[\Psi_t^*(\mp_t^*(H_t))\mid  H_{t-1}, A_{t-1}=i]=\mp_{t-1}^*(H_{t-1})_i\]
  by \eqref{def: sufficiency: pt star},  and applying \eqref{intheorem: necessity: IPW: supremum with other rv}, we obtain that
    \begin{align*}
   \MoveEqLeft \sup_{f_{t-1}\in\F_{t-1}}\E\lbt \prod_{j=1}^{t-2}\frac{\phi_j(f_T(H_t);A_t)}{\pi_j(A_j\mid H_j)}\lb \Psi_t^*(\mp_t^*(H_t))\frac{\phi_{t-1}(f_{t-1}(H_{t-1});A_{t-1})}{\pi_{t-1}(A_{t-1}\mid H_{t-1})}\rb\rbt\\
    &\ = 
   \sup_{f_{t-1}\in\F_{t-1}}\E\lbt \prod_{j=1}^{t-2}\frac{\phi_j(f_T(H_t);A_t)}{\pi_j(A_j\mid H_j)}\Psi_{t-1}^*(p_{t-1}^*(H_{t-1}))\rbt,
    \end{align*}
    which implies that the induction hypothesis holds for $t-1$. Therefore, the induction hypothesis holds for all $t\in[1:T]$. Thus, setting $t=1$, we obtain that
    \[V_*^\psi=\sup_{f_t\in\F_t:t\in[T]}\E\lbt \slb\sum_{i=1}^T Y_i\srb\prod_{j=1}^{T}\frac{\phi_j(f_j(H_j);A_j)}{\pi_j(A_j\mid H_j)}\rbt=\E[\Psi_1^*(\mp_1^*(H_1))].\]
   Finally, in the special case, $\Psi_1^*(\mp_1^*(H_1))=\max(\mp_1^*(H_1))$ by Assumption \ref{assump: N2}. Hence, $V_*^\psi=\E[\max(\mp_1^*(H_1))]$, which is $V_*$ by Lemma \ref{lemma: sufficiency: p t star and Q t star}.
\end{proof}



\paragraph{Auxiliary lemmas for induction steps}
Lemmas \ref{lemma: sufficiency: Induction: non-negative} and \ref{lemma: sufficiency: Induction: main step} are used during the induction steps in the proofs of Theorem \ref{theorem: sufficient conditions} and \ref{theorem: necessity}, respectively. 
\begin{lemma}
\label{lemma: sufficiency: Induction: non-negative}
Suppose $\PP$ satisfies Assumptions I-V, $\phi_1$ satisfies Condition \ref{assump: N1} and  for $t\in[2:T]$, $\phi_t$ satisfies Conditions \ref{assump: N1} and \ref{assump: N2}  with $C_{\phi_t}=1$. Let
\begin{align*}
    Z_{\text{diff,t}}=&\ \frac{\slb\sum_{i=1}^TY_i\srb 1[A_{t}=\pred(\mx)]}{\prod_{r=t}^T\pi_r(A_r\mid H_r)}
\prod_{r=t+1}^{T-1}\phi_r\slb f_r(H_r);\pred(f_r(H_r))\srb\\
&\ \times \left\{ \prod_{r=1+t}^T1[A_r=d_r^*(H_r)] -\prod_{r=1+t}^T1[A_r=\pred(f_r(H_r))]\right\},
\end{align*}
where the product terms in the expectation are one if the range is empty.
Then 
  for $t=1,\ldots,T$, for all $f\in\F$, and for any $\mx\in\RR^{k_t}$, $\E\lbtt Z_{\text{diff,t}}\mid\ H_t\rbtt$ is  non-negative and
\begin{align}
\label{instatement: main induction step}
\MoveEqLeft \Psi_{t}(\mx;\mp_{t}^*(H_{t}))-\Psi_{t}(\mx;\mp_{t}(H_{t}))
\geq  \slb \min_{1+t\leq i\leq T}\CC_{\phi_i}\srb \phi_t(\mx;\pred(\mx))\E\lbtt Z_{\text{diff,t}}\mid\ H_t\rbtt, 
\end{align}
where the $\CC_{\phi_t}$'s are as in Lemma \ref{lemma: necessity: linear bound} and $p_t(H_t)\equiv p_t(H_t;f)$. 
\end{lemma}
\begin{proof}[Proof of Lemma \ref{lemma: sufficiency: Induction: non-negative}]
First of all, the assumptions imply  Lemma \ref{lemma: necessity: linear bound} is applicable to $\phi_2,\ldots,\phi_T$.
When $t=T$, we can show that the LHS of \eqref{instatement: main induction step} is zero for all $\mx\in\RR^{k_t}$ because $\mp_T^*(h_T)=\mp_T(h_T)$ for all $h_T\in\H_T$ by definition. Also, the RHS of \eqref{instatement: main induction step} is zero in this case because the range $[1+T,T]$ is empty.  Thus,\eqref{instatement: main induction step} trivially follows for $t=T$.  Therefore, it suffices to consider the case $t\in[T-1]$, which implies we also consider $T\geq 2$.

We will prove the lemma by Induction. Our induction hypothesis is that \eqref{instatement: main induction step} holds and the RHS is non-negative  for $t\in[T-1]$. First, we will  show that \eqref{instatement: main induction step} holds for $t=T-1$. Then we will show that if \eqref{instatement: main induction step} holds for $t+1\in[2:T-1]$, then it holds for $t$ as well. When $t=T-1$, the LHS of \eqref{instatement: main induction step} equals 
\begin{align*}
  \MoveEqLeft \Psi_{T-1}(\mx;\mp_{T-1}^*(H_{T-1}))-\Psi_{T-1}(\mx;\mp_{T-1}(H_{T-1})) \\
  =&\ \sum_{i=1}^{k_{T-1}}\phi_{T-1}(\mx; i)\slb \mp^*_{T-1}(H_{T-1})_i-\mp_{T-1}(H_{T-1})_i\srb\\
  =&\ \sum_{i=1}^{k_{T-1}}\phi_{T-1}(\mx; i)\slb \E[\Psi_T^*(\mp_T^*(H_T))-\Psi_T(f_T(H_T);\mp_T(H_T))\mid H_{T-1}, A_{T-1}=i]\srb
\end{align*}
by \eqref{def: sufficiency: pt star}.
Since $\mp_T^*=\mp_T$, it follows that the above is non-negative.  Lemma \ref{lemma: necessity: linear bound} a applies to $\phi_T$ because $T\geq 2$. Applying Lemma \ref{lemma: necessity: linear bound}, we obtain that there exists $\CC_{\phi_T}>0$, depending only on $\phi_t$, so that 
\begin{align*}
  &\   \sum_{i=1}^{k_{T-1}}\phi_{T-1}(\mx; i)\slb \E[\Psi_t^*(\mp_T^*(H_T))-\Psi_t(f_T(H_T);\mp_T^*(H_T))\mid H_{T-1}, A_{T-1}=i]\srb\\
    \geq &\ \CC_{\phi_T}\sum_{i=1}^{k_{T-1}}\phi_{T-1}(\mx; i) \E[\max(\mp_T^*(H_T))-\myq(f_T(H_T);\mp_T(H_T))\mid H_{T-1}, A_{T-1}=i]
    \end{align*}
    which is bounded below by
    \[\CC_{\phi_T}\phi_{T-1}(\mx; \pred(\mx))\E[\max(\mp_T^*(H_T))-\myq(f_T(H_T);\mp_T(H_T))\mid H_{T-1}, A_{T-1}=\pred(\mx)],\]
which is non-negative because $\mp_T=\mp_T^*$. However,
\begin{align}
\label{intheorem: sufficiency: induction main step: non-negative T-1}
&\  \E[\max(\mp_T^*(H_T))-\myq(f_T(H_T);\mp_T(H_T))\mid H_{T-1}, A_{T-1}=\pred(\mx)]\\
  \stackrel{(a)}{=}&\ \E\lbt\slb \sum_{i=1}^TY_i\srb\frac{1[A_T=\pred(\mp_T^*(H_T))]-1[A_T=\pred(f_T(H_T))]}{\pi_T(A_T\mid H_T)}\bl H_{T-1}, A_{T-1}=\pred(\mx) \rbt\nn\\
  \stackrel{(b)}{=}&\ \E\lbt\slb\sum_{i=1}^TY_i\srb 1[ A_{T-1}=\pred(\mx)]\frac{1[A_T=\pred(\mp_T^*(H_T))]-1[A_T=\pred(f_T(H_T))]}{\prod_{t=T-1}^{T}\pi_t(A_t\mid H_t)}\bl H_{T-1} \rbt\nn\\
  \stackrel{(c)}{=}&\ \E\lbt\slb\sum_{i=1}^TY_i\srb 1[ A_{T-1}=\pred(\mx)]\frac{1[A_T=d_T^*(H_T)]-1[A_T=\pred(f_T(H_T))]}{\prod_{t=T-1}^{T}\pi_t(A_t\mid H_t)} \bl H_{T-1} \rbt \nn
\end{align}
where (a) follows from \eqref{intheorem: necessity: IPW: diff} with $i=\pred(f_T(H_T))$, (b) follows from an application of \eqref{intheorem: necessity: IPW: 0-1}, and (c) follows because $\pred(\mp_T^*(H_T))=d_T^*(H_T)$ by Lemma \ref{lemma: sufficiency: p t star and Q t star}. 
Since 
\[\max(\mp_T^*(H_T))\geq \myq(f_T(H_T);\mp^*_T(H_T))=\myq(f_T(H_T);\mp_T(H_T)),\]
\eqref{intheorem: sufficiency: induction main step: non-negative T-1} also implies that
\[ \E\lbt\slb\sum_{i=1}^TY_i\srb 1[ A_{T-1}=\pred(\mx)]\times\frac{1[A_T=d_T^*(H_T)]-1[A_T=\pred(f_T(H_T))]}{\prod_{t=T-1}^{T}\pi_t(A_t\mid H_t)}\ \bl\ H_{T-1} \rbt,\]
which is non-negative. 
Since $\CC_{\phi_T},\phi_{T-1}\geq 0$, it also follows that 
\begin{align*}
 &\ \Psi_{T-1}(\mx;\mp_{T-1}^*(H_{T-1}))-\Psi_{T-1}(\mx;\mp_{T-1}(H_{T-1}))  
   \geq \CC_{\phi_T}\phi_{T-1}(\mx;\pred(\mx))\\
   \times &\ \E\lbt\slb\sum_{i=1}^TY_i\srb 1[ A_{T-1}=\pred(\mx)]\frac{1[A_T=d_T^*(H_T)]-1[A_T=\pred(f_T(H_T))]}{\prod_{t=T-1}^{T}\pi_t(A_t\mid H_t)} \bl H_{T-1} \rbt,
\end{align*}
which is non-negative, implying that the induction hypothesis holds  for $t=T-1$. If $T=2$, then there is nothing to prove. Suppose $T\geq 3$ and that the induction hypothesis is true for $t+1$ where $t\in[1, T-2]$. We will show that the induction hypothesis is true for $t$, which will complete the proof of the current lemma.
Note that
\begin{align}
\label{inlemma: sufficiency: main induction step: t 1st deduction}
 &\ \Psi_{t}(\mx;\mp_{t}^*(H_{t}))-\Psi_{t}(\mx;\mp_{t}(H_{t}))\nn \\
  =&\ \sum_{i=1}^{k_{t}}\phi_t(\mx; i)\slb \mp^*_{t}(H_{t})_i-\mp_{t}(H_{t})_i\srb\\
  =&\ \sum_{i=1}^{k_{t}}\phi_t(\mx; i)\slb \E[\Psi_{1+t}^*(\mp_{1+t}^*(H_{1+t}))-\Psi_{1+t}(f_{1+t}(H_{1+t});\mp_{1+t}(H_{1+t}))\mid H_{t}, A_{t}=i]\srb\nn
\end{align}
by \eqref{def: sufficiency: pt star} since $t\in[ T-2]$. However, the above equals
\begin{align}
\label{inlemma: sufficiency: main induction step: t 2nd deduction}
 & \underbrace{\sum_{i=1}^{k_{t}}\phi_t(\mx; i)\slb \E[\Psi_{1+t}^*(\mp_{1+t}^*(H_{1+t}))-\Psi_{1+t}(f_{1+t}(H_{1+t});\mp^*_{1+t}(H_{1+t}))\mid H_{t}, A_{t}=i]\srb}_{S_1}\nn\\
     + & \underbrace{\sum_{i=1}^{k_{t}}\phi_t(\mx; i)\slb \E[\Psi_{1+t}(f_{1+t}(H_{1+t});\mp^*_{1+t}(H_{1+t}))-\Psi_{1+t}(f_{1+t}(H_{1+t});\mp_{1+t}(H_{1+t}))\mid H_{t}, A_{t}=i]\srb}_{S_2}.
\end{align}
 Since $1+t\in[2:T-1]$,  Lemma \ref{lemma: necessity: linear bound} can be applied to $\phi_{1+t}$, yielding
\begin{align}
\label{intheorem: sufficiency: induction: T1 and T2}
    S_1\geq &\  \CC_{\phi_{1+t}} \sum_{i=1}^{k_{t}}\phi_t(\mx; i)\E[\max(\mp_{1+t}^*(H_{1+t}))-\myq(f_{1+t}(H_{1+t});\mp^*_{1+t}(H_{1+t}))\mid H_{t}, A_{t}=i]\nn\\
    \geq &\ \CC_{\phi_{1+t}}  \phi_t(\mx; \pred(\mx))\E[\max(\mp_{1+t}^*(H_{1+t}))-\myq(f_{1+t}(H_{1+t});\mp^*_{1+t}(H_{1+t}))\mid H_{t}, A_{t}=\pred(\mx)]\nn\\
    \geq &\ \slb\min_{1+t\leq i\leq T}\CC_{\phi_i} \srb\phi_t(\mx; \pred(\mx))\\
    &\ \times\E\slbt\max(\mp_{1+t}^*(H_{1+t}))-\myq(f_{1+t}(H_{1+t});\mp^*_{1+t}(H_{1+t}))\mid H_{t}, A_{t}=\pred(\mx)\srbt.\nn
\end{align}
Since 
\begin{align}
    \label{inlemma: sufficiency: mp 1+t}
    \mp_{1+t}^*(H_{1+t})_i=\E[\Psi_{2+t}^*(\mp_{2+t}^*(H_{2+t}))\mid H_{1+t},A_{1+t}=i]
\end{align}
 for all $i\in[k_{1+t}]$, it follows that 
 \begin{align}
 \label{intheorem: sufficiency: induction: T1 reduction}
    \MoveEqLeft \E[\max(\mp_{1+t}^*(H_{1+t}))-\myq(f_{1+t}(H_{1+t});\mp^*_{1+t}(H_{1+t}))\mid H_{t}, A_{t}=\pred(\mx)]\nn\\
    =&\ \E\lbt \E\lbt \max(\mp_{1+t}^*(H_{1+t}))-\myq(f_{1+t}(H_{1+t});\mp^*_{1+t}(H_{1+t}))\ \bl\ H_{1+t}\rbt\ \bl\  H_t, A_t=\pred(\mx)\rbt\nn\\
    \stackrel{(a)}{=}&\ \E\lbt \E\lbt\frac{ 1[A_{1+t}=\pred(\mp_{1+t}^*(H_{1+t}))]-1[A_{1+t}=\pred(f_{1+t}(H_{1+t}))]}{\pi_{1+t}(A_{1+t}\mid H_{1+t})}\nn\\
    &\ \times \Psi_{2+t}^*(\mp_{2+t}^*(H_{2+t}))\ \bl\ H_{1+t}\rbt\ \bl\ H_t, A_t=\pred(\mx)\rbt
\end{align}
where (a) uses \eqref{intheorem: necessity: IPW: diff} on $H_{1+t}$ with $i=\pred(f_{1+t}(H_{1+t}))$ and $\mp(H_{1+t})=\mp_{1+t}^*(H_{1+t})$, whose expression is given by \eqref{inlemma: sufficiency: mp 1+t}. Also, note that
\begin{align}
\label{inlemma: hotchpotch}
  \MoveEqLeft  \E\lbt\frac{ 1[A_{1+t}=\pred(\mp_{1+t}^*(H_{1+t}))]-1[A_{1+t}=\pred(f_{1+t}(H_{1+t}))]}{\pi_{1+t}(A_{1+t}\mid H_{1+t})}\Psi_{2+t}^*(\mp_{2+t}^*(H_{2+t}))\ \bl\ H_{1+t}\rbt\nn\\
  =&\ \E\lbt \max(\mp_{1+t}^*(H_{1+t}))-\myq(f_{1+t}(H_{1+t});\mp^*_{1+t}(H_{1+t}))\bl H_{1+t}\rbt
\end{align}
is non-negative because $\max(\mp_{1+t}^*(H_{1+t}))\geq \myq(f_{1+t}(H_{1+t});\mp^*_{1+t}(H_{1+t}))$. Since $C_{\phi_t}=\Psi_t^*(\mo_{k_t})$ and $C_{\phi_t}= 1$ for all $t\in[2:T]$, it follows that $\phi_t(\mx;\pred(\mx))\leq 1$ for each $\mx\in\RR^{k_t}$ for all $t\in[2:T]$. In particular
\[0\leq \prod_{r=t+1}^{T-1}\phi_r\slb f_r(H_r);\pred(f_r(H_r))\srb\leq 1\text{ for all }t\geq 1.\]
Therefore, denoting
\begin{align}
\label{def: S11}
    S_{11}=&\ \E\lbt \prod_{r=t+1}^{T-1}\phi_r\slb f_r(H_r);\pred(f_r(H_r))\srb\nn\\
    &\ \times\E\lbt\frac{ 1[A_{1+t}=\pred(\mp_{1+t}^*(H_{1+t}))]-1[A_{1+t}=\pred(f_{1+t}(H_{1+t}))]}{\pi_{1+t}(A_{1+t}\mid H_{1+t})}\nn\\
    &\  \times \Psi_{2+t}^*(\mp_{2+t}^*(H_{2+t}))\ \bl\ H_{1+t}\rbt\ \bl\ H_t, A_t=\pred(\mx)\rbt,
\end{align}
and using \eqref{intheorem: sufficiency: induction: T1 reduction}, we obtain that
\begin{align}
\label{intheorem: sufficiency: induction: T1 reduction 2}
     \MoveEqLeft \E[\max(\mp_{1+t}^*(H_{1+t}))-\myq(f_{1+t}(H_{1+t});\mp^*_{1+t}(H_{1+t}))\mid H_{t}, A_{t}=\pred(\mx)]
     \geq  S_{11}.
\end{align}
 Since $0\leq \prod_{r=t+1}^{T-1}\phi_r\slb f_r(H_r);\pred(f_r(H_r))\srb$, from \eqref{inlemma: hotchpotch} it follows that $S_{11}\geq 0$.
Also,  Lemma \ref{lemma:  sufficiency: induction: Psi t star p t star} implies
\[\Psi_{2+t}^*(\mp_{2+t}^*(H_{2+t}))=\E\lbt \sum_{i=1}^TY_i \prod_{j=t+2}^T\frac{1[A_j=d_j^*(H_j)]}{\pi_j(A_j\mid H_j)}\ \bl\ H_{2+t}\rbt,\]
where $H_{2+t}$ is non-empty because the case under consideration has   $t\in[T-2]$. Therefore, 
\begin{align*}
 S_{11}=&\   \E\lbt \prod_{r=t+1}^{T-1}\phi_r\slb f_r(H_r);\pred(f_r(H_r))\srb\\
 &\ \times\frac{ 1[A_{1+t}=\pred(\mp_{1+t}^*(H_{1+t}))]-1[A_{1+t}=\pred(f_{1+t}(H_{1+t}))]}{\pi_{1+t}(A_{1+t}\mid H_{1+t})}\nn\\
    &\  \times \ \slb\sum_{i=1}^TY_i \srb\prod_{j=t+2}^T\frac{1[A_j=d_j^*(H_j)]}{\pi_j(A_j\mid H_j)} \bl\ H_t, A_t=\pred(\mx)\rbt.
\end{align*}
An application of \eqref{intheorem: necessity: IPW: 0-1}  yields that 
\begin{align*}
  S_{11}=&\ \E\lbt \prod_{r=t+1}^{T-1}\phi_r\slb f_r(H_r);\pred(f_r(H_r))\srb\\
  &\ \times\frac{ 1[A_{1+t}=\pred(\mp_{1+t}^*(H_{1+t}))]-1[A_{1+t}=\pred(f_{1+t}(H_{1+t}))]}{\pi_{1+t}(A_{1+t}\mid H_{1+t})}\nn\\
    &\  \times \ \slb\sum_{i=1}^TY_i \srb\prod_{j=t+2}^T\frac{1[A_j=d_j^*(H_j)]}{\pi_j(A_j\mid H_j)} \frac{1[A_t=\pred(\mx)]}{\pi_t(A_t\mid H_t)}\bl\ H_t\rbt.
\end{align*}
Combining the above with \eqref{intheorem: sufficiency: induction: T1 and T2} and \eqref{intheorem: sufficiency: induction: T1 reduction 2}, and using the fact that $\phi_t$'s  and $\CC_{\phi_t}$'s are non-negative, we obtain the following lower bound for $S_1:$
\begin{align}
\label{intheorem: sufficiency: induction: T1 lower bound}
    S_1\geq &\  \slb \min_{1+t\leq i\leq T}\CC_{\phi_i}\srb\phi_t(\mx;\pred(\mx)) S_{11}\nn\\
   =  &\ \slb \min_{1+t\leq i\leq T}\CC_{\phi_i}\srb\phi_t(\mx;\pred(\mx))\E\lbt  \slb\sum_{i=1}^TY_i \srb \prod_{r=t+1}^{T-1}\phi_r\slb f_r(H_r);\pred(f_r(H_r))\srb\\
    \times &\ \frac{\prod_{j=t+2}^T 1[A_j=d_j^*(H_j)]}{\prod_{j=t}^T\pi_j(A_j\mid H_j)}\nn\\
    \times &\ \slb 1[A_{1+t}=d_{1+t}^*(H_{1+t})]-1[A_{1+t}=\pred(f_{1+t}(H_{1+t}))]\srb 1[A_t=\pred(\mx)]\ \bl\ H_t\rbt\nn
\end{align}
Since we already argued that $S_{11}\geq 0$, the RHS of \eqref{intheorem: sufficiency: induction: T1 lower bound} is non-negative. 
Now we will bound $S_2$. 
By the induction hypothesis, each element of the sum $S_2$ is non-negative. Therefore, $S_2$ is bounded below by $\phi_t(\mx; \pred(\mx))$ times
\begin{align*}
 & \slb \E[\Psi_{1+t}^*(\mp_{1+t}^*(H_{1+t}))-\Psi_{1+t}(f_{1+t}(H_{1+t});\mp_{1+t}(H_{1+t}))\mid H_{t}, A_{t}=\pred(\mx)]\srb\\
 & =\E\lbt \slb\Psi_{1+t}^*(\mp_{1+t}^*(H_{1+t}))-\Psi_{1+t}(f_{1+t}(H_{1+t});\mp_{1+t}(H_{1+t}))\srb\frac{1[A_t=\pred(\mx)]}{\pi_t(A_t\mid H_t)}\ \bl\ H_t\rbt
\end{align*}
where the last step follows from  \eqref{intheorem: necessity: IPW: 0-1}.
Since the Induction hypothesis holds for $1+t$, it holds that 
\begin{align}
\label{inlemma: suff: s2 bound nonnegative}
 \MoveEqLeft \Psi_{1+t}(f_{1+t}(H_{1+t});\mp_{1+t}^*(H_{1+t}))-\Psi_{1+t}(f_{1+t}(H_{1+t});\mp_{1+t}(H_{1+t}))\nn\\
 \geq &\ \slb \min_{2+t\leq i\leq T}\CC_{\phi_i}\srb\phi_{1+t}\slb f_{1+t}(H_{1+t});\pred(f_{1+t}(H_{1+t}))\srb\nn\\
 &\ \times \E\lbt \lb\prod_{r=t+2}^{T-1}\phi_{r}\slb f_{r}(H_{r})); \pred(f_{r}(H_{r}))\srb\rb \\
  &\ \times \slb\sum_{i=1}^TY_i\srb\times  1[A_{1+t}=\pred(f_{1+t}(H_{1+t})]\\
  \times &\ \frac{\prod_{r=t+2}^T1[A_r=d_r^*(H_r)]-\prod_{r=t+2}^T1[A_r=\pred(f_r(H_r)]}{\prod_{r=t+1}^T\pi_r(A_r\mid H_r)}\ \bl\ H_{1+t}\rbt,\nn
\end{align}
which is also non-negative by the induction hypothesis.  
Since the $\CC_{\phi_t}$'s and the $\phi_t$'s are non-negative, we obtain that  
\begin{align*}
  S_2 \geq &\ \slb \min_{2+t\leq i\leq T}\CC_{\phi_i}\srb\phi_t(\mx; \pred(\mx))\E\lbt \lb\prod_{r=t+1}^{T-1}\phi_{r}\slb f_{r}(H_{r})); \pred(f_{r}(H_{r}))\srb\rb \\
  &\ \times \frac{1[A_t=\pred(\mx)]}{\pi_t(A_t\mid H_t)}\slb\sum_{i=1}^TY_i\srb\times  1[A_{1+t}=\pred(f_{1+t}(H_{1+t})]\\
  \times &\ \frac{\prod_{r=t+2}^T1[A_r=d_r^*(H_r)]-\prod_{r=t+2}^T1[A_r=\pred(f_r(H_r))]}{\prod_{r=t+1}^T\pi_r(A_r\mid H_r)}\ \bl\ H_t\rbt,
\end{align*}
which is also non-negative. 
Adding the above inequality with the inequality in  \eqref{intheorem: sufficiency: induction: T1 lower bound}, we obtain that 
\begin{align*}
S_1+S_2\geq  \slb \min_{1+t\leq i\leq T}\CC_{\phi_i}\srb\phi_t(\mx;\pred(\mx))\E\lbtt Z_{\text{diff,t}}\mid\ H_t\rbtt
\end{align*}
Since we have shown that  $S_{11}$ (the RHS of \eqref{intheorem: sufficiency: induction: T1 lower bound}) and the RHS of \eqref{inlemma: suff: s2 bound nonnegative} are non-negative, it also follows that $\E\lbtt Z_{\text{diff,t}}\mid\ H_t\rbtt$ is non-negative. 
The induction hypothesis follows for $t$ noting \eqref{inlemma: sufficiency: main induction step: t 1st deduction} and \eqref{inlemma: sufficiency: main induction step: t 2nd deduction} imply
\[\Psi_{t}(\mx;\mp_{t}^*(H_{t}))-\Psi_{t}(\mx;\mp_{t}(H_{t}))=S_1+S_2,\]
 thus completing the proof of Lemma \ref{lemma: sufficiency: Induction: non-negative}.
\end{proof}

\begin{lemma}
\label{lemma: sufficiency: Induction: main step}
Suppose $\PP$, $\phi_1,\ldots,\phi_T$, and $Z_{\text{diff,t}}$ are as in Lemma \ref{lemma: sufficiency: Induction: non-negative}. Additionally,  $\phi_t(\mx;\pred(\mx))\geq \J_t$ for some $\J_t\geq 0$ for all $t\in[T]$ and $\mx\in\RR^{k_t}$. Then 
  for $t=1,\ldots,T$, for all $f\in\F$, and for any $\mx\in\RR^{k_t}$, 
\begin{align*}
\MoveEqLeft \Psi_{t}(\mx;\mp_{t}^*(H_{t}))-\Psi_{t}(\mx;\mp_{t}(H_{t}))
\geq \slb \prod_{i=t}^T\J_i\srb\slb \min_{1+t\leq i\leq T}\CC_{\phi_i}\srb \E[Z_{\text{diff,t}}\mid H_t],
\end{align*}
and $ \E[Z_{\text{diff,t}}\mid H_t]\geq 0$. Here   the $\CC_{\phi_t}$'s are as in Lemma \ref{lemma: necessity: linear bound} and $p_t(H_t)\equiv p_t(H_t;f)$. 
\end{lemma}

\begin{proof}
Note that Lemma \ref{lemma: sufficiency: Induction: non-negative} implies $\phi_t(f_t(H_t),\pred(f_t(H_t)))\geq \J_t$ since $f_t(H_t)\in\RR^{k_t}$ for all $t\in[T]$. If we redo all steps of Lemma \ref{lemma: sufficiency: Induction: non-negative} replacing \\
$\phi_r(f_r(H_r),\pred(f_r(H_r)))$ and $\phi_t(\mx,\pred(\mx))$ with $\J_t$ in all lower bounds, Lemma \ref{lemma: sufficiency: Induction: main step} will follow.  Since the proof is similar, it has been skipped.
\end{proof}
\subsection{Proof of Theorem \ref{theorem: sufficient conditions}}
\label{sec: proof of theorem suff}
We need to show that if $V^{\psi}(\fm)\to_m V^\psi_*$ for some $\{\fm\}_{m\geq 1}\subset \F$, then $V(\fm)\to_m V_*$. 
Without loss of generality, let us assume $C_{\phi_t}=1$ for $t\in[2:T]$ because we can substitute $\phi_t$ by $\phi_t/C_{\phi_t}$ otherwise. 
We claim that the proof of Fisher consistency follows directly from Lemma \ref{lemma: suff: main lemma} below. 
\begin{lemma}
\label{lemma: suff: main lemma}
  Suppose $\phi_1$ satisfies Condition \ref{assump: N1} and  $\phi_2,\ldots,\phi_T$ satisfy both Conditions \ref{assump: N1} and \ref{assump: N2} with $C_{\phi_t}=1$ for $t\in[2;T]$. If   $V^\psi(\fm)\to_m V^*$, then  under Assumptions I-V, 
  \begin{align}
      \label{instatement: suff: main lemma}
  \MoveEqLeft  \E\lbt \slb \sum_{i=1}^T Y_i\srb  \prod_{j=1}^{t-1}\frac{1[A_1=\pred(\fm_1(H_1))] }{\pi_1(A_1\mid H_1)}\prod_{j=1+t}^{T}\frac{1[A_j=d_j^*(H_j)]}{\pi_j(A_j\mid H_j)}\nn\\
   &\ \times \frac{ \slb 1[A_t=\pred(\mp^*_{t}(H_{t}))]-1[A_t=\pred(\fm_{t}(H_{t}))]\srb}{\pi_t(A_t\mid H_t)}\rbt\to_m 0
  \end{align} 
  for all $t\in[T]$, where the products with empty range are taken to be one.
\end{lemma}
To see  how the proof follows the Lemma \ref{lemma: suff: main lemma}, first  note that 
\[V_*-V(f)=\E\lbt \slb \sum_{i=1}^T Y_i\srb\frac{\prod_{t=1}^T1[A_t=d_t^*(H_t)]-\prod_{t=1}^T1[A_t=\pred(\fm_{t}(H_{t}))]}{\prod_{t=1}^T\pi_t(A_t\mid H_t)}\rbt.\]
Letting  products over empty ranges  be one, we calculate 
\begin{align*}
  \MoveEqLeft \sum_{t=1}^T\lbs\prod_{j=1}^{t-1}\frac{1[A_1=\pred(\fm_1(H_1))] }{\pi_1(A_1\mid H_1)}\prod_{j=1+t}^{T}\frac{1[A_j=d_j^*(H_j)]}{\pi_j(A_j\mid H_j)}\\
   &\ \times\frac{ \slb 1[A_t=\pred(\mp^*_{t}(H_{t}))]-1[A_t=\pred(\fm_{t}(H_{t}))]\srb}{\pi_t(A_t\mid H_t)}\rbs\\
  \stackrel{(a)}{=}&\ \sum_{t=1}^T\lbs\prod_{j=1}^{t-1}\frac{1[A_1=\pred(\fm_1(H_1))] }{\pi_1(A_1\mid H_1)}\prod_{j=1+t}^{T}\frac{1[A_j=d_j^*(H_j)]}{\pi_j(A_j\mid H_j)}\\
  &\ \times\frac{ \slb 1[A_t=d_t^*(H_{t})]-1[A_t=\pred(\fm_{t}(H_{t}))]\srb}{\pi_t(A_t\mid H_t)}\rbs\\
  =&\ \sum_{t=1}^T\prod_{j=1}^{t-1}\frac{1[A_1=\pred(\fm_1(H_1))] }{\pi_1(A_1\mid H_1)}\prod_{j=1+t}^{T}\frac{1[A_j=d_j^*(H_j)]}{\pi_j(A_j\mid H_j)}\frac{ 1[A_t=d_t^*(H_{t})]}{\pi_t(A_t\mid H_t)}\\
  &\ -\sum_{t=1}^T\prod_{j=1}^{t-1}\frac{1[A_1=\pred(\fm_1(H_1))] }{\pi_1(A_1\mid H_1)}\prod_{j=1+t}^{T}\frac{1[A_j=d_j^*(H_j)]}{\pi_j(A_j\mid H_j)}\frac{1[A_t=\pred(\fm_{t}(H_{t}))]}{\pi_t(A_t\mid H_t)}\\
  =&\ \sum_{t=1}^T\prod_{j=1}^{t-1}\frac{1[A_1=\pred(\fm_1(H_1))] }{\pi_1(A_1\mid H_1)}\prod_{j=t}^{T}\frac{1[A_j=d_j^*(H_j)]}{\pi_j(A_j\mid H_j)}\\
  &\ -\sum_{t=1}^T\prod_{j=1}^{t}\frac{1[A_1=\pred(\fm_1(H_1))] }{\pi_1(A_1\mid H_1)}\prod_{j=1+t}^{T}\frac{1[A_j=d_j^*(H_j)]}{\pi_j(A_j\mid H_j)},
\end{align*}
where (a) follows from Lemmas \ref{lemma:  sufficiency: induction: Psi t star p t star} and \ref{lemma: sufficiency: p t star and Q t star}. The telescopic sum in the RHS of the above display equals
\[\frac{\prod_{t=1}^T1[A_t=d_t^*(H_t)]-\prod_{t=1}^T1[A_t=\pred(\fm_{t}(H_{t}))]}{\prod_{t=1}^T\pi_t(A_t\mid H_t)}.\]
Thus, we derived that $V_*-V(f)$ equals
 \[ \begin{split}
    \MoveEqLeft \E\lbt \slb \sum_{i=1}^T Y_i\srb  \sum_{t=1}^T\lbs\prod_{j=1}^{t-1}\frac{1[A_1=\pred(\fm_1(H_1))] }{\pi_1(A_1\mid H_1)}\prod_{j=1+t}^{T}\\
    &\  \times\frac{1[A_j=d_j^*(H_j)]}{\pi_j(A_j\mid H_j)}\frac{ \slb 1[A_t=\pred(\mp^*_{t}(H_{t}))]-1[A_t=\pred(\fm_{t}(H_{t}))]\srb}{\pi_t(A_t\mid H_t)}\rbs\rbt,
 \end{split}\]
 which goes to zero as $m\to\infty$ by Lemma \ref{lemma: suff: main lemma}. The rest of the proof is devoted towards proving  Lemma \ref{lemma: suff: main lemma}.
\subsubsection{Proof of Lemma \ref{lemma: suff: main lemma}}
Let $H_t(d^{\upm})$ be the counterfactual $H_t$ under the DTR $d^{\upm}_t(H_t)=\pred(\fm_t(H_t))$. In other word, the random variable $H_t(d^{\upm})$ is distributed as the $H_t$ in the world where the treatments are assigned according to $d^{\upm}$ instead of the propensity scores $\pi_t$. Thus 
\[H_1(d^{\upm})\equiv H_1,\quad H_2(d^{\upm})=(H_1, d^\upm_1(H_1), O_2(d^{\upm}_1(H_1)),Y_1(d^{\upm}_1(H_1))),\] and so on. Thus $H_t(d^{\upm})$'s distribution is not same as the distribution of $H_t$ under $\PP$ for $t\geq 2$. 
From Lemma 1 of \cite{orellana2010} \citep[see also][]{murphy2001marginal}, it follows that for $t\in[T-1]$, the distribution of $H_{1+t}(d^{\upm})$ is absolutely continuous with respect to the distribution of  $H_{1+t}$  under Assumptions I-III with  Radon-Nikodym derivative
\[\prod_{i=1}^{t}\frac{1[A_i=d^{\upm}(H_i)]}{\pi_i(A_i\mid H_i)},\]
which is well defined as $\pi_i(A_i\mid H_i)>0$ for all $i\in[T]$ by Assumption I. 
Therefore, under Assumptions I-III, for any measurable function $h:\H_{1+t}\mapsto\RR$, the expectation under $\PP_{d^{\upm}}$ is identified as 
\begin{align}
\label{g-formula}
    \dE[h(H_{1+t}(d^{\upm}))]=\E\lbt h(H_{1+t})\prod_{i=1}^t\frac{1[A_i=d^{\upm}(H_1)]}{\pi_i(A_i\mid H_i)}\rbt.
\end{align}
We will prove Lemma \ref{lemma: suff: main lemma} by mathematical induction. We consider the following induction hypothesis: under the setup of Lemma \ref{lemma: suff: main lemma}, for each $t\in[T]$, (P.i) \eqref{instatement: suff: main lemma} holds, and 
\begin{align}
    \label{intheorem: necessity: induction (ii)}
  (P.ii)\  \E\lbt\prod_{i=1}^{t-1} \frac{1[A_i=\pred(\fm_i(H_i))] }{\pi_i(A_i\mid H_i)}\slb\Psi_{t}^*(\mp^*_{t}(H_{t}))-\Psi_{t}(\fm_{t}(H_{t});\mp^*_{t}(H_{t}))\srb\rbt \to_m 0,
\end{align}
where the product over empty range is defined to be one.
We will denote $\mp_t(H_t;\fm)$ by $\mp_t^{(m)}(H_t)$ throughout the proof of Lemma \ref{lemma: suff: main lemma}. First we show that the induction hypotheses (P.i) and (P.ii) hold for $t=1$.
\subsubsection{Case $t=1$}
Note that 
\begin{align*}
 \MoveEqLeft   V^\psi_*-V^\psi(\fm)\nn\\
   \stackrel{(a)}{=}&\ \E[\Psi_1^*(\mp_1^*(H_1))]-\E[\Psi_1(\fm_1(H_1);\mpm(H_1))]\nn\\
    =&\ \E[\Psi_1^*(\mp_1^*(H_1))-\Psi_1(\fm_1(H_1);\mp_1^*(H_1))]\\
    &\ + \E\slbt \Psi_1(\fm_1(H_1);\mp_1^*(H_1))-\Psi_1(\fm_1(H_1);\mp_1^{(m)}(H_1))\srbt\\
    \stackrel{(b)}{\geq} &\ \E[\Psi_1^*(\mp_1^*(H_1))-\Psi_1(\fm_1(H_1);\mp_1^*(H_1))],
\end{align*}
where (a) follows from Lemma \ref{lemma: sufficiency} and (b) follows as  
Lemma \ref{lemma: sufficiency: Induction: non-negative} implies  
\[\Psi_1(\fm_1(H_1);\mp_1^*(H_1))\geq \Psi_1(\fm_1(H_1);\mpm_1(H_1)).\]
Also, the RHS of the above display is non-negative because $\Psi_1^*(\mp_1^*(H_1))\geq \Psi_1(\fm_1(H_1);\mp_1^*(H_1))$ by the definition of $\Psi^*$ in \eqref{def: Psi and psi star main text}. 
Since $V^\psi(\fm)\to_m V_*^\psi$, it then follows that 
\begin{align}
    \label{intheorem: suff: Psi1 f1 converges to Psi1 star }
    \E[\Psi_1^*(\mp_1^*(H_1))-\Psi_1(\fm_1(H_1);\mp_1^*(H_1))]\to_m & 0,
\end{align}
implying \eqref{intheorem: necessity: induction (ii)}. Hence Hypothesis P.ii holds for $t=1$.
Let $\B_1=[c_1,c_2]^{k_1}$ where $c_1,c_2>0$ is as in Lemma \ref{lemma: suff: p t star is bdd}. Then Lemma \ref{lemma: suff: p t star is bdd} implies that $\mp_1^*(h_1)\in\B_1$ for all $h_1\in\H_1$. Note that  $c_2$ depends on $\PP$. Since $\phi_1$  satisfies Condition \ref{assump: N1},  Lemma  \ref{lemma: necessity: Bartlett stuff} applies on $\phi_1$ with $\mp=\mp_1^*(h_1)$ and $\mathfrak{B}=\B_1$. 
Let $\varrho$ be as in Lemma \ref{lemma: necessity: Bartlett stuff}, which depends   on $\phi_1$ and $\PP$ (via $\B_1$). Since $\varrho$ is convex with $\dom(\varrho)\supset\RR_{\geq 0}$, $\max(\mp_1^*(H_1))-\myq(\fm_1(H_1);\mp_1^*(H_1))\in\dom(\varrho)$ for all $H_1$. Thus  Jensen's inequality implies that 
\begin{align*}
 \MoveEqLeft  \varrho\slb \E\slbt \max(\mp_1^*(H_1))-\myq(\fm_1(H_1);\mp_1^*(H_1))\srbt\srb\\
   \leq &\  \E\slbt \varrho\slb \max(\mp_1^*(H_1))-\myq(\fm_1(H_1);\mp_1^*(H_1))\srb\srbt\\
  \stackrel{(a)}{\leq} &\ \E\slbt \Psi_1^*(\mp_1^*(H_1))-\Psi_1(\fm_1(H_1);\mp_1^*(H_1))\srbt,
\end{align*}
where (a) follows from Lemma \ref{lemma: necessity: Bartlett stuff} because $\mp_1^*(H_1)\in\B_1$.   Lemma \ref{lemma: necessity: Bartlett stuff} states that $\varrho(x)\to 0$ implies $x\to 0$ for $x\geq 0$. Since $\max(\mp_1^*(H_1))\geq \myq(\fm_1(H_1);\mp_1^*(H_1))$ by the definition of $\myq$ in \eqref{def: myq},
\[\E[\max(\mp_1^*(H_1))- \myq(\fm_1(H_1);\mp_1^*(H_1))]\geq 0.\]
Therefore,  \eqref{intheorem: suff: Psi1 f1 converges to Psi1 star } leads to 
\begin{align}
  \label{intheorem: sufficiency: V convg: stage 1: intermediate}
    \E\slbt \max(\mp_1^*(H_1))-\myq(\fm_1(H_1);\mp_1^*(H_1))\srbt\to_m 0.
\end{align}
On one hand, Lemma \ref{lemma: sufficiency: p t star and Q t star} implies
\[\E[\max(\mp_1^*(H_1))]=V_*=\E\lbt \slb\sum_{i=1}^T Y_i\srb\prod_{j=1}^T\frac{1[A_j=d_j^*(H_j)]}{\pi_j(A_j\mid H_j)}\rbt.\]
 On the other hand, 
 \[\begin{split}
 \MoveEqLeft \myq(\fm_1(H_1);\mp_1^*(H_1))= \mp_1^*(H_1)_{\pred(\fm_1(H_1))}\\
  \stackrel{(a)}{=}&\ \E\lbt \slb\sum_{i=1}^T Y_i\srb\prod_{j=2}^T\frac{1[A_j=d_j^*(H_j)]}{\pi_j(A_j\mid H_j)}\rbt\ \bl\ H_1, A_1= \pred(\fm_1(H_1))\rbt,  
\end{split}\]
where (a) follows by \eqref{def: p t star}. 
Then \eqref{intheorem: necessity: IPW: 0-1} implies
\[\E[\myq(\fm_1(H_1);\mp_1^*(H_1))]=\E\lbt \slb\sum_{i=1}^T Y_i\srb\prod_{j=2}^T\frac{1[A_j=d_j^*(H_j)]}{\pi_j(A_j\mid H_j)}\frac{1[A_1=\pred(\fm_1(H_1))]}{\pi_1(A_1\mid H_1)}\rbt,\]
which, combined with \eqref{intheorem: sufficiency: V convg: stage 1: intermediate}, implies
\begin{align}
    \label{intheorem: sufficiency: V convg: stage 1}
    \E\lbt \slb\sum_{i=1}^T Y_i\srb\prod_{j=2}^T\frac{1[A_j=d_j^*(H_j)]}{\pi_j(A_j\mid H_j)}\frac{1[A_1=d_1^*(H_1)]-1[A_1=\pred(\fm_1(H_1))]}{\pi_1(A_1\mid H_1)}\rbt\to_m 0,
\end{align}
which implies induction hypothesis P.i 
holds when $t=1$. If $T=1$, then there is nothing else to prove. Therefore, we will consider the case $T\geq 2$ from now on.


\subsubsection{Case $t\in[2:T]$}
Let $T\geq 2$. Suppose the induction hypotheses P.i and P.ii hold for $t\in[1:T-1]$. We will show that they hold for $t+1$. Let us denote
\begin{equation}
\label{def: wt}
    w_t=\prod_{i=1}^{t-1} \frac{1[A_i=\pred(\fm_i(H_i))] }{\pi_i(A_i\mid H_i)},
\end{equation}
implying $w_1=1$.
Since $t<T$, $p_t$ and $p_t^*$  are as in  \eqref{def: suff: pt f} and \eqref{def: sufficiency: pt star}, respectively. 
Using the fact that $w_t\geq 0$ and equations \eqref{inlemma: sufficiency: main induction step: t 1st deduction} and \eqref{inlemma: sufficiency: main induction step: t 2nd deduction} from the proof of Lemma \ref{lemma: sufficiency: Induction: non-negative}, we obtain that
\begin{align*}
 \MoveEqLeft  \E\lbt w_t\slb \Psi_t(\fm_t(H_t);\mp_t^*(H_t))-\Psi_t(\fm_t(H_t);\mpm_t(H_t))\srb\rbt\\ 
 \geq &\ \E\lbt w_t\sum_{i=1}^{k_{t}}\phi_t(\fm_t(H_t);i) \E[\Psi_{1+t}^*(\mp^*_{1+t}(H_{1+t}))\\
 &\ -\Psi_{1+t}(f_{1+t}(H_{1+t});\mp^*_{1+t}(H_{1+t}))\mid H_t, A_t=i]\rbt\\
 &\ + \E\lbt w_t\sum_{i=1}^{k_{t}}\phi_t(\fm_t(H_t);i) \E[\Psi_{1+t}(f_{1+t}(H_{1+t});\mp^*_{1+t}(H_{1+t}))\\
 &\ -\Psi_{1+t}(f_{1+t}(H_{1+t});\mp_{1+t}(H_{1+t}))\mid H_{t}, A_{t}=i]\rbt,
\end{align*}
whose second term is non-negative by Lemma \ref{lemma: sufficiency: Induction: non-negative}. 
Since the $\phi_t$'s are non-negative and  \[\Psi_{1+t}^*(\mp^*_{1+t}(H_{1+t}))\geq \Psi_{1+t}(f_{1+t}(H_{1+t});\mp^*_{1+t}(H_{1+t}))\]
for all $t\in[T-1]$, and \eqref{intheorem: necessity: induction (ii)} holds for $t$, 
\begin{align}
    \label{intheorem: necessity: stage 2: first}
\MoveEqLeft\E\lbt w_t\phi_t(\fm_t(H_{t});\pred(\fm_t(H_{t}))) \E[\Psi_{1+t}^*(\mp^*_{1+t}(H_{1+t}))\\
&\ -\Psi_{1+t}(f_{1+t}(H_{1+t});\mp^*_{1+t}(H_{1+t}))\mid H_t, A_t=\pred(\fm_t(H_{t}))]\rbt\to_m 0.\nn
\end{align}
An application of \eqref{intheorem: necessity: IPW: 0-1} implies that
\begin{align*}
   & \E\lbt w_t\phi_{t}(\fm_{t}(H_{t});\pred(\fm_{t}(H_{t})))\E[\Psi_{1+t}^*(\mp^*_{1+t}(H_{1+t}))\\
    &\ -\Psi_{1+t}(f_{1+t}(H_{1+t});\mp_{1+t}^*(H_{1+t}))\mid H_{t}, A_{t}=\pred(\fm_{t}(H_{t}))]\rbt\\
    =&\ \E\lbt w_t\frac{1[A_{t}=\pred(\fm_{t}(H_{t}))]}{\pi_{t}(A_{t}\mid H_{t})}\phi_{t}(\fm_{t}(H_{t});\pred(\fm_{t}(H_{t})))\E[\Psi_{1+t}^*(\mp^*_{1+t}(H_{1+t}))\\
    &\ -\Psi_{1+t}(f_{1+t}(H_{1+t});\mp_{1+t}^*(H_{1+t}))\mid H_{t}]\rbt\\
   \stackrel{(a)}{=}&\ \E\lbt\lb\prod_{i=1}^t\frac{1[A_{i}=\pred(\fm_{i}(H_{i}))]}{\pi_{i}(A_{i}\mid H_{i})}\rb\phi_{t}(\fm_{t}(H_{t});\pred(\fm_{t}(H_{t})))\\
  &\  \times \slb\Psi_{1+t}^*(\mp^*_{1+t}(H_{1+t}))-\Psi_{1+t}(f_{1+t}(H_{1+t});\mp_{1+t}^*(H_{1+t}))\srb\rbt,
\end{align*}
where (a) follows from \eqref{def: wt}. However, the above, combined with \eqref{g-formula}, implies that 
\begin{align*}
 \MoveEqLeft  \dE\lbt\phi_{t}\slb \fm_{t}(H_{t}(d^{\upm}));\pred(\fm_{t}(H_{t}(d^{\upm})))\srb \slb\Psi_{1+t}^*(\mp^*_{1+t}(H_{1+t}(d^{\upm})))\\
   &\ -\Psi_{1+t}(f_{1+t}(H_{1+t}(d^{\upm}));\mp_{1+t}^*(H_{1+t}(d^{\upm})))\srb\rbt \to_m 0.
\end{align*}
While we use $H_{t}(d^{\upm})$, we want to remind the readers that the  functions  $\mp_t$ and $\mp_t^*$ do not depend on the distribution of  $H_{t}(d^{\upm})$, and their definition remains as in as in  \eqref{def: suff: pt f} and \eqref{def: sufficiency: pt star}, respectively, which depends only on $\PP$. 
 Lemma \ref{lemma: suff: lemma with counterfactuals} below states that the above convergence implies  
\begin{align*}
   \dE\lbt \Psi_{1+t}^*(\mp^*_{1+t}(H_{1+t}(d^{\upm})))-\Psi_{1+t}(f_{1+t}(H_{1+t}(d^{\upm}));\mp_{1+t}^*(H_{1+t}(d^{\upm})))\rbt \to_m 0,
\end{align*}
which combined with \eqref{g-formula}, implies that
\begin{align*}
   \E\lbt w_{1+t} \lb \Psi_{1+t}^*(\mp^*_{1+t}(H_{1+t}))-\Psi_{1+t}(f_{1+t}(H_{1+t});\mp_{1+t}^*(H_{1+t}))\rb\rbt \to_m 0.
\end{align*}
\begin{lemma}
\label{lemma: suff: lemma with counterfactuals}
Consider the setup of Lemma \ref{lemma: suff: main lemma}. Let $T\geq 2$ and 
\begin{align}
    \label{def: suff: Z upm}
    {\mathfrak Z}^{\upm}=\Psi_{1+t}^*(\mp^*_{1+t}(H_{1+t}(d^{\upm})))-\Psi_{1+t}(f_{1+t}(H_{1+t}(d^{\upm}));\mp_{1+t}^*(H_{1+t}(d^{\upm})))
\end{align}
for some $t\in[T-1]$. 
    Suppose \eqref{intheorem: necessity: induction (ii)} holds and 
    \begin{align}
        \label{instatement: lemma: counterfactual lemma}
        \dE\lbt\phi_{t}\slb \fm_{t}(H_{t}(d^{\upm}));\pred(\fm_{t}(H_{t}(d^{\upm})))\srb {\mathfrak Z}^{\upm}\rbt \to_m 0.
    \end{align}
    Then, under the setup of Lemma \ref{lemma: suff: main lemma},
    $\dE [{\mathfrak Z}^{\upm}] \to_m 0$.
    


\end{lemma}
Lemma \ref{lemma: suff: lemma with counterfactuals} is proved in Section \ref{sec: proof of main lemma: helper lemma}.
Since $t\geq 1$, we have $1+t\geq 2$. Therefore,  Lemma \ref{lemma: necessity: linear bound} implies
\begin{align}
\label{inlemma: suff: main lemma: t stage induction}
\MoveEqLeft \Psi_{1+t}^*(\mp^*_{1+t}(H_{1+t}))-\Psi_{1+t}(\fm_{1+t}(H_{1+t});\mp^*_{1+t}(H_{1+t})) \nn\\ 
\geq &\ \CC_{\phi_{1+t}}\slb \max(\mp^*_{1+t}(H_{1+t}))-\myq(\fm_{1+t}(H_{1+t});\mp^*_{1+t}(H_{1+t}))\srb.
\end{align}
Since $w_t$'s are non-negative by \eqref{def: wt}, we therefore obtain
\begin{align*}
 \MoveEqLeft  \E\lbt w_{1+t}\slb\Psi_{1+t}^*(\mp^*_{1+t}(H_{1+t}))-\Psi_{1+t}(\fm_{1+t}(H_{1+t});\mp^*_{1+t}(H_{1+t}))\srb\rbt\\
 \geq &\  \CC_{\phi_{1+t}}\E\lbt w_{1+t}\slb\max(\mp^*_{1+t}(H_{1+t}))-\myq(\fm_{1+t}(H_{1+t});\mp^*_{1+t}(H_{1+t}))\srb\rbt.
 \end{align*}
 Thus \eqref{inlemma: suff: main lemma: t stage induction} leads to
 \begin{align}
     \label{inlemma: suff: main lemma: t stage induction: final}
     \E\lbt w_{1+t}\slb\max(\mp^*_{1+t}(H_{1+t}))-\myq(\fm_{1+t}(H_{1+t});\mp^*_{1+t}(H_{1+t}))\srb\rbt\to_m 0.
 \end{align}
 On the other hand, 
\begin{align*}
 \MoveEqLeft  \E\lbt w_{1+t}\slb\max(\mp^*_{1+t}(H_{1+t}))-\myq(\fm_{1+t}(H_{1+t});\mp^*_{1+t}(H_{1+t}))\srb\rbt\\
 =&\ \E\lbt w_{1+t}\E\slbt \max(\mp^*_{1+t}(H_{1+t}))-\mp^*_{1+t}(H_{1+t})_{\pred(\fm_{1+t}(H_{1+t}))} \mid H_{1+t}\srbt \rbt.
\end{align*}
Now we want to use \eqref{intheorem: necessity: IPW: diff}  with $H_{1+t}$,  $h(H_{1+t})=1$, and 
\[\VV=\slb \sum_{i=1}^T Y_i\srb \prod_{j=2+t}^{T}\frac{1[A_j=d_j^*(H_j)]}{\pi_j(A_j\mid H_j)}.\]
Note that \eqref{def: p t star} implies
\[\E[\VV\mid H_{1+t},  A_{1+t}=i]=\mp_{1+t}^*(H_{1+t})_i\text{ for all }i\in[k_{1+t}]. \]
Therefore, \eqref{intheorem: necessity: IPW: diff} implies
\begin{align*}
 \MoveEqLeft  \E\lbt  \max(\mp_{1+t}^*(H_{1+t}))-\mp_{1+t}^*(H_{1+t})_{\pred(\fm_{1+t}(H_{1+t}))} \bl H_{1+t}\rbt\\
 =&\ \E\lbt  \frac{  1[A_{1+t}=\pred(\mp^*_{1+t}(H_{1+t}))]-1[A_{1+t}=\pred(\fm_{1+t}(H_{1+t}))]}{\pi_{1+t}(A_{1+t}\mid H_{1+t})}\VV\bl H_{1+t}\rbt.
\end{align*}
Thus we obtain that 
\begin{align*}
\MoveEqLeft \E\lbt w_{1+t}\slb\max(\mp^*_{1+t}(H_{1+t}))-\myq(\fm_{1+t}(H_{1+t});\mp^*_{1+t}(H_{1+t}))\srb\rbt\\
   =&\  \E\lbt w_{1+t}\lbs \slb \sum_{i=1}^T Y_i\srb  \prod_{j=2+t}^{T}\frac{1[A_j=d_j^*(H_j)]}{\pi_j(A_j\mid H_j)}\rbs\\
   &\ \times\frac{ \slb 1[A_{1+t}=\pred(\mp^*_{1+t}(H_{1+t}))]-1[A_{1+t}=\pred(\fm_{1+t}(H_{1+t}))]\srb}{\pi_{1+t}(A_{1+t}\mid H_{1+t})}\rbt,
\end{align*}
 If $2+t=T+1$, the product will be one since the range is empty. 
Hence, \eqref{inlemma: suff: main lemma: t stage induction: final} implies that 
\begin{equation*}
\begin{split}
       \MoveEqLeft \E\lbt \slb \sum_{i=1}^T Y_i\srb  w_{1+t}\prod_{j=2+t}^{T}\frac{1[A_j=d_j^*(H_j)]}{\pi_j(A_j\mid H_j)}\\
    &\ \times\frac{ \slb 1[A_{1+t}=\pred(\mp^*_{1+t}(H_{1+t}))]-1[A_{1+t}=\pred(\fm_{1+t}(H_{1+t}))]\srb}{\pi_{1+t}(A_{1+t}\mid H_{1+t})}\rbt\to_m 0,
\end{split}
\end{equation*}
which proves P.i for $1+t$.

\subsubsection{The proof of Lemma \ref{lemma: suff: lemma with counterfactuals}}
\label{sec: proof of main lemma: helper lemma}

\begin{proof}[Proof of Lemma \ref{lemma: suff: lemma with counterfactuals}]
First, we will show that the ${\mathfrak Z}^{\upm}$ in \eqref{def: suff: Z upm}  is non-negative and  bounded above.  Our ${\mathfrak Z}^{\upm}$ is also non-negative by the definition of $\Psi_t^*$ in \eqref{def: Psi and psi star main text}. The boundedness follows because the $\phi_t$'s are bounded and $\mp_{1+t}^*(h_{1+t})$ is bounded above by a constant $c_1$, potentially depending on $\PP$, for all $h_{1+t} \in \H_{1+t}$—as shown in Lemma~\ref{lemma: suff: p t star is bdd}. 

It suffices to show that given any subsequence of $\{m\}$, there exists a subsequence along which   $\dE [{\mathfrak Z}^{\upm}] \to_m 0$. First note that
    \[\phi_{t}\slb \fm_{t}(H_{t}(d^{\upm}));\pred(\fm_{t}(H_{t}(d^{\upm})))\srb {\mathfrak Z}^{\upm}\geq 0,\]
    \eqref{instatement: lemma: counterfactual lemma} implies
    \begin{align}
        \label{inlemma: suff: convg in Prob 1}
        \phi_{t}\slb \fm_{t}(H_{t}(d^{\upm}));\pred(\fm_{t}(H_{t}(d^{\upm})))\srb {\mathfrak Z}^{\upm}\to_{P} 0.
    \end{align}
  The convergence in \eqref{intheorem: necessity: induction (ii)} implies that
   \[\E\lbt\prod_{i=1}^{t-1} \frac{1[A_i=\pred(\fm_i(H_i))] }{\pi_i(A_i\mid H_i)}\slb\Psi_{t}^*(\mp^*_{t}(H_{t}))-\Psi_{t}(\fm_{t}(H_{t});\mp^*_{t}(H_{t}))\srb\rbt \to_m 0.\]
   Thus \eqref{g-formula} leads to
   \[\dE\lbt\Psi_{t}^*\slb\mp^*_{t}(H_{t}(d^{\upm}))\srb-\Psi_{t}\slb\fm_{t}(H_{t}(d^{\upm}));\mp^*_{t}(H_{t}(d^{\upm}))\srb\rbt \to_m 0.\]
 The random variable inside the above expectation is non-negative by \eqref{def: Psi and psi star main text}, which implies the above convergence is $L_1$-convergence. Since $L_1$ convergence implies convergence in probability, 
  \begin{align}
      \label{inlemma: suff: convg in Prob 2}
      \Psi_{t}^*\slb\mp^*_{t}(H_{t}(d^{\upm}))\srb-\Psi_{t}\slb\fm_{t}(H_{t}(d^{\upm}));\mp^*_{t}(H_{t}(d^{\upm}))\srb\to_P 0,
  \end{align}
  where $P$ corresponds to the underlying probability space. 
  Since convergence in probability a implies almost sure convergence along a subsequence, given any subsequence of $m$, we can find a further subsequence $\{m_r\}_{r\geq 1}\subset\{m\}$ so that
  the convergences in \eqref{inlemma: suff: convg in Prob 1} and \eqref{inlemma: suff: convg in Prob 2} are almost sure convergences w.r.t. the underlying probability $P$. The proof will be complete if we can show that $\E[Z^{(m_r)}]\to_r 0$. For the sake of notational simplicity, we will denote the subsequence $m_r$ by $m$. Hence, we need to show that if
  \begin{align}
        \label{inlemma: suff: convg as 1}
        \phi_{t}\slb \fm_{t}(H_{t}(d^{\upm}));\pred(\fm_{t}(H_{t}(d^{\upm})))\srb {\mathfrak Z}^{\upm}\as 0
    \end{align}
    and 
  \begin{align}
      \label{inlemma: suff: convg as 2}
      \Psi_{t}^*\slb\mp^*_{t}(H_{t}(d^{\upm}))\srb-\Psi_{t}\slb\fm_{t}(H_{t}(d^{\upm}));\mp^*_{t}(H_{t}(d^{\upm}))\srb\as 0
  \end{align}
  as $m\to\infty$, then 
  then it holds that $\E[Z^{(m)}]\to_m 0$. Now we will use Lemma \ref{lemma: necessity: phi bdd away from zero}  with $\mxm=\fm_{t}(H_{t}(d^{\upm}))$ and $\mpm=\mp^*_{t}(H_{t}(d^{\upm}))$. This lemma applies because Lemma \ref{lemma: suff: p t star is bdd} implies  $\mp_t^*(h_t)\in[c_1,c_2]^{k_t}$ for all $h_t\in\H_t$, which indicates that   $\|\mp_t^*(h_t)\|_1$ is bounded below by  a constant, perhaps depending on $\PP$, for all $h_t\in\H_t$.  Let $\J_t>0$ be the fixed constant appearing in Lemma \ref{lemma: necessity: phi bdd away from zero}, which  depends  only on $\phi_t$.
 Lemma \ref{lemma: necessity: phi bdd away from zero} implies 
 \[\begin{split}
  \MoveEqLeft P\lb\liminf_{m\to\infty}\phi_{t}\slb \fm_{t}(H_{t}(d^{\upm}));\pred(\fm_{t}(H_{t}(d^{\upm}))\srb>\J_t\rb\\ \geq &\ P\lb \Psi_{t}^*\slb\mp^*_{t}(H_{t}(d^{\upm}))\srb-\Psi_{t}\slb\fm_{t}(H_{t}(d^{\upm}));\mp^*_{t}(H_{t}(d^{\upm}))\srb\to_m 0\rb,  
 \end{split}
 \]
 which is one because of \eqref{inlemma: suff: convg as 2}.
  Then \eqref{inlemma: suff: convg as 1}, combined with the fact that ${\mathfrak Z}^{\upm}\geq 0$ implies ${\mathfrak Z}^{\upm}\as 0$ as $m\to \infty$. 
 Since ${\mathfrak Z}^{\upm}$ is a bounded random variable, Fact \ref{fact: as to L1} implies $\dE[{\mathfrak Z}^{\upm}]\to_m 0$, which completes the proof.

\end{proof}

 \subsubsection{New proof of main lemma commented below}

 \subsection{Proof of Proposition \ref{prop: multi-cat FC}}
 \label{sec: proof of dorollary  multi-cat FC}
 We show the proof first for the case  
 $C_{\phi_t}=1$ for each $t\in[T]$. 
For any $f\in\F$, 
\begin{align}
\label{intheorem: suff: V psi breakdown}
 \MoveEqLeft   V^\psi_*-V^\psi(f)\nn\\
    =&\ \E[\Psi_1^*(\mp_1^*(H_1))]-\E[\Psi_1(f_1(H_1);\mp_1(H_1))]\nn\\
    =&\  \E[\Psi_1^*(\mp_1^*(H_1))-\Psi_1(f_1(H_1);\mp_1^*(H_1))]\\
    &\ + \E\slbt \Psi_1(f_1(H_1);\mp_1^*(H_1))-\Psi_1(f_1(H_1);\mp_1(H_1))\srbt\nn,
\end{align}
where the first term is non-negative,  and by Lemma \ref{lemma: sufficiency: Induction: main step}, the second term is bounded below by
\[C\E\lbtt \frac{\begin{matrix}
      \slb\sum_{i=1}^TY_i\srb 1[A_{1}=\pred(f_1(H_1))]\\
      \times\slb \prod_{r=2}^T1[A_r=d_r^*(H_r)] -\prod_{r=2}^T1[A_r=\pred(f_r(H_r))]\srb\end{matrix}}{\prod_{r=1}^T\pi_r(A_r\mid H_r)
}
\ \Bigg|\ H_1\rbtt,\]
where $C=\slb \prod_{i=1}^T\J_{i}\srb \min_{2\leq i\leq T}\CC_{\phi_i}$. Lemma \ref{lemma: sufficiency: Induction: main step} also indicates that the above expression is  non-negative. 
 Lemma \ref{lemma: necessity: linear bound} implies that 
\begin{align*}
    \MoveEqLeft\E\slbt \Psi_1^*(\mp_1^*(H_1))-\Psi_1(f_1(H_1);\mp_1^*(H_1))\srbt 
  \geq  \CC_{\phi_1}\slb \E\slbt \max(\mp_1^*(H_1)-\myq(f_1(H_1),\mp_1^*(H_1))\srbt\srb
\end{align*}
where $\CC_{\phi_1}$ is as in Lemma \ref{lemma: necessity: linear bound}. Then
\eqref{intheorem: suff: V psi breakdown} reduces to
\begin{align*}
   \MoveEqLeft V^\psi_*-V^\psi(f)\geq \CC_{\phi_1}\slb \E\slbt \max(\mp_1^*(H_1)-\myq(f_1(H_1),\mp_1^*(H_1))\srbt\srb\\
   + &\ C\E\lbt \frac{\slb\sum_{i=1}^TY_i\srb  \times 1[A_{1}=\pred(f_1(H_1))]}{\prod_{r=1}^T\pi_r(A_r\mid H_r)}\\
   &\ \times
\slb \prod_{r=2}^T1[A_r=d_r^*(H_r)] -\prod_{r=2}^T1[A_r=\pred(f_r(H_r))]\srb\ \bl\ H_1\rbt.
\end{align*}
where $C=\slb \prod_{i=1}^T\J_{i}\srb \min_{2\leq i\leq T}\CC_{\phi_i}$. Lemma \ref{lemma: sufficiency: p t star and Q t star} implies that $\mp_1^*(H_1)=Q_1^*(H_1)$. Thus
\[\max(\mp_1^*(H_1))=\max(Q_1^*(H_1))=Q_1^*(H_1,d_1^*(H_1))\]
and
\[\myq(\fm_1(H_1),\mp_1^*(h_1))=Q_1^*(H_1,\pred(\fm_1(h_1))).\]
However, 
\[Q_1^*(H_1,i)=\E\lbt \slb\sum_{i=1}^TY_i\srb\prod_{j=2}^{T}\frac{1[A_j=d_j^*(H_j)]}{\pi_j(A_j\mid H_j)}\ \bl\ H_1, A_1=i\rbt\]
by Fact \ref{fact: Q function expression}. Therefore, 
\begin{align}
\label{intheorem: will need for cor}
  & E\slbt \max(\mp_1^*(h_1))-\myq(\fm_1(h_1),\mp_1^*(h_1))\srbt\nn\\
  =&\ \E\lbt \slb\sum_{i=1}^TY_i\srb\prod_{j=2}^{T}\frac{1[A_j=d_j^*(H_j)]}{\pi_j(A_j\mid H_j)}\ \bl\ H_1, A_1=d_1^*(H_1)\rbt\nn\\
  &\ - \E\lbt \slb\sum_{i=1}^TY_i\srb\prod_{j=2}^{T}\frac{1[A_j=d_j^*(H_j)]}{\pi_j(A_j\mid H_j)}\ \bl\ H_1, A_1=\pred(\fm_1(H_1))\rbt\nn\\
  =&\ \E\lbt \frac{\slb\sum_{i=1}^TY_i\srb\prod_{j=2}^{T}1[A_j=d_j^*(H_j)]}{\prod_{j=1}^{T}\pi_j(A_j\mid H_j)} \lb\leftfrac{ 1[A_1=d_1^*(H_1)]}{- 1[A_1=\pred(\fm_1(H_1))]}\rb \rbt\nn\\
  =&\  V^*- \E\lbt \sum_{i=1}^TY_i\frac{\prod_{j=2}^{T}1[A_j=d_j^*(H_j)]}{\prod_{j=1}^{T}\pi_j(A_j\mid H_j)} 1[A_1=\pred(\fm_1(H_1))] \rbt 
\end{align}
Now
\eqref{intheorem: will need for cor} implies that 
\begin{align*}
\MoveEqLeft \E\slbt \max(\mp_1^*(H_1)-\myq(f_1(H_1),\mp_1^*(H_1))\srbt\\
 =&\ V^*- \E\lbt \sum_{i=1}^TY_i\frac{\prod_{j=2}^{T}1[A_j=d_j^*(H_j)]}{\prod_{j=1}^{T}\pi_j(A_j\mid H_j)} 1[A_1=\pred(f_1(H_1))] \rbt,
\end{align*}
implying
\begin{align*}
   \MoveEqLeft V^\psi_*-V^\psi(f)\geq \CC_{\phi_1}\lb V^*- \E\lbt \sum_{i=1}^TY_i\frac{\prod_{j=2}^{T}1[A_j=d_j^*(H_j)]}{\prod_{j=1}^{T}\pi_j(A_j\mid H_j)} 1[A_1=\pred(f_1(H_1))] \rbt\rb\\
   + &\ C\E\left[ \frac{\slb\sum_{i=1}^TY_i\srb 1[A_{1}=\pred(f_1(H_1))]}{\prod_{r=1}^T\pi_r(A_r\mid H_r)}\right.
\\
&\ \left.\times\slb \prod_{r=2}^T1[A_r=d_r^*(H_r)] -\prod_{r=2}^T1[A_r=\pred(f_r(H_r))]\srb\ \bl\ H_1\right]\\
\geq &\ \min(\CC_{\phi_1},C)\slb V^*-V(f)\srbt.
\end{align*}
where the last step follows because both terms in the sum are non-negatives. Note that
\[\min(\CC_{\phi_1},C)\geq \slb \prod_{i=1}^T\min(\J_{i},1)\srb \min_{1\leq i\leq T}\CC_{\phi_i}.\]
However, since $\Psi^*_t(\mo_{k_t})=1$ for each $t\in[T]$, it follows that $\phi_t(\mx;\pred(\mx))\leq 1 $ for each $t\in[T]$. Therefore, $\J_t\leq 1$. Hence,
\[\min(\CC_{\phi_1},C)\geq \slb \prod_{t=1}^T\J_{t}\srb \min_{1\leq i\leq T}\CC_{\phi_i}.\]

Thus we have showed 
\begin{align}
     \label{instatement: sufficiency: lower bound linear under constraint}
     V^\psi_*-V^\psi(f)\geq\slb\prod_{t=1}^T\J_t\srb\min_{1\leq t\leq T}\CC_{\phi_t}\slb V_*-V(f)\srb
 \end{align}
 when $C_{\phi_t}=1$ for each $t\in[T]$. 

 Now suppose $C_{\phi_t}\neq 1$. 
 Then \eqref{instatement: sufficiency: lower bound linear under constraint} holds if we transform  $\phi_t\mapsto\phi_t/C_{\phi_t}$. However, $\J_t$ a also becomes $\J_t/C_{\phi_t}$ following this transformation. Therefore, 
 \eqref{instatement: sufficiency: lower bound linear under constraint} implies
 \[ \frac{V^\psi_*-V^\psi(f)}{\prod_{t=1}^TC_{\phi_t}}\geq\lb\prod_{t=1}^T\slb\frac{\J_t}{C_{\phi_t}}\srb\rb\min_{1\leq t\leq T}\CC_{\phi_t}\slb V_*-V(f)\srb,\]
which completes the proof.



\section{Proof of Theorem~\ref{theorem: necessity} }
\label{secpf: necessity}
 
  This section is organized as follows. Section \ref{secpf: lemmas for proving Cond 1} contains auxiliary lemmas necessary for proving the necessity of Condition \ref{assump: N1}. The proof of the necessity of Condition \ref{assump: N1} is then given in Section \ref{secpf: Cond 1}.  We  establish the necessity of Condition \ref{assump: N2} in Section  \ref{secpf: condition N2}.

\subsection{Lemmas required for proving Condition \ref{assump: N1}}
\label{secpf: lemmas for proving Cond 1}
\begin{lemma}
\label{lemma: necessity: sum for any t is like T}
    Suppose $\PP$ satisfies Assumption I-IV and $Y_t\geq 0$ for $t\in[T]$. Let $T\geq 2$. Additionally, $Y_t=0$ for $t>r$ where $r\in[T-1]$ is an integer, and
    \[Y_r\perp O_{r+1},A_{r+1},\ldots, O_T,A_T\mid H_r,A_r.\]
    Then $Q_t^*(H_t,A_t)=0$ for all $t>r$ and $Q_r^*(H_r,a_r)=\E[Y_r\mid H_r, A_r=a_r]$ for all $a_r\in[k]$. Additionally, 
    \begin{equation*}
\sup_{f\in\F}\E\lbt\prod_{t=1}^T\frac{\phi_t(f_t(H_t);A_t)}{\pi_t(A_t\mid H_t)} \sum_{t=1}^TY_t\rbt = \prod_{t=1+r}^T\Psi^*_t(\mo_{k_t})\sup_{f_1,\ldots,f_r} \E\lbt\prod_{t=1}^r\frac{\phi_t(f_t(H_t);A_t)}{\pi_t(A_t\mid H_t)} \sum_{t\in[r]}Y_t\rbt.
    \end{equation*}

\end{lemma}

\paragraph{Proof of Lemma \ref{lemma: necessity: sum for any t is like T}}
\label{sec: pf of lemma: necessity: sum for any t is like T}
\begin{proof}[Proof of Lemma \ref{lemma: necessity: sum for any t is like T}]

 First of all, note that for any $a\in[k_T]$,
    \[Q^*_T(H_T, a)=\E[Y_T\mid H_T, a]=0.\]
    Similarly, we can show that
    if $T-1>r$, for any $a\in[k_{T-1}]$,
    \begin{align*}
   Q^*_{T-1}(H_{T-1}, a)=&\ \E\slbt Y_{T-1}+\max_{a_T\in[k_T]}\E[Y_T\mid H_T, a_T]\ \bl\ H_{t-1},A_{t-1}=a\rbt=0.
    \end{align*}
   By induction, we can show that $Q^*_t(h_t,a_t)=0$ for all $h_t\in\H_t$ and $a_t\in[k_t]$ as long as  $t>r$.
   Therefore, for any $a_r\in[k_r]$, 
   \[Q^*_r(H_r,a_r)=\E[Y_r+Q^*_{r+1}(H_{r+1}, A_{r+1})\mid H_r, A_r=a_r]=\E[Y_r\mid H_r, A_r=a_r],\]
   which proves the first part of the lemma.

To prove the second part of the current lemma, we will show that for any $m\in[r,T-1]$, 
  \begin{equation}
      \label{inlemma: necessity: 1st lemma: induction statement}
      \begin{split}
\MoveEqLeft\sup_{f\in\F}\E\lbt\prod_{t=1}^T\frac{\phi_t(f_t(H_t);A_t)}{\pi_t(A_t\mid H_t)} \sum_{t=1}^TY_t\rbt\\
        =&\ \prod_{t=1+m}^T\Psi^*_t(\mo_{k_t}) \sup_{f_1,\ldots,f_{m}}\E\lbt\slb\sum_{t\in[r]}Y_t\srb\prod_{t=1}^{m}\frac{\phi_t( f_t(H_t);A_t)}{\pi_t(A_t\mid H_t)} \rbt.  
      \end{split}
  \end{equation}
  To prove \eqref{inlemma: necessity: 1st lemma: induction statement}, we will show that the statement holds for $m=T-1$. Then we will show that if \eqref{inlemma: necessity: 1st lemma: induction statement} holds for some $m$, then it holds for $m-1$ as well, provided  $m-1\geq r$. Then, by induction, it will follow that   $m=1+r$ satisfies \eqref{inlemma: necessity: 1st lemma: induction statement}, which will complete the proof of the current lemma.
  
   To this end, note that 
   \begin{align*}
       \MoveEqLeft \sup_{f\in\F}\E\lbt\prod_{t=1}^T\frac{\phi_t(f_t(H_t);A_t)}{\pi_t(A_t\mid H_t)} \sum_{t=1}^TY_t\rbt \\
       =&\ \sup_{f_1,\ldots,f_T}\E\lbt\slb\sum_{t\in[r]}Y_t\srb\prod_{t=1}^{T-1}\frac{\phi_t( f_t(H_t);A_t)}{\pi_t(A_t\mid H_t)} \E\lbt \frac{\phi_T( f_T(H_T);A_T)}{\pi_T(A_T\mid H_T)}\ \bl\ H_T,A_T\rbt\rbt\\
      \stackrel{(a)}{=}&\ \sup_{f_1,\ldots,f_T}\E\lbt\slb\sum_{t\in[r]}Y_t\srb\prod_{t=1}^{T-1}\frac{\phi_t( f_t(H_t);A_t)}{\pi_t(A_t\mid H_t)} \Psi_T(f_T(H_T);\mo_{k_T})\rbt\rbt\\
       \stackrel{(b)}{=}&\ \sup_{f_1,\ldots,f_{T-1}}\E\lbt\slb\sum_{t\in[r]}Y_t\srb\prod_{t=1}^{T-1}\frac{\phi_t( f_t(H_t);A_t)}{\pi_t(A_t\mid H_t)} \sup_{\mx\in\RR^{k_T}}\Psi_T(\mx;\mo_{k_T})\rbt,
   \end{align*}
   where (a) follows from \eqref{intheorem: necessity: IPW sum one} and step (b) uses the non-negativity of the $Y_t$'s  and the $\phi_t$'s (by our assumption). 
Note that $\sup_{\mx\in\RR^{k_T}}\sum_{i=1}^{k_T}\phi_T(\mx;i)=\Psi_t^*(\mo_{k_T})$.
   Thus we have shown that
   \[\sup_{f\in\F}\E\lbt\prod_{t=1}^T\frac{\phi_t(f_t(H_t);A_t)}{\pi_t(A_t\mid H_t)} \sum_{t=1}^TY_t\rbt=\Psi^*_T(\mo_{k_T})\sup_{f_1,\ldots,f_{T-1}}\E\lbt\slb\sum_{t\in[r]}Y_t\srb\prod_{t=1}^{T-1}\frac{\phi_t( f_t(H_t);A_t)}{\pi_t(A_t\mid H_t)} \rbt,\]
   which proves \eqref{inlemma: necessity: 1st lemma: induction statement} holds for $m=T-1$. 
  If $r=T-1$, then there is nothing to prove. Suppose $r\leq T-2$ (note that it requires $T>2$ for such $r\in\NN$ to exist).

   Now suppose \eqref{inlemma: necessity: 1st lemma: induction statement} holds for some $m\leq T-1$ such that $m\geq 1+r$. Then we will show that \eqref{inlemma: necessity: 1st lemma: induction statement} holds for  $m-1$ as well.
    Since $m\geq 1+r$, it holds that $m-1\geq r$ and $Y_1,\ldots,Y_r\subset H_m$. Hence, 
   \begin{align*}
  \MoveEqLeft \sup_{f_1,\ldots,f_{m}}\E\lbt\slb\sum_{t\in[r]}Y_t\srb\prod_{t=1}^{m}\frac{\phi_t( f_t(H_t);A_t)}{\pi_t(A_t\mid H_t)} \rbt  \\
   =&\ \sup_{f_1,\ldots,f_{m}}\E\lbt\slb\sum_{t\in[r]}Y_t\srb\prod_{t=1}^{m-1}\frac{\phi_t( f_t(H_t);A_t)}{\pi_t(A_t\mid H_t)} \E\lbt \frac{\phi_m( f_m(H_m);A_m)}{\pi_m(A_m\mid H_m)}\ \bl\ H_m, A_m\rbt \rbt 
   \end{align*}
Using \eqref{intheorem: necessity: IPW sum one}, we obtain that
\begin{align*}
     \E\lbt \frac{\phi_m( f_m(H_m);A_m)}{\pi_m(A_m\mid H_m)}\ \bl\ H_m, A_m\rbt =\Psi_m(f_m(H_m);\mo_{k_m}).
\end{align*}
Using the non-negativity of the $Y_t$'s and the non-negativity of he $\phi_t$'s, we obtain that  
\begin{align*}
\MoveEqLeft\sup_{f_1,\ldots,f_{m}}\E\lbt\slb\sum_{t\in[r]}Y_t\srb\prod_{t=1}^{m}\frac{\phi_t( f_t(H_t);A_t)}{\pi_t(A_t\mid H_t)} \rbt\\
=&\  \sup_{f_1,\ldots,f_{m-1}}\E\lbt\slb\sum_{t\in[r]}Y_t\srb\prod_{t=1}^{m-1}\frac{\phi_t( f_t(H_t);A_t)}{\pi_t(A_t\mid H_t)} \sup_{\mx\in\RR^{k_m}} \Psi_m(\mx;\mo_{k_m}) \rbt\\
 =&\sup_{\mx\in\RR^{k_m}}  \Psi_m(\mx;\mo_{k_m}) \sup_{f_1,\ldots,f_{m-1}}\E\lbt\slb\sum_{t\in[r]}Y_t\srb\prod_{t=1}^{m-1}\frac{\phi_t( f_t(H_t);A_t)}{\pi_t(A_t\mid H_t)} \rbt.
\end{align*}
Since $\sup_{\mx\in\RR^{k_m}}  \Psi_m(\mx;\mo_{k_m})=\Psi_m^*(\mo_{k_m})$ and \eqref{inlemma: necessity: 1st lemma: induction statement} holds for $m$, the above implies that \eqref{inlemma: necessity: 1st lemma: induction statement} holds for $m-1$ as well, thus completing the proof.
\end{proof}

\begin{lemma}
\label{lemma: necessity}
 Let $r\in[T]$. Suppose   $\PP$ is as in the setup of Theorem \ref{theorem: necessity}. Moreover, $Y_t\geq 0$ for $t\in[T]$ and $Y_t=0$ for all $t\neq r$. Further suppose 
$(H_{r-1},A_{r-1})\perp (A_r, O_r, Y_r)$, 
and $Y_r\perp (O_{r+1},A_{r+1},\ldots, O_T, A_T)\mid (O_r, A_r)$. Here $H_{t},A_t=\emptyset$ if $t<1$ or $t>T$. Then 
\[\sup_{f\in\F}\E\lbt \prod_{t=1}^T\frac{\phi_t(f_t(H_t);A_t)}{\pi_t(A_t\mid H_t)}\sum_{t=1}^TY_t\rbt=\E[\Psi_r^*(Q_r^*(O_r))]\prod_{t=1,t\neq r}^T \Psi_t^*(\mo_{k_t}),\]
where
\[Q_r^*(O_r)_{a_r}=Q_r^*(H_r,a_r)=\E[Y_r\mid O_r, A_r=a_r]\quad\text{ for }a_r\in[k_r].\]
\end{lemma}

\begin{proof}[Proof of Lemma \ref{lemma: necessity}]
Note that $Y_t=0$ is possible under the setup of Theorem \ref{theorem: necessity}. 
   By Lemma \ref{lemma: necessity: sum for any t is like T},  \[Q^*_r(H_r,a)=\E[Y_r\mid H_r, A_r=a].\]
Moreover, $Y_r\perp (H_{r-1}, A_{r-1}, Y_{r-1})$ where we used the fact that $Y_{r-1}=0$. Therefore, $\E[Y_r\mid H_r, A_r]=\E[Y_r\mid O_r, A_r]$. 
Therefore, 
\[Q^*_r(H_r,a)=\E[Y_r\mid O_r, A_r=a]\quad\text{ for all }a\in[k_r],\]
which implies $Q_r^*(O_r)_a=Q_r^*(H_r,a)$ for $a\in[k_r]$.

When $T=1$, the only possible value for $r$ is $1$. In that case, 
\[\sup_{f_1\in\F_1}\E\lbt \frac{\phi_1(f_1(H_1);A_1)}{\pi_t(A_1\mid H_1)}Y_1\rbt=\E[\Psi_1^*(Q_1^*(H_1))]\]
by 
\eqref{intheorem: necessity: IPW: general p supremum}. Therefore, the proof follows for $T=1$ case trivially. 
If $T>1$ and $r=1$, then from Lemma \ref{lemma: necessity: sum for any t is like T} it follows that
\[\sup_{f\in\F}\E\lbt \prod_{t=1}^T\frac{\phi_t(f_t(H_t);A_t)}{\pi_t(A_t\mid H_t)}\sum_{t=1}^TY_t\rbt=\prod_{t=2}^T \Psi_t^*(\mo_{k_t})\sup_{f_r} \E\lbt\frac{\phi_r(f_r(H_r);A_r)}{\pi_r(A_r\mid H_r)} Y_r\rbt,\]
which equals $\E[\Psi_r^*(Q_r(H_r))]\prod_{t=2}^T \Psi_t^*(\mo_{k_t})$ by 
\eqref{intheorem: necessity: IPW: general p supremum}. Therefore, the proof of the current lemma follows trivially also when $r=1$. 

Hence, we consider the case $r>1$. Note that if $r>1$, then $T\geq 2$. 
Lemma \ref{lemma: necessity: sum for any t is like T} implies when $r\in[T-1]$, 
\begin{align*}
\MoveEqLeft\sup_{f\in\F}\E\lbt\prod_{t=1}^T\frac{\phi_t(f_t(H_t);A_t)}{\pi_t(A_t\mid H_t)} \sum_{t=1}^TY_t\rbt
=  \prod_{t=1+r}^T \Psi_t^*(\mo_{k_t})\sup_{f_1,\ldots,f_r} \E\lbt\prod_{t=1}^r\frac{\phi_t(f_t(H_t);A_t)}{\pi_t(A_t\mid H_t)} Y_r\rbt.
\end{align*}
The product in the above expression  is one if $r=T$. It suffices to show that
\begin{align}
    \label{inlemma: necessity: product from 1 to r}
    \sup_{f_1,\ldots,f_r} \E\lbt\prod_{t=1}^r\frac{\phi_t(f_t(H_t);A_t)}{\pi_t(A_t\mid H_t)} Y_r\rbt= \E[\Psi^*_r(Q^*_r(O_r))]\prod_{t=1}^{r-1}\Psi_t^*(\mo_{k_t}).
\end{align}

Suppose $V$ is any non-negative random variable that is a function of $(H_{r-1},A_{r-1})$.
Using Assumption I and the non-negativity of the $\phi_t$'s, we obtain that
\begin{align*}
  \sup_{f_r}\E\lbt V\frac{\phi_r(f_r(H_r);A_r)}{\pi_r(A_r\mid H_r)} Y_r\rbt
   \stackrel{(a)}{=}&\ \sup_{f_r}\E\lbt V \E\lbt \frac{\phi_r(f_r(H_r);A_r)}{\pi_r(A_r\mid H_r)} Y_r\ \bl\ H_r\rbt\rbt\\
   \stackrel{(b)}{=} &\ \sup_{f_r}\E\lbt V\sum_{a_r\in[k_r]}\E[Y_r\mid H_r, A_r=a_r]\phi_r(f_r(H_r); a_r)\rbt\\
    \stackrel{(c)}{=} &\  \sup_{f_r}\E\lbt V\sum_{a_r\in[k_r]}\E[Y_r\mid O_r, A_r=a_r]\phi_r(f_r(H_r); a_r)\rbt\\
    \stackrel{(d)}{=}&\ \E\lbt V\sup_{\mx\in\RR^{k_r}}\Psi_r(\mx;Q_r^*(O_r))\rbt,
\end{align*}
where (a) follows because $V$ is a function of $H_{r-1}$ and $A_{r-1}$, (b) follows because the propensity scores are bounded away from zero by Assumption I, (c) follows because $Y_r\perp H_{r-1}, A_{r-1},Y_{r-1}$ (we again used the fact that $Y_{r-1}$, being zero, is independent of $Y_r$) and (d) follows using the non-negativity of $V$.
Therefore, we have proved that
\begin{align}
\label{inlemma: necessity: V main eq}
     \sup_{f_r}\E\lbt V\frac{\phi_r(f_r(H_r);A_r)}{\pi_r(A_r\mid H_r)} Y_r\rbt=\E[V\Psi^*_r(Q_r^*(O_r))]=\E[V]\E[\Psi^*_r(Q_r^*(O_r))]
\end{align}
because $O_r\perp (H_{r-1},A_{r-1})$ and $V$ is a function of $(H_{r-1},A_{r-1})$. 
Letting $V=1$ in \eqref{inlemma: necessity: V main eq}, we obtain that
\[ \sup_{f_r}\E\lbt \frac{\phi_r(f_r(H_r);A_r)}{\pi_r(A_r\mid H_r)} Y_r\rbt=\E[\Psi^*_r(Q_r^*(O_r))].\]
Letting
\begin{align*}
    V= \prod_{t=1}^{r-1}\frac{\phi_t(f_t(H_t);A_t)}{\pi_t(A_t\mid H_t)},
\end{align*}
we obtain that
\[ \sup_{f_r}\E\lbt \prod_{t=1}^{r}\frac{\phi_t(f_t(H_t);A_t)}{\pi_t(A_t\mid H_t)}Y_r\rbt=\E\lbt \prod_{t=1}^{r-1}\frac{\phi_t(f_t(H_t);A_t)}{\pi_t(A_t\mid H_t)}\rbt \E[\Psi^*_r(Q_r^*(O_r))].\]
Therefore,
\[ \sup_{f_1,\ldots,f_r}\E\lbt \prod_{t=1}^{r}\frac{\phi_t(f_t(H_t);A_t)}{\pi_t(A_t\mid H_t)}Y_r\rbt=\sup_{f_1,\ldots,f_{r-1}}\E\lbt \prod_{t=1}^{r-1}\frac{\phi_t(f_t(H_t);A_t)}{\pi_t(A_t\mid H_t)}\rbt \E[\Psi^*_r(Q_r^*(O_r))].\]
The proof will be complete if we can show that
\begin{align}
    \label{inlemma: necessity: product from 1 to r-1}
    \sup_{f_1,\ldots,f_{r-1}}\E\lbt \prod_{t=1}^{r-1}\frac{\phi_t(f_t(H_t);A_t)}{\pi_t(A_t\mid H_t)}\rbt=\prod_{t=1}^{r-1}\Psi^*_t(\mo_{k_t}).
\end{align}
For any $j\leq r-1$, 
\begin{align*}
\MoveEqLeft\sup_{f_1,\ldots,f_{j}}\E\lbt \prod_{t=1}^{j}\frac{\phi_t(f_t(H_t);A_t)}{\pi_t(A_t\mid H_t)}\rbt\\
 =&\ \sup_{f_1,\ldots,f_{j}}\E\lbt \prod_{t=1}^{j}\frac{\phi_t(f_t(H_t);A_t)}{\pi_t(A_t\mid H_t)}\E\lbt \frac{\phi_j(f_j(H_j);A_j)}{\pi_j(A_j\mid H_j)}\ \bl\ H_j, A_j\rbt\rbt\\
 \stackrel{(a)}{=}&\ \sup_{f_1,\ldots,f_{j}}\E\lbt \prod_{t=1}^{j}\frac{\phi_t(f_t(H_t);A_t)}{\pi_t(A_t\mid H_t)}\Psi_j(f_j(H_j);\mo_{k_j})\rbt
\end{align*}
where (a) follows from \eqref{intheorem: necessity: IPW sum one}.
Since the $\phi_t$'s are non-negative, the above equals
\[\sup_{f_1,\ldots,f_{j-1}}\E\lbt \prod_{t=1}^{j}\frac{\phi_t(f_t(H_t);A_t)}{\pi_t(A_t\mid H_t)}\Psi_j^*(\mo_{k_j})\rbt=\Psi_j^*(\mo_{k_j})\sup_{f_1,\ldots,f_{j-1}}\E\lbt \prod_{t=1}^{j}\frac{\phi_t(f_t(H_t);A_t)}{\pi_t(A_t\mid H_t)}\rbt.\]
Equation \ref{inlemma: necessity: product from 1 to r-1} follows by applying \eqref{intheorem: necessity: IPW sum one} repeatedly to $j=r-1,\ldots,1$.


\end{proof}

\subsection{Proving the necessity of  Condition \ref{assump: N1}}
\label{secpf: Cond 1}
We will prove by contradiction. 
Suppose there exists $r\in[T]$ so that $\phi_r$ does not satisfy Condition \ref{assump: N1}. Hence, there exists $\mp\in \RR^{k_r}_{\geq 0}$ that violates \eqref{inlemma: necessity: single-stage FC} with $\Psi_r$. 
Suppose $\mp$ is such that all elements of $\mp$ are equal. Then the left hand side of \eqref{inlemma: necessity: single-stage FC} is $\infty$ since we defined the supremum of emptyset to be $-\infty$. Therefore, \eqref{inlemma: necessity: single-stage FC} automatically holds. Therefore, in what follows, we assume that the elements of $\mp$ are not all equal. Since \eqref{inlemma: necessity: single-stage FC} fails, there exists $\mp\in\RR_{\geq 0}^{k_r}$ and a sequence $\{\xpmr\}_{m\geq 1}\subset\RR^{k_r}$ so that as $m\to \infty$, $\Psi_r(\xpmr;\mp)\mapsto \Psi_r^*(\mp)$ but $\mp_{\pred(\xpmr)}<\max(\mp)$ for all $m\in\NN$.   
We will show that if the above-mentioned $\xpmr$ exists, then we can construct a $\PP$ and a sequence of class score vectors (these are vector-valued functions) $\{\fm\}_{m\geq 1}\subset\F$ so that $V^\psi(\fm)\to_m V^\psi_*$ but $V(\fm)\not\to V_*$, which will complete the proof.  For $t\in[T]$ such that $t\neq r$, we define the sequences  $\{\xpmt\}_{m\geq 1}\subset\RR^{k_t}$ so  that $\Psi_t(\xpmt;\mo_{k_t})\to_m\Psi_t^*(\mo_{k_t})$.

We will define our $\fm$ and $\PP$ now. We take $\fpm$ to be so that $\fpm_t\equiv\xpmt$ for all $t\in[T]$, i.e., $\fpm_t(h_t)=\xpmt$ for all $h_t\in\H_t$. 
 We let  $\PP$ be as in Lemma \ref{lemma: necessity}, i.e., $\PP$ satisfies Assumptions I-IV, $Y_t = 0$ for all $t \neq r$, and $Y_r\geq 0$. Note that $Y_t=0$ is allowed because under the set up of Theorem \ref{theorem: necessity}, $\PP$ does not need to satisfy Assumption V.  Additionally, the independence conditions  
\[
(H_{r-1}, A_{r-1}) \perp (A_r, O_r, Y_r)
\]
and  
\[
Y_r \perp (O_{r+1}, A_{r+1}, \ldots, O_T, A_T) \mid (O_r, A_r)
\]  
hold for all $t\in[T]$, where $H_t$ and $A_t$ are empty sets if $t < 1$ or $t > T$. Furthermore, we assume that  
\[
\E[Y_r \mid H_r, A_r = i] = \mp_i \quad \text{for } i \in [k].
\]  
It is straightforward to verify the existence of such distributions.
From Lemma \ref{lemma: necessity}, it follows that
$Q_r^*(H_r,A_r)=\mp_{A_r}$ and
\begin{align}
    \label{inlemma: necessity: optimal when r}
    \sup_{f_1,\ldots,f_T}\E\lbt \prod_{t=1}^{T}\frac{\phi_t(f_t(H_t);A_t)}{\pi_t(A_t\mid H_t)}Y_r\rbt=\Psi_r^*(\mp)\prod_{t\neq r}\Psi^*_t(\mo_{k_t}).
\end{align}
First, we will show that $V^\psi(\fpm)\to_mV^\psi_*$. Next, we will show that $V(\fpm)\not\to_m V_*$, which will contradict the Fisher consistency of $\psi$, hence completing the proof by contradiction. 

\paragraph{Showing $\V^\psi(\fpm)\to_m V^\psi_*$} It suffices to show that for each $m\in[\NN]$,
\begin{align}
    \label{intheorem: necessity cond 1: V psi expression second deduction}
    V^\psi(\fpm)=\Psi_r(\xpmr;\mp)\prod_{t\neq r, t=1}^{T}\Psi_t(\xpmt;\mo_{k_t})
\end{align}
because the above converges to $\Psi_r^*(\mp)\prod_{t\neq r, t=1}^{T}\Psi_t^*(\mo_{k_t})$ due to $\Psi_r(\xpmr;\mp)\to_m\Psi_r^*(\mp)$ and $\Psi_t(\xpmt;\mo_{k_t})\to_m\Psi_t^*(\mo_{k_t})$ for $t\neq r$. However, \eqref{inlemma: necessity: optimal when r} implies 
\[V^\psi_*=\Psi_r^*(\mp)\prod_{t\neq r, t=1}^{T}\Psi_t^*(\mo_{k_t}),\]
which would complete the proof of $V^\psi(\fpm)\to_m V^\psi_*$.

To show \eqref{intheorem: necessity cond 1: V psi expression second deduction}, we will first show that for any $r\in[T]$,
\begin{equation}
    \label{intheorem: necessity Cond 1: V psi first deduction}
    \begin{split}
     V^\psi(\fm)= &\ \E\lbt Y_r\prod_{t=1}^{T}\frac{\phi_t(\fpm_t(H_t);A_t)}{\pi_t(A_t\mid H_t)}\rbt\\
  =&\ \prod_{t=1+r}^{T} \Psi_t(\xpmt;\mo_{k_t})\E\lbt Y_r\prod_{t=1}^{r}\frac{\phi_t(\fpm_t(H_t);A_t)}{\pi_t(A_t\mid H_t)}\rbt.    
    \end{split}
\end{equation}
If $r=T$, \eqref{intheorem: necessity Cond 1: V psi first deduction} holds trivially because the product in  \eqref{intheorem: necessity Cond 1: V psi first deduction} becomes one since the range of the product becomes empty. For $r\in[T-1]$, 
\begin{align*}
\MoveEqLeft \E\lbt Y_r\prod_{t=1}^{T}\frac{\phi_t(\fpm_t(H_t);A_t)}{\pi_t(A_t\mid H_t)}\rbt \\
 =&\ \E\lbt Y_r\prod_{t=1}^{T-1}\frac{\phi_t(\fpm_t(H_t);A_t)}{\pi_t(A_t\mid H_t)}\E\lbt \frac{\Psi_t(\fpm_T(H_T);A_T)}{\pi_T(A_T\mid H_T)}\ \bl\ H_T, A_T\rbt\rbt\\
 \stackrel{(a)}{=}&\ \E\lbt Y_r\prod_{t=1}^{T-1}\frac{\phi_t(\fpm_t(H_t);A_t)}{\pi_t(A_t\mid H_t)}\Psi_T(\xpmt;\mo_{k_T})\rbt\\
 =&\ \Psi_T(\xpmt;\mo_{k_T})\E\lbt Y_r\prod_{t=1}^{T-1}\frac{\phi_t(\fpm_t(H_t);A_t)}{\pi_t(A_t\mid H_t)}\rbt,
\end{align*}
where (a) follows from \eqref{intheorem: necessity: IPW sum one}. Proceeding similarly and applying \eqref{intheorem: necessity: IPW sum one} repeatedly, we can show \eqref{intheorem: necessity Cond 1: V psi first deduction} holds.
To finish the proof of \eqref{intheorem: necessity cond 1: V psi expression second deduction}, now observe that 
\begin{align*}
  \E\lbt Y_r\prod_{t=1}^{r}\frac{\phi_t(\fpm_t(H_t);A_t)}{\pi_t(A_t\mid H_t)}\rbt
    =&\ \E\lbt \prod_{t=1}^{r-1}\frac{\phi_t(\fpm_t(H_t);A_t)}{\pi_t(A_t\mid H_t)}\E\lbt Y_r\frac{\phi_r(\fpm_r(H_r);A_r)}{\pi_r(A_r\mid H_r)}\rbt\rbt\\
   \stackrel{(a)}{=} &\ \E\lbt \prod_{t=1}^{r-1}\frac{\phi_t(\fpm_t(H_t);A_t)}{\pi_t(A_t\mid H_t)}\Psi_r(\fpm_r(H_r);Q^*_r(O_r))\rbt\rbt,
\end{align*}
where (a) follows from \eqref{intheorem: necessity: IPW: general p} with 
\[Q^*_r(O_r)_j=\E[Y_r\mid H_r,A_r=j]=\mp_j\]
by our definition of $\PP$. 
Since $\fpm_r=\xpmr$, 
\[ \E\lbt Y_r\prod_{t=1}^{r}\frac{\phi_t(\fpm_t(H_t);A_t)}{\pi_t(A_t\mid H_t)}\rbt=\Psi_r(\xpmr;\mp)\E\lbt \prod_{t=1}^{r-1}\frac{\phi_t(\fpm_t(H_t);A_t)}{\pi_t(A_t\mid H_t)}\rbt.\]
Using \eqref{intheorem: necessity: IPW sum one} and the fact that $\fpm(H_t)=\xpmt$ is non-stochastic, we can show that
\[\E\lbt \prod_{t=1}^{r-1}\frac{\phi_t(\fpm_t(H_t);A_t)}{\pi_t(A_t\mid H_t)}\rbt=\prod_{t=1}^{r-1}\Psi_t(\xpmt;\mo_{k_t}).\]
Therefore, \eqref{intheorem: necessity cond 1: V psi expression second deduction} follows, completing the proof of $V^\psi(\fm)\to_m V^\psi_*$.

\paragraph{Showing $\limsup_{m\to\infty}V(\fpm)< V_*$}
It remains to prove that $V(\fpm)\not\to_m V_*$. First, we will prove that $V(\fm)=\mp_{\pred(\xpmr)}$. To see this, note that
\begin{align*}
V(\fpm)=&\ \E\lbt Y_r\prod_{t=1}^T\frac{1[A_t=\pred(\fpm(H_t))]}{\pi_t(A_t\mid H_t)}\rbt\\
=&\ \E\lbt Y_r\prod_{t=1}^{T-1}\frac{1[A_t=\pred(\fpm(H_t))]}{\pi_t(A_t\mid H_t)}\E\lbt \frac{1[A_T=\pred(\fpm(H_T))]}{\pi_T(A_T\mid H_T)}\ \bl\ H_T\rbt\rbt.
\end{align*}
Letting the $\phi$ in \eqref{intheorem: necessity: IPW sum one} be $\phi_{\text{dis}}$, we obtain that
\begin{align*}
   \MoveEqLeft \E\lbt \frac{1[A_T=\pred(\fpm(H_T))]}{\pi_T(A_T\mid H_T)}\rbt= 1.
\end{align*}
Thus
\[V(\fpm)=\E\lbt Y_r\prod_{t=1}^{T-1}\frac{1[A_t=\pred(\fpm(H_t))]}{\pi_t(A_t\mid H_t)}\rbt,\]
which rewrites as 
\begin{align*}
 V(\fpm)= \E\lbt \prod_{t=1}^{r-1}\frac{1[A_t=\pred(\fpm(H_t))]}{\pi_t(A_t\mid H_t)}\E\lbt Y_r\frac{1[A_r=\pred(\fpm(H_r))]}{\pi_r(A_r\mid H_r)}\ \bl\ H_r\rbt\rbt   
\end{align*}
where the product term is one if $r=1$. Letting the $\phi$ in \eqref{intheorem: necessity: IPW: general p} be $\phi_{\text{dis}}$, we obtain that
\begin{equation}
    \begin{split}
        \E\lbt Y_r\frac{1[A_r=\pred(\fpm(H_r))]}{\pi_r(A_r\mid H_r)}\ \bl\ H_r\rbt=\sum_{a\in [k_r]}\E[Y_r\mid H_r,a]1[a=\pred(\xpmr)],
    \end{split}
\end{equation}
which equals $\mp_{\pred(\xpmr)}$. 
Thus
\[V(\fpm)=\mp_{\pred(\xpmr)}\E\lbt \prod_{t=1}^{r-1}\frac{1[A_t=\pred(\fpm(H_t))]}{\pi_t(A_t\mid H_t)}\rbt.\]
Applying \eqref{intheorem: necessity: IPW sum one} repeatedly with $\phi=\phi_{\text{dis}}$, we can show that
\[\E\lbt \prod_{t=1}^{r-1}\frac{1[A_t=\pred(\fpm(H_t))]}{\pi_t(A_t\mid H_t)}\rbt=1.\]
Hence, we showed that $V(\fpm)=\mp_{\pred(\xpmr)}$. 
It remains to show that \\
$\limsup_{m\to\infty}\mp_{\pred(\xpmr)}<  V_*$.  

By our assumption, for each $m\in\NN$, $\mp_{\pred(\xpm)}<\max(\mp)$. Therefore, $\pred(\xpm)\notin\argmax(\mp)$ for all $m\in\NN$. Since $\mp$ has a finite length, this implies $\limsup_{m\to\infty}\mp_{\pred(\xpm)}<\max(\mp)$. Therefore, $\limsup_{m\to\infty} V(\fpm)<\max(\mp)$. Thus, it suffices to show that $V_*\geq \max(\mp)$.
To this end, note that, for any $\mx\in\RR^{k_r}$,  setting $\fm_r\equiv \mx$, $\fm_t\equiv\xpmt$ for $t\neq r$, $\fm=(\fm_1,\ldots, \fm_r,\ldots,\fm_T)$, and working as before, we can show that
\[V(\fm)\to_m\mp_{\pred(\mx)}.\]
If we take $\mx=\mp$, then $V(\fm)\to_m\mp_{\pred(\mp)}=\max(\mp)$. Thus $V_*\geq \max(\mp)$, which completes the proof.

\subsection{Proving the necessity of  Condition \ref{assump: N2}} 
\label{secpf: condition N2}
We only need to consider the $T\geq 2$ case in this proof.
Since we have already proved that $\phi_t$'s satisfy Condition \ref{assump: N1} for $t\in[T]$, we will use this fact in the proof. In particular, $\phi_t$'s are bounded and  $\Psi_t^*(\mp)\in(0,\infty)$ for all $\mp\in\RR^{k+t}_{\geq 0}$ such that $\mp\neq \mz_{k_t}$ by Lemma \ref{lemma: necessity: psi bounded} and $\Psi_t^*$ is continuous by Lemma \ref{lemma: necessity: Psi-t cont.}. Moreover, Fact \ref{fact: Cond N1 implies pred x in argmax p} will also apply.

We will prove the necessity of 
 Condition \ref{assump: N2} in 3 steps.  
To this end, let $r\in[1:T-1]$. 
\begin{itemize}
    \item[Step 1.] We will show that for any $\mp,\mq\in\RR_{\geq 0}^{k_{1+r}}$, $\max(\mp)>\max(\mq)$ implies $\Psi_{1+r}^*(\mp)\geq \Psi_{1+r}^*(\mq)$.
    \item[Step 2.] Using the result from Step 1, we will show that for any $\mp,\mq\in\RR_{\geq 0}^{k_{1+r}}$, $\max(\mp)=\max(\mq)$ implies $\Psi_{1+r}^*(\mp)=\Psi_{1+r}^*(\mq)$. This will imply that there exists a function $h:\RR\mapsto\RR$ so that $\Psi_{1+r}^*(\mp)=h_{1+r}(\max(\mp))$ for all $\mp\in\RR^{k_{1+r}}_{\geq 0}$.
    \item[Step 3.] We will show that $h_{1+r}(x)=\Psi_{1+r}^*(\mo_{k_{1+r}})x$ for all $x>0$. $\Psi_{1+r}^*(\mo_{k_{1+r}})<\infty$ by Lemma \ref{lemma: necessity: psi bounded}.
\end{itemize}
Step 3 implies $\Psi_{1+r}^*(\mp)=\Psi_{1+r}^*(\mo_{k_{1+r}})\max(\mp)$ for all $\mp\in\RR_{\geq 0}^{k_{1+r}}$. Since $\Psi_{1+r}^*(\mo_{k_{1+r}})$ is positive for each $r\in[T-1]$ by Lemma \ref{lemma: necessity: psi bounded},  $\phi_{1+r}$ satisfies Condition \ref{assump: N2}. Since $r$ can be any number in $[1:T-1]$, it follows that $\phi_t$ satisfies Condition \ref{assump: N2} if $t\in[2:T]$, thus completing the proof.

\subsubsection{Step 1: Showing $\max(\mp)>\max(\mq)$ implies $\Psi_{1+r}^*(\mp)\geq \Psi_{1+r}^*(\mq)$  for any $\mp,\mq\in\RR^{k_{1+r}}_{\geq 0}$ and any $r\in[1:T-1]$}
Note that if $\mq=\mz_{k_{1+r}}$, then $\Psi_{1+r}^*(\mq)=0$.  Since the $\phi(\cdot;i)$'s are non-negative, $\Psi_{1+r}^*(\mp)\geq \Psi_{1+r}^*(\mq)$ automatically holds. Hence, it suffices to consider the case where $\mq\neq \mz_{k_{1+r}}$ so that $\max(\mq)>0$. In this case, note that we also have  $\mp\neq \mz_{k_{1+r}}$. 
Let $r\in[1:T-1]$. We will prove Step 1 in some substeps.
\begin{itemize}
    \item \textbf{Step 1a:} We will define a simpler class of distributions $\Pms$, depending on $\mp$ and $\mq$, that satisfy Assumptions I-IV and $Y_t\geq 0$ for all $t\in[T]$.
    \item \textbf{Step 1b:} We will show that when $\PP\in\Pms$, for any collection of non-negative surrogates $\phi_1,\ldots,\phi_T$,  we have simpler expressions of $V^\psi_*$ and $V^\psi(f)$ for some classes of $f$.
    \item \textbf{Step 1c:} Using the form of $V^\psi_*$, we will derive the form of $V_*$ by letting $\phi$ to be the 0-1 loss.
    \item\textbf{Step 1d:} We will define a sequence of class-score functions $\{\fpm\}_{m\geq 1}\subset\F$ so that $V^\psi(\fpm)\to_m  V^\psi_*$.
     \item\textbf{Step 1d:} We calculate $V(\fpm)$ for the aforementioned sequence of score functions. 
     \item \textbf{Step 1e:}  Fisher consistency implies $V(\fpm)\to_m V_*$, which leads to the desired inequality $\Psi_{1+r}^*(\mp)\geq \Psi_{1+r}^*(\mq)$.
\end{itemize}

\paragraph*{Step 1a. Defining the small class $\Pms$}
Suppose $\mp$ and $\mq$ are fixed vectors in $\RR^{k_{1+r}}_{\geq 0}$ such that $\max(\mp)>\max(\mq)$. 
Let $\Pms$ be the class of all $\PP$'s that satisfy Assumption I-IV and $Y_t\geq 0$ for all $t\in[T]$ in addition to the followings:
\begin{enumerate}
    \item  $Y_i=0$ unless $i=1+r$. 
    \item $H_{r-1},A_{r-1}\perp (A_{r},O_r,A_{1+r},O_{1+r},Y_{1+r})$ for all $t\in[2:T-1]$. If $r=1$, we take $H_0=\emptyset$ and $A_0=0$. Hence, the above statement is vacuously true for $r=1$.
    \item $O_r=O_{1+r}=\emptyset$. Hence,
    \begin{align*}
        \E[Y_{1+r}\mid H_{1+r}, A_{1+r}]=\E[Y_{1+r}\mid A_r,A_{1+r}].
    \end{align*}
    Also, for $i\in[k_r]$ and $j\in[k_{1+r}]$, we let  $ \E[Y_{1+r}\mid A_r=i,A_{1+r}=j] =\mpi_j$ where 
    \begin{align}
    \label{intheorem: necessity: cond 2: pm(i) def}
     \mpi =\begin{cases}
         \mp & \text{ if }i=1\\
         \mq& \text{ if  }i\in[2:k_r].
     \end{cases}
     \end{align}
  \item $Y_{1+r}\perp O_{r+2},A_{2+r},\ldots\mid H_r$. We take $O_{T+1}=\emptyset$ and $A_{T+1}=0$.
  \end{enumerate}
 Note that $\Pms$ has  resembelence with a 2-stage DTR. The outcome $Y_r$ does not depend on any variable other than $A_r$ and $A_{1+r}$. Thus the stages $r$ and $1+r$ are the only stages that has connection to the outcome. The  $T$ stage DTR is formed from this 2-stage DTR by padding  $r-1$  stages before  and $T-1-r$ after these two stages. The proof builds on the idea that, since the $\phi_t$'s are Fisher consistent,  $(\phi_r,\phi_{1+r})$ are Fisher consistent for the two-stage DTR embedded in the $T$ stage DTR.

  \paragraph*{Step 1b. Properties under $\Pms$}
   Distributions in $\Pms$ have some interesting properties. 
  First we will show that $V^\psi_*$ has some closed form expression under $\PP\in\Pms$ for non-negative surrogates. 
 Let us define $\nu\in\RR^{k_r}_{\geq 0}$ so that 
\begin{align}
\label{intheorem: def: necessity: nu}
\nu_i=\Psi_{1+r}^*(\mpi)\text{ for each }i\in[k_r],
\end{align}
where the $\mpi$'s are as defined in \eqref{intheorem: necessity: cond 2: pm(i) def}. Here the $\nu_i$'s are non-negative by   Lemma \ref{lemma: necessity: psi bounded} because $\Psi_{1+r}^*(\mpi)>0$ since $\mp,\mq\neq \mz_{k_{1+r}}$.
\begin{lemma}
\label{lemma: necessity: optimal surrogate value function in small P}
 
    Suppose $\psi$ is as in \eqref{def: product psi}, where the losses $\phi_1,\ldots,\phi_T$ are bounded and non-negative. Let $r\in[1:T-1]$. Then for any $\PP\in\Pms$, it follows that 
    \[V^\psi_*=\Psi_r^*(\nu)\prod_{t=1}^{r-1}\Psi^*_t(\mo_{k_t})\prod_{t=2+r}^{T}\Psi^*_t(\mo_{k_t}),\]
where the products are one if the range is empty and $\nu$ is as defined in \eqref{intheorem: def: necessity: nu}
  
\end{lemma}

\begin{proof}[Proof of Lemma \ref{lemma: necessity: optimal surrogate value function in small P}]
First, we want to show
\begin{align}
\label{inlemma: necessity: optimal surrogate value small p: later stages}
   V^\psi_*=  \prod_{t=2+r}^{T}\Psi_t^*(\mo_{k_t})\sup_{f_1,\ldots,f_{1+r}}\E\lbt Y_{1+r}\prod_{t=1}^{1+r}\frac{\phi_t(f_t(H_t);A_t)}{\pi_t(A_t\mid H_t)}\rbt,
\end{align}
  where the products in the integrand are $1$ if the range is empty. The equation in \eqref{inlemma: necessity: optimal surrogate value small p: later stages} trivially holds if $r=T-1$. Let us consider the case $r<T-1$. Then
  \begin{align*}
V^\psi_*=&\ \sup_{f_1,\ldots,f_T}\E\lbt Y_{1+r}\prod_{t=1}^T\frac{\phi_t(f_t(H_t);A_t)}{\pi_t(A_t\mid H_t)}\rbt\\
=&\ \sup_{f_1,\ldots,f_T}\E\lbt Y_{1+r}\prod_{t=1}^{T-1}\frac{\phi_t(f_t(H_t);A_t)}{\pi_t(A_t\mid H_t)}\E\lbt \frac{\phi_T(f_T(H_T);A_T)}{\pi_T(A_T\mid H_T)}\ \bl\ H_T\rbt\rbt\\
\stackrel{(a)}{=}&\  \sup_{f_1,\ldots,f_T}\E\lbt Y_{1+r}\prod_{t=1}^{T-1}\frac{\phi_t(f_t(H_t);A_t)}{\pi_t(A_t\mid H_t)}\Psi_T(f_T(H_T);\mo_{k_T})\rbt\\
\stackrel{(b)}{=}&\ \sup_{f_1,\ldots,f_{T-1}}\E\lbt Y_{1+r}\prod_{t=1}^{T-1}\frac{\phi_t(f_t(H_t);A_t)}{\pi_t(A_t\mid H_t)}\sup_{\mx\in\RR^{k_T}}\Psi_T(\mx;\mo_{k_T})\rbt\\
=&\ \sup_{f_1,\ldots,f_{T-1}}\E\lbt Y_{1+r}\prod_{t=1}^{T-1}\frac{\phi_t(f_t(H_t);A_t)}{\pi_t(A_t\mid H_t)}\Psi_t^*(\mo_{k_T})\rbt\\
=&\ \Psi_t^*(\mo_{k_T})\sup_{f_1,\ldots,f_{T-1}}\E\lbt Y_{1+r}\prod_{t=1}^{T-1}\frac{\phi_t(f_t(H_t);A_t)}{\pi_t(A_t\mid H_t)}\rbt
  \end{align*}
  where  step (b) uses the non-negativity of $Y_{1+r}$ and the $\phi_t$'s
and step (a) follows from \eqref{intheorem: necessity: IPW sum one}. Proceeding as above, we can show \eqref{inlemma: necessity: optimal surrogate value small p: later stages} holds. Observe that 
\begin{align*}
\MoveEqLeft    \sup_{f_1,\ldots,f_{1+r}}\E\lbt Y_{1+r}\prod_{t=1}^{1+r}\frac{\phi_t(f_t(H_t);A_t)}{\pi_t(A_t\mid H_t)}\rbt\\
=&\ \sup_{f_1,\ldots,f_{1+r}}\E\lbt \prod_{t=1}^{r}\frac{\phi_t(f_t(H_t);A_t)}{\pi_t(A_t\mid H_t)}\E\lbt Y_{1+r}\frac{\phi_{1+r}(f_{1+r}(H_{1+r});A_{1+r})}{\pi_{1+r}(A_{1+r}\mid H_{1+r})}\ \bl\ H_{1+r}\rbt \rbt\\
=&\ \sup_{f_1,\ldots,f_{1+r}}\E\lbt \prod_{t=1}^{r}\frac{\phi_t(f_t(H_t);A_t)}{\pi_t(A_t\mid H_t)}\Psi_{1+r}(f_{1+r}(H_{1+r}); \mp_2(H_{1+r})) \rbt
\end{align*}  
by \eqref{intheorem: necessity: IPW: general p}, where
\[\mp_2(H_{1+r})_j=\E[Y_{1+r}\mid H_{1+r}, A_{1+r}=j].\]
However, for our $\PP$, 
\[\E[Y_{1+r}\mid H_{1+r},A_{1+r}=j]=\E[Y_{1+r}\mid A_{r},A_{1+r}=j]=\mp^{(A_r)}_j\text{ for all }j\in[k_{1+r}]\]
Therefore, $\mp_2(H_{1+r})=\mp^{(A_r)}$.
Hence, we have shown that
\begin{align}
    \label{inlemma: necessity: cond 2: small p V psi star: stage 1+r}
   \E\lbt Y_{1+r}\frac{\phi_{1+r}(f_{1+r}(H_{1+r});A_{1+r})}{\pi_{1+r}(A_{1+r}\mid H_{1+r})}\ \bl\ H_{1+r}\rbt  =\Psi_{1+r}(f_{1+r}(H_{1+r}); \mp^{(A_r)}),
\end{align}
and, in particular, 
\begin{align*}
\MoveEqLeft  \sup_{f_1,\ldots,f_{1+r}}\E\lbt Y_{1+r}\prod_{t=1}^{1+r}\frac{\phi_t(f_t(H_t);A_t)}{\pi_t(A_t\mid H_t)}\rbt\\
=&\  \sup_{f_1,\ldots,f_{1+r}}\E\lbt \prod_{t=1}^{r}\frac{\phi_t(f_t(H_t);A_t)}{\pi_t(A_t\mid H_t)}\Psi_{1+r}(f_{1+r}(H_{1+r}); \mp^{(A_r)}) \rbt\\
=&\ \sup_{f_1,\ldots,f_{r}}\E\lbt \prod_{t=1}^{r}\frac{\phi_t(f_t(H_t);A_t)}{\pi_t(A_t\mid H_t)}\sup_{\mx\in\RR^{k_{1+r}}}\Psi_{1+r}(\mx; \mp^{(A_r)}) \rbt\\
=&\ \sup_{f_1,\ldots,f_{r}}\E\lbt \prod_{t=1}^{r}\frac{\phi_t(f_t(H_t);A_t)}{\pi_t(A_t\mid H_t)} \Psi_{1+r}^*( \mp^{(A_r)})\rbt\\
=&\ \sup_{f_1,\ldots,f_{r}}\E\lbt \prod_{t=1}^{r-1}\frac{\phi_t(f_t(H_t);A_t)}{\pi_t(A_t\mid H_t)} \E\lbt \frac{\phi_r(f_r(H_r);A_r)}{\pi_r(A_r\mid H_r)}\Psi_{1+r}^*( \mp^{(A_r)})\ \bl\ H_r\rbt\rbt
\end{align*}
where the product term is one if $r=1$.
Another application of \eqref{intheorem: necessity: IPW: general p} yields that
the above equals
\begin{align*}
 \sup_{f_1,\ldots,f_{r}}\E\lbt \prod_{t=1}^{r-1}\frac{\phi_t(f_t(H_t);A_t)}{\pi_t(A_t\mid H_t)} \Psi_r(f_r(H_r);\nu)\rbt,   
\end{align*}
where
\[\nu_j=\E[\Psi_{1+r}^*(\mp^{(A_r)})\mid H_r, A_r=j]=\Psi_{1+r}^*(\mp^{(j)})\text{ for all }j\in[k_r],\]
i.e., $\nu$ is non-stochastic and as defined in \eqref{intheorem: def: necessity: nu}.
Using the fact that the $\phi_t$'s are non-negatve, it then follows that 
\begin{align*}
 \sup_{f_{r}}\E\lbt \prod_{t=1}^{r-1}\frac{\phi_t(f_t(H_t);A_t)}{\pi_t(A_t\mid H_t)} \Psi_r(f_r(H_r);\nu)\rbt=&\ \E\lbt \prod_{t=1}^{r-1}\frac{\phi_t(\mx;A_t)}{\pi_t(A_t\mid H_t)} \sup_{\mx\in\RR^{k_r}}\Psi_r(\mx;\nu)\rbt\\
 =&\ \E\lbt \prod_{t=1}^{r-1}\frac{\phi_t(\mx;A_t)}{\pi_t(A_t\mid H_t)} \Psi^*_r(\nu)\rbt.
\end{align*}
Thus we have shown that
\begin{align*}
     \sup_{f_1,\ldots,f_{1+r}}\E\lbt Y_{1+r}\prod_{t=1}^{1+r}\frac{\phi_t(f_t(H_t);A_t)}{\pi_t(A_t\mid H_t)}\rbt= \Psi_r^*(\nu) \sup_{f_1,\ldots,f_{r-1}}\E\lbt \prod_{t=1}^{r-1}\frac{\phi_t(f_t(H_t);A_t)}{\pi_t(A_t\mid H_t)}\rbt.
\end{align*}
The proof follows by combining the above with \eqref{inlemma: necessity: optimal surrogate value small p: later stages}, and noting that, for $r-1\geq 1$,  \eqref{inlemma: necessity: product from 1 to r-1} implies that
\[\sup_{f_1,\ldots,f_{r-1}}\E\lbt \prod_{t=1}^{r-1}\frac{\phi_t(f_t(H_t);A_t)}{\pi_t(A_t\mid H_t)}\rbt=\prod_{t=1}^{r-1}\Psi^*_t(\mo_{k_t}).\]
\end{proof}

 For simpler forms of $f$'s, we can find their surrogate value function $V^\psi(f)$ explicitly under $\Pms$. 
  \begin{lemma}
\label{lemma: necessity: calculation for V psi for positive surrogates}
Suppose $\PP\in\mP_{\text{small}}$ and $\psi$ is as in \eqref{def: product psi}, where $\phi_1,\ldots,\phi_T$ are bounded and non-negative. Consider $r\in[1,T-1]$.  For $t\notin\{1+r\}$,  we let  $f_t(h_t)=\mxt$ for all $h_t\in\H_t$ where $\mxt\in\RR^{k_t}$ are fixed vectors. For $H_{1+r}=(H_r, A_r,Y_r)$, we let $f_{1+r}(H_{1+r})=\mxmo^{(A_r)}$ where,  for each $i\in[k_r]$, $\mxmo^{(i)}\in\RR^{k_{1+r}}$ is a non-stochastic vector.
Let us define
  $\gamma\in\RR^{k_r}$ is defined by 
\begin{align}
    \label{def: necessity: theta}
    \gamma_i=\Psi_{1+r}(\mxmo^{(i)};\mp^{(i)})\text{ for all }i\in[k_r].
\end{align}
Then
\begin{align*}
V^\psi(f)=\Psi_r(\mxr;\gamma)\prod_{t=1}^{r-1}\Psi_t(\mxt;\mo_{k_t})\prod_{t=r+2}^{T}\Psi_t(\mxt;\mo_{k_t}),
\end{align*}
where $\gamma$ is as defined in \eqref{def: necessity: theta},  the $\mpi$'s are as in \eqref{intheorem: necessity: cond 2: pm(i) def}, and the products are one if the ranges are empty.
\end{lemma}

\begin{proof}[Proof of Lemma \ref{lemma: necessity: calculation for V psi for positive surrogates}]
Our first step is to show that
\begin{align}
\label{inlemma: necessity: V psi calculation: later stages}
    V^\psi(f)=\lb\prod_{t=2+r}^T  \Psi_t(\mxt;\mo_{k_t})\rb\E\lbt Y_{1+r}\prod_{t=1}^{1+r}\frac{\phi_t(f_t(H_t);A_t)}{\pi_t(A_t\mid H_t)}\rbt,
\end{align}
where the product is one if the range of the product is empty.
The above follows trivially if $r=T-1$. Suppose $T>1+r$. Then 
    \begin{align*}
        V^\psi(f)=&\ \E\lbt Y_{1+r}\prod_{t=1}^T\frac{\phi_t(f_t(H_t);A_t)}{\pi_t(A_t\mid H_t)}\rbt\\
=&\ \E\lbt Y_{1+r}\prod_{t=1}^{T-1}\frac{\phi_t(f_t(H_t);A_t)}{\pi_t(A_t\mid H_t)}\E\lbt \frac{\phi_T(f_T(H_T);A_T)}{\pi_T(A_T\mid H_T)}\ \bl\ H_T\rbt\rbt\\
\stackrel{(a)}{=}&\  \E\lbt Y_{1+r}\prod_{t=1}^{T-1}\frac{\phi_t(f_t(H_t);A_t)}{\pi_t(A_t\mid H_t)}\Psi_T(f_T(H_T);\mo_{k_T})\rbt\\
=&\ \Psi_T(\mx^{(T)};\mo_{k_T})\E\lbt Y_{1+r}\prod_{t=1}^{T-1}\frac{\phi_t(f_t(H_t);A_t)}{\pi_t(A_t\mid H_t)}\rbt
    \end{align*}
where (a) follows from \eqref{intheorem: necessity: IPW sum one}.
Proceeding as above, and applying  \eqref{intheorem: necessity: IPW sum one} repeatedly, \eqref{inlemma: necessity: V psi calculation: later stages} can be proved.
Using \eqref{inlemma: necessity: cond 2: small p V psi star: stage 1+r}, we can show that
\[\E\lbt Y_{1+r}\prod_{t=1}^{1+r}\frac{\phi_t(f_t(H_t);A_t)}{\pi_t(A_t\mid H_t)}\rbt=\E\lbt \Psi_{1+r}(f_{1+r}(H_{1+r});\mp^{(A_r)})\prod_{t=1}^{r}\frac{\phi_t(f_t(H_t);A_t)}{\pi_t(A_t\mid H_t)}\rbt.\]
However, according to our definition,
\[f_{1+r}(H_{1+r})=\mxmo^{(A_r)},\]
where $A_r\in H_{1+r}$.
Therefore,
\begin{align*}
\MoveEqLeft \E\lbt Y_{1+r}\prod_{t=1}^{1+r}\frac{\phi_t(f_t(H_t);A_t)}{\pi_t(A_t\mid H_t)}\rbt\\
 =&\  \E\lbt \Psi_{1+r}(\mxmo^{(A_r)};\mp^{(A_r)})\prod_{t=1}^{r}\frac{\phi_t(f_t(H_t);A_t)}{\pi_t(A_t\mid H_t)}\rbt\\ 
 =&\ \E\lbt \E\lbt \Psi_{1+r}(\mxmo^{(A_r)};\mp^{(A_r)})\frac{\phi_r(f_r(H_t);A_r)}{\pi_r(A_r\mid H_r)}\ \bl\ H_r\rbt\prod_{t=1}^{r-1}\frac{\phi_t(f_t(H_t);A_t)}{\pi_t(A_t\mid H_t)}\rbt,
\end{align*}
where the product over the range  $\prod_{t=1}^{r-1}$ is one if $r-1=0$. Another application of \eqref{intheorem: necessity: IPW: general p} yields that
\begin{align*}
 \E\lbt \Psi_{1+r}(\mxmo^{(A_r)};\mp^{(A_r)})\frac{\phi_r(f_r(H_r);A_r)}{\pi_r(A_r\mid H_r)}\ \bl\ H_r\rbt=\Psi_r(f_r(H_r);\gamma),
\end{align*}
where
\[\gamma_i=\E[\Psi_{1+r}(\mxmo^{(A_r)};\mp^{(A_r)})\mid H_r, A_r=i]=\Psi_{1+r}(\mxmo^{(i)};\mp^{(i)}).\]
Since $f_r(H_r)=\mxr$,
\begin{align*}
 \E\lbt \Psi_{1+r}(\mxmo^{(A_r)};\mp^{(A_r)})\frac{\phi_t(f_t(H_t);A_t)}{\pi_t(A_t\mid H_t)}\ \bl\ H_r\rbt=\Psi_r(\mxr;\gamma).
\end{align*}
Therefore,
\begin{align*}
  \MoveEqLeft E\lbt \E\lbt \Psi_{1+r}(\mxmo^{(A_r)};\mp^{(A_r)})\frac{\phi_t(f_r(H_r);A_r)}{\pi_r(A_r\mid H_r)}\ \bl\ H_r\rbt\prod_{t=1}^{r-1}\frac{\phi_t(f_t(H_t);A_t)}{\pi_t(A_t\mid H_t)}\rbt\\
  =&\ \Psi_r(\mxr;\gamma)\E\lbt \prod_{t=1}^{r-1}\frac{\phi_t(f_t(H_t);A_t)}{\pi_t(A_t\mid H_t)} \rbt,
\end{align*}
where, if $r-1\geq 1$, using \eqref{intheorem: necessity: IPW sum one} repeatedly, one can show that
\begin{align*}
 \E\lbt \prod_{t=1}^{r-1}\frac{\phi_t(f_t(H_t);A_t)}{\pi_t(A_t\mid H_t)} \rbt=\prod_{t=1}^{r-1}\Psi_t(\mxt;\mo_{k_t}).  
\end{align*}
Therefore, the proof follows.
\end{proof}
\paragraph*{Step 1c. Calculating $V_*$}
Lemma \ref{lemma: necessity: optimal surrogate value function in small P} implies that, for any non-negative $\phi_1,\ldots,\phi_T$,
\[V^\psi_*=\Psi_r^*(\nu)\prod_{t=1}^{r-1}\Psi^*_t(\mo_{k_t})\prod_{t=2+r}^{T}\Psi^*_t(\mo_{k_t})\]
where
$\nu$ is as in \eqref{intheorem: def: necessity: nu}.
Suppose for each $t\in[T]$,   $\phi_t(\mx;i)=1[\pred(\mx)=i]$ where $\mx\in\RR^{k_t}$.
Then $\Psi_t(\mx;\mo_{k_t})=1$. Therefore, $\Psi_t^*(\mo_{k_t})=1$. For any $\mp\in\RR_{\geq 0}^{k_t}$ in general, $\Psi_t(\mx;\mp)=\mp_{\pred(\mx)}$. Therefore, $\Psi_t^*(\mp)=\max(\mp)$. Hence,  \eqref{intheorem: def: necessity: nu} implies $\nu_i=\Psi^*_{1+r}(\mpi)=\max(\mpi)$. Therefore,  $V_*=\max(\nu)=\max_{i\in[k_r]}\max(\mpi)$. 

 \paragraph*{Step 1d. Defining the sequences}   Now we will define a sequence of class scores $\{f^\upm\}_{m\geq 1}\subset\F$. To this end, we define some sequences of vectors first. The following fact will be used while defining these sequences: since $\Psi_t^*(\mp)=\sup_{\mx\in\RR^{k_t}}\Psi_t(\mx;\mp)$ for each $\mp\in\RR^{k_t}$, there  exists a sequence $\{\mxm\}_{t\geq 1}\subset\RR^{k_t}$  so that $\Psi_t(\mxm;\mp)\to_m\Psi_t^*(\mp)$ for each $\mp\in\RR^{k_t}_{\geq 0}$ and  $t\in[T]$. 
 
For each $i\in[k_r]$, let  $\{\mxmo^{m,i}\}_{m\geq 1}\subset\RR^{k_{1+r}}$ be a sequence that satisfies  
    \begin{align}
    \label{intheorem: necessity: def of xpm 1+r}
       \Psi_{1+r}(\mxmo^{m,i};\mp^{(i)})\to_m \Psi^*_{1+r}(\mp^{(i)}),
    \end{align}
where $\mpi$ is as in \eqref{intheorem: necessity: cond 2: pm(i) def} for each $i\in[k_r]$. 
We let $\{\xpmr\}_{m\geq 1}\subset \RR^{k_r}$ be a sequence of vectors that satisfy 
    \begin{align}
\label{intheorem: necessity: def of xpm r}
\Psi_r(\xpmr;\nu)\to_m\Psi_r^*(\nu),
    \end{align}
    where $\nu$ as in \eqref{intheorem: def: necessity: nu}.
For each $t\notin\{ r,1+r\}$ and $m\in\NN$,  define the sequences $\{\mxmt\}_{m\geq 1}\subset\RR^{k_t}$ such that $\Psi_t(\mxmt;\mo_{k_t})\to_m \Psi_t^*(\mo_{k_t})$.

Now we are ready to define $\fm$.  For $t\notin\{ r,1+r\}$,  we let $f_t^{(m)}\equiv\mxmt$, i.e., $f_t^{(m)}(h_t)=\mxmt$ for all $h_t\in\H_t$.  When $t=r$, we let $f^\upm_r\equiv\xpmr$. Thus $\fm_t$ is constant vector for all $t\neq 1+r$. However, we will consider a non-constant $\fm_{1+r}$. Note that $H_{1+r}=(H_r,A_r)$ for $\PP\in\Pms$ because $Y_r=0$ and $O_{r}=O_{1+r}=\emptyset$. Therefore, $f_{1+r}$ is a function from $\H_r\times [k_r]$ to $\RR$. For $i\in[k_r]$,   we let $f_{1+r}^\upm(H_{1+r})=f_{1+r}^\upm(H_{r},i)=\mxmo^{m,i}$.

\paragraph*{Step 1e. Showing $V^\psi(\fpm)\to_m V^\psi_*$}
Our surrogates are non-negative by assumption and they are bounded by Lemma \ref{lemma: necessity: psi bounded}. Therefore,
Lemmas \ref{lemma: necessity: optimal surrogate value function in small P} and  \ref{lemma: necessity: calculation for V psi for positive surrogates} apply.  $V^\psi_*$ is given by Lemma \ref{lemma: necessity: optimal surrogate value function in small P}.  $V^\psi(\fpm)$ is given by Lemma \ref{lemma: necessity: calculation for V psi for positive surrogates}, which yields that 
\begin{align}
    \label{intheorem: necessity: V psi}
V^\psi(\fpm)=\Psi_r(\xpmr;\gamma^\upm)\prod_{t=1}^{r-1}\Psi_t(\xpmt;\mo_{k_t})\prod_{t=r+2}^{T}\Psi_t(\xpmt;\mo_{k_t}),
\end{align}
where 
\begin{equation}
    \label{def: intheorem: necessity: theta m}
    \gamma^\upm=\Psi_{1+r}(\mxmo^{m,i};\mpi)\text{ for all }i\in[k_r],
\end{equation}
and the products are one if the ranges are empty.
Since $\Psi_t(\xpmt;\mo_{k_t})\to_m\Psi_t^*(\mo_{k_t})$ for all $t\notin\{r,1+r\}$ by construction, to show $V^\psi(\fpm)\to_m V^\psi_*$, it suffices to show that $\Psi_r(\xpmr;\gamma^\upm)\to_m\Psi^*_r(\nu)$ where $\nu\in\RR^{k_r}$ is as in \eqref{intheorem: def: necessity: nu}. 
To this end, we note that
\begin{align*}
 \MoveEqLeft \abs{\Psi_r(\xpmr;\gamma^\upm)-\Psi_r^*(\nu)}\\
  \leq &\ \bl\sum_{i=1}^{k_r}\slb\phi_r(\xpmr;i)\gamma^\upm_i-\phi_r(\xpmr;i)\nu_i\srb\bl+\abs{\sum_{i=1}^{k_r}\phi_r(\xpmr;i)\nu_i-\Psi_r^*(\nu)}\\
  \leq &\ \sup_{\mx\in\RR^{k_r},i\in[k_r]}\phi_r(\mx;i)\sum_{i=1}^{k_r}\abs{\gamma^\upm_i-\nu_i}+\abs{\Psi_r(\xpmr;\nu)-\Psi_r^*(\nu)},
\end{align*}
where we used the non-negativity of $\phi_r$.
By Lemma \ref{lemma: necessity: psi bounded}, $\sup_{\mx\in\RR^{k_r},i\in[k_r]}\phi_r(\mx;i)$ is finite. For each $i\in[k_r]$,
\begin{align*}
   |\gamma^\upm_i-\nu_i|= |\Psi_{1+r}(\mxmo^{m,i};\mpi)-\Psi^*_{1+r}(\mpi)|
\end{align*}
which goes to $0$ as $m\to\infty$ by the definition of $\mxmo^{m,i}$ in \eqref{intheorem: necessity: def of xpm 1+r}. On the other hand, $\Psi_r(\xpmr;\nu)\to_m\Psi_r^*(\nu)$ by \eqref{intheorem: necessity: def of xpm r}. Therefore, we have shown that $\Psi_r(\xpmr;\gamma^\upm)\to_m\Psi_r^*(\nu)$, which completes the proof of this step.

\paragraph*{Step 1e. Calculating $V(\fpm)$}
We will apply Lemma \ref{lemma: necessity: calculation for V psi for positive surrogates} with $\phi_t(\mx;i)=\phi_{\text{dis}}(\mx;i)=1[\pred(\mx)=i]$ for  $\mx\in\RR^{k_t}$ and  $t\in[T]$.
 Equation \ref{intheorem: necessity: V psi} yields that 
\begin{align*}
V(\fpm)=\Psi_r(\xpmr;\gamma^\upm)\prod_{t=1}^{r-1}\Psi_t(\xpmt;\mo_{k_t})\prod_{t=r+2}^{T}\Psi_t(\xpmt;\mo_{k_t}),
\end{align*}
where $\phi_t\equiv\phi_{\text{dis}}$ and  $\gamma^\upm$ is as in \eqref{def: intheorem: necessity: theta m}. 
 For $\phi_{\text{dis}}$, we have already shown that $\Psi_t(\mx;\mp)=\mp_{\pred(\mx)}$ for any $\mx\in\RR^{k_t}$ and $\mp\in\RR_{\geq 0}^{k_t}$. Therefore, $\Psi_t(\xpmt;\mo_{k_t})=1$ for each $t\notin\{r,1+r\}$ and $\Psi_r(\xpmr;\gamma^\upm)=\gamma^\upm_{\pred(\xpmr)}$.
Therefore,
\[V(\fpm)=\gamma^\upm_{\pred(\xpmr)}.\]

\paragraph*{Step 1f. Implication of Fisher consistency}
Suppose $\phi_t$'s are Fisher consistent. Since  $V^\psi(\fpm)\to_m V^\psi_*$ by Step 1d, Fisher consistency implies $V(\fpm)\to_m V_*$.  Step 1e shows that that $V(\fpm)=\gamma^\upm_{\pred(\xpmr)}$  and  Step 1c shows that $V_*=\max_{i\in[k_r]}\max(\mpi)$.  Therefore, Fisher consistency yields that    $\gamma^\upm_{\pred(\xpmr)}\to_m\max_{i\in[k_r]}\max(\mpi)$. We will use use this convergence to  show that  $\pred(\xpmr)=1$ for all sufficiently large $m$. The definition of $\xpm$ in \eqref{intheorem: necessity: def of xpm r} implies that  $\Psi_r(\xpmr;\nu)\to_m \Psi_r^*(\nu)$. Since we proved that $\phi_r$ satisfies Condition \ref{assump: N1},  Fact \ref{fact: Cond N1 implies pred x in argmax p} applies. Hence, 
 $\pred(\xpmr)\in\argmax(\nu)$ for all sufficiently large $m$. Therefore, if $\pred(\xpmr)=1$ for all sufficiently large $m$, it will follow that  $1\in\argmax(\nu)$. Then the definition of $\nu$ in  \eqref{intheorem: def: necessity: nu} would imply that $\Psi_{1+r}^*(\mp)\geq \Psi_{1+r}^*(\mq)$, completing the proof of Step 1. Thus it suffices to show that $\pred(\xpmr)=1$ for all sufficiently large $m$.

To this end, note that  \eqref{def: necessity: theta} implies 
$\gamma^\upm_i=(\mpi)_{\pred(\mxmo^{m,i})}\leq \max(\mpi)$ for each $i\in[k_r]$. Let us denote the vector  $v^0=(\max(\mp^{(1)}),\ldots,\max(\mp^{(k_r)}))$. Therefore, $\gamma^\upm_i\leq \mv^0_i$ for each $i\in[k_r]$. Hence, it  holds that  
\[\gamma^\upm_{\pred(\xpmr)}\leq v^0_{\pred(\xpmr)}\leq \max(v^0).\]
 Since $\max(v^0)=\max_{i\in[k_r]}\max(\mpi)$ and  Fisher consistency implies $\gamma^\upm_{\pred(\xpmr)}\to_m\max_{i\in[k_r]}\max(\mpi)$, it follows that $\gamma^\upm_{\pred(\xpmr)}\to_m\max(v^0)$. Therefore, it follows that $v^0_{\pred(\xpmr)}\to_m\max(v^0)$. Since $\argmax(v^0)=\argmax_{i\in[k_r]}\max(\mpi)=1$ for $\PP\in\Pms$,  it follows that $\pred(\xpmr)=1$ for all sufficiently large $m$.

\subsubsection{Step 2: Proving $\max(\mp)=\max(\mq)$ 
implies $\Psi_t^*(\mp)=\Psi_t^*(\mq)$ for all $\mp,\mq\in\RR_{\geq 0}^{k_t}$ and $t>1$}
 
Suppose $t\in[2:T]$.
Consider $\mpm=\mp+(1/m)\mo_{k_2}$. Then $\max(\mpm)=\max(\mp)+1/m>\max(\mq)$. Taking $r=t-1$ in the previous step, we obtain that 
$\Psi_t^*(\mpm)\geq \Psi_t^*(\mq)$. Therefore,
$\lim_{m\to\infty}\Psi_t^*(\mpm)\geq \Psi_t^*(\mq)$. Since $\mpm\to_m\mp$ and $\Psi_t^*$ is continuous for each $t\in[T]$ by Lemma \ref{lemma: necessity: Psi-t cont.}, it follows that $\Psi^*_t(\mp)\geq \Psi^*_t(\mq)$. Switching the role between $\mp$ and $\mq$, we can similarly prove that $\Psi_t^*(\mp)\leq\Psi_t^*(\mq)$, which implies $\Psi_t^*(\mp)=\Psi_t^*(\mq)$. Therefore, the function $\Psi_t^*:\RR_{\geq 0}^{k_t}\mapsto\RR$ satisfies $\Psi_t^*(\mp)=G_t(\max(\mp))$ for all $\mp\in\RR_{\geq 0}^{k_t}$ for some $G_t:\RR\to\RR$.

\subsubsection{Step 3: Showing that $G_t$ is linear on $\RR_{\geq 0}$ for $t\in[2:T]$}
Now, we will show that $G_t$ is positively homogenous on $\RR_{\geq 0}$ in that $G_t(x)=C_t x$ for some $C_t>0$ for all $x\geq 0$. Note that if $x\geq 0$, then 
\[G_t(x)=G_t(\max(x\mo_{k_t}))=\Psi_t^*(x\mo_{k_t}).\]
However,
\[\Psi_t^*(x\mo_{k_2})=x\sup_{\mx\in\RR^{k_t}}\sum_{i=1}^{k_t}\phi_t(\mx;i)=x\Psi_t^*(\mo_{k_t}).\] 
Hence, we have shown that $\Psi_t^*(\mx)=\Psi_t^*(\mo_{k_t})\max(\mx)$ for all $\mx\in\RR_{\geq 0}^{k_t}$. The proof follows noting $\Psi_t^*(\mo_{k_t})>0$ by Lemma \ref{lemma: necessity: psi bounded}.

\section{Proofs of the lemmas from Section \ref{sec: necessity}}
\subsection{Single-stage case: proof of Lemma \ref{lemma: single-stage}}
\label{secpf: single-stage lemma}
\begin{proof}[Proof of Lemma \ref{lemma: single-stage}]
    The sufficiency part follows directly from  Theorem \ref{theorem: sufficient conditions}. We will only consider the proof of the necessity of Condition \ref{assump: N1}.
 In this case $T=1$ and $f=f_1$. 
 Note that given any $\mp\in\RR^k_{>0}$, we can find a $\PP$ satisfying Assumptions I-V so that  $\E[Y_1\mid H_1, A_1=i]=\mp_i$ for all $i\in[k]$. In this case, 
$Q_1^*(H_1,A_1)=\E[Y_1\mid H_1, A_1]=\mp_{A_1}$. 
From Lemma \ref{lemma: necessity}, it follows that
\[
 V^\psi(f)=   \sup_{f_1\in\F_1}\E\lbt \frac{\phi_1(f_1(H_1);A_1)}{\pi_1(A_1\mid H_1)}Y_1\rbt=\Psi_1^*(\mp).
\]
Equation \ref{intheorem: necessity Cond 1: V psi first deduction} implies that $V^\psi(f)=\E[\Psi_1(f_1(H_1);\mp)]$. 
The rest of the proof follows similarly to Section \ref{secpf: Cond 1}, and hence skipped. 
\end{proof}
\subsection{Proof of Lemma \ref{lemma: necessity: cond N2 and Cond N2'}}
\label{secpf: pf of cond N2 and N2'}
\begin{proof}[Proof of Lemma \ref{lemma: necessity: cond N2 and Cond N2'}]
In this proof, we will use that $\Psi^*(\mo_k)>0$, which follows from Lemma \ref{lemma: necessity: psi bounded} since $\phi$ satisfies Condition \ref{assump: N1}.

   \textbf{Only if part: }{Suppose Condition \ref{assump: N2} holds.  Without loss of generality, let us assume that $\Psi^*(\mo_k)=1$. Otherwise, we can replace $\phi$ by $\phi/\Psi^*(\mo_k)$. Since $\Psi^*(\mo_k)=1$,  $C_\phi$  is one. Thus, we need to show that $\S^{k-1}\subset\conv(\overline{\mV}_\phi)$.  Note that  it suffices to show
    \begin{align}
        \label{inlemma: necessity: Cond N2 equiv set}
        \{\mathbf{e}_k^{(1)},\ldots,\mathbf{e}_k^{(k)}\}\subset\overline{\mV}_\phi
    \end{align}
   because $\S^{k-1}\subset\conv(\overline{\mV}_\phi)$ would follow taking convex hull on both sides of \eqref{inlemma: necessity: Cond N2 equiv set}.  Convex hull of a set in $\RR^k$ is same as that of its closure, cf. page 31 of \cite{hiriart}. Therefore,  $\conv(\overline{\mV}_\phi)=\conv({\mV}_\phi)$, which would complete the proof of $\S^{k-1}\subset\conv({\mV}_\phi)$. 
  To prove \eqref{inlemma: necessity: Cond N2 equiv set}, will show that $\mathbf{e}_k^{(1)}\in \overline{\mV}_\phi$ because the proof will follow similarly for the other $\mathbf{e}_k^{(i)}$'s. Consider $\mp\in\RR^k_{>0}$ such that $\argmax(\mp)=1$. Then from Lemma \ref{lemma: sufficiency: tilde f form}, it holds that there exists a sequence $\{\mvm\}_{m\geq 1}\subset \mV_\phi$ such 
  that $\mvm_i\to_m 0$ if $i\neq 1$ and $\langle \mvm,\mp\rangle\to_m \mp_1$. Since $\mp_1>0$, $\mvm_1\to_m 1$. Here Lemma  \ref{lemma: sufficiency: tilde f form} applies because both Conditions \ref{assump: N1} and \ref{assump: N2} hold with $C^\phi=1$.  Therefore, $\mvm\to_m \mathbf{e}_k^{(1)}$. Hence, it follows that $\mathbf{e}_k^{(1)}\in \overline{\mV}_\phi$. Therefore, the proof of \eqref{inlemma: necessity: Cond N2 equiv set} follows.}

\vspace{1em}
  \textbf{If part:} Suppose $C_\phi\S^{k-1}\subset \conv({\mV}_\phi)$ where $C_\phi=\Psi^*(\mo_k)$. Without loss of generality, we assume that  $\Psi^*(\mo)=1$. Otherwise, we can replace $\phi$ by $\phi/\Psi^*(\mo_k)$. Thus, we have $\S^{k-1}\subset \conv({\mV}_\phi)$. Hence, 
 $ \varsigma_{\S^{k-1}}\leq \varsigma_{\conv({\mV}_\phi)}$ \citep[cf. Theorem 3.3.1, p. 151,][]{hiriart}. Also, the support function of $\S^{k-1}$ equals the support function of $  \{\mathbf{e}_k^{(1)},\ldots,\mathbf{e}_k^{(k)}\}$  \citep[cf. Prop 2.2.1, p. 137,][]{hiriart}. Hence, for any $\mp\in\RR^{k_1}$, 
 \[\varsigma_{\S^{k-1}}(\mp)=\max\{\mp_1,\ldots,\mp_k\}=\max(\mp).\]
 Therefore, we have obtained that for any $\mp\in\RR^k$,
 $\max(\mp)\leq \varsigma_{\conv({\mV}_\phi)}(\mp)=\Psi^*(\mp)$.
 If $\mp\in\RR^k_{\geq 0}$, then 
 \[\Psi^*(\mp)\leq \max(\mp)\Psi^*(\mo_k)=\max(\mp)\]
 because we assumed $\Psi^*(\mo_k)=1$. Therefore, for $\mp\in\RR^k_{>0}$, $\Psi^*(\mp)=\max(\mp)$. Hence, Condition \ref{assump: N2} holds.
\end{proof}
\subsection{Proof of Lemma  \ref{lemma: location-scale transformation}}
\label{secpf: lemma location-scale}
\begin{proof}[Proof of Lemma \ref{lemma: location-scale transformation}]
 Note that for any $\mp\in\RR_{\geq 0}^k$,
 \[\sup_{\mx\in\RR^k}\Psi_{a,b}(\mx;\mp)=b\sup_{\mx\in\RR^k}\Psi(a\mx;\mp)=b\Psi^*(\mp)=bC^{\phi}\max(\mp)\]
 because $\phi$ satisfies  Condition \ref{assump: N2}.
 Hence, $\phi_{a,b}$ satisfies Condition \ref{assump: N2}. 
 Also,
  \[\sup_{\mx:\pred(\mx)\notin\argmax(\mp)}\Psi_{a,b}(\mx;\mp)=b\sup_{\mx\in\RR^k:\pred(\mx)\notin\argmax(\mp)}\Psi(a\mx;\mp).\]
  Note that $\pred(a\mx)=\pred(\mx)$ since $a>0$. Therefore, \begin{align*}
 \sup_{\mx\in\RR^k:\pred(\mx)\notin\argmax(\mp)}\Psi(a\mx;\mp)
   =&\ \sup_{\mx:\pred(a\mx)\notin\argmax(\mp)}\Psi(a\mx;\mp)\\
   =&\ \sup_{\mx\in\RR^k:\pred(\mx)\notin\argmax(\mp)}\Psi(\mx;\mp)
 <  \Psi^*(\mp)
  \end{align*}
 since $\phi$ satisfies \eqref{inlemma: necessity: single-stage FC}. Thus it follows that 
 \[b\sup_{\mx\in\RR^k:\pred(\mx)\notin\argmax(\mp)}\Psi_{a,b}(\mx;\mp)<b\Psi^*(\mp)=\sup_{\mx\in\RR^k}\Psi_{a,b}(\mx;\mp),\]
 indicating that $\phi_{a,b}$ satisfies Condition \ref{assump: N1}. Hence, we showed that  $\phi_{a,b}$ satisfies Condition \ref{assump: N1} and Condition \ref{assump: N2}.

With an abuse of notation, we denote the surrogate $\phi+c$ by $\phi_c$ and the corresponding $\Psi$ and $\Psi^*$ transforms as $\Psi_c$ and $\Psi^*_c$, respectively. Observe that
\begin{align*}
 \sup_{\mx\in\RR^k:\pred(\mx)\notin\argmax(\mp)}\Psi_c(\mx;\mp)=&\ \sup_{\mx\in\RR^k:\pred(\mx)\notin\argmax(\mp)}\Psi(\mx;\mp)+c\sum_{i\in[k]}\mp_i   \\
\stackrel{(a)}{<} &\ \sup_{\mx\in\RR^p}\Psi(\mx;\mp)+c\sum_{i\in[k]}\mp_i= \sup_{\mx\in\RR^p}\Psi_c(\mx;\mp)=\Psi_c^*(\mp),
\end{align*}
where (a) follows because $\phi$ satisfies Condition \ref{assump: N1}. 
Therefore, $\phi_c$ satisfies Condition \ref{assump: N1}.
 However, 
 \[\Psi_c^*(\mp)=\sup_{\mx\in\RR^p}\Psi_c(\mx;\mp)=\Psi^*(\mp)+c\sum_{i\in[k]}\mp_i=C^\phi\max(\mp)+c\sum_{i\in[k]}\mp_i\]
 because $\phi$ satisfies Condition \ref{assump: N2}. Therefore, $\phi_c$ does not satisfy  Condition \ref{assump: N2} unless $c=0$.
\end{proof}

\subsection{Restricted class: Proof of Lemma \ref{lemma: theorem with linear policy}}
\label{sec: pf of linear}
As in Proposition~\ref{prop: multi-cat FC}, it suffices to prove the result for the case where $C_{\phi_t} = 1$ for all $t \in [T]$, since the general case then follows by replacing $\phi_t$ with $\phi_t / C_{\phi_t}$.
Note that since $af\in\mathcal L$ for all $a>0$ and $f\in\mathcal L$, if we can prove  that  $\lim_{a\to\infty}V^\psi(af^*)=\sup_{f\in\F}V^\psi(f)$, then 
$\sup_{f\in\F}V^\psi(f)=\sup_{f\in\mathcal L}V^\psi(f)$, which, combined with Proposition   \ref{prop: multi-cat FC}, completes the proof. Hence, it suffices to show that $\lim_{a\to\infty}V^\psi(af^*)=V^\psi_*$. Lemma \ref{lemma: sufficiency} implies that under Conditions \ref{assump: N1} and \ref{assump: N2}, $V^\psi_*=\E[\Psi_1^*(\mp_1^*(H_1))]$  where $\mp_1^*(H_1)$ is as in \eqref{def: sufficiency: pt star}. However, Condition \ref{assump: N2} implies $\Psi_1^*(\mp_1^*(H_1))=\max(\mp_1^*(H_1))$. Lemma \ref{lemma: sufficiency: p t star and Q t star} implies $\mp_1^*(H_1)=Q_1^*(H_1)$. Thus $V^\psi_*=\E[Q_1^*(H_1)]$. However,  Fact \ref{fact: Q function expression} implies $V_*=\E[Q_1^*(H_1)]$. Therefore, $V^\psi_*=V_*$. 
Hence, it suffices to show that $\lim_{a\to\infty}V^\psi(af^*)=V_*$.

 If $\argmax(f^*_t(H_t))$ is singleton, then  Condition  \ref{assump: N3} implies
 \begin{align*}
\lim_{a\to\infty}\phi(af^*_t(H_t);j)=1[j=\argmax(f^*_t(H_t))].
\end{align*}
Thus
\[\PP\lb\lim_{a\to\infty}\prod_{t=1}^T\frac{\phi(af^*_t(H_t);A_t)}{\pi_t(A_t\mid H_t)}=\prod_{t=1}^T\frac{1[j=\argmax(f_t^*(H_t))]}{\pi_t(A_t\mid H_t)}\rb=1.\]
Under the setup of the current lemma, $\argmax (f_t^*(H_t))=d_t^*(H_t)$ for some version of $d_t^*$ with $\PP$-probability one. Therefore, the above implies
\[\PP\lb\lim_{a\to\infty}\prod_{t=1}^T\frac{\phi(af^*_t(H_t);A_t)}{\pi_t(A_t\mid H_t)}=\prod_{t=1}^T\frac{1[j=d_t^*(H_t)]}{\pi_t(A_t\mid H_t)}\rb=1.\]

Since Condition \ref{assump: N1} is satisfied, 
Lemma \ref{lemma: necessity: psi bounded} implies that the $\phi$'s are all bounded.  Assumption I implies that the $\pi_t$'s are all bounded below and Assumption IV implies that the $Y_t$'s are bounded above. Therefore,
\[\slb\sum_{i=1}^TY_i \srb\prod_{t=1}^T\frac{\phi(af^*_t(H_t);A_t)}{\pi_t(A_t\mid H_t)}<C\]
for some constant $C>0$ potentially depending on $\PP$.
Hence, the dominated convergence theorem implies that  
\begin{align*}
   \lim_{a\to\infty} V^\psi(af^*)=&\ \lim_{a\to\infty}\E\lbt \slb\sum_{i=1}^TY_i \srb\prod_{t=1}^T\frac{\phi(af^*_t(H_t);A_t)}{\pi_t(A_t\mid H_t)}\rbt\\
   =&\ \E\lbt\slb\sum_{i=1}^TY_i \srb\prod_{t=1}^T\frac{1\slbt A_t=d_t^*(H_t)\srbt}{\pi_t(A_t\mid H_t)}\rbt,
\end{align*}
which equals $V_*$, thus completing the proof.

\subsection{Proof of the results for the kernel-based surrogates in Section \ref{sec: kernel based surrogates}}
\label{secpf: kernel based}

\begin{proof}[Proof of Lemma \ref{lemma: kernel based satisfies J condition}]
Without loss of generality, let us assume $C=1$ in   \eqref{kernel based: alt form}. 
Since $\pred(\mx)=\max(\argmax(\mx))$ for all $\mx\in\RR^k$,  \eqref{kernel based: alt form} implies
\begin{align}
 \label{def: psi smoothed}
 \phi(\mx,j)=&\  \dint_{\RR^k} 1[\pred(\mx-\mbu)=j]K(\mbu)d\mbu\nn\\
    =&\ \dint_{\RR^k}\prod_{i=j+1}^k 1[\mx_j-Z_j>\mx_i-Z_i]\prod_{i=1}^{j-1} 1[\mx_j-Z_j\geq \mx_i-Z_i]d P_Z\nn\\
   \stackrel{(a)}{=}&\ P_Z\lb Z_j\leq Z_i+\mx_j-\mx_i\text{ for all }i\neq j,\ i\in[k]\rb
 \end{align}
 where $(a)$ uses the fact that $P_Z(Z_j=Z_i+\mx_j-\mx_i)=0$  for all $i$ and $j$ in $[k]$, which follows because the $Z_i$'s have density. Now
 \begin{align*}
 \MoveEqLeft   P_Z\lb Z_j\leq Z_i+\mx_j-\mx_i\text{ for all }i\neq j,\ i\in[k]\rb\\
    =&\ E_{\KK}\lbt P_Z\lb Z_j\leq Z_i+\mx_j-\mx_i\text{ for all }i\neq j,\ i\in[k]\bl Z_j\rb\rbt\\
    =&\ E_{\KK}\lbt\prod_{i\neq j} P_Z\lb Z_j\leq Z_i+\mx_j-\mx_i\bl Z_j\rb\rbt
 \end{align*}
because $Z_1,\ldots,Z_k$ are i.i.d. given $Z_j$. 
Since they are independent of $Z_j$ as well, the above equals
\begin{align*}
    \MoveEqLeft E_{\KK}\lbt \prod_{i\neq j}\slb 1-F_{\KK}(Z_j+\mx_i-\mx_j)\srb\rbt= E_{\KK}\lbt \prod_{i\neq j}\slb  1-F_{\KK}(Z+\mx_i-\mx_j)\srb\rbt,
\end{align*}
where $Z$ is a random variable with distribution $F_{\KK}$. Thus the first part of the lemma follows. 
Suppose $j\in\argmax(\mx)$. Then $\mx_j-\mx_i\geq 0$ for all $i\in[k]$. Thus, 
\[1-F_{\KK}(Z+\mx_i-\mx_j)\geq 1-F_{\KK}(Z).\]
Therefore, $\phi(\mx;\pred(\mx))\geq E_{\KK}[(1-F_{\KK}(Z))^{k-1}]$.
Now note that since $Z\sim F_{\KK}$, $F_{\KK}(Z)\sim \mathbb U(0,1)$. Hence, $1-F_{\KK}(Z)\equiv U\sim \mathbb U(0,1)$ as well. Thus $  \phi(\mx,\pred(\mx))\geq \E[\mathbb U^{k-1}]=1/k$. Hence, the second part of the proof  follows.

\end{proof}

\subsubsection{Forms of the kernel-based $\phi$ in special cases}
\label{sec: short: calc for kernel logistic}

Lemma \ref{lemma: kernel based satisfies J condition} implies 
\[\phi(\mx;j)=CE_{\KK}\lbt \prod_{i\neq j} \overline F_{\KK}(Z+\mx_i-\mx_j)\rbt\quad\text{ for all } j\in[k],\]
where $\overline F_{\KK}=1-F_{\KK}$.
If $\KK$ is symmetric about 0, then 
\[\phi(\mx;j)=CE_{\KK}\lbt \prod_{i\neq j} F_{\KK}(\mx_j-\mx_i-Z)\rbt\quad\text{ for all } j\in[k].\]
Suppose $k=3$ and $C=1$. Then it follows that
\[\phi(\mx;i)=\begin{cases}
  \eta_s(\mx_1-\mx_2,\mx_1-\mx_3) & \text{ if }i=1\\
    \eta_s(\mx_2-\mx_3,\mx_2-\mx_1) & \text{ if }i=2\\
      \eta_s(\mx_3-\mx_1,\mx_3-\mx_2) & \text{ if }i=3
\end{cases}\]
where for $x,y\in\RR$, 
\begin{equation}
\label{kernel: common formula}
  \eta_s(x,y)=\begin{cases}
        E_{\KK}[F_{\KK}(x-Z)F_{\KK}(y-Z)] & \text{ if } \KK \text{ is symmetric about }0\\
         E_{\KK}[\overline{F}_{\KK}(Z-x)\overline F_{\KK}(Z-y)] & \text{ otherwise. }
    \end{cases}
\end{equation}
In Section~\ref{sec: convex loss}, we used $\eta$ to denote the template for two-stage PERM losses, so $\eta_s$ may appear to be an overloaded notation. However, the expressions above show that $\phi$ is a single-stage PERM loss with template $\eta_s$ \citep[see also][for the single-stage definition of PERM losses]{wang2023unified}. Here, the subscript $s$ indicates "single stage."

Since logistic is a symmetric (w.r.t. 0) distribution, when $\KK$ is  the standard logistic density, straightforward algebra yields
\[\eta_s(a,b)=\frac{e^{a+b}\left(a\left(e^b-1\right)^2+\left(e^a-1\right)\left(-e^ab+\left(e^b-1\right)\left(e^a-e^b\right)+b\right)\right)}{\left(e^a-1\right)^2\left(e^b-1\right)^2\left(e^a-e^b\right)}\]
for all $a,b\in\RR$.
When $\KK$ is the standard Gumbel density, 
straightforward algebra yields
\[\eta_s(a,b)=\frac{e^{a+b}-1}{(1+e^a)(1+e^b)}+\frac{1}{1+e^{a}+e^{b}}\text{ for all }a,b\in\RR.\]



\subsubsection{Proof of Lemma \ref{lemma: FC of smoothed FC}}
\label{secpf: kernel based surrogate N1 and N2}
\begin{proof}[Proof of Lemma \ref{lemma: FC of smoothed FC}]
Without loss of generality, we assume $C=1$. Therefore, to show $C_\phi=C$, it suffices to show that $C_\phi=1$. 

First we will show that
$\sup_{\mx\in\RR^k}\Psi(\mx;\mp)=\max(\mp)$, which would imply that $\phi$ satisfies Condition \ref{assump: N2} with $C_\phi=1$. Note that
\begin{align*}
    \Psi(\mx;\mp)= &\ \sum_{j\in[k]}\mp_j\dint_{\RR^k}1[\pred(\mx-\mbu)=j]K(\mbu)d\mbu\\
    &\ =\dint_{\RR^k}\slb \sum_{j\in[k]}\mp_j 1[\pred(\mx-\mbu)=j]\srb K(\mbu)d\mbu\\
    &\ =\dint_{\RR^k} \mp_{\pred(\mx-\mbu)} K(\mbu)d\mbu
    \leq  \max(\mp)\dint_{\RR^k} K(\mbu)d\mbu.
\end{align*}
Since $\int K(\mbu)d\mbu=1$, we have  $\Psi(\mx;\mp)\leq \max(\mp)$. Suppose $j\in\argmax(\mx)$. Let  $x_n=n\mathbf{e}^{(k)}_j$. Then 
\begin{align*}
\lim_{n\to\infty}\Psi(x_n;\mp)=\lim_{n\to\infty} \dint_{\RR^k} \mp_{\pred(\mx_n-\mbu)}K(\mbu)d\mbu=\dint_{\RR^k} \lim_{n\to\infty} \mp_{\pred(\mx_n-\mbu)}K(\mbu)d\mbu
\end{align*}
by the dominated convergence theorem since $\mp_{\pred(\mx_n-\mbu)}\leq \max(\mp)$ and $\int K(\mbu)d\mbu=1$. However, 
\[\lim_{n\to\infty} \mp_{\pred(\mx_n-\mbu)}=\mp_{\pred(\mx_n)}=\mp_{\pred(\mathbf{e}^{(k)}_j)}=\mp_j.\] Therefore, $  \lim_{n\to\infty}\Psi(x_n;\mp)=\mp_j=\max(\mp)$. Hence, we have proved that  $\sup_{\mx\in\RR^k}\Psi(\mx;\mp)=\max(\mp)$, which implies Condition \ref{assump: N2} holds with $C_\phi=1$. 

Now we will show that Condition \ref{assump: N1} holds. To this end, it suffices to show that \eqref{def of mathcal C psi} holds. 
Suppose $\pred(\mx)\notin\argmax(\mp)$. 
Then $\Psi(\mx;\mp)$ equals
\begin{align*}
     &\ \dint_{\RR^k} \mp_{\pred(\mx-\mbu)} K(\mbu)d\mbu\\
    =&\ \dint_{\mbu: \pred(\mx-\mbu)=\pred(\mx)}\mp_{\pred(\mx-\mbu)} K(\mbu)d\mbu+\dint_{\mbu: \pred(\mx-\mbu)\neq\pred(\mx)}\mp_{\pred(\mx-\mbu)} K(\mbu)d\mbu\\
    \leq &\ \mp_{\pred(\mx)}\dint_{\mbu: \pred(\mx-\mbu)=\pred(\mx)} K(\mbu)d\mbu+\max(\mp)\dint_{\mbu: \pred(\mx-\mbu)\neq\pred(\mx)} K(\mbu)d\mbu\\
    =&\ \mp_{\pred(\mx)}      P_{K}\slb \pred(\mx-\vec{Z})=\pred(\mx)\srb+\max(\mp)      P_{K}\slb \pred(\mx-\vec{Z})\neq\pred(\mx)\srb
\end{align*}
where $\vec{Z}=(Z_1,\ldots,Z_k)$ with the $Z_i$'s being i.i.d. with common density $\KK$, and $P_{K}$ is the  probability measure induced by $K$, the joint density  of the random vector  $\vec{Z}$.
Therefore, we have obtained that
\[\Psi(\mx;\mp)\leq \max(\mp)-   P_{K}\slb \pred(\mx-\vec{Z})=\pred(\mx)\srb\slb \max(\mp)-\mp_{\pred(\mx)}\srb,\]
which, noting $\Psi^*(\mp)=\max(\mp)$, implies
\begin{align}
\label{inlemma: kernel: oi FC stuff}
    \Psi^*(\mp)-\Psi(\mx;\mp)\geq      P_{K}\slb \pred(\mx-\vec{Z})=\pred(\mx)\srb\slb \max(\mp)-\mp_{\pred(\mx)}\srb.
\end{align}
Observe that if $\pred(\mx)=\pred(-\vec{Z})$, then $\argmax(\mx-\vec{Z})\ni\pred(\mx)$. 
Letting $P_Z$ denote the probability measure corresponding to the distribution of the $Z_i$'s,  we obtain that 
\begin{align*}
       P_{K}\slb \argmax(\mx-\vec{Z})\ni\pred(\mx)\srb
    \geq &\       P_{K}\slb \pred(-\vec{Z})=\pred(\mx)\srb\\
  \geq &\ P_{Z}\slb -Z_{\pred(\mx)}> -Z_i, i\neq \pred(\mx)\srb\\
  =&\ P_{Z}\slb -Z_1> -Z_i, i\in[2:k]\srb,
\end{align*}
where the last step follows because  $-Z_1,\ldots,-Z_k$ are i.i.d., and hence, exchangeable, and $\mx$ is non-stochastic. The RHS of the above display is $2^{-(k-1)}$ because the $Z_i$'s have a density. 
Now note that
\begin{align*}
          P_{K}\slb \argmax(\mx-\vec{Z})\ni\pred(\mx)\srb\leq &\      P_{K}\slb \pred(\mx-\vec{Z})=\pred(\mx)\srb\\
     &\ +      P_{K}\slb \argmax(\mx-\vec{Z})\text{ is not singleton}\srb,
\end{align*}
whose second term vanishes because $\vec{Z}$ has a density. Thus we have shown that
\begin{align*}
       P_{K}\slb \pred(\mx-\vec{Z})=\pred(\mx)\srb\geq 2^{-(k-1)}.  
\end{align*}
Therefore, \eqref{inlemma: kernel: oi FC stuff} implies  \eqref{def of mathcal C psi} holds with $\CC_\phi=1$.

It remains to show that Condition \ref{assump: N3} holds. To this end, it suffices to show that if $\argmax(\mx)=j$, then  $ \lim_{a_n\to\infty}   \phi(a_n\mx,j)=1$. By \eqref{kernel based: alt form}, 
    \begin{align*}
  \liminf_{a_n\to\infty}   \phi(a_n\mx,j)=&\ \liminf_{a_n\to\infty}  E_{\KK}\lbt \prod_{i\in[k]:i\neq j}\slb  1-F_{\KK}(Z+a_n\mx_i-a_n\mx_j)\srb\rbt\\
    \stackrel{(a)}{\geq }  &\ E_{\KK}\lbt  \liminf_{a_n\to\infty} \prod_{i\in[k]:i\neq j}\slb  1-F_{\KK}(Z+a_n\mx_i-a_n\mx_j)\srb\rbt\\
    =&\  E_{\KK}\lbt \prod_{i\in[k]:i\neq j}\slb  1- \liminf_{a_n\to\infty} F_{\KK}(Z+a_n\mx_i-a_n\mx_j)\srb\rbt\\
    \stackrel{(b)}{=} &\  E_{\KK}\lbt \prod_{i\in[k]:i\neq j}\slb  1-  F_{\KK}(Z+\liminf_{a_n\to\infty}(a_n\mx_i-a_n\mx_j))\srb\rbt,
         \end{align*}
         where (a) follows by  Fatou's lemma since $0\leq F_{\KK}\leq 1$ and (b) follows because $F_{\KK}$ is continuous on $\RR$ since it has a density.   Since $\mx_j-\mx_i>0$ for all $i\in[k]$ such that $i\neq j$, 
        $a_n\mx_i-a_n\mx_j\to_n -\infty$. Thus, the RHS on the above display is one.
Thus we have shown that $ \liminf_{a_n\to\infty}   \phi(a_n\mx,j)\geq 1$. However, $\phi(\mx;j)\leq \Psi(\mx,\mo_k)=1$ since $\Psi(\mx,\mo_k)=1$ for the kernel-based surrogate. Therefore,  $ \lim_{a_n\to\infty}   \phi(a_n\mx,j)= 1$ in this case.
         

Now suppose $\argmax(\mx)$ is singleton and $j\notin\argmax(\mx)$. Then 
\[ 1-F_{\KK}(Z+a_n\mx_{\pred(\mx)}-a_n\mx_j)\as 0.\]
Therefore,
\[\prod_{i\in[k]:i\neq j}\slb 1-F_{\KK}(Z+a_n\mx_{i}-a_n\mx_j)\srb \as 0.\]
Therefore, by the dominated convergence theorem, it follows that
\begin{align*}
  \MoveEqLeft \liminf_{a_n\to\infty}  E_{\KK}\lbt \prod_{i\in[k]:i\neq j}\slb  1-F_{\KK}(Z+a_n\mx_i-a_n\mx_j)\srb\rbt\\
   =&\   E_{\KK}\lbt \liminf_{a_n\to\infty}\prod_{i\in[k]:i\neq j}\slb  1-F_{\KK}(Z+a_n\mx_i-a_n\mx_j)\srb\rbt,
\end{align*}
which equals  zero. The proof follows since
\[\liminf_{a_n\to\infty}   \phi(a_n\mx,j)=\liminf_{a_n\to\infty}  E_{\KK}\lbt \prod_{i\in[k]:i\neq j}\slb  1-F_{\KK}(Z+a_n\mx_i-a_n\mx_j)\srb\rbt.\]


\end{proof}

\subsection{Proof of Lemma \ref{lemma: FC: product type surrogate loss} for the product-based losses}
 \label{secpf: product loss FC}

 \begin{proof}[Proof of Lemma \ref{lemma: FC: product type surrogate loss}]
 Without loss of generality, let us assume that $C=1$, which implies
 $\tau=1-F_\KK$. Because $\KK$ is symmetric about zero, $1-F_\KK=F_\KK$, implying $\tau=F_\KK$ in this case.  Since $\tau$ is a distribution function in this case, it is non-decreasing. Hence, for the surrogate in \eqref{def: multi-cat: Fisher consistent psi},  $\phi(\mx;\pred(\mx))\geq \tau(0)^{k-1}$.  If $\KK$ is symmetric about zero,  then $\tau(0)=\int_{-\infty}^\infty 1[y\leq 0]\KK(y)dy=1/2$.  Hence, $\J=2^{-(k-1)}$ for such $\phi$'s.

    \noindent\textbf{Proving Conditions \ref{assump: N1},  \ref{assump: N2}, and \eqref{def of mathcal C psi}:}  
   Since we assumed $C=1$, it suffices to show that Conditions \ref{assump: N1},  \ref{assump: N2}, and \eqref{def of mathcal C psi} hold with $C_\ph=1$. 
   If we can show that $\Psi^*(\mp)=\max(\mp)$, and
        \begin{align}
           \label{instatement: product loss: c psi=1} 
           \Psi^*(\mp)-\Psi(\mx;\mp)\geq  (\max \mp-\mp_{\pred(\mx)})/2,
        \end{align}
then it will follow that Condition \ref{assump: N2} holds with $C_\phi=1$ and \eqref{def of mathcal C psi} holds with $\CC_\phi=1/2$. The latter would imply that Condition \ref{assump: N1} holds. Thus it suffices to show \eqref{instatement: product loss: c psi=1}. 

        First, we will show that $\Psi^*(\mp)\geq \max(\mp)$. To see this, consider $\mxm=m\mathbf{e}^{(k)}_{\pred(\mp)}$. In this case, $\pred(\mxm)=\pred(\mp)$. Since $\tau$ is a distribution function, it follows that $\lim_{m\to\infty}\phi(\mxm;i)= 1$ if $i=\pred(\mp)$ and $\lim_{m\to\infty}\phi(\mxm;i)= 0$ for all other $i$'s in $[k]$. Therefore, $\Psi(\mxm;\mp)\to_m\max(\mp)$. Since $\Psi^*(\mp)=\sup_{\mx\in\RR^k}\Psi(\mx;\mp)$, it follows that
 \[\Psi^*(\mp)\geq \lim_{m\to\infty}\Psi(\mxm;\mp)=\max(\mp).\] 
 Therefore, to show that Condition \ref{assump: N2} holds, it suffices to show that $\Psi^*(\mp)\leq\max(\mp)$.
 
   We will show that given any $k\in\NN$, 
     $\Psi(\mx;\mp)\leq (\max(\mp)+\mp_{\pred(\mx)})/2$ for all $\mp\in\RR_{\geq 0}^k$ and $\mx\in\RR^k$. Note that the above implies 
     \begin{align}
        \label{inlemma: product loss: main ineq} 
        2\max(\mp)-2\Psi(\mx;\mp)\geq \max(\mp)-\mp_{\pred(\mx)},
     \end{align}
which, combined with  $\Psi^*(\mp)\geq \max(\mp)$, leads to \eqref{instatement: product loss: c psi=1} with $\CC_{\phi}=1/2$. 
Also,    $\Psi(\mx;\mp)\leq (\max(\mp)+\mp_{\pred(\mx)})/2$ implies  $\Psi^*(\mp)\leq \max(\mp)$, from which Condition \ref{assump: N2} follows. Therefore, it suffices to prove \eqref{inlemma: product loss: main ineq}. We will use mathematical induction to prove this. Our induction hypothesis is that \eqref{inlemma: product loss: main ineq}  holds  for some $k\in\NN$. We will first show that  \eqref{inlemma: product loss: main ineq}  holds for $k=1$. Next we will show that if  \eqref{inlemma: product loss: main ineq}  hold for some $k\in\NN$, it holds for $k+1$ as well, which will complete the proof.

   When $k=1$, $\mx=x\in\RR$ and $\mp=p\in\RR_{\geq 0}$. Note that $\Psi(\mx;\mp)=p\tau(x)$. Since $\tau(x)$ is a distribution function, $\sup_{x\in\infty}\tau(x)=1$ and $\Psi^*(\mp)=p=\max(\mp)$.   Also $\mp_{\pred(\mx)}=p$. Since $\tau$ is a distribution function, $\tau(x)\leq 1$. Thus 
   \[\Psi(\mx;\mp)=p\tau(x)\leq p=(\max(\mp)+\mp_{\pred(\mx)})/2=(p+\mp_{\pred(\mx)})/2,\]
  implying \eqref{inlemma: product loss: main ineq}  holds for $k=1$.
   
   Suppose the induction hypothesis holds for some $k\in\NN$.
 Consider  $\mx\in\RR^{k+1}$ and $\mp\in\RR_{\geq 0}^{k+1}$. Let $\pred(\mx)=i^*$. Suppose $r\in[k]$ is such that $r=\max\{\argmax_{i\in[k]:i\neq i^*}\mx_i\}$. Note that we allow the case $\mx_{r}=\max(\mx)$, which can happen if $\argmax(\mx)$ is not singleton. Also, $\mx_r\leq \mx_{i^*}$ and $\mx_r\geq \mx_i$ for all $i\in[k]\setminus\{r,i^*\}$.
Note that $\Psi(\mx;\mp)$ equals
   \begin{align*}
      & \mp_{i^*}\phi(\mx;i^*)+\sum_{j\in[k]:j\neq i^*}\mp_j\phi(\mx;j)\\
     =&\ \mp_{i^*}\tau(\mx_{i^*}-\mx_r)\prod_{i\notin\{r,i^*\}}\tau(\mx_{i^*}-\mx_i)+\sum_{j\neq i^*}\mp_j\tau(\mx_j-\mx_{i^*})\prod_{i\notin\{i^*,j\}}\tau(\mx_j-\mx_i).
   \end{align*}
   Here, the products will be one if the range is empty, which is the case when $k=2$ since $r\neq i^*$. 
   Since  $\tau$ is a distribution function,   $\tau\leq 1$. Also, since $\tau$ is a distribution function, it is non-decreasing, implying $\tau(\mx_j-\mx_{i^*})\leq \tau(\mx_r-\mx_{i^*})$ for all $j\neq i^*$. Therefore,
   \begin{align*}
      \Psi(\mx;\mp) \leq &   \mp_{i^*}\tau(\mx_{i^*}-\mx_r)+\tau(\mx_r-\mx_{i^*})\sum_{j\in[k_2]:j\neq i^*}\mp_j\prod_{i\notin\{i^*,j\}}\tau(\mx_j-\mx_i).
   \end{align*}
  Let us denote 
  \[\mp'=(\mp_1,\ldots, \mp_{i^*-1},\mp_{i^*+1},\ldots,\mp_{k+1})\text{ and   
 }\mx'=(\mx_1,\ldots, \mx_{i^*-1},\mx_{i^*+1},\ldots,\mx_{k+1}).\] Then $\mp'\in\RR_{\geq 0}^k$ and $\mx'\in\RR^k$. Let us also denote by $\Psi(\mx';\mp')$ the $k$-dimesnional version of $\phi$.    Since the induction hypothesis holds for $k$, $\Psi(\mx';\mp')\leq \max(\mp')$.

Note that the last display can be rewritten as 
\begin{align*}
  \Psi(\mx;\mp) \leq &   \mp_{i^*}\tau(\mx_{i^*}-\mx_r)+\tau(\mx_r-\mx_{i^*})\Psi(\mx';\mp')\\
  \leq &\ \mp_{i^*}\tau(\mx_{i^*}-\mx_r)+\tau(\mx_r-\mx_{i^*})\max(\mp')\\
   \leq &\ \mp_{i^*}\tau(\mx_{i^*}-\mx_r)+\tau(\mx_r-\mx_{i^*})\max(\mp)\\
  \leq &\ \sup_{x\geq 0}\slb \mp_{i^*}\tau(x)+\tau(-x)\max(\mp)\srb,
\end{align*}
where we used the fact that $\mx_{i^*}\geq \mx_r$. 
Since $\KK$ is symmetric about zero, $\tau(-x)=1-\tau(x)$. Therefore, 
The derivative of the function $x\mapsto \mp_{i^*}\tau(x)+\max(\mp)\tau(-x)$ is $\KK(x)(\mp_{i^*}-\max(\mp))$, which is  non-positive because $\mp_{i^*}\leq \max(\mp)$ and $\KK(x)\geq 0$ as $\KK$ is a density.  Thus on the positive half-line, $\mp_{i^*}\tau(x)+\max(\mp)\tau(-x)$ is maximized at $x=0$. However, since $\KK$ is symmetric about zero, $\tau(0)=1/2$, which implies $\Psi(x;\mp)\leq (\mp_{i^*}+\max(\mp))/2$. Thus, we have shown that the induction hypothesis holds for $k+1$. Hence, the induction hypothesis holds for all $k\in\NN$, which completes the proof.

\noindent\textbf{Proving Condition \ref{assump: N3}:}
Since $\tau$ is a distribution function for $C=1$, $\tau(\infty)=1$ and $\tau(-\infty)=0$. Note also that we have shown that $C_\phi=1$. If $\argmax(\mx)=j$, then
 \[\lim_{a\to\infty}\phi(a\mx;j)=\prod_{i\neq j}\lim_{a\to\infty}\tau(a(\mx_j-\mx_i))=1.\]
 If $j\notin\argmax(\mx)$, and $\argmax(\mx)$ is singleton, $a(\mx_j-\mx_{\pred(\mx)})\to -\infty$ as $a\to\infty.$ The proof follows noting that, in this case,
 \[\lim_{a\to\infty}\phi(a\mx;j)=\prod_{i\neq j}\lim_{a\to\infty}\tau(a(\mx_j-\mx_i))=0.\]
 \end{proof}

\section{Proofs for Section \ref{sec: approximation error}}
\label{sec: approx error proofs}


\subsubsection{Proof of Lemma \ref{lemma: error decomp}}
\label{secpf: error decomp}
\begin{proof}
 Proposition \ref{prop: multi-cat FC} implies
 \[ C_*(V_*-V(\hf))\leq  V^\psi_*-V^\psi(\hf).\]
 Since $\J_t>0$ for all $t\in[T]$, it follows that  $C_*>0$.
 Recall from  Algorithm \ref{alg: SDSS} that $\hf=\trans(\hg)$ where the $\trans$ operator is as in \eqref{eq: relative class cores}. Since we defined $\Vr(g)$ to be $V^\psi(\trans(g))$,  $V^\psi(\hf)=\Vr(\hg)$.  
 From Section \ref{sec: method}, it follows that for relative-margin-based surrogates, $V^\psi_*=\sup_{g\in\W}\Vr(g)$. Thus $ V^\psi_*-V^\psi(\hf) =\sup_{g\in\W}\Vr(g)-\Vr(\hg)$,
implying the following decomposition of the $\psi$-regret $ V^\psi_*-V^\psi(\hf)$:
 \begin{align*}
  \MoveEqLeft \sup_{g\in\W}\Vr(g)-\Vr(\tilde g)+ (\Vr-\hVr)(\tga)-(\Vr-\hVr)(\hg)\\
    &\ + \hVr(\tga)- \sup_{g\in\U_n}\hVr(g)+\underbrace{\sup_{g\in\U_n}\hVr(g)-\hVr(\hg)}_{\text{Optimization error: }\Opn}.
\end{align*}
The proof follows noting  $\widehat V^\psi(\tga)\leq  \sup_{g\in\U_n}\widehat V^\psi(g)\leq 0$ because  $\U_n\ni \tga$. 
\end{proof}
\subsection{Proof of Result \ref{result: smoothed psi satisfies approx error conditions}}
\label{secpf: kernel loss satisfies approx condition}
\begin{proof}[Proof of Result \ref{result: smoothed psi satisfies approx error conditions}]
  Proof follows trivially if $j\in\argmax(\mx)$. Suppose $j\notin\argmax(\mx)$.  Let $\pred(x)=i$. Then 
    \begin{align*}
       \psi(x;j)= &\ P_{\KK}\slb Z_j\leq Z_{i'} +x_j-x_{i'}\quad \text { for all }i'\neq j, i'\in[k] \srb\\
      \leq &\ P_{\KK} \slb Z_j\leq Z_i +x_j-x_i\srb \\
     =  &\  P_{\KK} \slb Z_j\leq Z_i -(x_i-x_j)\srb  
    \end{align*}
     However, since $x_i-x_j>0$,
    \[P_{\KK}(Z_i-Z_j\geq x_i-x_j)=P_{\KK}\slb (Z_i-Z_j)^2>(x_i-x_j)^2\srb\leq \frac{E_{\KK}[(Z_1-Z_2)^2]}{(x_i-x_j)^2}\]
    by Markov's inequality since $Z_1-Z_2$ has finite second moment. Moreover,
    $C_a=E_{\KK}[(Z_1-Z_2)^2]=E_{\KK}[Z_1^2]+E_{\KK}[Z_2^2]=2E_{\KK}[Z_1^2]>0$ because $E_{\KK}[Z_1^2]>0$.
\end{proof}

\subsection{Proof of Theorem \ref{theorem: approx error theorem}}
Without loss of generality, we assume that $C_{\phi_t}=1$ for all $t\in[T]$. If that is not the case, we can replace the $\phi_t$'s by $\phi_t/C_{\phi_t}$. 
Our first step is to express $V^\psi_*-V^\psi(f)$ in terms of the optimal Q-functions. To this end, we will need the following lemma.

\begin{lemma}
\label{lemma: approx:  difference decomposition in terms of Q-functions}
Under the setup of Theorem \ref{theorem: approx error theorem}, for any $f=(f_1,\ldots,f_T)\in\F$, 
\begin{align*}
   V^\psi_*-V^\psi(f)= \E\lbt \sum\limits_{j=1}^{T}\prod_{i=1}^j\frac{\phi_i(f_i(H_i);A_i)}{\pi_i(A_i\mid H_i)}\slb \max(Q_{j}^*(H_{j}))-Q_{j}^*(H_{j};A_{j})\srb\rbt.
\end{align*}
\end{lemma}

 Lemma \ref{lemma: approx:  difference decomposition in terms of Q-functions} is proved in 
Section \ref{sec: aux lemmas for approx error}. 
Since $\phi_t\geq 0$, and $\sum_{j\in[k_t]}\phi_t(\mx;j)\leq \Psi_t^*(\mo_{k_t})=1$, if follows that $\phi_t\leq 1$ for all $t\in[T]$. Moreover, since $\pi_t> \pb$ for all $t\in[T]$, using Lemma \ref{lemma: approx:  difference decomposition in terms of Q-functions} and \eqref{intheorem: necessity: IPW: general p}, we obtain
\begin{align}
\label{intheorem: approx: main inequality}
  V^\psi_*-V^\psi(f) &\ \leq 
    \frac{\sum\limits_{t\in[T]} \E\slbt \sum\limits_{i\in[k_t]:i\neq d_t^*(H_t)}\slb Q_t^*(H_t; d_t^*(H_t))-Q_t^*(H_t; i)\srb \phi_t(f_t(H_t);i) \srbt}{\pb^{T}}.
\end{align}
For fixed $i\in[k_t]$,
\begin{align*}
   \MoveEqLeft  \slb Q_t^*(H_t; d_t^*(H_t))-Q_t^*(H_t; i)\srb \phi_t(f_t(H_t);i) \\
   =&\  \underbrace{ \slb Q_t^*(H_t; d_t^*(H_t))-Q_t^*(H_t; i)\srb \phi_t(\myb_n Q_t^*(H_t);i)}_{S_1}\\
   &\ -\underbrace{\slb Q_t^*(H_t; d_t^*(H_t))-Q_t^*(H_t; i)\srb\slb \phi_t(\myb_n Q_t^*(H_t);i)- \phi_t(f_t(H_t);i) \srb}_{S_2}.
\end{align*}
Defining $\hn=\trans(\tga)$, we will now show that when  $f=\myb_n \hn$, then 
\begin{equation}
    \label{intheorem: approx: S1 inequality}
    S_1\lesssim C_a \myb_n^{-2}z^{-1}+z1[H_t\in\mE_{ti}]
\end{equation}
and
\begin{equation}
     \label{intheorem: approx: S2 inequality}
     S_2\lesssim C_a \myb_n^{-2}\delta_n^{-1}+\delta_n 1_{\mE_{ti}'}.
 \end{equation}
\subsubsection{Bound on $S_1$}
For fixed $t\in[T]$, $i\in[k_t]$, and $z>0$, let us denote $\mE_{ti}=\{h_t\in\H_t:  Q_t^*(h_t; d_t^*(h_t))-Q_t^*(h_t; i)\leq z\}$. Ideally, $\mE_{ti}$ should depend on $z$, but we omit the dependence for notational simplicity. Let us denote the random variables  $1_{\mE_{ti}^c}$ and $1_{\mE_{ti}^c}$ by $1_{\mE_{ti}}$ and $1_{\mE_{ti}^c}$, respectively.
Then 
\begin{align*}
    S_1=&\ \underbrace{\slb Q_t^*(H_t; d_t^*(H_t))-Q_t^*(H_t; i)\srb \phi_t(\myb_n Q_t^*(H_t);i) 1_{\mE_{ti}^c}}_{S_{11}}\\
    &\ + \underbrace{\slb Q_t^*(H_t; d_t^*(H_t))-Q_t^*(H_t; i)\srb \phi_t(\myb_n Q_t^*(H_t);i) 1_{\mE_{ti}}}_{S_{12}}.
\end{align*}
We apply Condition \ref{assump: N3: strong} on the first term to obtain
\[\psi(a_nQ_t^*(H_t);i)\leq C_a\slb \myb_n \slb Q_t^*(H_t, d_t^*(H_t))-Q_t^*(H_t,i)\srb\srb^{-2}.\]
 Thus
 \[S_{11}\lesssim C_a \myb_n^{-2}\slb Q_t^*(H_t, d_t^*(H_t))-Q_t^*(H_t,i))\srb^{-1}1_{\mE_{ti}^c}\leq C_a \myb_n^{-2}z^{-1}\]
 because for $H_t\in\mE^c_t$, $Q_t^*(H_t, d_t^*(H_t))-Q_t^*(H_t,i)> z>0$.
Finally, \eqref{intheorem: approx: S1 inequality}  follows, noting, since   $\phi_t\leq 1$ for all $t\in[T]$, we have $S_{12}\leq z1_{\mE_{ti}}$.

 
 

\subsubsection{Bound on $S_2$}
 Now for fixed $t\in[T]$, $i\in[k_t]$, and for any  $\delta_n>0$, we denote  $\mE'_{ti}=\{h_t\in\H_t:Q_t^*(h_t; d_t^*(h_t))-Q_t^*(h_t; i)\leq 3\delta_n\}$. Here we suppress the dependence of $\mE$ on $n$ for the sake of simplicity. Now notice that $S_2$ is bounded above by
 \begin{align*}
\MoveEqLeft \underbrace{\slb Q_t^*(H_t; d_t^*(H_t))-Q_t^*(H_t; i)\srb\slb \phi_t(\myb_n Q_t^*(H_t);i)- \phi_t(\myb_n\hn_t(H_t);i)\srb 1_{(\mE_{ti}')^c}}_{S_{21}}\\
&\ +\underbrace{\slb Q_t^*(H_t; d_t^*(H_t))-Q_t^*(H_t; i)\srb\slb \phi_t(\myb_n Q_t^*(H_t);i)- \phi_t(\myb_n\hn_t(H_t);i)\srb 1_{\mE'_{ti}}}_{S_{22}}.
 \end{align*}
 We will analyze $S_{21}$ first, for which, we use Lemma \ref{lemma: approx: connecting f and Q on E} below. This lemma is proved in Supplement \ref{secpf: proof of Lemma approx: connecting f and Q on E}.
 \begin{lemma}
\label{lemma: approx: connecting f and Q on E}
    Under the setup of Theorem \ref{theorem: approx error theorem}, for all $t\in[T]$, $h_t\in \H_t$, and $i\in[k_t]$, it holds that, 
\begin{align}
    \label{intheorem: approx: S21: lower bound }
    \hn_{t,d_t^*(h_t)}(h_t)-\hn_{ti}(h_t)\geq Q_t^*(h_t;{d_t^*(h_t)})-Q_t^*(h_t;i)-2\delta_n\text{ where }\hn_t=\trans(\tga_t).
\end{align}
\end{lemma}

Since $H_t\in(\mE'_{ti})^c$
  implies $Q_t^*(H_t;{d_t^*(H_t)})-Q_t^*(H_t;i)> 3\delta_n$, Lemma \ref{lemma: approx: connecting f and Q on E} implies the following for $H_t\in(\mE'_{ti})^c$:
 \[ \hn_{t,d_t^*(H_t)}(H_t)-\hn_{ti}(H_t)>\delta_n>0\quad\text{for all }i\neq {d_t^*(H_t)}.\]
Therefore, $\argmax(\hn_t(H_t))=\{{d_t^*(H_t)}\}$ for $H_t\in(\mE'_{ti})^c$. On the other hand, Condition \ref{assump: N3: strong} implies
\begin{align*}
    \phi_t(\myb_n\hn_t(H_t);i)
    \leq&\ \myb_n ^{-2}C_a\slb \hn_{t,d_t^*(H_t)}(H_t)-\hn_{ti}(H_t)\srb^{-2}\\
    \leq&\  \myb_n^{-2} C_a\slb Q_t^*(H_t;{d_t^*(H_t)})-Q_t^*(H_t;i)-2\delta_n\srb^{-2},
\end{align*}
where the last step follows from Lemma \ref{lemma: approx: connecting f and Q on E} and the fact that $Q_t^*(H_t;{d_t^*(H_t)})-Q_t^*(H_t;i)-2\delta_n>\delta_n>0$ for $H_t\in(\mE'_{ti})^c$.
Moreover, if $Q_t^*(H_t;{d_t^*(H_t)})-Q_t^*(H_t;i)> 3\delta_n$, then
\[2\delta_n< \frac{2}{3}\slb Q_t^*(H_t;{d_t^*(H_t)})-Q_t^*(H_t;i)\srb,\]
which implies, for $H_t\in(\mE'_{ti})^c$,
\[ Q_t^*(H_t;{d_t^*(H_t)})-Q_t^*(H_t;i)-2\delta_n\geq \frac{1}{3}\slb  Q_t^*(H_t;{d_t^*(H_t)})-Q_t^*(H_t;i)\srb.\]
Therefore, for $H_t\in(\mE'_{ti})^c$,
\[\phi_t(\myb_n\hn_t(H_t);i)\leq 9a_n^{-2}C_a \slb Q_t^*(H_t;{d_t^*(H_t)})-Q_t^*(H_t;i)\srb^{-2}.\]
On the other hand,  Condition \ref{assump: N3: strong} applied on $\phi_t(a_nQ_t^*(H_t);i)$ implies that 
\[\phi_t(a_nQ_t^*(H_t);i)\leq C_a  \myb_n^{-2} \slb Q_t^*(H_t;{d_t^*(H_t)})-Q_t^*(H_t;i)\srb^{-2}.\]
Therefore, we have shown that for $H_t\in(\mE'_{ti})^c$,
\begin{align*}
 \MoveEqLeft   \slb Q_t^*(H_t; {d_t^*(H_t)})-Q_t^*(H_t; i)\srb\slb \phi_t(\myb_n Q_t^*(H_t);i)- \phi_t(\myb_n\hn_t(H_t);i)\srb\\
 \leq&\ 10C_a  \myb_n^{-2} \slb Q_t^*(H_t;{d_t^*(H_t)})-Q_t^*(H_t;i)\srb^{-1},
\end{align*}
 which implies $ S_{21}\leq 10 C_a \myb_n^{-2}\delta_n^{-1}1_{(\mE'_{ti})^c}$. On the other hand, from the definition of $\mE_{ti}'$, and using $\phi_t\leq 1$, 
 it is not hard to see that
$S_{22}\leq 6\delta_n 1_{\mE_{ti}'}$. 
 Thus, \eqref{intheorem: approx: S2 inequality} is proved.

 Combining \eqref{intheorem: approx: main inequality}, \eqref{intheorem: approx: S1 inequality},  and \eqref{intheorem: approx: S2 inequality}, we obtain that
 \begin{align*}
     V^\psi_*-V^\psi(\myb_n\hn)\lesssim &\ \frac{1}{\pb^T}\lb T \myb_n^{-2} C_a(z^{-1}+\delta_n^{-1})+\sum\limits_{t\in[T]}\E\slbt \sum\limits_{i\in[k_t]:i\neq d_t^*(H_t)} (z1_{\mE_{ti}}+\delta_n 1_{\mE_{ti'}})\srbt\rb.
 \end{align*}
  Let $\mE_t=\{h_t: \mu(Q_t^*(h_t))\leq z\}$ and $\mE'_t=\{h_t: \mu(Q_t^*(h_t))\leq 3\delta_n\}$.
 Then
 \[\sum\limits_{i\in[k_t]:i\neq d_t^*(H_t)}  1_{\mE_{ti}}\leq (k_t-1) 1_{\mE_t}\text{ and } \sum\limits_{i\in[k_t]:\neq d_t^*(H_t)}  1_{\mE'_{ti}}\leq (k_t-1) 1_{\mE'_t}.\]
 By the small noise assumption, $\PP(\mE_t)\leq z^{\alpha}$ and  $\PP(\mE'_t)\lesssim \delta_n^\alpha$.
 Thus 
 \[ V^\psi_*-V^\psi(\myb_n\hn)\lesssim \frac{1}{\pb^T} \lb \myb_n^{-2}z^{-1}C_a+\myb_n^{-2}\delta_n^{-1}C_a+z^{1+\alpha}+\delta_n^{1+\alpha}\rb.\]
We want the $z^{1+\alpha}$ term to match $\delta_n^{1+\alpha}$. Thus, letting $z=\delta_n$
\[V^\psi_*-V^\psi(\myb_n\hn)\lesssim  \frac{1}{\pb^T} \lb \myb_n^{-2}\delta_n^{-1}C_a+\delta_n^{1+\alpha}\rb.\]
Since $\myb_n> \delta_n^{-(1+\alpha/2)}$, it follows that $\delta_n^{1+\alpha}>\myb_n^{-2}\delta_n^{-1}$. Thus
\[V^\psi_*-V^\psi(\myb_n\hn)\lesssim  \frac{1}{\pb^T} (1+C_a)\delta_n^{1+\alpha}.\]
Since $\sup_{g\in\W}\Vr(g)=V^\psi_*$ for relative-margin-based surrogates, and  $V^{\psi,\text{rel}}(\myb_n\tga)=V^\psi(\myb_n\trans(\tga))=V^\psi(\myb_n\hn)$, the proof follows.

\subsection{Proof of the auxiliary lemmas for proving Theorem \ref{theorem: approx error theorem}}
\label{sec: aux lemmas for approx error}

\begin{lemma}
\label{lemma: approx: p t star q t star}
    Under Assumptions I-V, for all $t\in[T]$ and $i\in[k-t]$, 
    \begin{equation}
\label{inlemma: approx: p star to Q}
 \max(p_{t}^*(H_{t}))-p_{t}^*(H_{t};i)   =Q_t^*(H_t,d_t^*(H_t))-Q_t^*(H_t;i).
\end{equation}
\end{lemma}
\begin{proof}[Proof of Lemma \ref{lemma: approx: p t star q t star}]
Note that for $j\in[k_T]$,  \eqref{def: sufficiency: pt star} implies
\begin{align*}
    p_T^*(H_T;j)= \E\slbt\sum\limits_{i=1}^T Y_i\mid H_T, A_T,A_T=j\srbt=&\ \E[Y_T\mid H_T,A_T=j]+\sum\limits_{i=1}^{T-1}Y_i,
\end{align*}
which equals $Q_T^*(H_T,j)+\sum_{i=1}^{T-1}Y_i$. 
Therefore,
\[ \max(p_{T}^*(H_{T}))-p_{T}^*(H_{T};i)=\max(Q_{T}^*(H_{T}))-Q_{T}^*(H_{T};i).\]
Hence, the proof follows for $t=T$. For $t=1$, the proof follows from Lemma \ref{lemma: sufficiency: p t star and Q t star}. Suppose $T>2$. 
For $t\in[2:T-1]$, using \eqref{def: p t star} we obtain that for $i\in[k_t]$, 
    \begin{align*}
     p_t^*(H_t;i)=&\   \E\lbt\E\lbt \slb\sum\limits_{i=1}^T Y_i\srb\prod_{j=1+t}^T\frac{1[A_j=d_j^*(H_j)]}{\pi_j(A_j\mid H_j)}\bl H_{1+t}\rbt\bl H_t, A_t=i\rbt\\
     =&\  \E\lbt\E\lbt \slb\sum\limits_{i=1}^{t-1} Y_i\srb\prod_{j=1+t}^T\frac{1[A_j=d_j^*(H_j)]}{\pi_j(A_j\mid H_j)}\bl H_{1+t}\rbt\bl H_t, A_t=i\rbt\\
     &\ +\E\lbt\E\lbt \slb\sum\limits_{i=t}^T Y_i\srb\prod_{j=1+t}^T\frac{1[A_j=d_j^*(H_j)]}{\pi_j(A_j\mid H_j)}\bl H_{1+t}\rbt\bl H_t, A_t=i\rbt,
    \end{align*}
    whose last term is $Q_t^*(H_t,i)$ by Fact \ref{fact: Q function expression} and the first term 
    \begin{align*}
  \MoveEqLeft  \slb\sum\limits_{i=1}^{t-1} Y_i\srb  \E\lbt  \E\lbt \prod_{j=1+t}^T\frac{1[A_j=d_j^*(H_j)]}{\pi_j(A_j\mid H_j)}\bl H_{1+t}\rbt \bl H_t, A_t=i\rbt =  \sum\limits_{i=1}^{t-1} Y_i
    \end{align*}
because 
\[ \E\lbt \prod_{j=1+t}^T\frac{1[A_j=d_j^*(H_j)]}{\pi_j(A_j\mid H_j)}\bl H_{1+t}\rbt =1\]
by \eqref{intheorem: fact: product of indicators}. 
Therefore, for $1<t<T$ and $i\in[k_t]$,
\[p_t^*(H_t;i)=Q_t^*(H_t,i)+\sum\limits_{i=1}^{t-1} Y_i.\]
Hence, the lemma follows for $t\in[2:T-1]$.
\end{proof}

\begin{proof}[Proof of Lemma \ref{lemma: approx:  difference decomposition in terms of Q-functions}]
Without loss of generality, we assume that $C_{\phi_t}=1$ for each $t\in[T]$. Otherwise, we can replace $\phi_t$ by $\phi_t/C_{\phi_t}$ for each $t\in[T]$. 
Since $\phi_t$ is symmetric for each $t\in[T]$, it  holds that $\Psi_t^*(\mo_{k_t})=\Psi_t(\mx;\mo_{k_t})$ for each $\mx\in\RR^{k_t}$. Since $\Psi_t^*(\mo_{k_t})=C_{\phi_t}=1$, it follows that $\Psi_t(\mx;\mo_{k_t})=\sum_{i\in[k_t]}\phi_t(\mx;i)=1$ for each $\mx\in\RR^{k_t}$. This fact will be used repeatedly in the proof. 
We will also use the notation introduced in Section~\ref{secpf: suff: new notation} for the proof of Theorem~\ref{theorem: sufficient conditions}.

For $t\in[T]$, let us define
\[\T_t=\E\lbt \Psi_t(f_t(H_t);p_t^*(H_t))-\Psi_t(f_t(H_t);p_t(H_t))\mid H_t\rbt.\]
Since $p_T^*(H_T)=p_T(H_T)$, we obtain that $\T_T=0$.
For $t\in[T-1]$, we obtain that $\T_t $ equals
\begin{align*}
 &\ \E\slbt \sum\limits_{i\in[k_t]}\phi_t(f_t(H_t);i)\slb p_t^*(H_t)_i-p_t(H_t)_i\srb\mid H_t\srbt\\
 \stackrel{(a)}{=}&\ \E\lbt \frac{\phi_t(f_t(H_t);A_t)}{\pi_t(A_t\mid H_t)}\slb p_t^*(H_t;A_t)-p_t(H_t;A_t)\srb\bl  H_t\rbt\\
 =&\ \E\lbt\frac{\phi_t(f_t(H_t);A_t)}{\pi_t(A_t\mid H_t)} \E\slbt \Psi_{1+t}^*(p_{1+t}^*(H_{1+t}))-\Psi_{1+t}(f_{1+t}(H_{1+t});p_{1+t}(H_{1+t}))\mid H_t, A_t\srbt \bl H_t\rbt \\
 =&\ \E\lbt\frac{\phi_t(f_t(H_t);A_t)}{\pi_t(A_t\mid H_t)} \E\slbt \Psi_{1+t}^*(p_{1+t}^*(H_{1+t}))-\Psi_{1+t}(f_{1+t}(H_{1+t});p_{1+t}^*(H_{1+t}))\mid H_t, A_t\srbt \bl H_t\rbt\\
 &\ + \E\lbt\frac{\phi_t(f_t(H_t);A_t)}{\pi_t(A_t\mid H_t)} \E\slbt \Psi_{1+t}(f_{1+t}(H_{1+t});p_{1+t}^*(H_{1+t}))\\
 &\ -\Psi_{1+t}(f_{1+t}(H_{1+t});p_{1+t}(H_{1+t}))\mid H_t, A_t\srbt \bl H_t\rbt\\
 \stackrel{(b)}{=}&\ \E\lbtt\frac{\phi_t(f_t(H_t);A_t)}{\pi_t(A_t\mid H_t)} \E\lbtt\sum\limits_{i\in[k_{1+t}]}\begin{matrix}\slb \max(p_{1+t}^*(H_{1+t}))-p_{1+t}^*(H_{1+t};i)\srb\\\times \phi_{1+t}(f_{1+t}(H_{1+t});i)\end{matrix}\bl H_t, A_t\rbtt \bl H_t\rbtt\\
 &\ +\E\lbtt\frac{\phi_t(f_t(H_t);A_t)}{\pi_t(A_t\mid H_t)} \E\lbtt \begin{matrix}
   \Psi_{1+t}(f_{1+t}(H_{1+t});p_{1+t}^*(H_{1+t}))\\
   -\Psi_{1+t}(f_{1+t}(H_{1+t});p_{1+t}(H_{1+t}))  
 \end{matrix}\bl H_{1+t}\rbtt \bl H_t\rbtt,
\end{align*}
where (a) uses \eqref{intheorem: necessity: IPW: general p} and (b) uses the facts: (i) $\Psi_t^*(\mp)=\max(\mp)$ for all $\mp\in\RR^{k+t}_{\geq 0}$ since $\phi_t$ satisfies Condition \ref{assump: N2} with $C_{\phi_t}=1$, (ii) $\sum_{i\in[k_t]}\phi_t(\mx;i)=1$ for each $t\in[T]$ and $\mx\in\RR^{k_t}$, as shown previously,  and (ii) $H_{1+t}\supset \{H_t,A_t\}$. The second term on the RHS of the above display equals
\[\E\lbt\frac{\phi_t(f_t(H_t);A_t)}{\pi_t(A_t\mid H_t)}\T_{1+t} \bl H_t\rbt.\]
Note that Lemma \ref{lemma: approx: p t star q t star} implies
\[
 \max(p_{t}^*(H_{t}))-p_{t}^*(H_{t};i)   =Q_t^*(H_t,d_t^*(H_t))-Q_t^*(H_t;i).
\]
Then it follows that for $t\in[T-1]$, $\T_t$ equals
\begin{align}
\label{inlemma: approx: Tt expression}
 \MoveEqLeft \E\lbt\frac{\phi_t(f_t(H_t);A_t)}{\pi_t(A_t\mid H_t)} \lb \E\slbt\sum\limits_{i\in[k_{1+t}]}\slb \max(Q_{1+t}^*(H_{1+t}))-Q_{1+t}^*(H_{1+t};i)\srb\nn\\
 &\ \times \phi_{1+t}(f_{1+t}(H_{1+t});i)\mid H_t, A_t\srbt +\T_{1+t} \rb\bl H_t\rbt\nn\\
    =&\ \E\lbt\frac{\phi_t(f_t(H_t);A_t)}{\pi_t(A_t\mid H_t)} \lbs\frac{\phi_{1+t}(f_{1+t}(H_{1+t});A_{1+t})}{\pi_{1+t}(A_{1+t}\mid H_{1+t})}\nn\\
    &\ \times\lb  \max(Q_{1+t}^*(H_{1+t}))-Q_{1+t}^*(H_{1+t};A_{1+t})\srb+\T_{1+t} \rbs\bl H_t\rbt
\end{align}
by \eqref{intheorem: necessity: IPW: general p}. 
We claim that for any $t\in[T-1]$,
\begin{equation}
    \label{inlemma: approx: Tt: induction}
    \T_t=\E\lbt \sum\limits_{j=t+1}^{T}\prod_{i=t}^j\frac{\phi_i(f_i(H_i);A_i)}{\pi_i(A_i\mid H_i)}\slb \max(Q_{j}^*(H_{j}))-Q_{j}^*(H_{j};A_{j})\srb\bl H_{t}\rbt.
\end{equation}
Note that since $\T_T=0$, \eqref{inlemma: approx: Tt expression} implies 
\[\T_{T-1}=\E\lbtt\frac{\splitfrac{\phi_{T-1}(f_{T-1}(H_{T-1});A_{T-1})}{\times \ \phi_T(f_{T}(H_{T});A_{T})}}{\splitfrac{\pi_{T-1}(A_{T-1}\mid H_{T-1})}{\times \ \pi_{T}(A_{T}\mid H_{T})}}\slb \max(Q_{T}^*(H_{T}))-Q_{T}^*(H_{T};A_{T})\srb\bl H_{T-1}\rbtt.\]
Hence, \eqref{inlemma: approx: Tt: induction} holds for $t=T-1$. Suppose \eqref{inlemma: approx: Tt: induction} holds for $t\in[2:T-1]$. Then \eqref{inlemma: approx: Tt expression} implies 
\begin{align*}
    \T_{t-1}= &\ \E\lbt \sum\limits_{j=t+1}^{T}\prod_{i=t-1}^j\frac{\phi_i(f_i(H_i);A_i)}{\pi_i(A_i\mid H_i)}\slb \max(Q_{j}^*(H_{j}))-Q_{j}^*(H_{j};A_{j})\srb\bl H_{t-1}\rbt\\
    &\ +\E\lbt \prod_{i=t-1}^t\frac{\phi_i(f_i(H_i);A_i)}{\pi_i(A_i\mid H_i)}\slb \max(Q_{t}^*(H_{t}))-Q_{t}^*(H_{t};A_{t})\srb\bl H_{t-1}\rbt\\
    =&\ \E\lbt \sum\limits_{j=t}^{T}\prod_{i=t-1}^j\frac{\phi_i(f_i(H_i);A_i)}{\pi_i(A_i\mid H_i)}\slb \max(Q_{j}^*(H_{j}))-Q_{j}^*(H_{j};A_{j})\srb\bl H_{t-1}\rbt.
\end{align*}
Hence, \eqref{inlemma: approx: Tt: induction} holds for $t-1$ as well. Therefore, by induction, \eqref{inlemma: approx: Tt: induction} holds for $t=1$ and
\[\T_1=\E\lbt \sum\limits_{j=2}^{T}\prod_{i=1}^j\frac{\phi_i(f_i(H_i);A_i)}{\pi_i(A_i\mid H_i)}\slb \max(Q_{j}^*(H_{j}))-Q_{j}^*(H_{j};A_{j})\srb\bl H_{1}\rbt.\]
By the definition of $\T_t$, 
\[\E[\T_1]=\E\slbt \Psi_1(f_1(H_1);p_1^*(H_1))-\Psi_1(f_1(H_1);p_1(H_1))\srbt.\]
Hence, it follows that
\begin{align*}
    \MoveEqLeft \E\slbt \Psi_1(f_1(H_1);p_1^*(H_1))-\Psi_1(f_1(H_1);p_1(H_1))\srbt\\
    =&\ \E\lbt \sum\limits_{j=2}^{T}\prod_{i=1}^j\frac{\phi_i(f_i(H_i);A_i)}{\pi_i(A_i\mid H_i)}\slb \max(Q_{j}^*(H_{j}))-Q_{j}^*(H_{j};A_{j})\srb\rbt.
\end{align*}
Using Lemma \ref{lemma: sufficiency}, we can show that $V^\psi_*-V^\psi(f)$ equals
\begin{align*}
  \MoveEqLeft \E\slbt\Psi_1^*(p_1^*(H_1))-\Psi_1(f_1(H_1);p_1^*(H_1))\srbt+ \E\slbt \Psi_1(f_1(H_1);p_1^*(H_1))-\Psi_1(f_1(H_1);p_1(H_1))\srbt\\
 \stackrel{(a)}{=}&\ \E[\max(p_1^*(H_1))-\Psi_1(f_1(H_1);p_1^*(H_1))]\\
 &\ +\E\lbt \sum\limits_{j=2}^{T}\prod_{i=1}^j\frac{\phi_i(f_i(H_i);A_i)}{\pi_i(A_i\mid H_i)}\slb \max(Q_{j}^*(H_{j}))-Q_{j}^*(H_{j};A_{j})\srb\rbt\\
\stackrel{(b)}{=}&\ \E\slbt \sum\limits_{j=1}^{k_1}\phi_1(f_1(H_1);i)\slb \max(p_1^*(H_1))-p_1^*(H_1)_i\srb\srbt\\
 &\ +\E\lbt \sum\limits_{j=2}^{T}\prod_{i=1}^j\frac{\phi_i(f_i(H_i);A_i)}{\pi_i(A_i\mid H_i)}\slb \max(Q_{j}^*(H_{j}))-Q_{j}^*(H_{j};A_{j})\srb\rbt\\
 \stackrel{(c)}{=}&\ \E\slbt \frac{\phi_1(f_1(H_1);A_1)}{\pi_1(A_1\mid H_1)}\slb \max(Q_1^*(H_1))-Q_1^*(H_1;A_1)\srb\srbt\\
 &\ +\E\lbt \sum\limits_{j=2}^{T}\prod_{i=1}^j\frac{\phi_i(f_i(H_i);A_i)}{\pi_i(A_i\mid H_i)}\slb \max(Q_{j}^*(H_{j}))-Q_{j}^*(H_{j};A_{j})\srb\rbt,
\end{align*}
where (a) follows since $\phi_1$ satisfies Condition \ref{assump: N1}, (b) follows because $\phi_1$ satisfies $\sum_{i\in[k_1]}\phi_t(\mx;i)=1$, and (c) follows from Lemma \ref{lemma: approx: p t star q t star} and \eqref{intheorem: necessity: IPW: general p}. Therefore, the proof follows.
\end{proof}

\subsubsection{Proof of Lemma \ref{lemma: approx: connecting f and Q on E}}
\label{secpf: proof of Lemma approx: connecting f and Q on E}
\begin{proof}[Proof of Lemma \ref{lemma: approx: connecting f and Q on E}]
    We consider three cases: (i) $d_t^*(h_t)\neq 1$, $i\neq 1$ (ii) $d_t^*(h_t)=1$, $i\neq 1$, and (iii) $d_t^*(h_t)\neq 1$, $i= 1$. We will show that \eqref{intheorem: approx: S21: lower bound } holds in case (i) and (ii). The proof  for case (iii) follows similarly to that of case (ii), and hence skipped.
\paragraph*{Case (i)}
Since $\hn_t=\trans(\tga_t)$, for $i\geq 2$, $\hn_{ti}(h_t)=-\tga_{t,i-1}(h_t)$. Thus,
\[ \hn_{t,d_t^*(h_t)}(h_t)-\hn_{ti}(h_t)=\tga_{t,i-1}(h_t)-\tga_{t,d_t^*(h_t)-1}(h_t),\]
which equals
 \begin{align*}
\MoveEqLeft \tga_{t,i-1}(h_t)-(Q_t^*(h_t,1)-Q_t^*(h_t,i))-\slb \tga_{t,d_t^*(h_t)-1}(h_t)-\slb Q_t^*(h_t,1)\\
&\ -Q_t^*(h_t,d_t^*(h_t))\srb\srb+Q_t^*(h_t,d_t^*(h_t))-Q_t^*(h_t,i)\\
   \geq &\ Q_t^*(h_t,d_t^*(h_t))-Q_t^*(h_t,i)-2\sup_{j\in[k_t-1]}\|\tga_{t,j}-(Q_t^*(\cdot,1)-Q_t^*(\cdot,1+j)\|_\infty.
 \end{align*}
Thus \eqref{intheorem: approx: S21: lower bound } follows noting that our assumption on $\tga$ and \eqref{def: blip cont} imply
\begin{equation}
    \label{intheorem: approx: main cond on functions}
    \sup_{j\in[k_t-1]}\|\tga_{tj}-(Q_t^*(\cdot,1)-Q_t^*(\cdot,1+j))\|_\infty=\|\tga_{tj}-\texttt{blip}_{tj}\|_\infty\leq \delta_n.
\end{equation}
\paragraph*{Case (ii)}
Since $d_t^*(h_t)=1$ and $\hn_t=\trans(\tga_t)$, it follows that  $\hn_{t,d_t^*(h_t)}(h_t)=0$. However, since $i\geq 2$, $\hn_{ti}(h_t)=-\tga_{t,i-1}(h_t)$.
Thus
\begin{align*}
  &\  \hn_{t,d_t^*(h_t)}(h_t)-\hn_{t,i}(h_t)=\tga_{t,i-1}(h_t)\\
   &\ =\tga_{t,i-1}(h_t)-(Q_t^*(h_t,1)-Q_t^*(h_t,i))+Q_t^*(h_t,d_t^*(h_t))-Q_t^*(h_t,i)
\end{align*}
 where we used $d_t^*(h_t)=1$. However, $|\tga_{t,i-1}(h_t)-(Q_t^*(h_t,1)-Q_t^*(h_t,i))|\leq\delta_n$ by \eqref{def: blip cont} and \eqref{intheorem: approx: main cond on functions}.
Thus
\[ \hn_{t,d_t^*(h_t)}(h_t)-\hn_{t,i}(h_t)\geq Q_t^*(h_t,d_t^*(h_t))-Q_t^*(h_t,i)-\delta_n.\]
\end{proof}
\section{Proofs for  Section \ref{sec: estimation error}: combined regret decay rate}
\label{sec: est error proof}

\subsection{Proof of Theorem \ref{thm: est error}}
Without loss of generality, we assume that $C_{\phi_t}=1$ for each $t\in[T]$. Otherwise, we can replace each  $\phi_t$ by $\phi_t/C_{\phi_t}$. As shown in the proof of Lemma Proof of Lemma \ref{lemma: approx:  difference decomposition in terms of Q-functions}, $C_{\phi_t}=1$ implies $\sum_{i\in[k_t]}\phi_t(\mx;i)=1$ for all $\mx\in\RR^{k_t}$. Also, as shown in the proof of Theorem \ref{theorem: approx error theorem},  $C_{\phi_t}=1$ implies $\Phi_t^*(\mo_{k_t})=1$, which indicates  $\phi_t\leq 1$ for all $t\in[T]$. These facts will be used repeatedly in our proof. Throughout this proof, we will denote $\trans(\tga)$ by $\tf$. We also remind the readers that $\hf=\trans(\hg)$. Since the $\phi_t$'s satisfy the conditions of Proposition \ref{prop: multi-cat FC} with positive $\J$, the regret of $\hf$ is bounded by a constant multiple of the $\psi$-regret $V^\psi_*-V^\psi(\hf)$, which we will denote by $r_n$ from now on.  Therefore, it suffices to prove the concentration bound on $r_n$. Since $\hf$ is stochastic, $r_n$ is also a random variable. Parts of the proof of Theorem~\ref{thm: est error} follow similar arguments as Theorem 5 in \cite{Laha2021surrogateu}. Therefore, we focus on the parts that differ from that proof. First, we introduce some new functions. 

It will be convenient for us to write $V^\psi_*$ as $V^\psi(f)$ for some $f$. There are extended-valued scores for which the above is possible. 
Let us define
$\fs=(\fs_1,\ldots,\fs_T)$ with  $\fs_t:\H_t\mapsto\bRR^{k_t}$ so that  
\begin{align}
\label{def: tilde f}
    \fs_{ti}(H_t)=\begin{cases}
        \infty & \text{ if }i=d_t^*(H_t)\\
        0 & \text{otherwise,}
    \end{cases}
\end{align}
where  $d^*$ is any version of the optimal DTR. 
 Note that we can write $\fs_t(H_t)=\infty\times \mathbf e^{(d_t^*(H_t))}_{k_t}$. 
We define $\phi_t(\fs_t(H_t);j):=\lim_{m\to\infty}\phi_t(m \mathbf e^{(d_t^*(H_t))}_{k_t};j)$. Let us fix $h_t\in\H_t$.  Condition \ref{assump: N3: strong} implies that for $j\neq d_t^*(h_t)$
 \[\limsup_{m\to\infty}\phi_t(m \mathbf e^{(d_t^*(h_t))}_{k_t};j)\leq \lim_{m\to\infty} C_a m^{-2}=0.\]
 Therefore,
 \begin{align}
     \label{intheorrem: est: f star suboptimal j}
     \lim_{m\to\infty}\phi_t(m \mathbf e^{(d_t^*(h_t))}_{k_t};j)=0\text{ for all }j\neq d_t^*(h_t).
 \end{align}
 Since we showed that under our assumptions,  $\sum_{i\in[k_t]}\phi_t(\mx;i)=1$, it follows that
\[\phi_t(\mx;i)=1-\sum\limits_{j\in[k_t]:j\neq i}\phi_t(\mx;j)\text{ for all }\mx\in\RR^{k_t},\]
implying that 
\[\phi_t(m \mathbf e^{(d_t^*(h_t))}_{k_t};d_t^*(h_t))=1-\sum\limits_{j\in[k_t]:j\neq d_t^*(h_t)}\phi_t(m \mathbf e^{(d_t^*(h_t))}_{k_t};j).\]
Hence, 
\[\lim_{m\to\infty}\phi_t(m \mathbf e^{(d_t^*(h_t))}_{k_t};d_t^*(h_t))=1-\sum\limits_{j\in[k_t]:j\neq d_t^*(h_t)}\lim_{m\to\infty}\phi_t(m \mathbf e^{(d_t^*(h_t))}_{k_t};j)=1\]
by \eqref{intheorrem: est: f star suboptimal j}.  Therefore, by our definition,
\begin{align}
    \label{intheorem: def: f star}
   \phi_t(\fs_t(h_t);j)=1[j=d_t^*(h_t)], \quad j\in[k_t], t\in[T]. 
\end{align}
With the above definition of $\phi_t(\fs_t(H_t);j)$'s, $V^\psi(\fs)$ is well-defined. 
 \begin{lemma}
\label{lemma: the sup of V psi equals V star}
Let $\fs$ is as defined in \eqref{def: tilde f}.  
Then under the conditions of Theorem \ref{thm: est error}, $V^\psi(\fs)=V^\psi_*=V_*$ when $C_{\phi_t}=1$ for all $t\in[T]$.
\end{lemma}

\begin{proof}[Proof of Lemma \ref{lemma: the sup of V psi equals V star}]


When $C_{\phi_t}=1$ for all $t\in[T]$
Lemma \ref{lemma: sufficiency} implies
\[V^\psi_*=V_*=\E\lbt\slb\sum\limits_{i=1}^T Y_i\srb\prod_{j=1}^T\frac{1[A_j=d_j^*(H_j)]}{\pi_j(A_j\mid H_j)}\rbt=\E\lbt\slb\sum\limits_{i=1}^T Y_i\srb\prod_{j=1}^T\frac{\phi_t(\fs_t(H_t);j)}{\pi_j(A_j\mid H_j)}\rbt,\]
which equals $V^\psi(\fs)$.
Here the last step follows from \eqref{intheorem: def: f star}.



\end{proof}

For any $u,u'\in\F$, define
\begin{equation}
    \label{def: xi f g}
    \mathfrak{U}_{u,u'}(\D):=\slb \sum\limits_{i=1}^T Y_i\srb \frac{\prod_{t=1}^T\phi_t(u_t(H_t);A_t)-\prod_{t=1}^T\phi_t(u'_t(H_t);A_t)}{\prod_{t=1}^T\pi_t(A_t\mid H_t)}.
\end{equation}
The proof of Lemma \ref{lemma: error decomp} shows that
\[r_n=V^\psi_*-V^\psi(\hf)\leq \text{Approximation error}+\text{ Estimation error }+\Opn,\]
where the approximation error is 
\[\sup_{g\in\W}V^{\psi,\text{rel}}(g)-V^{\psi,\text{rel}}(\tga)=V^\psi_*-V^{\psi}(\trans(\tga))=V^\psi_*-V^{\psi}(\tf),\]
the estimation error is 
\[(\Vr-\hVr)(\tga-\hg)=(\PP_n-\PP)[\L(\D;\hg)-\L(\D;\tga)],\] and  $\Opn$ is the optimization error. 
Let us denote the approximation error by $\App$. 
We denote the absolute value of the estimation error is by $\Est$, i.e., $\Est=|(\PP_n-\PP)[\L(\D;\hg)-\L(\D;\tga)]|$.
Thus
\[r_n\leq \App+\Est+\Opn.\]
Since for any 
$i\in[k_t]$,
$\Gamma_t(\mx;i)=\phi_t(\trans(\mx);i)$ by Definition \ref{def: relative margin}, it is not hard to see that
\[\L(\D;\hg)-\L(\D;\tga)= \mathfrak{U}_{\trans(\hg),\trans(\tga)}(\D)= \mathfrak{U}_{\hf,\tf}(\D)\]
because $\tf=\trans(\tga)$ and $\hf=\trans(\hg)$. Thus $\Est=|(\PP_n-\PP)\mathfrak{U}_{\hf,\tf}|$. 

Note that $\delta_n^{-(1+\alpha/2)}=n^{1/2}(\rho_n\log \mathcal{I}_n)^{-1/2}$, which is less than $\myb_n$ by our assumption on $\myb_n$. Therefore, $\App$ can be bounded directly by applying  Theorem \ref{theorem: approx error theorem}, which yields  $\App\lesssim \delta_n^{1+\alpha}$.  The non-trivial step is bounding $\Est$. To this end, first, we need a bound on $\|\mathfrak{U}_{\fs,f}\|_{\PP,2}$, which can be bounded using $r_n$ and $\App$.


 \begin{lemma}
\label{lemma: est: l2l1}
  Under the setup of Theorem \ref{thm: est error}, there exists a constant $C>0$ depending only on $\PP$,  so that
    \[\|\mathfrak{U}_{\fs,f}\|_{\PP,2}^2\leq C\lb V^\psi(\fs)-V^\psi(f)\rb^{\alpha/(1+\alpha)}\]
    where $\alpha$ is as in Assumption \ref{assump: small noise}.
\end{lemma}
The above lemma is proved in Section \ref{secpf: est error: main lemma}. It implies
\begin{align*}
\|\mathfrak{U}_{\fs,\tf}\|^2_{\PP,2}\leq C\lb V^\psi(\fs)-V^\psi(\tf)\rb^{\alpha/(1+\alpha)}   =C\lb V^\psi_*-V^\psi(\tf)\rb^{\alpha/(1+\alpha)} 
 \end{align*}
 because $V^\psi(\fs)=V^\psi_*$ by Lemma \ref{lemma: the sup of V psi equals V star}. However, since 
 \[ V^\psi_*-V^\psi(\tf)=\App\lesssim\delta_n^{1+\alpha},\]
we have $\|\mathfrak{U}_{\fs,\tf}\|^2_{\PP,2}\lesssim C\delta_n^\alpha$. Thus 
 \begin{align}
     \label{intheorem: est error: 1}
     [\|\mathfrak{U}_{\hf,\tf}\|_{\PP,2}\leq \|\mathfrak{U}_{\fs,\tf}\|_{\PP,2}+\|\mathfrak{U}_{\fs,\hf}\|_{\PP,2}\stackrel{(a)}{\lesssim} \delta_n^{\alpha/2}+r_n^{\alpha/\{2(1+\alpha)\}}:=\e_n,
 \end{align}
 where (a) follows since $\|\mathfrak{U}_{\fs,\tf}\|^2_{\PP,2}\lesssim C\delta_n^\alpha$ and by Lemma \ref{lemma: est: l2l1}, 
 \[\|\mathfrak{U}_{\fs,\hf}\|_{\PP,2}^2\leq C\lb V^\psi(\fs)-V^\psi(\hf)\rb^{\alpha/(1+\alpha)}= C\lb V^\psi_*-V^\psi(\hf)\rb^{\alpha/(1+\alpha)}\lesssim r_n^{\alpha/(1+\alpha)}.\]
 
Next, define the set
\[\mathcal G_n(\epsilon)=\lbs \mathfrak{U}_{f,\tf}: f=\trans(g),g\in\U_n,\|\mathfrak{U}_{f,\tf}\|_{\PP,2}\leq \epsilon\rbs.\]
Note that \eqref{intheorem: est error: 1} implies 
 $\mathfrak{U}_{\hf,\tf}\in \mathcal G_n(\e_n)$. Therefore, we can bound $\Est$ by 
 \begin{align*}
   \Est\leq &\ \sup_{u\in\mathcal G_n(\epsilon_n)}|(\PP_n-\PP)u|= \E[\sup_{u\in\mathcal G_n(\epsilon_n)}|(\PP_n-\PP)u|]\\
   &\ +\underbrace{\sup_{u\in\mathcal G_n(\epsilon_n)}|(\PP_n-\PP)u|-\E[\sup_{u\in\mathcal G_n(\epsilon_n)}|(\PP_n-\PP)u|]}_{\text{deviation term}}  
 \end{align*}
Bounding  the deviation term requires Talagrand's inequality, and this part is exactly similar to \cite{Laha2021surrogateu} if we plug in our $\mathcal G_n(\e_n)$ in the proof of step 3 of Theorem 5 therein.
 From the symmetrization inequality \citep[cf. Theorem 2.1 of][]{koltchinskii2009}, it follows that 
\[  \E[\sup_{u\in\mathcal G_n(\epsilon_n)}|(\PP_n-\PP)u|]\lesssim \E[\sup_{u\in\mathcal G_n(\epsilon_n)}|\mathfrak{R}_n(u)|]\]
where $\mathfrak{R}_n$ is the Rademacher complexity. The number $\E[\sup_{u\in\mathcal G_n(\epsilon_n)}|\mathfrak{R}_n(u)|]$ is called 
 the Rademacher complexity  of the class $\mathcal G_n(\epsilon)$.  Our next  step is to show that the bracketing entropy $N_{[]}(\epsilon,\mathcal G_n(\e_n),L_2(\PP_n))$ of $\mathcal G_n(\e_n)$ is of the order $(\mathcal{I}_n/\epsilon)^{\rho_n'}$. It can be then used to derive an upper bound on the Rademacher complexity of the class $\mathcal G_n(\e_n)$ using the following fact given by (3.12), pp.40, of \cite{koltchinskii2009}.

 \begin{fact}[Fact from \cite{koltchinskii2009}]\label{fact: rademacher complexity bund for basis expansion classes}
  Suppose  $\mathcal G$ is a function-class with envelope $F$ such that $\|F\|_{\infty}\leq C$, where $C>0$. Further suppose there exists $\rho_n>0$ such that
  \[ N_{[]}(\e,\mathcal G, L_2(\PP_n))\lesssim \lb\frac{\I_{\mathcal G}}{\e}\rb^{\rho_n}.\]
  Let $\sigma^2=\sup_{u\in\mathcal G}Ph^2$. Suppose there exists $c>0$ such that  $n\sigma^2>c$. 
  Then
  \[\E [\|\mathfrak{R}_n\|_{\mathcal G}]\lesssim \max\lbs\sigma\lb\frac{\rho_n\log(\I_{\mathcal G}/\sigma)}{n}\rb^{1/2},\frac{\rho_n \log (\I_{\mathcal G}/\sigma)}{n}\rbs.\]
 \end{fact}
By Lemma 9.1 of \cite{jonnotes}, $N_{[]}(\epsilon,\mathcal G_n(\e_n),L_2(P))\lesssim N(\epsilon/2,\mathcal G_n(\e_n),\|\cdot\|_\infty)$ for any measure $P$. 
 On the other hand, since $\tf$ is fixed, for any $\e>0$, the covering number  $N(\epsilon,\mathcal G_n(\e_n),\|\cdot\|_\infty)$ is not larger than the covering  number of the function-class 
\[\mathcal C_1=\lbs \mathcal D\mapsto \frac{(\sum\limits_{i=1}^T{Y_i})\prod_{t=1}^T\phi_t(f_t(H_t);a_t)}{\prod_{t=1}^T\pi_t(a_t\mid H_t)}\ \bl\  f=\trans(g), g\in\U_n\rbs.\]
Note that $\mathcal D\mapsto \sum_{i=1}^TY_i/\prod_{t=1}^T\pi_t(a_t\mid H_t)$ is a fixed function and it is also bounded  because the $Y_i$'s are bounded by a constant, say $C_{\text{max}}$, by Assumption IV,  and the $\pi_t$'s are  bounded away from $C_\pi>0$ by Assumption I. Since the $Y_t$'s are also positive by Assumption V, the above function takes value in  $[0,TC_{\text{max}}/\pb^T]$. Let us denote $c=TC_{\text{max}}/\pb^T$. Letting
\[\C_2=\lbs \mathcal D\mapsto \prod_{t=1}^T\phi_t(f_t(H_t);A_t)\mid  f=\trans(g),\ g\in\U_n\rbs,\]
we note that
\begin{equation}
    \label{intheorem: est: C1 and C2}
    N(\epsilon,\C_1,\|\cdot\|_\infty)\leq N(\epsilon/c,\C_2,\|\cdot\|_\infty).
\end{equation}
Thus to apply Fact \ref{fact: rademacher complexity bund for basis expansion classes}, it suffices to show that there exists $\rho_n'>0$ so that
\[N_{[]}(\epsilon,\C_2,\|\cdot\|_\infty)\lesssim \lb\frac{\mathcal{I}_n}{\e}\rb^{\rho'_n}\text{ for all }\e>0.\]
Recall that $g_t\in\U_{tn}^{k_t-1}$ and $g_{ti}\in\U_{tn}$ for each $i\in[k_t]$.   For  a fixed $(i,t)$ pair with $t\in[T]$ and $i\in[k_t]$, we define the class of functions
\[\C(i,t):=\lbs u:\H_t\mapsto\RR\mid u(h_t)=\phi_t(f_t(h_t);i),\  f_t=\trans(g_t),\ g_{t}\in\U_{tn}^{k_t-1}\rbs.\]
We will apply Lemma \ref{lemma: est: bracketing entropy of lipshitz} in Section \ref{secpf: est error: auxiliary lemmas} with $\G=\C(i,t)$, $\C=\U_{tn}$,  $\X=\H_t$, $k=k_t$, $\phi=\phi_t$, and  $u(\cdot)=\phi_t(f_t(\cdot),i)$ to obtain its covering number. Lemma \ref{lemma: est: bracketing entropy of lipshitz} applies here because, 
by Condition \ref{cond: estimation error}, for fixed $i\in[k_t]$, the function $\mx_t\mapsto \phi_t(\mx_t,i)$ is globally Lipschitz. Lemma \ref{lemma: est: bracketing entropy of lipshitz} yields
\begin{equation}
   \label{intheorem: est: bracketing entropy level 1}  
   N(\e,\C(i,t),\|\cdot\|_\infty)\lesssim  N(\e/(C\sqrt{k_t}),\U_{tn},\|\cdot\|_\infty)^{k_t-1}\leq N(\e/(C\sqrt{k_t}),\U_{tn},\|\cdot\|_\infty)^{k_t}.
\end{equation}
Now consider the function class 
\begin{align*}
    \C(t):=\lbs u: \H_t\times[k_t]\mapsto \RR\mid &\  u(h_t,i)=\phi_t(f_t(h_t),i) \text { for all }h_t\in\H_t,\\
    &\ i\in[k_t],\text{ where }f_t=\trans(g_t),\ g_{t}\in\U_{tn}^{k_t-1}\rbs.
\end{align*}
We now use Lemma \ref{lemma: est: support partition} in Section \ref{secpf: est error: auxiliary lemmas} to bound the covering number of $\C(t)$. To this end, we take $\X=\H_t$, $\G=\C(t)$, $u_a=\phi_t(f_t(\cdot),a)$,  $\C_i=\C(i,t)$, and $\A=[k_t]$ to obtain
\begin{align}
    \label{intheorem: est: bracketing entropy level 2}
    N(\e,\C(t),\|\cdot\|_\infty)\lesssim N(\e,\C(i,t),\|\cdot\|_\infty)^{k_t}.
\end{align}
Now note that the functions in $\C_2$ are $T$-fold products of functions from function-classes $\C(1),\ldots,\C(T)$. Moreover, the functions in  $\C(t)$ are positive for each $t\in[T]$ because $\phi_t$'s are positive. We also showed that the $\phi_t$'s are also uniformly bounded by one. 
Hence, the functions in  $\C(t)$  are also bounded by one. Therefore, 
\begin{align}
   \label{intheorem: est: bracketing entropy level 3}  
    N(\e,\C_2,\|\cdot\|_\infty)\lesssim  \prod_{t=1}^TN(\e,\C(t),\|\cdot\|_\infty).
\end{align}
Combining \eqref{intheorem: est: bracketing entropy level 1}, \eqref{intheorem: est: bracketing entropy level 2}, and \eqref{intheorem: est: bracketing entropy level 3} with \eqref{ineq: bracketing entropy: ub}, we obtain that 
\[   N(\e,\C_2,\|\cdot\|_\infty)\lesssim  \prod_{t=1}^TN(\e,\C(t),\|\cdot\|_\infty)\leq \prod_{t=1}^T\lb \frac{\mathcal{I}_n C\sqrt{k_t}}{\e}\rb^{\rho_nk_t^2}\lesssim \lb \frac{\mathcal{I}_n C\sqrt{\K}}{\e}\rb^{\rho_n\sum\limits_{t=1}^Tk_t^2}\]
where $\K=\sum_{t=1}^Tk_t$.
Since the $k_t$'s, $T$, and $C$ are finite, 
\[ \lb \frac{\mathcal{I}_n C\sqrt{\K}}{\e}\rb^{\rho_n\sum\limits_{t=1}^Tk_t^2}\lesssim  \lb \frac{\mathcal{I}_n}{\e}\rb^{\rho_n\sum\limits_{t=1}^Tk_t^2}=\lb \frac{\mathcal{I}_n}{\e}\rb^{\rho'_n},\]
where $\rho_n'=\rho_n\sum_{t=1}^Tk_t^2$.
Since the $k_t$'s are finite integers,   $\liminf_n\rho_n>0$, $\rho_n\log \mathcal{I}_n=o(n)$, and  $\liminf_n\rho_n\log \mathcal{I}_n>0$, it holds that $\liminf_n\rho'_n>0$, $\rho'_n\log \mathcal{I}_n=o(n)$, and  $\liminf_n\rho'_n\log \mathcal{I}_n>0$. Therefore, \eqref{intheorem: est: C1 and C2} implies $N(\e,\C_1,\|\cdot\|_\infty)\lesssim (\mathcal{I}_n /\e)^{\rho'_n}$. Therefore, we have established $N(\e,\G_n(\e_n),\|\cdot\|_\infty)\lesssim (\mathcal{I}_n /\e)^{\rho'_n}$. By Lemma 9.1 of \cite{jonnotes}, $N(\e,\G_n(\e_n),L_2(Q))\lesssim N(\e/2,\G_n(\e_n),\|\cdot\|_\infty)$ for any measure $Q$. In particular, when we take $Q=\PP_n$, we obtain that  $N(\e,\G_n(\e_n),L_2(\PP_n))\lesssim (\mathcal{I}_n/\epsilon)^{\rho_n'}$. Therefore, Fact \ref{fact: rademacher complexity bund for basis expansion classes} can be applied on  $\G_n(\e_n)$ with $\rho_n$ replaced by $\rho_n'$.
 The rest of the proof is similar to the proof of Step 2 of Theorem 5 of \cite{Laha2021surrogateu}, and hence omitted.

 \subsection{Proof of auxiliary lemmas for proving Theorem \ref{thm: est error}}
 \label{secpf: est error: auxiliary lemmas}
\begin{lemma}
\label{lemma: est: bracketing entropy of lipshitz}
Let $\mathcal X$ be a space, and  $\mathcal C$ is a space of functions mapping $\mathcal X$ to $\RR$. Let  $\phi$ be a fixed Lipschitz map  with constant $C>0$ from the metric space $(\RR^{k-1},l_2)$ to the metric space $\RR$ endowed with the usual absolute value metric.  Define 
\[\mathcal G=\lbs u:\mathcal X\mapsto\RR\ \mid\   u=\phi(0,g_1,\ldots,g_{k-1}),  g_i\in\C\text{ for all } i\in[k-1]\rbs.\] 
Further suppose the covering number $ N(\epsilon,\mathcal C,\|\cdot\|_\infty)$ is  finite for each $\epsilon>0$.
Then
\[ N(\epsilon,\mathcal G,\|\cdot\|_\infty)\leq  N_{}(\epsilon/(C\sqrt{k-1}),\mathcal C,\|\cdot\|_\infty)^{k-1}.\]
\end{lemma}

\begin{proof}[Proof of Lemma \ref{lemma: est: bracketing entropy of lipshitz}]
Any $u\in\mathcal G$ is of the form $u=\phi(0,g_1,\ldots,g_{k-1})$ where the $g_i$'s are in $\mathcal C$. Suppose $N= N(\epsilon/(C\sqrt{k-1}),\mathcal C,\|\cdot\|_\infty)$. Then there exists an $\epsilon/(C\sqrt{k-1})$ covering $\Theta$ of $\mathcal C$ of size $N$. 
Suppose  $\{\mathfrak{A}_1,\ldots,\mathfrak{A}_{k-1}\}\subset\Theta$ is such that
\[\|\mathfrak{A}_i-g_i\|_\infty<\epsilon/(C\sqrt{k-1}) \quad \text{ for all }i\in[k-1].\]
Let us denote $u_{cov}=\phi(0,\mathfrak{A}_1,\ldots,\mathfrak{A}_{k-1})$.
For any $\mx\in\mathcal X$,
\begin{align*}
|u(\mx)-u_{cov}(\mx)|= &\ |\phi(g_1(\mx),\ldots,g_{k-1}(\mx))-\phi(\mathfrak{A}_1(\mx),\ldots,\mathfrak{A}_{k-1}(\mx))|\\
\leq&\ C\sqrt{\sum\limits_{i=1}^{k-1}(g_i(\mx)-\mathfrak{A}_i(\mx))^2}
\end{align*}
because $\phi$ is Lipschitz with constant $C$.
 Thus
 \begin{align*}
    \|u-u_{cov}\|_\infty\leq C\sqrt{k-1}\sup_{1\leq i\leq k-1}\|g_i-\mathfrak{A}_i\|_\infty\leq\epsilon.
\end{align*}
Therefore, 
\[\mathcal G_{\Theta}=\lbs g:\mathcal X\mapsto\RR\ \mid\   g=\phi(0,\mathfrak{A}_1,\ldots,\mathfrak{A}_{k-1}), \mathfrak{A}_i\in\Theta\text{ for all } i\in[k-1]\rbs\] 
is an $\epsilon$-covering of $\mathcal G$. Note that the cardinality of $\mathcal G_{\Theta}$ is $N^{k-1}$. Thus, the proof follows. 

\end{proof}

\begin{lemma}
\label{lemma: est: support partition}
    Suppose $\mathcal X$ is a space, and $\mathcal A$ is a finite set.
    For each $a\in \mathcal A$, let $\C_a$ be a space of functions mapping $\mathcal X$ to $\RR$. Let us define
    \[\mathcal G=\lbs f: \X\times\A\mapsto\RR\mid f(\mx,a)=u_a(\mx) \text{ where } u_a\in \C_a \rbs.\]
    Then 
    \[ N(\epsilon,\mathcal G,\|\cdot\|_\infty)\leq \prod_{a\in\A} N(\epsilon,\mathcal C_a,\|\cdot\|_\infty).\]
\end{lemma}

\begin{proof}[Proof of Lemma \ref{lemma: est: support partition}]
We will construct an $\epsilon$-covering of $\mathcal G$.  For each $a\in\A$, let $\Theta_a$ be a covering of $\C_a$. Suppose the cardinality of $\Theta_a$ is $N_a$. Then 
\[N(\epsilon,\C_a,\|\cdot\|_\infty)\leq N_a.\]
Every  $f\in\mathcal G$ is of the form \[f(\mx,i)=\sum\limits_{a\in\A}u_a(\mx)1[a=i],\quad \text{for all }\mx\in\X,\ i\in\A,\] 
where the $u_a$'s are some functions satisfying  $u_a\in\C_a$ for each $a\in\A$. There exist $\mathfrak{A}_a\in \Theta_a$ for each $a\in\A$ so that $\|\mathfrak{A}_a-u_a\|_\infty\leq \epsilon$.
Then $f_{\text{cover}}(\mx,i)=\sum_{a\in\A}\mathfrak{A}_a(\mx)1[i=a]$
satisfies
\begin{align*}
    \|f_{\text{cover}}-f\|_\infty\leq \sup_{a\in\A}\|\mathfrak{A}_a-u_a\|_\infty =\epsilon.
\end{align*}
Therefore
 \[\lbs f: \X\times\A\mapsto\RR\mid f(\mx,i)=\sum\limits_{a\in\A}\mathfrak{A}_a(\mx)1[i=a] \text{ where } \mathfrak{A}_a\in \Theta_a \text{ for each }a\in\A\rbs\]
is an $\epsilon$-cover of $\mathcal G$. The cardinality of the above set is bounded by $\prod_{a\in\A}N_a$, which completes the proof of the current lemma.

\end{proof}

\subsubsection{Proof of Lemma \ref{lemma: est: l2l1}}
\label{secpf: est error: main lemma}
\begin{proof}[Proof of Lemma \ref{lemma: est: l2l1}]
    We will first show that for any $f\in\F$, 
     \[\|\mathfrak{U}_{\fs,f}\|_{\PP,1}\leq C\lb V^\psi(\fs)-V^\psi(f)\rb^{\alpha/(1+\alpha)}.\]
    Note that $\|\mathfrak{U}_{\fs,f}\|_{\PP,1}$ equals
    \begin{align*}
        \|\mathfrak{U}_{\fs,f}\|_{\PP,1}= \E\left[\frac{\sum\limits_{i=1}^TY_i}{\prod_{i=1}^T\pi_i(A_i\mid H_i)}\bl \prod_{t=1}^T\phi_t(\fs_t(H_t);A_t)-\prod_{t=1}^T\phi_t(f_t(H_t);A_t)\bl\right]
    \end{align*}
    where we used the fact that $Y_i$'s and $\pi_i$'s  are positive. Now we will use the following lemma, which is proved in Section \ref{secpf: lemma telescopic}. 
   \begin{lemma}
\label{lemma: approx: telescopic}
Under the setup of Theorem \ref{thm: est error}, any $f\equiv(f_1,\ldots,f_T)\in\F$ satisfies
\begin{align*}
&\ \prod_{i=1}^T\phi(\fs_i(H_i);A_i)  -\prod_{i=1}^T\phi(f_i(H_i);A_i)\\
=&\ \sum\limits_{t\in[T]}\slb\phi_t(\fs_t(H_t);A_t)-\phi_t(f_t(H_t);A_t)\srb\prod_{i=1}^{t-1}\phi_i(f_i(H_i);A_i)\prod_{i=t+1}^T\phi_i(\fs_i(H_i);A_i),
\end{align*}
    where products over $\prod_{i=1}^0$ or $\prod_{i=T+1}^T$ are taken to be one.
\end{lemma} 
    
 Since the $\phi_t$'s are positive,  using Lemma \ref{lemma: approx: telescopic}, we obtain
    \begin{align*}
\E[|\mathfrak{U}_{\fs,f}|]=&\ \E\left[\frac{(\sum\limits_{i=1}^TY_i)\prod_{i=1}^{t-1}\phi_i(f_i(H_i);A_i)\prod_{i=t+1}^T\phi_i(\fs_i(H_i);A_i)}{\prod_{i=1}^T\pi_i(A_i\mid H_i)}\right.\\
&\ \times\left. \bl\sum\limits_{t\in[T]}\slb\phi_t(\fs_t(H_t);A_t)-\phi_t(f_t(H_t);A_t)\srb\bl\right],
    \end{align*}
    where the products over $\prod_{1}^0$ and $\prod_{T+1}^T$ are taken to be one. 
By  triangle inequality, the above expression is bounded above by 
    \begin{align*}
    \MoveEqLeft \sum\limits_{t\in[T]}\E\left[\frac{(\sum\limits_{i=1}^TY_i)\prod_{i=1}^{t-1}\phi_i(f_i(H_i);A_i)\prod_{i=t+1}^T\phi_i(\fs_i(H_i);A_i)}{\prod_{i=1}^T\pi_i(A_i\mid H_i)}\right.\\
&\ \times\left. |\phi_t(\fs_t(H_t);A_t)-\phi_t(f_t(H_t);A_t)|\right]\\
=&\ \sum\limits_{t\in[T]}\E\left[\E\left[\frac{\slb\sum\limits_{i=1}^T Y_i\srb\prod_{i=t+1}^T\phi_i(\fs_i(H_i);A_i)}{\prod_{i=t+1}^T\pi_i(A_i\mid H_i)}\bl H_t,A_t\right]\right.\\
\times &\ \left.\frac{|\slb\phi_t(\fs_t(H_t);A_t)-\phi_t(f_t(H_t);A_t)\srb\prod_{i=1}^{t-1}\phi_i(f_i(H_i);A_i)|}{\prod_{i=1}^{t}\pi_i(A_i\mid H_i)}\right].
    \end{align*}
Combining \eqref{intheorem: def: f star}, which gives $\phi_i(\fs_i(H_i);A_i)=1[A_i=d_i^*(H_i)]$ for all $i\in[T]$, with Fact \ref{fact: Q function expression}, we obtain that
\[\E\left[\frac{(\sum\limits_{i=1}^TY_j)\prod_{i={t+1}}^{T}\phi_i(\fs_i(H_i);A_i)}{\prod_{i={t+1}}^{T}\pi_i(A_i\mid H_i)}\bl H_t, A_t\right]=Q_t^*(H_t,A_t). \]
Therefore, we have obtained that $\|\mathfrak{U}_{\fs,f}\|_{\PP,1}$ is bounded by 
\begin{align}
\label{inlemma: upper bound of l1 error}
\MoveEqLeft \sum\limits_{t\in[T]}\E\lbt \frac{\prod_{i=1}^{t-1}\phi_i(f_i(H_i);A_i)}{\prod_{i=1}^{t-1}\pi_i(A_i\mid H_i)}\frac{Q_t^*(H_t,A_t)|\phi_t(\fs_t(H_t);A_t)-\phi_t(f_t(H_t);A_t)|}{\pi_t(A_t\mid H_t)}\rbt\nn\\
 \stackrel{(a)}{=}&\ \sum\limits_{t\in[T]}\E\left [ \frac{\prod_{i=1}^{t-1}\phi_i(f_i(H_i);A_i)}{\prod_{i=1}^{t-1}\pi_i(A_i\mid H_i)}\E\left [ \frac{Q_t^*(H_t,A_t)|\phi_t(\fs_t(H_t);A_t)-\phi_t(f_t(H_t);A_t)|}{\pi_t(A_t\mid H_t)}\bl H_t\right ]\right ]\nn\\
 \stackrel{(b)}{=}&\ \sum\limits_{t\in[T]}\E\lbt \frac{\prod_{i=1}^{t-1}\phi_i(f_i(H_i);A_i)}{\prod_{i=1}^{t-1}\pi_i(A_i\mid H_i)}\sum\limits_{i\in[k_t]}{Q_t^*(H_t,i)|\phi_t(\fs_t(H_t);i)-\phi_t(f_t(H_t);i)|}\rbt\
\end{align}
where step (a) follows because $(H_i,A_i)\subset H_t$ for any $i\in[t-1]$ for all  $t\in[2:T]$, and step (b) follows because
\begin{align*}
   \MoveEqLeft \E\left [ \frac{Q_t^*(H_t,A_t)|\phi_t(\fs_t(H_t);A_t)-\psi(f_i(H_t);A_t)|}{\pi_t(A_t\mid H_t)}\mid H_t\right ]\\
   =&\ \sum\limits_{i\in[k_t]}{Q_t^*(H_t,i)|\phi_t(\fs_t(H_t);i)-\psi(f_i(H_t);i)|}, 
\end{align*}
which is  obtained by \eqref{intheorem: necessity: IPW: general p}.

\subsubsection{Lower bound on $V^\psi(\fs)-V^\psi(f)$}
Since $V^\psi_*=V^\psi(\fs)$ by Lemma \ref{lemma: the sup of V psi equals V star}, $V^\psi(\fs)-V^\psi(f)=V^\psi_*-V^\psi(f)$,  which, by  Lemma \ref{lemma: approx:  difference decomposition in terms of Q-functions}, equals
\begin{align*}
    \MoveEqLeft\E\left[ \sum\limits_{t\in[T]}\prod_{i=1}^{t-1}\frac{\phi_i(f_i(H_i);A_i)}{\pi_i(A_i\mid H_i)} \lbs\sum\limits_{i\in[k_t]:i\neq d_t^*(H_t)}\slb Q_t^*(H_t; d_t^*(H_t))-Q_t^*(H_t; i)\srb \phi_t(f_t(H_t);i) \rbs\right].
\end{align*}
Equation \ref{intheorem: def: f star} implies $\phi_t(\fs_t(H_t);i)=0$ for $i\neq d_t^*(H_t)$. Therefore,
\begin{align*}
  V(\fs)-V^\psi(f) =&\ \sum\limits_{t\in[T]} \E\left[ \prod_{i=1}^{t-1}\frac{\phi_i(f_i(H_i);A_i)}{\pi_i(A_i\mid H_i)}\lbs \sum\limits_{i\in[k_t]:i\neq d_t^*(H_t)}\slb Q_t^*(H_t; d_t^*(H_t))-Q_t^*(H_t; i)\srb \right.\\
\times   &\ \left.\slb \phi_t(f_t(H_t);i)-\phi_t(\fs_t(H_t);i)\srb \rbs\right] 
\end{align*}
Note that the integrated is always positive because for $i\neq d_t^*(H_t)$,
\[\slb Q_t^*(H_t; d_t^*(H_t))-Q_t^*(H_t; i)\srb \slb \phi_t(f_t(H_t);i)-\phi_t(\fs_t(H_t);i)\srb\geq 0.\]
Let us fix $z>0$ and denote
\[\mE_t=\{h_t\in\H_t: \mu(Q_t^*(h_t))\leq z\}. \]
Let us denote the random variables $1[H_t\in\mE_t]$ and $1[H_t\in\mE^c_t]$ by $1_{\mE_t}$ and $1_{\mE_t^c}$, respectively.

Then $ V^\psi(\fs)-V^\psi(f)$ is bounded below by 
\begin{align}
\label{inlemma: est: l1's upper bound}
 V^\psi(\fs)-V^\psi(f)\geq &\  \sum\limits_{t\in[T]}\E\left[ \prod_{i=1}^{t-1}\frac{\phi_i(f_i(H_i);A_i)}{\pi_i(A_i\mid H_i)}\lbs \sum\limits_{i\in[k_t]:i\neq d_t^*(H_t)}(Q_t^*(H_t; d_t^*(H_t))-Q_t^*(H_t; i)) \right.\nn\\
 &\ \left. \times \slb \phi_t(f_t(H_t);i)-\phi_t(\fs_t(H_t);i)\srb \rbs 1_{\mE_t^c}\right] 
\end{align}
By Assumptions IV, $Y_t\leq C_{max}$ for some $C_{max}>0$ for all $t\in[T]$. Thus, $Q_t^*(H_t;i)+Q_t^*(H_t,d_t^*(H_t))\leq 2C_{max}$. Hence, 
For $H_t\in\mE_t^c$ and $i\neq d_t^*(H_t)$,
\begin{equation*}
    Q_t^*(H_t,d_t^*(H_t))-Q_t^*(H_t,i)\geq \mu(Q_t^*(H_t))> z\stackrel{(a)}{\geq} \frac{z \slb Q_t^*(H_t;i)+Q_t^*(H_t,d_t^*(H_t))\srb}{2C_{max}},
\end{equation*}
where (a) follows because 
$\{Q_t^*(H_t;i)+Q_t^*(H_t,d_t^*(H_t))\}/( 2C_{max})\leq 1$. 
When  $i\neq d_t^*(H_t)$,
\[\phi_t(f_t(H_t);i)-\phi_t(\fs_t(H_t);i)=\phi_t(f_t(H_t);i)\geq 0,\]
which implies that if  $i\neq d_t^*(H_t)$ and  $H_t\in\mE^c_t$, then
\begin{align}
  \label{intheorem: est: Q minus to Q plus}
   \MoveEqLeft  \slb Q_t^*(H_t,d_t^*(H_t))-Q_t^*(H_t,i)\srb \slb \phi_t(f_t(H_t);i)-\phi_t(\fs_t(H_t);i)\srb\nn\\
   \geq &\ \frac{z}{2C_{max}}\slb Q_t^*(H_t,d_t^*(H_t))+Q_t^*(H_t,i)\srb \slb \phi_t(f_t(H_t);i)-\phi_t(\fs_t(H_t);i)\srb. 
\end{align}
However, 
\begin{align}
\label{inlemma: est: q-functions}
   \MoveEqLeft \sum\limits_{i\in[k_t]:i\neq d_t^*(H_t)}\slb Q_t^*(H_t; d_t^*(H_t))+Q_t^*(H_t; i)\srb \slb \phi_t(f_t(H_t);i)-\phi_t(\fs_t(H_t);i)\srb\nn\\
  \stackrel{(a)}{=} &\ Q_t^*(H_t; d_t^*(H_t))\sum\limits_{i\in[k_t]:i\neq d_t^*(H_t)}\phi_t(f_t(H_t);i)\nn\\
  &\ +\sum\limits_{i\in[k_t]:i\neq d_t^*(H_t)}Q_t^*(H_t; i) \slb \phi_t(f_t(H_t);i)-\phi_t(\fs_t(H_t);i)\srb
  \end{align}
where in step (a), we used the fact that $\phi_t(\fs_t(H_t);i)=0$ for all $i\neq d_t^*(H_t)$, which follows from \eqref{intheorem: def: f star}. Since $\sum_{i\in[k_t]}\phi_t(\mx;i)=1$ for all $\mx\in\RR^{k_t}$, the above equals
  \begin{align}
  \label{inlemma: est: q-functions 2}
  \MoveEqLeft Q_t^*(H_t; d_t^*(H_t))\slb 1-\phi_t(f_t(H_t);d_t^*(H_t))\srb\nn\\
  &\ +\sum\limits_{i\in[k_t]:i\neq d_t^*(H_t)}Q_t^*(H_t; i) \slb \phi_t(f_t(H_t);i)-\phi_t(\fs_t(H_t);i)\srb\nn\\
  \stackrel{(a)}{=}&\  Q_t^*(H_t; d_t^*(H_t))\slb \phi_t(\fs_t(H_t);d_t^*(H_t))-\phi_t(f_t(H_t);d_t^*(H_t))\srb\nn\\
  &\ +\sum\limits_{i\in[k_t]:i\neq d_t^*(H_t)}Q_t^*(H_t; i) \slb \phi_t(f_t(H_t);i)-\phi_t(\fs_t(H_t);i)\srb\nn\\
  =&\ \sum\limits_{i\in[k_t]}Q_t^*(H_t; i) \slb \phi_t(f_t(H_t);i)-\phi_t(\fs_t(H_t);i)\srb
\end{align}
 where in step (a), we have used the fact that $\phi_t(\fs_t(H_t);d_t^*(H_t))=1$, which follows from  \eqref{intheorem: def: f star}. 
 Now we will show that
 \begin{align}
     \label{intheorem: est: non-negativity}
     \phi_t(f_t(H_t);i)-\phi_t(\fs_t(H_t);i)\geq 0\text{ for all }i\in[k_t]\text{ and }t\in[T].
 \end{align}
To this end, first consider the case when $i=d_t^*(H_t)$. Since $\phi_t\leq 1$ and\\ $\phi_t(\fs_t(H_t);d_t^*(H_t))=1$,
\[\phi_t(\fs_t(H_t);d_t^*(H_t))-\phi_t(f_t(H_t);d_t^*(H_t))\geq 0.\]
When $i\neq d_t^*(H_t)$, \eqref{intheorem: est: non-negativity} still holds since $\phi_t(\fs_t(H_t);i)=0$ for all $i\neq d_t^*(H_t)$ by \eqref{intheorem: def: f star}, and $\phi_t\geq 0$.
Then \eqref{intheorem: est: non-negativity}, combined with \eqref{inlemma: est: q-functions} and \eqref{inlemma: est: q-functions 2}, implies that 
\begin{align*}
 \MoveEqLeft  \sum\limits_{i\in[k_t]:i\neq d_t^*(H_t)}\slb Q_t^*(H_t; d_t^*(H_t))+Q_t^*(H_t; i)\srb \slb \phi_t(f_t(H_t);i)-\phi_t(\fs_t(H_t);i)\srb\\  
  \geq&\ \sum\limits_{i\in[k_t]}Q_t^*(H_t; i) |\phi_t(f_t(H_t);i)-\phi_t(\fs_t(H_t);i)|.  
\end{align*}
The above, combined with \eqref{inlemma: est: l1's upper bound} and \eqref{intheorem: est: Q minus to Q plus}, implies that $V^\psi(\fs)-V^\psi(f)$ is bounded below by 
\begin{align*}
 & \frac{z}{2C_{max}}\sum\limits_{t\in[T]}\E\left[ \prod_{i=1}^{t-1}\frac{\phi_i(f_i(H_i);A_i)}{\pi_i(A_i\mid H_i)}\lbs \sum\limits_{i\in[k_t]}Q_t^*(H_t;i)| \phi_t(f_t(H_t);i)-\phi_t(\fs_t(H_t);i)|  \rbs 1_{\mE_t^c}\right]\\
    =&\   \frac{z}{2C_{max}} \sum\limits_{t\in[T]}\E\left[ \prod_{i=1}^{t-1}\frac{\phi_i(f_i(H_i);A_i)}{\pi_i(A_i\mid H_i)}\lbs \sum\limits_{i\in[k_t]}Q_t^*(H_t;i)| \phi_t(f_t(H_t);i)-\phi_t(\fs_t(H_t);i)|  \rbs \right]\\
    -  &\ \frac{z}{2C_{max}}\sum\limits_{t\in[T]}\E\left[ \prod_{i=1}^{t-1}\frac{\phi_i(f_i(H_i);A_i)}{\pi_i(A_i\mid H_i)}\lbs \sum\limits_{i\in[k_t]}Q_t^*(H_t;i)| \phi_t(f_t(H_t);i)-\phi_t(\fs_t(H_t);i)|  \rbs 1_{\mE_t}\right]
  \end{align*}

Because $\phi_t\leq 1$ for all $t\in[T]$, $\sum_{i=1}^T|Y_i|\leq TC_{\text{max}}$, and $\pi_t\geq C_{\pi}$ for all $t\in[T]$,
it follows that
\[\E\left[ \prod_{i=1}^{t-1}\frac{\phi_i(f_i(H_i);A_i)}{\pi_i(A_i\mid H_i)}\left(\mathlarger{\sum}\limits_{i\in[k_t]}\begin{matrix}
    | \phi_t(f_t(H_t);i)-\phi_t(\fs_t(H_t);i)|\\
    \times Q_t^*(H_t;i) 
\end{matrix} \right ) 1_{\mE_t}\right]\leq  \frac{k_tC_{max}\PP(\mE_t)}{C_{\pi}^T},\]
which is less than $Cz^{\alpha}$, 
where the last step follows from Assumption \ref{assump: small noise}. Here $C>0$ is a constant depending only on $\PP$. Thus we have showed that
\begin{align*}
&\  V(\fs)-V^\psi(f)
     \geq  -Cz^{1+\alpha}\\
     &\ +\frac{z}{2C_{max}}\sum\limits_{t\in[T]}\E\left[ \prod_{i=1}^{t-1}\frac{\phi_i(f_i(H_i);A_i)}{\pi_i(A_i\mid H_i)}\lbs \sum\limits_{i\in[k_t]}Q_t^*(H_t;i)| \phi_t(f_t(H_t);i)-\phi_t(\fs_t(H_t);i)|  \rbs\right],
\end{align*}
which, combined with the upper bound on $\|\mathfrak{U}_{\fs,f}\|_{\PP,1}$
from \eqref{inlemma: upper bound of l1 error}, leads to
\begin{align}
\label{inlemma: est: the bound with l1}
   V^\psi(\fs)-V^\psi(f)
     \geq Cz\|\mathfrak{U}_{\fs,f}\|_{\PP,1} -Cz^{1+\alpha}  \end{align}
 for some constant $C>0$. Since $\pi_t\geq \pb$, $\sum_{t\in[T]}|Y_t|\leq TC_{\text{max}}$ and $\phi_t\leq 1$, it follows that
   \[\abs{\slb\sum\limits_{t\in[T]} Y_t\srb\prod_{j=1}^T\frac{\psi(f_j(H_j);A_j)}{\pi_j(A_j\mid H_j)}}\leq \frac{TC_{\text{max}}}{\pb^T}.\]  
    Therefore $\mathfrak{U}_{\fs,f}$ is uniformly bounded by $TC_{max}C_\pi^{-T}$. Therefore, 
    \[\|\mathfrak{U}_{\fs,f}\|_{\PP,2}\leq TC_{max}C_\pi^{-T} \|\mathfrak{U}_{\fs,f}\|_{\PP,1}.\]
    Hence, \eqref{inlemma: est: the bound with l1} leads to 
    \[ V(\fs)-V^\psi(f)
     \geq C\slb z\|\mathfrak{U}_{\fs,f}\|_{\PP,2} -z^{1+\alpha}\srb,\]
     which leads to
     \[\|\mathfrak{U}_{\fs,f}\|_{\PP,2}\leq z^\alpha+\slb  V(\fs)-V^\psi(f)\srb z^{-1}. \]
Optimizing over $z$, we find that the upper bound is smallest when $z=(V^{\psi}(\fs)-V^\psi(f))^{1/(1+\alpha)}$, which completes the proof.
 \end{proof} 

\subsubsection{Proof of Lemma \ref{lemma: approx: telescopic}}
\label{secpf: lemma telescopic}
 \begin{proof}[Proof of Lemma \ref{lemma: approx: telescopic}]
The proof follows noting the following telescoping sum 
\begin{align}
\label{inlemma: approx: telescopic: 1st}
 \MoveEqLeft \sum\limits_{t\in[T]}\slb\phi_t(\fs_t(H_t);A_t)-\phi_t(f_t(H_t);A_t)\srb\prod_{i=1}^{t-1}\phi_t(f_i(H_i);A_i)\prod_{i=t+1}^T\phi_i(\fs_i(H_i);A_i)\nn \\
 =&\ \prod_{i=1}^T\phi_i(\fs_i(H_i);A_i)+\sum\limits_{t=2}^{T}\prod_{i=1}^{t-1}\phi_t(f_i(H_i);A_i)\prod_{i=t}^T\phi_i(\fs_i(H_i);A_i)\nn\\
 &\ - \sum\limits_{t=1}^{T-1}\prod_{i=1}^{t}\phi_t(f_i(H_i);A_i)\prod_{i=t+1}^T\phi_i(\fs_i(H_i);A_i)-\prod_{i=1}^T\phi_i(f_i(H_i);A_i).
\end{align}
Replacing $t$ by $t'=t-1$ in the second sum on the RHS of  \eqref{inlemma: approx: telescopic: 1st}, we obtain that
\begin{align*}
\MoveEqLeft \sum\limits_{t=2}^{T}\prod_{i=1}^{t-1}\phi_i(f_i(H_i);A_i)\prod_{i=t}^T\phi_i(\fs_i(H_i);A_i)\\
 =&\ \sum\limits_{t'=1}^{T-1}\prod_{i=1}^{t'}\phi_i(f_i(H_i);A_i)\prod_{i=t'+1}^T\phi_i(\fs_i(H_i);A_i),
\end{align*}
which cancels the third term on the RHS of  \eqref{inlemma: approx: telescopic: 1st}. Hence,  the proof follows from \eqref{inlemma: approx: telescopic: 1st}.
\end{proof}
 \subsection{Proof of Corollary \ref{cor: neural network}}
 \label{secpf: cor: NN}
\begin{proof}[Proof of Corollary \ref{cor: neural network}]
We have already mentioned in Section \ref{sec: estimation error} that $\U_{tn}={\mathcal F}(\N_n,W_n,s_n)$ satisfies \eqref{ineq: bracketing entropy: ub} with $\rho_n=s_n+1$ and $\mathcal{I}_n=(\N_n+1)(s_n+1)^{2\N_n+4}$. We will apply Theorem \ref{thm: est error}. 
To this end, we first  show that $\rho_n\log \mathcal{I}_n=o(n)$ and $\liminf_n \rho_n\log \mathcal{I}_n>0$. Since $\N_n,s_n\geq 1$, we have $\I_n\geq 2\times 2^{6}$. Thus $\log \mathcal{I}_n>0$ trivially follows. Also, $\rho_n=s_n>0$. Therefore $\rho_n\log \mathcal{I}_n>0$ for all $n>e$. Also since $\log(s_n)=O(\log n)$, $\log \mathcal{I}_n$ is of the order $O(\log n)$. Noting $\rho_n=O(s_n)=O\left(n^{\frac{q}{(2+\alpha)\mybeta+q}}\right)$, we thus conclude that $\rho_n\log \mathcal{I}_n=o(n)$.
Therefore, it suffices  to show that \eqref{ineq: bracketing entropy: ub} holds for some $\myb_n\gg \sqrt{n}$. We will take $\myb_n=n$.  Since $s_n=c_2n^{q/((2+\alpha)\mybeta+q)}$ and $\N_n=c_1\log n$, we have
\begin{align}
\label{inlemma: NN: approx}
  \MoveEqLeft   \lb \frac{\rho_n\log \mathcal{I}_n}{n}\rb^{1/(2+\alpha)}\geq C \lb \frac{s_n(2\N_n+4)\log (s_n)+s_n\log\log n}{n}\rb^{1/(2+\alpha)}\nn\\
     \geq &\ C\lb(\log n)^2\frac{s}{n}\rb^{1/(2+\alpha)}\geq C (\log n)^{\frac{2}{2+\alpha}} n^{-\frac{\mybeta}{(2+\alpha)\mybeta+ q}} \gtrsim(\log n)^{\frac{2}{2+\alpha}} s_n^{-\mybeta/q},
\end{align}  
where the last step uses $s_n^{-\mybeta/q}=n^{-\frac{\mybeta}{(2+\alpha)\mybeta+ q}}$.
  Let $t\in[T]$. Under Assumption \ref{assumption: smmothness}, 
  following the proof of Supplementary Lemma M.2 of  \cite{Laha2021surrogateu}, we can show that, for all $i\in[k_t-1]$, there exists a network $\tilde g_{ti}\in\F(\N_n,W_n,s_n)$ and a constant  $C_{q,\mybeta}>0$ so that   under our condition on $\N_n$, $W_n$, and $s_n$, the following holds for all $t\in[T]$:
\[\|\tilde g_{ti}/\myb_n-Q_t^*(\cdot,1)+Q_t^*(\cdot,i+1)\|_\infty \leq C_{q,\mybeta} s_n^{-\mybeta/q},\]
which is bounded above by a constant multiple of $\lb \frac{\rho_n\log \mathcal{I}_n}{n}\rb^{1/(2+\alpha)}$ by \eqref{inlemma: NN: approx}. Hence, \eqref{ineq: bracketing entropy: ub} holds. Therefor, the rest of the proof follows from Theorem \ref{thm: est error} noting 
\[\lb \frac{\rho_n\log \mathcal{I}_n}{n}\rb^{\frac{1+\alpha}{2+\alpha}}\lesssim (\log n)^{\frac{2(1+\alpha)}{2+\alpha}}n^{-\frac{1+\alpha}{(2+\alpha)+q/\mybeta}}.\]
\end{proof}
\subsection{Proof of Result \ref{result: kernel surrogate is Lipshitcz}}
\label{secpf: result kernel surrogate is Lipshitcz}
\begin{proof}[Proof of Result \ref{result: kernel surrogate is Lipshitcz}]
 Fix $j\in[k]$. If we can show that $\phi(\cdot;j)$ is differentiable and the gradient is uniformly bounded in $l_2$ norm, then the global Lipschitz continuity follows from the mean value property. First, we will show $\phi(\cdot;j)$ is differentiable. To this end, we first  show that $K$ is differentiable with uniformly bounded partial derivatives.

  Since $K$ satisfies $K(\mx)=\prod_{i=1}^k\KK(\mx_i)$ for all $\mx\in\RR^k$ and $\KK$ is differentiable,  the partial derivatives of $K$ exist  everywhere and equals 
\[\pdv{K(\mx)}{\mx_i}=\KK'(\mx_i)\prod_{r\in[k]:r\neq i}\KK(\mx_r)\]
for all $i\in[k]$. 
Since $\KK'$ and $\KK$ are bounded, the partial derivatives are also bounded. 

To show that  $\phi_t(\cdot,j)$ is differentiable, it suffices to show that its partial derivatives exist and they are  continuous. 
Since all partial derivatives of $K$ exist and bounded for all $\mx\in\RR^k$,  by Leibniz integral rule,
\begin{align*}
 \pdv{\phi(\mx;j)}{\mx_i}=&\ \pdv{\lbs\dint_{\RR^k}1[\pred(\mt)=j] K(\mx-\mt)d\mt\rbs}{\mx_i}\\
 =&\ \dint_{\RR^k}1[\pred(\mt)=j]\pdv{ K(\mx-\mt)}{\mx_i}d\mt\\
 =&\ \dint_{\RR^k}1[\pred(\mt)=j]\KK'(\mx_i-\mt_i)\prod_{r\in[k]:r\neq i}\KK(\mx_r-\mt_r)d\mt   
\end{align*}
for all $i\in[k]$. Thus all partial derivatives of $\phi_t(\cdot,j)$ exist.

Now we will show that the partial derivatives of $\phi(\cdot;j)$ are continuous. 
Suppose $\mx^{(n)}\to \mx$. Then 
\[\lim_{n\to\infty}\pdv{\phi(\mx_n;j)}{\mx_i}=\lim_{n\to\infty}\dint_{\RR^k}1[\pred(\mt)=j]\KK'(\mx^{(n)}_i-\mt_i)\prod_{r\in[k]:r\neq i}\KK(\mx^{(n)}_r-\mt_r)d\mt.\]
Since $\KK'$ and $\KK$ are bounded, the continuity of $\partial\phi(\cdot;j)/\partial \mx_i$ follows from the dominated convergence theorem and the continuity of $\KK$ and $\KK'$. 

Thus we have shown that all partial derivatives of $\phi(\cdot;j)$ exist and they are continuous at every $\mx\in\RR^k$. Hence, it follows that $\phi(\cdot;j)$ is differentiable with gradient
\[\grad{\phi(\mx;j)}=\dint_{\RR^k}1[\pred(\mt)=j]\grad K(\mx-\mt)d\mt.\]

To show that $\phi(\cdot;j)$ is Lipschitz, it remains to show that $\grad{\phi(\mx;j)}$ is bounded in $l_2$ norm for each $\mx\in\RR^k$. To this end, it is enough to show that the partial derivatives are uniformly bounded. However, the above follows since for any $i$ and $j$ in $[k]$, 
\begin{align*}
 \sup_{\mx\in\RR^k} \abs{\pdv{\phi(\mx;j)}{\mx_i}}
  = &\  \sup_{\mx\in\RR^k} \abs{\dint_{\RR^k}1[\pred(\mt)=j]\KK'(\mx_i-\mt_i)\prod_{r\in[k]:r\neq i}\KK(\mx_r-\mt_r)d\mt}\\
   \leq &\ \sup_{\mx\in\RR^k} \dint_{\RR^k}\abs{\KK'(\mx_i-\mt_i)}\prod_{r\in[k]:r\neq i}\KK(\mx_r-\mt_r)d\mt\\
   =&\ \sup_{\mx_i\in\RR}\dint_{\RR}|\KK'(\mx_i-\mt_i)|d\mt_i\\
   =&\ \edint |\KK'(t)|dt,
\end{align*}
which is bounded because $\KK'$ is integrable.  
\end{proof}

\section{Additional facts}
 \label{sec: additional facts}
 \begin{fact}
\label{fact: closedness}
 Suppose $f_1$, $\ldots$, $f_k$ are real-valued functions on $\RR^k$ for some $k\in\NN$. If they are closed and bounded below, then their sum is closed.   
\end{fact}
\begin{proof}[Proof of Fact \ref{fact: closedness}]
    Suppose $f_1$, $\ldots$, $f_k$ are  bounded below by $C\in\RR$.  Consider the function $h=\sum_{i=1}^k(f_k-C)$. Clearly, $h$ is a sum of non-negative closed functions. Closedness is equivalent to lower-semicontinuity \citep[pp. 78]{hiriart}.  However, sum of non-negative lower semicontinuous functions is lower semicontinuous (this fact follows trivially noting $f$ is lower semicontinuous if and only if $\liminf_{\mx\to \mx_0}f(\mx)\geq f(\mx_0)$ for all $\mx_0\in\RR^k$). Therefore, $h$ is lower semicontinuous, and hence closed. However, since $\sum_{i=1}^kf_i=h+kC$, the function $\sum_{i=1}^kf_i$ is closed as well. 
\end{proof}

\begin{fact}
    \label{fact: bivariate lsc}
   Let $\mbu\in\RR^{k_1-1}$ and $\mw\in\RR^{k_2-1}$ be two  fixed vectors.
    If $f:\RR^{k_1+k_2}\to\RR$ is closed, then so is the bivariate function $(x,y)\mapsto f(x \mbu, y\mw)$, where $x,y\in\RR$.
\end{fact}

\begin{proof}[Proof of Fact \ref{fact: bivariate lsc}]
To show closedness, it suffices to show that the above function is lower-semicontinuous everywhere.  
    Let $(x_k)$ and $(y_k)$ be two real sequences. Let us denote $\mx_k=x_k \mbu$ and $\my_k=y_k\mv$.  For any $x^0,y^0\in\RR$, it holds that 
\[\liminf_{x_k\to x^0,\ y_k\to y^0}f(x_k\mbu,y_k\mv)=\liminf_{\mx_k\to x^0\mbu,\ \my_k\to y^0\mv}f(\mx_k,\my_k)\geq f(x^0\mbu,y^0\mv)\]
because $f$ is lower semicontinuous \citep[pp. 78]{hiriart}.
\end{proof}

 \begin{fact}[Lemma 5 of \cite{schmidt2020}]\label{fact: NN covering number}
 Suppose $\F(\N,W,s)$ is a class of ReLU neural networks with linear output layer, depth $\N$, width vector $W$, sparsity $s$, and weights uniformly bounded by one.  
Denote by  $C_{\N}$ the constant$(\N+1)\prod_{i=0}^{L+1}(W_i+1)^2$. Then for any $\e>0$,
\[\log \mathcal N(\e,\F(\N,W,s),\|\cdot\|_{\infty})\lesssim  (s+1)\log(2\epsilon^{-1} C_{\N}).\]
Also since $W_i\leq s$ for each $i\in\{0,\ldots,L+1\}$, $C_{\N}\lesssim \N(s+1)^{2\N+4}$.
\end{fact}

\begin{fact}
\label{fact: Q function expression}
For $t\in[T]$, the Q-functions satisfy
    \[Q_t^*(H_t, a_t)=\E\lbt \slb\sum_{i=t}^T Y_i\srb\prod_{j=1+t}^T \frac{1[A_j=d_j^*(H_j)]}{\pi_j(A_j\mid H_j)}\bl H_t, A_t=a_t\rbt.\]
\end{fact}

\begin{proof}[Proof of Fact \ref{fact: Q function expression}]
    We will prove this fact by induction where the induction hypothesis 
    \[Q_t^*(H_t, a_t)=\E\lbt \slb\sum_{i=t}^T Y_i\srb\prod_{j=1+t}^T \frac{1[A_j=d_j^*(H_j)]}{\pi_j(A_j\mid H_j)}\bl H_t, A_t=a_t\rbt\]
   holds for $t=T$ because $Q_T^*(H_T,A_T)=\E[Y_T\mid H_T, A_T]$ and we have defined products over an empty range to be one. Suppose the induction hypothesis holds for $1+t$ where $t\in[1:T-1]$. We will show that the induction hypothesis holds for $t$ as well. 
Note that
    \begin{align}
    \label{infact: Q function}
        \MoveEqLeft \E\lbt \slb\sum_{i=t}^T Y_i\srb\prod_{j=1+t}^T\frac{1[A_j=d_j^*(H_j)]}{\pi_j(A_j\mid H_j)}\bl H_t, A_t=i\rbt\nn\\
        =&\  \E\lbt Y_t\E\lbt \prod_{j=1+t}^T\frac{1[A_j=d_j^*(H_j)]}{\pi_j(A_j\mid H_j)}\bl H_{1+t}\rbt\\
        &\ + \slb\sum_{i=t+1}^T Y_i\srb\prod_{j=1+t}^T\frac{1[A_j=d_j^*(H_j)]}{\pi_j(A_j\mid H_j)}\bl H_t, A_t=i\rbt\nn\\
        \stackrel{(a)}{=}&\  \E[Y_t\mid H_t, A_t=i]+\E\lbt \E\lbt\slb\sum_{i=t+1}^T Y_i\srb\prod_{j=1+t}^T\frac{1[A_j=d_j^*(H_j)]}{\pi_j(A_j\mid H_j)}\ \bl\ H_{1+t}\rbt\ \bl\  H_t, A_t=i\rbt
    \end{align}
    where (a) follows by applying \eqref{intheorem: fact: product of indicators} on the first term. Now  applying \eqref{intheorem: necessity: IPW: 0-1} on the second term of the RHS of \eqref{infact: Q function} with
    \[V=\slb\sum_{i=t+1}^T Y_i\srb\prod_{j=2+t}^T\frac{1[A_j=d_j^*(H_j)]}{\pi_j(A_j\mid H_j)},\]
    we see that the RHS of \eqref{infact: Q function} equals
    \begin{align*}
      \MoveEqLeft  \E[Y_t\mid H_t, A_t=i]\nn\\
         +&\ \E\lbt \E\lbt \slb\sum_{i=1+t}^T Y_i\srb\prod_{j=2+t}^T\frac{1[A_j=d_j^*(H_j)]}{\pi_j(A_j\mid H_j)}\bl H_{1+t}, A_{1+t}=d_{1+t}^*(H_{1+t})\rbt \bl H_t, A_t=i\rbt,
    \end{align*}
    where the product is one if $t=T-1$, when the range of the product is empty. 
    Since the induction hypothesis holds for $1+t$, it follows that $Q_{1+t}^*(H_{1+t},d_{1+t}^*(H_{1+t}))$ equals
    \[\E\lbt \slb\sum_{i=1+t}^T Y_i\srb\prod_{j=2+t}^T\frac{1[A_j=d_j^*(H_j)]}{\pi_j(A_j\mid H_j)}\bl H_{1+t}, A_{1+t}=d_{1+t}^*(H_{1+t})\rbt.\]
    Thus \eqref{infact: Q function} leads to  
    \begin{align*}
   \MoveEqLeft \E\lbt \slb\sum_{i=t}^T Y_i\srb\prod_{j=1+t}^T\frac{1[A_j=d_j^*(H_j)]}{\pi_j(A_j\mid H_j)}\bl H_t, A_t=i\rbt\\
    =&\  \E\slbt Y_t+Q_{1+t}^*(H_{1+t},d_{1+t}^*(H_{1+t}))\mid H_t, A_t=i\srbt,
    \end{align*}
    which equals $Q_t^*(H_t,i)$.
\end{proof}

\begin{fact}
\label{fact: as to L1}
    If $X_n$ is bounded and $X_n\to_P 0$, then $E[X_n]\to_n 0$.
\end{fact}

\begin{proof}
    Uniform integrability and convergence in probability implies $L_1$ convergence. Bounded sequences of random variables are uniformly integrable \citep{billingsley2017probability}. Hence $E[|X_n|\to_n 0$.
\end{proof}

\bibliographystyle{natbib}
\bibliography{DTR,DTR2,publist,location_estimation_1}


\end{document}